\documentclass[a4paper, 12pt]{amsart}
\usepackage[inner=3cm, outer=3cm, bottom=3cm]{geometry}

\usepackage{graphicx}
\usepackage{amssymb, amsmath, amsfonts}
\usepackage{url}
\usepackage[pdfborder={0 0 0}]{hyperref}
\usepackage{verbatim} 
\usepackage[all]{xy}
\usepackage{color}
\usepackage{marginnote}

\allowdisplaybreaks
\usepackage{enumitem}
\newcommand{\mylabel}[2]{#2\def\@currentlabel{#2}\label{#1}}

\DeclareMathOperator{\id}{Id}
\DeclareMathOperator{\meas}{meas}

\newcommand{\tnorm}[1]{\vert\hspace*{-1pt}\vert\hspace*{-1pt}\vert#1\vert\hspace*{-1pt}\vert\hspace*{-1pt}\vert}

\newcommand{\Linf}{{L^\infty}}
\newcommand{\Winf}{{W^{1,\infty}}}

\newcommand{\WX}{\ensuremath{W_X^{1,\infty}}}
\newcommand{\WY}{\ensuremath{W_Y^{1,\infty}}}

\renewcommand{\H}{\ensuremath{\mathcal{H}}}

\newcommand{\N}{\ensuremath{\mathcal{N}}}
\newcommand{\U}{\ensuremath{\mathcal{U}}}
\newcommand{\X}{\ensuremath{\mathcal{X}}}
\newcommand{\Y}{\ensuremath{\mathcal{Y}}}

\newcommand{\C}{\ensuremath{\mathcal{C}}}
\newcommand{\B}{\ensuremath{\mathcal{B}}}

\newcommand{\V}{\ensuremath{\mathcal{V}}}

\newcommand{\W}{\ensuremath{\mathcal{W}}}
\newcommand{\p}{\ensuremath{\mathfrak{p}}}
\newcommand{\q}{\ensuremath{\mathfrak{q}}}

\newcommand{\dXY}[1]{{#1}_{XY}}

\newcommand{\Real}{\mathbb R}

\newcommand{\norm}[1]{\left\Vert#1\right\Vert}

\newcommand{\Z}{\ensuremath{\mathcal{Z}}}
\newcommand{\E}{\ensuremath{\mathcal{E}}}

\newcommand{\D}{\ensuremath{\mathcal{D}}}
\newcommand{\G}{\ensuremath{\mathcal{G}}}

\newcommand{\F}{\ensuremath{\mathcal{F}}}

\allowdisplaybreaks
\newtheorem{theorem}{Theorem}[section]

\newtheorem{lemma}[theorem]{Lemma}
\newtheorem{definition}[theorem]{Definition}

\newtheorem{corollary}[theorem]{Corollary}

\numberwithin{equation}{section}

\newif\ifcolor
\colortrue  

\newenvironment{aalign}
{ 
	\csname align*\endcsname
}
{ 
	\csname endalign*\endcsname
}

\usepackage{etoolbox}
\newcommand\numberthis{\addtocounter{equation}{1}\tag{\theequation}}
\expandafter\appto\csname aalign\endcsname{\numberthis}
\allowdisplaybreaks

\begin{document}

\raggedbottom

\title{A Regularized System for the Nonlinear Variational Wave Equation}
    
\author[K. Grunert]{Katrin Grunert}
\address{Department of Mathematical Sciences\\ Norwegian University of Science and Technology\\ NO-7491 Trondheim\\ Norway}
\email{katrin.grunert@ntnu.no}
\urladdr{\url{http://www.ntnu.no/ansatte/katrin.grunert}}

\author[A. Reigstad]{Audun Reigstad}
\address{Department of Mathematical Sciences\\ Norwegian University of Science and Technology\\ NO-7491 Trondheim\\ Norway}
\email{audun.reigstad@ntnu.no}
\urladdr{\url{http://www.ntnu.no/ansatte/audun.reigstad}}

\thanks{Research supported by the grants {\it Waves and Nonlinear Phenomena (WaNP)} and {\it Wave Phenomena and Stability --- a Shocking Combination (WaPheS)} from the Research Council of Norway.}  

\begin{abstract} 
We derive a new generalization of the nonlinear variational wave equation. We prove existence of local, smooth solutions for this system. As a limiting case, we recover the nonlinear variational
wave equation.
\end{abstract}

\maketitle

\tableofcontents

\section{Introduction}

The nonlinear variational wave equation (NVW) is given by
\begin{equation}
\label{eq:nvw}
	u_{tt}-c(u)(c(u)u_{x})_{x}=0,
\end{equation}
where $u=u(t,x)$, $t\geq 0$ and $x\in\mathbb{R}$, with initial data
\begin{equation}
\label{eq:initdat}
	u|_{t=0}=u_{0} \quad \text{and} \quad u_{t}|_{t=0}=u_{1}.
\end{equation}
It was introduced by Saxton in \cite{Saxt:89}, where it is derived by applying  the variational principle to the functional
\begin{equation*}
	\int_{0}^{\infty}\int_{-\infty}^{\infty}(u_{t}^{2}-c^{2}(u)u_{x}^{2})\,dx\,dt.
\end{equation*}
It is well known that derivatives of solutions of this equation can develop singularities in finite time even for smooth initial data, see e.g. \cite{GlaHunZh:96}. The continuation past singularities is highly nontrivial, and allows for various distinct solutions. The most common way of continuing the solution is to require that the energy is non-increasing, which naturally leads to the two following notions of solutions: Dissipative solutions for which the energy is decreasing in time, see \cite{BreHua16,ZhaZhen:03,ZhaZhen:05a,ZhaZhen:05}, and conservative solutions for which the energy is constant in time. In the latter case a semigroup of solutions has been constructed in \cite{BreZhe:06,HolRay:11}. 

In this paper we modify \eqref{eq:nvw} by adding two transport equations and coupling terms. The resulting system is given by 
\begin{subequations}
	\label{eq:nvwsys}
	\begin{align}
	\label{eq:nvwsys1}
	u_{tt}-c(u)(c(u)u_{x})_{x}&=-\frac{c'(u)}{4}(\rho^{2}+\sigma^{2}),\\
	\label{eq:nvwsys2}
	\rho_{t}-(c(u)\rho)_{x}&=0,\\
	\label{eq:nvwsys3}
	\sigma_{t}+(c(u)\sigma)_{x}&=0,
	\end{align}
\end{subequations}
with initial data
\begin{equation}
\label{eq:nvwInitialData}
	u|_{t=0}=u_{0}, \quad u_{t}|_{t=0}=u_{1}, \quad \rho|_{t=0}=\rho_{0}, \quad \sigma|_{t=0}=\sigma_{0}.
\end{equation}
It is clear that when $\rho=\sigma=0$ we recover \eqref{eq:nvw}. We assume that $c\in C^{2}(\mathbb{R})$ and satisfies
\begin{equation}
\label{eq:cassumption}
	\frac{1}{\kappa}\leq c(u)\leq \kappa
\end{equation}  
for some $\kappa\geq 1$. In addition, we assume that
\begin{equation}
\label{eq:cderassumption}
	\max_{u\in \mathbb{R}}|c'(u)|\leq k_{1} \quad \text{and} \quad \max_{u\in \mathbb{R}}|c''(u)|\leq k_{2}
\end{equation}
for positive constants $k_{1}$ and $k_{2}$. 

We are interested in studying conservative solutions of the initial value problem \eqref{eq:nvwsys}-\eqref{eq:nvwInitialData} for initial data $u_{0}$, $u_{0,x}$, $u_{1}$, $\rho_{0}$, $\sigma_{0}\in L^{2}(\mathbb{R})$. For smooth and bounded solutions such that $u$, $u_{t}$, $u_{x}$, $\rho$ and $\sigma$ vanish at $\pm\infty$ the energy is given by
\begin{equation}
\label{eq:smoothEnergy}
	E(t)=\frac{1}{2}\int_{\mathbb{R}}\Big(u_{t}^{2}+c^{2}(u)u_{x}^{2}+\frac{1}{2}c(u)\rho^{2}+\frac{1}{2}c(u)\sigma^{2}\Big)\,dx,
\end{equation} 
and independent of time. One way to see this is to consider the quantity
\begin{equation*}
	K(t)=\frac{1}{2}\int_{\mathbb{R}}u_{t}^{2}\,dx,
\end{equation*}
which we can think of as the "kinetic energy". We compute $K'(t)$ and find by using \eqref{eq:nvwsys1},
\begin{equation*}
	K'(t)=\int_{\mathbb{R}}
	\Big(c(u)u_{t}(c(u)u_{x})_{x}-\frac{1}{4}c'(u)u_{t}(\rho^{2}+\sigma^{2})\Big)\,dx.
\end{equation*}
For the first term we get, by integration by parts,
\begin{equation*}
	\int_{\mathbb{R}}c(u)u_{t}(c(u)u_{x})_{x}\,dx=-\int_{\mathbb{R}}\big(c^{2}(u)u_{x}u_{xt}+c(u)c'(u)u_{t}u_{x}^{2}\big)\,dx=-\frac{d}{dt}\int_{\mathbb{R}}\frac{1}{2}c^{2}(u)u_{x}^{2}\,dx,
\end{equation*}
and for the second term we obtain from \eqref{eq:nvwsys2} and \eqref{eq:nvwsys3},
\begin{align*}
	\int_{\mathbb{R}}c'(u)u_{t}(\rho^{2}+\sigma^{2})\,dx&=\frac{d}{dt}\int_{\mathbb{R}}c(u)(\rho^{2}+\sigma^{2})\,dx-2\int_{\mathbb{R}}c(u)(\rho\rho_{t}+\sigma\sigma_{t})\,dx\\
	&=\frac{d}{dt}\int_{\mathbb{R}}c(u)(\rho^{2}+\sigma^{2})\,dx-\int_{\mathbb{R}}\Big((c^{2}(u)\rho^{2})_{x}-(c^{2}(u)\sigma^{2})_{x}\Big)\,dx\\
	&=\frac{d}{dt}\int_{\mathbb{R}}c(u)(\rho^{2}+\sigma^{2})\,dx.
\end{align*}
Therefore we get
\begin{equation*}
	K'(t)=-\frac{d}{dt}\bigg(\int_{\mathbb{R}}\frac{1}{2}c^{2}(u)u_{x}^{2}\,dx+\frac{1}{4}\int_{\mathbb{R}}c(u)(\rho^{2}+\sigma^{2})\,dx\bigg),
\end{equation*}
which implies that $E(t)$ is constant. In particular, we have
\begin{equation*}
	E(t)=\frac{1}{2}\int_{\mathbb{R}}\Big(u_{1}^{2}+c^{2}(u_{0})u_{0,x}^{2}+\frac{1}{2}c(u_{0})\rho_{0}^{2}+\frac{1}{2}c(u_{0})\sigma_{0}^{2}\Big)\,dx.
\end{equation*} 

Next, we introduce the functions $R$ and $S$ defined as
\begin{equation}
\label{eq:defRS}
\left\{
\begin{aligned}
	R&=u_{t}+c(u)u_{x}, \\
	S&=u_{t}-c(u)u_{x}.
\end{aligned}
\right.
\end{equation}
Note that $R$ and $S$ are smooth by assumption. Using \eqref{eq:defRS} we can express the energy in \eqref{eq:smoothEnergy} as
\begin{equation}
\label{eq:smoothEnergy2}
	E(t)=\frac{1}{4}\int_{\mathbb{R}}\big(R^{2}+c(u)\rho^{2}+S^{2}+c(u)\sigma^{2}\big)\,dx.
\end{equation} 
As we shall see, we can think of $R^{2}+c(u)\rho^{2}$ and $S^{2}+c(u)\sigma^{2}$ as the left and right traveling part of the energy density, respectively. Indeed, from \eqref{eq:nvwsys1} we have
\begin{equation}
\label{eq:evolRS}
\left\{
\begin{aligned}
	R_{t}-c(u)R_{x}&=\frac{c'(u)}{4c(u)}(R^{2}-S^{2})-\frac{c'(u)}{4}(\rho^{2} +\sigma^{2}),\\
	S_{t}+c(u)S_{x}&=-\frac{c'(u)}{4c(u)}(R^{2}-S^{2})-\frac{c'(u)}{4}(\rho^{2} +\sigma^{2}).
\end{aligned}
\right.
\end{equation}
Multiplying the first equation in \eqref{eq:evolRS} by $R$ and the second by $S$, using
\begin{equation*}
	(c(u)R^{2})_{x}=\frac{c'(u)}{2c(u)}R^{2}(R-S)+c(u)(R^{2})_{x},
\end{equation*}
and
\begin{equation*}
	(c(u)S^{2})_{x}=\frac{c'(u)}{2c(u)}S^{2}(R-S)+c(u)(S^{2})_{x},
\end{equation*}
yields
\begin{equation*}
\left\{
\begin{aligned}
	(R^{2})_{t}-(c(u)R^{2})_{x}&=\frac{c'(u)}{2c(u)}(R^{2}S-RS^{2})-\frac{c'(u)}{2}R(\rho^{2} +\sigma^{2}),\\
	(S^{2})_{t}+(c(u)S^{2})_{x}&=-\frac{c'(u)}{2c(u)}(R^{2}S-RS^{2})-\frac{c'(u)}{2}S(\rho^{2} +\sigma^{2}).
\end{aligned}
\right.
\end{equation*}
Moreover, using \eqref{eq:nvwsys2} and \eqref{eq:nvwsys3} we get
\begin{equation*}
\left\{
\begin{aligned}
	(\rho^{2})_{t}-(c(u)\rho^{2})_{x}&=\frac{c'(u)}{2c(u)}(R-S)\rho^{2},\\
	(\sigma^{2})_{t}+(c(u)\sigma^{2})_{x}&=-\frac{c'(u)}{2c(u)}(R-S)\sigma^{2},
\end{aligned}
\right.
\end{equation*}
which implies
\begin{equation*}
\left\{
\begin{aligned}
	(c(u)\rho^{2})_{t}-(c^{2}(u)\rho^{2})_{x}&=\frac{c'(u)}{2}(R+S)\rho^{2},\\
	(c(u)\sigma^{2})_{t}+(c^{2}(u)\sigma^{2})_{x}&=\frac{c'(u)}{2}(R+S)\sigma^{2}.
\end{aligned}
\right.
\end{equation*}
This leads to
\begin{equation}
\label{eq:evolEnergy}
\left\{
\begin{aligned}
	&(R^{2}+c(u)\rho^{2})_{t}-\Big(c(u)(R^{2}+c(u)\rho^{2})\Big)_{x}\\
	&=\frac{c'(u)}{2c(u)}(R^{2}S-RS^{2})+\frac{c'(u)}{2}(\rho^{2}S-\sigma^{2}R),\\
	&(S^{2}+c(u)\sigma^{2})_{t}+\Big(c(u)(S^{2}+c(u)\sigma^{2})\Big)_{x}\\
	&=-\frac{c'(u)}{2c(u)}(R^{2}S-RS^{2})-\frac{c'(u)}{2}(\rho^{2}S-\sigma^{2}R).
\end{aligned}
\right.
\end{equation}
From \eqref{eq:evolEnergy} we get
\begin{equation}
\label{eq:evolEnergy2}
\left\{
\begin{aligned}
	&\bigg(\frac{1}{c(u)}\big(R^{2}+c(u)\rho^{2}\big)\bigg)_{t}-(R^{2}+c(u)\rho^{2})_{x}\\
	&=-\frac{c'(u)}{2c^{2}(u)}(R^{2}S+RS^{2})-\frac{c'(u)}{2c(u)}(\rho^{2}S+\sigma^{2}R),\\
	&\bigg(\frac{1}{c(u)}\big(S^{2}+c(u)\sigma^{2}\big)\bigg)_{t}+(S^{2}+c(u)\sigma^{2})_{x}\\&=-\frac{c'(u)}{2c^{2}(u)}(R^{2}S+RS^{2})-\frac{c'(u)}{2c(u)}(\rho^{2}S+\sigma^{2}R).
\end{aligned}
\right.
\end{equation}
Combining \eqref{eq:evolEnergy} and \eqref{eq:evolEnergy2}, we finally obtain
\begin{equation}
\label{eq:conslaw}
\left\{
\begin{aligned}
	&\Big(R^{2}+c(u)\rho^{2}+S^{2}+c(u)\sigma^{2}\Big)_{t}\\
	&-\Big(c(u)\big(R^{2}+c(u)\rho^{2}-S^{2}-c(u)\sigma^{2}\big)\Big)_{x}=0,\\
	&\bigg(\frac{1}{c(u)}\big(R^{2}+c(u)\rho^{2}-S^{2}-c(u)\sigma^{2}\big)\bigg)_{t}\\
	&-\Big(R^{2}+c(u)\rho^{2}+S^{2}+c(u)\sigma^{2}\Big)_{x}=0.
\end{aligned}
\right.
\end{equation}
Let
\begin{equation*}
	v=R^{2}+c(u)\rho^{2}+S^{2}+c(u)\sigma^{2} \quad \text{and} \quad w=\frac{1}{c(u)}\big(R^{2}+c(u)\rho^{2}-S^{2}-c(u)\sigma^{2}\big),
\end{equation*}
then \eqref{eq:conslaw} rewrites as
\begin{equation}
\label{eq:conslawsys}
	\begin{pmatrix} v \\ w \end{pmatrix}_{t}-\begin{pmatrix} c^{2}(u)w \\
	v \end{pmatrix}_{x}=0,
\end{equation}
which is a system of conservation laws, see \cite[(2.4)-(2.6)]{BreZhe:06} for the NVW equation.

Conservation of $v$ and $w$ here means that
\begin{equation}
\label{eq:cons1}
	\frac{d}{dt}\int_{x_{1}}^{x_{2}}v(t,x)\,dx=(c^{2}(u)w)(t,x_{2})-(c^{2}(u)w)(t,x_{1})
\end{equation}
and
\begin{equation*}
	\frac{d}{dt}\int_{x_{1}}^{x_{2}}w(t,x)\,dx=v(t,x_{2})-v(t,x_{1}).
\end{equation*}
Note that we have, by \eqref{eq:smoothEnergy2}
\begin{equation*}
	E(t)=\frac{1}{4}\int_{\mathbb{R}}v(t,x)\,dx.
\end{equation*}
So far we assumed that $u$, $u_{t}$, $u_{x}$, $\rho$ and $\sigma$ are smooth and bounded functions that vanish at $\pm\infty$. Under these assumptions we get by letting $x_{1}\rightarrow-\infty$ and $x_{2}\rightarrow+\infty$ in \eqref{eq:cons1},
\begin{equation*}
	\frac{d}{dt}\int_{-\infty}^{\infty}v(t,x)\,dx=0
\end{equation*}
and we recover the condition $E'(t)=0$.

In the view of \eqref{eq:evolEnergy}, we interpret $R^{2}+c(u)\rho^{2}$ and $S^{2}+c(u)\sigma^{2}$ as the left and right traveling part of the energy density, respectively. Moreover, the right-hand sides of the two equations in \eqref{eq:evolEnergy} are equal with opposite sign, which means that the right and the left part can interact with each other. That is, energy can swap back and forth between the two parts, while the total energy remains unchanged because of \eqref{eq:conslaw}.

In contrast to the linear wave equation, solutions to \eqref{eq:nvw}, and hence also to \eqref{eq:nvwsys}, can develop singularities in finite time, even for smooth initial data, see e.g. \cite{GlaHunZh:96}. Here, a singularity means that either $u_{x}$ or $u_{t}$ becomes unbounded pointwise while $u$ remains continuous, and $u(t,\cdot)$,$u_{x}(t,\cdot)$, $u_t(t,\cdot)\in L^{2}(\mathbb{R})$ for all $t\geq 0$. This means that the energy densities $\frac{1}{4}(R^{2}+c(u)\rho^{2})$ and $\frac{1}{4}(S^{2}+c(u)\sigma^{2})$ may become unbounded pointwise. In other words, the energy density measures $\frac{1}{4}(R^{2}+c(u)\rho^{2})\,dx$ and $\frac{1}{4}(S^{2}+c(u)\sigma^{2})\,dx$ can have singular parts, meaning that energy concentrates on sets of measure zero. Thus if we want to obtain a semigroup of solutions of \eqref{eq:nvwsys} we must be able to deal with both singular initial data and singularities turning up at later times. Assume that we have a singularity at time $t=t_{0}$. A central question is: if we want to solve the equation for $t\geq t_{0}$, how do we prescribe initial data at $t=t_{0}$? By computing $\frac{1}{4}(R^{2}+c(u)\rho^{2})(t_{0},\cdot)$ and $\frac{1}{4}(S^{2}+c(u)\sigma^{2})(t_{0},\cdot)$ in $L^2(\mathbb{R})$, we cannot conclude whether or not energy has concentrated. On the other hand, the presence of singularities in the initial data greatly affects the analysis of the equation and this is information we need to have available at the initial time. The solution to this problem is to add to the initial data two positive Radon measures $\mu_{0}$ and $\nu_{0}$, such that the absolutely continuous parts equal the classical energy densities, i.e., $(\mu_{0})_{\text{ac}}=\frac{1}{4}(R_{0}^{2}+c(u_{0})\rho_{0}^{2})\,dx$ and $(\nu_{0})_{\text{ac}}=\frac{1}{4}(S_{0}^{2}+c(u_{0})\sigma_{0}^{2})\,dx$. The singular parts of the measures on the other hand contain information about the concentration of energy.  

Next, we illustrate the formation of a singularity with the following example. We consider a function $f(t,x)$, where $f^{2}$ should be thought of as either $\frac{1}{4}(R^{2}+c(u)\rho^{2})$ or $\frac{1}{4}(S^{2}+c(u)\sigma^{2})$.  
The function $f(t,\cdot)$ belongs to $L^{2}(\mathbb{R})$ for all $t\geq 0$. At $t=0$, $f$ is smooth and bounded for all $x$. At a later time $t=t_{0}>0$, $f$ becomes unbounded at the origin and $f^{2}(t,x)\,dx$ converges weak-star in the sense of measures to the Dirac delta at zero as $t\rightarrow t_{0}$.

Let $t_{0}>0$ and consider the function
\begin{equation*}
	f(t,x)=\Big(\frac{2}{\pi}\Big)^{\frac{1}{4}}\frac{1}{\sqrt{|t-t_{0}|}}e^{-(\frac{x}{t-t_{0}})^{2}},
\end{equation*}
where $t\geq 0$. We have $f(t,\cdot)\in L^{2}(\mathbb{R})$ for all $t\geq 0$ since
\begin{equation}
\label{eq:fL2norm}
	\int_{\mathbb{R}}f^{2}(t,x)\,dx=1.
\end{equation}
Note that $f(t,x)\rightarrow 0$ for $x\neq 0$ and $f(t,0)\rightarrow +\infty$ as $t\rightarrow t_{0}$. Moreover, direct calculations yield
\begin{equation*}
	\lim_{t\rightarrow t_{0}}\int_{\mathbb{R}}\phi(x)f(t,x)\,dx=0 \quad \text{and} \quad
	\lim_{t\rightarrow t_{0}}\int_{\mathbb{R}}\phi(x)f^{2}(t,x)\,dx=\phi(0)
\end{equation*}
for all $\phi\in C^{\infty}_{c}(\mathbb{R})$. In other words, $f(t,\cdot)\overset{*}{\rightharpoonup}0$ and $f^{2}(t,x)\,dx\overset{*}{\rightharpoonup}\delta_{0}$, where $\delta_{0}$ is the Dirac delta at zero. Also note from \eqref{eq:fL2norm} that $f(t,\cdot)$ does not converge to zero in $L^{2}(\mathbb{R})$, and since $f(t,x)\rightarrow 0$ almost everywhere it means that $f(t,\cdot)$ does not converge in $L^{2}(\mathbb{R})$. In fact we have
\begin{equation*}
	\lim_{t\rightarrow t_{0}}\int_{\mathbb{R}}f^{p}(t,x)\,dx=\begin{cases}
	0, & 1\leq p<2,\\
	1, & p=2,\\
	\infty, & 2<p<\infty,
	\end{cases}
\end{equation*} 
and since $f\geq 0$ this implies that $f(t,\cdot)\rightarrow 0$ in $L^{p}(\mathbb{R})$ for $1\leq p<2$.

\section{Equivalent System}
\label{sec:equivsys}

In this section we introduce a change of coordinates based on the method of characteristics. As a motivation for the approach we use for \eqref{eq:nvw} and \eqref{eq:nvwsys} we start out with the linear wave equation.

\subsection{The Linear Wave Equation}

Consider the linear wave equation
\begin{equation}
\label{eq:wave}
	u_{tt}-c^2u_{xx}=0,
\end{equation}
where $c$ is constant. We factorize the wave operator and can write the equation either as
\begin{equation}
\label{eq:waveSplit1}
	\Big[\frac{\partial}{\partial t}-c\frac{\partial}{\partial x}\Big]\Big[\frac{\partial}{\partial t}+c\frac{\partial}{\partial x}\Big]u=0
\end{equation}
or
\begin{equation}
\label{eq:waveSplit2}
	\Big[\frac{\partial}{\partial t}+c\frac{\partial}{\partial x}\Big]\Big[\frac{\partial}{\partial t}-c\frac{\partial}{\partial x}\Big]u=0.
\end{equation}
In both cases we find that the characteristics are given by $x\pm ct=\text{constant}$, and that every solution is of the form
\begin{equation}
\label{eq:waveSol}
	u(t,x)=F(x+ct)+G(x-ct),
\end{equation}
for some functions $F$ and $G$. In other words, the solution consists of a left and right traveling part. 

Let $X(t,x)=x+ct$ and $Y(t,x)=x-ct$. Then the functions $X$ and $Y$ satisfy
\begin{equation}
\label{eq:charEqWave}
	X_{t}-cX_{x}=0 \quad \text{and} \quad Y_{t}+cY_{x}=0.
\end{equation}  
Note that the operators acting on $X$ and $Y$ are the two factors of the wave operator. We consider the mapping from $(t,x)$-plane to $(X,Y)$-plane defined by the above equations. To make sure that the transformation is non-degenerate we compute
\begin{equation*}
	\det\bigg(\begin{bmatrix}
	X_{t} & X_{x}\\
	Y_{t} & Y_{x}
	\end{bmatrix}\bigg)=2c,
\end{equation*}   
which implies that we must have $c\neq 0$. We observe that the characteristics $x+ct=\text{constant}$ and $x-ct=\text{constant}$ are mapped to horizontal and vertical lines in the $(X,Y)$-plane, respectively. Let $U(X,Y)=u(t(X,Y),x(X,Y))$. We compute the derivatives of $u(t,x)=U(X(t,x),Y(t,x))$ and get 
\begin{subequations}
\label{eq:uDer}	
\begin{align}
	u_{t}&=U_{X}X_{t}+U_{Y}Y_{t},\\
	u_{tt}&=U_{XX}X_{t}^{2}+U_{X}X_{tt}+2U_{XY}X_{t}Y_{t}+U_{YY}Y_{t}^{2}+U_{Y}Y_{tt},\\
	u_{x}&=U_{X}X_{x}+U_{Y}Y_{x},\\
	u_{xx}&=U_{XX}X_{x}^{2}+U_{X}X_{xx}+2U_{XY}X_{x}Y_{x}+U_{YY}Y_{x}^{2}+U_{Y}Y_{xx}.
\end{align}
\end{subequations}
Inserting \eqref{eq:uDer} in \eqref{eq:wave} yields
\begin{align*}
	0&=u_{tt}-c^2u_{xx}\\
	&=U_{XX}(X_{t}^{2}-c^{2}X_{x}^{2})+U_{X}(X_{tt}-c^{2}X_{xx})+2U_{XY}(X_{t}Y_{t}-c^{2}X_{x}Y_{x})\\
	&\quad +U_{YY}(Y_{t}^{2}-c^{2}Y_{x}^{2})+U_{Y}(Y_{tt}-c^{2}Y_{xx})\\
	&=-4c^{2}X_{x}Y_{x}U_{XY},
\end{align*}
where we used \eqref{eq:charEqWave}. The functions $X_{x}$ and $Y_{x}$ are nonzero and finite, as $X_{x}=1$ and $Y_{x}=1$. Furthermore, since we assume $c\neq 0$, we get
\begin{equation}
\label{eq:wesys}
	U_{XY}=0.
\end{equation}
In particular, we have $U(X,Y)=F(X)+G(Y)$ and once again we obtain \eqref{eq:waveSol}.

\subsection{The Nonlinear Variational Wave Equation}\label{sec:nvw}

Now we turn to the nonlinear variational wave equation
\begin{equation*}
	u_{tt}-c(u)(c(u)u_{x})_{x}=u_{tt}-c(u)c'(u)u_{x}^{2}-c^{2}(u)u_{xx}=0.
\end{equation*}
We first note that a factorization of the operator like we did in \eqref{eq:waveSplit1} and \eqref{eq:waveSplit2} is not possible because of the function $c(u)$. Instead we look for a factorization of the terms containing the highest order derivatives. We compute
\begin{equation}
\label{eq:nvwSplit1}
	\Big[\frac{\partial}{\partial t}-c(u)\frac{\partial}{\partial x}\Big]\Big[\frac{\partial}{\partial t}+c(u)\frac{\partial}{\partial x}\Big]u=u_{tt}-c(u)(c(u)u_{x})_{x}+c'(u)u_{t}u_{x}=+c'(u)u_{t}u_{x}
\end{equation}
and
\begin{equation}
\label{eq:nvwSplit2}
	\Big[\frac{\partial}{\partial t}+c(u)\frac{\partial}{\partial x}\Big]\Big[\frac{\partial}{\partial t}-c(u)\frac{\partial}{\partial x}\Big]u=u_{tt}-c(u)(c(u)u_{x})_{x}-c'(u)u_{t}u_{x}=-c'(u)u_{t}u_{x}.
\end{equation}
Both these factorizations take care of the higher order derivatives, and we end up with a lower order term on the right-hand side. Note the difference in sign of this term depending on which operator is used first, showing that the operators do not commute. Therefore it is natural to consider both equations corresponding to the two factorizations. Note that \eqref{eq:nvwSplit1} and \eqref{eq:nvwSplit2} are the equations for $R=u_{t}+c(u)u_{x}$ and $S=u_{t}-c(u)u_{x}$ and we see once more that it is convenient to work with these functions. Note that $R$ and $S$ are the directional derivatives of $u$ in the directions $(1,c(u))$ and $(1,-c(u))$, respectively. From \eqref{eq:nvwSplit1} and \eqref{eq:nvwSplit2} we see that the directional derivative of $R$ in the direction $(1,-c(u))$ and the directional derivative of $S$ in the direction $(1,c(u))$ are equal with opposite sign.

In the following we assume that $u$ is sufficiently smooth and bounded.

We consider the characteristics corresponding to the highest order derivatives, i.e., the characteristics corresponding to the two factors $\frac{\partial}{\partial t}-c(u)\frac{\partial}{\partial x}$ and $\frac{\partial}{\partial t}+c(u)\frac{\partial}{\partial x}$ of the nonlinear variational wave operator. More specifically, we consider the functions $X(t,x)$ and $Y(t,x)$ satisfying
\begin{equation}
\label{eq:charXY}
	X_{t}-c(u)X_{x}=0 \quad \text{and} \quad Y_{t}+c(u)Y_{x}=0.
\end{equation}

First, we want to solve the equation for $X(t,x)$ with the method of characteristics. Let $t$ and $x$ be functions of parameters $s$ and $\xi$. We compute
\begin{equation}\label{eq:char1_2}
	\frac{d}{ds}X(t(s,\xi),x(s,\xi))=X_{t}t_{s}+X_{x}x_{s}
\end{equation} 
which is equal to zero if
\begin{equation}
\label{eq:char1_1}
	t_{s}(s,\xi)=1 \quad \text{and} \quad x_{s}(s,\xi)=-c(u(t(s,\xi),x(s,\xi))).
\end{equation}
We assume that $t(0,\xi)=0$ and $x(0,\xi)=\xi$ for all $\xi\in\mathbb{R}$. Then we get
\begin{equation}
\label{eq:char1}
	t(s,\xi)=s \quad \text{and} \quad x_{s}(s,\xi)=-c(u(s,x(s,\xi))).
\end{equation}
We integrate the last equation in \eqref{eq:char1} and get
\begin{equation}
\label{eq:char3}
	x(s,\xi)=\xi-\int_{0}^{s}c(u(r,x(r,\xi)))\,dr.
\end{equation}
Recalling assumption \eqref{eq:cassumption}, we get $-\kappa\leq x_{s}(s,\xi)\leq -\frac{1}{\kappa}$ and
\begin{equation*}
	\xi-\kappa s\leq x(s,\xi)\leq \xi-\frac{1}{\kappa}s
\end{equation*}
for all $s\geq 0$ and $\xi\in\mathbb{R}$. 

For fixed $\xi$ we consider the differential equation in \eqref{eq:char1}. Assuming that $u_{x}(t,\cdot)\in L^\infty(\mathbb{R})$, the right-hand side is Lipschitz continuous with respect to the $x$-argument, i.e.,
\begin{equation*}
	|c(u(s,x_{2}))-c(u(s,x_{1}))|\leq k_{1}\sup_{t\geq 0}||u_{x}(t,\cdot)||_{L^{\infty}(\mathbb{R})}|x_{2}-x_{1}|
\end{equation*}
for all $x_{1}$, $x_{2}\in\mathbb{R}$. Thus, there exists a unique local solution $x(\cdot,\xi)$ with initial data $x(0,\xi)=\xi$. This means that only one characteristic starts from the point given by $t=0$ and $x=\xi$.

Now we can have three scenarios: 

If $0<x_{\xi}(s,\xi)<\infty$ for all $s<s_0$, then the solution $X(s,x(s,\xi))$ is well-defined and there is at least a chance of continuing $X(s,x(s,\xi))$ for $s\geq s_{0}$.

If $x_{\xi}(s_{0},\xi)=0$ at some point $(s_{0}, \xi)$ with $s_0>0$ , characteristics starting from different values of $\xi$ may intersect at $s=s_{0}$ and it is not clear that $X(s,x(s,\xi))$ is well-defined for $s>s_0$. 

If $x_{\xi}(s,\xi)\rightarrow \infty$ as $s$ tends to some point $s_{0}>0$ and some $\xi$, then the solution $x(s,\xi)$ is not defined for $s\geq s_{0}$.

Differentiating the last equation in \eqref{eq:char1} with respect to $\xi$ gives us
\begin{equation*}
	x_{s\xi}(s,\xi)=-c'(u(s,x(s,\xi)))u_{x}(s,x(s,\xi))x_{\xi}(s,\xi),
\end{equation*}
which we integrate to get
\begin{equation}
\label{eq:xXi}
	x_{\xi}(s,\xi)=\exp\bigg\{-\int_{0}^{s}c'(u(r,x(r,\xi)))u_{x}(r,x(r,\xi))\,dr\bigg\}.
\end{equation}
Since $u_{x}$ is bounded we have $0<x_{\xi}(s,\xi)<\infty$ for all $0\leq s<\infty$ and all $\xi$. This is because
\begin{equation*}
	\exp\Big\{-k_{1}s\sup_{t\in[0,s]}||u_{x}(t,\cdot)||_{L^{\infty}(\mathbb{R})}\Big\}\leq 
	x_{\xi}(s,\xi)\leq \exp\Big\{k_{1}s\sup_{t\in[0,s]}||u_{x}(t,\cdot)||_{L^{\infty}(\mathbb{R})}\Big\}.
\end{equation*}
Thus, in the smooth case we do not end up with the two challenging scenarios described above. 

We compute the determinant of the Jacobian corresponding to the map $(s,\xi)\rightarrow (t,x)$ and get 
\begin{equation*}
	\det\bigg(\begin{bmatrix}
	t_{s} & t_{\xi}\\
	x_{s} & x_{\xi}
	\end{bmatrix}\bigg)=t_{s}x_{\xi}-t_{\xi}x_{s}=x_{\xi}.
\end{equation*}
Since $0<x_{\xi}(s,\xi)<\infty$ we have from the inverse function theorem that the Jacobian corresponding to the map $(t,x)\rightarrow (s,\xi)$ satisfies
\begin{equation*}
	\begin{bmatrix}
	s_{t} & s_{x}\\
	\xi_{t} & \xi_{x}
	\end{bmatrix}=\frac{1}{x_{\xi}}\begin{bmatrix}
	x_{\xi} & -t_{\xi}\\
	-x_{s} & t_{s}
	\end{bmatrix}.
\end{equation*}
From \eqref{eq:char1} we get
\begin{equation}
\label{eq:sxiId}
	s_{t}=1, \quad s_{x}=0, \quad \xi_{t}=-\frac{x_{s}}{x_{\xi}}, \quad \xi_{x}=\frac{1}{x_{\xi}},
\end{equation}
so that
\begin{equation*}
	s(t,x)=t
\end{equation*}
and
\begin{equation*}
	\xi_{t}(t,x)=-x_{s}(t,\xi(t,x))\xi_{x}(t,x)=c(u(t,\xi(t,x)))\xi_{x}(t,x).
\end{equation*} 

Furthermore, \eqref{eq:charXY}--\eqref{eq:char1_1} imply that
\begin{equation*}
	X(t(s,\xi),x(s,\xi))=X(0,\xi)=g(\xi),
\end{equation*}
for some strictly increasing function $g\in C^1(\mathbb{R})$.
Differentiation, combined with \eqref{eq:char1_1} and \eqref{eq:sxiId} yields
\begin{equation}
\label{eq:XtXx}
	X_{t}=g'(\xi)\xi_{t}=-g'(\xi)\frac{x_{s}}{x_{\xi}} \quad \text{and} \quad X_{x}=g'(\xi)\xi_{x}=g'(\xi)\frac{1}{x_{\xi}}, 
\end{equation}
which implies $0<X_{t}<\infty$ and $0<X_{x}<\infty$.

Next, we study $Y(t,x)$ with the method of characteristics. We obtain
\begin{equation*}
	\frac{d}{ds}Y(t(s,\xi),x(s,\xi))=0
\end{equation*} 
with the characteristics given by
\begin{equation}
\label{eq:forwardchar}
	t_{s}(s,\xi)=1 \quad \text{and} \quad x_{s}(s,\xi)=c(u(t(s,\xi),x(s,\xi))).
\end{equation}
Assuming that $t(0,\xi)=0$ and $x(0,\xi)=\xi$ for all $\xi\in\mathbb{R}$ we get
\begin{equation}
\label{eq:forwardchar2}
	t(s,\xi)=s \quad \text{and} \quad x_{s}(s,\xi)=c(u(s,x(s,\xi))).
\end{equation}
If $Y(0,\xi)=h(\xi)$ for some strictly increasing function $h\in C^1(\mathbb{R})$, then
\begin{equation*}
	Y(s,x(s,\xi))=h(\xi).
\end{equation*}
As in the computations above we find
\begin{equation}
\label{eq:xXi2}
	x_{\xi}(s,\xi)=\exp\bigg\{\int_{0}^{s}c'(u(r,x(r,\xi)))u_{x}(r,x(r,\xi))\,dr\bigg\},
\end{equation}
and since $u_{x}$ is bounded, $0<x_{\xi}(s,\xi)<\infty$ for all $0\leq s<\infty$ and all $\xi$. We also find that \eqref{eq:sxiId} holds with $x_{\xi}$ as defined in \eqref{eq:xXi2}, and
\begin{equation}
\label{eq:YtYx}
	Y_{t}=h'(\xi)\xi_{t}=-h'(\xi)\frac{x_{s}}{x_{\xi}} \quad \text{and} \quad Y_{x}=h'(\xi)\xi_{x}=h'(\xi)\frac{1}{x_{\xi}},
\end{equation}
so that $-\infty<Y_{t}<0$ and $0<Y_{x}<\infty$. 

Figure \ref{fig:FigLinChar} and \ref{fig:FigNonlinChar} show the characteristics for the linear wave equation and the NVW equation, respectively.

\begin{figure}
	\centerline{\hbox{\includegraphics[width=10cm]{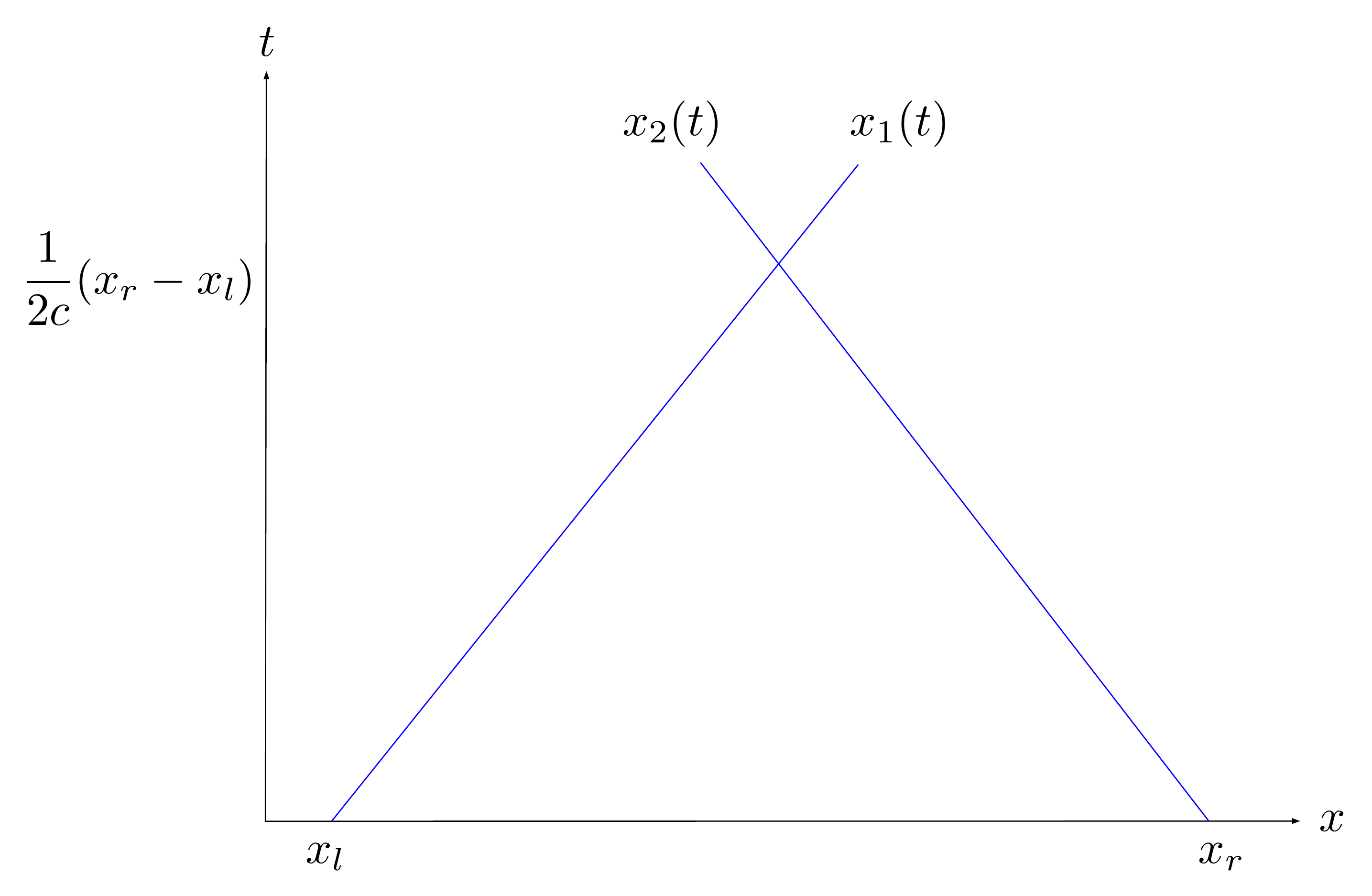}}}
	\caption{Characteristics of the linear wave equation, i.e., $c$ is constant. The forward characteristic $x_{1}(t)=x_{l}+ct$ starting from $x_{l}$, and the backward characteristic $x_{2}(t)=x_{r}-ct$ starting from $x_{r}$, intersect at $t=\frac{1}{2c}(x_{r}-x_{l})$.}
	\label{fig:FigLinChar}
\end{figure}

\begin{figure}
	\centerline{\hbox{\includegraphics[width=10cm]{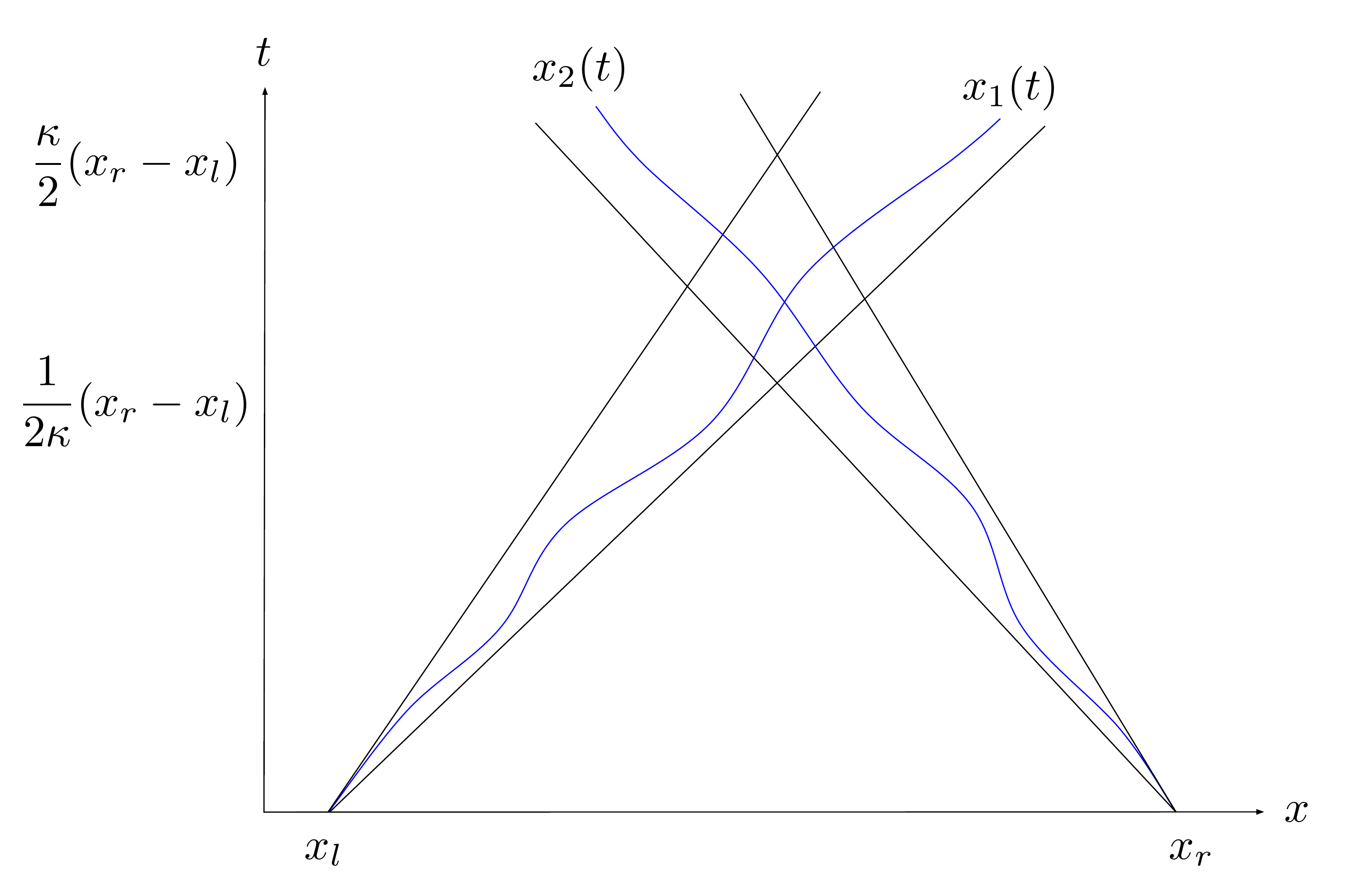}}}
	\caption{Characteristics of the NVW equation. The forward characteristic $x_{1}(t)$ starting from $x_{l}$ is given by $x_{1,t}(t)=c(u(t,x_{1}(t)))$, $x_{1}(0)=x_{l}$, and the backward characteristic $x_{2}(t)$ starting from $x_{r}$ is given by $x_{2,t}(t)=-c(u(t,x_{2}(t)))$, $x_{2}(0)=x_{r}$. Because of \eqref{eq:cassumption}, they intersect at a time $t$ such that $\frac{1}{2\kappa}(x_{r}-x_{l})\leq t\leq \frac{\kappa}{2}(x_{r}-x_{l})$.}
	\label{fig:FigNonlinChar}
\end{figure}

Now we consider the mapping from the $(t,x)$-plane to the $(X,Y)$-plane. The determinant of the Jacobian of this map reads
\begin{equation}
\label{eq:txtoXY}
    d=\det\bigg(\begin{bmatrix}
	X_{t} & X_{x}\\
	Y_{t} & Y_{x}
\end{bmatrix}\bigg)=X_{t}Y_{x}-X_{x}Y_{t}=2c(u)X_{x}Y_{x}=-\frac{2X_{t}Y_{t}}{c(u)}
\end{equation}
and once again we see that we must assume that $c(u)$ is strictly positive and finite. Since $0<X_{x}<\infty$ and $0<Y_{x}<\infty$, we have $0<d<\infty$. The inverse function theorem then implies that the Jacobian corresponding to the map $(X,Y)\rightarrow (t,x)$ satisfies
\begin{equation*}
	\begin{bmatrix}
	t_{X} & t_{Y}\\
	x_{X} & x_{Y}
	\end{bmatrix}=\frac{1}{d}\begin{bmatrix}
	Y_{x} & -X_{x}\\
	-Y_{t} & X_{t}
	\end{bmatrix}.
\end{equation*}
From the above equality many identities can be read off, and we only mention some of them. By using \eqref{eq:charXY} and \eqref{eq:txtoXY}, we obtain
\begin{equation}
\label{eq:NiceId}
	2c(u)t_{X}X_{x}=1, \quad -2c(u)t_{Y}Y_{x}=1, \quad 2x_{X}X_{x}=1, \quad 2x_{Y}Y_{x}=1,
\end{equation}
which imply
\begin{equation}
\label{eq:xXxY}
	x_{X}=c(u)t_{X} \quad \text{and} \quad x_{Y}=-c(u)t_{Y}.
\end{equation}
We observe from \eqref{eq:NiceId} that $t_{X}$, $t_{Y}$, $x_{X}$ and $x_{Y}$ are nonzero and finite.

Let $U(X,Y)=u(t(X,Y),x(X,Y))$. We insert the derivatives of $u(t,x)=\\U(X(t,x),Y(t,x))$ from \eqref{eq:uDer} into \eqref{eq:nvw} and get
\begin{aalign}
\label{eq:nvw2}
	0&=u_{tt}-c(u)(c(u)u_{x})_{x}\\
	&=U_{XX}(X_{t}^{2}-c^{2}(u)X_{x}^{2})+U_{X}(X_{tt}-c^{2}(u)X_{xx})\\
	&\quad+2U_{XY}(X_{t}Y_{t}-c^{2}(u)X_{x}Y_{x})\\
	&\quad +U_{YY}(Y_{t}^{2}-c^{2}(u)Y_{x}^{2})+U_{Y}(Y_{tt}-c^{2}(u)Y_{xx})\\
	&\quad-c(u)c'(u)(U_{X}^{2}X_{x}^{2}+2U_{X}U_{Y}X_{x}Y_{x}+U_{Y}^{2}Y_{x}^{2}).
\end{aalign}
Due to \eqref{eq:charXY} all second order derivatives of $U$ drop out except for the term containing the mixed derivative $U_{XY}$. We compute the remaining terms. From \eqref{eq:uDer} and \eqref{eq:charXY} we have
\begin{equation*}
	R=u_{t}+c(u)u_{x}=U_{X}(X_{t}+c(u)X_{x})+U_{Y}(Y_{t}+c(u)Y_{x})=2c(u)U_{X}X_{x}
\end{equation*}
and
\begin{equation*}
	S=u_{t}-c(u)u_{x}=U_{X}(X_{t}-c(u)X_{x})+U_{Y}(Y_{t}-c(u)Y_{x})=-2c(u)U_{Y}Y_{x},
\end{equation*}
and after using \eqref{eq:NiceId} we get
\begin{equation}
\label{eq:Rid}
	R=c(u)\frac{U_{X}}{x_{X}}
\end{equation}
and
\begin{equation}
\label{eq:Sid}
	S=-c(u)\frac{U_{Y}}{x_{Y}}.
\end{equation}
By differentiating \eqref{eq:charXY} and using \eqref{eq:NiceId} we obtain
\begin{equation*}
	X_{tt}-c^{2}(u)X_{xx}=c'(u)X_{x}R=\frac{c'(u)R}{2x_{X}} \quad \text{and} \quad Y_{tt}-c^{2}(u)Y_{xx}=-c'(u)Y_{x}S=-\frac{c'(u)S}{2x_{Y}}.
\end{equation*}
From \eqref{eq:charXY} and \eqref{eq:NiceId} we have
\begin{equation*}
	X_{t}Y_{t}-c^{2}(u)X_{x}Y_{x}=-2c^{2}(u)X_{x}Y_{x}=-\frac{c^{2}(u)}{2x_{X}x_{Y}}.
\end{equation*}
Thus  \eqref{eq:nvw2} is equivalent to
\begin{equation}
\label{eq:UXY3}
	U_{XY}=\frac{c'(u)}{4c^{3}(u)}(R^{2}x_{X}x_{Y}+S^{2}x_{Y}x_{X})-\frac{c'(u)}{2c(u)}U_{X}U_{Y}.
\end{equation}
Let
\begin{equation}
\label{eq:JDer}
	J_{X}=\frac{1}{2}R^{2}x_{X} \quad \text{and} \quad J_{Y}=\frac{1}{2}S^{2}x_{Y},
\end{equation} 
which we think of as the left and right traveling part of the energy density in the new variables, respectively. Now \eqref{eq:UXY3} yields
\begin{equation*}
	U_{XY}=\frac{c'(u)}{2c^{3}(U)}(J_{X}x_{Y}+J_{Y}x_{X})-\frac{c'(u)}{2c(u)}U_{X}U_{Y}.
\end{equation*}
Using \eqref{eq:Rid}, \eqref{eq:Sid} and \eqref{eq:JDer} we get
\begin{equation}
\label{eq:xJURel}
	2x_{X}J_{X}=c^{2}(U)U_{X}^{2} \quad \text{and} \quad 2x_{Y}J_{Y}=c^{2}(U)U_{Y}^{2}. 
\end{equation}

We find it convenient to introduce the function $K$ defined by
\begin{equation}
\label{eq:KDer}
	K_{X}=\frac{1}{2c(u)}R^{2}x_{X} \quad \text{and} \quad K_{Y}=-\frac{1}{2c(u)}S^{2}x_{Y},
\end{equation}
which satisfies
\begin{equation}
\label{eq:JXKX}
	J_{X}=c(U)K_{X} \quad \text{and} \quad J_{Y}=-c(U)K_{Y}.
\end{equation}
In view of \eqref{eq:evolEnergy2} and \eqref{eq:conslaw} (with $\rho=\sigma=0$) we can think of $K_{X}$ and $K_{Y}$ as the left and right traveling part of the second conserved quantity $\frac{1}{c(u)}(R^{2}-S^{2})$ in the new coordinates, respectively. 

Next, let us derive the equations for $t_{XY}$, $x_{XY}$, $J_{XY}$ and $K_{XY}$.
We have $x_{XY}=x_{YX}$, which by using \eqref{eq:xXxY} is the same as $(c(U)t_{X})_{Y}=(-c(U)t_{Y})_{X}$. This leads to  
\begin{equation*}
	t_{XY}=-\frac{c'(U)}{2c(U)}(U_{Y}t_{X}+U_{X}t_{Y}).
\end{equation*}
We find the equation for $x_{XY}$ by using \eqref{eq:xXxY} in $t_{XY}=t_{YX}$, which yields $(\frac{x_{X}}{c(U)})_{Y}=(-\frac{x_{Y}}{c(U)})_{X}$ and finally
\begin{equation*}
	x_{XY}=\frac{c'(U)}{2c(U)}(U_{Y}x_{X}+U_{X}x_{Y}).
\end{equation*}
Using $J_{XY}=J_{YX}$, $K_{XY}=K_{YX}$ and \eqref{eq:JXKX} we get
\begin{equation*}
	J_{XY}=\frac{c'(U)}{2c(U)}(U_{Y}J_{X}+U_{X}J_{Y})
\end{equation*}
and
\begin{equation*}
	K_{XY}=-\frac{c'(U)}{2c(U)}(U_{Y}K_{X}+U_{X}K_{Y}).
\end{equation*}
Finally we end up with the following system of differential equations
\begin{subequations}\label{eq:nvwgoveq}
\begin{align}
	t_{XY}&=-\frac{c'(U)}{2c(U)}(U_{Y}t_{X}+U_{X}t_{Y}),\\
	x_{XY}&=\frac{c'(U)}{2c(U)}(U_{Y}x_{X}+U_{X}x_{Y}),\\
	U_{XY}&=\frac{c'(u)}{2c^{3}(U)}(J_{X}x_{Y}+J_{Y}x_{X})-\frac{c'(u)}{2c(u)}U_{X}U_{Y},\\
	J_{XY}&=\frac{c'(U)}{2c(U)}(U_{Y}J_{X}+U_{X}J_{Y}),\\
	K_{XY}&=-\frac{c'(U)}{2c(U)}(U_{Y}K_{X}+U_{X}K_{Y}).
\end{align}
\end{subequations}

\subsection{The Regularized System}

We derive a set of equations corresponding to \eqref{eq:nvwsys} in the new variables. We consider characteristics $X(t,x)$ and $Y(t,x)$ given by \eqref{eq:charXY}. We assume that $u$, $R$, $S$, $\rho$, and $\sigma$ are smooth and bounded. As above we get that $t_{X}$, $t_{Y}$, $x_{X}$, and $x_{Y}$ are nonzero and finite. 

Denote $u(t,x)=U(X(t,x),Y(t,x))$. By calculations like those that led to \eqref{eq:UXY3} we find
\begin{equation*}
	U_{XY}=\frac{c'(u)}{4c^{3}(u)}\Big((R^{2}+c(u)\rho^{2})x_{X}x_{Y}+(S^{2}+c(u)\sigma^{2})x_{Y}x_{X}\Big)-\frac{c'(u)}{2c(u)}U_{X}U_{Y}.
\end{equation*}
We introduce
\begin{equation}
\label{eq:JDer2}
	J_{X}=\frac{1}{2}(R^{2}+c(u)\rho^{2})x_{X} \quad \text{and} \quad J_{Y}=\frac{1}{2}(S^{2}+c(u)\sigma^{2})x_{Y}.
\end{equation}
Using \eqref{eq:Rid} and \eqref{eq:Sid} we obtain
\begin{equation*}
	U_{XY}=\frac{c'(u)}{2c^{3}(U)}(J_{X}x_{Y}+J_{Y}x_{X})-\frac{c'(u)}{2c(u)}U_{X}U_{Y}.
\end{equation*}
Accordingly, we define
\begin{equation}
\label{eq:KDer2}
	K_{X}=\frac{1}{2c(u)}(R^{2}+c(u)\rho^{2})x_{X} \quad \text{and} \quad K_{Y}=-\frac{1}{2c(u)}(S^{2}+c(u)\sigma^{2})x_{Y}.
\end{equation}

The derivation of the system of equations is similar to the one in Section~\ref{sec:nvw} for the NVW equation \eqref{eq:nvw}.  In fact, we obtain the same equations as in \eqref{eq:nvwgoveq}. In addition we get two equations corresponding to \eqref{eq:nvwsys2} and \eqref{eq:nvwsys3}. Let $\rho(t,x)=P(X(t,x),Y(t,x))$. By \eqref{eq:nvwsys2} we get
\begin{align*}
	0&=\rho_{t}-(c(u)\rho)_{x}\\
	&=P_{X}(X_{t}-c(u)X_{x})+P_{Y}(Y_{t}-c(u)Y_{x})-c'(u)P(U_{X}X_{x}+U_{Y}Y_{x}).
\end{align*}
From \eqref{eq:charXY} and \eqref{eq:NiceId} we have
\begin{equation*}
	c(U)P_{Y}x_{X}+\frac{c'(U)}{2}P(U_{X}x_{Y}+U_{Y}x_{X})=0
\end{equation*}
and from \eqref{eq:nvwgoveq} we see that this is the same as
\begin{equation*}
	P_{Y}x_{X}+Px_{XY}=(Px_{X})_{Y}=0.
\end{equation*}
We define $p=Px_{X}$, so that 
\begin{equation*}
	p_{Y}=0.
\end{equation*}
Let $\sigma(t,x)=Q(X(t,x),Y(t,x))$. From \eqref{eq:nvwsys3} we have 
\begin{align*}
	0&=\sigma_{t}+(c(u)\sigma)_{x}\\
	&=Q_{X}(X_{t}+c(u)X_{x})+Q_{Y}(Y_{t}+c(u)Y_{x})+c'(u)Q(U_{X}X_{x}+U_{Y}Y_{x}).
\end{align*}
Using \eqref{eq:charXY} and \eqref{eq:NiceId} we get
\begin{equation*}
	c(U)Q_{X}x_{Y}+\frac{c'(U)}{2}Q(U_{X}x_{Y}+U_{Y}x_{X})=0
\end{equation*}
 and by  \eqref{eq:nvwgoveq} we find
\begin{equation*}
	Q_{X}x_{Y}+Qx_{XY}=(Qx_{Y})_{X}=0.
\end{equation*}
We define $q=Qx_{Y}$, so that
\begin{equation*}
	q_{X}=0.
\end{equation*}

By \eqref{eq:Rid}, \eqref{eq:Sid} and \eqref{eq:JDer2} we get
\begin{equation}
\label{eq:xJUpqRel}
	2x_{X}J_{X}=c^{2}(U)U_{X}^{2}+c(U)p^{2} \quad \text{and} \quad 2x_{Y}J_{Y}=c^{2}(U)U_{Y}^{2}+c(U)q^{2}. 
\end{equation}
Furthermore, we note that the relations 
\begin{subequations}\label{eq:xtJKRel}
\begin{align}
	\label{eq:xtRel}
	x_{X}&=c(U)t_{X}, & x_{Y}&=-c(U)t_{Y},\\
	\label{eq:JKRel}
	J_{X}&=c(U)K_{X}, & J_{Y}&=-c(U)K_{Y}
\end{align}
\end{subequations}
hold. To summarize, we obtain the following system of equations
\begin{subequations}
\label{eq:goveq}
\begin{align}
	\label{eq:goveqt}
    \dXY{t}&=-\frac{c'(U)}{2c(U)}(U_{Y}t_{X}+U_{X}t_{Y}),\\
    \label{eq:goveqx}
    \dXY{x}&=\frac{c'(U)}{2c(U)}(U_{Y}x_{X}+U_{X}x_{Y}),\\
    \label{eq:goveqU}
    \dXY{U}&=\frac{c'(U)}{2c^3(U)}(x_{Y}J_{X}+x_{X}J_{Y})-\frac{c'(U)}{2c(U)}U_{X}U_{Y},\\
    \label{eq:goveqJ}
    \dXY{J}&=\frac{c'(U)}{2c(U)}(U_{Y}J_{X}+U_{X}J_{Y}),\\
    \label{eq:goveqK}
    \dXY{K}&=-\frac{c'(U)}{2c(U)}(U_{Y}K_{X}+U_{X}K_{Y}),\\
    \label{eq:goveqp}
    p_{Y}&=0,\\
    \label{eq:goveqq}
    q_{X}&=0.
\end{align}
\end{subequations}
We introduce the vector $Z=(t,x,U,J,K)$. The
system \eqref{eq:goveqt}-\eqref{eq:goveqK} then rewrites  as
\begin{equation}
\label{eq:condgoveq}
  Z_{XY}=F(Z)(Z_X,Z_Y),
\end{equation}
where $F(Z)$ is a bilinear and symmetric tensor from $\Real^5\times\Real^5$ to $\Real^5$. Due to the relations \eqref{eq:xtRel}, either one of the equations in \eqref{eq:goveqt} and
\eqref{eq:goveqx} is redundant: one could remove one of them, and the system would remain well-posed, and one retrieves $t$ or $x$ by using \eqref{eq:xtRel}. Similarly, either one of the equations \eqref{eq:goveqJ} and \eqref{eq:goveqK} becomes redundant by \eqref{eq:JKRel}. However,
we find it convenient to work with the complete set of variables, i.e., $Z=(t,x,U,J,K)$. 

We observe that the equations \eqref{eq:goveqt}-\eqref{eq:goveqK} are not coupled with the last two \eqref{eq:goveqp}-\eqref{eq:goveqq}. Furthermore they are identical to the set of equations \eqref{eq:nvwgoveq} found for the NVW equation. However, we see that \eqref{eq:xJURel} is different from \eqref{eq:xJUpqRel}, so the solutions will not be identical. Also, from \eqref{eq:xJUpqRel} we see that the solutions $t,x,U,J,K$ of the five first equations are not independent of the solutions $p,q$ of the last two equations.

To prove the existence of solutions of \eqref{eq:goveq} we use a fixed point argument which is similar to the one found in \cite{HolRay:11}. In order to do so we need a curve $(\X(s),\Y(s))$ parametrized by $s\in\mathbb{R}$ in the $(X,Y)$-plane that corresponds to the initial time, i.e., it consists of all points $(X,Y)\in\mathbb{R}^{2}$ such that $t(X,Y)=0$. We will admit curves of the following type.  

\begin{definition}\label{def:curves}
	We denote by $\C$ the set of curves in the plane $\mathbb{R}^2$ parametrized by $(\X(s),\Y(s))$ with $s \in \mathbb{R}$, such that
	\begin{subequations}
		\label{eq:initialcurve}
		\begin{align}
		\label{eq:initialcurvereg}
		&\X-\id, \ \Y-\id \in \Winf(\Real), \\
		&\dot{\X}\geq 0, \ \dot{\Y}\geq 0 
		\end{align}	
		with the normalization
		\begin{equation}
		\label{eq:initialcurvenormalization}
		\frac{1}{2}(\X(s)+\Y(s))=s \quad \text{for all } s \in \Real.
		\end{equation}
		We set
		\begin{equation}
		\norm{(\X,\Y)}_{\C}=\norm{\X-\id}_{\Linf(\mathbb{R})}+\norm{\Y-\id}_{\Linf(\mathbb{R})}.
		\end{equation}
	\end{subequations}
\end{definition}

In the above derivation where we assumed that the solutions are smooth and bounded we found that $0<t_{X}<\infty$ and $-\infty<t_{Y}<0$. This implies that both $\X(s)$ and $\Y(s)$ are strictly increasing functions. Indeed, by differentiating $t(\X(s),\Y(s))=0$ and using \eqref{eq:initialcurvenormalization} we get
\begin{equation*}
	\dot{\X}=-\frac{2t_{Y}}{t_{X}-t_{Y}} \quad \text{and} \quad \dot{\Y}=\frac{2t_{X}}{t_{X}-t_{Y}},
\end{equation*}
which implies $\dot{\X}>0$ and $\dot{\Y}>0$. Thus, in this case $(\X(s),\Y(s))$ is a strictly monotone curve.

For general initial data the set 
\begin{equation*}
\Gamma_{0}=\{(X,Y)\in \mathbb{R}^2\mid t(X,Y)=0\}
\end{equation*}
will be the union of strictly monotone curves, horizontal and vertical lines, and boxes. We define this set implicitly in Definition \ref{def:mapfromDtoF} and \ref{def:mapfromFtoG0}, and the examples following the definitions show how the set $\Gamma_{0}$ depends on the initial data $(u_{0},R_{0},S_{0},\rho_{0},\sigma_{0},\mu_{0},\nu_{0})$.

The idea is the following. The backward characteristics transports the energy described by the measure $\mu_{0}$, while the forward characteristics transports the energy described by the measure $\nu_{0}$. 

At points $(0,x_{0})$ where the initial data is smooth and bounded and the measures are absolutely continuous, there is a finite amount of energy, and there is exactly one forward and one backward characteristic starting from $(0,x_{0})$. This point is mapped to one point in Lagrangian coordinates. An interval of such points yields a strictly monotone curve in Lagrangian coordinates, like we showed above.

At points $(0,x_{0})$ where only one of the measures is singular, say $\mu_{0}$, there is a finite amount of forward energy and an infinite amount of backward energy. Thus there are
infinitely many backward characteristics, but only one forward characteristic starting from $(0,x_0)$. If we think of characteristics as particles, then the infinite amount of backward energy at $(0,x_0)$ is distributed over infinitely many particles. 
To label these particles, we map this point to a horizontal line in the $(X,Y)$-plane.

If only $\nu_{0}$ is singular at $(0,x_{0})$, the point is mapped to a vertical line in the $(X,Y)$-plane. 

At points $(0,x_{0})$ where both measures are singular, there is an infinite amount of both forward and backward energy. Thus there are infinitely many forward and backward characteristics starting from $(0,x_{0})$. If we think of characteristics as particles, then the infinite amount of both forward and backward energy at $(0,x_0)$ is distributed over infinitely many particles and we need a rectangular box in the $(X,Y)$-plane to label all these particles. 

From the set $\Gamma_{0}$ we have to choose a unique curve $(\X,\Y)\in\C$. In the case of a box there are in principle infinitely many possible ways of doing this. We define the curve in Definition \ref{def:mapfromFtoG0}, where we in the case of a box roughly speaking define it to be the union of the left vertical side and the upper horizontal side of the box.    

Having defined the curve $(\X,\Y)$, we have to assign the values of $Z$, $Z_{X}$, $Z_{Y}$, $p$ and $q$ on it in order to solve \eqref{eq:goveq}. In addition we require that several properties derived in this section for the smooth case hold on the curve, see Definition \ref{def:initialDataH}. Later we will prove that solutions of \eqref{eq:goveq} satisfy the same properties. In Section \ref{sec:3} we explain how to define the functions on the curve for general initial data $u_{0},R_{0},S_{0},\rho_{0},\sigma_{0}\in L^{2}(\mathbb{R})$ and measures $\mu_{0}$ and $\nu_{0}$. Here we present how to proceed for initial data such that the functions $u_{0},R_{0},S_{0},\rho_{0},\sigma_{0}$ are smooth and bounded, and such that the measures $\mu_{0}$ and $\nu_{0}$ are absolutely continuous. We have to specify the values of 19 functions. Let us see how many equations we have available to determine these values. Let
\begin{equation}
\label{eq:initialtime}
	t(\X(s),\Y(s))=0,
\end{equation}
and
\begin{equation}
\label{eq:Uinitialtime}
	U(\X(s),\Y(s))=u_{0}(x(\X(s),\Y(s))).
\end{equation}
From \eqref{eq:Rid}, \eqref{eq:Sid} and \eqref{eq:JDer2} we have
\begin{align}	
	\label{eq:initialU}
	U_{X}(\X,\Y)&=x_{X}(\X,\Y)\frac{R_{0}(x(\X,\Y))}{c(u_{0}(x(\X,\Y)))},\\ U_{Y}(\X,\Y)&=-x_{Y}(\X,\Y)\frac{S_{0}(x(\X,\Y))}{c(u_{0}(x(\X,\Y)))},\\
	\label{eq:initialJX}
	J_{X}(\X,\Y)&=\frac{1}{2}x_{X}(\X,\Y)(R_{0}^{2}+c(u_{0})\rho_{0}^{2})(x(\X,\Y)),\\
	\label{eq:initialJY}
	J_{Y}(\X,\Y)&=\frac{1}{2}x_{Y}(\X,\Y)(S_{0}^{2}+c(u_{0})\sigma_{0}^{2})(x(\X,\Y)),\\
	\label{eq:initialPandQ}
	p(\X,\Y)&=x_{X}(\X,\Y)\rho_{0}(x(\X,\Y)),\\
	\label{eq:initialq}
	q(\X,\Y)&=x_{Y}(\X,\Y)\sigma_{0}(x(\X,\Y))
\end{align}	
and from \eqref{eq:xtRel} and \eqref{eq:JKRel} we have the relations
\begin{align}
\label{eq:initialxXi}
	x_{X}(\X,\Y)&=c(u_{0}(x(\X,\Y)))t_{X}(\X,\Y),\\
\label{eq:initialxXii}	
	x_{Y}(\X,\Y)&=-c(u_{0}(x(\X,\Y)))t_{Y}(\X,\Y),\\
\label{eq:initialJXJYi}
	J_{X}(\X,\Y)&=c(u_{0}(x(\X,\Y)))K_{X}(\X,\Y),\\
\label{eq:initialJXJYii}	
	J_{Y}(\X,\Y)&=-c(u_{0}(x(\X,\Y)))K_{Y}(\X,\Y).
\end{align}
We also want to use the fact that $Z_{X}$ and $Z_{Y}$ are derivatives of $Z$ to assign values of $t$, $ x$, $U$, $J$, $K$, and we can write this condition as	
\begin{equation}
\label{eq:compatibility}
	\dot{Z}(\X(s),\Y(s))=Z_{X}(\X(s),\Y(s))\dot{\X}(s)+Z_{Y}(\X(s),\Y(s))\dot{\Y}(s),
\end{equation}	
where the notation means $\dot{Z}(\X(s),\Y(s))=\frac{d}{ds}Z(\X(s),\Y(s))$. We have 19 unknowns $(\X,\Y,p,q,Z,Z_{X},Z_{Y})$ and 17 equations, given by \eqref{eq:initialcurvenormalization} and \eqref{eq:initialtime}-\eqref{eq:compatibility}. We use the two remaining degrees of freedom to obtain $Z_{X}$, $p$, $Z_{Y}$ and $q$ bounded. We set
\begin{align}
\label{eq:degfree1}
	2x_{X}(\X(s),\Y(s))+J_{X}(\X(s),\Y(s))&=1,\\
\label{eq:degfree2}	
	2x_{Y}(\X(s),\Y(s))+J_{Y}(\X(s),\Y(s))&=1.
\end{align}
In view of \eqref{eq:xJUpqRel}, $x_{X}$ and $J_{X}$ have the same sign, so that \eqref{eq:degfree1} implies that they are non-negative and bounded. Similarly we find that $x_{Y}\geq 0$, $J_{Y}\geq 0$, and that they are bounded from above. From \eqref{eq:initialxXi} and \eqref{eq:initialxXii} it then follows that $t_{X}$, $K_{X}$, $t_{Y}$ and $K_{Y}$ are bounded and $t_{X}\geq 0$, $K_{X}\geq 0$, $t_{Y}\leq 0$ and $K_{Y}\leq 0$. The relation \eqref{eq:xJUpqRel} also implies that $U_{X}$, $p$, $U_{Y}$ and $q$ are bounded.

Using \eqref{eq:initialJX} and \eqref{eq:initialJY} in \eqref{eq:degfree1} and \eqref{eq:degfree2} yields
\begin{align}
\label{eq:initialxX2}
	x_{X}(\X,\Y)&=\bigg(\frac{2}{4+R_{0}^{2}+c(u_{0})\rho_{0}^{2}}\bigg)(x(\X,\Y)),\\ 
\label{eq:initialxY2}	
	x_{Y}(\X,\Y)&=\bigg(\frac{2}{4+S_{0}^{2}+c(u_{0})\sigma_{0}^{2}}\bigg)(x(\X,\Y)),
\end{align}
which implies by \eqref{eq:initialxXi} and \eqref{eq:initialxXii} that
\begin{align*}
	t_{X}(\X,\Y)&=\bigg(\frac{2}{c(u_{0})(4+R_{0}^{2}+c(u_{0})\rho_{0}^{2})}\bigg)(x(\X,\Y)),\\ 	
	t_{Y}(\X,\Y)&=-\bigg(\frac{2}{c(u_{0})(4+S_{0}^{2}+c(u_{0})\sigma_{0}^{2})}\bigg)(x(\X,\Y)).
\end{align*} 
Now \eqref{eq:initialU}-\eqref{eq:initialq} take the form
\begin{align*}
	U_{X}(\X,\Y)&=\bigg(\frac{2R_{0}}{c(u_{0})(4+R_{0}^{2}+c(u_{0})\rho_{0}^{2})}\bigg)(x(\X,\Y)),\\ U_{Y}(\X,\Y)&=-\bigg(\frac{2S_{0}}{c(u_{0})(4+S_{0}^{2}+c(u_{0})\sigma_{0}^{2})}\bigg)(x(\X,\Y)), \\
	J_{X}(\X,\Y)&=\bigg(\frac{R_{0}^{2}+c(u_{0})\rho_{0}^{2}}{4+R_{0}^{2}+c(u_{0})\rho_{0}^{2}}\bigg)(x(\X,\Y)),\\
	J_{Y}(\X,\Y)&=\bigg(\frac{S_{0}^{2}+c(u_{0})\sigma_{0}^{2}}{4+S_{0}^{2}+c(u_{0})\sigma_{0}^{2}}\bigg)(x(\X,\Y)),\\
	p(\X,\Y)&=\bigg(\frac{2\rho_{0}}{4+R_{0}^{2}+c(u_{0})\rho_{0}^{2}}\bigg)(x(\X,\Y)),\\ q(\X,\Y)&=\bigg(\frac{2\sigma_{0}}{4+S_{0}^{2}+c(u_{0})\sigma_{0}^{2}}\bigg)(x(\X,\Y))
\end{align*}
and from \eqref{eq:initialJXJYi} and \eqref{eq:initialJXJYii} we get
\begin{align*}
	K_{X}(\X,\Y)&=\bigg(\frac{R_{0}^{2}+c(u_{0})\rho_{0}^{2}}{c(u_{0})(4+R_{0}^{2}+c(u_{0})\rho_{0}^{2})}\bigg)(x(\X,\Y)),\\
	K_{Y}(\X,\Y)&=-\bigg(\frac{S_{0}^{2}+c(u_{0})\sigma_{0}^{2}}{c(u_{0})(4+S_{0}^{2}+c(u_{0})\sigma_{0}^{2})}\bigg)(x(\X,\Y)).
\end{align*}
It remains to determine $(\X,\Y)$ and the value of $x$, $J$ and $K$ on the curve.

Differentiating \eqref{eq:initialtime} with respect to $s$ yields
\begin{equation*}
	t_{X}(\X(s),\Y(s))\dot{\X}(s)+t_{Y}(\X(s),\Y(s))\dot{\Y}(s)=0
\end{equation*}
and after using \eqref{eq:initialxXi} and \eqref{eq:initialxXii} we get
\begin{equation*}
	x_{X}(\X(s),\Y(s))\dot{\X}(s)=x_{Y}(\X(s),\Y(s))\dot{\Y}(s).
\end{equation*}
Using this in \eqref{eq:compatibility} implies
\begin{equation}
\label{eq:initialxrelation}
	\dot{x}(\X(s),\Y(s))=2x_{X}(\X(s),\Y(s))\dot{\X}(s)=2x_{Y}(\X(s),\Y(s))\dot{\Y}(s).
\end{equation}
We use \eqref{eq:initialxrelation} then \eqref{eq:initialxX2} and \eqref{eq:initialxY2} in \eqref{eq:initialcurvenormalization}, and get
\begin{align*}
	2&=\dot{\X}(s)+\dot{\Y}(s)\\
	&=\frac{1}{2}\bigg(\frac{1}{x_{X}(\X(s),\Y(s))}+\frac{1}{x_{Y}(\X(s),\Y(s))}\bigg)\dot{x}(\X(s),\Y(s))\\
	&=\bigg(2+\frac{1}{4}(R_{0}^{2}+c(u_{0})\rho_{0}^{2}+S_{0}^{2}+c(u_{0})\sigma_{0}^{2})\bigg)(x(\X(s),\Y(s)))\dot{x}(\X(s),\Y(s)).
\end{align*}
We define $x(\X(s),\Y(s))$ implicitly as
\begin{equation}
\label{eq:initialdefx}
	2x(\X(s),\Y(s))+\frac{1}{4}\int_{-\infty}^{x(\X(s),\Y(s))}(R_{0}^{2}+c(u_{0})\rho_{0}^{2}+S_{0}^{2}+c(u_{0})\sigma_{0}^{2})(z)\,dz=2s.
\end{equation}
Note that the left-hand side is a strictly increasing function with respect to $x$, so that \eqref{eq:initialdefx} uniquely defines $x(\X(s),\Y(s))$.

From \eqref{eq:initialxrelation} and \eqref{eq:initialxX2} it follows that
\begin{equation}
\label{eq:Xdot}
	\dot{\X}(s)=\bigg(1+\frac{1}{4}(R_{0}^{2}+c(u_{0})\rho_{0}^{2})\bigg)(x(\X(s),\Y(s)))\dot{x}(\X(s),\Y(s))
\end{equation} 
and we define
\begin{equation*}
	\X(s)=x(\X(s),\Y(s))+\frac{1}{4}\int_{-\infty}^{x(\X(s),\Y(s))}(R_{0}^{2}+c(u_{0})\rho_{0}^{2})(z)\,dz.
\end{equation*}

Similarly, by \eqref{eq:initialxrelation} and \eqref{eq:initialxY2}, we get
\begin{equation}
\label{eq:Ydot}
	\dot{\Y}(s)=\bigg(1+\frac{1}{4}(S_{0}^{2}+c(u_{0})\sigma_{0}^{2})\bigg)(x(\X(s),\Y(s)))\dot{x}(\X(s),\Y(s))
\end{equation}
and we set
\begin{equation*}
	\Y(s)=x(\X(s),\Y(s))+\frac{1}{4}\int_{-\infty}^{x(\X(s),\Y(s))}(S_{0}^{2}+c(u_{0})\sigma_{0}^{2})(z)\,dz.
\end{equation*}

From \eqref{eq:degfree1}, \eqref{eq:degfree2} and \eqref{eq:initialcurvenormalization} we have
\begin{aalign}
\label{eq:xandJdot}
	\dot{J}(\X,\Y)&=J_{X}(\X,\Y)\dot{\X}+J_{Y}(\X,\Y)\dot{\Y}\\
	&=(1-2x_{X}(\X,\Y))\dot{\X}+(1-2x_{Y}(\X,\Y))\dot{\Y}\\
	&=2-2\dot{x}(\X,\Y),
\end{aalign}
so that
\begin{equation*}
	2x(\X(s),\Y(s))+J(\X(s),\Y(s))=2s,
\end{equation*}
which combined with \eqref{eq:initialdefx} yields
\begin{equation*}
	J(\X(s),\Y(s))=\frac{1}{4}\int_{-\infty}^{x(\X(s),\Y(s))}(R_{0}^{2}+c(u_{0})\rho_{0}^{2}+S_{0}^{2}+c(u_{0})\sigma_{0}^{2})(z)\,dz.
\end{equation*}

From \eqref{eq:initialJXJYi} and \eqref{eq:initialJXJYii} we have
\begin{equation*}
	\dot{K}(\X,\Y)=K_{X}(\X,\Y)\dot{\X}+K_{Y}(\X,\Y)\dot{\Y}=\frac{J_{X}(\X,\Y)\dot{\X}-J_{Y}(\X,\Y)\dot{\Y}}{c(u_{0}(x(\X,\Y)))}.
\end{equation*}	
Multiplying \eqref{eq:degfree1} by $\dot{\X}$ and \eqref{eq:degfree2} by $\dot{\Y}$, yields
\begin{equation*}
	J_{X}(\X,\Y)\dot{\X}=\dot{\X}-2x_{X}(\X,\Y)\dot{\X} \quad \text{and} \quad J_{Y}(\X,\Y)\dot{\Y}=\dot{\Y}-2x_{Y}(\X,\Y)\dot{\Y}.
\end{equation*}
Using \eqref{eq:initialxrelation}, we get
\begin{equation*}
	J_{X}(\X,\Y)\dot{\X}-J_{Y}(\X,\Y)\dot{\Y}=\dot{\X}-\dot{\Y},
\end{equation*}
which implies by \eqref{eq:Xdot} and \eqref{eq:Ydot} that
\begin{equation*}	
	\dot{K}(\X,\Y)=\frac{(R_{0}^{2}+c(u_{0})\rho_{0}^{2}-S_{0}^{2}-c(u_{0})\sigma_{0}^{2})(x(\X,\Y))}{4c(u_{0}(x(\X,\Y)))}\dot{x}(\X,\Y).
\end{equation*}
We define
\begin{equation*}	
	K(\X(s),\Y(s))=\int_{-\infty}^{x(\X(s),\Y(s))}\frac{1}{4c(u_{0})}(R_{0}^{2}+c(u_{0})\rho_{0}^{2}-S_{0}^{2}-c(u_{0})\sigma_{0}^{2})(z)\,dz.
\end{equation*}

In Section \ref{sec:WeakSoln} we prove the existence of global, weak, conservative solutions of \eqref{eq:nvwsys}. Our approach follows closely \cite{HolRay:11}. The solutions we construct will be conservative in the sense that for all $t\geq 0$,
\begin{equation*}
	\mu(t)(\mathbb{R})+\nu(t)(\mathbb{R})=\mu_{0}(\mathbb{R})+\nu_{0}(\mathbb{R}),
\end{equation*}
where we denote the solution at time $t$ by $(u,R,S,\rho,\sigma,\mu,\nu)(t)$, see Theorem \ref{thm:cons}. This is a consequence of the fact that the energy function $J$ in Lagrangian coordinates satisfies that the limit
\begin{equation*}
	\lim_{s\rightarrow\pm\infty}J(\X(s),\Y(s))
\end{equation*}
is independent of the curve $(\X,\Y)\in\C$. Thus, the same limiting values of $J$ are obtained for curves corresponding to different times, see Lemma \ref{lemma:curveind}. We do not address uniqueness of conservative solutions. 

The main results of this paper is contained in Section \ref{sec:Reg}, where we first prove that under certain conditions we have local smooth solutions of \eqref{eq:nvwsys}. More specifically, on a finite interval $[x_{l},x_{r}]$ we assume that the initial data satisfies the following: $u_{0},R_{0},S_{0},\rho_{0},\sigma_{0}$ are smooth and bounded, $\mu_{0}$ and $\nu_{0}$ are absolutely continuous, and the functions $\rho_{0}$ and $\sigma_{0}$ are strictly positive. Then we prove that for every time $t\in\big[0,\frac{1}{2\kappa}(x_{r}-x_{l})\big]$, the functions $\rho(t,x)$ and $\sigma(t,x)$ are strictly positive for all $x\in[x_{l}+\kappa t,x_{r}-\kappa t]$. This has a regularizing effect on the solution in the sense that the solution at time $t$ will then satisfy the same regularity conditions as the initial data does on the interval $[x_{l}+\kappa t,x_{r}-\kappa t]$, see Corollary \ref{cor:smooth}. Roughly speaking, the variables $p$ and $q$ contain information about $\rho_{0}$ and $\sigma_{0}$, respectively. In particular, since $p_{Y}=0$ and $q_{X}=0$, the strict positivity of $p$ and $q$ is preserved in the characteristic directions. The identities in \eqref{eq:xJUpqRel} then imply the strict positivity of $x_X$ and $x_Y$.

In Theorem \ref{thm:main} we prove that we can locally approximate weak solutions of \eqref{eq:nvw} by smooth solutions of \eqref{eq:nvwsys} in $L^{\infty}$, provided that certain regularity and convergence conditions hold, see Section \ref{sec:Reg}.

\section{From Eulerian to Lagrangian Coordinates}
\label{sec:3}

We first define the set $\D$ which consists of possible initial data corresponding to \eqref{eq:nvwsys} in Eulerian coordinates.   

\begin{definition}
	The set $\D$ consists of the elements $(u,R,S,\rho,\sigma,\mu,\nu)$ such that
	\begin{equation*}
	u,R,S,\rho,\sigma \in L^{2}(\mathbb{R}),
	\end{equation*}
	\begin{equation}
	\label{eq:setDux}
	u_{x}=\frac{1}{2c(u)}(R-S),
	\end{equation}
	and $\mu$ and $\nu$ are finite positive Radon measures with
	\begin{equation}
	\label{eq:setDabscont}
	\mu_{\emph{ac}}=\frac{1}{4}(R^{2}+c(u)\rho^{2})\,dx \quad \text{and} \quad \nu_{\emph{ac}}=\frac{1}{4}(S^{2}+c(u)\sigma^{2})\,dx.
	\end{equation}
\end{definition}

Note that this definition allows for initial data with concentrated energy. The next step is to map elements from $\D$ to a set $\F$ which is defined as follows.

\begin{definition}\label{def:G}
	The group $G$ is given by all invertible functions $f$ such that
	\begin{equation}
	\label{eq:groupcond}
	f-\id \text{ and } f^{-1}-\id \text{ both belong to } \Winf(\mathbb{R}),
	\end{equation}
	and
	\begin{equation}
	\label{eq:groupcond2}
	(f-\id)'\in L^{2}(\mathbb{R}).
	\end{equation}
\end{definition}

Note that if $f$, $g\in G$, then also $f^{-1}$, $g^{-1}$ and $f\circ g$ belong to $G$.

\begin{definition}
	\label{def:setF}
	The set $\F$ consists of all functions $\psi=(\psi_{1},\psi_{2})$ such that
	\begin{equation*}
	\psi_{1}(X)=(x_{1}(X),U_{1}(X),J_{1}(X),K_{1}(X),V_{1}(X),H_{1}(X))
	\end{equation*}
	and
	\begin{equation*}	
	\psi_{2}(Y)=(x_{2}(Y),U_{2}(Y),J_{2}(Y),K_{2}(Y),V_{2}(Y),H_{2}(Y))
	\end{equation*}
	satisfy the following regularity and decay conditions
	\begin{subequations}
		\label{eqns:setF}
		\begin{equation}
		\label{eq:setF1}
		x_{1}-\id,\ x_{2}-\id,\ J_{1},\ J_{2},\ K_{1},\ K_{2} \in \Winf(\mathbb{R}),
		\end{equation}
		\begin{equation}
		\label{eq:setF2}
		x_{1}'-1,\ x_{2}'-1,\ J_{1}',\ J_{2}',\ K_{1}',\ K_{2}',\ H_{1},\ H_{2} \in L^{2}(\mathbb{R})\cap L^{\infty}(\mathbb{R}), 
		\end{equation}
		\begin{equation}
		\label{eq:setF3}
		U_{1},\ U_{2} \in L^2(\mathbb{R})\cap L^\infty(\mathbb{R}),
		\end{equation}
		\begin{equation}
		\label{eq:setF4}
		V_{1},\ V_{2} \in L^{2}(\mathbb{R})\cap L^{\infty}(\mathbb{R}),
		\end{equation}
	\end{subequations}
	and the additional conditions
	\begin{equation}
	\label{eq:setFrel1}
	x_{1}',x_{2}',J_{1}',J_{2}'\geq 0,
	\end{equation}
	\begin{equation}
	\label{eq:setFrel2}
	J_{1}'=c(U_{1})K_{1}',\quad J_{2}'=-c(U_{2})K_{2}',
	\end{equation}
	\begin{equation}
	\label{eq:setFrel3}
	x_{1}'J_{1}'=(c(U_{1})V_{1})^{2}+c(U_{1})H_{1}^{2}, \quad
	x_{2}'J_{2}'=(c(U_{2})V_{2})^{2}+c(U_{2})H_{2}^{2},
	\end{equation}
	\begin{equation}
	\label{eq:setFrel4}
	x_{1}+J_{1},\ x_{2}+J_{2} \in G, 
	\end{equation}
	\begin{equation}
	\label{eq:setFrel5}
	\displaystyle\lim_{X \rightarrow -\infty} J_{1}(X)=\displaystyle\lim_{Y \rightarrow -\infty} J_{2}(Y)=0.
	\end{equation}
	Moreover, for any curve $(\X,\Y)\in \C$ such that
	\begin{equation*}
	x_{1}(\X(s))=x_{2}(\Y(s)) \text{ for all } s\in \mathbb{R},
	\end{equation*}
	we have
	\begin{subequations}
		\begin{equation}
		\label{eq:setFrel6}
		U_{1}(\X(s))=U_{2}(\Y(s))
		\end{equation}
		for all $s \in \mathbb{R}$ and
		\begin{equation}
		\label{eq:setFrel7}
		\frac{d}{ds}U_{1}(\X(s))=\frac{d}{ds}U_{2}(\Y(s))=V_{1}(\X(s))\dot{\X}(s)+V_{2}(\Y(s))\dot{\Y}(s)
		\end{equation}
		for almost all $s\in \mathbb{R}$.
	\end{subequations}
\end{definition}

Condition \eqref{eq:setFrel4} will be important in Section \ref{sec:SemigroupD}
where we prove that the solution operator from $\D$ to $\D$ is a semigroup. In the proof we use $x_{i}+J_{i}$ for $i=1,2$ as relabeling functions.  

For any strictly monotone curve $(\X,\Y)\in\C$ we introduce 
\begin{equation*}
\X(Y)=\X(\Y^{-1}(Y))\quad  \text{ and }\quad \Y(X)=\Y(\X^{-1}(X)). 
\end{equation*}
In the context of the previous section where we derived the system \eqref{eq:goveq} for smooth solutions $(Z,p,q)$, the elements $(\psi_{1},\psi_{2})$ should be thought of 
\begin{align*}
	x_{1}(X)&=x(X,\Y(X)), & x_{2}(Y)&=x(\X(Y),Y), \\ 
	U_{1}(X)&=U(X,\Y(X)), & U_{2}(Y)&=U(\X(Y),Y), \\
	J_{1}(X)&=\int_{-\infty}^{X}J_{X}(Z,\Y(Z))\,dZ, & J_{2}(Y)&=\int_{-\infty}^{Y}J_{Y}(\X(Z),Z)\,dZ,\\
	K_{1}(X)&=\int_{-\infty}^{X}K_{X}(Z,\Y(Z))\,dZ, & K_{2}(Y)&=\int_{-\infty}^{Y}K_{Y}(\X(Z),Z)\,dZ, \\
	V_{1}(X)&=U_{X}(X,\Y(X)), & V_{2}(Y)&=U_{Y}(\X(Y),Y), \\
	H_{1}(X)&=p(X,\Y(X)), & H_{2}(Y)&=q(\X(Y),Y)
\end{align*}
and
\begin{align*}
	x_{1}'(X)&=2x_{X}(X,\Y(X)), & x_{2}'(Y)&=2x_{Y}(\X(Y),Y), \\
	J_{1}'(X)&=J_{X}(X,\Y(X)), & J_{2}'(Y)&=J_{Y}(\X(Y),Y), \\
	K_{1}'(X)&=K_{X}(X,\Y(X)), & K_{2}'(Y)&=K_{Y}(\X(Y),Y).
\end{align*}

We define the map from $\D$ to $\F$.

\begin{definition}
	\label{def:mapfromDtoF}
	Given $(u,R,S,\rho,\sigma,\mu,\nu)\in \D$, we define $\psi_{1}=(x_{1},U_{1},J_{1},K_{1},V_{1},H_{1})$ and $\psi_{2}=(x_{2},U_{2},J_{2},K_{2},V_{2},H_{2})$ as
	\begin{subequations}
		\begin{align}
		\label{eq:mapfromDtoF1}
		x_{1}(X)&=\sup \{x\in \mathbb{R} \ | \ x'+\mu((-\infty, x'))<X \text{ for all } x'<x\}, \\
		\label{eq:mapfromDtoF2}
		x_{2}(Y)&=\sup \{x\in \mathbb{R} \ | \ x'+\nu((-\infty, x'))<Y \text{ for all } x'<x\}
		\end{align}
		and
		\begin{align}
		\label{eq:mapfromDtoF3}
		&J_{1}(X)=X-x_{1}(X), &&J_{2}(Y)=Y-x_{2}(Y), \\
		\label{eq:mapfromDtoF4}
		&U_{1}(X)=u(x_{1}(X)), &&U_{2}(Y)=u(x_{2}(Y)), \\
		\label{eq:mapfromDtoF5}
		&V_{1}(X)=x_{1}'(X)\frac{R(x_{1}(X))}{2c(U_{1}(X))}, &&V_{2}(Y)=-x_{2}'(Y)\frac{S(x_{2}(Y))}{2c(U_{2}(Y))}, \\
		\label{eq:mapfromDtoF6}
		&K_{1}(X)=\int_{-\infty}^{X}\frac{J_{1}'(\bar{X})}{c(U_{1}(\bar{X}))}\,d\bar{X},
		&&K_{2}(Y)=-\int_{-\infty}^{Y}\frac{J_{2}'(\bar{Y})}{c(U_{2}(\bar{Y}))}\,d\bar{Y}, \\
		\label{eq:mapfromDtoF7}
		&H_{1}(X)=\frac{1}{2}\rho(x_{1}(X))x_{1}'(X), &&H_{2}(Y)=\frac{1}{2}\sigma(x_{2}(Y))x_{2}'(Y).
		\end{align}
	\end{subequations}
	We let $\bf{L}: \D \rightarrow \F$ denote the mapping which to any $(u,R,S,\rho,\sigma,\mu,\nu)\in \D$ associates the element $\psi=(\psi_{1},\psi_{2})\in \F$ as defined above.
\end{definition}

As mentioned before solutions can develop singularities in finite time and energy can concentrate on sets of measure zero. If this is the case one has to put some extra effort into understanding \eqref{eq:mapfromDtoF5} and \eqref{eq:mapfromDtoF7} since they might be of the form $0\cdot \infty$, when $x_1'(X)=0$. One has, in the smooth case for $X_1<X_2$ that
\begin{equation*}
\int_{x_1(X_1)}^{x_1(X_2)} \frac{R}{2c(u)}(x)\,dx=\int_{X_1}^{X_2} \frac{R(x_1(\tilde X)}{2c(u(x_1(\tilde X)))}x_1'(\tilde X)\,d\tilde X = \int_{X_1}^{X_2} V_1(\tilde X)\,d\tilde X
\end{equation*}
and 
\begin{align}
\label{eq:expl}
\left\vert \int_{x_1(X_1)}^{x_1(X_2)} \frac{R}{2c(u)}(x)\,dx\right\vert &\leq \kappa\sqrt{ x_1(X_2)-x_1(X_1)}\sqrt{\mu_{\text{ac}}((x_1(X_1),x_1(X_2))}\\ \nonumber &\leq \kappa\sqrt{ x_1(X_2)-x_1(X_1)}\sqrt{ J_1(X_2)-J_1(X_1)}.
\end{align}
If we now consider the nonsmooth case, \eqref{eq:expl} still holds and the above calculations imply that $V_1(X)$ exists and is bounded.  Furthermore, if $x_1'(X)=0$, we must have that $V_1(X)=0$.

In the case of initial data $(u_{0},R_{0},S_{0},\rho_{0},\sigma_{0},\mu_{0},\nu_{0})\in\D$ such that $\mu_{0}$ and $\nu_{0}$ are absolutely continuous with respect to the Lebesgue measure, we check that we end up with the same expressions as at the end of Section \ref{sec:equivsys}. By \eqref{eq:setDabscont} we have
\begin{equation*}
	\mu_{0}((-\infty,x))=\frac{1}{4}\int_{-\infty}^{x}(R_{0}^{2}+c(u_{0})\rho_{0}^{2})(z)\,dz
\end{equation*}
and
\begin{equation*}
	\nu_{0}((-\infty,x))=\frac{1}{4}\int_{-\infty}^{x}(S_{0}^{2}+c(u_{0})\sigma_{0}^{2})(z)\,dz.
\end{equation*}
Since the functions $x+\mu_{0}((-\infty,x))$ and $x+\nu_{0}((-\infty,x))$ are continuous and strictly increasing, we get from \eqref{eq:mapfromDtoF1} and \eqref{eq:mapfromDtoF2},
\begin{equation*}
	x_{1}(X)+\frac{1}{4}\int_{-\infty}^{x_{1}(X)}(R_{0}^{2}+c(u_{0})\rho_{0}^{2})(z)\,dz=X	
\end{equation*}
and
\begin{equation*}
	x_{2}(Y)+\frac{1}{4}\int_{-\infty}^{x_{2}(Y)}(S_{0}^{2}+c(u_{0})\sigma_{0}^{2})(z)\,dz=Y.
\end{equation*}
We add these equalities and since $x_{1}(\X(s))=x_{2}(\Y(s))=x(\X(s),\Y(s))$ and $\X(s)+\Y(s)=2s$ we get
\begin{equation*}
	2x(\X(s),\Y(s))+\frac{1}{4}\int_{-\infty}^{x(\X(s),\Y(s))}(R_{0}^{2}+c(u_{0})\rho_{0}^{2}+S_{0}^{2}+c(u_{0})\sigma_{0}^{2})(z)\,dz=2s
\end{equation*}
and we recover (\ref{eq:initialdefx}). In a similar way we can show by using Definition \ref{def:mapfromDtoF} that we get the other expressions that we derived in Section \ref{sec:equivsys}.

We illustrate the mappings in this section with a series of examples. We study three possible situations where the initial measures are absolutely continuous and discrete. We want to illustrate how the region where $x_{1}(X)=x_{2}(Y)$ and $Y=2s-X$ in the $(X,Y)$-plane looks like in each situation. This is important in \eqref{eq:mapFtoGX}, the definition of the initial curve $(\X,\Y)$.  

\textbf{Example 1.} We first consider the case where both $\mu_{0}$ and $\nu_{0}$ are absolutely continuous. More specifically, let
\begin{equation*}
	\mu_{0}((-\infty,x])=\nu_{0}((-\infty,x])=\arctan(x)+\frac{\pi}{2}.
\end{equation*}
The measures are absolutely continuous.
Let $f(x)=\arctan(x)+x+\frac{\pi}{2}$, which is strictly increasing and continuous. We have $x_{1}(X)=f^{-1}(X)$ and $x_{2}(Y)=f^{-1}(Y)$. By differentiating the identity $f(f^{-1}(X))=X$ we obtain
\begin{equation*}
	x_{1}'(X)=\frac{1}{1+\frac{1}{1+x_{1}(X)^{2}}}\geq \frac{1}{2},
\end{equation*} 
and similarly we get $x_{2}'(Y)\geq\frac{1}{2}$, so that both $x_{1}$ and $x_{2}$ are strictly increasing functions. 

\textbf{Example 2.}  We consider the case where one of the measures is absolutely continuous and the other is not. Let 
\begin{equation*}
	\mu_{0}((-\infty,x))=\begin{cases} 0, & x\leq 0,\\ 1, & x> 0 \end{cases} \quad \text{and} \quad \nu_{0}((-\infty,x))=\arctan(x)+\frac{\pi}{2}.
\end{equation*}
The measure $\mu_{0}$ is not absolutely continuous with respect to the Lebesgue measure, as
\begin{equation*}
	\mu_{0}(\{0\})=\mu_{0}\bigg(\bigcap_{n=1}^{\infty}\bigg[0,\frac{1}{n}\bigg)\bigg)=\lim_{n\rightarrow\infty}\mu_{0}\bigg(\bigg[0,\frac{1}{n}\bigg)\bigg)=1.
\end{equation*} 
Using Definition \ref{def:mapfromDtoF}, we find that
\begin{equation*}
	x_{1}(X)=\begin{cases} X & \text{if } X\leq 0,\\ 0 &\text{if } 0\leq X\leq 1,\\ X-1, &\text{if } 1\leq X, \end{cases}
\end{equation*}
and as in Example 1 we have $x_{2}(Y)=f^{-1}(Y)$.

\textbf{Example 3.} We consider the case where both measures are singular at the same point. Let $(u_{0},R_{0},S_{0},\rho_{0},\sigma_{0},\mu_{0},\nu_{0})\in\D$ be such that
\begin{equation*}
	\mu_{0}((-\infty,x))=\begin{cases} 0, & x\leq 0,\\ 1, & x> 0 \end{cases} \quad \text{and} \quad \nu_{0}((-\infty,x))=\begin{cases} 0, & x\leq 0,\\ 1, & x> 0. \end{cases}
\end{equation*}
From Definition \ref{def:mapfromDtoF} we get
\begin{equation*}
	x_{1}(X)=\begin{cases} X & \text{if } X\leq 0,\\ 0 &\text{if } 0\leq X\leq 1,\\ X-1, &\text{if } 1\leq X \end{cases}  \quad \text{and} \quad  x_{2}(Y)=\begin{cases} Y & \text{if } Y\leq 0,\\ 0 &\text{if } 0\leq Y\leq 1,\\ Y-1, &\text{if } 1\leq Y. \end{cases}
\end{equation*} 

\begin{proof}[Proof of the well-posedness of Definition \ref{def:mapfromDtoF}]	
	We only show that $\psi_{1}$ as defined above satisfies the conditions in the definition of $\F$. The corresponding proof for $\psi_{2}$ is similar. First we prove that the derivatives of $x_{1}$ and $J_{1}$ are well-defined. Let us show that $x_{1}$ is Lipschitz continuous. Consider $X,X'\in \mathbb{R}$ such that $X<X'$ and $x_{1}(X)<x_{1}(X')$. The definition of $x_{1}$ implies that there exists an increasing sequence, $z_{i}'$, and a decreasing one, $z_{i}$, such that $\displaystyle\lim_{i\rightarrow \infty}z_{i}'=x_{1}(X')$ and $\displaystyle\lim_{i\rightarrow \infty}z_{i}=x_{1}(X)$ with $z_{i}'+\mu((-\infty,z_{i}'))<X'$ and $z_{i}+\mu((-\infty,z_{i}))\geq X$. Combining these two inequalities gives  
	\begin{equation*}
	\mu((-\infty,z_{i}'))-\mu((-\infty,z_{i}))+z_{i}'-z_{i}<X'-X.
	\end{equation*}
	For sufficiently large $i$, we have $z_{i}'>z_{i}$, so that $\mu((-\infty,z_{i}'))-\mu((-\infty,z_{i}))=\mu([z_{i},z_{i}'))\geq 0$. Hence, $z_{i}'-z_{i}<X'-X$. Letting $i$ tend to infinity, we obtain $x_{1}(X')-x_{1}(X)\leq X'-X$ and $x_{1}$ is Lipschitz continuous with Lipschitz constant at most one. Thus, $x_{1}$ is differentiable almost everywhere. Then, by \eqref{eq:mapfromDtoF3} it follows that $J_{1}$ is Lipschitz continuous with Lipschitz constant at most two, so that $J_{1}$ is differentiable almost everywhere. 
		
	Next, we show \eqref{eq:setF1}-\eqref{eq:setF4} and that $K_{1}$ is well-defined and differentiable almost everywhere. It is clear from \eqref{eq:mapfromDtoF1} that $x_{1}$ yields a nondecreasing function. For any $z>x_{1}(X)$, we have $z+\mu((-\infty,z))\geq X$. Hence, $X-z \leq \mu(\mathbb{R})$ and, since we can choose $z$ arbitrarily close to $x_{1}(X)$, we obtain $X-x_{1}(X) \leq \mu(\mathbb{R})$. Since $x_{1}(X)\leq X$, we have
	\begin{equation}
	\label{eq:mapfromDtoFproofid}
	|X-x_{1}(X)| \leq \mu(\mathbb{R})
	\end{equation}
	and $x_{1}-\id \in L^{\infty}(\mathbb{R})$. Since $x_{1}$ is nondecreasing and has Lipschitz constant at most one, we have $0\leq x_{1}'\leq 1$ almost everywhere, so that $x_{1}'-1\in \Linf(\mathbb{R})$. From (\ref{eq:mapfromDtoFproofid}), we obtain $|J_{1}(X)|\leq \mu(\mathbb{R})$ and $J_{1}\in \Linf(\mathbb{R})$. We have $J_{1}'=1-x_{1}'$ a.e. and therefore $0\leq J_{1}'\leq 1$ a.e., which implies that $J_{1}'\in \Linf(\mathbb{R})$. Thus, $J_{1}'\in L^{1}(\mathbb{R})$ as $\int_{-\infty}^{X}J_{1}'(\bar{X})\,d\bar{X}\leq ||J_{1}||_{\Linf(\mathbb{R})}$. By H{\"o}lder's inequality, we obtain
	\begin{equation*}
	\norm{J_{1}'}_{L^{2}(\mathbb{R})}^{2}\leq \norm{J_{1}'}_{L^{\infty}(\mathbb{R})}\norm{J_{1}'}_{L^{1}(\mathbb{R})}\leq \norm{J_{1}}_{L^{\infty}(\mathbb{R})}\leq \mu(\mathbb{R}).
	\end{equation*}
	Hence, $J_{1}'\in L^{2}(\mathbb{R})$ and since $J_{1}'=1-x_{1}'$ a.e., we have that $x_{1}'-1\in L^{2}(\mathbb{R})$. Note that, by the above, the inequalities for $x_{1}'$ and $J_{1}'$ in \eqref{eq:setFrel1} are satisfied. The fact that $J_{1}'$ is integrable also implies that $K_{1}$ is well-defined and differentiable almost everywhere. By differentiating \eqref{eq:mapfromDtoF6}, we obtain $K_{1}'=\frac{J_{1}'}{c(U_{1})}$, so that $K_{1}\in \Winf(\mathbb{R})$ and $K_{1}'\in L^{2}(\mathbb{R})\cap L^\infty(\mathbb{R})$. By a change of variables and, using the fact that $x_{1}'\leq 1$ a.e., we get
	\begin{equation*}
	\int_{\mathbb{R}}H_{1}^{2}(X)\,dX\leq \frac{1}{4}\int_{\mathbb{R}}\rho^{2}(x)\,dx<\infty
	\end{equation*}
	and
	\begin{equation*}
	\int_{\mathbb{R}}V_{1}^{2}(X)\,dX\leq \frac{\kappa^{2}}{4}\int_{\mathbb{R}}R^{2}(x)\,dx<\infty,
	\end{equation*}
	so that $H_{1}$ and $V_{1}$ belong to $L^{2}(\mathbb{R})$.
	
	Next, we prove that $U_{1}$ is in $L^{2}(\mathbb{R})$. Let $B_{3}=\{X\in \mathbb{R} \ | \ x_{1}'(X)<\frac{1}{2} \}$. Since $J_{1}'=1-x_{1}'$, $B_{3}=\{X\in \mathbb{R} \ | \ J_{1}'(X)>\frac{1}{2} \}$, and $J_{1}'\in L^{2}(\mathbb{R})$, we obtain $\text{meas}(B_{3})<\infty$ after using Chebyshev's inequality. We have, since $x_{1}' \geq \frac{1}{2}$ in $B_{3}^{c}$,
	\begin{align*}
	\int_{\mathbb{R}}U_{1}^{2}(X)\,dX&=\int_{B_{3}}U_{1}^{2}(X)\,dX+\int_{B_{3}^{c}}U_{1}^{2}(X)\,dX \\
	&\leq \text{meas}(B_{3})\norm{u}_{L^{\infty}(\mathbb{R})}^{2}+2\int_{B_{3}^{c}}u^{2}(x_{1}(X))x_{1}'(X)\,dX \\
	&\leq \text{meas}(B_{3})\norm{u}_{L^{\infty}(\mathbb{R})}^{2}+2\norm{u}_{L^{2}(\mathbb{R})}^{2}.
	\end{align*}
	Since $u\in H^{1}(\mathbb{R})$, we have that $u\in \Linf(\mathbb{R})$ and we conclude that $U_{1}\in L^{2}(\mathbb{R})\cap L^\infty(\mathbb{R})$. It remains to show that $H_{1}$ and $V_{1}$ belong to $\Linf(\mathbb{R})$. In order to prove this and \eqref{eq:setFrel3}, we have to compute the derivative of $x_{1}$. Following \cite{Folland}, we decompose $\mu$ into its absolutely continuous, singular continuous and discrete part, denoted by $\mu_{\text{ac}}$, $\mu_{\text{sc}}$ and $\mu_{\text{d}}$, respectively. The support of $\mu_{\text{d}}$ consists of a countable set of points. The function $G(x)=\mu((-\infty,x))$ is lower semi-continuous and its points of discontinuity coincide exactly with the support of $\mu_{\text{d}}$. Let $A$ denote the complement of $x_{1}^{-1}(\text{supp}(\mu_{\text{d}}))$, that is, $A=\{X\in \mathbb{R} \ | \ x_{1}(X)\in \text{supp}(\mu_{\text{d}})^{c}\}$. We claim that for any $X\in A$, we have
	\begin{equation}
	\label{eq:mapfromDtoFproofclaim}
	\mu((-\infty,x_{1}(X)))+x_{1}(X)=X.
	\end{equation}  
	By \eqref{eq:mapfromDtoF1} there exists an increasing sequence $z_{i}$ which converges to $x_{1}(X)$ such that $G(z_{i})+z_{i}<X$. Since $G$ is lower semi-continuous, $\displaystyle\lim_{i\rightarrow \infty}G(z_{i})=G(x_{1}(X))$ and therefore
	\begin{equation*}
	G(x_{1}(X))+x_{1}(X)\leq X.
	\end{equation*}
	Assume that $G(x_{1}(X))+x_{1}(X)<X$. Since $x_{1}(X)$ is a point of continuity of $G$, we can find an $x$ such that $x>x_{1}(X)$ and $G(x)+x<X$. This contradicts the definition of $x_{1}(X)$ and proves our claim (\ref{eq:mapfromDtoFproofclaim}). Let
	\begin{equation*}
	B_{1}=\Big\{ x\in \mathbb{R} \ | \ \displaystyle\lim_{\varepsilon \downarrow 0}\frac{1}{2\varepsilon}\mu((x-\varepsilon,x+\varepsilon))=\frac{1}{4}(R^{2}(x)+c(u(x))\rho^{2}(x)) \Big\}.
	\end{equation*}
	Since $\frac{1}{4}(R^{2}+c\rho^{2})\,dx$ is the absolutely continuous part of $\mu$, we have from Besicovitch's derivation theorem that $\text{meas}(B_{1}^{c})=0$. The proof can be found in \cite{Ambrosio}. Given $X\in x_{1}^{-1}(B_{1})$, we denote $x=x_{1}(X)$. We claim that for all $i\in \mathbb{N}$, there exists $0<\varepsilon< \frac{1}{i}$ such that $x-\varepsilon$ and $x+\varepsilon$ both belong to $\text{supp}(\mu_{\text{d}})^{c}$. Let us assume the opposite. Then, there exists $i\in \mathbb{N}$ such that for all $0<\varepsilon< \frac{1}{i}$, $x-\varepsilon$ and $x+\varepsilon$ both belong to $\text{supp}(\mu_{\text{d}})$. Since the set $(0,\frac{1}{i})$ is uncountable, this implies that uncountably many points belong to $\text{supp}(\mu_{\text{d}})$. This is a contradiction, and our claim is proved. Hence, we can find two sequences $X_{i}$ and $X_{i}'$ in $A$ such that $\frac{1}{2}(x_{1}(X_{i})+x_{1}(X_{i}'))=x_{1}(X)$ and $0<X_{i}'-X_{i}<\frac{1}{i}$. We have by (\ref{eq:mapfromDtoFproofclaim}), since $X_{i}$ and $X_{i}'$ belong to $A$,
	\begin{equation}
	\label{eq:mapfromDtoFproofmeas}
	\mu([x_{1}(X_{i}),x_{1}(X'_{i})))+x_{1}(X'_{i})-x_{1}(X_{i})=X_{i}'-X_{i}.
	\end{equation} 
	Since $x_{1}(X_{i})\notin \text{supp}(\mu_{\text{d}})$, we infer that $\mu(\{x_{1}(X_{i})\})=0$ and $\mu([x_{1}(X_{i}),x_{1}(X'_{i})))=\mu((x_{1}(X_{i}),x_{1}(X'_{i})))$. Dividing (\ref{eq:mapfromDtoFproofmeas}) by $X_{i}'-X_{i}$, we get
	\begin{equation*}
	\frac{x_{1}(X'_{i})-x_{1}(X_{i})}{X_{i}'-X_{i}}\frac{\mu((x_{1}(X_{i}),x_{1}(X'_{i})))}{x_{1}(X'_{i})-x_{1}(X_{i})}+\frac{x_{1}(X'_{i})-x_{1}(X_{i})}{X_{i}'-X_{i}}=1
	\end{equation*}
	and letting $i$ tend to infinity, we obtain
	\begin{equation}
	\label{eq:mapfromDtoFproofx1}
	x_{1}'(X)\frac{1}{4}(R^{2}+c(u)\rho^{2})(x_{1}(X))+x_{1}'(X)=1
	\end{equation}
	for almost every $X\in x_{1}^{-1}(B_{1})$. Since $\mathbb{R}=x_{1}^{-1}(B_{1})\cup x_{1}^{-1}(B_{1}^{c})$, it remains to study the behavior of $x_{1}'$ in $x_{1}^{-1}(B_{1}^{c})$. We proved above that $\text{meas}(B_{1}^{c})=0$, which does not imply in general that $\text{meas}(x_{1}^{-1}(B_{1}^{c}))=0$.\footnote{If $\mu=\delta_{0}$, then $B_{1}^{c}=\{0\}$, but $x_{1}^{-1}(\{0\})=[0,1]$.} Therefore, we need the following result.
	
	\begin{lemma}[{\cite[Lemma 3.9]{HolRay:07}}]
		\label{lemma:increasinglipschitz}
		Given an increasing Lipschitz continuous function $f:\mathbb{R}\rightarrow \mathbb{R}$, for any set $B$ of measure zero, we have $f'=0$ almost everywhere in $f^{-1}(B)$.
	\end{lemma}
	
	We apply Lemma \ref{lemma:increasinglipschitz} and get, since $\text{meas}(B_{1}^{c})=0$, that $x_{1}'=0$ almost everywhere in $x_{1}^{-1}(B_{1}^{c})$. From (\ref{eq:mapfromDtoFproofx1}), we get
	\begin{align*}
	x_{1}'(X)J_{1}'(X)&=x_{1}'(X)^{2}\frac{1}{4}(R^{2}(x_{1}(X))+c(U_{1}(X))\rho^{2}(x_{1}(X))) \\
	&=(c(U_{1}(X))V_{1}(X))^{2}+c(U_{1}(X))H_{1}^{2}(X)
	\end{align*}
	and \eqref{eq:setFrel3} follows. The relation in \eqref{eq:setFrel2} follows by differentiating \eqref{eq:mapfromDtoF6}. Now we can prove that $H_{1}$ and $V_{1}$ belong to $\Linf(\mathbb{R})$. By \eqref{eq:setFrel3}, we have
	\begin{align*}
	0&\leq \bigg(\frac{1}{\kappa}|V_{1}|+\frac{1}{\sqrt{\kappa}}|H_{1}|\bigg)^{2}\\
	&\leq \big(c(U_{1})|V_{1}|+\sqrt{c(U_{1})}|H_{1}|\big)^{2}\\
	&\leq 2\big(c^{2}(U_{1})V_{1}^{2}+c(U_{1})H_{1}^{2}\big)\\
	&=2x_{1}'J_{1}'\leq 2
	\end{align*}
	since $x_{1}',J_{1}'\in [0,1]$. This implies that $H_{1},V_{1}\in \Linf(\mathbb{R})$. 
	
	Since $x_1+J_1=\id$, all conditions in Definition~\ref{def:G} are satisfied and hence also \eqref{eq:setFrel4}. 
		
	The function $J_{1}$ vanishes at $-\infty$ since $J_{1}$ is non-decreasing and non-negative and $x_{1}(X)\leq X$. Hence, \eqref{eq:setFrel5} is satisfied for $J_{1}$. Let us verify that \eqref{eq:setFrel6} and \eqref{eq:setFrel7} hold. Consider a curve $(\X,\Y)\in \C$ such that $x_{1}(\X(s))=x_{2}(\Y(s))$. By \eqref{eq:mapfromDtoF4}, we have
	\begin{equation*}
	U_{1}(\X(s))=u(x_{1}(\X(s)))=u(x_{2}(\Y(s)))=U_{2}(\Y(s)).
	\end{equation*}
	We obtain
	\begin{align*}
	U_1(\X(\bar s))-U_1(\X(s))&=\int_{x_1(\X(s))}^{x_1(\X(\bar s))} u_{x}(x)\,dx\\
	& = \int_{x_1(\X(s))}^{x_1(\X(\bar s))}\frac{(R-S)}{2c(u)}(x)\,dx\\
	& =  \int_{x_1(\X(s))}^{x_1(\X(\bar s))} \frac{R}{2c(u)}(x)dx-\int_{x_2(\Y(s))}^{x_2(\Y(\bar s))} \frac{S}{2c(u)}(x)\,dx\\
	& = \int_s^{\bar s}(V_{1}(\X)\dot{\X}+V_{2}(\Y)\dot{\Y})(r)\,dr 
	\end{align*}
	where we used that $u_{x}=\frac{1}{2c(u)}(R-S)$ and $x_{1}(\X)=x_{2}(\Y)$.
	Dividing both sides by $\bar s-s$ and letting $\bar s\to s$ yields \eqref{eq:setFrel7}.
\end{proof}

Given an element in $\F$ we want to define a curve $(\X,\Y)$ and the values of $\psi$ on that curve. We define the set $\G$ which consists of curves $(\X,\Y)$ and five functions $\Z$, $\V$, $\W$, $\p$ and $\q$, next. Recalling Section \ref{sec:equivsys} the idea is that these functions in the smooth case are given through
\begin{equation*}
	\Z(s)=Z(\X(s),\Y(s))
\end{equation*} 
and
\begin{align*}
	&\V(\X(s))=Z_{X}(\X(s),\Y(s)), &&\W(\Y(s))=Z_{Y}(\X(s),\Y(s)), \\
	&\p(\X(s))=p(\X(s),\Y(s)), &&\q(\Y(s))=q(\X(s),\Y(s)),
\end{align*}
and hence motivate some of the regularity conditions that are imposed in the definition of the set $\G$. For example, from the derivation in the previous section we know that the function $x(X,Y)$ is increasing with respect to both its arguments and is therefore unbounded. However, from \eqref{eq:initialdefx} we get
\begin{equation*}
	|x(\X(s),\Y(s))-s|\leq\frac{1}{2}(\mu_{0}(\mathbb{R})+\nu_{0}(\mathbb{R}))
\end{equation*}
which belongs to $L^{\infty}(\mathbb{R})$. Therefore, we require that $\Z_{2}-\id$ belongs to $L^{\infty}(\mathbb{R})$. 

It is convenient to introduce the following notation: to any triplet $(\Z,\V,\W)$ of five dimensional vector functions we associate a triplet $(\Z^{a},\V^{a},\W^{a})$ given by
\begin{subequations}\label{eq:atriplet}
\begin{align}
	&\Z_{1}^{a}=\Z_{1}-\frac{1}{c(0)}(\X-\id), &&\V_{1}^{a}=\V_{1}-\frac{1}{2c(0)}, &&\W_{1}^{a}=\W_{1}+\frac{1}{2c(0)}, \\
	&\Z_{2}^{a}=\Z_{2}-\id, &&\V_{2}^{a}=\V_{2}-\frac{1}{2}, &&\W_{2}^{a}=\W_{2}-\frac{1}{2}, \\
	&\Z_{i}^{a}=\Z_{i}, &&\V_{i}^{a}=\V_{i}, &&\W_{i}^{a}=\W_{i}
\end{align}
for $i \in \{3,4,5 \}$.
\end{subequations}

\begin{definition}
	\label{def:setG}
	The set $\G$ is the set of all elements $\Theta=(\X,\Y,\Z,\V,\W,\p,\q)$ which consist of a curve $(\X(s),\Y(s))\in \C$, three vector-valued functions 
	\begin{align*}
		\Z(s)&=(\Z_{1}(s),\Z_{2}(s),\Z_{3}(s),\Z_{4}(s),\Z_{5}(s)),\\
		\V(X)&=(\V_{1}(X),\V_{2}(X),\V_{3}(X),\V_{4}(X),\V_{5}(X)),\\
		\W(Y)&=(\W_{1}(Y),\W_{2}(Y),\W_{3}(Y),\W_{4}(Y),\W_{5}(Y)),
	\end{align*}
	and two functions $\p(X)$ and $\q(Y)$. We set
	\begin{equation}
	\label{eq:setGnorm}
	\norm{\Theta}_{\G}^{2}=\norm{\Z_{3}}_{L^{2}(\mathbb{R})}^{2}+\norm{\V^{a}}_{L^{2}(\mathbb{R})}^{2}+\norm{\W^{a}}_{L^{2}(\mathbb{R})}^{2}+\norm{\p}_{L^{2}(\mathbb{R})}^{2}+\norm{\q}_{L^{2}(\mathbb{R})}^{2}
	\end{equation}
	and	
	\begin{aalign}
	\label{eq:setGseminorm}
	\tnorm{\Theta}_{\G}&=\norm{(\X,\Y)}_{\C}+\norm{\frac{1}{\V_{2}+\V_{4}}}_{L^{\infty}(\mathbb{R})}+\norm{\frac{1}{\W_{2}+\W_{4}}}_{L^{\infty}(\mathbb{R})} \\
	&\quad +\norm{\Z^{a}}_{L^{\infty}(\mathbb{R})}+\norm{\V^{a}}_{L^{\infty}(\mathbb{R})}+\norm{\W^{a}}_{L^{\infty}(\mathbb{R})}+\norm{\p}_{L^{\infty}(\mathbb{R})}+\norm{\q}_{L^{\infty}(\mathbb{R})}.
	\end{aalign}
	The element $\Theta$ belongs to $\G$ if
	\begin{enumerate}
		\item [(i)]
		\begin{equation}
		\label{eq:setGmember}
		\norm{\Theta}_{\G}<\infty \quad \text{and} \quad \tnorm{\Theta}_{\G}<\infty;
		\end{equation}
		\item [(ii)]
		\begin{equation}
		\label{eq:setGpositive}
		\V_{2},\W_{2},\V_{4},\W_{4}\geq 0;
		\end{equation}
		\item [(iii)] for almost every $s\in\mathbb{R}$, we have
		\begin{equation}
		\label{eq:setGcomp}
		\dot{\Z}(s)=\V(\X(s))\dot{\X}(s)+\W(\Y(s))\dot{\Y}(s)
		\end{equation}
		\item [(iv)]
		\begin{subequations}
			\begin{align}
			\label{eq:setGrel1}
			&\V_{2}(\X(s))=c(\Z_{3}(s))\V_{1}(\X(s)), &&\W_{2}(\Y(s))=-c(\Z_{3}(s))\W_{1}(\Y(s)), \\
			\label{eq:setGrel2}
			&\V_{4}(\X(s))=c(\Z_{3}(s))\V_{5}(\X(s)), &&\W_{4}(\Y(s))=-c(\Z_{3}(s))\W_{5}(\Y(s))
			\end{align}	
			and	
			\begin{align}
			\label{eq:setGrel3}
			2\V_{4}(\X(s))\V_{2}(\X(s))&=(c(\Z_{3}(s))\V_{3}(\X(s)))^{2}+c(\Z_{3}(s))\p^{2}(\X(s)), \\ 
			\label{eq:setGrel4}
			2\W_{4}(\Y(s))\W_{2}(\Y(s))&=(c(\Z_{3}(s))\W_{3}(\Y(s)))^{2}+c(\Z_{3}(s))\q^{2}(\Y(s));
			\end{align}
		\end{subequations}
		\item [(v)]
		\begin{equation}
		\label{eq:setGrel5}
		\displaystyle\lim_{s \rightarrow -\infty} \Z_{4}(s)=0.
		\end{equation}   
	\end{enumerate}
	We denote by $\G_{0}$ the subset of $\G$ which parametrize the data at time $t=0$, that is,
	\begin{equation*}
	\G_{0}=\{ \Theta \in \G \ | \ \Z_{1}=0 \}.
	\end{equation*}
\end{definition}

For $\Theta \in \G_{0}$, we get by using (\ref{eq:setGcomp}) and (\ref{eq:setGrel1}), that
\begin{equation}
\label{eq:setG0rel}
\V_{2}(\X(s))\dot{\X}(s)=\W_{2}(\Y(s))\dot{\Y}(s).
\end{equation}
This implies that
\begin{equation}
\label{eq:setG0rel2}
\dot{\Z}_{2}(s)=2\V_{2}(\X(s))\dot{\X}(s)=2\W_{2}(\Y(s))\dot{\Y}(s)
\end{equation}

Note that for an element $\Theta\in\G$ we have $\V_{2}+\V_{4}>0$ and $\W_{2}+\W_{4}>0$ almost everywhere. As we shall see, this property is preserved in the solution and is important in proving that the solution operator from $\D$ to $\D$ is a semigroup.

\begin{definition}
	\label{def:mapfromFtoG0}
	For any $\psi=(\psi_{1},\psi_{2})\in \F$, we define $(\X,\Y,\Z,\V,\W,\p,\q)$ as
	\begin{equation}
	\label{eq:mapFtoGX}
	\X(s)=\sup\{X\in \mathbb{R} \ | \ x_{1}(X')<x_{2}(2s-X') \text{ for all } X'<X \}
	\end{equation}
	and set $\Y(s)=2s-\X(s)$. We have
	\begin{equation}
	\label{eq:mapFtoG2}
	x_{1}(\X(s))=x_{2}(\Y(s)).
	\end{equation}
	We define
	\begin{subequations}
		\begin{align}
		\label{eq:mapFtoGZ1}
		\Z_{1}(s)&=0, \\
		\label{eq:mapFtoGZ2}	
		\Z_{2}(s)&=x_{1}(\X(s))=x_{2}(\Y(s)), \\
		\label{eq:mapFtoGZ3}	
		\Z_{3}(s)&=U_{1}(\X(s))=U_{2}(\Y(s)), \\
		\label{eq:mapFtoGZ4}	
		\Z_{4}(s)&=J_{1}(\X(s))+J_{2}(\Y(s)), \\
		\label{eq:mapFtoGZ5}	
		\Z_{5}(s)&=K_{1}(\X(s))+K_{2}(\Y(s))
		\end{align}
	\end{subequations}
	and
	\begin{subequations}
		\begin{align}
		\label{eq:mapFtoGV1}
		&\V_{1}(X)=\frac{1}{2c(U_{1}(X))}x_{1}'(X), &&\W_{1}(Y)=-\frac{1}{2c(U_{2}(Y))}x_{2}'(Y), \\
		\label{eq:mapFtoGV2}
		&\V_{2}(X)=\frac{1}{2}x_{1}'(X), &&\W_{2}(Y)=\frac{1}{2}x_{2}'(Y), \\
		\label{eq:mapFtoGV3}
		&\V_{3}(X)=V_{1}(X), &&\W_{3}(Y)=V_{2}(Y), \\
		\label{eq:mapFtoGV4}
		&\V_{4}(X)=J_{1}'(X), &&\W_{4}(Y)=J_{2}'(Y), \\
		\label{eq:mapFtoGV5}
		&\V_{5}(X)=K_{1}'(X), &&\W_{5}(Y)=K_{2}'(Y), \\
		\label{eq:mapFtoGp}
		&\p(X)=H_{1}(X), &&\q(Y)=H_{2}(Y).
		\end{align}
	\end{subequations}
	Let $\bf{C}:\F\rightarrow \G_{0}$ denote the mapping which to any $\psi \in \F$ associates the element $(\X,\Y,\Z,\V,\W,\p,\q)\in \G_{0}$ as defined above.
\end{definition}

\textbf{Example 1 continued.} The function $\X(s)$ is given as the unique point of intersection $X$ between $x_{1}(X)$ and $x_{2}(2s-X)$, i.e.,
\begin{equation*}
	x_{1}(\X(s))=x_{2}(2s-\X(s)),
\end{equation*}
which implies that
\begin{equation*}
	\X(s)=s \quad \text{and} \quad \Y(s)=s.
\end{equation*}
Hence, $(\X,\Y)$ is a strictly monotone curve.

\textbf{Example 2 continued.} For $s\leq\frac{\pi}{4}$, $\X(s)$ is given implicitly as the solution of the equation $X=x_{2}(2s-X)$, or
\begin{equation*}
	\arctan(\X(s))+2\X(s)+\frac{\pi}{2}=2s.
\end{equation*}
By differentiating, we get 
\begin{equation*}
	\dot{\X}(s)=\frac{1}{1+\frac{1}{2+2\X(s)^{2}}},
\end{equation*}
so that $\frac{2}{3}\leq\dot{\X}(s)\leq 1$ which implies that $\dot{\Y}(s)\geq 1$. Hence, for $s\leq\frac{\pi}{4}$  $(\X,\Y)$ is a strictly monotone curve.

For $\frac{\pi}{4}\leq s\leq \frac{\pi}{4}+\frac{1}{2}$, we have
\begin{equation*}
	\X(s)=2s-\frac{\pi}{2}
\end{equation*}
and $\Y(s)=\frac{\pi}{2}$, that is, $\Y(s)$ is constant.

For $s>\frac{\pi}{4}+\frac{1}{2}$, $\X(s)$ is given as the solution of the equation $X-1=x_{2}(2s-X)$, that is
\begin{equation*}
	\arctan(\X(s)-1)+2\X(s)-1+\frac{\pi}{2}=2s.
\end{equation*}
We differentiate and get 
\begin{equation*}
	\dot{\X}(s)=\frac{1}{1+\frac{1}{2+2(\X(s)-1)^{2}}},
\end{equation*}
so that $\frac{2}{3}\leq\dot{\X}(s)\leq 1$ and $\dot{\Y}(s)\geq 1$. Hence, for $s>\frac{\pi}{4}+\frac{1}{2}$ the curve $(\X,\Y)$ is the graph of a strictly increasing function. 

We conclude that the curve $(\X,\Y)$ consists of two strictly increasing parts (when $s\leq\frac{\pi}{4}$ and $s>\frac{\pi}{4}+\frac{1}{2}$) which are joined by a horizontal line segment (when $\frac{\pi}{4}\leq s\leq \frac{\pi}{4}+\frac{1}{2}$). 

\textbf{Example 3 continued.} In order to compute $\X$ as defined in \eqref{eq:mapFtoGX} we study the region $(X,Y)$ such that $Y=2s-X$ and $x_{1}(X)=x_{2}(Y)$. We have
\begin{equation*}
x_{2}(2s-X)=\begin{cases} 2s-X-1 & \text{if } X\leq 2s-1,\\ 0 &\text{if } 2s-1\leq X\leq 2s,\\ 2s-X, &\text{if } 2s\leq X. \end{cases}
\end{equation*}
If $2s\leq 0$, the two functions intersect at only one point $X=s$. The same holds for $2s-1\geq 1$, where they intersect at $X=s$.

If $0\leq 2s\leq 1$ then $x_{1}(X)=x_{2}(2s-X)$ for all $0\leq X\leq 2s$. Since $Y=2s-X$ this corresponds to straight line segments in the $(X,Y)$-plane with endpoints in $(0,2s)$ and $(2s,0)$. When we consider all $0\leq 2s\leq 1$ we therefore get a triangle in the $(X,Y)$-plane with corner points at $(0,0)$, $(1,0)$ and $(0,1)$.

If $0\leq 2s-1\leq 1$ then $x_{1}(X)=x_{2}(2s-X)$ for all $2s-1\leq X\leq 1$. Since $Y=2s-X$ this corresponds to straight line segments in the $(X,Y)$-plane with endpoints in $(2s-1,1)$ and $(1,2s-1)$. When we consider all $0\leq 2s-1\leq 1$ we therefore get a triangle in the $(X,Y)$-plane with corner points at $(1,0)$, $(1,1)$ and $(0,1)$.

Therefore, for $0\leq s\leq 1$ the region in the $(X,Y)$-plane where $x_{1}(X)=x_{2}(Y)$ for $Y=2s-X$ is a box with corners at $(0,0)$, $(1,0)$, $(1,1)$ and $(0,1)$. 

In principle we could pick any curve $(\X(s),\Y(s))$ in the box which satisfies Definition~\ref{def:curves}. This is because for any such curve in the box we have $x_{1}(\X(s))=x_{2}(\Y(s))$. From \eqref{eq:mapFtoGZ1} we would then have $\Z_{1}(s)=0$ so that $t(\X(s),\Y(s))=0$ for any such curve. Hence, $t(X,Y)=0$ for all $(X,Y)=[0,1]\times[0,1]$. In other words, time is equal to zero in the box. In order to proceed we must pick one of these curves, and from \eqref{eq:mapFtoGX} this curve consists of the straight line between $(0,0)$ and $(0,1)$, and the straight line between $(0,1)$ and $(1,1)$, that is,
\begin{equation*}
\X(s)=\begin{cases}
s & \text{if } s\leq 0,\\
0 & \text{if } 0\leq s\leq \frac{1}{2},\\
2s-1 & \text{if } \frac{1}{2}\leq s\leq 1,\\
s & \text{if } s\geq 1
\end{cases} \quad \text{and} \quad 
\Y(s)=\begin{cases}
s & \text{if } s\leq 0,\\
2s & \text{if } 0\leq s\leq \frac{1}{2},\\
1 & \text{if } \frac{1}{2}\leq s\leq 1,\\
s & \text{if } s\geq 1,
\end{cases}
\end{equation*}
see Figure \ref{fig:FigEx3i} and \ref{fig:FigEx3ii}.

\begin{figure}
	\centerline{\hbox{\includegraphics[width=10cm]{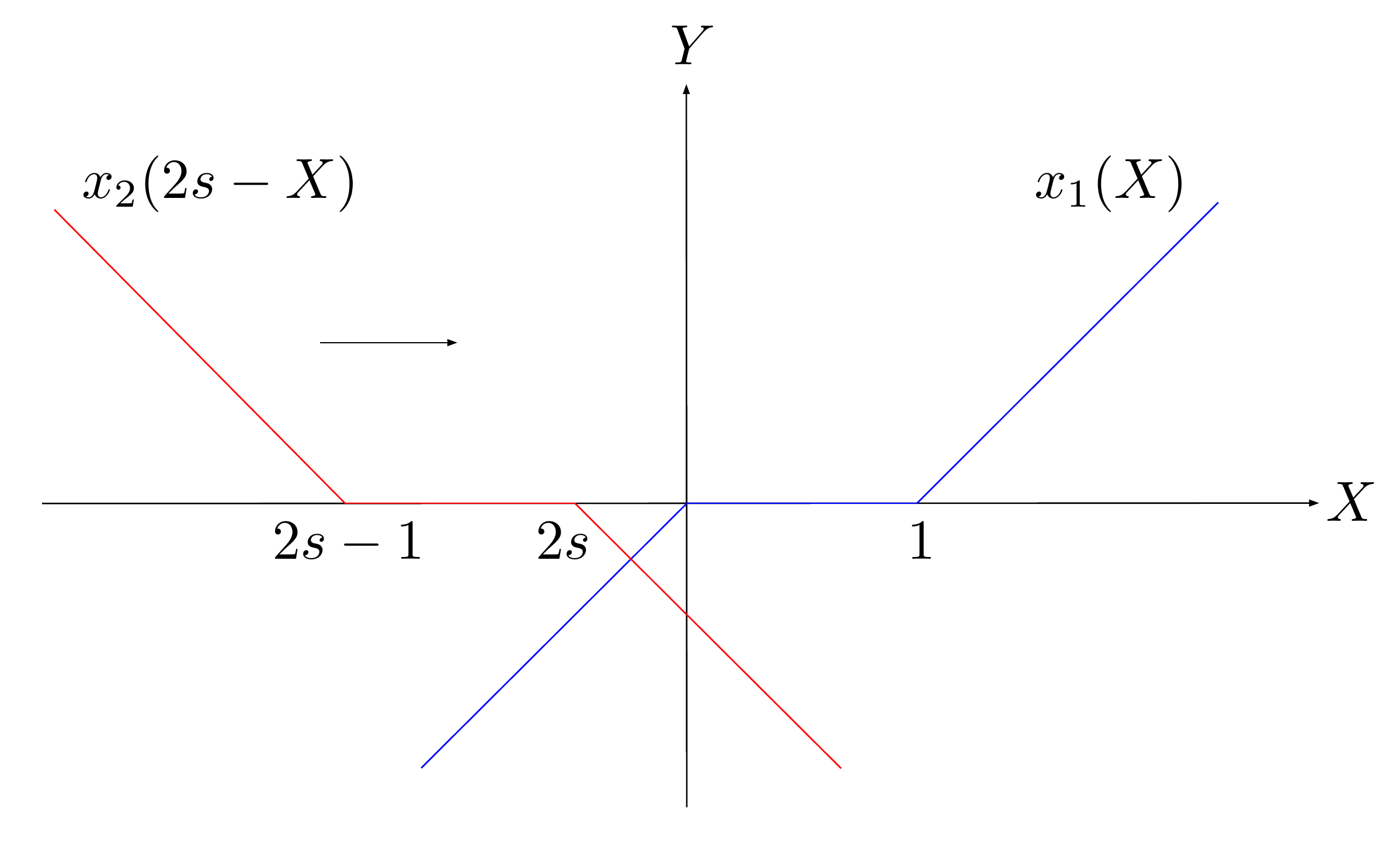}}}
	\caption{The functions $x_{1}(X)$ and $x_{2}(2s-X)$, from Example 3, for some $s<0$. The functions intersect at $X=\X(s)$. As $s$ increases, the graph of $x_{2}(2s-X)$ moves to the right. Eventually, the functions will intersect on intervals, which correspond to the box in Figure \ref{fig:FigEx3ii}.}
	\label{fig:FigEx3i}
\end{figure}

\begin{figure}
	\centerline{\hbox{\includegraphics[width=10cm]{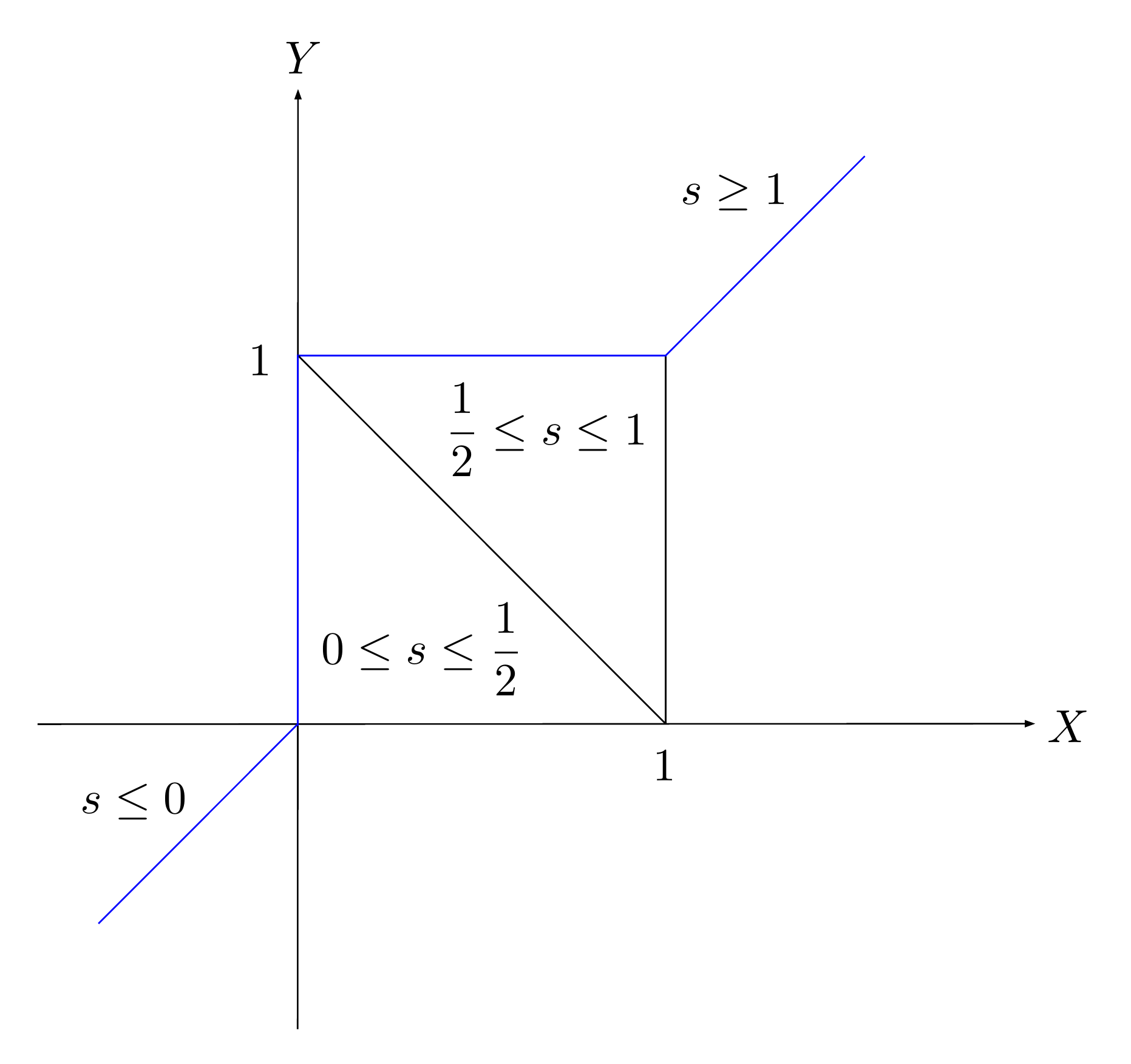}}}
	\caption{The set of all $X$ such that $x_{1}(X)=x_{2}(2s-X)$, from Example 3, for different values of $s$. The curve $(\X(s), \Y(s))$ is marked in blue.}
	\label{fig:FigEx3ii}
\end{figure}

We mention that if the measures are not discrete at the same point, we do not get boxes. Instead, the region in the $(X,Y)$-plane where $x_{1}(X)=x_{2}(Y)$ for $Y=2s-X$ is a curve consisting of the graph of strictly increasing functions and horizontal and vertical line segments.

\begin{proof}[Proof of the well-posedness of Definition \ref{def:mapfromFtoG0}]
	Let us verify that $(\X,\Y)$ belongs to $\C$. We first prove that $\X$ is nondecreasing. Let $s<\bar{s}$ and consider a sequence $X_{i}$ such that $\displaystyle\lim_{i\rightarrow \infty}X_{i}=\X(s)$ with $X_{i}<\X(s)$. By \eqref{eq:mapFtoGX} and since $x_{2}$ is nondecreasing, we have
	\begin{equation*}
	x_{1}(X_{i})<x_{2}(2s-X_{i})\leq x_{2}(2\bar{s}-X_{i}).
	\end{equation*}
	Hence, $X_{i}<\X(\bar{s})$. By letting $i$ tend to infinity, we conclude that $\X(s)\leq \X(\bar{s})$. By the continuity of $x_{1}$ and $x_{2}$, we obtain \eqref{eq:mapFtoG2}. We show that $\X$ is differentiable almost everywhere. We claim that $\X$ is Lipschitz continuous with Lipschitz constant bounded by two, that is,
	\begin{equation}
	\label{eq:Xlipschitz}
	|\X(\bar{s})-\X(s)|\leq 2|\bar{s}-s|.
	\end{equation}
	We may assume without loss of generality that $s<\bar{s}$. Assume that (\ref{eq:Xlipschitz}) does not hold, so that
	\begin{equation}
	\label{eq:Xlipschitzcontra}
	\X(\bar{s})-\X(s)>2(\bar{s}-s)
	\end{equation}
	for some $\bar{s}>s\in \mathbb{R}$. Thus, $\Y(s)>\Y(\bar{s})$. Then, since $x_{2}$ is nondecreasing,
	\begin{equation*}
	x_{1}(\X(s))=x_{2}(\Y(s))\geq x_{2}(\Y(\bar{s}))=x_{1}(\X(\bar{s})).
	\end{equation*}
	This implies that $x_{1}(\X(s))=x_{1}(\X(\bar{s}))$ because $x_{1}'\geq 0$ and $\X(s)<\X(\bar{s})$. Hence, $x_{1}$ is constant on $[\X(s),\X(\bar{s})]$. One proves similarly that $x_{2}$ is constant on $[\Y(\bar{s}),\Y(s)]$. Consider the point $(X,Y)$ given by $Y=\Y(s)$ and $X=2\bar{s}-\Y(s)$. We have
	\begin{equation*}
	\X(s)=2s-\Y(s)<X<2\bar{s}-\Y(\bar{s})=\X(\bar{s}),
	\end{equation*}
	so that $(X,Y)\in [\X(s),\X(\bar{s})]\times [\Y(\bar{s}),\Y(s)]$. It follows that
	$x_{1}(X)=x_{1}(\X(s))=x_{2}(\Y(s))=x_{2}(2\bar{s}-X)$ and $X<\X(\bar{s})$, which contradicts the definition of $\X$. Therefore, (\ref{eq:Xlipschitzcontra}) cannot hold and we have proved $(\ref{eq:Xlipschitz})$. Then, by Rademacher's theorem, $\X$ is differentiable almost everywhere. Let us prove that $\X-\id \in \Winf(\mathbb{R})$. This follows since
	\begin{equation*}
	\X(s)-s=\frac{1}{2}(\X(s)-\Y(s))=\frac{1}{2}(\X(s)-x_{1}(\X(s))+x_{2}(\Y(s))-\Y(s))
	\end{equation*} 
	and $x_{1}-\id,x_{2}-\id \in \Winf(\mathbb{R})$. Since $\dot{\X}\leq 2$, it follows that $\dot{\Y}=2-\dot{\X}\geq 0$. As above, one can show that $\Y-\id \in \Winf(\mathbb{R})$. Hence, $(\X,\Y)\in \C$. We prove that $||\Theta||_{\G}$ and $|||\Theta|||_{\G}$ are finite. In order to prove that $\Z_{3}\in L^{2}(\mathbb{R})$ we define the set
	\begin{equation*}
	B=\{s\in \mathbb{R} \ | \ \dot{\X}(s)\geq 1\}. 
	\end{equation*}
	Since $\dot{\X}+\dot{\Y}=2$, we have $\dot{\Y}>1$ on $B^{c}$. Thus,
	\begin{align*}
	\int_{\mathbb{R}}\Z_{3}^{2}(s)\,ds&=\int_{B}U_{1}^{2}(\X(s))\,ds+\int_{B^{c}}U_{2}^{2}(\Y(s))\,ds \\
	&\leq\int_{B}U_{1}^{2}(\X(s))\dot{\X}(s)\,ds+\int_{B^{c}}U_{2}^{2}(\Y(s))\dot{\Y}(s)\,ds \\
	&\leq \norm{U_{1}}_{L^{2}(\mathbb{R})}^2+\norm{U_{2}}_{L^{2}(\mathbb{R})}^2
	\end{align*}
	and $\Z_{3}\in L^{2}(\mathbb{R})$. The fact that $\Z_{3}^{a}\in L^{\infty}(\mathbb{R})$ follows from $U_{1},U_{2}\in L^2(\mathbb{R})\cap L^\infty(\mathbb{R})$. Next we show that the components of $\V^{a}$ belong to $L^{2}(\mathbb{R})\cap \Linf(\mathbb{R})$. By \eqref{eq:atriplet} and \eqref{eq:mapFtoGV1}, we have
	\begin{align*}
	|\V_{1}^{a}(X)|&=\bigg|\frac{1}{2c(U_{1}(X))}x_{1}'(X)-\frac{1}{2c(0)}\bigg|\\
	&=\bigg|\frac{1}{2c(U_{1}(X))}(x_{1}'(X)-1)+\frac{c(0)-c(U_{1}(X))}{2c(U_{1}(X))c(0)}\bigg|\\
	&\leq \frac{\kappa}{2}|x_{1}'(X)-1|+\frac{\kappa^{2}}{2}|c(0)-c(U_{1}(X))|\\
	&=\frac{\kappa}{2}|x_{1}'(X)-1|+\frac{\kappa^{2}}{2}\bigg| \int_{U_{1}(X)}^{0}c'(\tilde{U})\,d\tilde{U}\bigg|\\
	&\leq \frac{\kappa}{2}|x_{1}'(X)-1|+\frac{\kappa^{2}k_1}{2}|U_{1}(X)|,
	\end{align*}
	which implies that $\V_{1}^{a}$ belongs to $L^{2}(\mathbb{R})\cap \Linf(\mathbb{R})$, as  $x_{1}'-1,U_{1}\in L^{2}(\mathbb{R})\cap \Linf(\mathbb{R})$. We have
	\begin{equation*}
	\V_{2}^{a}=\V_{2}-\frac{1}{2}=\frac{1}{2}(x_{1}'-1),
	\end{equation*}
	so that $\V_{2}^{a}\in L^{2}(\mathbb{R})\cap \Linf(\mathbb{R})$. Since $\V_{3}^{a}=V_{1}$, $\V_{4}^{a}=J_{1}'$ and $\V_{5}^{a}=K_{1}'$, we conclude that all the components of $\V^{a}$ belong to $L^{2}(\mathbb{R})\cap \Linf(\mathbb{R})$. Similarly, one shows that the components of $\W^{a}$ belong to $L^{2}(\mathbb{R})\cap \Linf(\mathbb{R})$. Since $\p=H_{1}$ and $\q=H_{2}$, we have $\p,\q \in L^{2}(\mathbb{R})\cap \Linf(\mathbb{R})$. We have that $\Z_{2}^{a}\in \Linf(\mathbb{R})$ because
	\begin{equation*}
	\Z_{2}^{a}(s)=\Z_{2}(s)-s=x_{1}(\X(s))-s=x_{1}(\X(s))-\X(s)+\X(s)-s
	\end{equation*}
	and $x_{1}-\id,\X-\id \in \Linf(\mathbb{R})$. Since $\Z_{1}=0$, we have
	\begin{equation*}
	\Z_{1}^{a}(s)=-\frac{1}{c(0)}(\X(s)-s),
	\end{equation*}
	so that $\Z_{1}^{a}\in \Linf(\mathbb{R})$. From the relations $\Z_{4}^{a}=J_{1}(\X)+J_{2}(\Y)$ and $\Z_{5}^{a}=K_{1}(\X)+K_{2}(\Y)$, it follows from \eqref{eq:setF1} that they belong to $\Linf(\mathbb{R})$. To check that $\frac{1}{\V_{2}+\V_{4}}$ and $\frac{1}{\W_{2}+\W_{4}}$ are bounded, we need the following result.

	\begin{lemma}[{\cite[Lemma 3.2]{HolRay:07}}]
		\label{lemma:auxiliaryG}
		If $f\in G$ satisfies $||f-\id||_{\Winf(\mathbb{R})}+||f^{-1}-\id||_{\Winf(\mathbb{R})}\leq \alpha$ for some $\alpha\geq 0$, then $\frac{1}{1+\alpha}\leq f'\leq 1+\alpha$ almost everywhere. Conversely, if $f$ is absolutely continuous, $f-\id \in L^{\infty}(\mathbb{R})$, $f'-1\in L^{2}(\mathbb{R})$ and there exists $c\geq 1$ such that $\frac{1}{c}\leq f'\leq c$ almost everywhere, then $f$ belongs to $G$ and satisfies $||f-\id||_{\Winf(\mathbb{R})}+||f^{-1}-\id||_{\Winf(\mathbb{R})}\leq \alpha$,
		where $\alpha\geq 0$ only depends on $c$ and $||f-\id||_{L^{\infty}(\mathbb{R})}$.
	\end{lemma}

	Since $x_{1}+J_{1}\in G$, Lemma \ref{lemma:auxiliaryG} implies that for some $\alpha \geq 0$, $1/(1+\alpha)\leq x_{1}'+J_{1}'\leq 1+\alpha$ almost everywhere and it follows that $\frac{1}{\V_{2}+\V_{4}}\in L^{\infty}(\mathbb{R})$. Similarly, one can show that $\frac{1}{\W_{2}+\W_{4}}\in L^{\infty}(\mathbb{R})$. Hence, \eqref{eq:setGmember} holds. From \eqref{eq:mapFtoGV2}, \eqref{eq:mapFtoGV4} and \eqref{eq:setFrel1}, we can check that \eqref{eq:setGpositive} is satisfied. We verify that \eqref{eq:setGcomp} holds. By differentiating \eqref{eq:mapFtoG2}, we obtain $x_{1}'(\X)\dot{\X}=x_{2}'(\Y)\dot{\Y}$, which after using \eqref{eq:mapFtoGV2} yields \eqref{eq:setG0rel}. It follows that
	\begin{equation*}
	\dot{\Z}_{2}=\frac{1}{2}x_{1}'(\X)\dot{\X}+\frac{1}{2}x_{2}'(\Y)\dot{\Y}=\V_{2}(\X)\dot{\X}+\W_{2}(\Y)\dot{\Y}.
	\end{equation*}
	By \eqref{eq:mapFtoGZ1}, we have $\dot{\Z}_{1}=0$, and by \eqref{eq:mapFtoGV1}, \eqref{eq:mapFtoGZ2} and \eqref{eq:mapFtoGZ3}, we obtain 
	\begin{equation*}
	\V_{1}(\X)\dot{\X}+\W_{1}(\Y)\dot{\Y}=\frac{1}{2c(U_{1}(\X))}x_{1}'(\X)\dot{\X}-\frac{1}{2c(U_{2}(\Y))}x_{2}'(\Y)\dot{\Y}=0.
	\end{equation*}
	Using \eqref{eq:mapFtoGZ3}, \eqref{eq:setFrel7} and \eqref{eq:mapFtoGV3}, we find
	\begin{equation*}
	\dot{\Z}_{3}=\frac{1}{2}U_{1}'(\X)\dot{\X}+\frac{1}{2}U_{2}'(\Y)\dot{\Y}=V_{1}(\X)\dot{\X}+V_{2}(\Y)\dot{\Y}=\V_{3}(\X)\dot{\X}+\W_{3}(\Y)\dot{\Y}.
	\end{equation*}
	The relations for $\dot{\Z}_{4}$ and $\dot{\Z}_{5}$ in \eqref{eq:setGcomp} follow by differentiating \eqref{eq:mapFtoGZ4} and \eqref{eq:mapFtoGZ5}, respectively. The two identities in \eqref{eq:setGrel1} follow from \eqref{eq:mapFtoGV1} and \eqref{eq:mapFtoGV2}. Using \eqref{eq:setFrel2}, we can verify that \eqref{eq:setGrel2} holds. The last two identities \eqref{eq:setGrel3} and \eqref{eq:setGrel4} follows from \eqref{eq:setFrel3}. It remains to prove \eqref{eq:setGrel5}. Since $\Z_{4}(s)=J_{1}(\X(s))+J_{2}(\Y(s))$, $\displaystyle\lim_{s\rightarrow-\infty}\X(s)=-\infty$ and $\displaystyle\lim_{s\rightarrow-\infty}\Y(s)=-\infty$, it follows from \eqref{eq:setFrel5} that $\displaystyle\lim_{s\rightarrow -\infty}\Z_{4}(s)=0$, so that \eqref{eq:setGrel5} is satisfied.  
\end{proof}

\section{Existence of Solutions for the Equivalent System}

\subsection{Existence of Short-Range Solutions}

In the following we denote rectangular domains by 
\begin{equation*}
	\Omega=[X_{l},X_{r}]\times[Y_{l},Y_{r}]
\end{equation*}
and we set $s_{l}=\frac{1}{2}(X_{l}+Y_{l})$ and $s_{r}=\frac{1}{2}(X_{r}+Y_{r})$. We define curves in rectangular domains as follows.

\begin{definition}
Given $\Omega=[X_{l},X_{r}]\times[Y_{l},Y_{r}]$, we denote by $\C(\Omega)$ the set of curves in $\Omega$ parametrized by $(\X(s),\Y(s))$ with $s \in [s_{l},s_{r}]$ such that $(\X(s_{l}),\Y(s_{l}))=(X_{l},Y_{l})$, $(\X(s_{r}),\Y(s_{r}))=(X_{r}, Y_{r})$ and
\begin{subequations}
\begin{align}
	&\X-\id,\quad \Y-\id \in \Winf([s_{l},s_{r}]), \\
	&\dot{\X}\geq 0,\quad \dot{\Y}\geq 0, \\
	\label{eq:curvenormalization}
	&\frac{1}{2}(\X(s)+\Y(s))=s \quad \text{for all } s \in [s_{l},s_{r}]. 
\end{align}
\end{subequations}
We set
\begin{equation*}
	\norm{(\X,\Y)}_{\C(\Omega)}=\norm{\X-\id}_{\Linf([s_{l},s_{r}])}+\norm{\Y-\id}_{\Linf([s_{l},s_{r}])}.
\end{equation*}
\end{definition}

We introduce the counterpart of $\G$ on bounded domains, which we denote by $\G(\Omega)$.

\begin{definition}
Given $\Omega=[X_{l},X_{r}]\times[Y_{l},Y_{r}]$, we denote by $\G(\Omega)$ the set of all elements which consist of a curve $(\X,\Y)\in \C(\Omega)$, three vector-valued functions $\Z(s)$, $\V(X)$ and $\W(Y)$, and two functions $\p(X)$ and $\q(Y)$. We denote \ $\Theta=(\X,\Y,\Z,\V,\W,\p,\q)$ and set
\begin{equation*}
	\norm{\Theta}_{\G(\Omega)}^{2}=\norm{\Z_{3}}_{L^{2}([s_{l},s_{r}])}^{2}+\norm{\V^{a}}_{L^{2}([X_{l},X_{r}])}^{2}+\norm{\W^{a}}_{L^{2}([Y_{l},Y_{r}])}^{2}+\norm{\p}_{L^{2}([X_{l},X_{r}])}^{2}+\norm{\q}_{L^{2}([Y_{l},Y_{r}])}^{2}
\end{equation*}
and	
\begin{align*}
	\tnorm{\Theta}_{\G(\Omega)}&=\norm{(\X,\Y)}_{\C(\Omega)}+\norm{\frac{1}{\V_{2}+\V_{4}}}_{L^{\infty}([X_{l},X_{r}])}+\norm{\frac{1}{\W_{2}+\W_{4}}}_{L^{\infty}([Y_{l},Y_{r}])}\\
	&\quad+\norm{\Z^{a}}_{L^{\infty}([s_{l},s_{r}])}
	+\norm{\V^{a}}_{L^{\infty}([X_{l},X_{r}])}+\norm{\W^{a}}_{L^{\infty}([Y_{l},Y_{r}])}\\
	&\quad+\norm{\p}_{L^{\infty}([X_{l},X_{r}])}+\norm{\q}_{L^{\infty}([Y_{l},Y_{r}])}.
\end{align*}
The element $\Theta$ belongs\footnote{Note that condition (i) implies $\norm{\Theta}_{\G(\Omega)}<\infty$ because $\Omega$ is bounded.} to $\G(\Omega)$,
 if
\begin{enumerate}
\item [(i)]
\begin{equation*}
	 \tnorm{\Theta}_{\G(\Omega)}<\infty,
	 \end{equation*}
\item [(ii)]
\begin{equation*}
	\V_{2},\W_{2},\Z_{4},\V_{4},\W_{4}\geq 0,
\end{equation*}
\item [(iii)] for almost every $s\in\mathbb{R}$, we have
\begin{equation}
\label{eq:localGcomp}
	\dot{\Z}(s)=\V(\X(s))\dot{\X}(s)+\W(\Y(s))\dot{\Y}(s),
\end{equation}
\item [(iv)]
\begin{subequations}
\begin{align}
	\label{eq:setGlocrel1}
	&\V_{2}(\X(s))=c(\Z_{3}(s))\V_{1}(\X(s)), &&\W_{2}(\Y(s))=-c(\Z_{3}(s))\W_{1}(\Y(s)), \\
	&\V_{4}(\X(s))=c(\Z_{3}(s))\V_{5}(\X(s)), &&\W_{4}(\Y(s))=-c(\Z_{3}(s))\W_{5}(\Y(s))
\end{align}	
and	
\begin{align}
	\label{eq:setGlocrel3}
	2\V_{4}(\X(s))\V_{2}(\X(s))&=(c(\Z_{3}(s))\V_{3}(\X(s)))^{2}+c(\Z_{3}(s))\p^{2}(\X(s)), \\
	2\W_{4}(\Y(s))\W_{2}(\Y(s))&=(c(\Z_{3}(s))\W_{3}(\Y(s)))^{2}+c(\Z_{3}(s))\q^{2}(\Y(s)).
\end{align}
\end{subequations}
\end{enumerate}
\end{definition}

By definition we have for any $(\X,\Y,\Z,\V,\W,\p,\q)\in \G(\Omega)$ that the functions $\X$ and $\Y$ are nondecreasing. To any nondecreasing function one can associate its generalized inverse, a concept which is presented in, e.g., \cite{brenier:09}.

\begin{definition}
\label{def:GenInv}
Given $\Omega=[X_{l},X_{r}]\times[Y_{l},Y_{r}]$ and $(\X,\Y)\in \C(\Omega)$, we define the generalized inverse of $\X$ and $\Y$ as
\begin{align*}
	\alpha(X)&=\sup \{ s\in [s_{l},s_{r}] \ | \ \X(s)<X \} \quad \text{ for } X\in (X_{l},X_{r}], \\
	\beta(Y)&=\sup \{ s\in [s_{l},s_{r}] \ | \ \Y(s)<Y \} \quad \text{ for } Y\in (Y_{l},Y_{r}],
\end{align*}
respectively. We denote $\X^{-1}=\alpha$ and $\Y^{-1}=\beta$.
\end{definition}

The generalized inverse functions $\X^{-1}$ and $\Y^{-1}$ satisfy the following properties.

\begin{lemma}
\label{lemma:GenInv}
The functions $\X^{-1}$ and $\Y^{-1}$ are lower semicontinuous and nondecreasing. We have
\begin{subequations}
\begin{equation}
\label{eq:geninv1}
	\X \circ \X^{-1}=\emph{Id} \quad \text{and} \quad \Y \circ \Y^{-1}=\emph{Id},
\end{equation}
\begin{equation}
\label{eq:geninv2}
	\X^{-1} \circ \X(s)=s \text{ for any } s \text{ such that } \dot{\X}(s)>0
\end{equation}	
and
\begin{equation}
\label{eq:geninv3}
	\Y^{-1} \circ \Y(s)=s \text{ for any } s \text{ such that } \dot{\Y}(s)>0.
\end{equation}
\end{subequations}
\end{lemma}

We refer to \cite[Lemma 3]{HolRay:11} for a proof.

Now we define solutions of (\ref{eq:goveq}) on rectangular domains.
Consider the Banach spaces
\begin{align*}
	L^{\infty}_{X}(\Omega)&=L^{\infty}([Y_{l},Y_{r}],C([X_{l},X_{r}])), & L^{\infty}_{Y}(\Omega)&=L^{\infty}([X_{l},X_{r}],C([Y_{l},Y_{r}])), \\
	W^{1,\infty}_{X}(\Omega)&=L^{\infty}([Y_{l},Y_{r}],\Winf([X_{l},X_{r}])), & W^{1,\infty}_{Y}(\Omega)&=L^{\infty}([X_{l},X_{r}],\Winf([Y_{l},Y_{r}])).
\end{align*}
The corresponding norms for $f:\Omega\mapsto\mathbb{R}$ are defined as
\begin{align*}
	||f||_{L^{\infty}_{X}(\Omega)}&=\text{ess}\, \text{sup}_{Y\in [Y_{l},Y_{r}]} ||f(\cdot,Y)||_{\Linf([X_{l},X_{r}])},\\
	||f||_{L^{\infty}_{Y}(\Omega)}&=\text{ess}\, \text{sup}_{X\in [X_{l},X_{r}]} ||f(X,\cdot)||_{\Linf([Y_{l},Y_{r}])},\\
	||f||_{W^{1,\infty}_{X}(\Omega)}&=\text{ess}\, \text{sup}_{Y\in [Y_{l},Y_{r}]} ||f(\cdot,Y)||_{\Winf([X_{l},X_{r}])},\\
	||f||_{W^{1,\infty}_{Y}(\Omega)}&=\text{ess}\, \text{sup}_{X\in [X_{l},X_{r}]} ||f(X,\cdot)||_{\Winf([Y_{l},Y_{r}])}.
\end{align*}

We introduce the function $Z^{a}$, defined as
\begin{subequations}\label{eq:decayinfty}
\begin{align}
	Z_{1}^{a}(X,Y)&=Z_{1}(X,Y)-\frac{1}{2c(0)}(X-Y), \\
	Z_{2}^{a}(X,Y)&=Z_{2}(X,Y)-\frac{1}{2}(X+Y), \\
	Z_{i}^{a}(X,Y)&=Z_{i}(X,Y) \quad \text{for } i\in \{3,4,5\}
\end{align}
\end{subequations}
in order to conveniently express the decay of $Z$ at infinity in the diagonal direction. Although we are not yet concerned with the behavior at infinity, the notation will be useful when introducing global solutions.

\begin{definition}
\label{def:soln}
We say that $(Z,p,q)$ is a solution of (\ref{eq:goveq}) in $\Omega=[X_{l},X_{r}]\times[Y_{l},Y_{r}]$ if
\begin{enumerate}
	\item [(i)]
	\begin{aalign}
	\label{eq:soln1}
		&Z^{a}\in [\Winf(\Omega)]^{5}, \quad Z^{a}_{X}\in [W^{1,\infty}_{Y}(\Omega)]^{5}, \quad Z^{a}_{Y}\in [W^{1,\infty}_{X}(\Omega)]^{5}, \\ &\hspace{60pt}p\in W^{1,\infty}_{Y}(\Omega), \quad q\in W^{1,\infty}_{X}(\Omega),
	\end{aalign}
	\item [(ii)] for almost every $X\in [X_{l},X_{r}]$,
	\begin{equation}
	\label{eq:soln2}
		(Z_{X}(X,Y))_{Y}=F(Z)(Z_{X},Z_{Y})(X,Y),
	\end{equation}
	\item [(iii)] for almost every $Y\in [Y_{l},Y_{r}]$,
	\begin{equation}
	\label{eq:soln3}
	(Z_{Y}(X,Y))_{X}=F(Z)(Z_{X},Z_{Y})(X,Y),
	\end{equation}
	\item [(iv)] for almost every $X\in [X_{l},X_{r}]$,
	\begin{equation}
	\label{eq:soln4}
		p_{Y}(X,Y)=0,
	\end{equation}
	\item [(v)] for almost every $Y\in [Y_{l},Y_{r}]$,
	\begin{equation}
	\label{eq:soln5}
		q_{X}(X,Y)=0.
	\end{equation}
We say that $(Z,p,q)$ is a global solution of (\ref{eq:goveq}), if these conditions hold for any rectangular domain $\Omega$.
\end{enumerate}
\end{definition}

The following lemma, whose proof follows the same lines as the one of \cite[Lemma 4]{HolRay:11}, shows that the imposed regularity in Definition \ref{def:soln} is necessary to extract relevant data from a curve. Slightly abusing the notation, we denote 
\begin{equation}
\label{eq:invabuse}
	\X(Y)=\X \circ \Y^{-1}(Y) \quad \text{and} \quad \Y(X)=\Y \circ \X^{-1}(X).
\end{equation}

\begin{lemma}
Let $\Omega$ be a rectangular domain in $\mathbb{R}^{2}$ and assume that
\begin{align*}
	&Z^{a}\in [\Winf(\Omega)]^{5}, \quad Z^{a}_{X}\in [W^{1,\infty}_{Y}(\Omega)]^{5}, \quad Z^{a}_{Y}\in [W^{1,\infty}_{X}(\Omega)]^{5}, \\ &\hspace{60pt}p\in W^{1,\infty}_{Y}(\Omega), \quad q\in W^{1,\infty}_{X}(\Omega).
\end{align*}
Given a curve $(\X,\Y)\in \C(\Omega)$, let $(\Z,\V,\W,\p,\q)$ be defined as
\begin{equation*}
	\Z(s)=Z(\X(s),\Y(s)) \text{ for all } s\in \mathbb{R}
\end{equation*}
and
\begin{align*}
	\V(X)&=Z_{X}(X,\Y(X)) \text{ for a.e. } X\in \mathbb{R}, \\
	\W(Y)&=Z_{Y}(\X(Y),Y) \text{ for a.e. } Y\in \mathbb{R}, \\
	\p(X)&=p(X,\Y(X))  \text{ for a.e. } X\in \mathbb{R}, \\
	\q(Y)&=q(\X(Y),Y) \text{ for a.e. } Y\in \mathbb{R} \\
\end{align*}
or equivalently
\begin{align*}
	\V(\X(s))&=Z_{X}(\X(s),\Y(s)) \text{ for a.e. } s\in \mathbb{R} \text{ such that } \dot{\X}(s)>0, \\
	\W(\Y(s))&=Z_{Y}(\X(s),\Y(s)) \text{ for a.e. } s\in \mathbb{R} \text{ such that } \dot{\Y}(s)>0, \\
	\p(\X(s))&=p(\X(s),\Y(s)) \text{ for a.e. } s\in \mathbb{R} \text{ such that } \dot{\X}(s)>0, \\
	\q(\Y(s))&=q(\X(s),\Y(s)) \text{ for a.e. } s\in \mathbb{R} \text{ such that } \dot{\Y}(s)>0.
\end{align*}
Then $\Z,\V,\W,\p,\q\in L^{\infty}_{\emph{loc}}(\mathbb{R})$ and we denote $\Theta=(\X,\Y,\Z,\V,\W,\p,\q)$ by
\begin{equation*}
	(Z,p,q) \bullet (\X,\Y).
\end{equation*}
\end{lemma}

We now introduce the set $\H(\Omega)$ of all solutions of (\ref{eq:goveq}) on rectangular domains, which satisfy \eqref{eq:xJUpqRel}, \eqref{eq:xtJKRel}, and some additional constraints.
\begin{definition}
\label{def:initialDataH}
Given $\Omega=[X_{l},X_{r}]\times[Y_{l},Y_{r}]$, let $\H(\Omega)$ be the set of all solutions $(Z,p,q)$ to (\ref{eq:goveq}) in the sense of Definition \ref{def:soln} which satisfy the following properties
\begin{subequations}
\label{eqns:setH}
\begin{align}
	\label{eq:setH1}
	x_{X}&=c(U)t_{X}, & x_{Y}&=-c(U)t_{Y}, \\
	\label{eq:setH2}
	J_{X}&=c(U)K_{X}, & J_{Y}&=-c(U)K_{Y}, \\
	\label{eq:setH3}
	2J_{X}x_{X}&=(c(U)U_{X})^{2}+c(U)p^{2}, & 2J_{Y}x_{Y}&=(c(U)U_{Y})^{2}+c(U)q^{2}, \\
	\label{eq:setH4}
	x_{X}&\geq 0, & x_{Y}&\geq 0, \\
	\label{eq:setH5}
	J_{X}&\geq 0, & J_{Y}&\geq 0, \\
	\label{eq:setH6}
	x_{X}+J_{X}&>0, & x_{Y}+J_{Y}&>0. 
\end{align}  
\end{subequations}
\end{definition}

We have the following short-range existence theorem.

\begin{theorem}
\label{thm:solnsmallrect}
Given $\Omega=[X_l,X_r]\times [Y_l,Y_r]$, then for any $\Theta=(\X,\Y,\Z,\V,\W,\p,\q)\in \G(\Omega)$, there exists a unique solution $(Z,p,q)\in \H(\Omega)$ such that
\begin{equation}
\label{eq:shortrange}
	\Theta=(Z,p,q) \bullet (\X,\Y),
\end{equation}
if $s_{r}-s_{l}\leq 1/C(|||\Theta|||_{\G(\Omega)})$. Here $C$ denotes an increasing function dependent on $\Omega$, $\kappa$, $k_1$, and $k_2$.
\end{theorem}

\begin{proof}
We aim to use the Banach fixed-point theorem. Define $\mathcal{B}$ as the set of all elements $(Z_{h},Z_{v},V,W)$ such that
\begin{equation*}
	Z_{h}\in [L^{\infty}_{X}(\Omega)]^{5}, \quad Z_{v}\in [L^{\infty}_{Y}(\Omega)]^{5}, \quad V\in [L^{\infty}_{Y}(\Omega)]^{5}, \quad W\in [L^{\infty}_{X}(\Omega)]^{5}
\end{equation*}
and
\begin{equation}
\label{eq:Bnorm}
	\sum_{i=1}^{5}(||Z_{h,i}^{a}||_{L^{\infty}_{X}(\Omega)}+||Z_{v,i}^{a}||_{L^{\infty}_{Y}(\Omega)}+||V_{i}^{a}||_{L^{\infty}_{Y}(\Omega)}+||W_{i}^{a}||_{L^{\infty}_{X}(\Omega)})\leq 2|||\Theta|||_{\G(\Omega)},
\end{equation}
where we used the same notation for $Z_{h}$ and $Z_{v}$ as in (\ref{eq:decayinfty}) for $Z$. For $V$ and $W$ we used the same notation as in (\ref{eq:atriplet}) for $\V$ and $\W$, that is,
\begin{subequations}\label{eq:auxarel}
\begin{align}
	&V_{1}^{a}=V_{1}-\frac{1}{2c(0)}, \ &&W_{1}^{a}=W_{1}+\frac{1}{2c(0)}, \\
	&V_{2}^{a}=V_{2}-\frac{1}{2}, \ &&W_{2}^{a}=W_{2}-\frac{1}{2}, \\
	&V_{i}^{a}=V_{i}, \ &&W_{i}^{a}=W_{i} \hspace{50pt} \text{for } i \in \{3,4,5 \}.
\end{align}
\end{subequations}
As we shall see, for the fixed point, the functions $Z_{h}$ and $Z_{v}$ coincide and are equal to the solution $Z$, but we find it convenient to define both quantities in order to keep the symmetry of the problem with respect to the $X$ and $Y$ variables. For any $(Z_{h},Z_{v},V,W)\in \mathcal{B}$, we introduce the mapping $\mathcal{T}$ given by 
\begin{equation*}
	(\bar{Z_{h}},\bar{Z_{v}},\bar{V},\bar{W})=\mathcal{T}(Z_{h},Z_{v},V,W),
\end{equation*}
where
\begin{equation}
\label{eq:Zhbar}
	\bar{Z_{h}}(X,Y)=\Z(\Y^{-1}(Y))+\int_{\X(Y)}^{X} V(\tilde{X},Y)\,d\tilde{X}
\end{equation}
for a.e. $Y\in [Y_{l},Y_{r}]$ and all $X\in [X_{l},X_{r}]$,
\begin{equation}
\label{eq:Zvbar}
	\bar{Z_{v}}(X,Y)=\Z(\X^{-1}(X))+\int_{\Y(X)}^{Y} W(X,\tilde{Y})\,d\tilde{Y}
\end{equation}
for a.e. $X\in [X_{l},X_{r}]$ and all $Y\in [Y_{l},Y_{r}]$,	
\begin{equation}
\label{eq:Vbar}
	\bar{V}(X,Y)=\V(X)+\int_{\Y(X)}^{Y} F(Z_{h})(V,W)(X,\tilde{Y})\,d\tilde{Y}
\end{equation}
for a.e. $X\in [X_{l},X_{r}]$ and all $Y\in [Y_{l},Y_{r}]$,	
\begin{equation}
\label{eq:Wbar}
	\bar{W}(X,Y)=\W(Y)+\int_{\X(Y)}^{X} F(Z_{h})(V,W)(\tilde{X},Y)\,d\tilde{X}
\end{equation}	
for a.e. $Y\in [Y_{l},Y_{r}]$ and all $X\in [X_{l},X_{r}]$.

Let us compare this mapping with a solution $Z$ of \eqref{eq:goveqt}-\eqref{eq:goveqK} (in the sense of Definition \ref{def:soln}) which satisfies (\ref{eq:shortrange}). For any $(X,Y)\in \Omega$, we have
\begin{equation*}
	Z(X,\Y(s))=\Z(s)+\int_{\X(s)}^{X}Z_{X}(\tilde{X},\Y(s))\,d\tilde{X},
\end{equation*}
which after setting $s=\Y^{-1}(Y)$ yields	
\begin{equation*}
	Z(X,Y)=\Z(\Y^{-1}(Y))+\int_{\X(Y)}^{X}Z_{X}(\tilde{X},Y)\,d\tilde{X}.
\end{equation*}	
Similarly, we obtain, by setting $s=\X^{-1}(X)$,
\begin{equation*}
	Z(X,Y)=\Z(\X^{-1}(X))+\int_{\Y(X)}^{Y}Z_{Y}(X,\tilde{Y})\,d\tilde{Y}.
\end{equation*}
For a.e. $X\in [X_{l},X_{r}]$ and all $Y\in [Y_{l},Y_{r}]$, we have
\begin{equation*}
	Z_{X}(X,Y)=Z_{X}(X,\Y(X))+\int_{\Y(X)}^{Y}F(Z)(Z_{X},Z_{Y})(X,\tilde{Y})\,d\tilde{Y}
\end{equation*}
which by (\ref{eq:shortrange}) rewrites as
\begin{equation*}
	Z_{X}(X,Y)=\V(X)+\int_{\Y(X)}^{Y}F(Z)(Z_{X},Z_{Y})(X,\tilde{Y})\,d\tilde{Y}.
\end{equation*}	
By a similar argument, we get
\begin{equation*}
	Z_{Y}(X,Y)=\W(Y)+\int_{\X(Y)}^{X}F(Z)(Z_{X},Z_{Y})(\tilde{X},Y)\,d\tilde{X}
\end{equation*}
for a.e. $Y\in [Y_{l},Y_{r}]$ and all $X\in [X_{l},X_{r}]$. Thus, if $Z$ is a solution of \eqref{eq:goveqt}-\eqref{eq:goveqK} which satisfies \eqref{eq:shortrange}, then $(Z,Z,Z_{X},Z_{Y})$ is a fixed point of $\mathcal{T}$. 

Let us show that the mapping $\mathcal{T}$ maps $\mathcal{B}$ into $\mathcal{B}$. To begin with we derive some estimates. We set $\delta=s_{r}-s_{l}$. By \eqref{eq:geninv1} and since $0\leq \dot{\X}\leq 2$, we have	
\begin{equation}
\label{eq:deltabound1}
	|X-\X(Y)|=|\X(\alpha(X))-\X(\beta(Y))|\leq \X(s_{r})-\X(s_{l})\leq 2(s_{r}-s_{l})=2\delta
\end{equation}
and similarly, we get
\begin{equation}
\label{eq:deltabound2}
	|Y-\Y(X)|\leq 2(s_{r}-s_{l})=2\delta.
\end{equation}
For the first component of $\bar{Z}_{h}^{a}$, we have
\begin{align*}
	\bar{Z}_{h,1}^{a}(X,Y)&=\bar{Z}_{h,1}(X,Y)-\frac{1}{2c(0)}(X-Y) \quad \text{by } \eqref{eq:decayinfty}\\
	&=\Z_{1}(\Y^{-1}(Y))+\int_{\X(Y)}^{X} V_{1}(\tilde{X},Y)\,d\tilde{X}-\frac{1}{2c(0)}(X-Y) \quad \text{by } \eqref{eq:Zhbar} \\
	&=\Z_{1}^{a}(\Y^{-1}(Y))+\frac{1}{c(0)}(\X(Y)-\Y^{-1}(Y)) \quad \text{by } \eqref{eq:atriplet} \text{ and } \eqref{eq:auxarel}\\
	&\quad+\int_{\X(Y)}^{X} \bigg(V_{1}^{a}(\tilde{X},Y)+\frac{1}{2c(0)}\bigg)\,d\tilde{X}-\frac{1}{2c(0)}(X-Y)\\	
	&=\Z_{1}^{a}(\Y^{-1}(Y))+\int_{\X(Y)}^{X} V_{1}^{a}(\tilde{X},Y)\,d\tilde{X}\\
	&\quad+\frac{1}{c(0)}\bigg(\frac{1}{2}\X(Y)+\frac{1}{2}Y-\Y^{-1}(Y)\bigg).
\end{align*}
From \eqref{eq:curvenormalization}, we obtain
\begin{equation*}
	\frac{1}{2}\X(Y)+\frac{1}{2}Y-\Y^{-1}(Y)=\frac{1}{2}\X(Y)+\frac{1}{2}Y-\frac{1}{2}(\X(Y)+Y)=0.
\end{equation*}
Hence, by \eqref{eq:deltabound1},
\begin{equation*}
	||\bar{Z}_{h,1}^{a}||_{L^{\infty}_{X}(\Omega)}\leq ||\Z_{1}^{a}||_{L^{\infty}([s_{l},s_{r}])}+2\delta||V_{1}^{a}||_{L^{\infty}_{Y}(\Omega)}.
\end{equation*}
We proceed similarly for the second component of $\bar{Z}_{h}^{a}$ and get
\begin{align*}
	\bar{Z}_{h,2}^{a}(X,Y)&=\bar{Z}_{h,2}(X,Y)-\frac{1}{2}(X+Y) \\
	&=\Z_{2}(\Y^{-1}(Y))+\int_{\X(Y)}^{X}V_{2}(\tilde{X},Y)\,d\tilde{X}-\frac{1}{2}(X+Y) \\
	&=\Z_{2}^{a}(\Y^{-1}(Y))+\Y^{-1}(Y)+\int_{\X(Y)}^{X} \bigg(V_{2}^{a}(\tilde{X},Y)+\frac{1}{2}\bigg)\,d\tilde{X}-\frac{1}{2}(X+Y)\\
	&=\Z_{2}^{a}(\Y^{-1}(Y))+\int_{\X(Y)}^{X}V_{2}^{a}(\tilde{X},Y)\,d\tilde{X}+\Y^{-1}(Y)-\frac{1}{2}\X(Y)-\frac{1}{2}Y.
\end{align*} 
By \eqref{eq:curvenormalization}, we have
\begin{align*}
	\Y^{-1}(Y)-\frac{1}{2}\X(Y)-\frac{1}{2}Y=\frac{1}{2}(\X(Y)+Y)-\frac{1}{2}\X(Y)-\frac{1}{2}Y=0,
\end{align*}
which leads to the estimate
\begin{equation*}
	||\bar{Z}_{h,2}^{a}||_{L^{\infty}_{X}(\Omega)}\leq ||\Z_{2}^{a}||_{L^{\infty}([s_{l},s_{r}])}+2\delta||V_{2}^{a}||_{L^{\infty}_{Y}(\Omega)}.
\end{equation*}
For $i\in\{3,4,5\}$, we have
\begin{equation*}
	||\bar{Z}_{h,i}^{a}||_{L^{\infty}_{X}(\Omega)}\leq ||\Z_{i}^{a}||_{L^{\infty}([s_{l},s_{r}])}+2\delta||V_{i}^{a}||_{L^{\infty}_{Y}(\Omega)}.
\end{equation*}
Similar bounds hold for the components of $\bar{Z}_{v}^{a}$. Let us consider $\bar{V}_{1}^{a}$. By using the governing equations (\ref{eq:goveq}), we obtain
\begin{align*}
	\bar{V}_{1}^{a}(X,Y)&=\bar{V}_{1}(X,Y)-\frac{1}{2c(0)} \\
	&=\V_{1}(X)-\frac{1}{2c(0)}+\int_{\Y(X)}^{Y} F_{1}(Z_{h})(V,W)(X,\tilde{Y})\,d\tilde{Y} \\
	&=\V_{1}^{a}(X)-\int_{\Y(X)}^{Y} \bigg( \frac{c'(Z_{h,3})}{2c(Z_{h,3})}(V_{3}W_{1}+W_{3}V_{1})\bigg)(X,\tilde{Y})\,d\tilde{Y}.
\end{align*} 
We have
\begin{align*}
	|V_{3}W_{1}+W_{3}V_{1}|&=\bigg|V_{3}^{a}W_{1}^{a}+W_{3}^{a}V_{1}^{a}-\frac{1}{2c(0)}(W_{3}^{a}-V_{3}^{a})\bigg|\\
	&\leq ||V_{3}^{a}||_{L^{\infty}_{Y}(\Omega)}||W_{1}^{a}||_{L^{\infty}_{X}(\Omega)}+||W_{3}^{a}||_{L^{\infty}_{X}(\Omega)}||V_{1}^{a}||_{L^{\infty}_{Y}(\Omega)}\\
	&\quad+\frac{\kappa}{2}\big(||W_{3}^{a}||_{L^{\infty}_{X}(\Omega)}+||V_{3}^{a}||_{L^{\infty}_{Y}(\Omega)}\big)\\
	&\leq 4|||\Theta|||_{\G(\Omega)}^{2}+\kappa |||\Theta|||_{\G(\Omega)} \quad \text{by } \eqref{eq:Bnorm}.
\end{align*}
Hence,
\begin{equation*}
	||\bar{V}_{1}^{a}||_{L^{\infty}_{Y}(\Omega)}\leq ||\V_{1}^{a}||_{L^{\infty}([X_{l},X_{r}])}+\delta \kappa k_1 \big(4|||\Theta|||_{\G(\Omega)}^{2}+\kappa|||\Theta|||_{\G(\Omega)}\big).
\end{equation*}
After doing the same for the other components of $\bar{V}$ and $\bar{W}$, we get
\begin{equation*}
	\sum_{i=1}^{5}(||\bar{Z}_{h,i}^{a}||_{L^{\infty}_{X}(\Omega)}+||\bar{Z}_{v,i}^{a}||_{L^{\infty}_{Y}(\Omega)}+||\bar{V}_{i}^{a}||_{L^{\infty}_{Y}(\Omega)}+||\bar{W}_{i}^{a}||_{L^{\infty}_{X}(\Omega)})
	\leq |||\Theta|||_{\G(\Omega)}+\delta C_1(|||\Theta|||_{\G(\Omega)})
\end{equation*}
for some increasing function $C_1(|||\Theta|||_{\G(\Omega)})$. Hence, by setting $\delta$ small enough, the mapping $\mathcal{T}$ maps $\B$ into $\B$. 

It remains to show that $\mathcal{T}$ is a contraction. Let $(Z_{h},Z_{v},V,W)$ and $(Z_{h}',Z_{v}',V',W')$ belong to $\B$, and fix $(X,Y)\in \Omega$. For the first component of $F$, we have 
\begin{aalign}
\label{eq:V1contracting}
	&|\bar{V}_{1}^{a}(X,Y)-(\bar{V}_{1}')^{a}(X,Y)|\\
	&=|\bar{V_{1}}(X,Y)-\bar{V_{1}'}(X,Y)|\\
	&\leq\int_{\Y(X)}^{Y}|F_{1}(Z_{h})(V,W)-F_{1}(Z_{h}')(V',W')|(X,\tilde{Y})\,d\tilde{Y}
\end{aalign} 
and by \eqref{eq:goveq}, we get
\begin{align*}
	&F_{1}(Z_{h})(V,W)-F_{1}(Z_{h}')(V',W')\\
	&=-\frac{c'(Z_{h,3})}{2c(Z_{h,3})}(V_{3}W_{1}+W_{3}V_{1})+\frac{c'(Z_{h,3}')}{2c(Z_{h,3}')}(V_{3}'W_{1}'+W_{3}'V_{1}')\\
	&=-\frac{c'(Z_{h,3}^{a})}{2c(Z_{h,3}^{a})}\bigg(V_{3}^{a}\bigg(W_{1}^{a}-\frac{1}{2c(0)}\bigg)+W_{3}^{a}\bigg(V_{1}^{a}+\frac{1}{2c(0)}\bigg)\bigg)\\
	&\quad+\frac{c'((Z_{h,3}')^{a})}{2c((Z_{h,3}')^{a})}\bigg((V_{3}')^{a}\bigg((W_{1}')^{a}-\frac{1}{2c(0)}\bigg)+(W_{3}')^{a}\bigg((V_{1}')^{a}+\frac{1}{2c(0)}\bigg)\bigg)\\
	&=\frac{c'((Z_{h,3}')^{a})}{2c((Z_{h,3}')^{a})}\bigg( (V_{3}')^{a}((W_{1}')^{a}-W_{1}^{a})+W_{1}^{a}((V_{3}')^{a}-V_{3}^{a})\\
	&\hspace{80pt}+(W_{3}')^{a}((V_{1}')^{a}-V_{1}^{a})+V_{1}^{a}((W_{3}')^{a}-W_{3}^{a})\\
	&\hspace{80pt}-\frac{1}{2c(0)}((V_{3}')^{a}-V_{3}^{a})+\frac{1}{2c(0)}((W_{3}')^{a}-W_{3}^{a})\bigg)\\
	&\quad+\frac{1}{2}\bigg(V_{3}^{a}W_{1}^{a}+W_{3}^{a}V_{1}^{a}-\frac{1}{2c(0)}V_{3}^{a}+\frac{1}{2c(0)}W_{3}^{a}\bigg)\bigg(\frac{c'((Z_{h,3}')^{a})}{c((Z_{h,3}')^{a})}-\frac{c'(Z_{h,3}^{a})}{c(Z_{h,3}^{a})}\bigg).
\end{align*}	
Since
\begin{equation*}
	\frac{c'((Z_{h,3}')^{a})}{c((Z_{h,3}')^{a})}-\frac{c'(Z_{h,3}^{a})}{c(Z_{h,3}^{a})}=
	\int_{Z_{h,3}^{a}}^{(Z_{h,3}')^{a}}\bigg(\frac{c''c-(c')^{2}}{c^{2}}\bigg)(Z)\,dZ,
\end{equation*}
this leads to the estimate
\begin{align*}
	&|F_{1}(Z_{h})(V,W)-F_{1}(Z_{h}')(V',W')|\\
	&\leq \frac{\kappa k_{1}}{2}\bigg(2|||\Theta|||_{\G(\Omega)}\Big(||(W_{1}')^{a}-W_{1}^{a}||_{L^{\infty}_{X}(\Omega)}+||(V_{3}')^{a}-V_{3}^{a}||_{L^{\infty}_{Y}(\Omega)}\\
	&\hspace{110pt}+||(V_{1}')^{a}-V_{1}^{a}||_{L^{\infty}_{Y}(\Omega)}+||(W_{3}')^{a}-W_{3}^{a}||_{L^{\infty}_{X}(\Omega)}\Big)\\
	&\hspace{50pt}+\frac{\kappa}{2}\Big(||(V_{3}')^{a}-V_{3}^{a}||_{L^{\infty}_{Y}(\Omega)}+||(W_{3}')^{a}-W_{3}^{a}||_{L^{\infty}_{X}(\Omega)}\Big)\bigg)\\
	&\quad+\frac{1}{2}\Big(4|||\Theta|||_{\G(\Omega)}^{2}+\kappa|||\Theta|||_{\G(\Omega)}\Big)||(Z_{h,3}')^{a}-Z_{h,3}^{a}||_{L^{\infty}_{X}(\Omega)}\\
	&\quad \quad \times (\kappa k_{2}+(\kappa k_{1})^{2}),
\end{align*}
where we used \eqref{eq:Bnorm}. We insert this into \eqref{eq:V1contracting} and obtain
\begin{align*}
	&|\bar{V}_{1}^{a}(X,Y)-(\bar{V}_{1}')^{a}(X,Y)|\\
	&\leq \delta K\big(||W^{a}-(W')^{a}||_{L^{\infty}_{X}(\Omega)}+||V^{a}-(V')^{a}||_{L^{\infty}_{Y}(\Omega)}+||Z_{h}^{a}-(Z_{h}')^{a}||_{L^{\infty}_{X}(\Omega)}\big),
\end{align*}
where $K$ depends on $|||\Theta|||_{\G(\Omega)}$, $\kappa$, $k_{1}$ and $k_{2}$ because of \eqref{eq:cassumption} and \eqref{eq:cderassumption}. Following the same lines, we obtain
\begin{align*}
	&||\bar{Z}_{h}^{a}-(\bar{Z}_{h}')^{a}||_{L^{\infty}_{X}(\Omega)}+||\bar{Z}_{v}^{a}-(\bar{Z}_{v}')^{a}||_{L^{\infty}_{Y}(\Omega)}+||\bar{V}^{a}-(\bar{V}')^{a}||_{L^{\infty}_{Y}(\Omega)}+||\bar{W}^{a}-(\bar{W}')^{a}||_{L^{\infty}_{X}(\Omega)} \\
	&\leq \delta C_2 \big( ||Z_{h}^{a}-(Z_{h}')^{a}||_{L^{\infty}_{X}(\Omega)}+||Z_{v}^{a}-(Z_{v}')^{a}||_{L^{\infty}_{Y}(\Omega)}\\
	&\hspace{40pt}+||V^{a}-(V')^{a}||_{L^{\infty}_{Y}(\Omega)}+||W^{a}-(W')^{a}||_{L^{\infty}_{X}(\Omega)}\big)
\end{align*}
for some increasing function $C_2$ depending on $|||\Theta|||_{\G(\Omega)}$, $\Omega$, $\kappa$, $k_{1}$ and $k_{2}$. By setting $\delta>0$ so small that $\delta C_2<1$, we conclude that $\mathcal{T}$ is a contraction. Hence, $\mathcal{T}$ admits a unique fixed point that we denote $(Z_{h},Z_{v},V,W)$. 

Let us prove that $Z_{h}=Z_{v}$, $V=Z_{X}$ and $W=Z_{Y}$. We denote by $\N_{X}$ the set of points $X\in [X_{l},X_{r}]$ for which (\ref{eq:Zvbar}) and (\ref{eq:Vbar}) hold. Similarly, we denote by $\N_{Y}$ the set of points $Y\in [Y_{l},Y_{r}]$ for which (\ref{eq:Zhbar}) and (\ref{eq:Wbar}) hold. We have $\text{meas}([X_{l},X_{r}]\setminus \N_{X})=\text{meas}([Y_{l},Y_{r}]\setminus \N_{Y})=0$ such that $\text{meas}(\Omega \setminus \N_{X}\times \N_{Y})=0$. For any $(X,Y)\in \N_{X}\times \N_{Y}$, we have
\begin{aalign}
\label{eq:ZhZvdiff}
	Z_{h}(X,Y)-Z_{v}(X,Y)&=\Z(\Y^{-1}(Y))+\int_{\X(Y)}^{X} V(\tilde{X},Y)\,d\tilde{X} \\
	&\quad-\Z(\X^{-1}(X))-\int_{\Y(X)}^{Y} W(X,\tilde{Y})\,d\tilde{Y}.
\end{aalign}
From (\ref{eq:Vbar}) and (\ref{eq:Wbar}), we obtain
\begin{align*}
	&\int_{\X(Y)}^{X} V(\tilde{X},Y)\,d\tilde{X}-\int_{\Y(X)}^{Y} W(X,\tilde{Y})\,d\tilde{Y} \\
	&=\int_{\X(Y)}^{X}\bigg(\V(\tilde{X})+\int_{\Y(\tilde{X})}^{Y} F(Z_{h})(V,W)(\tilde{X},\tilde{Y})\,d\tilde{Y}\bigg)\,d\tilde{X}\\
	&\quad -\int_{\Y(X)}^{Y}\bigg(\W(\tilde{Y})+\int_{\X(\tilde{Y})}^{X} F(Z_{h})(V,W)(\tilde{X},\tilde{Y})\,d\tilde{X}\bigg)\,d\tilde{Y}\\	
	&=\int_{\X(Y)}^{X} \V(\tilde{X})\,d\tilde{X}-\int_{\Y(X)}^{Y} \W(\tilde{Y})\,d\tilde{Y} \\
	&=\int_{\X\circ \Y^{-1}(Y)}^{\X\circ \X^{-1}(X)} \V(\tilde{X})\,d\tilde{X}-\int_{\Y\circ \X^{-1}(X)}^{\Y\circ \Y^{-1}(Y)} \W(\tilde{Y})\,d\tilde{Y} \quad \text{by } (\ref{eq:geninv1}) \text{ and } (\ref{eq:invabuse}) \\
	&=-\int_{\X^{-1}(X)}^{\Y^{-1}(Y)} \big(\V(\X(s))\dot{\X}(s)+\W(\Y(s))\dot{\Y}(s)\big)\, ds \quad \text{by a change of variables} \\
	&=-\int_{\X^{-1}(X)}^{\Y^{-1}(Y)} \dot{\Z}(s)\,ds \quad \text{by } (\ref{eq:localGcomp}) \\
	&=\Z(\X^{-1}(X))-\Z(\Y^{-1}(Y)).
\end{align*}	
By inserting this into (\ref{eq:ZhZvdiff}), we conclude that $Z_{h}(X,Y)=Z_{v}(X,Y)$ for all $(X,Y)\in \N_{X}\times \N_{Y}$, that is, almost everywhere in $\Omega$. 

We denote $Z=Z_{h}=Z_{v}$. The function $Z$ is only defined in $\N_{X}\times \N_{Y}$. We define $Z(X,Y)$ for all $(X,Y)\in\Omega$ by setting 
\begin{equation}
\label{eq:Zdangerzone}
	Z(X,Y)=\displaystyle\lim_{n\rightarrow\infty}Z(X_{n},Y_{n}),
\end{equation}
where $(X_{n},Y_{n})$ is a squence in $\N_{X}\times \N_{Y}$ such that $(X_{n},Y_{n})\rightarrow(X,Y)$ as $n\to \infty$. Let us prove that this is well-defined. First we show that $Z$ is Lipschitz continuous in $\N_{X}\times \N_{Y}$. Let $(X_{1},Y_{1}),(X_{2},Y_{2})\in\N_{X}\times \N_{Y}$. By \eqref{eq:Zhbar}, we have
\begin{align*}
	Z(X_{2},Y_{2})-Z(X_{1},Y_{2})&=\int_{\X(Y_{2})}^{X_{2}} V(\tilde{X},Y_{2})\,d\tilde{X}-\int_{\X(Y_{2})}^{X_{1}} V(\tilde{X},Y_{2})\,d\tilde{X}\\
	&=\int_{X_{1}}^{X_{2}} V(\tilde{X},Y_{2})\,d\tilde{X} 
\end{align*}
and, from \eqref{eq:Zvbar}, we get
\begin{align*}
	Z(X_{1},Y_{2})-Z(X_{1},Y_{1})&=\int_{\Y(X_{1})}^{Y_{2}} W(X_{1},\tilde{Y})\,d\tilde{Y}-\int_{\Y(X_{1})}^{Y_{1}} W(X_{1},\tilde{Y})\,d\tilde{Y}\\
	&=\int_{Y_{1}}^{Y_{2}} W(X_{1},\tilde{Y})\,d\tilde{Y}.
\end{align*}
This implies that
\begin{aalign}
\label{eq:ZXdiffandZYdiff}
	&|Z(X_{2},Y_{2})-Z(X_{1},Y_{1})|\\
	&\leq|Z(X_{2},Y_{2})-Z(X_{1},Y_{2})|+|Z(X_{1},Y_{2})-Z(X_{1},Y_{1})|\\
	&\leq (||V||_{L_{Y}^{\infty}(\Omega)}+||W||_{L_{X}^{\infty}(\Omega)})(|X_{2}-X_{1}|+|Y_{2}-Y_{1}|)\\
	&\leq C(|X_{2}-X_{1}|+|Y_{2}-Y_{1}|),
\end{aalign}
where $C$ only depends on $|||\Theta|||_{\G(\Omega)}$, and we conclude that $Z$ is Lipschitz continuous in $\N_{X}\times \N_{Y}$.
 
For any $(X,Y)\in\Omega$, there exists a sequence $(X_{n},Y_{n})\in \N_{X}\times\N_{Y}$ such that $(X_{n},Y_{n})\rightarrow(X,Y)$ as $n\rightarrow\infty$, since $\text{meas}(\Omega\setminus \N_{X}\times\N_{Y})=0$. From \eqref{eq:ZXdiffandZYdiff}, we have
\begin{equation*}
	|Z(X_{n},Y_{n})-Z(X_{m},Y_{m})|\leq C(|X_{n}-X_{m}|+|Y_{n}-Y_{m}|),
\end{equation*}
so that the limit $\displaystyle\lim_{n\rightarrow\infty}Z(X_{n},Y_{n})$ exists. We claim that the limit is independent of the particular choice of the sequence in $\N_{X}\times\N_{Y}$ converging to $(X,Y)$. Let $(\bar{X}_{m},\bar{Y}_{m})$ be another sequence in $\N_{X}\times\N_{Y}$ such that $(\bar{X}_{m},\bar{Y}_{m})\rightarrow(X,Y)$ as $n\rightarrow\infty$. By \eqref{eq:ZXdiffandZYdiff}, we have
\begin{align*}
	|Z(X_{n},Y_{n})-Z(\bar{X}_{m},\bar{Y}_{m})|&\leq C(|X_{n}-\bar{X}_{m}|+|Y_{n}-\bar{Y}_{m}|)\\
	&\leq C(|X_{n}-X|+|X-\bar{X}_{m}|+|Y_{n}-Y|+|Y-\bar{Y}_{m}|)
\end{align*}
which implies that $\displaystyle\lim_{n\rightarrow\infty}Z(X_{n},Y_{n})=\displaystyle\lim_{m\rightarrow\infty}Z(\bar{X}_{m},\bar{Y}_{m})$. This proves the claim and \eqref{eq:Zdangerzone} is well-defined. It now follows from \eqref{eq:ZXdiffandZYdiff} that $Z$ is Lipschitz continuous in $\Omega$. Indeed, for any $(X,Y),(\bar{X},\bar{Y})\in\Omega$, there exist sequences $(X_{n},Y_{n}),(\bar{X}_{n},\bar{Y}_{n})\in\N_{X}\times\N_{Y}$, such that $(X_{n},Y_{n})\rightarrow(X,Y)$ and $(\bar{X}_{n},\bar{Y}_{n})\rightarrow(\bar{X},\bar{Y})$, which yields
\begin{align*}
	|Z(X,Y)-Z(\bar{X},\bar{Y})|&=\displaystyle\lim_{n\rightarrow\infty}|Z(X_{n},Y_{n})-Z(\bar{X}_{n},\bar{Y}_{n})|\\
	&\leq C\displaystyle\lim_{n\rightarrow\infty} (|X_{n}-\bar{X}_{n}|+|Y_{n}-\bar{Y}_{n}|)\\
	&=C(|X-\bar{X}|+|Y-\bar{Y}|).
\end{align*}
Thus, $Z$ is Lipschitz continuous and therefore differentiable almost everywhere in $\Omega$. It follows from
(\ref{eq:Zhbar}) and (\ref{eq:Zvbar}), that
\begin{equation*}
	Z_{X}(X,Y)=V(X,Y) \quad \text{and} \quad Z_{Y}(X,Y)=W(X,Y)
\end{equation*}
for almost every $X\in [X_{l},X_{r}]$ and $Y\in [Y_{l},Y_{r}]$. Now we define the solutions $p$ and $q$ of \eqref{eq:goveqp} and \eqref{eq:goveqq}, respectively. For almost every $X\in [X_{l},X_{r}]$ and all $Y\in [Y_{l},Y_{r}]$, we set
\begin{equation}
\label{eq:pdef}
	p(X,Y)=\p(X)
\end{equation}
and for almost every $Y\in [Y_{l},Y_{r}]$ and all $X\in [X_{l},X_{r}]$, we set
\begin{equation}
\label{eq:qdef}
	q(X,Y)=\q(Y).
\end{equation}
Let us check that \eqref{eq:soln1} is satisfied. Since $Z$ is Lipschitz continuous, it follows that all the components of $Z$ belong to $\Winf(\Omega)$. As in the argument above where we showed that $Z$ is Lipschitz continuous, one can show that for almost every $X$, $Z_{X}$ is Lipschitz continuous with respect to $Y$, and that for almost every $Y$, $Z_{Y}$ is Lipschitz continuous with respect to $X$. It implies that $Z_{XY}(X,\cdot)\in \Linf([Y_{l},Y_{r}])$ for almost every $X\in [X_{l},X_{r}]$ and $Z_{YX}(\cdot,Y)\in \Linf([X_{l},X_{r}])$ for almost every $Y\in [X_{l},X_{r}]$, which in turn implies that	
\begin{equation*}
\text{ess}\, \text{sup}_{X\in [X_{l},X_{r}]} ||Z_{X}(X,\cdot)||_{\Winf([Y_{l},Y_{r}])} 
\end{equation*}
and 
\begin{equation*}
\text{ess}\, \text{sup}_{Y\in [Y_{l},Y_{r}]} ||Z_{Y}(\cdot,Y)||_{\Winf([X_{l},X_{r}])} 
\end{equation*}
are bounded by a constant. Hence, $Z_{X}\in W^{1,\infty}_{Y}(\Omega)$ and $Z_{Y}\in W^{1,\infty}_{X}(\Omega)$. The fact that $p\in W^{1,\infty}_{Y}(\Omega)$ and $q\in W^{1,\infty}_{X}(\Omega)$ follows by \eqref{eq:pdef} and \eqref{eq:qdef}, since $\p \in \Linf([X_{l},X_{r}])$ and $\q \in \Linf([Y_{l},Y_{r}])$. Thus, \eqref{eq:soln1} holds. 

Next, we verify that the relations \eqref{eq:soln2}-\eqref{eq:soln5} are satisfied. Since $(Z,Z,Z_{X},Z_{Y})$ is a fixed point of $\mathcal{T}$, we have, by differentiating \eqref{eq:Vbar} and \eqref{eq:Wbar}, that \eqref{eq:soln2} and \eqref{eq:soln3} hold. The relations \eqref{eq:soln4} and \eqref{eq:soln5} follow by differentiating \eqref{eq:pdef} and \eqref{eq:qdef}, respectively. 

We prove that $(Z,p,q)$ satisfies (\ref{eq:shortrange}). Since $(Z,Z,Z_{X},Z_{Y})$ is a fixed point of $\mathcal{T}$, we get, by \eqref{eq:Vbar} and \eqref{eq:Wbar}, that 	
\begin{equation*}
	Z_{X}(X,\Y(X))=\V(X)\quad \text{and} \quad Z_{Y}(\X(Y),Y)=\W(Y).
\end{equation*}
By \eqref{eq:Zvbar}, we have	
\begin{equation*}
	Z(\X(s),\Y(s))=\Z(\X^{-1}(\X(s)))
\end{equation*}
and, by \eqref{eq:geninv2}, this implies \eqref{eq:shortrange} for all $s\in [s_{l},s_{r}]$ such that $\dot{\X}(s)>0$. Similarly, by \eqref{eq:Zhbar}, we have	
\begin{equation*}
	Z(\X(s),\Y(s))=\Z(\Y^{-1}(\Y(s)))
\end{equation*}
and, by \eqref{eq:geninv3}, this implies \eqref{eq:shortrange} for all $s\in [s_{l},s_{r}]$ such that $\dot{\Y}(s)>0$. Since $\dot{\X}+\dot{\Y}=2$, the set of all $s\in [s_{l},s_{r}]$ such that both $\dot{\X}(s)=0$ and $\dot{\Y}(s)=0$ has zero measure. Hence, for almost every $s\in [s_{l},s_{r}]$, \eqref{eq:shortrange} holds, and since $Z$ is continuous, we get that $Z(\X(s),\Y(s))=\Z(s)$ for all	$s\in [s_{l},s_{r}]$. By \eqref{eq:pdef} and \eqref{eq:qdef}, it follows that
\begin{equation*}
	p(X,\Y(X))=\p(X)
\end{equation*}
for almost every $X\in [X_{l},X_{r}]$ and	
\begin{equation*}
	q(\X(Y),Y)=\q(Y)
\end{equation*}
for almost every $Y\in [Y_{l},Y_{r}]$, respectively, and we conclude that \eqref{eq:shortrange} holds. 	
Hence, we have shown that $Z$ is a solution of \eqref{eq:goveqt}-\eqref{eq:goveqK} which satisfies \eqref{eq:shortrange} if and only if it is a fixed point for $\mathcal{T}$. Since the fixed point exists and is unique, we have proved the existence and uniqueness of the solution $Z$ to \eqref{eq:goveqt}-\eqref{eq:goveqK}. Furthermore, since the functions $p$ and $q$, as defined in \eqref{eq:pdef} and \eqref{eq:qdef}, respectively, satisfy Definition \ref{def:soln} and \eqref{eq:shortrange}, we have proved the existence of a unique solution $(Z,p,q)$ of \eqref{eq:goveq}. Next we prove that the solution $(Z,p,q)$ belongs to $\H(\Omega)$. We define the functions $v\in W^{1,\infty}_{Y}(\Omega)$ and $w\in W^{1,\infty}_{X}(\Omega)$ as
\begin{equation*}
	v=x_{X}-c(U)t_{X} \quad \text{and} \quad w=x_{Y}+c(U)t_{Y}.
\end{equation*}
We want to prove that $v$ and $w$ are both zero. By using the governing equations (\ref{eq:goveq}), we obtain
\begin{equation*}
	v_{Y}=x_{XY}-c'(U)U_{Y}t_{X}-c(U)t_{XY}=\frac{c'(U)}{2c(U)}(U_{Y}v+U_{X}w)
\end{equation*}
and
\begin{equation*}
	w_{X}=x_{XY}+c'(U)U_{X}t_{Y}+c(U)t_{XY}=\frac{c'(U)}{2c(U)}(U_{Y}v+U_{X}w).
\end{equation*}
By \eqref{eq:shortrange} and \eqref{eq:setGlocrel1}, we have $v(X,\Y(X))=0$ and $w(\X(Y),Y)=0$. It follows that
\begin{align*}
	v(X,Y)&=\int_{\Y(X)}^{Y}\bigg(\frac{c'(U)}{2c(U)}(U_{Y}v+U_{X}w)\bigg)(X,\tilde{Y})\,d\tilde{Y}, \\
	w(X,Y)&=\int_{\X(Y)}^{X}\bigg(\frac{c'(U)}{2c(U)}(U_{Y}v+U_{X}w)\bigg)(\tilde{X},Y)\,d\tilde{X},
\end{align*}
which implies, by using (\ref{eq:deltabound1}) and (\ref{eq:deltabound2}), that  
\begin{align*}
	||v||_{L^{\infty}_{Y}(\Omega)}&\leq \delta \kappa k_1\big( ||U_{Y}||_{L^{\infty}_{X}(\Omega)}+||U_{X}||_{L^{\infty}_{Y}(\Omega)})(||v||_{L^{\infty}_{Y}(\Omega)}+||w||_{L^{\infty}_{X}(\Omega)}\big), \\
	||w||_{L^{\infty}_{X}(\Omega)}&\leq \delta \kappa k_1\big( ||U_{Y}||_{L^{\infty}_{X}(\Omega)}+||U_{X}||_{L^{\infty}_{Y}(\Omega)} )(||v||_{L^{\infty}_{Y}(\Omega)}+||w||_{L^{\infty}_{X}(\Omega)}\big).
\end{align*}
Hence, by setting $\delta>0$ smaller if necessary, we get $||v||_{L^{\infty}_{Y}(\Omega)}=||w||_{L^{\infty}_{X}(\Omega)}=0$. One proceeds similarly in order to prove (\ref{eq:setH2}). We show (\ref{eq:setH3}). Define $z\in W^{1,\infty}_{Y}(\Omega)$ as 
\begin{equation*}
	z=2J_{X}x_{X}-(c(U)U_{X})^{2}-c(U)p^{2}.
\end{equation*}
We have
\begin{align*}
	z_{Y}&=2J_{XY}x_{X}+2J_{X}x_{XY}-2c(U)^{2}U_{X}U_{XY}-2c(U)c'(U)U_{Y}U_{X}^{2} \\
	&\quad-c'(U)U_{Y}p^{2}-2c(U)pp_{Y} \\
	&=\frac{c'(U)}{c(U)}U_{Y}z
\end{align*}
and by \eqref{eq:setGlocrel3}, $z(X,\Y(X))=0$ for $X\in [X_{l},X_{r}]$, since $(Z,p,q)\bullet (\X,\Y)=\Theta \in \G(\Omega)$. After integrating, we obtain, since $|U_{Y}|\leq 2|||\Theta|||_{\G(\Omega)}$ and $|Y-\Y(X)|\leq 2\delta$, that
\begin{align*}
	|z(X,Y)|\leq |z(X,\Y(X))|e^{4\delta\kappa k_1|||\Theta|||_{\G(\Omega)}}.
\end{align*} 
Hence, $z=0$. Similarly, one proves that $2J_{Y}x_{Y}=(c(U)U_{Y})^{2}+c(U)q^{2}$. Now we prove (\ref{eq:setH4})-(\ref{eq:setH6}). Since
\begin{equation*}
	\frac{1}{x_{X}+J_{X}}(X,\Y(X))=\frac{1}{\V_{2}+\V_{4}}(X)
\end{equation*}
for almost every $X\in [X_{l},X_{r}]$, and since $\V_{2}$ and $\V_{4}$ belong to $\G(\Omega)$, we have
\begin{equation*}
	\norm{\frac{1}{x_{X}+J_{X}}(\cdot,\Y(\cdot))}_{L^{\infty}([X_{l},X_{r}])}\leq \tnorm{\Theta}_{\G(\Omega)}.
\end{equation*}
Let $X\in [X_l,X_r]$ such that
\begin{equation*}
	\frac{1}{x_{X}+J_{X}}(X,\Y(X))\leq \tnorm{\Theta}_{\G(\Omega)},
\end{equation*}	
and we define	
\begin{equation*}
	Y_{*}=\inf\{Y\in[Y_{l},Y_{r}] \ | \ Y\leq \Y(X) \text{ and } (x_{X}+J_{X})(X,Y')>0 \text{ for all } Y'>Y\}
\end{equation*}	
and	
\begin{equation*}
	Y^{*}=\sup\{Y\in[Y_{l},Y_{r}] \ | \ Y\geq \Y(X) \text{ and } (x_{X}+J_{X})(X,Y')>0 \text{ for all } Y'<Y\}.
\end{equation*}	
For $Y\in (Y_{*},Y^{*})$, we have $(x_{X}+J_{X})(X,Y)>0$ and we define
\begin{equation*}
	\eta(Y)=\frac{1}{(x_{X}+J_{X})(X,Y)}.
\end{equation*}	
Let us assume that $Y^{*}<Y_{r}$. Then
\begin{equation}
\label{eq:gronwallcontra}
	(x_{X}+J_{X})(X,Y^{*})=0.
\end{equation}	
Since $\eta(Y)\geq0$ for $Y\in (Y_{*},Y^{*})$, and $J_{X}x_{X}\geq0$ by (\ref{eq:setH3}), we have that
\begin{equation*}
	x_{X}(X,Y)\geq 0 \quad \text{and} \quad J_{X}(X,Y)\geq 0
\end{equation*}
for $Y\in (Y_{*},Y^{*})$. By (\ref{eq:goveq}), we have
\begin{equation*}
	\eta_{Y}=-\frac{x_{XY}+J_{XY}}{(x_{X}+J_{X})^{2}}=-\frac{c'(U)}{2c(U)}\frac{U_{Y}(x_{X}+J_{X})+U_{X}(x_{Y}+J_{Y})}{(x_{X}+J_{X})^{2}}
\end{equation*}
and from (\ref{eq:setH3}), we obtain
\begin{equation*}
	|U_{X}|\leq \frac{1}{c(U)}\sqrt{2J_{X}x_{X}}\leq \frac{1}{\sqrt{2}c(U)}(J_{X}+x_{X}).
\end{equation*} 	
Hence,
\begin{equation*}
	\eta_{Y}\leq \frac{|c'(U)|}{2c(U)}\Big(|U_{Y}|+\frac{1}{\sqrt{2}c(U)}(x_{Y}+J_{Y})\Big)\eta\leq C\eta
\end{equation*}	
for some constant $C$ which only depends on $\tnorm{\Theta}_{\G(\Omega)}$. From Gronwall's inequality it follows that
\begin{equation}
\label{eq:xXandJXbounded}
	\frac{1}{x_{X}+J_{X}}(X,Y)\leq \frac{1}{\V_{2}+\V_{4}}(X)e^{C|Y-\Y(X)|},
\end{equation}	
which contradicts (\ref{eq:gronwallcontra}), so that we must have $Y^{*}=Y_{r}$. In the same way one proves that $Y_{*}=Y_{l}$. Hence, 
\begin{equation*}
	x_{X}(X,Y)\geq 0, \quad J_{X}(X,Y)\geq 0 \quad \text{and} \quad (x_{X}+J_{X})(X,Y)>0
\end{equation*}
for almost every $X\in [X_{l},X_{r}]$ and all $Y\in [Y_{l},Y_{r}]$. This concludes the proof of the first identities in \eqref{eq:setH4}--\eqref{eq:setH6}. The second identities in \eqref{eq:setH4} -- \eqref{eq:setH6} can be proved in a similar way.
\end{proof}

\subsection{Existence of Local Solutions}

We begin with some a priori estimates.

Given a positive constant $L$, we call domains of the type
\begin{equation*}
\{(X,Y)\in \mathbb{R}^{2} \ | \ |Y-X|\leq2L \}
\end{equation*}
\emph{strip domains}, which correspond to domains where time is bounded. We have the following a priori estimates for the solution of \eqref{eq:goveq}.

\begin{lemma}
	\label{lemma:stripest}
	Given $\Omega=[X_{l},X_{r}]\times[Y_{l},Y_{r}]$ and $\Theta=(\X,\Y,\Z,\V,\W,\p,\q)\in \G(\Omega)$, let 
	$(Z,p,q)\in \H(\Omega)$ be a solution of (\ref{eq:goveq}) such that $\Theta=(Z,p,q)\,\bullet\, (\X,\Y)$. Let $\E_{0}=||\Z_{4}||_{L^{\infty}([s_{l},s_{r}])}+||\Z_{5}||_{L^{\infty}([s_{l},s_{r}])}$. Then the following statements hold:
	\begin{enumerate}
		\item[(i)] Boundedness of the energy, that is,
		\begin{subequations}
		\begin{equation}
		\label{eq:striplemmaJ}
		0\leq J(X,Y)\leq \E_{0} \quad \text{for all } (X,Y)\in \Omega
		\end{equation}
		and
		\begin{equation}
		\label{eq:striplemmaK}
		||K||_{L^{\infty}(\Omega)}\leq (1+\kappa)\E_{0}.
		\end{equation} 
		\end{subequations}
		\item[(ii)] The functions $Z$, $Z_{X}$, $Z_{Y}$, $p$ and $q$ remain uniformly bounded in strip domains which contain $\Omega$, that is, there exists a nondecreasing function $C_{1}=C_{1}(\tnorm{\Theta}_{\G(\Omega)},L)$ such that for any $L>0$ and any $(X,Y)\in \Omega$ such that $|X-Y|\leq 2L$, we have
		\begin{subequations}
		\begin{equation}
		\label{eq:striplemmaZ}
		|Z^{a}(X,Y)|\leq C_{1}, \quad |Z_{X}(X,Y)|\leq C_{1}, \quad |Z_{Y}(X,Y)|\leq C_{1},
		\end{equation}
		\begin{equation}
		\label{eq:striplemmapq}
		|p(X,Y)|\leq C_{1}, \quad |q(X,Y)|\leq C_{1}
		\end{equation}
		and
		\begin{equation}
		\label{eq:striplemma2}
		\frac{1}{x_{X}+J_{X}}(X,Y)\leq C_{1}, \quad \frac{1}{x_{Y}+J_{Y}}(X,Y)\leq C_{1}.
		\end{equation}
		\end{subequations}
		Condition \emph{(ii)} is equivalent to the following:
		\item[(iii)] For any curve $(\bar{\X},\bar{\Y})\in \C(\Omega)$, we have
		\begin{equation}
		\label{eq:striplemma3}
		\tnorm{(Z,p,q) \bullet (\bar{\X},\bar{\Y})}_{\G(\Omega)}\leq C_{1},
		\end{equation}
		where $C_{1}=C_{1}(||(\bar{\X},\bar{\Y})||_{\C(\Omega)},|||\Theta|||_{\G(\Omega)})$ is an increasing function with respect to both its arguments.
	\end{enumerate}

\end{lemma}

\begin{proof}
	Given $P=(X,Y)\in \Omega$, let $s_{0}=\Y^{-1}(Y)$ and $s_{1}=\X^{-1}(X)$, where we assume that $s_{0}\leq s_{1}$ (the proof for the other case is similar). We denote $P_{0}=(\X(s_{0}),\Y(s_{0}))$ and $P_{1}=(\X(s_{1}),\Y(s_{1}))$. Since $\X$ and $\Y$ are increasing functions, we have that $X=\X(s_{1})\geq \X(s_{0})$ and $Y=\Y(s_{0})\leq \Y(s_{1})$. Then, because $\Z_{4}\geq0$, $J_{X}\geq0$ and $J_{Y}\geq0$, we have
	\begin{equation}
	\label{eq:stripJbound}
	0\leq \Z_{4}(s_{0})=J(P_{0})\leq J(P)\leq J(P_{1})=\Z_{4}(s_{1})\leq \E_{0},
	\end{equation}
	which proves (\ref{eq:striplemmaJ}). By (\ref{eq:setH2}), we have
	\begin{equation*}
	K(P)-K(P_{0})=\int_{\X(s_{0})}^{X}\bigg(\frac{J_{X}}{c(U)}\bigg)(\tilde{X},Y)\,d\tilde{X}.
	\end{equation*}
	Hence,
	\begin{equation*}
	|K(P)|\leq |K(P_{0})|+\kappa(J(P)-J(P_{0}))\leq (1+\kappa)\E_{0}
	\end{equation*}
	by (\ref{eq:striplemmaJ}). This proves (\ref{eq:striplemmaK}). Next we show \eqref{eq:striplemmaZ}-\eqref{eq:striplemmapq}. Since $x_{X}\geq0$, we have
	\begin{equation*}
	x(P)\geq x(P_{0})=\Z_{2}(s_{0})=\Z_{2}^{a}(s_{0})+s_{0}\geq -\tnorm{\Theta}_{\G(\Omega)}+s_{0} 
	\end{equation*}
	and since $\frac{1}{2}(X+Y)=Y+\frac{1}{2}(X-Y)\leq \Y(s_{0})+L$, it follows that
	\begin{equation*}
	x(P)-\frac{1}{2}(X+Y) \geq -\tnorm{\Theta}_{\G(\Omega)}+(s_{0}-\Y(s_{0}))-L \geq -2\tnorm{\Theta}_{\G(\Omega)}-L.
	\end{equation*}
	Similarly, we find $x(P)\leq \tnorm{\Theta}_{\G(\Omega)}+s_{1}$ and $\frac{1}{2}(X+Y)\geq \X(s_{1})-L$, which implies that
	\begin{equation*}
	x(P)-\frac{1}{2}(X+Y) \leq 2\tnorm{\Theta}_{\G(\Omega)}+L.
	\end{equation*}
	Hence,
	\begin{equation*}
	|Z_{2}^{a}(P)|=\Big|x(P)-\frac{1}{2}(X+Y)\Big|\leq 2\tnorm{\Theta}_{\G(\Omega)}+L.
	\end{equation*}
	By \eqref{eq:setH1}, we have
	\begin{align*}
	|t(P)|&\leq |t(P_{1})|+\int_{Y}^{\Y(s_{1})}|t_{Y}(X,\tilde{Y})|\,d\tilde{Y}\\
	&=|t(P_{1})|+\int_{Y}^{\Y(s_{1})}\bigg(\frac{x_{Y}}{c(U)}\bigg)(X,\tilde{Y})\,d\tilde{Y}\\
	&\leq |t(P_{1})|+\kappa(x(P_{1})-x(P))
	\end{align*}
	and since
	\begin{equation}
	\label{eq:stripestxdiff}
	x(P_{1})-x(P)\leq x(P_{1})-x(P_{0})=\Z_{2}^{a}(s_{1})+s_{1}-\Z_{2}^{a}(s_{0})-s_{0}\leq 2\tnorm{\Theta}_{\G(\Omega)}+s_{1}-s_{0}
	\end{equation}
	and
	\begin{equation}
	\label{eq:stripestsdiff}
	s_{1}-s_{0}=(s_{1}-\X(s_{1}))+(\Y(s_{0})-s_{0})+(X-Y)\leq 2\tnorm{\Theta}_{\G(\Omega)}+2L,
	\end{equation}
	it follows that
	\begin{equation*}
	|t(P)|=|Z_{1}(P)|\leq \tnorm{\Theta}_{\G(\Omega)}+\kappa(4\tnorm{\Theta}_{\G(\Omega)}+2L).
	\end{equation*}
	Hence,
	\begin{equation*}
	|Z_{1}^{a}(P)|=\Big|t(P)-\frac{1}{2c(0)}(X-Y)\Big|\leq \tnorm{\Theta}_{\G(\Omega)}+\kappa(4\tnorm{\Theta}_{\G(\Omega)}+3L).
	\end{equation*}
	We have
	\begin{equation*}
	|U(P)|\leq |U(P_{1})|+\int_{Y}^{\Y(s_{1})}|U_{Y}(X,\tilde{Y})|\,d\tilde{Y}.
	\end{equation*}
	By (\ref{eq:setH3}), we have $2J_{Y}x_{Y}\geq (c(U)U_{Y})^{2}$, which implies that
	\begin{equation}
	\label{eq:stripestUbound}
	|U_{Y}|\leq \frac{\kappa}{\sqrt{2}}(J_{Y}+x_{Y}).
	\end{equation}
	Hence, by \eqref{eq:stripJbound}, \eqref{eq:stripestxdiff} and \eqref{eq:stripestsdiff}, we obtain
	\begin{aalign}
	\label{eq:stripestU2bound}
	|U(P)|&\leq |U(P_{1})|+\frac{\kappa}{\sqrt{2}}(J(P_{1})+x(P_{1})-J(P)-x(P)) \\
	&\leq \tnorm{\Theta}_{\G(\Omega)}+\frac{\kappa}{\sqrt{2}}(\E_{0}+4\tnorm{\Theta}_{\G(\Omega)}+2L).
	\end{aalign}
	We prove that $Z_{X}$ and $Z_{Y}$ remain bounded. As above, we assume that $Y\leq \Y(X)$. For almost every $X\in [X_{l},X_{r}]$, we have
	\begin{equation*}
	|Z_{X}(X,Y)|\leq |Z_{X}(X,\Y(X))|+\int_{Y}^{\Y(X)}(|t_{XY}|+|x_{XY}|+|U_{XY}|+|J_{XY}|+|K_{XY}|)(X,\tilde{Y})\,d\tilde{Y}.
	\end{equation*} 
	From the governing equations (\ref{eq:goveq}), we find that
	\begin{equation*}
	|t_{XY}|+|x_{XY}|+|U_{XY}|+|J_{XY}|+|K_{XY}|\leq C|Z_{X}||Z_{Y}|
	\end{equation*}
	for some constant $C$ dependent on $\Omega$, $\kappa$, and $k_1$. By Gronwall's lemma we obtain
	\begin{aalign}
	\label{eq:stripestZXbound}
	|Z_{X}(X,Y)|&\leq |Z_{X}(X,\Y(X))|\exp\Big(\int_{Y}^{\Y(X)}C|Z_{Y}(X,\tilde{Y})|\,d\tilde{Y}\Big) \\
	&=|\V(X)|\exp\Big(C\int_{Y}^{\Y(X)}|Z_{Y}(X,\tilde{Y})|\,d\tilde{Y}\Big).
	\end{aalign}
	From \eqref{eq:atriplet}, we have
	\begin{equation*}
	|\V(X)|\leq |\V^{a}(X)|+C\leq \tnorm{\Theta}_{\G(\Omega)}+C
	\end{equation*}
	for some constant $C$ dependent on $\Omega$, $\kappa$, and $k_1$. Furthermore, using (\ref{eq:setH1}), (\ref{eq:setH2}) and (\ref{eq:stripestUbound}), we obtain
	\begin{equation*}
	|Z_{Y}|=\frac{1}{c(U)}(x_{Y}+J_{Y})+x_{Y}+J_{Y}+|U_{Y}|\leq C(x_{Y}+J_{Y}).
	\end{equation*} 
	Hence, 
	\begin{align*}
	\int_{Y}^{\Y(X)}|Z_{Y}(X,\tilde{Y})|\,d\tilde{Y}&\leq C\int_{Y}^{\Y(X)}(x_{Y}+J_{Y})(X,\tilde{Y})\,d\tilde{Y} \\
	&=C(x(X,\Y(X))+J(X,\Y(X))-x(X,Y)-J(X,Y)) \\
	&\leq C(\E_{0}+4\tnorm{\Theta}_{\G(\Omega)}+2L),
	\end{align*}
	where we used the same estimate as in (\ref{eq:stripestU2bound}). Combined with (\ref{eq:stripestZXbound}), this yields
	\begin{equation*}
	|Z_{X}(X,Y)|\leq C_1
	\end{equation*}
	for some constant $C_1$ which only depends on $\tnorm{\Theta}_{\G(\Omega)}$ and $L$. Similarly, one proves the bound on $Z_{Y}$. The estimates in (\ref{eq:striplemmapq}) follows from the fact that $p_{Y}=0$ and $q_{X}=0$. Indeed, we have
	\begin{equation*}
	|p(X,Y)|=|p(X,\Y(X))|=|\p(X)|\leq \tnorm{\Theta}_{\G(\Omega)}
	\end{equation*}
	and similarly for $q$. Let us prove (\ref{eq:striplemma2}). In (\ref{eq:xXandJXbounded}), we found that
	\begin{equation}
	\label{eq:stripestxXJXbound}
	\frac{1}{x_{X}+J_{X}}(X,Y)\leq \frac{1}{\V_{2}+\V_{4}}(X)e^{C|Y-\Y(X)|}
	\end{equation}
	for a constant $C$ which only depends on $\tnorm{\Theta}_{\G(\Omega)}$. We have
	\begin{align*}
	|Y-\Y(X)|&=|Y-X+\X(\X^{-1}(X))-\X^{-1}(X)+\X^{-1}(X)-\Y(\X^{-1}(X))| \\
	&\leq L+\norm{\X-\id}_{\Linf([s_{l},s_{r}])}+\norm{\Y-\id}_{\Linf([s_{l},s_{r}])} \\
	&\leq L+\tnorm{\Theta}_{\G(\Omega)},
	\end{align*}
	which combined with \eqref{eq:stripestxXJXbound} yields the first inequality in (\ref{eq:striplemma2}). The other inequality in \eqref{eq:striplemma2} can be proved in a similar way. We show that the conditions (ii) and (iii) are equivalent. Given a curve $(\bar{\X},\bar{\Y})\in \C(\Omega)$, we have, since $\bar{\X}+\bar{\Y}=2s$, that
	\begin{equation*}
	||\bar{\X}-\bar{\Y}||_{\Linf}=2||\bar{\X}-\id||_{\Linf}=||\bar{\X}-\id||_{\Linf}+||\bar{\Y}-\id||_{\Linf}=||(\bar{\X},\bar{\Y})||_{\C(\Omega)}
	\end{equation*}  
	and \eqref{eq:striplemma3} follows.
\end{proof}

We have the following existence and uniqueness result.

\begin{lemma}[Existence and uniqueness on arbitrarily large rectangles]
	\label{lemma:SolnBigRectangle}
	Given a rectangular domain $\Omega=[X_{l},X_{r}]\times[Y_{l},Y_{r}]$ and $\Theta=(\X,\Y,\Z,\V,\W,\p,\q)\in \G(\Omega)$, there exists a unique solution $(Z,p,q)\in \H(\Omega)$ such that
	\begin{equation*}
	\Theta=(Z,p,q) \bullet (\X,\Y).
	\end{equation*}
\end{lemma}

\begin{proof}
	Let $\delta=\frac{s_{r}-s_{l}}{N}$, where $N$ is an integer that we will specifiy later. For $i=0,\dots,N$, let $s_{i}=i\delta+s_{l}$ and $P_{i}=(X_{i},Y_{i})=(\X(s_{i}),\Y(s_{i}))$. For $i,j=0,\dots,N$, we construct a grid which consists of the points $P_{i,j}=(X_{i,j},Y_{i,j})$, where $X_{i,j}=X_{i}$ and $Y_{i,j}=Y_{j}$. We denote by $\Omega_{i,j}$ the rectangle with diagonal points $P_{i,j}$ and $P_{i+1,j+1}$, and by $\Omega_{n}$ the rectangle with diagonal points $(X_{0},Y_{0})$ and $(X_{n},Y_{n})$. We prove by induction that there exists a unique $(Z,p,q)\in \H(\Omega_{n})$ such that $\Theta=(Z,p,q) \bullet (\X,\Y)\in \G(\Omega_{n})$. There is an increasing function $C=C(|||\Theta|||_{\G(\Omega)})$ such that 
	\begin{equation*}
	s_{1}-s_{0}=\delta\leq 1/C(|||\Theta|||_{\G(\Omega)})\leq 1/C(|||\Theta|||_{\G(\Omega_{1})}),
	\end{equation*}
	provided that $N$ is large enough. Hence, by Theorem \ref{thm:solnsmallrect}, there exists a unique solution $(Z,p,q)\in \H(\Omega_{1})$ such that $\Theta=(Z,p,q) \bullet (\X,\Y)\in \G(\Omega_{1})$. We assume that there exists a unique solution $(Z,p,q)$ on $\Omega_{n}$ and prove that there exists a solution on $\Omega_{n+1}$. By Theorem \ref{thm:solnsmallrect}, there exists a unique solution on $\Omega_{n,n}$ since 
	\begin{equation*}
	s_{n+1}-s_{n}=\delta\leq 1/C(|||\Theta|||_{\G(\Omega)})\leq 1/C(|||\Theta|||_{\G(\Omega_{n,n})}).
	\end{equation*}
	For $j=n-1,\dots,0$, we iteratively construct the unique solution in $\Omega_{n,j}$ and $\Omega_{j,n}$ as follows. We only treat the case of $\Omega_{n,j}$. We assume that the solution is known on $\Omega_{n,j+1}$, then we define $\tilde{\Theta}=(\tilde{\X},\tilde{\Y},\tilde{\Z},\tilde{\V},\tilde{\W},\tilde{\p},\tilde{\q})\in \G(\Omega_{n,j})$ as follows: the curve $(\tilde{\X}(s),\tilde{\Y}(s))$ is given by
	\begin{equation*}
	\tilde{\X}(s)=X_{n}, \quad \tilde{\Y}(s)=2s-X_{n}
	\end{equation*}
	for $\frac{1}{2}(X_{n}+Y_{j})\leq s\leq \frac{1}{2}(X_{n}+Y_{j+1})$ and
	\begin{equation*}
	\tilde{\X}(s)=2s-Y_{j+1}, \quad \tilde{\Y}(s)=Y_{j+1}
	\end{equation*}
	for $\frac{1}{2}(X_{n}+Y_{j+1})\leq s\leq \frac{1}{2}(X_{n+1}+Y_{j+1})$. We set
	\begin{align*}
	\tilde{Z}(s)&=Z(\tilde{\X}(s),\tilde{\Y}(s)) \quad \text{for } s\in \Big[\frac{1}{2}(X_{n}+Y_{j}) , \frac{1}{2}(X_{n+1}+Y_{j+1})\Big], \\
	\tilde{\V}(X)&=Z_{X}(X,Y_{j+1}) \quad \text{for a.e. } X\in [X_{n},X_{n+1}], \\
	\tilde{\W}(Y)&=Z_{Y}(X_{n},Y) \quad \text{for a.e. } Y\in [Y_{j},Y_{j+1}], \\
	\tilde{\p}(X)&=p(X,Y_{j+1}) \quad \text{for a.e. } X\in [X_{n},X_{n+1}], \\
	\tilde{\q}(Y)&=q(X_{n},Y) \quad \text{for a.e. } Y\in [Y_{j},Y_{j+1}],
	\end{align*}
	where $(Z,p,q)$ is the solution on $\Omega_{n}\cup (\cup_{i=j+1}^{n}\Omega_{n,i})$. By Lemma \ref{lemma:stripest}, we have that $|||\tilde{\Theta}|||_{\G(\Omega_{n,j})}\leq C_{1}(|||\Theta|||_{\G(\Omega)},L)$. We have
	\begin{equation*}
	\frac{1}{2}(X_{n+1}+Y_{j+1})-\frac{1}{2}(X_{n}+Y_{j})=\frac{1}{2}(\X(s_{n+1})-\X(s_{n})+\Y(s_{j+1})-\Y(s_{j}))\leq 2\delta,
	\end{equation*}
	because $\X$ and $\Y$ are Lipschitz continuous with Lipschitz constant smaller than 2. By setting $N$ so large that $2\delta\leq 1/C(C_{1})$ it follows that $2\delta \leq 1/C(|||\tilde{\Theta}|||_{\G(\Omega_{n,j})})$, so that we can apply Theorem \ref{thm:solnsmallrect} to $\Omega_{n,j}$ and obtain the existence of a unique solution in $\H(\Omega_{n,j})$. Similarly we obtain the existence of a unique solution in $\H(\Omega_{j,n})$. Since
	\begin{equation*}
	\Omega_{n+1}=\Omega_{n}\cup (\cup_{j=0}^{n}\Omega_{n,j})\cup(\cup_{j=0}^{n-1}\Omega_{j,n}),
	\end{equation*}
	we have proved the existence of a unique solution in $\Omega_{n+1}$.
\end{proof}

\begin{lemma}[A Gronwall lemma for curves]
	\label{lemma:groncurve}
	Let $\Omega=[X_{l},X_{r}]\times [Y_{l},Y_{r}]$ and assume that $(Z,p,q)\in \H(\Omega)$ and $(\X,\Y)\in \C(\Omega)$. Then, for any $(\bar{\X},\bar{\Y})\in \C(\Omega)$, we have
	\begin{equation*}
	||(Z,p,q)\bullet (\bar{\X},\bar{\Y})||_{\G(\Omega)}\leq C||(Z,p,q)\bullet (\X,\Y)||_{\G(\Omega)},
	\end{equation*}
	where $C=C(||(\bar{\X},\bar{\Y})||_{\C(\Omega)},|||(Z,p,q)\bullet (\X,\Y)|||_{\G(\Omega)})$ is an increasing function with respect to both its arguments.
\end{lemma}

\begin{proof}
	Note that for any function $f\in \WY(\Omega)$, $f_{Y}$ is well-defined, but not $f_{X}$. This means that the form $f(X,Y)\,dX$ is well-defined, while the form $f(X,Y)\,dY$ is not. Similarly, for any function $g\in \WX(\Omega)$, the form $g(X,Y)\,dY$ is well-defined and $g(X,Y)\,dX$ is not. Thus, given $(Z,p,q)\in \H(\Omega)$, we can consider the forms $U^{2}\,dX$, $U^{2}\,dY$, $|Z_{X}^{a}|^{2}\,dX$, $|Z_{Y}^{a}|^{2}\,dY$, $p^{2}\,dX$ and $q^{2}\,dY$. For any curve $\bar{\Gamma}=(\bar{\X},\bar{\Y})\in \C(\Omega)$ and $\bar{\Theta}=(\bar{\X},\bar{\Y},\bar{\Z},\bar{\V},\bar{\W},\bar{\p},\bar{\q})$ such that $\bar{\Theta}=(Z,p,q)\bullet (\bar{\X},\bar{\Y})$, we have
	\begin{equation*}
	\int_{\bar{\Gamma}}U^{2}(X,Y)\,dX+U^{2}(X,Y)\,dY=2\int_{s_{l}}^{s_{r}}\bar{\Z}_{3}^{2}(s)\,ds
	\end{equation*}
	and, since $Z_{X}^{a}(X,Y)=Z_{X}^{a}(X,\bar{\Y}(X))$ and $Z_{Y}^{a}(X,Y)=Z_{Y}^{a}(\bar{\X}(Y),Y)$ on $\bar{\Gamma}$, we have
	\begin{equation*}
	\int_{\bar{\Gamma}}|Z_{X}^{a}(X,Y)|^{2}\,dX=\int_{X_{l}}^{X_{r}}|\bar{\V}^{a}(X)|^{2}\,dX,\quad \int_{\bar{\Gamma}}|Z_{Y}^{a}(X,Y)|^{2}\,dY=\int_{Y_{l}}^{Y_{r}}|\bar{\W}^{a}(Y)|^{2}\,dY.
	\end{equation*}
	Similarly, since $p(X,Y)=p(X,\bar{\Y}(X))$ and $q(X,Y)=q(\bar{\X}(Y),Y)$ on $\bar{\Gamma}$, we obtain
	\begin{align*}
	\int_{\bar{\Gamma}}p(X,Y)^{2}\,dX=\int_{X_{l}}^{X_{r}}\bar{\p}(X)^{2}\,dX,\quad \int_{\bar{\Gamma}}q(X,Y)^{2}\,dY=\int_{Y_{l}}^{Y_{r}}\bar{\q}(Y)^{2}\,dY.
	\end{align*} 
	Now we can rewrite
	\begin{aalign}
	\label{eq:Gnormequivalent}
	&||(Z,p,q)\bullet (\bar{\X},\bar{\Y})||_{\G(\Omega)}^{2}\\
	&=\int_{\bar{\Gamma}}\bigg( \frac{1}{2}U^{2}\,(dX+dY)+|Z_{X}^{a}|^{2}\,dX+|Z_{Y}^{a}|^{2}\,dY+p^{2}\,dX+q^{2}\,dY \bigg).
	\end{aalign}  	
	Thus, we want to prove that
	\begin{aalign}
	\label{eq:GronCurveProveThis}
	&\int_{\bar{\Gamma}}\bigg( \frac{1}{2}U^{2}\,(dX+dY)+|Z_{X}^{a}|^{2}\,dX+|Z_{Y}^{a}|^{2}\,dY+p^{2}\,dX+q^{2}\,dY \bigg)\\
	&\leq C\int_{\Gamma}\bigg( \frac{1}{2}U^{2}\,(dX+dY)+|Z_{X}^{a}|^{2}\,dX+|Z_{Y}^{a}|^{2}\,dY+p^{2}\,dX+q^{2}\,dY \bigg).
	\end{aalign}
	We decompose the proof into three steps.
	
	\textbf{Step 1.} We first prove that \eqref{eq:GronCurveProveThis} holds for small domains. We claim that there exist constants $\delta$ and $C$, which depend on $||(\bar{\X},\bar{\Y})||_{\C(\Omega)}$ and $|||(Z,p,q)\bullet (\X,\Y)|||_{\G(\Omega)}$, such that for any rectangular domain $\Omega=[X_{l},X_{r}]\times [Y_{l},Y_{r}]$ with $s_{r}-s_{l}\leq \delta$, \eqref{eq:GronCurveProveThis} holds. By Lemma \ref{lemma:stripest}, we have
	\begin{equation*}
	||U||_{\Linf(\Omega)}+||Z_{X}^{a}||_{\Linf(\Omega)}+||Z_{Y}^{a}||_{\Linf(\Omega)}+||p||_{\Linf(\Omega)}+||q||_{\Linf(\Omega)}\leq C,
	\end{equation*}
	where $C=C(\Omega,|||(Z,p,q)\bullet (\X,\Y)|||_{\G(\Omega)})$ which is increasing with respect to its second argument. Let
	\begin{equation*}
	A=\sup_{\bar{\Gamma}}\int_{\bar{\Gamma}}\bigg( \frac{1}{2}U^{2}\,(dX+dY)+|Z_{X}^{a}|^{2}\,dX+|Z_{Y}^{a}|^{2}\,dY+p^{2}\,dX+q^{2}\,dY \bigg),
	\end{equation*}
	where the supremum is taken over all $\bar{\Gamma}=(\bar{\X},\bar{\Y})\in \C(\Omega)$. We have
	\begin{aalign}
	\label{eq:gronlemmaxXest}
	\bigg(x_{X}-\frac{1}{2}\bigg)^{2}(X,Y)&=\bigg(x_{X}-\frac{1}{2}\bigg)^{2}(X,\Y(X))\\
	&\quad+\int_{\Y(X)}^{Y}2\bigg(\bigg(x_{X}-\frac{1}{2}\bigg)x_{XY}\bigg)(X,\tilde{Y})\,d\tilde{Y}.
	\end{aalign}
	Since $(Z,p,q)$ is a solution of \eqref{eq:goveq}, we have for almost every $X\in [X_{l},X_{r}]$ and all $Y\in [Y_{l},Y_{r}]$, that
	\begin{align*}
	2\bigg(x_{X}-\frac{1}{2}\bigg)x_{XY}&=\bigg(x_{X}-\frac{1}{2}\bigg)\frac{c'(U)}{c(U)}(U_{Y}x_{X}+U_{X}x_{Y})\\
	&=\frac{c'(U)}{c(U)}\bigg(U_{Y}\bigg(x_{X}-\frac{1}{2}\bigg)^{2}
	+\frac{1}{2}U_{Y}\bigg(x_{X}-\frac{1}{2}\bigg)\\
	&\hspace{52pt}+U_{X}\bigg(x_{Y}-\frac{1}{2}\bigg)\bigg(x_{X}-\frac{1}{2}\bigg)+
	\frac{1}{2}U_{X}\bigg(x_{X}-\frac{1}{2}\bigg)\bigg).
	\end{align*}
	Using \eqref{eq:cassumption}, \eqref{eq:cderassumption} and Young's inequality, we get
	\begin{aalign}
	\label{eq:gronwallxXest}
	&\bigg|2\bigg(x_{X}-\frac{1}{2}\bigg)x_{XY}\bigg|\\
	&\leq\kappa k_{1}\bigg(||Z_{Y}^{a}||_{\Linf(\Omega)}\bigg(x_{X}-\frac{1}{2}\bigg)^{2}
	+\frac{1}{2}|U_{Y}|\bigg|x_{X}-\frac{1}{2}\bigg|\\
	&\hspace{40pt}+||Z_{X}^{a}||_{\Linf(\Omega)}\bigg|x_{Y}-\frac{1}{2}\bigg|\bigg|x_{X}-\frac{1}{2}\bigg|+\frac{1}{2}|U_{X}|\bigg|x_{X}-\frac{1}{2}\bigg|\bigg)\\
	&\leq\kappa k_{1}\bigg(||Z_{Y}^{a}||_{\Linf(\Omega)}\bigg(x_{X}-\frac{1}{2}\bigg)^{2}
	+\frac{1}{4}U_{Y}^{2}+\frac{1}{4}\bigg(x_{X}-\frac{1}{2}\bigg)^{2}\\
	&\hspace{40pt}+\frac{1}{2}||Z_{X}^{a}||_{\Linf(\Omega)}\bigg(x_{Y}-\frac{1}{2}\bigg)^{2}+\frac{1}{2}||Z_{X}^{a}||_{\Linf(\Omega)}\bigg(x_{X}-\frac{1}{2}\bigg)^{2}\\
	&\hspace{40pt}+\frac{1}{4}U_{X}^{2}+\frac{1}{4}\bigg(x_{X}-\frac{1}{2}\bigg)^{2}\bigg)\\
	&\leq\kappa k_{1}\bigg(C\bigg(\bigg(x_{X}-\frac{1}{2}\bigg)^{2}+\bigg(x_{Y}-\frac{1}{2}\bigg)^{2}\bigg)\\
	&\hspace{40pt}+\frac{1}{2}\bigg(\bigg(x_{X}-\frac{1}{2}\bigg)^{2}+U_{X}^{2}+U_{Y}^{2}\bigg)\bigg)\\
	&\leq\kappa k_{1}\bigg(C+\frac{1}{2}\bigg)\big(|Z_{X}^{a}|^{2}+|Z_{Y}^{a}|^{2}\big).
	\end{aalign}
	Inserting this into \eqref{eq:gronlemmaxXest} and integrating over $[X_{l},X_{r}]$ gives
	\begin{aalign}
	\label{eq:GronCurve2ndComp}
	\int_{\bar{\Gamma}}\bigg(x_{X}-\frac{1}{2}\bigg)^{2}\,dX&\leq \int_{\Gamma}\bigg(x_{X}-\frac{1}{2}\bigg)^{2}\,dX\\
	&\quad+\kappa k_{1}\bigg(C+\frac{1}{2}\bigg)\int_{X_{l}}^{X_{r}}\int_{Y_{l}}^{Y_{r}}\big(|Z_{X}^{a}|^{2}+|Z_{Y}^{a}|^{2}\big)\,dY\,dX.
	\end{aalign}
	For any $Y\in [Y_{l},Y_{r}]$, the integral $\int_{X_{l}}^{X_{r}}|Z_{X}^{a}|^{2}(X,Y)\,dX$ can be seen as part of the integral of the form $|Z_{X}^{a}|^{2}\,dX$ on the piecewise linear path going through the points $(X_{l},Y_{l})$, $(X_{l},Y)$, $(X_{r},Y)$ and $(X_{r},Y_{r})$, which implies that $\int_{X_{l}}^{X_{r}}|Z_{X}^{a}|^{2}(X,Y)\,dX\leq A$. Similarly, for any $X\in [X_{l},X_{r}]$, $\int_{Y_{l}}^{Y_{r}}|Z_{Y}^{a}|^{2}(X,Y)\,dY\leq A$. Hence, by \eqref{eq:GronCurve2ndComp},
	\begin{align*}
	\int_{\bar{\Gamma}} \big(Z_{2,X}^{a}\big)^{2}\,dX&\leq \int_{\Gamma} \big(Z_{2,X}^{a}\big)^{2}\,dX+\kappa k_{1}\big(C+\frac{1}{2}\big)A(Y_{r}-Y_{l}+X_{r}-X_{l})\\
	&\leq \int_{\Gamma} \big(Z_{2,X}^{a}\big)^{2}\,dX+2\delta \kappa k_{1}\big(C+\frac{1}{2}\big)A
	\end{align*} 
	since $Y_{r}-Y_{l}+X_{r}-X_{l}=2(s_{r}-s_{l})$. By treating the other components of $Z_{X}^{a}$ and $Z_{Y}^{a}$ similarly, we get
	\begin{align}
	\label{eq:GronCurveStep1i}
	\int_{\bar{\Gamma}}|Z_{X}^{a}|^{2}\,dX&\leq \int_{\Gamma}|Z_{X}^{a}|^{2}\,dX+\delta \bar{C}A,\\
	\label{eq:GronCurveStep1ii}
	\int_{\bar{\Gamma}}|Z_{Y}^{a}|^{2}\,dY&\leq \int_{\Gamma}|Z_{Y}^{a}|^{2}\,dY+\delta \bar{C}A.
	\end{align}
	where $\bar{C}$ depends on $||(\bar{\X},\bar{\Y})||_{\C(\Omega)}$, $|||(Z,p,q)\bullet (\X,\Y)|||_{\G(\Omega)}$, $\kappa$ and $k_{1}$. For almost every $X\in[X_{l},X_{r}]$ and all $Y\in[Y_{l},Y_{r}]$, we have $p_{Y}(X,Y)=0$, so that $p(X,Y)=p(X,\Y(X))$. By squaring this expression and integrating over $[X_{l},X_{r}]$, we obtain
	\begin{equation}
	\label{eq:GronCurveStep1iiip}
	\int_{\bar{\Gamma}}p^{2}\,dX=\int_{\Gamma}p^{2}\,dX.
	\end{equation}
	Since $q_{X}(X,Y)=0$ for almost every $Y\in[Y_{l},Y_{r}]$ and all $X\in[X_{l},X_{r}]$, we get $q(X,Y)=q(\X(Y),Y)$ which implies, after squaring and integrating over $[Y_{l},Y_{r}]$, that
	\begin{equation}
	\label{eq:GronCurveStep1iiiq} 
	\int_{\bar{\Gamma}}q^{2}\,dY=\int_{\Gamma}q^{2}\,dY.
	\end{equation}
	For $U$, we have 
	\begin{align*}
	U^{2}(X,Y)&=U^{2}(X,\Y(X))+2\int_{\Y(X)}^{Y}(UU_{Y})(X,\tilde{Y})\,d\tilde{Y}\\\
	&\leq U^{2}(X,\Y(X))+\left\vert\int_{\Y(X)}^{Y}U^{2}(X,\tilde{Y})\,d\tilde{Y}\right\vert+\left\vert\int_{\Y(X)}^{Y}U_{Y}^{2}(X,\tilde{Y})\,d\tilde{Y}\right\vert.
	\end{align*}	
	As above, it follows that
	\begin{equation}
	\label{eq:GronCurveStep1iv}
	\int_{\bar{\Gamma}}U^{2}\,dX\leq \int_{\Gamma}U^{2}\,dX+2\delta A.
	\end{equation}	
	Similarly, we obtain	
	\begin{equation}
	\label{eq:GronCurveStep1v}
	\int_{\bar{\Gamma}}U^{2}\,dY\leq \int_{\Gamma}U^{2}\,dY+2\delta A.
	\end{equation}	
	By adding \eqref{eq:GronCurveStep1i} -- \eqref{eq:GronCurveStep1v}, we obtain
	\begin{align*}
	&\int_{\bar{\Gamma}}\bigg( \frac{1}{2}U^{2}\,(dX+dY)+|Z_{X}^{a}|^{2}\,dX+|Z_{Y}^{a}|^{2}\,dY+p^{2}\,dX+q^{2}\,dY \bigg)\\
	&\leq \int_{\Gamma}\bigg( \frac{1}{2}U^{2}\,(dX+dY)+|Z_{X}^{a}|^{2}\,dX+|Z_{Y}^{a}|^{2}\,dY+p^{2}\,dX+q^{2}\,dY \bigg)+2\delta\bar{C}A+2\delta A,
	\end{align*}	
	which yields, after taking the supremum over all curves $\bar{\Gamma}$,	
	\begin{equation*}
	(1-2\delta\bar{C}-2\delta)A\leq \int_{\Gamma}\bigg( \frac{1}{2}U^{2}\,(dX+dY)+|Z_{X}^{a}|^{2}\,dX+|Z_{Y}^{a}|^{2}\,dY+p^{2}\,dX+q^{2}\,dY \bigg)
	\end{equation*}
	and \eqref{eq:GronCurveProveThis} follows.	
	
	\textbf{Step 2.} For an arbitrarily large rectangular domain $\Omega=[X_{l},X_{r}]\times [Y_{l},Y_{r}]$, we now prove that \eqref{eq:GronCurveProveThis} holds for curves $\bar{\Gamma}=(\bar{\X},\bar{\Y})\in \C(\Omega)$ such that
	\begin{equation}
	\label{eq:barcurveovercurve}
	\bar{\Y}(s)-\bar{\X}(s)>\Y(s)-\X(s) \quad \text{for all } s\in (s_{l},s_{r}). 
	\end{equation}
	We claim that \eqref{eq:barcurveovercurve} implies that the curve $\bar{\Gamma}$ lies above $\Gamma$ and intersects $\Gamma$ only at the end points. From \eqref{eq:barcurveovercurve} and \eqref{eq:curvenormalization}, we find that $\bar{\Y}(s)>\Y(s)$ and $\X(s)>\bar{\X}(s)$. If $\bar{\X}(\bar{s})=\X(s)$ for some $\bar{s}\in(s_{l},s_{r})$, then $\bar{\X}(\bar{s})=\X(s)>\bar{\X}(s)$, so that $\bar{s}\geq s$ because $\bar{\X}$ is nondecreasing. This implies, since also $\bar{\Y}$ is nondecreasing, that $\Y(s)<\bar{\Y}(s)\leq\bar{\Y}(\bar{s})$. Hence, $\bar{\Gamma}$ is above $\Gamma$ except at the end points. The proof in the case when $\bar{\Gamma}$ is below $\Gamma$ is similar. For some constant $K>0$ that will be determined later, we have for almost every $X\in[X_{l},X_{r}]$, that
	\begin{align*}
	&e^{-K(\bar{\Y}(X)-X)}\bigg(x_{X}-\frac{1}{2}\bigg)^{2}(X,\bar{\Y}(X))-e^{-K(\Y(X)-X)}\bigg(x_{X}-\frac{1}{2}\bigg)^{2}(X,\Y(X))\\
	&=-K\int_{\Y(X)}^{\bar{\Y}(X)}e^{-K(Y-X)}\bigg(x_{X}-\frac{1}{2}\bigg)^{2}(X,Y)\,dY\\
	&\quad+\int_{\Y(X)}^{\bar{\Y}(X)}e^{-K(Y-X)}\bigg(2\bigg(x_{X}-\frac{1}{2}\bigg)x_{XY}\bigg)(X,Y)\,dY.
	\end{align*}
	By integrating over $[X_{l},X_{r}]$ and using \eqref{eq:gronwallxXest}, we get that
	\begin{align*}
	&\int_{\bar{\Gamma}}e^{-K(\bar{\Y}(X)-X)}\bigg(x_{X}-\frac{1}{2}\bigg)^{2}(X,\bar{\Y}(X))\,dX-\int_{\Gamma}e^{-K(\Y(X)-X)}\bigg(x_{X}-\frac{1}{2}\bigg)^{2}(X,\Y(X))\,dX\\
	&\leq-K\int_{X_{l}}^{X_{r}}\int_{\Y(X)}^{\bar{\Y}(X)}e^{-K(Y-X)}\bigg(x_{X}-\frac{1}{2}\bigg)^{2}(X,Y)\,dY\,dX\\
	&\quad+\kappa k_{1}\bigg(C+\frac{1}{2}\bigg) \int_{X_{l}}^{X_{r}}\int_{\Y(X)}^{\bar{\Y}(X)}e^{-K(Y-X)}\big(|Z_{X}^{a}|^{2}+|Z_{Y}^{a}|^{2}\big)(X,Y)\,dY\,dX.
	\end{align*}
	We treat the other components of $Z_{X}^{a}$ in the same way and obtain
	\begin{aalign}
	\label{eq:GronCurveStep2i}
	&\int_{\bar{\Gamma}}e^{-K(\bar{\Y}(X)-X)}|Z_{X}^{a}(X,\bar{\Y}(X))|^{2}\,dX-\int_{\Gamma}e^{-K(\Y(X)-X)}|Z_{X}^{a}(X,\Y(X))|^{2}\,dX\\
	&\leq-K\int_{X_{l}}^{X_{r}}\int_{\Y(X)}^{\bar{\Y}(X)}e^{-K(Y-X)}|Z_{X}^{a}(X,Y)|^{2}\,dY\,dX\\
	&\quad+M\int_{X_{l}}^{X_{r}}\int_{\Y(X)}^{\bar{\Y}(X)}e^{-K(Y-X)}\big(|Z_{X}^{a}|^{2}+|Z_{Y}^{a}|^{2}\big)(X,Y)\,dY\,dX,
	\end{aalign}
	where $M$ depends on $||(\bar{\X},\bar{\Y})||_{\C(\Omega)}$, $|||(Z,p,q)\bullet (\X,\Y)|||_{\G(\Omega)}$, $\kappa$ and $k_{1}$. Similarly, for $Z_{Y}^{a}$, we obtain 
	\begin{aalign}
	\label{eq:GronCurveStep2ii}
	&\int_{\bar{\Gamma}}e^{-K(Y-\bar{\X}(Y))}|Z_{Y}^{a}(\bar{\X}(Y),Y)|^{2}\,dY-\int_{\Gamma}e^{-K(Y-\X(Y))}|Z_{Y}^{a}(\X(Y),Y)|^{2}\,dY\\
	&\leq-K\int_{Y_{l}}^{Y_{r}}\int_{\bar{\X}(Y)}^{\X(Y)}e^{-K(Y-X)}|Z_{Y}^{a}(X,Y)|^{2}\,dX\,dY\\
	&\quad+M\int_{Y_{l}}^{Y_{r}}\int_{\bar{\X}(Y)}^{\X(Y)}e^{-K(Y-X)}\big(|Z_{X}^{a}|^{2}+|Z_{Y}^{a}|^{2}\big)(X,Y)\,dX\,dY.
	\end{aalign}
	We claim that the sets
	\begin{equation*}
	\N_{1}=\{(X,Y) \ | \ X_{l}<X<X_{r},\ \Y(X)<Y<\bar{\Y}(X)\}
	\end{equation*}
	and
	\begin{equation*}	
	\N_{2}=\{(X,Y) \ | \ Y_{l}<Y<Y_{r},\ \bar{\X}(Y)<X<\X(Y)\}
	\end{equation*}	
	are equal up to a set of zero measure. Let $(X,Y)\in \N_{1}$ and set $s_{1}=\X^{-1}(X)$, $s_{2}=\Y^{-1}(Y)$, $s_{3}=\bar{\Y}^{-1}(Y)$ and $s_{4}=\bar{\X}^{-1}(X)$. We have
	\begin{equation*}
	\Y(s_{1})=\Y(X)<Y=\Y(s_{2})=\bar{\Y}(s_{3})<\bar{\Y}(X)=\bar{\Y}(s_{4})
	\end{equation*}
	so that $s_{1}<s_{2}$ and $s_{3}<s_{4}$. It follows that 
	\begin{equation*}
	\bar{\X}(Y)=\bar{\X}(s_{3})\leq \bar{\X}(s_{4})=X=\X(s_{1})\leq \X(s_{2})=\X(Y).
	\end{equation*}
	Hence, $\N_{1}\subset \N_{2}$ up to a set of measure zero. Similarly, one proves the reverse inclusion. Now \eqref{eq:GronCurveStep2i} and \eqref{eq:GronCurveStep2ii} take the form
	\begin{aalign}
	\label{eq:gronwallZX}
	&\int_{\bar{\Gamma}}e^{-K(\bar{\Y}(X)-X)}|Z_{X}^{a}(X,\bar{\Y}(X))|^{2}\,dX-\int_{\Gamma}e^{-K(\Y(X)-X)}|Z_{X}^{a}(X,\Y(X))|^{2}\,dX\\
	&\leq-K\iint_{\N_{1}}e^{-K(Y-X)}|Z_{X}^{a}(X,Y)|^{2}\,dX\,dY\\
	&\quad+M\iint_{\N_{1}}e^{-K(Y-X)}\big(|Z_{X}^{a}|^{2}+|Z_{Y}^{a}|^{2}\big)(X,Y)\,dX\,dY
	\end{aalign}
	and 
	\begin{aalign}
	\label{eq:gronwallZY}
	&\int_{\bar{\Gamma}}e^{-K(Y-\bar{\X}(Y))}|Z_{Y}^{a}(\bar{\X}(Y),Y)|^{2}\,dY-\int_{\Gamma}e^{-K(Y-\X(Y))}|Z_{Y}^{a}(\X(Y),Y)|^{2}\,dY\\
	&\leq-K\iint_{\N_{1}}e^{-K(Y-X)}|Z_{Y}^{a}(X,Y)|^{2}\,dX\,dY\\
	&\quad+M\iint_{\N_{1}}e^{-K(Y-X)}\big(|Z_{X}^{a}|^{2}+|Z_{Y}^{a}|^{2}\big)(X,Y)\,dX\,dY.
	\end{aalign}
	A similar computation as above yields
	\begin{aalign}
	\label{eq:gronwallUX}
	&\int_{\bar{\Gamma}}e^{-K(\bar{\Y}(X)-X)}U^{2}(X,\bar{\Y}(X))\,dX-\int_{\Gamma}e^{-K(\Y(X)-X)}U^{2}(X,\Y(X))\,dX\\
	&=\iint_{\N_{1}}e^{-K(Y-X)}(-KU^{2}+2UU_{Y})(X,Y)\,dX\,dY\\
	&\leq\iint_{\N_{1}}e^{-K(Y-X)}(-KU^{2}+U^{2}+U_{Y}^{2})(X,Y)\,dX\,dY
	\end{aalign}
	and
	\begin{aalign}
	\label{eq:gronwallUY}
	&\int_{\bar{\Gamma}}e^{-K(Y-\bar{\X}(Y))}U^{2}(\bar{\X}(Y),Y)\,dY-\int_{\Gamma}e^{-K(Y-\X(Y))}U^{2}(\X(Y),Y)\,dY\\
	&\leq\iint_{\N_{1}}e^{-K(Y-X)}(-KU^{2}+U^{2}+U_{X}^{2})(X,Y)\,dX\,dY.
	\end{aalign}
	Furthermore, we have
	\begin{aalign}
	\label{eq:gronwallp}
	&\int_{\bar{\Gamma}}e^{-K(\bar{\Y}(X)-X)}p^{2}(X,\bar{\Y}(X))\,dX-\int_{\Gamma}e^{-K(\Y(X)-X)}p^{2}(X,\Y(X))\,dX\\
	&=-K\iint_{\N_{1}}e^{-K(Y-X)}p^{2}(X,Y)\,dX\,dY
	\end{aalign} 
	and
	\begin{aalign}
	\label{eq:gronwallq}
	&\int_{\bar{\Gamma}}e^{-K(Y-\bar{\X}(Y))}q^{2}(\bar{\X}(Y),Y)\,dY-\int_{\Gamma}e^{-K(Y-\X(Y))}q^{2}(\X(Y),Y)\,dY\\
	&=-K\iint_{\N_{1}}e^{-K(Y-X)}q^{2}(X,Y)\,dX\,dY.
	\end{aalign}
	By combining \eqref{eq:gronwallZX}-\eqref{eq:gronwallq}, we obtain
	\begin{align*}
	&\int_{\bar{\Gamma}}e^{-K(\bar{\Y}(X)-X)}\bigg(\frac{1}{2}U^{2}+|Z_{X}^{a}|^{2}+p^{2}\bigg)(X,\bar{\Y}(X))\,dX\\
	&\quad+\int_{\bar{\Gamma}}e^{-K(Y-\bar{\X}(Y))}\bigg(\frac{1}{2}U^{2}+|Z_{Y}^{a}|^{2}+q^{2}\bigg)(\bar{\X}(Y),Y)\,dY\\
	&\quad-\int_{\Gamma}e^{-K(\Y(X)-X)}\bigg(\frac{1}{2}U^{2}+|Z_{X}^{a}|^{2}+p^{2}\bigg)(X,\Y(X))\,dX\\
	&\quad-\int_{\Gamma}e^{-K(Y-\X(Y))}\bigg(\frac{1}{2}U^{2}+|Z_{Y}^{a}|^{2}+q^{2}\bigg)(\X(Y),Y)\,dY\\
	&\leq\iint_{\N_{1}}e^{-K(Y-X)}\bigg(-KU^{2}+U^{2}+\frac{1}{2}U_{Y}^{2}+\frac{1}{2}U_{X}^{2}-K|Z_{X}^{a}|^{2}-K|Z_{Y}^{a}|^{2}\\
	&\hspace{103pt}+2M|Z_{X}^{a}|^{2}+2M|Z_{Y}^{a}|^{2}-Kp^{2}-Kq^{2}\bigg)(X,Y)\,dX\,dY\\
	&\leq(2M+1-K)\iint_{\N_{1}}e^{-K(Y-X)}(U^{2}+|Z_{X}^{a}|^{2}+|Z_{Y}^{a}|^{2}+p^{2}+q^{2})(X,Y)\,dX\,dY.
	\end{align*}
	By choosing $K$ so large that the right-hand side is negative and get that
	\begin{align*}
	&e^{-K||(\bar{\X},\bar{\Y})||_{\C(\Omega)}}\int_{\bar{\Gamma}}\bigg( \frac{1}{2}U^{2}\,(dX+dY)+|Z_{X}^{a}|^{2}\,dX+|Z_{Y}^{a}|^{2}\,dY+p^{2}\,dX+q^{2}\,dY \bigg)\\
	&\leq e^{K||(\X,\Y)||_{\C(\Omega)}}\int_{\Gamma}\bigg( \frac{1}{2}U^{2}\,(dX+dY)+|Z_{X}^{a}|^{2}\,dX+|Z_{Y}^{a}|^{2}\,dY+p^{2}\,dX+q^{2}\,dY \bigg)
	\end{align*}
	and \eqref{eq:GronCurveProveThis} follows.
	
	\textbf{Step 3.} Given any rectangle $\Omega=[X_{l},X_{r}]\times [Y_{l},Y_{r}]$, we consider a sequence of rectangular domains $\Omega_{i}=[X_{i},X_{i+1}]\times [Y_{i},Y_{i+1}]$ for $i=0,\dots,N-1$ such that $X_{i}$ and $Y_{i}$ are increasing, $(X_{0}, Y_{0})=(X_{l}, Y_{l})$, $(X_{N}, Y_{N})=(X_{r}, Y_{r})$, and $(\X,\Y),(\bar{\X},\bar{\Y})\in \C(\Omega_{i})$ for $s\in [s_{i},s_{i+1}]$. We construct the sequence of rectangles such that either $s_{i+1}-s_{i}\leq \delta$ (and Step 1 applies) or $\bar{\Y}(s)-\bar{\X}(s)>\Y(s)-\X(s)$ or $\bar{\Y}(s)-\bar{\X}(s)<\Y(s)-\X(s)$ for $s\in (s_{i},s_{i+1})$ (and Step 2 applies). Then
	\begin{align*}
	||(Z,p,q)\bullet (\bar{\X},\bar{\Y})||_{\G(\Omega)}^{2}
	&=\sum_{i=0}^{N-1}||(Z,p,q)\bullet (\bar{\X},\bar{\Y})||_{\G(\Omega_{i})}^{2}\\
	&\leq \sum_{i=0}^{N-1}C||(Z,p,q)\bullet (\X,\Y)||_{\G(\Omega_{i})}^{2}\\
	&=C||(Z,p,q)\bullet (\X,\Y)||_{\G(\Omega)}^{2}.
	\end{align*}
\end{proof}

\begin{lemma}[Stability in $L^{2}$]
	Let $\Omega=[X_{l},X_{r}]\times [Y_{l},Y_{r}]$ and assume that $(Z,p,q)$, $(\tilde{Z},\tilde{p},\tilde{q})\in \H(\Omega)$ and $(\X,\Y)\in \C(\Omega)$. Then, for any $(\bar{\X},\bar{\Y})\in \C(\Omega)$, we have
	\begin{equation*}
	||(Z-\tilde{Z},p-\tilde{p},q-\tilde{q})\bullet (\bar{\X},\bar{\Y})||_{\G(\Omega)}\leq D||(Z-\tilde{Z},p-\tilde{p},q-\tilde{q})\bullet (\X,\Y)||_{\G(\Omega)},
	\end{equation*}
	where $D=D(||(\bar{\X},\bar{\Y})||_{\C(\Omega)},|||(Z,p,q)\bullet (\X,\Y)|||_{\G(\Omega)},|||(\tilde{Z},\tilde{p},\tilde{q})\bullet (\X,\Y)|||_{\G(\Omega)})$ is an increasing function with respect to all its arguments.
\end{lemma}

\begin{proof}
	As in the proof of Lemma \ref{lemma:groncurve}, we can consider the forms $(U-\tilde{U})^{2}\,dX$, $(U-\tilde{U})^{2}\,dY$, $|Z_{X}^{a}-\tilde{Z}_{X}^{a}|^{2}\,dX$, $|Z_{Y}^{a}-\tilde{Z}_{Y}^{a}|^{2}\,dY$, $(p-\tilde{p})^{2}\,dX$ and $(q-\tilde{q})^{2}\,dY$. For any curve $\bar{\Gamma}=(\bar{\X},\bar{\Y})\in \C(\Omega)$, we find
	\begin{align*}
	&||(Z-\tilde{Z},p-\tilde{p},q-\tilde{q})\bullet (\bar{\X},\bar{\Y})||_{\G(\Omega)}^{2}\\
	&=\int_{\bar{\Gamma}}\bigg( \frac{1}{2}(U-\tilde{U})^{2}\,(dX+dY)+|Z_{X}^{a}-\tilde{Z}_{X}^{a}|^{2}\,dX\\
	&\hspace{37pt}+|Z_{Y}^{a}-\tilde{Z}_{Y}^{a}|^{2}\,dY+(p-\tilde{p})^{2}\,dX+(q-\tilde{q})^{2}\,dY \bigg).
	\end{align*}  	
	Thus, we want to prove that
	\begin{aalign}
	\label{eq:GronCurveProveThis2}
	&\int_{\bar{\Gamma}}\bigg(\frac{1}{2}(U-\tilde{U})^{2}\,(dX+dY)+|Z_{X}^{a}-\tilde{Z}_{X}^{a}|^{2}\,dX\\
	&\hspace{23pt}+|Z_{Y}^{a}-\tilde{Z}_{Y}^{a}|^{2}\,dY+(p-\tilde{p})^{2}\,dX+(q-\tilde{q})^{2}\,dY \bigg)\\
	&\leq D\int_{\Gamma}\bigg( \frac{1}{2}(U-\tilde{U})^{2}\,(dX+dY)+|Z_{X}^{a}-\tilde{Z}_{X}^{a}|^{2}\,dX\\
	&\hspace{49pt}+|Z_{Y}^{a}-\tilde{Z}_{Y}^{a}|^{2}\,dY+(p-\tilde{p})^{2}\,dX+(q-\tilde{q})^{2}\,dY \bigg).
	\end{aalign} 	
	We decompose the proof into three steps.
	
	\textbf{Step 1.} We prove that \eqref{eq:GronCurveProveThis2} holds for small domains. We claim that there exist constants $\delta$ and $D$, which depend on 
	$||(\bar{\X},\bar{\Y})||_{\C(\Omega)}$, $|||(Z,p,q)\bullet (\X,\Y)|||_{\G(\Omega)}$ and $|||(\tilde{Z},\tilde{p},\tilde{q})\bullet (\X,\Y)|||_{\G(\Omega)}$, such that for any rectangular domain $\Omega=[X_{l},X_{r}]\times [Y_{l},Y_{r}]$ with $s_{r}-s_{l}\leq \delta$, \eqref{eq:GronCurveProveThis2} holds. By Lemma \ref{lemma:stripest}, we have
	\begin{aalign}
	\label{eq:CandCtilde}
	||U||_{\Linf(\Omega)}+||Z_{X}^{a}||_{\Linf(\Omega)}+||Z_{Y}^{a}||_{\Linf(\Omega)}+||p||_{\Linf(\Omega)}+||q||_{\Linf(\Omega)}&\leq C,\\
	||\tilde{U}||_{\Linf(\Omega)}+||\tilde{Z}_{X}^{a}||_{\Linf(\Omega)}+||\tilde{Z}_{Y}^{a}||_{\Linf(\Omega)}+||\tilde{p}||_{\Linf(\Omega)}+||\tilde{q}||_{\Linf(\Omega)}&\leq \tilde{C}
	\end{aalign}
	where $C=C(\Omega,|||(Z,p,q)\bullet (\X,\Y)|||_{\G(\Omega)})$ and $\tilde{C}=\tilde{C}(\Omega,|||(\tilde{Z},\tilde{p},\tilde{q})\bullet (\X,\Y)|||_{\G(\Omega)})$ are increasing with respect to the second argument. Let
	\begin{align*}
	A&=\sup_{\bar{\Gamma}}\int_{\bar{\Gamma}}\bigg(\frac{1}{2}(U-\tilde{U})^{2}\,(dX+dY)+|Z_{X}^{a}-\tilde{Z}_{X}^{a}|^{2}\,dX\\
	&\hspace{48pt}+|Z_{Y}^{a}-\tilde{Z}_{Y}^{a}|^{2}\,dY+(p-\tilde{p})^{2}\,dX+(q-\tilde{q})^{2}\,dY \bigg),
	\end{align*}
	where the supremum is taken over all $\bar{\Gamma}=(\bar{\X},\bar{\Y})\in \C(\Omega)$. We have
	\begin{aalign}
	\label{eq:L2stabilitystep1xXdiff}
	(x_{X}-\tilde{x}_{X})^{2}(X,Y)&=
	(x_{X}-\tilde{x}_{X})^{2}(X,\Y(X))\\
	&\quad+\int_{\Y(X)}^{Y}2(x_{X}-\tilde{x}_{X})(x_{XY}-\tilde{x}_{XY})(X,\tilde{Y})\,d\tilde{Y}.
	\end{aalign}
	Since $(Z,p,q)$ is a solution of \eqref{eq:goveq}, we have for almost every $X\in [X_{l},X_{r}]$ and all $Y\in [Y_{l},Y_{r}]$, that		
	\begin{align*}
	x_{XY}-\tilde{x}_{XY}&=\frac{c'(U)}{2c(U)}(U_{Y}x_{X}+U_{X}x_{Y})-\frac{c'(\tilde{U})}{2c(\tilde{U})}(\tilde{U}_{Y}\tilde{x}_{X}+\tilde{U}_{X}\tilde{x}_{Y})\\
	&=\frac{c'(U)}{2c(U)}(U_{Y}x_{X}+U_{X}x_{Y}-\tilde{U}_{Y}\tilde{x}_{X}-\tilde{U}_{X}\tilde{x}_{Y})\\
	&\quad+\bigg(\frac{c'(U)}{2c(U)}-\frac{c'(\tilde{U})}{2c(\tilde{U})}\bigg)(\tilde{U}_{Y}\tilde{x}_{X}+\tilde{U}_{X}\tilde{x}_{Y})\\
	&=\frac{c'(U)}{2c(U)}\bigg(U_{Y}(x_{X}-\tilde{x}_{X})+\bigg(\bigg(\tilde{x}_{X}-\frac{1}{2}\bigg)+\frac{1}{2}\bigg)(U_{Y}-\tilde{U}_{Y})\\
	&\hspace{55pt}+U_{X}(x_{Y}-\tilde{x}_{Y})+\bigg(\bigg(\tilde{x}_{Y}-\frac{1}{2}\bigg)+\frac{1}{2}\bigg)(U_{X}-\tilde{U}_{X})\bigg)\\
	&\quad+\frac{1}{2}\bigg(\tilde{U}_{Y}\bigg(\bigg(\tilde{x}_{X}-\frac{1}{2}\bigg)+\frac{1}{2}\bigg)+\tilde{U}_{X}\bigg(\bigg(\tilde{x}_{Y}-\frac{1}{2}\bigg)+\frac{1}{2}\bigg)\bigg)\\
	&\hspace{24pt}\times \int_{\tilde{U}}^{U}\bigg(\frac{c''(u)c(u)-c'(u)^{2}}{c(u)^{2}}\bigg)\,du.
	\end{align*}
	Using \eqref{eq:cassumption}, \eqref{eq:cderassumption} and \eqref{eq:CandCtilde}, this implies that
	\begin{align*}
	&\big|2(x_{X}-\tilde{x}_{X})(x_{XY}-\tilde{x}_{XY})\big|\\
	&\leq\kappa k_{1}\bigg(||Z_{Y}^{a}||_{\Linf(\Omega)}(x_{X}-\tilde{x}_{X})^{2}+\bigg(||\tilde{Z}_{X}^{a}||_{\Linf(\Omega)}+\frac{1}{2}\bigg)|U_{Y}-\tilde{U}_{Y}||x_{X}-\tilde{x}_{X}|\\
	&\hspace{35pt}+||Z_{X}^{a}||_{\Linf(\Omega)}|x_{Y}-\tilde{x}_{Y}||x_{X}-\tilde{x}_{X}|+\bigg(||\tilde{Z}_{Y}^{a}||_{\Linf(\Omega)}+\frac{1}{2}\bigg)|U_{X}-\tilde{U}_{X}||x_{X}-\tilde{x}_{X}|\bigg)\\
	&\quad+\bigg(||\tilde{Z}_{Y}^{a}||_{\Linf(\Omega)}\bigg(||\tilde{Z}_{X}^{a}||_{\Linf(\Omega)}+\frac{1}{2}\bigg)+||\tilde{Z}_{X}^{a}||_{\Linf(\Omega)}\bigg(||\tilde{Z}_{Y}^{a}||_{\Linf(\Omega)}+\frac{1}{2}\bigg)\bigg)\\
	&\hspace{20pt}\times(\kappa k_{2}+\kappa^{2}k_{1}^{2})|U-\tilde{U}||x_{X}-\tilde{x}_{X}|\\
	&\leq\kappa k_{1}\bigg(C+\tilde{C}+\frac{1}{2}\bigg)\bigg((x_{X}-\tilde{x}_{X})^{2}+|U_{Y}-\tilde{U}_{Y}||x_{X}-\tilde{x}_{X}|\\
	&\hspace{115pt}+|x_{Y}-\tilde{x}_{Y}||x_{X}-\tilde{x}_{X}|+|U_{X}-\tilde{U}_{X}||x_{X}-\tilde{x}_{X}|\bigg)\\
	&\quad+2\bigg(\tilde{C}+\frac{1}{2}\bigg)^{2}(\kappa k_{2}+\kappa^{2}k_{1}^{2})|U-\tilde{U}||x_{X}-\tilde{x}_{X}|\\
	&\leq\kappa k_{1}\bigg(C+\tilde{C}+\frac{1}{2}\bigg)\bigg((x_{X}-\tilde{x}_{X})^{2}+\frac{1}{2}(U_{Y}-\tilde{U}_{Y})^{2}+\frac{1}{2}(x_{X}-\tilde{x}_{X})^{2}\\
	&\hspace{96pt}+\frac{1}{2}(x_{Y}-\tilde{x}_{Y})^{2}+\frac{1}{2}(x_{X}-\tilde{x}_{X})^{2}+\frac{1}{2}(U_{X}-\tilde{U}_{X})^{2}+\frac{1}{2}(x_{X}-\tilde{x}_{X})^{2}\bigg)\\
	&\quad+2\bigg(\tilde{C}+\frac{1}{2}\bigg)^{2}(\kappa k_{2}+\kappa^{2}k_{1}^{2})\bigg(\frac{1}{2}(U-\tilde{U})^{2}+\frac{1}{2}(x_{X}-\tilde{x}_{X})^{2}\bigg)\\
	&\leq\frac{5}{2}\kappa k_{1}\bigg(C+\tilde{C}+\frac{1}{2}\bigg)\bigg((x_{X}-\tilde{x}_{X})^{2}+(U_{Y}-\tilde{U}_{Y})^{2}+(x_{Y}-\tilde{x}_{Y})^{2}+(U_{X}-\tilde{U}_{X})^{2}\bigg)\\
	&\quad+\bigg(\tilde{C}+\frac{1}{2}\bigg)^{2}(\kappa k_{2}+\kappa^{2}k_{1}^{2})\bigg((U-\tilde{U})^{2}+(x_{X}-\tilde{x}_{X})^{2}\bigg). 
	\end{align*}
	We set
	\begin{equation*}
	m=\max\bigg\{\frac{5}{2}\kappa k_{1}\bigg(C+\tilde{C}+\frac{1}{2}\bigg),\bigg(\tilde{C}+\frac{1}{2}\bigg)^{2}(\kappa k_{2}+\kappa^{2}k_{1}^{2})\bigg\}
	\end{equation*}
	and get
	\begin{aalign}
	\label{eq:L2stabilityleq2m}
	&\big|2(x_{X}-\tilde{x}_{X})(x_{XY}-\tilde{x}_{XY})\big|\\
	&\leq m\big((U_{Y}-\tilde{U}_{Y})^{2}+(x_{Y}-\tilde{x}_{Y})^{2}+(U_{X}-\tilde{U}_{X})^{2}+(U-\tilde{U})^{2}\big)\\
	&\quad+2m(x_{X}-\tilde{x}_{X})^{2}\\
	&\leq 2m\big((U-\tilde{U})^{2}+|Z_{X}^{a}-\tilde{Z}_{X}^{a}|^{2}+|Z_{Y}^{a}-\tilde{Z}_{Y}^{a}|^{2}\big).
	\end{aalign}
	We insert this into \eqref{eq:L2stabilitystep1xXdiff} and get
	\begin{align*}
	(x_{X}-\tilde{x}_{X})^{2}(X,Y)&\leq
	(x_{X}-\tilde{x}_{X})^{2}(X,\Y(X))\\
	&\quad+2m\int_{Y_{l}}^{Y_{r}}\big((U-\tilde{U})^{2}+|Z_{X}^{a}-\tilde{Z}_{X}^{a}|^{2}+|Z_{Y}^{a}-\tilde{Z}_{Y}^{a}|^{2}\big)(X,\tilde{Y})\,d\tilde{Y}
	\end{align*}	
	which, after integrating over $[X_{l},X_{r}]$, yields
	\begin{align*}
	\int_{\bar{\Gamma}}(x_{X}-\tilde{x}_{X})^{2}\,dX&\leq
	\int_{\Gamma}(x_{X}-\tilde{x}_{X})^{2}\,dX\\
	&\quad+2m\int_{X_{l}}^{X_{r}}\int_{Y_{l}}^{Y_{r}}\big((U-\tilde{U})^{2}+|Z_{X}^{a}-\tilde{Z}_{X}^{a}|^{2}+|Z_{Y}^{a}-\tilde{Z}_{Y}^{a}|^{2}\big)\,dY\,dX.
	\end{align*}
	For any $Y\in [Y_{l},Y_{r}]$, $\int_{X_{l}}^{X_{r}}\big(\frac{1}{2}(U-\tilde{U})^{2}+|Z_{X}^{a}-\tilde{Z}_{X}^{a}|^{2}\big)(X,Y)\,dX$ can be seen as part of the integral of the form $\big(\frac{1}{2}(U-\tilde{U})^{2}+|Z_{X}^{a}-\tilde{Z}_{X}^{a}|^{2}\big)\,dX$ on the piecewise linear path going through the points $(X_{l},Y_{l})$, $(X_{l},Y)$, $(X_{r},Y)$ and $(X_{r},Y_{r})$. This implies that $\int_{X_{l}}^{X_{r}}\big(\frac{1}{2}(U-\tilde{U})^{2}+|Z_{X}^{a}-\tilde{Z}_{X}^{a}|^{2}\big)(X,Y)\,dX\leq A$. Similarly, for any $X\in [X_{l},X_{r}]$, we get $\int_{Y_{l}}^{Y_{r}}\big(\frac{1}{2}(U-\tilde{U})^{2}+|Z_{Y}^{a}-\tilde{Z}_{Y}^{a}|^{2}\big)(X,Y)\,dY\leq A$. Hence,
	\begin{align*}
	\int_{\bar{\Gamma}}\big(Z_{2,X}^{a}-\tilde{Z}_{2,X}^{a}\big)^{2}\,dX&\leq
	\int_{\Gamma}\big(Z_{2,X}^{a}-\tilde{Z}_{2,X}^{a}\big)^{2}\,dX+2mA(Y_{r}-Y_{l}+X_{r}-X_{l})\\
	&\leq \int_{\Gamma}\big(Z_{2,X}^{a}-\tilde{Z}_{2,X}^{a}\big)^{2}\,dX+4\delta mA
	\end{align*}
	because $Y_{r}-Y_{l}+X_{r}-X_{l}=2(s_{r}-s_{l})$. By treating the other components of $Z_{X}^{a}-\tilde{Z}_{X}^{a}$ and $Z_{Y}^{a}-\tilde{Z}_{Y}^{a}$ similarly, we get
	\begin{align}
	\label{eq:L2stabilityZXineq}
	\int_{\bar{\Gamma}}|Z_{X}^{a}-\tilde{Z}_{X}^{a}|^{2}\,dX&\leq \int_{\Gamma}|Z_{X}^{a}-\tilde{Z}_{X}^{a}|^{2}\,dX+MA\delta,\\
	\label{eq:L2stabilityZYineq}
	\int_{\bar{\Gamma}}|Z_{Y}^{a}-\tilde{Z}_{Y}^{a}|^{2}\,dY&\leq \int_{\Gamma}|Z_{Y}^{a}-\tilde{Z}_{Y}^{a}|^{2}\,dY+MA\delta,
	\end{align}
	where $M$ depends on $||(\bar{\X},\bar{\Y})||_{\C(\Omega)}$, $|||(Z,p,q)\bullet (\X,\Y)|||_{\G(\Omega)}$, $|||(\tilde{Z},\tilde{p},\tilde{q})\bullet (\X,\Y)|||_{\G(\Omega)}$, $\kappa$, $k_{1}$ and $k_{2}$. For almost every $X\in[X_{l},X_{r}]$ and all $Y\in[Y_{l},Y_{r}]$, we have $p_{Y}(X,Y)=0$ and $\tilde{p}_{Y}(X,Y)=0$, so that $p(X,Y)-\tilde{p}(X,Y)=p(X,\Y(X))-\tilde{p}(X,\Y(X))$. By squaring this expression and integrating over $[X_{l},X_{r}]$, we obtain
	\begin{equation}
	\label{eq:L2stabilitypineq}
	\int_{\bar{\Gamma}}(p-\tilde{p})^{2}\,dX=\int_{\Gamma}(p-\tilde{p})^{2}\,dX.
	\end{equation}
	Since $q_{X}(X,Y)=0$ and $\tilde{q}_{X}(X,Y)=0$ for almost every $Y\in[Y_{l},Y_{r}]$ and all $X\in[X_{l},X_{r}]$, we get $q(X,Y)-\tilde{q}(X,Y)=q(\X(Y),Y)-\tilde{q}(\X(Y),Y)$ which implies, after squaring and integrating over $[Y_{l},Y_{r}]$, that
	\begin{equation} 
	\label{eq:L2stabilityqineq}
	\int_{\bar{\Gamma}}(q-\tilde{q})^{2}\,dY=\int_{\Gamma}(q-\tilde{q})^{2}\,dY.
	\end{equation}
	We have
	\begin{align*}
	&(U-\tilde{U})^{2}(X,Y)\\
	&=(U-\tilde{U})^{2}(X,\Y(X))+2\int_{\Y(X)}^{Y}(U-\tilde{U})(U_{Y}-\tilde{U}_{Y})(X,\tilde{Y})\,d\tilde{Y}\\
	&\leq (U-\tilde{U})^{2}(X,\Y(X))+\int_{\Y(X)}^{Y}(U-\tilde{U})^{2}(X,\tilde{Y})\,d\tilde{Y}+\int_{\Y(X)}^{Y}(U_{Y}-\tilde{U}_{Y})^{2}(X,\tilde{Y})\,d\tilde{Y}.
	\end{align*}
	As above, it follows that
	\begin{equation}
	\label{eq:L2stabilityU1ineq}
	\int_{\bar{\Gamma}}(U-\tilde{U})^{2}\,dX\leq \int_{\Gamma}(U-\tilde{U})^{2}\,dX+2A\delta.
	\end{equation}
	Similarly, we obtain	
	\begin{equation}
	\label{eq:L2stabilityU2ineq}
	\int_{\bar{\Gamma}}(U-\tilde{U})^{2}\,dY\leq \int_{\Gamma}(U-\tilde{U})^{2}\,dY+2A\delta.
	\end{equation}
	By adding \eqref{eq:L2stabilityZXineq}-\eqref{eq:L2stabilityU2ineq}, we obtain
	\begin{align*}
	&\int_{\bar{\Gamma}}\bigg(\frac{1}{2}(U-\tilde{U})^{2}\,(dX+dY)+|Z_{X}^{a}-\tilde{Z}_{X}^{a}|^{2}\,dX\\
	&\hspace{23pt}+|Z_{Y}^{a}-\tilde{Z}_{Y}^{a}|^{2}\,dY+(p-\tilde{p})^{2}\,dX+(q-\tilde{q})^{2}\,dY \bigg)\\
	&\leq \int_{\Gamma}\bigg( \frac{1}{2}(U-\tilde{U})^{2}\,(dX+dY)+|Z_{X}^{a}-\tilde{Z}_{X}^{a}|^{2}\,dX\\
	&\hspace{37pt}+|Z_{Y}^{a}-\tilde{Z}_{Y}^{a}|^{2}\,dY+(p-\tilde{p})^{2}\,dX+(q-\tilde{q})^{2}\,dY \bigg)+2\delta MA+2\delta A
	\end{align*}	
	which yields, after taking the supremum over all curves $\bar{\Gamma}$,	
	\begin{align*}
	(1-2\delta M-2\delta)A&\leq \int_{\Gamma}\bigg( \frac{1}{2}(U-\tilde{U})^{2}\,(dX+dY)+|Z_{X}^{a}-\tilde{Z}_{X}^{a}|^{2}\,dX\\
	&\hspace{37pt}+|Z_{Y}^{a}-\tilde{Z}_{Y}^{a}|^{2}\,dY+(p-\tilde{p})^{2}\,dX+(q-\tilde{q})^{2}\,dY \bigg)
	\end{align*}
	and \eqref{eq:GronCurveProveThis2} follows.
	
	\textbf{Step 2.} For an arbitrarily large rectangular domain $\Omega=[X_{l},X_{r}]\times [Y_{l},Y_{r}]$, we now prove that \eqref{eq:GronCurveProveThis2} holds for curves $\bar{\Gamma}=(\bar{\X},\bar{\Y})\in \C(\Omega)$ such that
	\begin{equation*}
	\bar{\Y}(s)-\bar{\X}(s)>\Y(s)-\X(s) \quad \text{for all } s\in (s_{l},s_{r}), 
	\end{equation*}
	that is, the curve $\bar{\Gamma}$ lies above $\Gamma$ and intersects $\Gamma$ only at the end points, as in the proof of Lemma~\ref{lemma:groncurve}. The proof in the case when $\bar{\Gamma}$ is below $\Gamma$ is similar. For a constant $K>0$ that will be determined later, we have for almost every $X\in[X_{l},X_{r}]$, that
	\begin{align*}
	&e^{-K(\bar{\Y}(X)-X)}(x_{X}-\tilde{x}_{X})^{2}(X,\bar{\Y}(X))-e^{-K(\Y(X)-X)}(x_{X}-\tilde{x}_{X})^{2}(X,\Y(X))\\
	&=-K\int_{\Y(X)}^{\bar{\Y}(X)}e^{-K(Y-X)}(x_{X}-\tilde{x}_{X})^{2}(X,Y)\,dY\\
	&\quad+\int_{\Y(X)}^{\bar{\Y}(X)}e^{-K(Y-X)}(2(x_{X}-\tilde{x}_{X})(x_{XY}-\tilde{x}_{XY}))(X,Y)\,dY.
	\end{align*}
	We integrate over $[X_{l},X_{r}]$ and get, by using \eqref{eq:L2stabilityleq2m}, that
	\begin{align*}
	&\int_{\bar{\Gamma}}e^{-K(\bar{\Y}(X)-X)}(x_{X}-\tilde{x}_{X})^{2}(X,\bar{\Y}(X))\,dX-\int_{\Gamma}e^{-K(\Y(X)-X)}(x_{X}-\tilde{x}_{X})^{2}(X,\Y(X))\,dX\\
	&\leq-K\int_{X_{l}}^{X_{r}}\int_{\Y(X)}^{\bar{\Y}(X)}e^{-K(Y-X)}(x_{X}-\tilde{x}_{X})^{2}(X,Y)\,dY\,dX\\
	&\quad+2m\int_{X_{l}}^{X_{r}}\int_{\Y(X)}^{\bar{\Y}(X)}e^{-K(Y-X)}
	\big((U-\tilde{U})^{2}+|Z_{X}^{a}-\tilde{Z}_{X}^{a}|^{2}+|Z_{Y}^{a}-\tilde{Z}_{Y}^{a}|^{2}\big)(X,Y)
	\,dY\,dX.
	\end{align*} 
	By treating the other components of $Z_{X}^{a}-\tilde{Z}_{X}^{a}$ in the same way, we obtain
	\begin{align*}
	&\int_{\bar{\Gamma}}e^{-K(\bar{\Y}(X)-X)}|Z_{X}^{a}-\tilde{Z}_{X}^{a}|^{2}(X,\bar{\Y}(X))\,dX-\int_{\Gamma}e^{-K(\Y(X)-X)}|Z_{X}^{a}-\tilde{Z}_{X}^{a}|^{2}(X,\Y(X))\,dX\\
	&\leq-K\int_{X_{l}}^{X_{r}}\int_{\Y(X)}^{\bar{\Y}(X)}e^{-K(Y-X)}|Z_{X}^{a}-\tilde{Z}_{X}^{a}|^{2}(X,Y)\,dY\,dX\\
	&\quad+M\int_{X_{l}}^{X_{r}}\int_{\Y(X)}^{\bar{\Y}(X)}e^{-K(Y-X)}
	\big((U-\tilde{U})^{2}+|Z_{X}^{a}-\tilde{Z}_{X}^{a}|^{2}+|Z_{Y}^{a}-\tilde{Z}_{Y}^{a}|^{2}\big)(X,Y)
	\,dY\,dX
	\end{align*}
	where $M$ depends on $||(\bar{\X},\bar{\Y})||_{\C(\Omega)}$, $|||(Z,p,q)\bullet (\X,\Y)|||_{\G(\Omega)}$, $|||(\tilde{Z},\tilde{p},\tilde{q})\bullet (\X,\Y)|||_{\G(\Omega)}$, $\kappa$, $k_{1}$ and $k_{2}$. Similarly, for $Z_{Y}^{a}-\tilde{Z}_{Y}^{a}$, we get
	\begin{align*}
	&\int_{\bar{\Gamma}}e^{-K(Y-\bar{\X}(Y))}|Z_{Y}^{a}-\tilde{Z}_{Y}^{a}|^{2}(\bar{\X}(Y),Y)\,dY-\int_{\Gamma}e^{-K(Y-\X(Y))}|Z_{Y}^{a}-\tilde{Z}_{Y}^{a}|^{2}(\X(Y),Y)\,dY\\
	&\leq-K\int_{Y_{l}}^{Y_{r}}\int_{\bar{\X}(Y)}^{\X(Y)}e^{-K(Y-X)}|Z_{Y}^{a}-\tilde{Z}_{Y}^{a}|^{2}(X,Y)\,dX\,dY\\
	&\quad+M\int_{Y_{l}}^{Y_{r}}\int_{\bar{\X}(Y)}^{\X(Y)}e^{-K(Y-X)}
	\big((U-\tilde{U})^{2}+|Z_{X}^{a}-\tilde{Z}_{X}^{a}|^{2}+|Z_{Y}^{a}-\tilde{Z}_{Y}^{a}|^{2}\big)(X,Y)
	\,dX\,dY.
	\end{align*}
	In the proof of Lemma \ref{lemma:groncurve}, we showed that the sets
	\begin{equation*}
	\N_{1}=\{(X,Y) \ | \ X_{l}<X<X_{r},\ \Y(X)<Y<\bar{\Y}(X)\}
	\end{equation*}
	and
	\begin{equation*}	
	\N_{2}=\{(X,Y) \ | \ Y_{l}<Y<Y_{r},\ \bar{\X}(Y)<X<\X(Y)\}
	\end{equation*}	
	are equal up to a set of zero measure. Hence, we get
	\begin{aalign}
	\label{eq:stabilityL2expZX}
	&\int_{\bar{\Gamma}}e^{-K(\bar{\Y}(X)-X)}|Z_{X}^{a}-\tilde{Z}_{X}^{a}|^{2}(X,\bar{\Y}(X))\,dX\\
	&-\int_{\Gamma}e^{-K(\Y(X)-X)}|Z_{X}^{a}-\tilde{Z}_{X}^{a}|^{2}(X,\Y(X))\,dX\\
	&\leq-K\iint_{\N_{1}}e^{-K(Y-X)}|Z_{X}^{a}-\tilde{Z}_{X}^{a}|^{2}(X,Y)\,dX\,dY\\
	&\quad+M\iint_{\N_{1}}e^{-K(Y-X)}
	\big((U-\tilde{U})^{2}+|Z_{X}^{a}-\tilde{Z}_{X}^{a}|^{2}\\
	&\hspace{125pt}+|Z_{Y}^{a}-\tilde{Z}_{Y}^{a}|^{2}\big)(X,Y)
	\,dX\,dY
	\end{aalign}
	and
	\begin{aalign}
	\label{eq:stabilityL2expZY}
	&\int_{\bar{\Gamma}}e^{-K(Y-\bar{\X}(Y))}|Z_{Y}^{a}-\tilde{Z}_{Y}^{a}|^{2}(\bar{\X}(Y),Y)\,dY\\
	&-\int_{\Gamma}e^{-K(Y-\X(Y))}|Z_{Y}^{a}-\tilde{Z}_{Y}^{a}|^{2}(\X(Y),Y)\,dY\\
	&\leq-K\iint_{\N_{1}}e^{-K(Y-X)}|Z_{Y}^{a}-\tilde{Z}_{Y}^{a}|^{2}(X,Y)\,dX\,dY\\
	&\quad+M\iint_{\N_{1}}e^{-K(Y-X)}
	\big((U-\tilde{U})^{2}+|Z_{X}^{a}-\tilde{Z}_{X}^{a}|^{2}\\
	&\hspace{125pt}+|Z_{Y}^{a}-\tilde{Z}_{Y}^{a}|^{2}\big)(X,Y)
	\,dX\,dY.
	\end{aalign}
	A similar computation as above yields
	\begin{aalign}
	\label{eq:stabilityL2expUX}
	&\int_{\bar{\Gamma}}e^{-K(\bar{\Y}(X)-X)}(U-\tilde{U})^{2}(X,\bar{\Y}(X))\,dX\\
	&-\int_{\Gamma}e^{-K(\Y(X)-X)}(U-\tilde{U})^{2}(X,\Y(X))\,dX\\
	&=\iint_{\N_{1}}e^{-K(Y-X)}(-K(U-\tilde{U})^{2}+2(U-\tilde{U})(U_{Y}-\tilde{U}_{Y}))(X,Y)\,dX\,dY\\
	&\leq\iint_{\N_{1}}e^{-K(Y-X)}(-K(U-\tilde{U})^{2}+(U-\tilde{U})^{2}+(U_{Y}-\tilde{U}_{Y})^{2})(X,Y)\,dX\,dY
	\end{aalign}
	and
	\begin{aalign}
	\label{eq:stabilityL2expUY}
	&\int_{\bar{\Gamma}}e^{-K(Y-\bar{\X}(Y))}(U-\tilde{U})^{2}(\bar{\X}(Y),Y)\,dY\\
	&-\int_{\Gamma}e^{-K(Y-\X(Y))}(U-\tilde{U})^{2}(\X(Y),Y)\,dY\\
	&\leq\iint_{\N_{1}}e^{-K(Y-X)}(-K(U-\tilde{U})^{2}+(U-\tilde{U})^{2}+(U_{X}-\tilde{U}_{X})^{2})(X,Y)\,dX\,dY.
	\end{aalign}
	Furthermore, we have
	\begin{aalign}
	\label{eq:stabilityL2expp}
	&\int_{\bar{\Gamma}}e^{-K(\bar{\Y}(X)-X)}(p-\tilde{p})^{2}(X,\bar{\Y}(X))\,dX\\
	&-\int_{\Gamma}e^{-K(\Y(X)-X)}(p-\tilde{p})^{2}(X,\Y(X))\,dX\\
	&=-K\iint_{\N_{1}}e^{-K(Y-X)}(p-\tilde{p})^{2}(X,Y)\,dX\,dY
	\end{aalign} 
	and
	\begin{aalign}
	\label{eq:stabilityL2expq}
	&\int_{\bar{\Gamma}}e^{-K(Y-\bar{\X}(Y))}(q-\tilde{q})^{2}(\bar{\X}(Y),Y)\,dY\\
	&-\int_{\Gamma}e^{-K(Y-\X(Y))}(q-\tilde{q})^{2}(\X(Y),Y)\,dY\\
	&=-K\iint_{\N_{1}}e^{-K(Y-X)}(q-\tilde{q})^{2}(X,Y)\,dX\,dY.
	\end{aalign} 
	Combining \eqref{eq:stabilityL2expZX}-\eqref{eq:stabilityL2expq}, we obtain
	\begin{align*}
	&\int_{\bar{\Gamma}}e^{-K(\bar{\Y}(X)-X)}\bigg(\frac{1}{2}(U-\tilde{U})^{2}+|Z_{X}^{a}-\tilde{Z}_{X}^{a}|^{2}+(p-\tilde{p})^{2}\bigg)(X,\bar{\Y}(X))\,dX\\
	&\quad+\int_{\bar{\Gamma}}e^{-K(Y-\bar{\X}(Y))}\bigg(\frac{1}{2}(U-\tilde{U})^{2}+|Z_{Y}^{a}-\tilde{Z}_{Y}^{a}|^{2}+(q-\tilde{q})^{2}\bigg)(\bar{\X}(Y),Y)\,dY\\
	&\quad-\int_{\Gamma}e^{-K(\Y(X)-X)}\bigg(\frac{1}{2}(U-\tilde{U})^{2}+|Z_{X}^{a}-\tilde{Z}_{X}^{a}|^{2}+(p-\tilde{p})^{2}\bigg)(X,\Y(X))\,dX\\
	&\quad-\int_{\Gamma}e^{-K(Y-\X(Y))}\bigg(\frac{1}{2}(U-\tilde{U})^{2}+|Z_{Y}^{a}-\tilde{Z}_{Y}^{a}|^{2}+(q-\tilde{q})^{2}\bigg)(\X(Y),Y)\,dY\\
	&\leq\iint_{\N_{1}}e^{-K(Y-X)}\bigg(-K(U-\tilde{U})^{2}+(U-\tilde{U})^{2}+\frac{1}{2}(U_{Y}-\tilde{U}_{Y})^{2}\\
	&\hspace{102pt}+\frac{1}{2}(U_{X}-\tilde{U}_{X})^{2}-K|Z_{X}^{a}-\tilde{Z}_{X}^{a}|^{2}-K|Z_{Y}^{a}-\tilde{Z}_{Y}^{a}|^{2}\\
	&\hspace{102pt}+2M(U-\tilde{U})^{2}+2M|Z_{X}^{a}-\tilde{Z}_{X}^{a}|^{2}+2M|Z_{Y}^{a}-\tilde{Z}_{Y}^{a}|^{2}\\
	&\hspace{102pt}-K(p-\tilde{p})^{2}-K(q-\tilde{q})^{2}\bigg)(X,Y)\,dX\,dY\\
	&\leq(2M+1-K)\iint_{\N_{1}}e^{-K(Y-X)}\bigg((U-\tilde{U})^{2}+|Z_{X}^{a}-\tilde{Z}_{X}^{a}|^{2}+|Z_{Y}^{a}-\tilde{Z}_{Y}^{a}|^{2}\\
	&\hspace{175pt}+(p-\tilde{p})^{2}+(q-\tilde{q})^{2}\bigg)(X,Y)\,dX\,dY.
	\end{align*}
	We choose $K$ so large that the right-hand side becomes negative. This implies that
	\begin{align*}
	&e^{-K||(\bar{\X},\bar{\Y})||_{\C(\Omega)}}\int_{\bar{\Gamma}}\bigg(\frac{1}{2}(U-\tilde{U})^{2}\,(dX+dY)+|Z_{X}^{a}-\tilde{Z}_{X}^{a}|^{2}\,dX\\
	&\hspace{90pt}+|Z_{Y}^{a}-\tilde{Z}_{Y}^{a}|^{2}\,dY+(p-\tilde{p})^{2}\,dX+(q-\tilde{q})^{2}\,dY \bigg)\\
	&\leq e^{K||(\X,\Y)||_{\C(\Omega)}}\int_{\Gamma}\bigg( \frac{1}{2}(U-\tilde{U})^{2}\,(dX+dY)+|Z_{X}^{a}-\tilde{Z}_{X}^{a}|^{2}\,dX\\
	&\hspace{100pt}+|Z_{Y}^{a}-\tilde{Z}_{Y}^{a}|^{2}\,dY+(p-\tilde{p})^{2}\,dX+(q-\tilde{q})^{2}\,dY \bigg)
	\end{align*}	
	and \eqref{eq:GronCurveProveThis2} follows.	
	
	\textbf{Step 3.} Given any rectangle $\Omega=[X_{l},X_{r}]\times [Y_{l},Y_{r}]$, we consider a sequence of rectangular domains $\Omega_{i}=[X_{i},X_{i+1}]\times [Y_{i},Y_{i+1}]$ for $i=0,\dots,N-1$ such that $X_{i}$ and $Y_{i}$ are increasing, $(X_{0}, Y_{0})=(X_{l}, Y_{l})$, $(X_{N}, Y_{N})=(X_{r}, Y_r)$, and $(\X,\Y),(\bar{\X},\bar{\Y})\in \C(\Omega_{i})$ for $s\in [s_{i},s_{i+1}]$. We construct the sequence of rectangles such that either $s_{i+1}-s_{i}\leq \delta$ (and Step 1 applies) or $\bar{\Y}(s)-\bar{\X}(s)>\Y(s)-\X(s)$ or $\bar{\Y}(s)-\bar{\X}(s)<\Y(s)-\X(s)$ for $s\in (s_{i},s_{i+1})$ (and Step 2 applies). Then
	\begin{align*}
	||(Z-\tilde{Z},p-\tilde{p},q-\tilde{q})\bullet (\bar{\X},\bar{\Y})||_{\G(\Omega)}^{2}
	&=\sum_{i=0}^{N-1}||(Z-\tilde{Z},p-\tilde{p},q-\tilde{q})\bullet (\bar{\X},\bar{\Y})||_{\G(\Omega_{i})}^{2}\\
	&\leq \sum_{i=0}^{N-1}D||(Z-\tilde{Z},p-\tilde{p},q-\tilde{q})\bullet (\X,\Y)||_{\G(\Omega_{i})}^{2}\\
	&=D||(Z-\tilde{Z},p-\tilde{p},q-\tilde{q})\bullet (\X,\Y)||_{\G(\Omega)}^{2}.
	\end{align*}	
\end{proof}

\subsection{Existence of Global Solutions in $\H$}

\begin{definition}[Global solutions]
	\label{def:globalsolutions}	
	Let $\H$ be the set of all functions $(Z,p,q)$ such that
	\begin{enumerate}
		\item[(i)] $(Z,p,q)\in \H(\Omega)$ for all rectangular domains $\Omega$;
		\item[(ii)] there exists a curve $(\X,\Y)\in \C$ such that $(Z,p,q)\bullet (\X,\Y)\in \G$. 
	\end{enumerate}
\end{definition}

The following lemma shows that condition (ii) does not depend on the particular curve for which it holds. In particular, we can replace this condition by the requirement that $(Z,p,q)\bullet (\X_{d},\Y_{d})\in \G$ for the diagonal, which is given by $\X_{d}(s)=\Y_{d}(s)=s$.

\begin{lemma}
	\label{lemma:curveind}
	Given $(Z,p,q)\in \H$, we have $(Z,p,q)\bullet (\X,\Y)\in \G$ for any curve $(\X,\Y)\in \C$. Moreover, the limit $\displaystyle\lim_{s\rightarrow \infty}J(\X(s),\Y(s))$ is independent of the curve $(\X,\Y)\in \C$. 
\end{lemma}

\begin{proof}
	Since $(Z,p,q)\in \H$, we know that there exists a curve $(\X,\Y)\in \C$ such that $(Z,p,q)\, \bullet \,(\X,\Y)\in \G$. We have to check that the conditions (i)-(v) of Definition \ref{def:setG} are satisfied for $\bar{\Theta}=(Z,p,q)\,\bullet \,(\bar{\X},\bar{\Y})$. For any curve $(\bar{\X},\bar{\Y})\in \C$, we have to prove that
	\begin{equation}
	\label{eq:anycurveprovethis1}
	|||(Z,p,q)\bullet (\bar{\X},\bar{\Y})|||_{\G}<\infty \quad \text{and} \quad ||(Z,p,q)\bullet (\bar{\X},\bar{\Y})||_{\G}<\infty.
	\end{equation}
	For any positive number $\bar{s}$, we denote $\Omega_{\bar{s}}=[\bar{\X}(-\bar{s}),\bar{\X}(\bar{s})]\times [\bar{\Y}(-\bar{s}),\bar{\Y}(\bar{s})]$. Let 
	\begin{equation*}
	s_{\text{max}}=\begin{cases}
	\Y^{-1}(\bar{\Y}(\bar{s})) & \text{if } \Y(\X^{-1}(\bar{\X}(\bar{s})))\leq \bar{\Y}(\bar{s}) \\
	\X^{-1}(\bar{\X}(\bar{s})) & \text{otherwise}
	\end{cases}
	\end{equation*}
	and
	\begin{equation*}
	s_{\text{min}}=\begin{cases}
	\X^{-1}(\bar{\X}(-\bar{s})) & \text{if } \bar{\X}(-\bar{s})\leq \X(\Y^{-1}(\bar{\Y}(-\bar{s}))) \\
	\Y^{-1}(\bar{\Y}(-\bar{s})) & \text{otherwise}.
	\end{cases}
	\end{equation*}
	We claim that $s_{\text{min}}\leq-\bar{s}\leq\bar{s}\leq s_{\text{max}}$. 
	
	If $\Y(\X^{-1}(\bar{\X}(\bar{s})))\leq \bar{\Y}(\bar{s})$, then $\X^{-1}(\bar{\X}(\bar{s}))\leq \Y^{-1}(\bar{\Y}(\bar{s}))=s_{\text{max}}$, so that $\bar{\X}(\bar{s})\leq \X(s_{\text{max}})$. This implies, since $\Y(s_{\text{max}})=\bar{\Y}(\bar{s})$, that
	\begin{equation*}
	2\bar{s}-\bar{\Y}(\bar{s})\leq 2s_{\text{max}}-\Y(s_{\text{max}})=2s_{\text{max}}-\bar{\Y}(\bar{s})
	\end{equation*}
	and we conclude that $\bar{s}\leq s_{\text{max}}$. 
	
	If $\bar{\X}(-\bar{s})\leq \X(\Y^{-1}(\bar{\Y}(-\bar{s})))$, then $s_{\text{min}}=\X^{-1}(\bar{\X}(-\bar{s}))\leq\Y^{-1}(\bar{\Y}(-\bar{s}))$ and we get $\Y(s_{\text{min}})\leq \bar{\Y}(-\bar{s})$, which implies, since $\X(s_{\text{min}})=\bar{\X}(-\bar{s})$, that
	\begin{equation*}
	2s_{\text{min}}-\X(s_{\text{min}})\leq -2\bar{s}-\bar{\X}(-\bar{s})=-2\bar{s}-\X(s_{\text{min}}).
	\end{equation*}
	Hence, $s_{\text{min}}\leq -\bar{s}$. The other cases can be treated in a similar way.
	
	We denote $\tilde{\Omega}_{\bar{s}}=[\X(s_{\text{min}}),\X(s_{\text{max}})]\times [\Y(s_{\text{min}}),\Y(s_{\text{max}})]$ and, since $s_{\text{min}}\leq -\bar{s}\leq \bar{s}\leq s_{\text{max}}$, we have $\Omega_{\bar{s}}\subset \tilde{\Omega}_{\bar{s}}$. We define the curve 
	\begin{align*}
	&(\tilde{\X}(s),\tilde{\Y}(s))\\
	&=\begin{cases}
	(\X(s),\Y(s)) & \text{if } s<s_{\text{min}},\\
	\text{straight line joining } (\X(s_{\text{min}}),\Y(s_{\text{min}})) \text{ and } (\bar{\X}(-\bar{s}),\Y(-\bar{s})) & \text{if } s_{\text{min}}\leq s<-\bar{s},\\
	(\bar{\X}(s),\bar{\Y}(s)) & \text{if } -\bar{s}\leq s \leq \bar{s},\\
	\text{straight line joining } (\bar{\X}(\bar{s}),\Y(\bar{s})) \text{ and } (\X(s_{\text{max}}),\Y(s_{\text{max}})) & \text{if } \bar{s}<s\leq s_{\text{max}},\\
	(\X(s),\Y(s)) & \text{if } s_{\text{max}}<s.
	\end{cases}
	\end{align*}
	We have that $(\tilde{\X},\tilde{\Y})$ and $(\X,\Y)$ belong to $\C(\tilde{\Omega}_{\bar{s}})$. By Lemma \ref{lemma:stripest}, we get
	\begin{align*}
	|||(Z,p,q)\bullet (\bar{\X},\bar{\Y})|||_{\G(\Omega_{\bar{s}})}&\leq |||(Z,p,q)\bullet (\tilde{\X},\tilde{\Y})|||_{\G(\tilde{\Omega}_{\bar{s}})}\\
	&\leq C_{1}\Big(||(\tilde{\X},\tilde{\Y})||_{\C(\tilde{\Omega}_{\bar{s}})},|||\Theta|||_{\G(\tilde{\Omega}_{\bar{s}})}\Big)\\
	&\leq C_{1}\big(||(\tilde{\X},\tilde{\Y})||_{\C},|||\Theta|||_{\G}\big),
	\end{align*}
	and by letting $\bar{s}$ tend to infinity, we obtain $|||(Z,p,q)\bullet (\bar{\X},\bar{\Y})|||_{\G}<\infty$. By Lemma \ref{lemma:groncurve}, we obtain
	\begin{align*}
	||(Z,p,q)\bullet (\bar{\X},\bar{\Y})||_{\G(\Omega_{\bar{s}})}
	&\leq ||(Z,p,q)\bullet (\tilde{\X},\tilde{\Y})||_{\G(\tilde{\Omega}_{\bar{s}})}\\
	&\leq C||(Z,p,q)\bullet (\X,\Y)||_{\G(\tilde{\Omega}_{\bar{s}})}\\
	&\leq C||(Z,p,q)\bullet (\X,\Y)||_{\G},
	\end{align*}
	where 
	\begin{align*}
	C&=C\Big(||(\tilde{\X},\tilde{\Y})||_{\C(\Omega_{\bar{s}})},|||(Z,p,q)\bullet (\X,\Y)|||_{\G(\tilde{\Omega}_{\bar{s}})}\Big)\\
	&\leq C\big(||(\tilde{\X},\tilde{\Y})||_{\C},|||(Z,p,q)\bullet (\X,\Y)|||_{\G}\big).
	\end{align*}
	By letting $\bar{s}$ tend to infinity, we get $||(Z,p,q)\bullet (\bar{\X},\bar{\Y})||_{\G}<\infty$ and we have proved \eqref{eq:anycurveprovethis1}. Hence, condition (i) of Definition \ref{def:setG} is satisfied. The conditions (ii)-(iv) follow directly since $(Z,p,q)\in\H$ and $(\bar{\X},\bar{\Y})\in \C$. We prove that
	\begin{equation}
	\label{eq:anycurveprovethis2}
	\displaystyle\lim_{s\rightarrow \pm\infty}J(\bar{\X}(s),\bar{\Y}(s))=\displaystyle\lim_{s\rightarrow \pm\infty}J(\X(s),\Y(s)).
	\end{equation}
	For any $s\in \mathbb{R}$, let $s_{1}=\X^{-1}(\bar{\X}(s))$ and $s_{2}=\Y^{-1}(\bar{\Y}(s))$. If $s_{1}\leq s_{2}$, then $\bar{\X}(s)=\X(s_{1})\leq \X(s_{2})$ and $\bar{\Y}(s)=\Y(s_{2})\geq \Y(s_{1})$. Since $J_{X},J_{Y}\geq 0$, we get
	\begin{equation*}
	J(\X(s_{1}),\Y(s_{1}))\leq J(\bar{\X}(s),\bar{\Y}(s))\leq J(\X(s_{2}),\Y(s_{2})).
	\end{equation*}
	Similarly, if $s_{2}\leq s_{1}$, we obtain
	\begin{equation*}
	J(\X(s_{2}),\Y(s_{2}))\leq J(\bar{\X}(s),\bar{\Y}(s))\leq J(\X(s_{1}),\Y(s_{1})),
	\end{equation*}
	so that
	\begin{equation}
	\label{eq:squeeze1}
	\min\{J(\X(s_{1}),\Y(s_{1})),J(\X(s_{2}),\Y(s_{2}))\}\leq J(\bar{\X}(s),\bar{\Y}(s))
	\end{equation}
	and 
	\begin{equation}
	\label{eq:squeeze2}
	J(\bar{\X}(s),\bar{\Y}(s))\leq \max\{J(\X(s_{1}),\Y(s_{1})),J(\X(s_{2}),\Y(s_{2}))\}.
	\end{equation}
	Since
	\begin{equation*}
	|s_{1}-s|\leq |\X(s_{1})-s_{1}|+|\bar{\X}(s)-s|\leq ||(\X,\Y)||_{\C}+||(\bar{\X},\bar{\Y})||_{\C},
	\end{equation*}
	we have that $\displaystyle\lim_{s\rightarrow \pm\infty}s_{1}=\displaystyle\lim_{s\rightarrow \pm\infty}\X^{-1}(\bar{\X}(s))=\pm\infty$. Similarly, we find that $\displaystyle\lim_{s\rightarrow \pm\infty}s_{2}=\displaystyle\lim_{s\rightarrow \pm\infty}\Y^{-1}(\bar{\Y}(s))=\pm\infty$. Hence, \eqref{eq:squeeze1} and \eqref{eq:squeeze2} yield \eqref{eq:anycurveprovethis2}, since $J$ is bounded and monotone. 
	In particular, we have that these limits are independent of which curve $(\bar{\X},\bar{\Y})$ is chosen. Furthermore, by \eqref{eq:setGrel5},
	\begin{equation*}
	\displaystyle\lim_{s\rightarrow -\infty}J(\bar{\X}(s),\bar{\Y}(s))=\displaystyle\lim_{s\rightarrow -\infty}J(\X(s),\Y(s))=0,
	\end{equation*}
	which shows that the last condition (v) in Definition \ref{def:setG} is satisfied for $\bar{\Theta}=(Z,p,q)\,\bullet \,(\bar{\X},\bar{\Y})$. Hence, $\bar{\Theta}\in \G$.
\end{proof}

We have the following global existence theorem.

\begin{theorem}[Existence and uniqueness of global solutions]
	\label{thm:globalsoln}
	For any initial data \newline $\Theta=(\X,\Y,\Z,\V,\W,\p,\q)\in\G$, there exists a unique solution $(Z,p,q)\in \H$ such that $\Theta=(Z,p,q)\bullet(\X,\Y)$. We denote this solution mapping by 
	\begin{equation}
	\label{eq:Soperator}
	\bf{S}: \G\rightarrow \H.
	\end{equation}
\end{theorem}

\begin{proof}
	First we show how to construct the solution on rectangles with diagonal points which lie on the curve $(\X,\Y)\in\C$. 
	
	We consider two points, $(\bar{X}_{l},\bar{Y}_{l})$ and $(\bar{X}_{r},\bar{Y}_{r})$, on the curve $(\X,\Y)\in\C$ such that $\bar{X}_{l}<\bar{X}_{r}$ and $\bar{Y}_{l}<\bar{Y}_{r}$. Set $\bar{s}_{l}=\frac{1}{2}(\bar{X}_{l}+\bar{Y}_{l})$ and $\bar{s}_{r}=\frac{1}{2}(\bar{X}_{r}+\bar{Y}_{r})$. Let the restriction of $\Theta$ to $\bar{\Omega}=[\bar{X}_{l},\bar{X}_{r}]\times[\bar{Y}_{l},\bar{Y}_{r}]$, which we denote by $\bar{\Theta}=(\bar{\X},\bar{\Y},\bar{\Z},\bar{\V},\bar{\W},\bar{\p},\bar{\q})$, be such that $\bar{\X}(s)=\X(s)$, $\bar{\Y}(s)=\Y(s)$, $\bar{\Z}(s)=\Z(s)$, $\bar{\V}(\bar{\X}(s))=\V(\X(s))$, $\bar{\W}(\bar{\Y}(s))=\W(\Y(s))$, $\bar{\p}(\bar{\X}(s))=\p(\X(s))$ and $\bar{\q}(\bar{\Y}(s))=\q(\Y(s))$ for $s\in[\bar{s}_{l},\bar{s}_{r}]$. Then, $\bar{\Theta}\in\G(\bar{\Omega})$ and, by Lemma \ref{lemma:SolnBigRectangle}, there exists a unique solution $(\bar{Z},\bar{p},\bar{q})\in \H(\bar{\Omega})$ such that $\bar{\Theta}=(\bar{Z},\bar{p},\bar{q})\bullet(\bar{\X},\bar{\Y})$. 
	
	We can consider a new rectangle $\tilde{\Omega}=[\tilde{X}_{l},\tilde{X}_{r}]\times[\tilde{Y}_{l},\tilde{Y}_{r}]$ such that the upper right diagonal point of $\bar{\Omega}$ is the lower left diagonal point of $\tilde{\Omega}$, that is, 
	$\tilde{X}_{l}=\bar{X}_{r}$ and $\tilde{Y}_{l}=\bar{Y}_{r}$. We set the upper right diagonal point $(\tilde{X}_{r},\tilde{Y}_{r})$ to be on the curve $(\X,\Y)$. By the above argument, we obtain a unique solution $(\tilde{Z},\tilde{p},\tilde{q})\in\H(\tilde{\Omega})$ such that $\tilde{\Theta}=(\tilde{Z},\tilde{p},\tilde{q})\, \bullet\, (\tilde{\X},\tilde{\Y})$, where $\tilde{\Theta}=(\tilde{\X},\tilde{\Y},\tilde{\Z},\tilde{\V},\tilde{\W},\tilde{\p},\tilde{\q})$ is the restriction of $\Theta$ to $\tilde{\Omega}$. By repeating this argument on rectangles above $\tilde{\Omega}$ and below $\bar{\Omega}$, we obtain unique solutions in rectangles which cover the curve $(\X,\Y)$.
	
	Now we show how to construct solutions on rectangles with diagonal points which does not lie on the given curve $(\X,\Y)\in\C$. We consider the above setting. 
	
	Let $\hat{\Omega}=[\hat{X}_{l},\hat{X}_{r}]\times[\hat{Y}_{l},\hat{Y}_{r}]$ be the rectangle which lies above $\bar{\Omega}$ and to the left of $\tilde{\Omega}$, that is, $\hat{X}_{l}=\bar{X}_{l}$, $\hat{X}_{r}=\bar{X}_{r}=\tilde{X}_{l}$, $\hat{Y}_{l}=\tilde{Y}_{l}=\bar{Y}_{r}$ and $\hat{Y}_{r}=\tilde{Y}_{r}$. Set
	\begin{equation*}
	(\hat{\X}(s),\hat{\Y}(s))=\begin{cases}
	(2s-\hat{Y}_{l},\hat{Y}_{l}) & \text{if } \hat{s}_{l}\leq s\leq \tilde{s}_{l},\\
	(\hat{X}_{r},2s-\hat{X}_{r}) & \text{if } \tilde{s}_{l}<s\leq \hat{s}_{r},
	\end{cases}
	\end{equation*} 
	where $\hat{s}_{l}$, $\tilde{s}_{l}$ and $\hat{s}_{r}$ are defined similarly to $\bar{s}_{l}$ and $\bar{s}_{r}$ above. We have $(\hat{\X},\hat{\Y})\in\C(\hat{\Omega})$, and we denote $\hat{\Theta}=(\hat{\X},\hat{\Y},\hat{\Z},\hat{\V},\hat{\W},\hat{\p},\hat{\q})$, where
	\begin{align*}
	&\hat{\Z}(s)=\bar{\Z}(s), \quad \hat{\V}(\hat{\X}(s))=\bar{\V}(\hat{\X}(s)), \quad \hat{\W}(\hat{\Y}(s))=\bar{\W}(\hat{\Y}(s)),\\ &\hat{\p}(\hat{\X}(s))=\bar{\p}(\hat{\X}(s)), \quad \hat{\q}(\hat{\Y}(s))=\bar{\q}(\hat{\Y}(s))
	\end{align*}
	for $\hat{s}_{l}\leq s\leq \tilde{s}_{l}$ and
	\begin{align*}
	&\hat{\Z}(s)=\tilde{\Z}(s), \quad \hat{\V}(\hat{\X}(s))=\tilde{\V}(\hat{\X}(s)), \quad \hat{\W}(\hat{\Y}(s))=\tilde{\W}(\hat{\Y}(s)),\\ &\hat{\p}(\hat{\X}(s))=\tilde{\p}(\hat{\X}(s)), \quad \hat{\q}(\hat{\Y}(s))=\tilde{\q}(\hat{\Y}(s))
	\end{align*}
	for $\tilde{s}_{l}<s\leq \hat{s}_{r}$. We have $\hat{\Theta}\in\G(\hat{\Omega})$. By Lemma
	\ref{lemma:SolnBigRectangle}, there exists a unique solution $(\hat{Z},\hat{p},\hat{q})\in \H(\hat{\Omega})$ such that $\hat{\Theta}=(\hat{Z},\hat{p},\hat{q})\bullet(\hat{\X},\hat{\Y})$. 
	By repeatedly applying this argument to rectangles that are adjacent to rectangles where we have a solution, we obtain unique solutions in any rectangular domain. Hence, condition (i) of Definition \ref{def:globalsolutions} is satisfied.
	
	We define $(Z,p,q)$ to be the unique solution in each rectangle. Then, we have $(Z,p,q)\bullet(\X,\Y)=\Theta\in\G$ and condition (ii) of Definition \ref{def:globalsolutions} is satisfied, so that $(Z,p,q)\in\H$. 
\end{proof}

\section{From Lagrangian to Eulerian Coordinates}

\subsection{Mapping from $\H$ to $\F$}

Given an element $(Z,p,q)$ in $\H$ we now want to map it to an element in the set $\G$ and then further to one in $\F$. For a solution in $\H$ corresponding to time $T>0$, i.e., $t(X,Y)=T$, we find it convenient to first shift the time to zero so that we can map the solution to an element in $\G_{0}$ in the next step.  

\begin{definition}
	\label{lemma:tTmap}
	Given $T\geq 0$ and $(Z,p,q)\in \H$, we define
	\begin{subequations}
		\label{eqns:tTmap}
		\begin{equation}
		\bar{t}(X,Y)=t(X,Y)-T
		\end{equation}
		and
		\begin{align}	
		\bar{x}(X,Y)&=x(X,Y), & \bar{U}(X,Y)&=U(X,Y),\\
		\bar{J}(X,Y)&=J(X,Y), & \bar{K}(X,Y)&=K(X,Y),\\
		\bar{p}(X,Y)&=p(X,Y), & \bar{q}(X,Y)&=q(X,Y).
		\end{align}
	\end{subequations} 
	We denote by $\mathbf{t}_{T}:\H \rightarrow \H$ the mapping which associates to any $(Z,p,q)\in \H$ the element $(\bar{Z},\bar{p},\bar{q})\in \H$. We have
	\begin{equation}
	\label{eq:tsemigroup}
	\mathbf{t}_{T+T'}=\mathbf{t}_{T}\circ \mathbf{t}_{T'}.
	\end{equation}
\end{definition}

\begin{definition}
\label{def:mapE}
Given $(Z,p,q)\in \H$, we define
\begin{equation}
\label{eq:Emap}
	\X(s)=\sup\{X\in \mathbb{R}\ | \ t(X',2s-X')<0 \text{ for all } X'<X\}
\end{equation}
and $\Y(s)=2s-\X(s)$. Then, we have $(\X,\Y)\in \C$ and $(Z,p,q)\bullet (\X,\Y)\in \G_{0}$. We denote by $\bf{E}:\H\rightarrow \G_{0}$ the mapping which associates to any $(Z,p,q)\in \H$ the element $(Z,p,q)\bullet (\X,\Y)\in \G_{0}$.
\end{definition}

\begin{proof}[Proof of the well-posedness of Definition \ref{def:mapE}]
We prove that $(\X,\Y)$ belongs to $\C$. Let us verify that $\X$ is nondecreasing. Let $s<\bar{s}$ and consider a sequence $X_{i}$ such that $X_{i}<\X(s)$ and $\displaystyle\lim_{i\rightarrow \infty}X_{i}=\X(s)$. We have $t(X_{i},2s-X_{i})<0$ and, since $t_{Y}\leq 0$, $t(X_{i},2\bar{s}-X_{i})<0$, which implies that $X_{i}<\X(\bar{s})$. By letting $i$ tend to infinity, we obtain $\X(s)\leq \X(\bar{s})$, so that $\X$ is nondecreasing. 

Next we show that $\X$ is differentiable almost everywhere. We claim that $\X$ is Lipschitz continuous with Lipschitz constant at most two. Let us assume the opposite, that is, there exists $\bar{s}>s$ such that
\begin{equation*}
	\X(\bar{s})-\X(s)>2(\bar{s}-s).
\end{equation*}
This implies that $\Y(s)>\Y(\bar{s})$ and we denote $\Omega=[\X(s),\X(\bar{s})]\times [\Y(\bar{s}),\Y(s)]$. Since $t_{X}\geq 0$ and $t_{Y}\leq 0$, we have, for any $(X,Y)\in \Omega$, that
\begin{equation*}
	0=t(\X(s),\Y(s))\leq t(X,\Y(s))\leq t(X,Y)\leq t(X,\Y(\bar{s}))\leq t(\X(\bar{s}),\Y(\bar{s}))=0,
\end{equation*}
so that $t(X,Y)=0$ for all $(X,Y)\in \Omega$. Consider the point $(X,Y)$ given by $Y=\Y(s)$ and $X=2\bar{s}-\Y(s)$. It belongs to $\Omega$ since $X=2\bar{s}-\Y(s)<2\bar{s}-\Y(\bar{s})=\X(\bar{s})$. We have $t(X,Y)=0$, $X+Y=2\bar{s}$ and $X<\X(\bar{s})$, which contradicts the definition of $\X(\bar{s})$. Hence, we have proved that $\X$ is Lipschitz continuous with Lipschitz constant at most two and therefore differentiable almost everywhere. Furthermore, it follows that $\Y$ is nondecreasing and differentiable almost everywhere. Since $\dot{\X},\dot{\Y}\in [0,2]$, we find that $\dot{\X}-1,\dot{\Y}-1\in \Linf (\mathbb{R})$. It remains to prove that $\X-\id$ and $\Y-\id$ belong to $\Linf (\mathbb{R})$. We will prove that
\begin{equation}
\label{eq:limsup}
	\displaystyle\limsup_{s\rightarrow \pm \infty}|\X(s)-s|\leq \frac{L}{2}
\end{equation}
for a constant $L$ that will be set later, and which depends on $\kappa$ and $|||(Z,p,q)\bullet (\X_{d},\Y_{d})|||_{\G}$. First we prove that 
\begin{equation}
\label{eq:firstClaim}
	\displaystyle\limsup_{s\rightarrow \infty}(\X(s)-s)\leq \frac{L}{2}.
\end{equation}
Assume the opposite. Introducing $f(s)=\displaystyle\sup_{r\geq s}(\X(r)-r)$, which is a nonincreasing function, we then have
\begin{equation*}
	\inf_{s\geq 0}f(s)>\frac{L}{2}.
\end{equation*}
This implies that
\begin{equation*}
	f(s)>\frac{L}{2}
\end{equation*}
for all $s\geq 0$. The function $\X-\id$ can only be unbounded at infinity since it is continuous with bounded derivative. If $\X-\id$ is bounded at infinity, \eqref{eq:firstClaim} is immediately satisfied. Thus, we assume that $\X(s)-s$ tends to infinity as $s\rightarrow\infty$. This implies that $f(s)\rightarrow\infty$ as $s\rightarrow\infty$. Then there is an increasing sequence $s_{n}\rightarrow\infty$ as $n\rightarrow\infty$, such that
\begin{equation*}
	\X(s_{n})-s_{n}>\frac{L}{2}
\end{equation*} 
for all $n$, and $\X(s_{n})-s_{n}\rightarrow\infty$ as $n\rightarrow\infty$. We have $\Y(s_{n})=2s_{n}-\X(s_{n})\leq s_{n}-\frac{L}{2}$. Since $t_{X}\geq 0$ and $t_{Y}\leq 0$ we have
\begin{equation*}
	0=t(\X(s_{n}),\Y(s_{n}))\geq t\Big(s_{n}+\frac{L}{2},s_{n}-\frac{L}{2}\Big).
\end{equation*}
Next we prove that
\begin{equation}
\label{eq:HtoG0mapliminft}
	\displaystyle\liminf_{n\rightarrow \infty}t\Big(s_{n}+\frac{L}{2},s_{n}-\frac{L}{2}\Big)\geq 1,
\end{equation}  
which will lead to the contradiction
\begin{equation*}
	0\geq\displaystyle\liminf_{n\rightarrow \infty}t\Big(s_{n}+\frac{L}{2},s_{n}-\frac{L}{2}\Big)\geq 1,
\end{equation*}
and \eqref{eq:firstClaim} follows.

By \eqref{eq:setH1} and \eqref{eq:setH4}, we have
\begin{aalign}
	\label{eq:HtoG0maptestimate}
	&t\Big(s_{n}+\frac{L}{2},s_{n}-\frac{L}{2}\Big)\\
	&=t\Big(s_{n}-\frac{L}{2},s_{n}-\frac{L}{2}\Big)+\int_{s_{n}-\frac{L}{2}}^{s_{n}+\frac{L}{2}}\bigg(\frac{x_{X}}{c(U)}\bigg)\Big(\tilde{X},s_{n}-\frac{L}{2}\Big)\,d\tilde{X}\\
	&\geq -|||(Z,p,q)\bullet (\X_{d},\Y_{d})|||_{\G}+\frac{L}{2\kappa}+\frac{1}{\kappa}\int_{s_{n}-\frac{L}{2}}^{s_{n}+\frac{L}{2}}\bigg(x_{X}-\frac{1}{2}\bigg)\Big(\tilde{X},s_{n}-\frac{L}{2}\Big)\,d\tilde{X}.
\end{aalign}
Let $\Omega_{n,L}=[s_{n}-\frac{L}{2},s_{n}+\frac{L}{2}]\times[s_{n}-\frac{L}{2},s_{n}+\frac{L}{2}]$ and consider the curve
\begin{equation*}
	(\bar{\X}(s),\bar{\Y}(s))=\begin{cases}
	(\X_{d}(s),\Y_{d}(s)) & \text{for } s<s_{n}-\frac{L}{2},\\
	(2s-(s_{n}-\frac{L}{2}),s_{n}-\frac{L}{2}) & \text{for } s_{n}-\frac{L}{2}\leq s\leq s_{n},\\
	(s_{n}+\frac{L}{2},2s-(s_{n}+\frac{L}{2})) & \text{for } s_{n}\leq s\leq s_{n}+\frac{L}{2},\\
	(\X_{d}(s),\Y_{d}(s)) & \text{for } s>s_{n}+\frac{L}{2}.
\end{cases}
\end{equation*}
Both $(\X_{d},\Y_{d})$ and $(\bar{\X},\bar{\Y})$ belong to $\C(\Omega_{n,L})$. By the Cauchy--Schwarz inequality and Lemma \ref{lemma:groncurve}, we find
\begin{aalign}
	\label{eq:HtoG0mapxestimate}
	&\bigg|\int_{s_{n}-\frac{L}{2}}^{s_{n}+\frac{L}{2}}\bigg(x_{X}-\frac{1}{2}\bigg)\Big(\tilde{X},s_{n}-\frac{L}{2}\Big)\,d\tilde{X}\bigg|\\
	&\leq\sqrt{L}\bigg(\int_{s_{n}-\frac{L}{2}}^{s_{n}+\frac{L}{2}}\bigg(x_{X}-\frac{1}{2}\bigg)^{2}\Big(\tilde{X},s_{n}-\frac{L}{2}\Big)\,d\tilde{X}\bigg)^{\frac{1}{2}}\\
	&\leq \sqrt{L}\,||(Z,p,q)\bullet (\bar{\X},\bar{\Y})||_{\G(\Omega_{n,L})}\\
	&\leq \sqrt{L}C||(Z,p,q)\bullet (\X_{d},\Y_{d})||_{\G(\Omega_{n,L})}.
\end{aalign}
Here, $C$ is an increasing function with respect to both its arguments and we have
\begin{align*}
	C&=C\big(||(\bar{\X},\bar{\Y})||_{\C(\Omega_{n,L})},|||(Z,p,q)\bullet (\X_{d},\Y_{d})|||_{\G(\Omega_{n,L})}\big)\\
	&\leq C\big(L,|||(Z,p,q)\bullet (\X_{d},\Y_{d})|||_{\G}\big),
\end{align*} 
where we used that $||(\bar{\X},\bar{\Y})||_{\C(\Omega_{n,L})}=L$. From \eqref{eq:Gnormequivalent}, we have
\begin{align*}
	&\displaystyle\lim_{n\rightarrow \infty}||(Z,p,q)\bullet (\X_{d},\Y_{d})||_{\G(\Omega_{n,L})}^{2}\\
	&=\displaystyle\lim_{n\rightarrow \infty}\int_{s_{n}-\frac{L}{2}}^{s_{n}+\frac{L}{2}}\big(U^{2}+|Z_{X}^{a}|^{2}+|Z_{Y}^{a}|^{2}+p^{2}+q^{2}\big)(\tilde{X},\tilde{X})\,d\tilde{X}=0,
\end{align*}
which combined with \eqref{eq:HtoG0mapxestimate} and \eqref{eq:HtoG0maptestimate} yields
\begin{equation*}
	\displaystyle\liminf_{n\rightarrow \infty}t\Big(s_{n}+\frac{L}{2},s_{n}-\frac{L}{2}\Big)\geq -|||(Z,p,q)\bullet (\X_{d},\Y_{d})|||_{\G}+\frac{L}{2\kappa}.
\end{equation*} 
Setting $L\geq 2\kappa(|||(Z,p,q)\bullet (\X_{d},\Y_{d})|||_{\G}+1)$ implies \eqref{eq:HtoG0mapliminft}. Thus, we have proved \eqref{eq:firstClaim}.

It remains to prove that $\displaystyle\liminf_{s\rightarrow \infty}(\X(s)-s)\geq -\frac{L}{2}$ in order to conclude that
\begin{equation*}
	\displaystyle\limsup_{s\rightarrow\infty}|\X(s)-s|\leq \frac{L}{2}.
\end{equation*}
The proof is similar to the one above. Now one has to show that
\begin{equation}
\label{eq:HtoG0mapliminft2}
	\displaystyle\limsup_{n\rightarrow \infty}t\Big(s_{n}-\frac{L}{2},s_{n}+\frac{L}{2}\Big)\leq -1,
\end{equation}
for an increasing sequence $s_{n}\rightarrow\infty$ as $n\rightarrow\infty$, in order to get a contradiction.

The proof of
\begin{equation*}
	\displaystyle\limsup_{s\rightarrow-\infty}|\X(s)-s|\leq \frac{L}{2}
\end{equation*}
is similar to the argument above. To show
\begin{equation*}
	\displaystyle\limsup_{s\rightarrow -\infty}\,(\X(s)-s)\leq \frac{L}{2} \quad \text{and} \quad \displaystyle\liminf_{s\rightarrow -\infty}\,(\X(s)-s)\geq -\frac{L}{2},
\end{equation*}
one proves
\begin{equation*}
	\displaystyle\liminf_{n\rightarrow \infty}t\Big(s_{n}+\frac{L}{2},s_{n}-\frac{L}{2}\Big)\geq 1 \quad \text{and} \quad \displaystyle\limsup_{n\rightarrow \infty}t\Big(s_{n}-\frac{L}{2},s_{n}+\frac{L}{2}\Big)\leq -1,
\end{equation*}
respectively, for a carefully chosen decreasing sequence $s_{n}\rightarrow-\infty$ as $n\rightarrow\infty$.

This concludes the proof of \eqref{eq:limsup}, and we have showed that $\X-\id$ and $\Y-\id$ belong to $\Linf(\mathbb{R})$, so that $(\X,\Y)\in \C$. Then, by Lemma \ref{lemma:curveind}, we have
\begin{equation*}
	(\X,\Y,\Z,\V,\W,\p,\q)=(Z,p,q)\bullet (\X,\Y)\in \G
\end{equation*} 
and by construction $\Z_{1}(s)=t(\X(s),\Y(s))=0$ so that $(Z,p,q)\bullet (\X,\Y)\in \G_{0}$.
\end{proof}

\begin{definition}
	\label{def:mapfromG0toF}
	Given $(\X,\Y,\Z,\V,\W,\p,\q) \in \G_{0}$, let $\psi_{1}=(x_{1},U_{1},J_{1},K_{1},V_{1},H_{1})$ and $\psi_{2}=(x_{2},U_{2},J_{2},K_{2},V_{2},H_{2})$ be defined as
	\begin{subequations}
	\begin{align}
	\label{eq:mapfromG0toF1}
	x_{1}(\X(s))&=x_{2}(\Y(s))=\Z_{2}(s), \\
	\label{eq:mapfromG0toF2}
	U_{1}(\X(s))&=U_{2}(\Y(s))=\Z_{3}(s),
	\end{align}
	\begin{align}
	\label{eq:mapfromG0toF3}
	&J_{1}(\X(s))=\int_{-\infty}^{s}\V_{4}(\X(\tau))\dot{\X}(\tau)\,d\tau, 
	&&J_{2}(\Y(s))=\int_{-\infty}^{s}\W_{4}(\Y(\tau))\dot{\Y}(\tau)\,d\tau, \\ 
	\label{eq:mapfromG0toF4}			&K_{1}(\X(s))=\int_{-\infty}^{s}\V_{5}(\X(\tau))\dot{\X}(\tau)\,d\tau, 
	&&K_{2}(\Y(s))=\int_{-\infty}^{s}\W_{5}(\Y(\tau))\dot{\Y}(\tau)\,d\tau,
	\end{align}
	and
	\begin{align}
	\label{eq:mapfromG0toF5}
	&V_{1}=\V_{3}, &&V_{2}=\W_{3}, \\
	\label{eq:mapfromG0toF6}
	&H_{1}=\p, &&H_{2}=\q.
	\end{align}
	\end{subequations}
	We denote by $\bf{D}:\G_{0}\rightarrow \F$ the mapping which to any $(\X,\Y,\Z,\V,\W,\p,\q) \in \G_{0}$ associates the element $\psi\in \F$ as defined above.
\end{definition}

\begin{proof}[Proof of the well-posedness of Definition \ref{def:mapfromG0toF}]
	We check the well-posedness of (\ref{eq:mapfromG0toF1}) and (\ref{eq:mapfromG0toF2}). Consider $s<\bar{s}$ such that $\X(s)=\X(\bar{s})$. Since $\X$ is nondecreasing and continuous, we have $\dot{\X}(\tilde{s})=0$ and $\dot{\Y}(\tilde{s})=2$ for all $\tilde{s}\in [s,\bar{s}]$. From (\ref{eq:setG0rel}), it follows that $\W_{2}(\Y(\tilde{s}))=0$ for all $\tilde{s}\in [s,\bar{s}]$. Hence,
	\begin{equation*}
	\dot{\Z}_{2}(\tilde{s})=\V_{2}(\X(\tilde{s}))\dot{\X}(\tilde{s})+\W_{2}(\Y(\tilde{s}))\dot{\Y}(\tilde{s})=0
	\end{equation*}
	for all $\tilde{s}\in [s,\bar{s}]$, so that $\Z_{2}(\tilde s)=\Z_{2}(\bar{s})$ and (\ref{eq:mapfromG0toF1}) is well-posed. 
	
	From (\ref{eq:setGrel4}), we obtain
	\begin{equation*}
	0=2\W_{4}(\Y(\tilde{s}))\W_{2}(\Y(\tilde{s}))=(c(\Z_{3}(\tilde{s}))\W_{3}(\Y(\tilde{s})))^{2}+c(\Z_{3}(\tilde{s}))\q^{2}(\Y(\tilde{s})),
	\end{equation*}
	which implies that $\W_{3}(\Y(\tilde{s}))=0$ and $\q(\Y(\tilde{s}))=0$ for all $\tilde{s}\in [s,\bar{s}]$. Thus,
	\begin{equation*}
	\dot{\Z}_{3}(\tilde{s})=\V_{3}(\X(\tilde{s}))\dot{\X}(\tilde{s})+\W_{3}(\Y(\tilde{s}))\dot{\Y}(\tilde{s})=0
	\end{equation*}
	so that $\Z_{3}(\tilde s)=\Z_{3}(\bar{s})$ and (\ref{eq:mapfromG0toF2}) is well-posed.
	
	 Let us prove that  $J_1$ and $K_1$ given by (\ref{eq:mapfromG0toF3}) and (\ref{eq:mapfromG0toF4}) are well-posed. The proof is similar for $J_{2}$ and $K_{2}$. Since $\V_{4},\W_{4},\dot{\X},\dot{\Y}\geq 0$, we have $J_{1}\geq 0$ and
	\begin{aalign}
		\label{eq:mapfromG0toFV4inL1}
		J_{1}(\X(s))&=\int_{-\infty}^{s}\V_{4}(\X(\tau))\dot{\X}(\tau)\,d\tau\\
		&\leq \int_{-\infty}^{s}(\V_{4}(\X(\tau))\dot{\X}(\tau)+\W_{4}(\Y(\tau))\dot{\Y}(\tau))\,d\tau\\
		&=\int_{-\infty}^{s}\dot{\Z}_{4}(\tau)\,d\tau \leq ||\Z_{4}^{a}||_{\Linf(\mathbb{R})},
	\end{aalign}
	so that the function $\V_{4}$ belongs to $L^{1}(\mathbb{R})$ and $J_{1}$ is bounded. If $s,\bar{s}\in\mathbb{R}$ are such that $s<\bar{s}$ and $\X(s)=\X(\bar{s})$, we have $\dot{\X}(\tilde{s})=0$ for all $\tilde{s}\in [s,\bar{s}]$ since $\X$ is nondecreasing. Hence,
	\begin{equation}
	\label{eq:mapfromG0toFJ1wellposed}
	\int_{-\infty}^{\tilde{s}}\V_{4}(\X(\tau))\dot{\X}(\tau)\,d\tau=\int_{-\infty}^{s}\V_{4}(\X(\tau))\dot{\X}(\tau)\,d\tau,	
	\end{equation}
	and $J_1$ is well-posed. For $K_{1}$, we have by \eqref{eq:setGrel2}, that
	\begin{equation}
	\label{eq:mapfromG0toFK1wellposed}
	K_{1}(\X(s))=\int_{-\infty}^{s}\V_{5}(\X(\tau))\dot{\X}(\tau)\,d\tau
	=\int_{-\infty}^{s}\frac{\V_{4}(\X(\tau))}{c(\Z_{3}(\tau))}\dot{\X}(\tau)\,d\tau\leq \kappa J_{1}(\X(s)),
	\end{equation}
	which by \eqref{eq:mapfromG0toFV4inL1} and since $\V_{4}\geq 0$ and $c>0$, implies that $0\leq K_{1}(X)\leq \kappa ||\Z_{4}^{a}||_{\Linf(\mathbb{R})}$. By an argument as in \eqref{eq:mapfromG0toFJ1wellposed} applied to $K_{1}$, we conclude that also $K_1$ is well-posed. 
	
	Next we show that $\psi_{1}=(x_{1},U_{1},J_{1},K_{1},V_{1},H_{1})$ as defined in \eqref{eq:mapfromG0toF1}-\eqref{eq:mapfromG0toF6} satisfies the conditions in the definition of the set $\F$. The proof for $\psi_{2}$ is similar. 
	
	Let us show that $x_{1}$ is Lipschitz continuous and therefore differentiable almost everywhere. Consider $s,\bar{s}\in \mathbb{R}$ and set $X=\X(s)$ and $\bar{X}=\X(\bar{s})$. We have
	\begin{align*}
	|x_{1}(\bar{X})-x_{1}(X)|&=|\Z_{2}(\bar{s})-\Z_{2}(s)|\\
	&=\bigg|\int_{s}^{\bar{s}}\dot{\Z}_{2}(\tilde{s})\,d\tilde{s}\bigg| \\
	&=\bigg|\int_{s}^{\bar{s}}(\V_{2}(\X(\tilde{s}))\dot{\X}(\tilde{s})+\W_{2}(\Y(\tilde{s}))\dot{\Y}(\tilde{s}))\,d\tilde{s}\bigg| \\
	&=\bigg|2\int_{s}^{\bar{s}}\V_{2}(\X(\tilde{s}))\dot{\X}(\tilde{s})\,d\tilde{s}\bigg| \quad \text{by } \eqref{eq:setG0rel}\\
	&=\bigg|2\int_{s}^{\bar{s}}\V_{2}^{a}(\X(\tilde{s}))\dot{\X}(\tilde{s})\,d\tilde{s}+\int_{s}^{\bar{s}}\dot{\X}(\tilde{s})\,d\tilde{s}\bigg| \quad \text{by } \eqref{eq:atriplet}\\
	&\leq (2\norm{\V_{2}^{a}}_{L^{\infty}(\mathbb{R})}+1)|\bar{X}-X|. 
	\end{align*}
	
	From \eqref{eq:mapfromG0toFV4inL1}, we have that $J_{1}$ is increasing and hence differentiable almost everywhere.
	Similarly, one shows that $K_{1}$ is differentiable almost everywhere. 
	
	Next we show \eqref{eq:setF1}-\eqref{eq:setF4}. We have
	\begin{equation}
	\label{eq:mapfromG0toFx1-id}
	x_{1}(\X(s))-\X(s)=\Z_{2}(s)-s+s-\X(s)=\Z_{2}^{a}+s-\X(s),
	\end{equation}
	so that $x_{1}-\id \in \Linf(\mathbb{R})$ since $\Z_{2}^{a}$ and $\X-\id$ belong to $\Linf(\mathbb{R})$. Differentiating \eqref{eq:mapfromG0toF1} and using \eqref{eq:setG0rel}, we obtain $x_{1}'(\X)\dot{\X}=\dot{\Z}_{2}=2\V_{2}(\X)\dot{\X}$. Hence, $x_{1}'=2\V_{2}$ and we get that
	\begin{equation*}
	x_{1}'-1=2\V_{2}-1=2\V_{2}^{a},
	\end{equation*}
	which shows that $x_{1}'-1\in L^{2}(\mathbb{R})\cap \Linf(\mathbb{R})$. By \eqref{eq:mapfromG0toFV4inL1}, we have $J_{1}\in \Linf(\mathbb{R})$. We differentiate \eqref{eq:mapfromG0toF3} and obtain $J_{1}'=\V_{4}=\V_{4}^{a}$, which implies that $J_{1}'\in L^{2}(\mathbb{R})\cap \Linf(\mathbb{R})$. Then, from \eqref{eq:mapfromG0toFK1wellposed} it follows that $K_{1}\in \Linf(\mathbb{R})$ and $K_{1}'\in L^{2}(\mathbb{R})\cap \Linf(\mathbb{R})$. Since $\p,\q \in L^{2}(\mathbb{R})\cap \Linf(\mathbb{R})$, we have by \eqref{eq:mapfromG0toF6} that $H_{1}$ and $H_{2}$ belong to $L^{2}(\mathbb{R})\cap \Linf(\mathbb{R})$. The function $U_{1}$ belongs to $L^{2}(\mathbb{R})$ and $\Linf(\mathbb{R})$, as $U_{1}(\X)=\Z_{3}\in L^{2}(\mathbb{R})\cap \Linf(\mathbb{R})$. We have $V_{1}\in L^{2}(\mathbb{R})\cap \Linf(\mathbb{R})$ by \eqref{eq:mapfromG0toF5} and since $\V_{3}=\V_{3}^{a}\in L^{2}(\mathbb{R})\cap \Linf(\mathbb{R})$. Hence, we have proved \eqref{eq:setF1}-\eqref{eq:setF4}. Let us verify \eqref{eq:setFrel1}. We showed above that $x_{1}'=2\V_{2}$ and $J_{1}'=\V_{4}$. Thus, $x_{1}',J_{1}'\geq 0$ because $\V_{2},\V_{4}\geq 0$. The identity \eqref{eq:setFrel2} follows from \eqref{eq:setGrel2} since $J_{1}'=\V_{4}$ and $K_{1}'=\V_{5}$. We can check that the relation \eqref{eq:setFrel3} holds by using \eqref{eq:setGrel3}, \eqref{eq:mapfromG0toF5} and \eqref{eq:mapfromG0toF6}. Let us prove (\ref{eq:setFrel4}) by using Lemma \ref{lemma:auxiliaryG}. We found above that $J_{1}$ is absolutely continuous. Since $x_{1}$ is Lipschitz continuous, it follows that $x_{1}+J_{1}$ is absolutely continuous. By \eqref{eq:mapfromG0toFV4inL1}, we have
	\begin{equation*}
	|x_{1}+J_{1}-\id|\leq|x_{1}-\id|+||\Z_{4}^{a}||_{\Linf(\mathbb{R})}
	\end{equation*}
	which, by \eqref{eq:mapfromG0toFx1-id}, implies that $x_{1}+J_{1}-\id \in L^{\infty}(\mathbb{R})$.
	
	We proved above that $x_{1}'-1,J_{1}'\in L^{2}(\mathbb{R})$, which implies that $x_{1}'+J_{1}'-1\in L^{2}(\mathbb{R})$.
	
	The fact that $\frac{1}{\V_{2}+\V_{4}}\in L^{\infty}(\mathbb{R})$ implies that there exists a number $k>0$ such that $\V_{2}(X)+\V_{4}(X)\geq k$ for almost every $X\in \mathbb{R}$. Then, since $x_{1}'+J_{1}'=2\V_{2}+\V_{4}$, we obtain
	\begin{equation*}
	k\leq \V_{2}+\V_{4}\leq x_{1}'+J_{1}'\leq 2(\V_{2}+\V_{4})=2(\V_{2}^{a}+\V_{4}^{a})+1\leq 2\big(||\V_{2}^{a}||_{\Linf(\mathbb{R})}+||\V_{4}^{a}||_{\Linf(\mathbb{R})}\big)+1,
	\end{equation*}
	so that the remaining condition in Lemma \ref{lemma:auxiliaryG} holds. Hence, $x_{1}+J_{1}\in G$ and we have proved \eqref{eq:setFrel4}. 
	
	By \eqref{eq:mapfromG0toF3} and \eqref{eq:setGrel5}, we have
	\begin{equation*}
	0\leq J_{1}(\X(s))+J_{2}(\Y(s))=\int_{-\infty}^{s}\dot{\Z}_{4}(\tau)\,d\tau=\Z_{4}(s),
	\end{equation*}
	and since $\displaystyle\displaystyle\lim_{s\rightarrow-\infty}\X(s)=-\infty$ and $\displaystyle\displaystyle\lim_{s\rightarrow-\infty}\Y(s)=-\infty$, \eqref{eq:setGrel5} implies \eqref{eq:setFrel5}. The relation \eqref{eq:setFrel6} follows directly from \eqref{eq:mapfromG0toF2}. Using \eqref{eq:mapfromG0toF2}, \eqref{eq:setGcomp}, and \eqref{eq:mapfromG0toF5}, we obtain
	\begin{equation*}
	\frac{d}{ds}U_{1}(\X)=\frac{d}{ds}U_{2}(\Y)=\dot{\Z}_{3}=\V_{3}(\X)\dot{\X}+\W_{3}(\Y)\dot{\Y}=V_{1}(\X)\dot{\X}+V_{2}(\Y)\dot{\Y},
	\end{equation*}
	so that \eqref{eq:setFrel7} holds.

\end{proof}

\subsection{Semigroup of Solutions in $\F$}

We define the solution operator on the set $\F$.

\begin{definition}
	For any $T\geq 0$, we define the mapping $S_{T}:\F\rightarrow \F$ by
	\begin{equation*}
	S_{T}=\mathbf{D}\circ\mathbf{E}\circ\mathbf{t}_{T}\circ\mathbf{S}\circ\mathbf{C}.
	\end{equation*}		
\end{definition}

In order to show that $S_{T}$ is a semigroup we need the following result.

\begin{lemma}
\label{lemma:setrelations}
We have
\begin{equation}
\label{eq:FHbijection1}
	\mathbf{C}\circ\mathbf{D}\circ\mathbf{E}=\mathbf{E}, \quad \mathbf{D}\circ\mathbf{C}=\id
\end{equation}
and
\begin{equation}
\label{eq:FHbijection2}
	\mathbf{E}\circ\mathbf{S}\circ\mathbf{C}=\mathbf{C}, \quad \mathbf{S}\circ\mathbf{E}=\id.
\end{equation}
It follows that $\mathbf{S}\circ\mathbf{C}=(\mathbf{D}\circ\mathbf{E})^{-1}$ and the sets $\F$ and $\H$ are in bijection.
\end{lemma}

\begin{proof}
We first prove \eqref{eq:FHbijection1}. Given $(Z,p,q)\in \H$, let
\begin{align*}
	(\X,\Y,\Z,\V,\W,\p,\q)&=\mathbf{E}(Z,p,q),\\
	(\psi_{1},\psi_{2})&=\mathbf{D}(\X,\Y,\Z,\V,\W,\p,\q),\\
	(\bar{\X},\bar{\Y},\bar{\Z},\bar{\V},\bar{\W},\bar{\p},\bar{\q})&=\mathbf{C}(\psi_{1},\psi_{2}).
\end{align*}
We want to prove that $(\bar{\X},\bar{\Y},\bar{\Z},\bar{\V},\bar{\W},\bar{\p},\bar{\q})=(\X,\Y,\Z,\V,\W,\p,\q)$. Let us show that $\bar{\X}=\X$. We claim that for any $s\in \mathbb{R}$ and $(X,Y)$ such that $X<\X(s)$ and $X+Y=2s$, we have either
\begin{equation}
\label{eq:FHbijection1x1andx2}
	x_{1}(X)<x_{1}(\X(s)) \quad \text{or} \quad x_{2}(Y)>x_{2}(\Y(s)).
\end{equation}
Let us assume the opposite, that is, there exist $\bar{s}$ and $(\bar{X},\bar{Y})$ such that $\bar{X}<\X(\bar{s})$, $\bar{X}+\bar{Y}=2\bar{s}$ and
\begin{equation*}
	x_{1}(\bar{X})=x_{1}(\X(\bar{s}))=\Z_{2}(\bar{s})=x_{2}(\Y(\bar{s}))=x_{2}(\bar{Y}),
\end{equation*}
where we used \eqref{eq:mapfromG0toF1}. Let $s_{0}=\X^{-1}(\bar{X})$ and $s_{1}=\Y^{-1}(\bar{Y})$. Since $\bar{X}<\X(\bar{s})$, $\Y(\bar{s})<\bar{Y}$ and $\dot{\X},\dot{\Y}\geq 0$, we have $s_{0}<\bar{s}<s_{1}$. Consider the rectangular domain $\Omega=[\X(s_{0}),\X(s_{1})]\times [\Y(s_{0}),\Y(s_{1})]$. We want to construct a solution $(\tilde{Z},\tilde{p},\tilde{q})$ of 
\eqref{eq:goveq} in $\Omega$.

Since
\begin{align*}
	\Z_{2}(s_{0})&=x_{1}(\X(s_{0}))=x_{1}(\bar{X})=\Z_{2}(\bar{s}),\\
	\Z_{2}(s_{1})&=x_{2}(\Y(s_{1}))=x_{2}(\bar{Y})=\Z_{2}(\bar{s})
\end{align*}
and $\Z_{2}$ is nondecreasing, we have $\Z_{2}(s)=\Z_{2}(s_{0})=\Z_{2}(s_{1})$ for all $s\in [s_{0},s_{1}]$. We have $\dot{\Z}_{2}(s)=\V_{2}(\X(s))\dot{\X}(s)+\W_{2}(\Y(s))\dot{\Y}(s)=0$ for all $s\in [s_{0},s_{1}]$, which implies that $\V_{2}(X)=0$ for almost every $X\in [\X(s_{0}),\X(s_{1})]$ and $\W_{2}(Y)=0$ for almost every $Y\in [\Y(s_{0}),\Y(s_{1})]$. Then, by \eqref{eq:setGrel1}, \eqref{eq:setGrel3} and \eqref{eq:setGrel4}, we have $\V_{1}(X)=\V_{3}(X)=\p(X)=0$ for almost every $X\in [\X(s_{0}),\X(s_{1})]$ and $\W_{1}(Y)=\W_{3}(Y)=\q(Y)=0$ for almost every $Y\in [\Y(s_{0}),\Y(s_{1})]$. Hence, $\Z_{1}(s)$ is constant for all $s\in [s_{0},s_{1}]$ and we define\footnote{In a rectangular domain where $t=0$, $(\X,\Y)$ as defined by \eqref{eq:Emap} has to consist of the vertical straight line connecting the lower left diagonal point with the upper left corner, and the horizontal straight line connecting the upper left corner with the upper right diagonal point, since $t_{X}\geq0$ and $t_{Y}\leq0$} $\tilde{t}(X,Y)=0$ in $\Omega$. Let
\begin{align*}
	\tilde{x}(X,Y)&=\Z_2(s), & \tilde{U}(X,Y)&=\Z_3(s), \\
	\tilde{J}(X,Y)&=J_1(X)+J_2(Y), & \tilde{K}(X,Y)&=K_1(X)+K_2(Y), \\
	\tilde{p}(X,Y)&=\p(X), & \tilde{q}(X,Y)&=\q(Y).
\end{align*}
Then, $(\tilde{Z},\tilde{p},\tilde{q})$ is a solution of \eqref{eq:goveq} in $\Omega$. By the uniqueness of the solution, we get $(\tilde{Z},\tilde{p},\tilde{q})=(Z,p,q)$. In particular, we have $t(\bar{X},\bar{Y})=0$ such that $\bar{X}<\X(\bar{s})$ and $\bar{X}+\bar{Y}=2\bar{s}$, which contradicts the definition of $\X$ given by \eqref{eq:Emap}. Hence, we conclude that \eqref{eq:FHbijection1x1andx2} holds. By \eqref{eq:mapfromG0toF1}, we have $x_{1}(\X(s))=x_{2}(2s-\X(s))$. Thus, \eqref{eq:mapFtoGX} implies that $\bar{\X}(s)\leq \X(s)$ and it follows that $\bar{\Y}(s)\geq \Y(s)$. From \eqref{eq:mapFtoG2}, we have
\begin{equation}
\label{eq:FHbijection1x1equalsx2}
	x_{1}(\bar{\X}(s))=x_{2}(\bar{\Y}(s)).
\end{equation}  
Let us assume that $\bar{\X}(s)<\X(s)$. Then, by \eqref{eq:FHbijection1x1andx2}, we have either
$x_{1}(\bar{\X}(s))<x_{1}(\X(s))$ or $x_{2}(\bar{\Y}(s))>x_{2}(\Y(s))$. If $x_{1}(\bar{\X}(s))<x_{1}(\X(s))$, then
\begin{equation*}
	x_{1}(\bar{\X}(s))<x_{1}(\X(s))=x_{2}(\Y(s))\leq x_{2}(\bar{\Y}(s)),
\end{equation*}
which contradicts \eqref{eq:FHbijection1x1equalsx2}. Similarly, if $x_{2}(\bar{\Y}(s))>x_{2}(\Y(s))$, we obtain the contradiction
\begin{equation*}
	x_{2}(\bar{\Y}(s))>x_{2}(\Y(s))=x_{1}(\X(s))\geq x_{1}(\bar{\X}(s)).
\end{equation*}
Hence, $\bar{\X}=\X$ and therefore $\bar{\Y}=\Y$. Then, by \eqref{eq:mapFtoGZ2} and \eqref{eq:mapfromG0toF1}, we have $\bar{\Z}_{2}(s)=x_{1}(\bar{\X}(s))=x_{1}(\X(s))=\Z_{2}(s)$. Similarly, one finds that $\bar{\Z}_{3}=\Z_{3}$. By \eqref{eq:mapFtoGZ1} and since $(\X,\Y,\Z,\V,\W,\p,\q)\in \G_{0}$, we have $\bar{\Z}_{1}=\Z_{1}=0$. We have
\begin{align*}
	\bar{\Z}_{4}(s)&=J_{1}(\bar{\X}(s))+J_{2}(\bar{\Y}(s)) \quad \text{by } \eqref{eq:mapFtoGZ4}\\
	&=J_{1}(\X(s))+J_{2}(\Y(s))\\
	&=\int_{-\infty}^{s}(\V_{4}(\X(\tau))\dot{\X}(\tau)+\W_{4}(\Y(\tau))\dot{\Y}(\tau))\,d\tau \quad \text{by } \eqref{eq:mapfromG0toF3}\\
	&=\int_{-\infty}^{s}\dot{\Z}_{4}(\tau)\,d\tau=\Z_{4}(s) \quad \text{by } \eqref{eq:setGcomp}
\end{align*}
and by a similar calculation, we obtain $\bar{\Z}_{5}=\Z_{5}$. Let us verify that $\bar{\V}=\V$ (one shows that $\bar{\W}=\W$ in a similar way). By differentiating $x_{1}(\X(s))=\Z_{2}(s)$ and using \eqref{eq:setG0rel}, we obtain $x_{1}'=2\V_{2}$. This yields
\begin{equation*}
	\bar{\V}_{1}(\bar{\X})=\bar{\V}_{1}(\X)=\frac{1}{2c(U_{1}(\X))}x_{1}'(\X)=\frac{1}{c(\Z_{3})}\V_{2}(\X)=\V_{1}(\X)
\end{equation*}
by \eqref{eq:mapFtoGV1} and \eqref{eq:setGrel1}, and
\begin{equation*}
	\bar{\V}_{2}(\bar{\X})=\bar{\V}_{2}(\X)=\frac{1}{2}x_{1}'(\X)=\V_{2}(\X)
\end{equation*}
by \eqref{eq:mapFtoGV2}. From \eqref{eq:mapFtoGV3}-\eqref{eq:mapFtoGp} and \eqref{eq:mapfromG0toF3}-\eqref{eq:mapfromG0toF6}, we obtain
\begin{align*}
	\bar{\V}_{3}(\bar{\X})&=\bar{\V}_{3}(\X)=V_{1}(\X)=\V_{3}(\X),\\
	\bar{\V}_{4}(\bar{\X})&=\bar{\V}_{4}(\X)=J_{1}'(\X)=\V_{4}(\X),\\
	\bar{\V}_{5}(\bar{\X})&=\bar{\V}_{5}(\X)=K_{1}'(\X)=\V_{5}(\X),\\
	\bar{\p}(\bar{\X})&=\bar{\p}(\X)=H_{1}(\X)=\p(\X),\\
	\bar{\q}(\bar{\Y})&=\bar{\q}(\Y)=H_{2}(\Y)=\q(\X).
\end{align*}
Hence, we have proved that $\mathbf{C}\circ\mathbf{D}\circ\mathbf{E}=\mathbf{E}$. By a straightforward calculation, using Definition \ref{def:mapfromFtoG0} and Definition \ref{def:mapfromG0toF}, one proves that $\mathbf{D}\circ \mathbf{C}=\id$. This concludes the proof of \eqref{eq:FHbijection1}. 

Next we prove \eqref{eq:FHbijection2}. Given $(\psi_{1},\psi_{2})\in \F$, let
\begin{align*}
	(\X,\Y,\Z,\V,\W,\p,\q)&=\mathbf{C}(\psi_{1},\psi_{2}),\\
	(Z,p,q)&=\mathbf{S}(\X,\Y,\Z,\V,\W,\p,\q),\\
	(\bar{\X},\bar{\Y},\bar{\Z},\bar{\V},\bar{\W},\bar{\p},\bar{\q})&=\mathbf{E}(Z,p,q).
\end{align*}
As before, we first show that $\bar{\X}=\X$. Since $(Z,p,q)\in \H$ is a solution with $(Z,p,q)\bullet (\X,\Y)=(\X,\Y,\Z,\V,\W,\p,\q)\in \G_{0}$, we have that $t(\X(s),\Y(s))=0$. Hence, by \eqref{eq:Emap}, we get $\bar{\X}(s)\leq \X(s)$. Assume that there exists $s\in \mathbb{R}$ such that $\bar{\X}(s)<\X(s)$. Let $s_{0}=\X^{-1}(\bar{\X}(s))$ and $s_{1}=\Y^{-1}(\bar{\Y}(s))$. Since
$\X(s_{0})=\bar{\X}(s)<\X(s)$ and $\Y(s_{1})=\bar{\Y}(s)>\Y(s)$, we have $s_{0}<s<s_{1}$. By \eqref{eq:setH1}, \eqref{eq:setH4} and since $t(\X(s_{0}),\Y(s_{0}))=t(\bar{\X}(s),\bar{\Y}(s))=t(\X(s_{1}),\Y(s_{1}))=0$, we get $t_{Y}(\X(s_{0}),Y)=0$ for $Y\in [\Y(s_{0}),\Y(s_{1})]$ and $t_{X}(X,\Y(s_{1}))=0$ for $X\in [\X(s_{0}),\X(s_{1})]$. This implies, by \eqref{eq:setH1}, that $x_{Y}(\X(s_{0}),Y)=0$ for $Y\in [\Y(s_{0}),\Y(s_{1})]$ and $x_{X}(X,\Y(s_{1}))=0$ for $X\in [\X(s_{0}),\X(s_{1})]$. Then,
\begin{align*}
	x_{1}(\bar{\X}(s))&=x_{1}(\X(s_{0}))\\
	&=\Z_{2}(s_{0}) \quad \text{by } \eqref{eq:mapFtoGZ2}\\
	&=x(\X(s_{0}),\Y(s_{0})) \quad \text{since } (Z,p,q)\bullet (\X,\Y)=(\X,\Y,\Z,\V,\W,\p,\q)\\
	&=x(\X(s_{0}),\Y(s_{1}))\\
	&=x(\X(s_{1}),\Y(s_{1}))=\Z_{2}(s_{1})=x_{2}(\Y(s_{1}))=x_{2}(\bar{\Y}(s)).
\end{align*}
However, the fact that $x_{1}(\bar{\X}(s))=x_{2}(\bar{\Y}(s))$ and $\bar{\X}(s)<\X(s)$ contradicts the definition of $\X$ in \eqref{eq:mapFtoGX}. Hence, we must have $\bar{\X}=\X$, which yields $\bar{\Y}=\Y$ and
\begin{align*}
	\bar{\Z}(s)&=Z(\bar{\X}(s),\bar{\Y}(s))=Z(\X(s),\Y(s))=\Z(s),\\
	\bar{\V}(\bar{\X})&=\bar{\V}(\X)=Z_{X}(\X,\Y)=\V(\X),\\
	\bar{\W}(\bar{\Y})&=\bar{\W}(\Y)=Z_{Y}(\X,\Y)=\W(\Y),\\
	\bar{\p}(\bar{\X})&=\bar{\p}(\X)=p(\X,\Y)=\p(\X),\\
	\bar{\q}(\bar{\Y})&=\bar{\q}(\Y)=q(\X,\Y)=\q(\Y).
\end{align*}  
Hence, we have proved that $\mathbf{E}\circ\mathbf{S}\circ\mathbf{C}=\mathbf{C}$. By the uniqueness of the solution for given data $(\X,\Y,\Z,\V,\W,\p,\q)\in \G_{0}$, we have that $\mathbf{S}\circ\mathbf{E}=\id$. This concludes the proof of \eqref{eq:FHbijection2}.
\end{proof}

\begin{theorem}
\label{thm:STsemigroup}
The mapping $S_{T}$ is a semigroup.
\end{theorem}

\begin{proof}
We have
\begin{align*}
	S_{T} \circ S_{T'}&=
	\mathbf{D}\circ\mathbf{E}\circ\mathbf{t}_{T}\circ\mathbf{S}\circ\mathbf{C}\circ\mathbf{D}\circ\mathbf{E}\circ\mathbf{t}_{T'}\circ\mathbf{S}\circ\mathbf{C}\\
	&=\mathbf{D}\circ\mathbf{E}\circ\mathbf{t}_{T}\circ\mathbf{t}_{T'}\circ\mathbf{S}\circ\mathbf{C} \quad \text{by Lemma } \ref{lemma:setrelations}\\
	&=\mathbf{D}\circ\mathbf{E}\circ\mathbf{t}_{T+T'}\circ\mathbf{S}\circ\mathbf{C} \quad \text{by } \eqref{eq:tsemigroup}\\
	&=S_{T+T'}.
\end{align*}
\end{proof}

\subsection{Mapping from $\F$ to $\D$}

\begin{definition}
\label{def:mapMFtoD}
Given $\psi=(\psi_{1},\psi_{2})\in \F$, we define $(u,R,S,\rho,\sigma,\mu,\nu)$ as
\begin{subequations}
\label{eqns:mapFtoD1}
\begin{equation}
\label{eq:mapFtoD1}
	u(x)=U_{1}(X) \quad \text{if } x_{1}(X)=x
\end{equation}
or, equivalently,
\begin{equation}
\label{eq:mapFtoD2}
	u(x)=U_{2}(Y) \quad \text{if } x_{2}(Y)=x,
\end{equation}
\begin{align}
	\label{eq:mapFtoD3}
	R(x)\,dx&=(x_{1})_{\#}(2c(U_{1}(X))V_{1}(X)\,dX),\\
	\label{eq:mapFtoD4}
	S(x)\,dx&=(x_{2})_{\#}(-2c(U_{2}(Y))V_{2}(Y)\,dY),\\
	\label{eq:mapFtoD5}
	\rho(x)\,dx&=(x_{1})_{\#}(2H_{1}(X)\,dX),\\
	\label{eq:mapFtoD6}
	\sigma(x)\,dx&=(x_{2})_{\#}(2H_{2}(Y)\,dY),\\
	\label{eq:mapFtoD7}
	\mu&=(x_{1})_{\#}(J_{1}'(X)\,dX),\\
	\label{eq:mapFtoD8}
	\nu&=(x_{2})_{\#}(J_{2}'(Y)\,dY).
\end{align}
\end{subequations}
The relations \eqref{eq:mapFtoD3}-\eqref{eq:mapFtoD6} are equivalent to
\begin{subequations}
\begin{align}
\label{eq:mapFtoD9}
	R(x_{1}(X))x_{1}'(X)&=2c(U_{1}(X))V_{1}(X),\\
	S(x_{2}(Y))x_{2}'(Y)&=-2c(U_{2}(Y))V_{2}(Y),\\
	\label{eq:mapFtoD11}
	\rho(x_{1}(X))x_{1}'(X)&=2H_{1}(X),\\
	\label{eq:mapFtoD12}
	\sigma(x_{2}(Y))x_{2}'(Y)&=2H_{2}(Y),
\end{align}
\end{subequations}
respectively, for almost every $X$ and $Y$. We denote by $\bf{M}:\F\rightarrow \D$ the mapping which to any $\psi\in \F$ associates the element $(u,R,S,\rho,\sigma,\mu,\nu)\in \D$ as defined above.
\end{definition}

The push-forward of a measure $\lambda$ by a function $f$ is the measure $f_{\#}\lambda$ defined by $f_{\#}\lambda(B)=\lambda(f^{-1}(B))$ for Borel sets $B$. 

The well-posedness of Definition \ref{def:mapMFtoD} is part of the proof of the following lemma. 

\begin{lemma}
\label{lemma:mapG0toD}
Given $\psi=(\psi_{1},\psi_{2})\in \F$, let $(u,R,S,\rho,\sigma,\mu,\nu)=\mathbf{M}(\psi_{1},\psi_{2})$. Then, for any $\Theta=(\X,\Y,\Z,\V,\W,\p,\q)\in \G_{0}$ such that $(\psi_{1},\psi_{2})=\mathbf{D}(\X,\Y,\Z,\V,\W,\p,\q)$, we have
\begin{subequations}
\label{eqns:mapGtoD1}
\begin{equation}
\label{eq:mapGtoD1}
	u(x)=\Z_{3}(s) \quad \text{if } x=\Z_{2}(s),
\end{equation}
\begin{align}
	\label{eq:mapGtoD3}
	R(x)\,dx&=(\Z_{2})_{\#}(2c(\Z_{3}(s))\V_{3}(\X(s))\dot{\X}(s)\,ds),\\
	\label{eq:mapGtoD4}
	S(x)\,dx&=(\Z_{2})_{\#}(-2c(\Z_{3}(s))\W_{3}(\Y(s))\dot{\Y}(s)\,ds),\\
	\label{eq:mapGtoD5}
	\rho(x)\,dx&=(\Z_{2})_{\#}(2\p(\X(s))\dot{\X}(s)\,ds),\\
	\label{eq:mapGtoD6}
	\sigma(x)\,dx&=(\Z_{2})_{\#}(2\q(\Y(s))\dot{\Y}(s)\,ds),\\
	\label{eq:mapGtoD7}
	\mu&=(\Z_{2})_{\#}(\V_{4}(\X(s))\dot{\X}(s)\,ds),\\
	\label{eq:mapGtoD8}
	\nu&=(\Z_{2})_{\#}(\W_{4}(\Y(s))\dot{\Y}(s)\,ds).
\end{align}
\end{subequations}
The relations \eqref{eq:mapGtoD3} and \eqref{eq:mapGtoD5} are equivalent to
\begin{subequations}
\label{eq:mapGtoD9and10}
\begin{align}
	\label{eq:mapGtoD9}
	R(\Z_{2}(s))\V_{2}(\X(s))&=c(\Z_{3}(s))\V_{3}(\X(s)),\\
	\label{eq:mapGtoD10}
	\rho(\Z_{2}(s))\V_{2}(\X(s))&=\p(\X(s))
\end{align}
\end{subequations}
for any $s$ such that $\dot{\X}(s)>0$, respectively. The relations \eqref{eq:mapGtoD4} and \eqref{eq:mapGtoD6} are equivalent to
\begin{subequations}
\begin{align}
	\label{eq:mapGtoD11}
	S(\Z_{2}(s))\W_{2}(\Y(s))&=-c(\Z_{3}(s))\W_{3}(\Y(s)),\\
	\label{eq:mapGtoD12}
	\sigma(\Z_{2}(s))\W_{2}(\Y(s))&=\q(\Y(s))
\end{align}
\end{subequations}
for any $s$ such that $\dot{\Y}(s)>0$, respectively.
\end{lemma}

\begin{proof}
We decompose the proof into five steps.
	
\textbf{Step 1.} We first prove that \eqref{eqns:mapFtoD1} implies \eqref{eqns:mapGtoD1}. If $x=x_{1}(X)$, let $s=\X^{-1}(X)$. Then, by \eqref{eq:mapfromG0toF1}, $x=x_{1}(X)=x_{1}(\X(s))=\Z_{2}(s)$, and, by \eqref{eq:mapfromG0toF2}, $u(x)=U_{1}(X)=U_{1}(\X(s))=\Z_{3}(s)$. Similarly, if $x=x_{2}(Y)$, we let $s=\Y^{-1}(Y)$ and obtain $x=x_{2}(Y)=x_{2}(\Y(s))=\Z_{2}(s)$ and $u(x)=U_{2}(Y)=U_{2}(\Y(s))=\Z_{3}(s)$. Hence, both \eqref{eq:mapFtoD1} and \eqref{eq:mapFtoD2} imply 
\eqref{eq:mapGtoD1}. The identity \eqref{eq:mapGtoD3} follows from \eqref{eq:mapFtoD3} since, for any Borel set $A$, we have
\begin{align*}
	\int_{A}R(x)\,dx&=\int_{x_{1}^{-1}(A)}2c(U_{1}(X))V_{1}(X)\,dX\\
	&=\int_{(x_{1}\circ \X)^{-1}(A)}2c(U_{1}(\X(s)))V_{1}(\X(s))\dot{\X}(s)\,ds \quad \text{by a change of variables}\\
	&=\int_{\Z_{2}^{-1}(A)}2c(\Z_{3}(s))\V_{3}(\X(s))\dot{\X}(s)\,ds \quad \text{by \eqref{eq:mapfromG0toF1}, \eqref{eq:mapfromG0toF2} and \eqref{eq:mapfromG0toF5}.} 
\end{align*}
In the same way, one proves that \eqref{eq:mapFtoD4} implies \eqref{eq:mapGtoD4}. By a similar calculation as above, we obtain
\begin{align*}
	\int_{A}\rho(x)\,dx&=\int_{x_{1}^{-1}(A)}2H_{1}(X)\,dX\\
	&=\int_{(x_{1}\circ \X)^{-1}(A)}2H_{1}(\X(s))\dot{\X}(s)\,ds\\
	&=\int_{\Z_{2}^{-1}(A)}2\p(\X(s))\dot{\X}(s)\,ds
\end{align*}
and
\begin{align*}
	\int_{A}\sigma(x)\,dx&=\int_{x_{2}^{-1}(A)}2H_{2}(Y)\,dY\\
	&=\int_{(x_{2}\circ \Y)^{-1}(A)}2H_{2}(\Y(s))\dot{\Y}(s)\,ds\\
	&=\int_{\Z_{2}^{-1}(A)}2\q(\Y(s))\dot{\Y}(s)\,ds,
\end{align*}
which shows that \eqref{eq:mapFtoD5} and \eqref{eq:mapFtoD6} imply \eqref{eq:mapGtoD5} and \eqref{eq:mapGtoD6}, respectively. From \eqref{eq:mapFtoD7}, we find 
\begin{equation*}
	\mu(A)=\int_{x_{1}^{-1}(A)}J_{1}'(X)\,dX=\int_{(x_{1}\circ \X)^{-1}(A)}J_{1}'(\X(s))\dot{\X}(s)\,ds
	=\int_{\Z_{2}^{-1}(A)}\V_{4}(\X(s))\dot{\X}(s)\,ds, 
\end{equation*}
so that \eqref{eq:mapFtoD7} leads to \eqref{eq:mapGtoD7}. By a similar calculation, one shows that \eqref{eq:mapFtoD8} yields \eqref{eq:mapGtoD8}.  
	
\textbf{Step 2.} We prove that $u$ is a well-defined function that belongs to $L^{2}(\mathbb{R})$. Consider $s_{0},s_{1}\in \mathbb{R}$ such that $s_{0}<s_{1}$ and $x=\Z_{2}(s_{0})=\Z_{2}(s_{1})$. Since $\Z_{2}$ is continuous and nondecreasing, we have $\dot{\Z}_{2}(s)=\V_{2}(\X(s))\dot{\X}(s)+\W_{2}(\Y(s))\dot{\Y}(s)=0$ for $s\in [s_{0},s_{1}]$, which implies that $\V_{2}(\X(s))\dot{\X}(s)=\W_{2}(\Y(s))\dot{\Y}(s)=0$. Then, by multiplying \eqref{eq:setGrel3} with $\dot{\X}(s)^{2}$ and \eqref{eq:setGrel4} with $\dot{\Y}(s)^{2}$, we obtain $\V_{3}(\X(s))\dot{\X}(s)=\W_{3}(\Y(s))\dot{\Y}(s)=0$ and therefore $\dot{\Z}_{3}(s)=\V_{3}(\X(s))\dot{\X}(s)+\W_{3}(\Y(s))\dot{\Y}(s)=0$ for $s\in [s_{0},s_{1}]$. Hence, $\Z_{3}(s_{0})=\Z_{3}(s_{1})$ and \eqref{eq:mapGtoD1} is well-defined. We have
\begin{align*}
	\int_{\mathbb{R}}u^{2}(x)\,dx&=\int_{\mathbb{R}}u^{2}(\Z_{2}(s))\dot{\Z}_{2}(s)\,ds \quad \text{by a change of variables}\\
	&=\int_{\mathbb{R}}\Z_{3}^{2}(s)\dot{\Z}_{2}(s)\,ds \quad \text{by } \eqref{eq:mapGtoD1}\\
	&=2\int_{\mathbb{R}}\Z_{3}^{2}(s)\V_{2}(\X(s))\dot{\X}(s)\,ds \quad \text{by } \eqref{eq:setG0rel}\\
	&\leq 4\int_{\mathbb{R}}\Z_{3}^{2}(s)\V_{2}(\X(s))\,ds \quad \text{since } 0\leq \dot{\X}\leq 2 \text{ and } \V_{2}\geq0\\
	&=4\int_{\mathbb{R}}\Z_{3}^{2}(s)\bigg(\V^{a}_{2}(\X(s))+\frac{1}{2}\bigg)\,ds \quad \text{by } \eqref{eq:atriplet}\\
	&\leq 4\bigg(||\V^{a}_{2}||_{\Linf(\mathbb{R})}+\frac{1}{2}\bigg)||\Z_{3}||_{L^{2}(\mathbb{R})}^{2}
\end{align*}
and $u\in L^{2}(\mathbb{R})$. 
	
\textbf{Step 3.} We show that the definitions \eqref{eq:mapGtoD3}-\eqref{eq:mapGtoD6} are well-defined, and that the relations \eqref{eq:mapGtoD9}-\eqref{eq:mapGtoD12} hold. First we prove that the measures 
\begin{align*}
	&(\Z_{2})_{\#}(2c(\Z_{3}(s))\V_{3}(\X(s))\dot{\X}(s)\,ds),\\
	&(\Z_{2})_{\#}(-2c(\Z_{3}(s))\W_{3}(\Y(s))\dot{\Y}(s)\,ds),\\
	&(\Z_{2})_{\#}(2\p(\X(s))\dot{\X}(s)\,ds),\\
	&(\Z_{2})_{\#}(2\q(\Y(s))\dot{\Y}(s)\,ds)
\end{align*}
are absolutely continuous with respect to Lebesgue measure. We claim that the function
\begin{equation*}
	F(x)=\int_{\Z_{2}^{-1}((-\infty,x])}2c(\Z_{3}(s))\V_{3}(\X(s))\dot{\X}(s)\,ds
\end{equation*}
is absolutely continuous. Let $(x_{i},\bar{x}_{i})$, $i=1,\dots,N$, be non-intersecting intervals. We have
\begin{equation*}
	\sum_{i=1}^{N}|F(\bar{x}_{i})-F(x_{i})|=\sum_{i=1}^{N}\bigg|\int_{\Z_{2}^{-1}((x_{i},\bar{x}_{i}])}2c(\Z_{3}(s))\V_{3}(\X(s))\dot{\X}(s)\,ds\bigg|.
\end{equation*}
The set $\Z_{2}^{-1}((x_{i},\bar{x}_{i}])$ is an interval, since the function $\Z_{2}$ is nondecreasing and continuous. Denote $s_{i}=\sup\{s\in\mathbb{R} \ | \ \Z_{2}(s)\leq x_{i}\}$ and $\bar{s}_{i}=\sup\{s\in\mathbb{R} \ | \ \Z_{2}(s)\leq\bar{x}_{i}\}$. We have $\Z_{2}(s_{i})=x_{i}$, $\Z_{2}(\bar{s}_{i})=\bar{x}_{i}$ and  $\Z_{2}^{-1}((x_{i},\bar{x}_{i}])=(s_{i},\bar{s}_{i}]$. Then
\begin{equation*}
	\sum_{i=1}^{N}|F(\bar{x}_{i})-F(x_{i})|
	\leq 2\kappa \int_{\displaystyle{\bigcup_{i=1}^{N}}(s_{i},\bar{s}_{i}]}|\V_{3}(\X(s))|\dot{\X}(s)\,ds. 
\end{equation*}
From \eqref{eq:setGrel3}, we obtain $|\V_{3}(\X)|\leq \kappa(2\V_{4}(\X)\V_{2}(\X))^{\frac{1}{2}}$. This implies, by the Cauchy--Schwarz inequality, that
\begin{aalign}
	\label{eq:mapGtoDFest}
	&\sum_{i=1}^{N}|F(\bar{x}_{i})-F(x_{i})|\\
	&\leq 2\kappa^{2}\int_{\displaystyle{\bigcup_{i=1}^{N}}(s_{i},\bar{s}_{i}]}(\V_{4}(\X(s))\dot{\X}(s))^{\frac{1}{2}}(2\V_{2}(\X(s))\dot{\X}(s))^{\frac{1}{2}}\,ds\\
	&\leq 2\kappa^{2}\bigg(\int_{\displaystyle{\bigcup_{i=1}^{N}}(s_{i},\bar{s}_{i}]}\V_{4}(\X(s))\dot{\X}(s)\,ds\bigg)^{\frac{1}{2}}\bigg(\int_{\displaystyle{\bigcup_{i=1}^{N}}(s_{i},\bar{s}_{i}]} 2\V_{2}(\X(s))\dot{\X}(s)\,ds\bigg)^{\frac{1}{2}}.
\end{aalign}	
Inserting the estimates
\begin{align*}	
	\int_{\displaystyle{\bigcup_{i=1}^{N}}(s_{i},\bar{s}_{i}]}\V_{4}(\X(s))\dot{\X}(s)\,ds
	&\leq\int_{\mathbb{R}}\V_{4}(\X(s))\dot{\X}(s)\,ds\\
	&\leq\int_{\mathbb{R}}(\V_{4}(\X(s))\dot{\X}(s)+\W_{4}(\Y(s))\dot{\Y}(s))\,ds\\
	&=\int_{\mathbb{R}}\dot{\Z}_{4}(s)\,ds\\
	&\leq||\Z_{4}^{a}||_{\Linf(\mathbb{R})} \quad \text{by } \eqref{eq:setGrel5}
\end{align*}
and
\begin{equation*}
	\int_{\displaystyle{\bigcup_{i=1}^{N}}(s_{i},\bar{s}_{i}]}2\V_{2}(\X(s))\dot{\X}(s)\,ds=\int_{\displaystyle{\bigcup_{i=1}^{N}}(s_{i},\bar{s}_{i}]}\dot{\Z}_{2}(s)\,ds=\sum_{i=1}^{N}|\bar{x}_{i}-x_{i}|
\end{equation*}
into \eqref{eq:mapGtoDFest}, we get
\begin{equation*}
	\sum_{i=1}^{N}|F(\bar{x}_{i})-F(x_{i})|\leq C\bigg(\sum_{i=1}^{N}|\bar{x}_{i}-x_{i}|\bigg)^{\frac{1}{2}}
\end{equation*}
for a constant $C$ which only depends on $|||\Theta|||_{\G}$ and $\kappa$. This implies that $F$ is absolutely continuous. Then, the measure $(\Z_{2})_{\#}(2c(\Z_{3}(s))\V_{3}(\X(s))\dot{\X}(s)\,ds)$ is absolutely continuous. In the same way, one proves that measures \\$(\Z_{2})_{\#}(-2c(\Z_{3}(s))\W_{3}(\Y(s))\dot{\Y}(s)\,ds)$, $(\Z_{2})_{\#}(2\p(\X(s))\dot{\X}(s)\,ds)$ and
\\$(\Z_{2})_{\#}(2\q(\Y(s))\dot{\Y}(s)\,ds)$ are absolutely continuous, so that the functions $R$, $S$, $\rho$ and $\sigma$ as given by \eqref{eq:mapGtoD3}-\eqref{eq:mapGtoD6} are well-defined. 
	
Let us prove \eqref{eq:mapGtoD9}. We have
\begin{equation}
	\label{eq:mapGtoDZ2}
	\int_{\Z_{2}^{-1}(A)}R(\Z_{2}(s))\dot{\Z}_{2}(s)\,ds=\int_{\Z_{2}^{-1}(A)}2c(\Z_{3}(s))\V_{3}(\X(s))\dot{\X}(s)\,ds
\end{equation}
for any Borel set $A$, and we want to show that for any measurable set $B$, 
\begin{equation}
	\label{eq:mapGtoDB}
	\int_{B}R(\Z_{2}(s))\dot{\Z}_{2}(s)\,ds=\int_{B}2c(\Z_{3}(s))\V_{3}(\X(s))\dot{\X}(s)\,ds.
\end{equation}
For any measurable set $B$, we have the decomposition $\Z_{2}^{-1}(\Z_{2}(B))=B\cup(B^{c}\cap\Z_{2}^{-1}(\Z_{2}(B)))$. Let us prove that $\dot{\Z}_{2}=0$ on $B^{c}\cap\Z_{2}^{-1}(\Z_{2}(B))$. Consider a point $\bar{s}\in B^{c}\cap\Z_{2}^{-1}(\Z_{2}(B))$. There exists $\tilde{s}\in B$ such that $\Z_{2}(\bar{s})=\Z_{2}(\tilde{s})$, which implies, since $\Z_{2}$ is nondecreasing, that $\dot{\Z}_{2}=0$ on the interval joining the points $\bar{s}$ and $\tilde{s}$. Since $\bar{s}$ was arbitrary, we conclude that $\dot{\Z}_{2}(s)=0$ for all $s\in B^{c}\cap\Z_{2}^{-1}(\Z_{2}(B))$. Then, by an estimate as above, we get
\begin{align*}
	&\bigg|\int_{B^{c}\cap\Z_{2}^{-1}(\Z_{2}(B))}2c(\Z_{3}(s))\V_{3}(\X(s))\dot{\X}(s)\,ds\bigg|\\ &\leq
	2\kappa^{2}||\Z_{4}^{a}||_{\Linf(\mathbb{R})}^{\frac{1}{2}}\bigg(\int_{B^{c}\cap\Z_{2}^{-1}(\Z_{2}(B))}\dot{\Z}_{2}(s)\,ds \bigg)^{\frac{1}{2}}=0.
\end{align*}
Hence, by taking $A=\Z_{2}(B)$ in \eqref{eq:mapGtoDZ2}, we obtain \eqref{eq:mapGtoDB}. Thus, 
\begin{equation*}
	R(\Z_{2}(s))\dot{\Z}_{2}(s)=2c(\Z_{3}(s))\V_{3}(\X(s))\dot{\X}(s)
\end{equation*}
which yields, because $\dot{\Z}_{2}(s)=2\V_{2}(\X(s))\dot{\X}(s)$,
\begin{equation*}
	R(\Z_{2}(s))\V_{2}(\X(s))=c(\Z_{3}(s))\V_{3}(\X(s))
\end{equation*}
for any $s$ such that $\dot{\X}(s)>0$. Similarly, one proves \eqref{eq:mapGtoD10}, \eqref{eq:mapGtoD11} and \eqref{eq:mapGtoD12}.
	
\textbf{Step 4.} We show that $R$, $S$, $\rho$ and $\sigma$ belong to $L^{2}(\mathbb{R})$, and $u_{x}=\frac{R-S}{2c(u)}$. Since
\begin{align*}
	\int_{\mathbb{R}}R^{2}(x)\,dx&=\int_{\mathbb{R}}R^{2}(\Z_{2}(s))\dot{\Z}_{2}(s)\,ds \quad \text{by a change of variables}\\
	&=2\int_{\mathbb{R}}R^{2}(\Z_{2}(s))\V_{2}(\X(s))\dot{\X}(s)\,ds \quad \text{by \eqref{eq:setG0rel}}\\
	&=2\int_{\{s\in \mathbb{R}\, | \, \V_{2}(\X(s))>0 \}}\frac{\big(R(\Z_{2}(s))\V_{2}(\X(s))\big)^{2}}{\V_{2}(\X(s))}\dot{\X}(s)\,ds\\
	&=2\int_{\{s\in \mathbb{R}\, | \, \V_{2}(\X(s))>0 \}}\frac{\big(c(\Z_{3}(s))\V_{3}(\X(s))\big)^{2}}{\V_{2}(\X(s))}\dot{\X}(s)\,ds \quad \text{by \eqref{eq:mapGtoD9}}\\
	&\leq 4\int_{\mathbb{R}}\V_{4}(\X(s))\dot{\X}(s)\,ds \quad \text{by } \eqref{eq:setGrel3}\\
	&\leq 4\int_{\mathbb{R}}\big(\V_{4}(\X(s))\dot{\X}(s)+\W_{4}(\Y(s))\dot{\Y}(s)\big)\,ds \quad \text{since } \W_{4}\geq 0\\
	&=4\int_{\mathbb{R}}\dot{\Z_{4}}(s)\,ds\\
	&\leq 4||\Z_{4}^{a}||_{\Linf(\mathbb{R})} \quad \text{by } \eqref{eq:setGrel5}, 
\end{align*}
$R$ belongs to $L^{2}(\mathbb{R})$. Similarly, using \eqref{eq:mapGtoD10}, \eqref{eq:mapGtoD11} and \eqref{eq:mapGtoD12}, one proves that $\rho$, $S$ and $\sigma$ belong to $L^{2}(\mathbb{R})$, respectively. 
	
Let $\phi$ be a smooth test function with compact support. We have
\begin{align*}
	&\int_{\mathbb{R}}u(x)\phi_{x}(x)\,dx\\
	&=\int_{\mathbb{R}}u(\Z_{2}(s))\phi_{x}(\Z_{2}(s))\dot{\Z}_{2}(s)\,ds \quad \text{by a change of variables}\\
	&=\int_{\mathbb{R}}\Z_{3}(s)(\phi(\Z_{2}(s)))_{s}\,ds \quad \text{by \eqref{eq:mapGtoD1}}\\
	&=-\int_{\mathbb{R}}\dot{\Z}_{3}(s)\phi(\Z_{2}(s))\,ds \quad \text{by integrating by parts}\\
	&=-\int_{\mathbb{R}}\big(\V_{3}(\X(s))\dot{\X}(s)+\W_{3}(\Y(s))\dot{\Y}(s)\big)\phi(\Z_{2}(s))\,ds\\
	&=-\int_{\mathbb{R}}\frac{1}{c(\Z_{3}(s))}\big(R(\Z_{2}(s))\V_{2}(\X(s))\dot{\X}(s)\\
	&\hspace{90pt}-S(\Z_{2}(s))\W_{2}(\Y(s))\dot{\Y}(s)\big)\phi(\Z_{2}(s))\,ds \quad \text{by \eqref{eq:mapGtoD9and10}}\\
	&=-\int_{\mathbb{R}}\frac{1}{2c(\Z_{3}(s))}\big(R(\Z_{2}(s))-S(\Z_{2}(s))\big)\phi(\Z_{2}(s))\dot{\Z}_{2}(s)\,ds \quad \text{by \eqref{eq:setG0rel}}\\
	&=-\int_{\mathbb{R}}\frac{1}{2c(u(x))}\big(R(x)-S(x)\big)\phi(x)\,ds \quad \text{by a change of variables}.
\end{align*} 
Hence, $u_{x}=\frac{R-S}{2c(u)}$ in the sense of distributions.
	
\textbf{Step 5.}
We prove that $\mu_{\text{ac}}=\frac{1}{4}(R^{2}+c(u)\rho^{2})\,dx$ and $\nu_{\text{ac}}=\frac{1}{4}(S^{2}+c(u)\sigma^{2})\,dx$. Let
\begin{equation}
\label{eq:setAandB}
	A=\{s\in\mathbb{R} \ | \ \V_{2}(\X(s))>0 \} \quad \text{and} \quad B=(\Z_{2}(A^{c}))^{c}.
\end{equation}
Since $\dot{\Z}_{2}=2\V_{2}(\X)\dot{\X}$, we have $\dot{\Z}_{2}=0$ on $A^{c}$, so that 
\begin{equation*}
	\text{meas}(B^{c})=\int_{A^{c}}\dot{\Z}_{2}(s)\,ds=0.
\end{equation*}
Since $\Z_{2}^{-1}(B^{c})=\Z_{2}^{-1}(\Z_{2}(A^{c}))\supset A^{c}$, we have $\Z_{2}^{-1}(B)\subset A$\footnote{The following example is useful to have in mind. Suppose that $\Z_{2}$ is strictly increasing outside an interval $I$ on which $\dot{\Z}_{2}=0$. Assume further that $A^{c}$ is a subinterval of $I$. We have $\V_{2}(\X)=0$ on $A^{c}$, and $\dot{\X}=0$ on $I\setminus A^{c}$. In this case it is not hard to check that $\Z_{2}^{-1}(B)\subset A$}.
Let $M$ be any Borel set. We have $\mu(M)=\mu(M\cap B)+\mu(M\cap B^{c})$, and we claim that $\mu(M\cap B)$ is the absolutely continuous part of $\mu$. Since $M\cap B\subset B$, $\Z_{2}^{-1}(M\cap B)\subset A$. Hence,
\begin{align*}
	&\mu(M\cap B)\\
	&=\int_{\Z_{2}^{-1}(M\cap B)}\V_{4}(\X(s))\dot{\X}(s)\,ds\\
	&=\int_{\Z_{2}^{-1}(M\cap B)}\frac{\V_{4}(\X(s))\V_{2}(\X(s))}{\V_{2}(\X(s))}\dot{\X}(s)\,ds\\
	&=\int_{\Z_{2}^{-1}(M\cap B)}\frac{(c(\Z_{3}(s))\V_{3}(\X(s)))^{2}+c(\Z_{3}(s))\p^{2}(\X(s)) }{2\V_{2}(\X(s))}\dot{\X}(s)\,ds \quad \text{by } \eqref{eq:setGrel3}\\
	&=\frac{1}{4}\int_{\Z_{2}^{-1}(M\cap B)}(R^{2}(\Z_{2}(s))+c(\Z_{3}(s))\rho^{2}(\Z_{2}(s)))2\V_{2}(\X(s))\dot{\X}(s)\,ds \quad \text{by } \eqref{eq:mapGtoD9and10}\\
	&=\frac{1}{4}\int_{\Z_{2}^{-1}(M\cap B)}(R^{2}(\Z_{2}(s))+c(\Z_{3}(s))\rho^{2}(\Z_{2}(s)))\dot{\Z}_{2}(s)\,ds\\
	&=\frac{1}{4}\int_{M\cap B}(R^{2}(x)+c(u(x))\rho^{2}(x))\,dx \quad \text{by a change of variables}.
\end{align*}
It follows that for any Borel set $M$ with measure zero, $\mu(M\cap B)=0$, so that $\mu_{\text{ac}}=\frac{1}{4}(R^{2}(x)+c(u(x))\rho^{2}(x))\,dx$. Similarly, one proves that $\nu_{\text{ac}}=\frac{1}{4}(S^{2}(x)+c(u(x))\sigma^{2}(x))\,dx$.
	
For further reference, let us prove that the singular part of $\mu$, $\mu_{\text{sing}}(M)=\mu(M\cap B^{c})=\mu(M\cap\Z_{2}(A^{c}))$,  can be written as
\begin{equation}
	\label{eq:mapGtoDsing1}
	\mu_{\text{sing}}(M)=\int_{\Z_{2}^{-1}(M)\cap A^{c}}\V_{4}(\X(s))\dot{\X}(s)\,ds.
\end{equation}
We have
\begin{equation}
	\label{eq:mapGtoDsing2}
	\mu_{\text{sing}}(M)=\int_{\Z_{2}^{-1}(M\cap\Z_{2}(A^{c}))}\V_{4}(\X(s))\dot{\X}(s)\,ds
\end{equation}
and
\begin{align*}
	\Z_{2}^{-1}(M\cap\Z_{2}(A^{c}))&=\Z_{2}^{-1}(M)\cap\Z_{2}^{-1}(\Z_{2}(A^{c}))\\
	&=\Z_{2}^{-1}(M)\cap(A^{c}\cup(A\cap\Z_{2}^{-1}(\Z_{2}(A^{c}))))\\
	&=(\Z_{2}^{-1}(M)\cap A^{c})\cup(\Z_{2}^{-1}(M)\cap(A\cap\Z_{2}^{-1}(\Z_{2}(A^{c})))). 
\end{align*}
Either the set $A\cap\Z_{2}^{-1}(\Z_{2}(A^{c}))$ is empty, or $\Z_{2}$ is constant on $A\cap\Z_{2}^{-1}(\Z_{2}(A^{c}))$, and that $\dot{\Z}_{2}=2\V_{2}(\X)\dot{\X}=0$ and, since $\V_{2}(\X)>0$ on $A$, we must have that $\dot{\X}=0$ on $A\cap\Z_{2}^{-1}(\Z_{2}(A^{c}))$. Then \eqref{eq:mapGtoDsing1} follows from \eqref{eq:mapGtoDsing2}. In a similar way, one shows that 
\begin{equation*}
	\nu_{\text{sing}}(M)=\int_{\Z_{2}^{-1}(M)\cap A^{c}}\W_{4}(\Y(s))\dot{\Y}(s)\,ds.
\end{equation*}
\end{proof}

By using the semigroup $S_{T}$ we can, together with the mappings from $\D$ to $\F$ and vica versa, study the solution in the original set of variables, for given initial data in $\D$.

\begin{lemma}
\label{lemma:semigroupDtoDproperties}
Given $(u_{0},R_{0},S_{0},\rho_{0},\sigma_{0},\mu_{0},\nu_{0})\in \D$, let
\begin{equation*}
	(u,R,S,\rho,\sigma,\mu,\nu)(T)=\mathbf{M}\circ S_{T}\circ\mathbf{L}(u_{0},R_{0},S_{0},\rho_{0},\sigma_{0},\mu_{0},\nu_{0})
\end{equation*}
and
\begin{equation*}
	(Z,p,q)=\mathbf{S}\circ\mathbf{C}\circ\mathbf{L}(u_{0},R_{0},S_{0},\rho_{0},\sigma_{0},\mu_{0},\nu_{0}).
\end{equation*}
Then, we have
\begin{equation}
\label{eq:lemmasemigprop1}
	u(t(X,Y),x(X,Y))=U(X,Y)
\end{equation}
for all $(X,Y)\in \mathbb{R}^{2}$,
\begin{subequations}
\begin{align}
	\label{eq:lemmasemigprop2}
	R(t(X,Y),x(X,Y))x_{X}(X,Y)&=c(U(X,Y))U_{X}(X,Y),\\
	\label{eq:lemmasemigprop3}
	\rho(t(X,Y),x(X,Y))x_{X}(X,Y)&=p(X,Y)
\end{align}
\end{subequations}
for almost every $(X,Y)\in \mathbb{R}^{2}$ such that $x_{X}(X,Y)>0$, and
\begin{subequations}
\begin{align}
	\label{eq:lemmasemigprop4}
	S(t(X,Y),x(X,Y))x_{Y}(X,Y)&=-c(U(X,Y))U_{Y}(X,Y),\\
	\label{eq:lemmasemigprop5}
	\sigma(t(X,Y),x(X,Y))x_{Y}(X,Y)&=q(X,Y)
\end{align}
\end{subequations}
for almost every $(X,Y)\in \mathbb{R}^{2}$ such that $x_{Y}(X,Y)>0$. Furthermore, we have
\begin{equation}
\label{eq:lemmasemigprop6}
	u_{t}=\frac{1}{2}(R+S) \quad \text{and} \quad u_{x}=\frac{1}{2c(u)}(R-S)
\end{equation}
in the sense of distributions.
\end{lemma}

\begin{proof}
Given $(X,Y)\in \mathbb{R}^{2}$, we denote $\bar{t}=t(X,Y)$ and $\bar{x}=x(X,Y)$. Let 
\begin{equation*}
	(\X,\Y,\Z,\V,\W,\p,\q)=\mathbf{E}\circ\mathbf{t}_{\bar{t}}(Z,p,q).
\end{equation*}
We have
$t(\X(s),\Y(s))=\bar{t}$, $\Z_{2}(s)=x(\X(s),\Y(s))$ and $\Z_{3}(s)=U(\X(s),\Y(s))$. Notice that we can write
\begin{align*}
	(u,R,S,\rho,\sigma,\mu,\nu)(\bar{t})&=\mathbf{M}\circ S_{\bar{t}}\circ\mathbf{L}(u_{0},R_{0},S_{0},\rho_{0},\sigma_{0},\mu_{0},\nu_{0})\\
	&=\mathbf{M}\circ \mathbf{D}\circ\mathbf{E}\circ\mathbf{t}_{\bar{t}}\circ\mathbf{S}\circ\mathbf{C}\circ\mathbf{L}(u_{0},R_{0},S_{0},\rho_{0},\sigma_{0},\mu_{0},\nu_{0})\\
	&=\mathbf{M}\circ \mathbf{D}\circ\mathbf{E}\circ\mathbf{t}_{\bar{t}}(Z,p,q)\\
	&=\mathbf{M}\circ \mathbf{D}(\X,\Y,\Z,\V,\W,\p,\q),
\end{align*}
so that we can apply Lemma \ref{lemma:mapG0toD}, from which we have that $u(\bar{t},\bar{x})=\Z_{3}(s)$ for any $s$ such that $\bar{x}=\Z_{2}(s)$. This implies that, for any $\bar{s}$ such that
\begin{equation}
\label{eq:tandxbar}
	t(\X(\bar{s}),\Y(\bar{s}))=\bar{t}=t(X,Y)\quad \text{and} \quad x(\X(\bar{s}),\Y(\bar{s}))=\bar{x}=x(X,Y),
\end{equation} 
we have
\begin{equation*}
	u(\bar{t},\bar{x})=U(\X(\bar{s}),\Y(\bar{s})).
\end{equation*}
Then, \eqref{eq:lemmasemigprop1} will be proved once we have proved that
\begin{equation}
\label{eq:Ubar}
	U(\X(\bar{s}),\Y(\bar{s}))=U(X,Y).
\end{equation}
We show that when \eqref{eq:tandxbar} holds, then either $(X,Y)=(\X(\bar{s}),\Y(\bar{s}))$ or
\begin{equation}
\label{eq:allvanish}
	x_{X}=x_{Y}=U_{X}=U_{Y}=p=q=0
\end{equation}
in the rectangle with corners at $(X,Y)$ and $(\X(\bar{s}),\Y(\bar{s}))$, so that \eqref{eq:Ubar} holds in both cases. We first consider the rectangle where $\X(\bar{s})\leq X$ and $\Y(\bar{s})\leq Y$. Since $x_{X}\geq0$ and $x_{Y}\geq 0$, \eqref{eq:tandxbar} implies that $x_{X}=0$ and $x_{Y}=0$ in $[\X(\bar{s}),X]\times [\Y(\bar{s}),Y]$. By \eqref{eq:setH3}, we have $U_{X}=U_{Y}=p=q=0$ in
$[\X(\bar{s}),X]\times [\Y(\bar{s}),Y]$, so that $U$ is constant and we have proved \eqref{eq:Ubar}.
In the case where $\X(\bar{s})\leq X$ and $\Y(\bar{s})\geq Y$, we find, since $t_{X}\geq 0$ and $t_{Y}\leq 0$, that $t_{X}=0$ and $t_{Y}=0$ in $[\X(\bar{s}),X]\times [Y,\Y(\bar{s})]$. By \eqref{eq:setH1}, it follows that $x_{X}=x_{Y}=0$ and we prove \eqref{eq:Ubar} as before. The other cases can be treated in the same way. Thus, \eqref{eq:Ubar} holds and we have proved \eqref{eq:lemmasemigprop1}. We prove \eqref{eq:lemmasemigprop2} and \eqref{eq:lemmasemigprop3}. By \eqref{eq:mapGtoD9}, \eqref{eq:mapGtoD10} and the definition of $\mathbf{E}$, we have
\begin{align*}
	R(\bar{t},x(\X(s),\Y(s)))x_{X}(\X(s),\Y(s))&=c(U(\X(s),\Y(s)))U_{X}(\X(s),\Y(s)),\\
	\rho(\bar{t},x(\X(s),\Y(s)))x_{X}(\X(s),\Y(s))&=p(\X(s),\Y(s)),
\end{align*}
so that
\begin{align*}
	R(t(X,Y),x(X,Y))x_{X}(\X(\bar{s}),\Y(\bar{s}))&=c(U(\X(\bar{s}),\Y(\bar{s})))U_{X}(\X(\bar{s}),\Y(\bar{s})),\\
	\rho(t(X,Y),x(X,Y))x_{X}(\X(\bar{s}),\Y(\bar{s}))&=p(\X(\bar{s}),\Y(\bar{s}))
\end{align*}
for any $\bar{s}$ such that \eqref{eq:tandxbar} holds. This implies \eqref{eq:lemmasemigprop2} and \eqref{eq:lemmasemigprop3} because when \eqref{eq:tandxbar} is satisfied, then either $(X,Y)=(\X(\bar{s}),\Y(\bar{s}))$ or \eqref{eq:allvanish} holds. Similarly, from \eqref{eq:mapGtoD11} and \eqref{eq:mapGtoD12}, we obtain
\begin{align*}
	S(t(X,Y),x(X,Y))x_{Y}(\X(\bar{s}),\Y(\bar{s}))&=-c(U(\X(\bar{s}),\Y(\bar{s})))U_{Y}(\X(\bar{s}),\Y(\bar{s})),\\
	\sigma(t(X,Y),x(X,Y))x_{Y}(\X(\bar{s}),\Y(\bar{s}))&=q(\X(\bar{s}),\Y(\bar{s}))
\end{align*}
for any $\bar{s}$ such that \eqref{eq:tandxbar} holds, so that \eqref{eq:lemmasemigprop4} and \eqref{eq:lemmasemigprop5} follows. Now we prove \eqref{eq:lemmasemigprop6}. Let $\phi(t,x)$ be a smooth test function with compact support. We have
\begin{align*}
	&\iint_{\mathbb{R}^{2}}(u\phi_{t})(t,x)\,dt\,dx\\
	&=\iint_{\mathbb{R}^{2}}\big((u\phi_{t})\circ(t,x)(t_{X}x_{Y}-t_{Y}x_{X})\big)(X,Y)\,dX\,dY \quad \text{by a change of variables}\\
	&=\iint_{\mathbb{R}^{2}}\big(U\phi_{t}\circ(t,x)(t_{X}x_{Y}-t_{Y}x_{X})\big)(X,Y)\,dX\,dY \quad \text{by } \eqref{eq:lemmasemigprop1}\\
	&=\iint_{\mathbb{R}^{2}}\big(U(\phi_{X}\circ(t,x)x_{Y}-\phi_{Y}\circ(t,x)x_{X})\big)(X,Y)\,dX\,dY	\quad \text{by calculating } \phi_{X} \text{ and } \phi_{Y}\\
	&=-\iint_{\mathbb{R}^{2}}\big(((Ux_{Y})_{X}-(Ux_{X})_{Y})\phi\circ(t,x)\big)(X,Y)\,dX\,dY \quad \text{by integrating by parts}\\
	&=-\iint_{\mathbb{R}^{2}}\big((U_{X}x_{Y}-U_{Y}x_{X})\phi\circ(t,x)\big)(X,Y)\,dX\,dY\\
	&=-\iint_{\mathbb{R}^{2}}\bigg(\bigg(\frac{R+S}{c(u)}\phi\bigg)\circ(t,x)x_{X}x_{Y}\bigg)(X,Y)\,dX\,dY \quad \text{by } \eqref{eq:lemmasemigprop2} \text{ and } \eqref{eq:lemmasemigprop4}\\
	&=-\iint_{\mathbb{R}^{2}}\bigg(\bigg(\frac{1}{2}(R+S)\phi\bigg)\circ(t,x)(t_{X}x_{Y}-t_{Y}x_{X})\bigg)(X,Y)\,dX\,dY \quad \text{by } \eqref{eq:setH1}\\
	&=-\iint_{\mathbb{R}^{2}}\bigg(\frac{1}{2}(R+S)\phi\bigg)(t,x)\,dt\,dx,
\end{align*}
which proves the first identity in \eqref{eq:lemmasemigprop6}. The second one is proven in the same way.
\end{proof}

\subsection{Semigroup of Solutions in $\D$}
\label{sec:SemigroupD}

Now we can define a mapping on the original set of variables, $\D$.

\begin{definition}
\label{def:modsemigroup}
For any $T>0$, let $\bar{S}_{T}:\D\rightarrow \D$ be defined as
\begin{equation*}
	\bar{S}_{T}=\mathbf{M}\circ S_{T}\circ\mathbf{L}.
\end{equation*}	
\end{definition}

Since
\begin{equation*}
	\bar{S}_{T}\circ\bar{S}_{T'}=\mathbf{M}\circ S_{T}\circ\mathbf{L}\circ\mathbf{M}\circ S_{T'}\circ\mathbf{L}
\end{equation*}
it would immediately follow from the semigroup property of $S_{T}$ that $\bar{S}_{T}$ is also a semigroup if we had $\mathbf{L}\circ\mathbf{M}=\id$, but this identity does not hold in general. To see this consider an element $\psi$ in $\F$. By Definition \ref{def:setF} we have $x_{1}+J_{1}\in G$, which in particular means that $x_{1}+J_{1}-\id\in\Linf(\mathbb{R})$. Let $\xi=\mathbf{M}(\psi)$ and $\bar{\psi}=\mathbf{L}(\xi)$. From Definition \ref{def:mapfromDtoF} we have $\bar{x}_{1}+\bar {J}_{1}=\id$, and it is clear that in general we have $\psi\neq\bar{\psi}$. 

It is the aim of this section to prove that $\bar{S}_{T}$ is a semigroup. The idea is loosely speaking the following. Assume that $x_{1}(X)+J_{1}(X)=f(X)$ where $f\in G$. We associate to $x_{1}$ and $J_{1}$ the functions $\bar{x}_{1}$ and $\bar{J}_{1}$ such that $\bar{x}_{1}+\bar{J}_{1}=\id$. We observe that the transformation $\bar{x}_{1}(X)=x_{1}(f^{-1}(X))$ and $\bar{J}_{1}(X)=J_{1}(f^{-1}(X))$ is such a mapping. The identities \eqref{eq:setFrel2} and \eqref{eq:setFrel3} allow us to define the remaining elements of $\bar{\psi}$, see Definition \ref{def:Faction} below. The transformation of $\psi$ to $\bar{\psi}$ defines an action of $G^{2}$ on the set $\F$. It defines a projection $\Pi$ from $\F$ on the set
\begin{equation*}
	\F_{0}=\{\psi=(\psi_{1},\psi_{2})\in\F \ | \ x_{1}+J_{1}=\id \text{ and } x_{2}+J_{2}=\id\}.
\end{equation*}
Thus, we have $\bar{\psi}=\Pi(\psi)$.

We prove that the map $S_{T}$ is invariant under the group acting on $\F$. Thus, we have to define the action of $G^{2}$ on the sets $\C$, $\G$ and $\H$. The definition of the action on $\F$ will naturally lead to the definition of the action on the set of curves $\C$ and the set $\G$. We define the action of $G^{2}$ on the set of solutions $\H$ so that it commutes with the $\bullet$ operation, see Lemma \ref{lemma:GHCactioncommute} below.

Then we prove that the map $\mathbf{M}$ satisfies $\mathbf{M}=\mathbf{M}\circ\Pi$, and that $\F_{0}$ contains exactly one element of each equivalence class of $\F$ with respect to $G^{2}$.      
    
\begin{definition}
	\label{def:Faction}
	For any $\psi=(\psi_{1},\psi_{2})\in \F$ and $f,g\in G$, we define $\bar{\psi}=(\bar{\psi}_{1},\bar{\psi}_{2})$ as
	\begin{subequations}
		\begin{align}
		\label{eq:Faction1}
		\bar{x}_{1}(X)&=x_{1}(f(X)), & \bar{x}_{2}(Y)&=x_{2}(g(Y)),\\
		\label{eq:Faction2}
		\bar{U}_{1}(X)&=U_{1}(f(X)), & \bar{U}_{2}(Y)&=U_{2}(g(Y)),\\
		\label{eq:Faction3}
		\bar{J}_{1}(X)&=J_{1}(f(X)), & \bar{J}_{2}(Y)&=J_{2}(g(Y)),\\
		\label{eq:Faction4}
		\bar{K}_{1}(X)&=K_{1}(f(X)), & \bar{K}_{2}(Y)&=K_{2}(g(Y)),\\
		\label{eq:Faction5}
		\bar{V}_{1}(X)&=f'(X)V_{1}(f(X)), & \bar{V}_{2}(Y)&=g'(Y)V_{2}(g(Y)),\\
		\label{eq:Faction6}
		\bar{H}_{1}(X)&=f'(X)H_{1}(f(X)), & \bar{H}_{2}(Y)&=g'(Y)H_{2}(g(Y)).
		\end{align}
	\end{subequations}
	The mapping $\F\times G^{2}\rightarrow\F$ given by $\psi\times (f,g)\mapsto \bar{\psi}$ defines an action of the group $G^{2}$ on $\F$ and we denote $\bar{\psi}=\psi\cdot (f,g)$.
\end{definition}

\begin{proof}[Proof of the well-posedness of Definition \ref{def:Faction}] We prove that $\bar{\psi}=(\bar{\psi}_{1},\bar{\psi}_{2})$ belongs to $\F$. We only show that $\bar{\psi}_{1}$ satisfies the conditions in Definition \ref{def:setF}. The proof is similar for $\bar{\psi}_{2}$. First we show that $\bar{\psi}_{1}$ satisfies the regularity conditions in \eqref{eqns:setF}. We will use the following result throughout the proof. By Lemma \ref{lemma:auxiliaryG}, there exists $\alpha\geq0$ such that $\frac{1}{1+\alpha}\leq f'(X)\leq 1+\alpha$ for almost every $X\in \mathbb{R}$ and $\frac{1}{1+\alpha}\leq g'(Y)\leq 1+\alpha$ for almost every $Y\in \mathbb{R}$. 
	
	Since $x_{1}-\id,f-\id\in\Linf(\mathbb{R})$,
	\begin{equation*}
	\bar{x}_{1}(X)-X=(x_{1}(f(X))-f(X))+(f(X)-X)
	\end{equation*}	
	and $\bar x_1-\id$ belongs to $\Linf(\mathbb{R})$. We differentiate and obtain
	\begin{equation*}
	\bar{x}_{1}'(X)-1=(x_{1}'(f(X))-1)f'(X)+(f'(X)-1).
	\end{equation*}	
	Since $x_{1}'-1,f'-1\in\Linf(\mathbb{R})$ and $\frac{1}{1+\alpha}\leq f'\leq 1+\alpha$, $\bar{x}_{1}'-1$ belongs to $\Linf(\mathbb{R})$. Moreover, by a straightforward calculation, we get
	\begin{align*}
	\int_{\mathbb{R}}(\bar{x}_{1}'(X)-1)^{2}\,dX&\leq 2\int_{\mathbb{R}}(x_{1}'(f(X))-1)^{2}f'(X)^{2}\,dX+2\int_{\mathbb{R}}(f'(X)-1)^{2}\,dX\\
	&\leq 2(1+\alpha)\int_{\mathbb{R}}(x_{1}'(f(X))-1)^{2}f'(X)\,dX+2\int_{\mathbb{R}}(f'(X)-1)^{2}\,dX\\
	&=2(1+\alpha)\int_{\mathbb{R}}(x_{1}'(X)-1)^{2}\,dX+2\int_{\mathbb{R}}(f'(X)-1)^{2}\,dX,
	\end{align*}
	where we used a change of variables in the last equality. This shows that $\bar{x}_{1}'-1\in L^{2}(\mathbb{R})$ because $x_{1}'-1,f'-1\in L^{2}(\mathbb{R})$. Since $J_{1}\in\Linf(\mathbb{R})$ it follows immediately from the definition of $\bar{J}_{1}$ that it also belongs to $\Linf(\mathbb{R})$.
	We have $\bar{J}_{1}'(X)=J_{1}'(f(X))f'(X)$, so that $\bar{J}_{1}'\in\Linf(\mathbb{R})$ because $J_{1}'\in\Linf(\mathbb{R})$ and $\frac{1}{1+\alpha}\leq f'\leq 1+\alpha$. The function $\bar{J}_{1}'$ also belongs to $L^{2}(\mathbb{R})$ since
	\begin{aalign}
		\label{eq:FactionJ1derL2norm}
		\int_{\mathbb{R}}\bar{J}_{1}'(X)^{2}\,dX&=\int_{\mathbb{R}}J_{1}'(f(X))^{2}f'(X)^{2}\,dX\\
		&\leq (1+\alpha)\int_{\mathbb{R}}J_{1}'(f(X))^{2}f'(X)\,dX\\
		&=(1+\alpha)\int_{\mathbb{R}}J_{1}'(X)^{2}\,dX \quad \text{by a change of variables}
	\end{aalign}
	and $J_{1}'\in L^{2}(\mathbb{R})$. In a similar way, one shows that $\bar{K}_{1}\in\Linf(\mathbb{R})$ and $\bar{K}_{1}',\bar{H}_{1},\bar{V}_{1}\in L^{2}(\mathbb{R})\cap\Linf(\mathbb{R})$. We have
	\begin{align*}
	\int_{\mathbb{R}}\bar{U}_{1}(X)^{2}\,dX&=\int_{\mathbb{R}}U_{1}(f(X))^{2}\frac{f'(X)}{f'(X)}\,dX\\
	&\leq (1+\alpha)\int_{\mathbb{R}}U_{1}(f(X))^{2}f'(X)\,dX\\
	&=(1+\alpha)\int_{\mathbb{R}}U_{1}(X)^{2}\,dX \quad \text{by a change of variables},
	\end{align*}
	so that $\bar{U}_{1}\in L^{2}(\mathbb{R})\cap L^\infty(\mathbb{R})$ as $U_{1}\in L^{2}(\mathbb{R})\cap L^\infty(\mathbb{R})$. Hence, we have proved \eqref{eqns:setF}. By differentiating $\bar{x}_{1}$ and $\bar{J}_{1}$, we find that the inequalities in \eqref{eq:setFrel1} are satisfied since $x_{1}', J_1'\geq 0$ and $f'\geq\frac{1}{1+\alpha}>0$. 
	
	The relations in \eqref{eq:setFrel2}-\eqref{eq:setFrel3} follow by direct calculation, for example, we have  
	\begin{align*}
	\bar{x}_{1}'(X)\bar{J}_{1}'(X)&=f'(X)^{2}x_{1}'(f(X))J_{1}'(f(X))\\
	&=f'(X)^{2}\big[(c(U_{1}(f(X)))V_{1}(f(X)))^{2}+c(U_{1}(f(X)))H_{1}^{2}(f(X))\big]\\
	&=(c(\bar{U}_{1}(X))\bar{V}_{1}(X))^{2}+c(\bar{U}_{1}(X))\bar{H}_{1}^{2}(X).
	\end{align*}
	
	Let us prove that $\bar{x}_{1}+\bar{J}_{1}\in G$. We proved above that $\bar{x}_{1}-\id,\bar{J}_{1}\in\Linf(\mathbb{R})$ and $\bar{x}_{1}'-1,\bar{J}_{1}'\in L^{2}(\mathbb{R})\cap\Linf(\mathbb{R})$, which implies that $\bar{x}_{1}+\bar{J}_{1}-\id\in\Winf(\mathbb{R})$ and $\bar{x}_{1}'+\bar{J}_{1}'-1 \in L^{2}(\mathbb{R})$. Since $x_{1}+J_{1}\in G$ we get by Lemma \ref{lemma:auxiliaryG} that there exists $\alpha_{1}\geq 0$ such that $\frac{1}{1+\alpha_{1}}\leq x_{1}'+J_{1}'\leq 1+\alpha_{1}$ and $\bar{x}_{1}+\bar{J}_{1}$ is invertible because
	\begin{equation*}
	\bar{x}_{1}'(X)+\bar{J}_{1}'(X)=f'(X)(x_{1}'(f(X))+J_{1}'(f(X)))\geq \frac{1}{(1+\alpha)(1+\alpha_{1})}>0.
	\end{equation*}
	We prove that $(\bar{x}_{1}+\bar{J}_{1})^{-1}-\id$ belongs to $\Winf(\mathbb{R})$. Since
	\begin{align*}
	(\bar{x}_{1}+\bar{J}_{1})^{-1}-\id&=f^{-1}\circ(x_{1}+J_{1})^{-1}-\id\\
	&=f^{-1}\circ(x_{1}+J_{1})^{-1}-(x_{1}+J_{1})^{-1}+(x_{1}+J_{1})^{-1}-\id
	\end{align*}
	and $f^{-1}-\id,(x_{1}+J_{1})^{-1}-\id\in\Linf(\mathbb{R})$, $(\bar{x}_{1}+\bar{J}_{1})^{-1}-\id$ belongs to $\Linf(\mathbb{R})$. Let $v=\bar{x}_{1}+\bar{J}_{1}$. We have
	\begin{equation*}
	(v^{-1})'=\frac{1}{v'(v^{-1})}=\frac{1}{f'(v^{-1})(x_{1}'(f(v^{-1}))+J_{1}'(f(v^{-1})))},
	\end{equation*}
	so that
	\begin{equation*}
	\frac{1}{(1+\alpha)(1+\alpha_{1})}-1\leq (v^{-1})'-1\leq (1+\alpha)(1+\alpha_{1})-1
	\end{equation*}
	and $((\bar{x}_{1}+\bar{J}_{1})^{-1})'-1\in\Linf(\mathbb{R})$. Hence, we have proved \eqref{eq:setFrel4}. 
	
	We prove \eqref{eq:setFrel5}. Since $f-\id\in\Linf(\mathbb{R})$, we have $\displaystyle\lim_{X\rightarrow-\infty}f(X)=-\infty$. This implies, by \eqref{eq:setFrel5}, that 
	\begin{equation*}
	\displaystyle\lim_{X\rightarrow -\infty}\bar{J}_{1}(X)=\displaystyle\lim_{X\rightarrow -\infty}J_{1}(f(X))=0.
	\end{equation*}
	In order to prove the identities \eqref{eq:setFrel6} and \eqref{eq:setFrel7} we have to define the action of $G^{2}$ on $\C$. 

\begin{definition}
	\label{def:Caction}
	For any $(\X,\Y)\in \C$ and $f,g\in G$, we define $(\bar{\X},\bar{\Y})$ as
	\begin{equation}
	\label{eq:Caction}
	\bar{\X}=f^{-1}\circ\X\circ h \quad \text{and} \quad \bar{\Y}=g^{-1}\circ\Y\circ h,
	\end{equation}
	where $h\in G$ is the re-normalizing function which yields $\bar{\X}+\bar{\Y}=2\id$, that is,
	\begin{equation}
	\label{eq:Caction2}
	(f^{-1}\circ\X+g^{-1}\circ\Y)\circ h=2\id.
	\end{equation}
	The mapping $\C\times G^{2}\rightarrow\C$ given by $(\X,\Y)\times (f,g)\mapsto (\bar{\X},\bar{\Y})$ defines an action of the group $G^{2}$ on $\C$ and we denote $(\bar{\X},\bar{\Y})=(\X,\Y)\cdot (f,g)$.
\end{definition}

The action corresponds to a stretching of the curve in the $X$ and $Y$ directions.  

\begin{proof}[Proof of the well-posedness of Definition \ref{def:Caction}]	
	Let $v=\frac{1}{2}(f^{-1}\circ\X+g^{-1}\circ\Y)$. We want to prove that $v$ belongs to $G$ by using Lemma \ref{lemma:auxiliaryG}. We have
	\begin{align*}
	v-\id&=\frac{1}{2}(f^{-1}\circ\X-\X+\X+g^{-1}\circ\Y)-\id\\
	&=\frac{1}{2}(f^{-1}\circ\X-\X+2\id-\Y+g^{-1}\circ\Y)-\id\\
	&=\frac{1}{2}(f^{-1}\circ\X-\X+g^{-1}\circ\Y-\Y)
	\end{align*}
	which belongs to $\Linf(\mathbb{R})$ because $f-\id,g-\id\in \Linf(\mathbb{R})$. Since $f$ and $g$ belong to $G$, we have, by Lemma \ref{lemma:auxiliaryG}, that there exists $\alpha\geq0$ such that $\frac{1}{1+\alpha}\leq
	(f^{-1})'\leq 1+\alpha$ and $\frac{1}{1+\alpha}\leq
	(g^{-1})'\leq 1+\alpha$ almost everywhere. Then,
	\begin{equation*}
	v'=\frac{1}{2}(((f^{-1})'\circ\X)\dot{\X}+((g^{-1})'\circ\Y)\dot{\Y})\leq \frac{1}{2}(1+\alpha)(\dot{\X}+\dot{\Y})=1+\alpha
	\end{equation*}
	and, similarly, we obtain that 
	\begin{equation}
	\label{eq:Cgroupvder}
	v'\geq\frac{1}{1+\alpha}.
	\end{equation}
	We show that $v$ is absolutely continuous. Let $(s_{i},\bar{s}_{i})$, $i=1,\dots,N$, be non-intersecting intervals.
	We have
	\begin{align*}
	\sum_{i=1}^{N}|v(\bar{s}_{i})-v(s_{i})|&=\frac{1}{2}\sum_{i=1}^{N}|f^{-1}\circ\X(\bar{s}_{i})+g^{-1}\circ\Y(\bar{s}_{i})-f^{-1}\circ\X(s_{i})-g^{-1}\circ\Y(s_{i})|\\
	&\leq\frac{1}{2}\sum_{i=1}^{N}\bigg(\bigg|\int_{\X(s_{i})}^{\X(\bar{s}_{i})}(f^{-1})'(X)\,dX\bigg|+\bigg|\int_{\Y(s_{i})}^{\Y(\bar{s}_{i})}(g^{-1})'(Y)\,dY\bigg|\bigg)\\
	&\leq\frac{1}{2}(1+\alpha)\sum_{i=1}^{N}(|\X(\bar{s}_{i})-\X(s_{i})|+|\Y(\bar{s}_{i})-\Y(s_{i})|) \\
	&\leq 2(1+\alpha)\sum_{i=1}^{N}|\bar{s}_{i}-s_{i}|, 
	\end{align*}
	where we used that $\dot{\X}$ and $\dot{\Y}$ are bounded by 2. Hence, $v$ is absolutely continuous.
	
	Let us prove that if $w\in G$, then $(w^{-1})'-1$ belongs to $L^{2}(\mathbb{R})$. Since $w\circ w^{-1}=\id$, we have $(w'\circ w^{-1})(w^{-1})'=1$, so that, 
	by Lemma \ref{lemma:auxiliaryG}, $(w^{-1})'\leq 1+\beta$ for some $\beta>0$. This implies that
	\begin{aalign}
		\label{eq:invOfwInG}
		&\int_{\mathbb{R}}((w^{-1})'(x)-1)^{2}\,dx\\
		&=\int_{\mathbb{R}}((w^{-1})'(x)-(w'\circ (w^{-1})(x))(w^{-1})'(x))^{2}\,dx\\
		&=\int_{\mathbb{R}}(1-w'\circ (w^{-1})(x))^{2}(w^{-1})'(x)^{2}\,dx\\
		&\leq (1+\beta)\int_{\mathbb{R}}(1-w'\circ (w^{-1})(x))^{2}(w^{-1})'(x)\,dx\\
		&\leq (1+\beta)\int_{\mathbb{R}}(1-w'(x))^{2}\,dx \quad \text{by a change of variables},
	\end{aalign}     
	which is bounded because $w'-1\in L^{2}(\mathbb{R})$, and we conclude that $(w^{-1})'-1\in L^{2}(\mathbb{R})$.
	
	Since $\dot{\X}+\dot{\Y}=2$, we have
	\begin{align*}
	v'-1&=\frac{1}{2}(((f^{-1})'\circ\X)\dot{\X}+((g^{-1})'\circ\Y)\dot{\Y}-2)\\
	&=\frac{1}{2}(((f^{-1})'\circ\X)\dot{\X}-\dot{\X}+((g^{-1})'\circ\Y)\dot{\Y}-\dot{\Y}).
	\end{align*} 
	Hence,
	\begin{align*}
	&\int_{\mathbb{R}}(v'(s)-1)^{2}\,ds\\
	&=\frac{1}{4}\int_{\mathbb{R}}(((f^{-1})'\circ\X(s))\dot{\X}(s)-\dot{\X}(s)+((g^{-1})'\circ\Y(s))\dot{\Y}(s)-\dot{\Y}(s))^{2}\,ds\\
	&\leq\frac{1}{2}\int_{\mathbb{R}}(((f^{-1})'\circ\X(s))\dot{\X}(s)-\dot{\X}(s))^{2}\,ds\\
	&\quad+\frac{1}{2}\int_{\mathbb{R}}(((g^{-1})'\circ\Y(s))\dot{\Y}(s)-\dot{\Y}(s))^{2}\,ds\\
	&\leq\int_{\mathbb{R}}((f^{-1})'\circ\X(s)-1)^{2}\dot{\X}(s)\,ds\\
	&\quad+\int_{\mathbb{R}}((g^{-1})'\circ\Y(s)-1)^{2}\dot{\Y}(s)\,ds \quad \text{since } \dot{\X}\leq 2 \text{ and } \dot{\Y}\leq 2\\
	&=\int_{\mathbb{R}}((f^{-1})'(X)-1)^{2}\,dX+\int_{\mathbb{R}}((g^{-1})'(Y)-1)^{2}\,dY \quad \text{by a change of variables}
	\end{align*}
	and by \eqref{eq:invOfwInG}, we conclude that $v'-1\in L^{2}(\mathbb{R})$. Then, by Lemma \ref{lemma:auxiliaryG}, $v\in G$.
	
	We define $h=v^{-1}$. Since $v\circ h=\id$, \eqref{eq:Caction2} holds. We claim that $h\in G$. Since $v$ belongs to $G$, we have that $h-\id$ and $h^{-1}-\id$ belong to $\Winf(\mathbb{R})$. From \eqref{eq:invOfwInG}, we get that $h'-1\in L^{2}(\mathbb{R})$. Therefore, $h\in G$.
	
	Now we prove that $(\bar{\X},\bar{\Y})\in\C$. We have $\bar{\X}-\id\in\Winf(\mathbb{R})$ since
	\begin{equation*}
	\bar{\X}-\id=f^{-1}\circ\X\circ h-\X\circ h+\X\circ h-h+h-\id
	\end{equation*}
	and $f^{-1}-\id,\X-\id,h-\id\in\Winf(\mathbb{R})$. Since $\dot{\X}\geq 0$, we get
	\begin{equation*}
	\dot{\bar{\X}}=((f^{-1})'\circ\X\circ h)(\dot{\X}\circ h)h'\geq \frac{1}{(1+\alpha)^{2}}\dot{\X}\circ h\geq 0.
	\end{equation*}
	Similarly, one shows that $\bar{\Y}-\id\in\Winf(\mathbb{R})$ and $\dot{\bar{\Y}}\geq 0$. The identity \eqref{eq:initialcurvenormalization} is satisfied since $v\circ h=\id$, which we proved above. Hence, $(\bar{\X},\bar{\Y})\in\C$.
\end{proof}

\emph{End of proof of the well-posedness of Definition \ref{def:Faction}.}

It remains to prove \eqref{eq:setFrel6} and \eqref{eq:setFrel7}. Let $(\bar{\X},\bar{\Y})\in \C$. Then $(\X,\Y)=(\bar \X, \bar \Y)\cdot(f^{-1}, g^{-1})$ belongs to $\C$. In particular,  $x_{1}(\X(s))=x_{2}(\Y(s))$ for all $s\in\mathbb{R}$ and the identities \eqref{eq:setFrel6} and \eqref{eq:setFrel7} hold for the elements corresponding to $(\psi_{1},\psi_{2})$. Furthermore, $(\bar \X, \bar \Y)=(\X,\Y)\cdot (f,g)$, which implies, using the same notation as in Definition~\ref{def:Caction}, that
\begin{equation*}
	\bar{x}_{1}(\bar{\X}(s))=x_{1}(\X(h(s)))=x_{2}(\Y(h(s)))=\bar{x}_{2}(\bar{\Y}(s)).
\end{equation*}
We find
\begin{equation*}
	\bar{U}_{1}(\bar{\X}(s))=U_{1}(\X(h(s)))=U_{2}(\Y(h(s)))=\bar{U}_{2}(\bar{\Y}(s)),
\end{equation*}
which proves \eqref{eq:setFrel6}. Since
\begin{equation*}
	\dot{\bar{\X}}(s)=\frac{\dot{\X}(h(s))h'(s)}{f'(\bar{\X}(s))} \quad \text{and} \quad \dot{\bar{\Y}}(s)=\frac{\dot{\Y}(h(s))h'(s)}{g'(\bar{\Y}(s))}
\end{equation*}
we have
\begin{equation*}
	\bar{V}_{1}(\bar{\X}(s))\dot{\bar{\X}}(s)+\bar{V}_{2}(\bar{\Y}(s))\dot{\bar{\Y}}(s)=h'(s)\big[V_{1}(\X(h(s)))\dot{\X}(h(s))+V_{2}(\Y(h(s)))\dot{\Y}(h(s))
	\big].
\end{equation*}
and we obtain
\begin{align*}
	\bar{V}_{1}(\bar{\X}(s))\dot{\bar{\X}}(s)+\bar{V}_{2}(\bar{\Y}(s))\dot{\bar{\Y}}(s)&=\frac{d}{ds}U_{1}(\X(h(s)))=\frac{d}{ds}\bar{U}_{1}(\bar{\X}(s))\\
	&=\frac{d}{ds}U_{2}(\Y(h(s)))=\frac{d}{ds}\bar{U}_{2}(\bar{\Y}(s)).
\end{align*}
This proves \eqref{eq:setFrel7}.
\end{proof}

\begin{definition}
	\label{def:Gaction}
	For any $\Theta=(\X,\Y,\Z,\V,\W,\p,\q)\in \G$ and $f,g\in G$, we define $\bar{\Theta}=(\bar{\X},\bar{\Y},\bar{\Z},\bar{\V},\bar{\W},\bar{\p},\bar{\q})$ as
	\begin{subequations}
		\begin{align}
		(\bar{\X},\bar{\Y})&=(\X,\Y)\cdot (f,g),\\
		\label{eq:Gaction2}
		\bar{\Z}&=\Z\circ h,
		\end{align}
		where $h$ is given by \eqref{eq:Caction2}, and
		\begin{align}
		\label{eq:Gaction3}
		\bar{\V}(X)&=f'(X)\V(f(X)), & \bar{\W}(Y)&=g'(Y)\W(g(Y)),\\
		\label{eq:Gaction4}
		\bar{\p}(X)&=f'(X)\p(f(X)), & \bar{\q}(Y)&=g'(Y)\q(g(Y)).
		\end{align}
	\end{subequations}
	The mapping $\G\times G^{2}\rightarrow\G$ given by $\Theta\times (f,g)\mapsto \bar{\Theta}$ defines an action of the group $G^{2}$ on $\G$ and we denote $\bar{\Theta}=\Theta\cdot (f,g)$.
\end{definition}

\begin{proof}[Proof of the well-posedness of Definition \ref{def:Gaction}]
	We have to prove that $\bar{\Theta}$ belongs to $\G$. In the proof of the well-posedness of Definition \ref{def:Caction}, we showed that $(\bar{\X},\bar{\Y})\in \C$. We check that $||\bar{\Theta}||_{\G}$ and $|||\bar{\Theta}|||_{\G}$ are finite. Since $f,h\in G$, we have from \eqref{eq:groupcond} that $f-\id,f^{-1}-\id,h-\id\in\Winf(\mathbb{R})$ and therefore, by \eqref{eq:atriplet}, we conclude that the quantities
	\begin{align*}
	\bar{\Z}_{1}^{a}&=\bar{\Z}_{1}-\frac{1}{c(0)}(\bar{\X}-\id)\\
	&=\Z_{1}\circ h-\frac{1}{c(0)}(f^{-1}\circ\X\circ h-\id)\\
	&=\Z_{1}\circ h-\frac{1}{c(0)}(f^{-1}\circ\X\circ h-\X\circ h+\X\circ h-h+h-\id)\\
	&=\Z_{1}^{a}\circ h-\frac{1}{c(0)}(f^{-1}\circ\X\circ h-\X\circ h+h-\id),
	\end{align*} 
	\begin{equation*}
	\bar{\Z}_{2}^{a}=\bar{\Z}_{2}-\id=\Z_{2}\circ h-\id=\Z_{2}\circ h-h+h-\id=\Z_{2}^{a}\circ h+h-\id
	\end{equation*}
	and
	\begin{equation*}
	\bar{\Z}_{i}^{a}=\bar{\Z}_{i}=\Z_{i}\circ h=\Z_{i}^{a}\circ h \quad \text{for } i\in \{3,4,5\}
	\end{equation*}
	belong to $\Linf(\mathbb{R})$. 
	
	By Lemma \ref{lemma:auxiliaryG}, there exists $\delta>0$ such that 
	\begin{equation}
	\label{eq:Gactionfandgder}
	f'(X)\geq\delta, \quad g'(X)\geq\delta,  \quad \text{ and }\quad h'(X)\geq \delta,
	\end{equation}
	for almost every $X\in \mathbb{R}$. This yields
	\begin{equation*}
	\int_{\mathbb{R}}\bar{\Z}_{3}(s)^{2}\,ds=\int_{\mathbb{R}}(\Z_{3}\circ h(s))^{2}\,ds\leq \frac{1}{\delta}\int_{\mathbb{R}}(\Z_{3}\circ h(s))^{2}h'(s)\,ds=\frac{1}{\delta}\int_{\mathbb{R}}\Z_{3}(s)^{2}\,ds
	\end{equation*}
	by a change of variables, and we conclude that $\bar{\Z}_{3}\in L^{2}(\mathbb{R})$. 
	Furthermore, we have
	\begin{align*}
	\bar{\V}_{1}^{a}&=\bar{\V}_{1}-\frac{1}{2c(0)}\\
	&=f'\V_{1}\circ f-\frac{1}{2c(0)}\\
	&=f'\bigg(\V_{1}\circ f-\frac{1}{2c(0)}\bigg)+\frac{1}{2c(0)}(f'-1)\\
	&=f'\V_{1}^{a}\circ f+\frac{1}{2c(0)}(f'-1),
	\end{align*}
	\begin{equation*}
	\bar{\V}_{2}^{a}=\bar{\V}_{2}-\frac{1}{2}=f'\V_{2}\circ f-\frac{1}{2}=f'\bigg(\V_{2}\circ f-\frac{1}{2}\bigg)+\frac{1}{2}(f'-1)=f'\V_{2}^{a}\circ f+\frac{1}{2}(f'-1),
	\end{equation*}
	\begin{equation*}
	\bar{\V}_{i}^{a}=\bar{\V}_{i}=f'\V_{i}\circ f=f'\V_{i}^{a}\circ f \quad \text{for } i\in \{3,4,5\},
	\end{equation*}
	\begin{equation*}	
	\bar{\p}=f'\p\circ f
	\end{equation*}
	and
	\begin{equation*}	
	\bar{\q}=g'\p\circ g
	\end{equation*}
	which implies, by \eqref{eq:groupcond}, \eqref{eq:groupcond2} and \eqref{eq:Gactionfandgder}, that all the components of $\bar{\V}^{a}$ and $\bar{\p}$ and $\bar{\q}$ belong to $L^{2}(\mathbb{R})\cap\Linf(\mathbb{R})$. Similarly, one shows that the components of $\bar{\W}^{a}$ belong to the same set. From \eqref{eq:Gaction3} and \eqref{eq:Gactionfandgder}, we see that the inequalities in \eqref{eq:setGpositive} are satisfied for $\bar{\Theta}$. Since 
	\begin{equation*}
	\frac{1}{\bar{\V}_{2}+\bar{\V}_{4}}=\frac{1}{f'(\V_{2}\circ f+\V_{4}\circ f)}\leq \frac{1}{\delta(\V_{2}\circ f+\V_{4}\circ f)}
	\end{equation*}
	and
	\begin{equation*}
	\frac{1}{\bar{\W}_{2}+\bar{\W}_{4}}=\frac{1}{g'(\W_{2}\circ g+\W_{4}\circ g)}\leq \frac{1}{\delta(\W_{2}\circ g+\W_{4}\circ g)},
	\end{equation*}
	$\frac{1}{\bar{\V}_{2}+\bar{\V}_{4}}$ and $\frac{1}{\bar{\W}_{2}+\bar{\W}_{4}}$ belong to $\Linf(\mathbb{R})$. Hence, $||\bar{\Theta}||_{\G}$ and $|||\bar{\Theta}|||_{\G}$ are finite. By \eqref{eq:Caction}, we obtain
	\begin{align*}
	\dot{\bar{\Z}}(s)&=\dot{\Z}(h(s))\dot{h}(s)\\
	&=\big[\V(\X(h(s)))\dot{\X}(h(s))+\W(\Y(h(s)))\dot{\Y}(h(s))\big]\dot{h}(s)\\
	&=\V(f(\bar{\X}(s)))f'(\bar{\X}(s))\dot{\bar{\X}}(s)+\W(g(\bar{\Y}(s)))g'(\bar{\Y}(s))\dot{\bar{\Y}}(s)\\
	&=\bar{\V}(\bar{\X}(s))\dot{\bar{\X}}(s)+\bar{\W}(\bar{\Y}(s))\dot{\bar{\Y}}(s)
	\end{align*}
	and we have proved that $\Theta$ satisfies \eqref{eq:setGcomp}. The relations \eqref{eq:setGrel1}-\eqref{eq:setGrel4} follows by direct computation, for instance, we have
	\begin{align*}
	2\bar{\V}_{4}(\bar{\X})\bar{\V}_{2}(\bar{\X})
	&=2f'(\bar{\X})^{2}\V_{4}(f(\bar{\X}))\V_{2}(f(\bar{\X}))\\
	&=2f'(\bar{\X})^{2}\V_{4}(\X(h))\V_{2}(\X(h))\\
	&=f'(\bar{\X})^{2}\big[(c(\Z_{3}(h))\V_{3}(\X(h)))^{2}+c(\Z_{3}(h))\p(\X(h))^{2}\big]\\
	&=f'(\bar{\X})^{2}\big[(c(\bar{\Z}_{3})\V_{3}(f(\bar{\X})))^{2}+c(\bar{\Z}_{3})\p(f(\bar{\X}))^{2}\big]\\
	&=(c(\bar{\Z}_{3})\bar{\V}_{3}(\bar{\X}))^{2}+c(\bar{\Z}_{3})\bar{\p}(\bar{\X})^{2}.
	\end{align*}
	It remains to prove that \eqref{eq:setGrel5} holds for $\bar{\Theta}$. Since $h-\id\in\Linf(\mathbb{R})$, $\displaystyle\lim_{s\rightarrow-\infty}h(s)=-\infty$, so that
	\begin{equation*}
	\displaystyle\lim_{s\rightarrow -\infty}\bar{\Z}_{4}(s)=\displaystyle\lim_{s\rightarrow -\infty}\Z_{4}(h(s))=0
	\end{equation*}
	and we conclude that $\bar{\Theta}\in\G$.
\end{proof}

\begin{definition}
\label{def:Haction}
For any $(Z,p,q)\in \H$ and $f,g\in G$, we define
\begin{subequations}
\label{eqns:Haction}
\begin{align}
	\label{eq:HactionZ}
	\bar{Z}(X,Y)&=Z(f(X),g(Y)),\\ 
	\label{eq:Hactionp}
	\bar{p}(X,Y)&=f'(X)p(f(X),g(Y)),\\ 
	\label{eq:Hactionq}
	\bar{q}(X,Y)&=g'(Y)q(f(X),g(Y)).
\end{align}
\end{subequations}
The mapping $\H\times G^{2}\rightarrow\H$ given by $(Z,p,q)\times (f,g)\mapsto (\bar{Z},\bar{p},\bar{q})$ defines an action of the group $G^{2}$ on $\H$ and we denote $(\bar{Z},\bar{p},\bar{q})=(Z,p,q)\cdot (f,g)$.
\end{definition}

\begin{proof}[Proof of the well-posedness of Definition \ref{def:Haction}]
Condition (i) of Definition \ref{def:globalsolutions} implies that for any $(Z,p,q)\in\H$, we have 
\begin{aalign}
\label{eq:Hactionreg}
	&Z^{a}\in [\Winf(\mathbb{R}^{2})]^{5}, \quad Z^{a}_{X}\in [W^{1,\infty}_{Y}(\mathbb{R}^{2})]^{5}, \quad Z^{a}_{Y}\in [W^{1,\infty}_{X}(\mathbb{R}^{2})]^{5}, \\ &\hspace{60pt}p\in W^{1,\infty}_{Y}(\mathbb{R}^{2}), \quad q\in W^{1,\infty}_{X}(\mathbb{R}^{2}),
\end{aalign}
\begin{equation}
\label{eq:Hactionsoln2}
	(Z_{X}(X,Y))_{Y}=F(Z)(Z_{X},Z_{Y})(X,Y)
\end{equation}
for almost every $X\in\mathbb{R}$,
\begin{equation}
\label{eq:Hactionsoln3}
	(Z_{Y}(X,Y))_{X}=F(Z)(Z_{X},Z_{Y})(X,Y)
\end{equation}
for almost every $Y\in\mathbb{R}$,
\begin{equation}
\label{eq:Hactionsoln4}
	p_{Y}(X,Y)=0
\end{equation}
for almost every $X\in\mathbb{R}$,
\begin{equation}
\label{eq:Hactionsoln5}
	q_{X}(X,Y)=0
\end{equation}
for almost every $Y\in\mathbb{R}$. We first prove that the same regularity conditions hold for $(\bar{Z},\bar{p},\bar{q})$. Let us consider the first component of $\bar{Z}$. We have, by \eqref{eq:decayinfty},
\begin{align*}
	&\bar{Z}_{1}^{a}(X,Y)\\
	&=Z_{1}(f(X),g(Y))-\frac{1}{2c(0)}(X-Y)\\
	&=Z_{1}(f(X),g(Y))-\frac{1}{2c(0)}(f(X)-g(Y))+\frac{1}{2c(0)}(f(X)-g(Y))-\frac{1}{2c(0)}(X-Y)\\
	&=Z_{1}^{a}(f(X),g(Y))+\frac{1}{2c(0)}(f(X)-X+Y-g(Y)).
\end{align*}
Since $f-\id,g-\id\in\Linf(\mathbb{R})$ and $c\geq\frac{1}{\kappa}$, we get, by \eqref{eq:Hactionreg}, that $\bar{Z}_{1}^{a}\in\Linf(\mathbb{R}^{2})$. We differentiate and obtain
\begin{equation*}
	\bar{Z}_{1,X}^{a}(X,Y)=f'(X)Z_{1,X}^{a}(f(X),g(Y))+\frac{1}{2c(0)}(f'(X)-1)
\end{equation*}
and
\begin{equation*}
	\bar{Z}_{1,Y}^{a}(X,Y)=g'(Y)Z_{1,Y}^{a}(f(X),g(Y))-\frac{1}{2c(0)}(g'(Y)-1).
\end{equation*}
Since $f$ and $g$ belong to $G$, we have, by Lemma \ref{lemma:auxiliaryG}, that there exists $\alpha\geq0$ such that $\frac{1}{1+\alpha}\leq f'(X)\leq 1+\alpha$ for almost every $X\in \mathbb{R}$ and $\frac{1}{1+\alpha}\leq g'(Y)\leq 1+\alpha$ for almost every $Y\in \mathbb{R}$. This implies, by \eqref{eq:Hactionreg}, that $\bar{Z}_{1,X}^{a},\bar{Z}_{1,Y}^{a}\in\Linf(\mathbb{R}^{2})$ as $f'-1,g'-1\in\Linf(\mathbb{R})$ and $c\geq\frac{1}{\kappa}$. We have
\begin{equation*}
	\sup_{X\in\mathbb{R}}||\bar{Z}_{1,X}^{a}(X,\cdot)||_{\Linf(\mathbb{R})}\leq (1+\alpha)\sup_{X\in\mathbb{R}}||Z_{1,X}^{a}(X,\cdot)||_{\Linf(\mathbb{R})}+\frac{\kappa}{2}(\alpha+2)
\end{equation*}
and
\begin{equation*}
	\sup_{Y\in\mathbb{R}}||\bar{Z}_{1,Y}^{a}(\cdot,Y)||_{\Linf(\mathbb{R})}\leq(1+\alpha)\sup_{Y\in\mathbb{R}}||Z_{1,Y}^{a}(\cdot,Y)||_{\Linf(\mathbb{R})}+\frac{\kappa}{2}(\alpha+2).
\end{equation*}
Differentiating $\bar{Z}_{1,X}^{a}$ and $\bar{Z}_{1,Y}^{a}$ yields
\begin{equation*}
	\sup_{X\in\mathbb{R}}||\bar{Z}_{1,XY}^{a}(X,\cdot)||_{\Linf(\mathbb{R})}\leq (1+\alpha)^{2}\sup_{X\in\mathbb{R}}||Z_{1,XY}^{a}(X,\cdot)||_{\Linf(\mathbb{R})}
\end{equation*}
and
\begin{equation*}
	\sup_{Y\in\mathbb{R}}||\bar{Z}_{1,YX}^{a}(\cdot,Y)||_{\Linf(\mathbb{R})}\leq(1+\alpha)^{2}\sup_{Y\in\mathbb{R}}||Z_{1,YX}^{a}(\cdot,Y)||_{\Linf(\mathbb{R})}.
\end{equation*}
From \eqref{eq:Hactionreg}, we conclude that $\bar{Z}^{a}_{1,X}\in W^{1,\infty}_{Y}(\mathbb{R}^{2})$ and $\bar{Z}^{a}_{1,Y}\in W^{1,\infty}_{X}(\mathbb{R}^{2})$. One proves the same inclusions for the other components of $\bar{Z}$ in a similar way. We have
\begin{equation*}
	\sup_{X\in\mathbb{R}}||\bar{p}(X,\cdot)||_{\Linf(\mathbb{R})}\leq (1+\alpha)\sup_{X\in\mathbb{R}}||p(X,\cdot)||_{\Linf(\mathbb{R})}
\end{equation*}
and
\begin{equation*}
	\sup_{Y\in\mathbb{R}}||\bar{q}(\cdot,Y)||_{\Linf(\mathbb{R})}\leq(1+\alpha)\sup_{Y\in\mathbb{R}}||q(\cdot,Y)||_{\Linf(\mathbb{R})}
\end{equation*}
which are bounded by \eqref{eq:Hactionreg}. From \eqref{eq:Hactionsoln4} and \eqref{eq:Hactionsoln5}, we have
\begin{equation*}
	\bar{p}_{Y}(X,Y)=f'(X)g'(Y)p_{Y}(f(X),g(Y))=0
\end{equation*}
and
\begin{equation*}
	\bar{q}_{X}(X,Y)=f'(X)g'(Y)q_{X}(f(X),g(Y))=0,
\end{equation*}
so that condition (iv) and (v) of Definition \ref{def:soln} are satisfied. Moreover, $\bar{p}\in W^{1,\infty}_{Y}(\mathbb{R}^{2})$ and $\bar{q}\in W^{1,\infty}_{X}(\mathbb{R}^{2})$, so that conditions (i) of Definition \ref{def:soln} holds. By using the linearity of the mapping $F(Z)$, we obtain
\begin{align*}
	\bar{Z}_{XY}&=f'g'Z_{XY}(f,g)\\
	&=f'g'F(Z(f,g))(Z_{X}(f,g),Z_{Y}(f,g))\\
	&=F(Z(f,g))(f'Z_{X}(f,g),g'Z_{Y}(f,g))\\
	&=F(\bar{Z})(\bar{Z}_{X},\bar{Z}_{Y}),
\end{align*} 
so that conditions (ii) and (iii) of Definition \ref{def:soln} are satisfied. It remains to show the relations in \eqref{eqns:setH}. The identities \eqref{eq:setH1}-\eqref{eq:setH3} follows by direct computation. For instance, we have
\begin{align*}
	2\bar{J}_{X}\bar{x}_{X}&=2(f')^{2}J_{X}(f,g)x_{X}(f,g)\\
	&=(f')^{2}[(c(U(f,g))U_{X}(f,g))^{2}+c(U(f,g))p^{2}(f,g)]\\
	&=(c(U(f,g))f'U_{X}(f,g))^{2}+c(U(f,g))(f'p(f,g))^{2}\\
	&=(c(\bar{U})\bar{U}_{X})^{2}+c(\bar{U})\bar{p}^{2}.
\end{align*}
Since $x_{X}+J_{X}>0$ and $f'\geq\frac{1}{1+\alpha}>0$, we get
\begin{align*}
	\bar{x}_{X}+\bar{J}_{X}=f'(x_{X}(f,g)+J_{X}(f,g))\geq \frac{1}{1+\alpha}(x_{X}(f,g)+J_{X}(f,g))>0.
\end{align*}
Similarly, one shows the other inequalities in \eqref{eq:setH4}-\eqref{eq:setH6}. Hence, we have proved that $(\bar{Z},\bar{p},\bar{q})$ satisfies condition (i) of Definition \ref{def:globalsolutions}. To prove that the second condition is satisfied, we need the following lemma.

\begin{lemma}
\label{lemma:GHCactioncommute}
For any $(Z,p,q)\in \H$, $\Gamma=(\X,\Y)\in\C$ and $\phi=(f,g)\in G^{2}$, we have
\begin{equation*}
	((Z,p,q)\bullet\Gamma)\cdot\phi=((Z,p,q)\cdot\phi)\bullet(\Gamma\cdot\phi).
\end{equation*}
\end{lemma}

\begin{proof}
Let 
\begin{align*}
	\Theta&=(\X,\Y,\Z,\V,\W,\p,\q)=(Z,p,q)\bullet\Gamma,\\
	\bar{\Theta}&=(\bar{\X},\bar{\Y},\bar{\Z},\bar{\V},\bar{\W},\bar{\p},\bar{\q})=\Theta\cdot\phi,\\
	\bar{\Gamma}&=(\bar{\X},\bar{\Y})=\Gamma\cdot\phi,\\
	(\bar{Z},\bar{p},\bar{q})&=(Z,p,q)\cdot\phi,\\
	\tilde{\Theta}&=(\bar{\X},\bar{\Y},\tilde{\Z},\tilde{\V},\tilde{\W},\tilde{\p},\tilde{\q})=(\bar{Z},\bar{p},\bar{q})\bullet\bar{\Gamma}.
\end{align*}
We want to prove that $\bar{\Theta}=\tilde{\Theta}$. By \eqref{eq:HactionZ}, \eqref{eq:Caction} and \eqref{eq:Gaction2}, we get
\begin{equation*}
	\tilde{\Z}=\bar{Z}(\bar{\X},\bar{\Y})=Z(f\circ\bar{\X},g\circ\bar{\Y})=Z(\X\circ h,\Y\circ h)=\Z\circ h=\bar{\Z}.
\end{equation*}
We have
\begin{align*}
	\tilde{\V}(\bar{\X})&=\bar{Z}_{X}(\bar{\X},\bar{\Y})\\
	&=f'(\bar{\X})Z_{X}(f\circ\bar{\X},g\circ\bar{\Y}) \quad \text{by } \eqref{eq:HactionZ}\\
	&=f'(\bar{\X})Z_{X}(\X\circ h,\Y\circ h) \quad \text{by } \eqref{eq:Caction}\\
	&=f'(\bar{\X})\V(\X\circ h)\\
	&=f'(\bar{\X})\V(f\circ\bar{\X}) \quad \text{by } \eqref{eq:Caction}\\
	&=\bar{\V}(\bar{\X}) \quad \text{by } \eqref{eq:Gaction3}. 
\end{align*}
In a similar way, one shows that $\tilde{\W}=\bar{\W}$. Moreover,
\begin{align*}
	\tilde{\p}(\bar{\X})&=\bar{p}(\bar{\X},\bar{\Y})\\
	&=f'(\bar{\X})p(f\circ\bar{\X},g\circ\bar{\Y}) \quad \text{by } \eqref{eq:Hactionp}\\
	&=f'(\bar{\X})p(\X\circ h,\Y\circ h) \quad \text{by } \eqref{eq:Caction}\\
	&=f'(\bar{\X})\p(\X\circ h)\\
	&=f'(\bar{\X})\p(f\circ\bar{\X}) \quad \text{by } \eqref{eq:Caction}\\
	&=\bar{\p}(\bar{\X}) \quad \text{by } \eqref{eq:Gaction4}. 
\end{align*}
and
\begin{align*}
	\tilde{\q}(\bar{\Y})&=\bar{q}(\bar{\X},\bar{\Y})\\
	&=g'(\bar{\Y})q(f\circ\bar{\X},g\circ\bar{\Y}) \quad \text{by } \eqref{eq:Hactionq}\\
	&=g'(\bar{\Y})q(\X\circ h,\Y\circ h) \quad \text{by } \eqref{eq:Caction}\\
	&=g'(\bar{\Y})\q(\Y\circ h)\\
	&=g'(\bar{\Y})\q(g\circ\bar{\Y}) \quad \text{by } \eqref{eq:Caction}\\
	&=\bar{\q}(\bar{\Y}) \quad \text{by } \eqref{eq:Gaction4}
\end{align*}
and we have proved that $\bar{\Theta}=\tilde{\Theta}$.
\end{proof}

\emph{End of proof of the well-posedness of Definition \ref{def:Haction}.} Now we prove that for any $\phi=(f,g)\in G^{2}$ and $(Z,p,q)\in\H$, $(\bar{Z},\bar{p},\bar{q})$, as defined in \eqref{eqns:Haction}, satisfies condition (ii) in Definition \ref{def:globalsolutions}. Since $(Z,p,q)\in\H$, there exists a curve $\Gamma\in\C$ such that $(Z,p,q)\bullet\Gamma\in\G$. Consider the curve $\bar{\Gamma}=\Gamma\cdot\phi$ which, by Definition \ref{def:Caction}, belongs to $\C$. By Definition \ref{def:Gaction}, $((Z,p,q)\bullet\Gamma)\cdot\phi$ belongs to $\G$. This implies that $(\bar{Z},\bar{p},\bar{q})\bullet\bar{\Gamma}\in\G$, as
\begin{align*}
	(\bar{Z},\bar{p},\bar{q})\bullet\bar{\Gamma}&=((Z,p,q)\cdot\phi)\bullet(\Gamma\cdot\phi) && \text{by } \eqref{eqns:Haction}\\
	&=((Z,p,q)\bullet\Gamma)\cdot\phi && \text{by Lemma } \ref{lemma:GHCactioncommute}
\end{align*}
and we have proved the last condition. Hence, we conclude that $(\bar{Z},\bar{p},\bar{q})\in\H$.
\end{proof}

\begin{lemma}
The mappings $\mathbf{E}$, $\mathbf{t}_{T}$, $\mathbf{S}$, $\mathbf{D}$ and $\mathbf{C}$ are $G^{2}$-equivariant, that is, for all $\phi=(f,g)\in G^{2}$, we have
\begin{subequations}
\begin{align}
	\label{eq:equivar1}
	\mathbf{E}((Z,p,q)\cdot\phi)&=\mathbf{E}(Z,p,q)\cdot\phi,\\
	\label{eq:equivar2}
	\mathbf{t}_{T}((Z,p,q)\cdot\phi)&=\mathbf{t}_{T}(Z,p,q)\cdot\phi
\end{align}
for all $(Z,p,q)\in\H$,
\begin{equation}
\label{eq:equivar3}
	\mathbf{S}(\Theta\cdot\phi)=\mathbf{S}(\Theta)\cdot\phi
\end{equation}
for all $\Theta\in\G$,
\begin{equation}
\label{eq:equivar4}
	\mathbf{D}(\Theta\cdot\phi)=\mathbf{D}(\Theta)\cdot\phi
\end{equation}
for all $\Theta\in\G_{0}$, and
\begin{equation}
\label{eq:equivar5}
	\mathbf{C}(\psi\cdot\phi)=\mathbf{C}(\psi)\cdot\phi
\end{equation}
for all $\psi\in\F$. Therefore $S_{T}$ is $G^{2}$-equivariant, that is,
\begin{equation}
\label{eq:equivar6}
	S_{T}(\psi\cdot\phi)=S_{T}(\psi)\cdot\phi
\end{equation}
for all $\psi\in\F$.
\end{subequations}
\end{lemma}

\begin{proof}	
We decompose the proof into six steps.
	
\textbf{Step 1.}
We prove \eqref{eq:equivar1}. Let
\begin{align*}
	\Theta&=(\X,\Y,\Z,\V,\W,\p,\q)=\mathbf{E}(Z,p,q),\\
	\tilde{\Theta}&=(\tilde{\X},\tilde{\Y},\tilde{\Z},\tilde{\V},\tilde{\W},\tilde{\p},\tilde{\q})=\Theta\cdot\phi,\\
	(\bar{Z},\bar{p},\bar{q})&=(Z,p,q)\cdot\phi,\\
	\bar{\Theta}&=(\bar{\X},\bar{\Y},\bar{\Z},\bar{\V},\bar{\W},\bar{\p},\bar{\q})=\mathbf{E}(\bar{Z},\bar{p},\bar{q}).
\end{align*}
We want to prove that $\tilde{\Theta}=\bar{\Theta}$. First we show that $(\tilde{\X},\tilde{\Y})=(\bar{\X},\bar{\Y})$. By \eqref{eq:Emap} and \eqref{eqns:Haction}, we have
\begin{equation}
\label{eq:EequivX}
	\bar{\X}(s)=\sup\{X\in \mathbb{R}\ | \ t(f(X'),g(2s-X'))<0 \text{ for all } X'<X\}.
\end{equation}
From \eqref{eq:Emap}, we have $t(\X(s),\Y(s))=0$, so that $t(\X\circ h(s),\Y\circ h(s))=0$ which implies, by \eqref{eq:Caction}, that $t(f\circ\tilde{\X}(s),g\circ\tilde{\Y}(s))=0$. Hence, by \eqref{eq:EequivX}, $\bar{\X}\leq \tilde{\X}$. Let us assume that $\bar{\X}(s)<\tilde{\X}(s)$ for some $s\in\mathbb{R}$. Since $f$ and $g$ are strictly increasing functions, we have $f(\bar{\X}(s))<f(\tilde{\X}(s))$
and $g(\tilde{\Y}(s))<g(\bar{\Y}(s))$ which implies, by \eqref{eq:setH1} and \eqref{eq:setH4}, that
\begin{equation*}
	t(f\circ\bar{\X}(s),g\circ\bar{\Y}(s))\leq t(f\circ\tilde{\X}(s),g\circ\bar{\Y}(s))\leq t(f\circ\tilde{\X}(s),g\circ\tilde{\Y}(s)). 
\end{equation*}
By \eqref{eq:EequivX}, we have $t(f\circ\bar{\X}(s),g\circ\bar{\Y}(s))=0$ and since $t(f\circ\tilde{\X}(s),g\circ\tilde{\Y}(s))=t(\X\circ h(s),\Y\circ h(s))=0$, the monotonicity of $t$
implies that $t(X,Y)=0$ for all $(X,Y)\in[f\circ\bar{\X}(s),f\circ\tilde{\X}(s)]\times[g\circ\tilde{\Y}(s),g\circ\bar{\Y}(s)]$. If $f\circ\bar{\X}(s)\leq 2h(s)-g\circ\bar{\Y}(s)$, set $X'=2h(s)-g\circ\bar{\Y}(s)$ and $Y'=g\circ\bar{\Y}(s)$. We get
\begin{equation*}
	f\circ\bar{\X}(s)\leq X'<2h(s)-g\circ\tilde{\Y}(s)=2h(s)-\Y\circ h(s)=\X\circ h(s)=f\circ\tilde{\X}(s),
\end{equation*}
that is, $X'\in[f\circ\bar{\X}(s),f\circ\tilde{\X}(s)]$, so that $t(X',Y')=0$. Thus, we have $t(X',Y')=0$, $X'<\X\circ h(s)$ and $X'+Y'=2h(s)$, which contradicts the definition \eqref{eq:Emap} of $(\X,\Y)$ at $h(s)$. If $f\circ\bar{\X}(s)>2h(s)-g\circ\bar{\Y}(s)$, let $X'=f\circ\bar{\X}(s)$ and $Y'=2h(s)-f\circ\bar{\X}(s)$, which implies that $X'<f\circ\tilde{\X}(s)=\X\circ h(s)$ and
\begin{align*}
	g\circ\tilde{\Y}(s)&=\Y\circ h(s)=2h(s)-\X\circ h(s)=2h(s)-f\circ\tilde{\X}(s)\\
	&<2h(s)-f\circ\bar{\X}(s)=Y'<g\circ\bar{\Y}(s),
\end{align*}
so that $Y'\in[g\circ\tilde{\Y}(s),g\circ\bar{\Y}(s)]$. Thus, $t(X',Y')=0$ which is a contradiction, because $X'<\X\circ h(s)$ and $X'+Y'=2h(s)$. Hence, we must have $\bar{\X}=\tilde{\X}$, which implies that $\bar{\Y}=\tilde{\Y}$. Then, we get
\begin{align*}
	\mathbf{E}((Z,p,q)\cdot\phi)&=((Z,p,q)\cdot\phi)\bullet(\bar{\X},\bar{\Y})\\
	&=((Z,p,q)\cdot\phi)\bullet(\tilde{\X},\tilde{\Y})\\
	&=((Z,p,q)\cdot\phi)\bullet((\X,\Y)\cdot\phi)\\
	&=((Z,p,q)\bullet(\X,\Y))\cdot\phi \quad \text{by Lemma } \ref{lemma:GHCactioncommute}\\
	&=\mathbf{E}(Z,p,q)\cdot\phi
\end{align*}
and we have proved \eqref{eq:equivar1}.

\textbf{Step 2.}
Let us prove \eqref{eq:equivar2}. We denote
\begin{align*}
	(\hat{Z},\hat{p},\hat{q})&=(Z,p,q)\cdot\phi,\\
	(\tilde{Z},\tilde{p},\tilde{q})&=\mathbf{t}_{T}(\hat{Z},\hat{p},\hat{q}),\\
	(\check{Z},\check{p},\check{q})&=\mathbf{t}_{T}(Z,p,q),\\
	(\bar{Z},\bar{p},\bar{q})&=(\check{Z},\check{p},\check{q})\cdot\phi
\end{align*}
and we want to show that $(\tilde{Z},\tilde{p},\tilde{q})=(\bar{Z},\bar{p},\bar{q})$. By \eqref{eqns:Haction} and \eqref{eqns:tTmap}, we obtain
\begin{align*}
	\tilde{t}&=\hat{t}-T=t(f,g)-T=\check{t}(f,g)=\bar{t},\\
	\tilde{x}&=\hat{x}=x(f,g)=\check{x}(f,g)=\bar{x},\\
	\tilde{p}&=\hat{p}=f'p(f,g)=f'\check{p}(f,g)=\bar{p},\\
	\tilde{q}&=\hat{q}=g'q(f,g)=g'\check{q}(f,g)=\bar{q}.
\end{align*}
By a calculation similar to the one where it is shown that $\tilde{x}=\bar{x}$, one obtains for the remaining components that  $\tilde{Z}=\bar{Z}$.

\textbf{Step 3.}
We prove \eqref{eq:equivar3}. For any $\Theta=(\X,\Y,\Z,\V,\W,\p,\q)$, we denote
\begin{align*}
	(\tilde{Z},\tilde{p},\tilde{q})&=\mathbf{S}(\Theta\cdot\phi),\\
	(Z,p,q)&=\mathbf{S}(\Theta),\\
	(\bar{Z},\bar{p},\bar{q})&=(Z,p,q)\cdot\phi.
\end{align*}
We want to prove that $(\tilde{Z},\tilde{p},\tilde{q})=(\bar{Z},\bar{p},\bar{q})$. From the definition of the solution operator $\mathbf{S}$ in \eqref{eq:Soperator}, we have
\begin{equation*}
	(\tilde{Z},\tilde{p},\tilde{q})\bullet((\X,\Y)\cdot\phi)=\Theta\cdot\phi \quad \text{and} \quad (Z,p,q)\bullet(\X,\Y)=\Theta
\end{equation*}
which implies, along with Lemma \ref{lemma:GHCactioncommute}, that
\begin{align*}
	(\bar{Z},\bar{p},\bar{q})\bullet((\X,\Y)\cdot\phi)&=((Z,p,q)\cdot\phi)\bullet((\X,\Y)\cdot\phi)\\
	&=((Z,p,q)\bullet(\X,\Y))\cdot\phi=\Theta\cdot\phi.
\end{align*}
Hence, $(\tilde{Z},\tilde{p},\tilde{q})$ and $(\bar{Z},\bar{p},\bar{q})$ are two solutions to the same data $\Theta\cdot\phi$. By Theorem \ref{thm:globalsoln}, the solution is unique, so that $(\tilde{Z},\tilde{p},\tilde{q})=(\bar{Z},\bar{p},\bar{q})$.

\textbf{Step 4.}
Now we prove \eqref{eq:equivar4}. For any $\Theta=(\X,\Y,\Z,\V,\W,\p,\q)\in\G_{0}$, let
\begin{align*}
	\tilde{\Theta}&=(\tilde{\X},\tilde{\Y},\tilde{\Z},\tilde{\V},\tilde{\W},\tilde{\p},\tilde{\q})=\Theta\cdot\phi,\\ 
	\tilde{\psi}&=(\tilde{\psi}_{1},\tilde{\psi}_{2})=\mathbf{D}(\tilde{\Theta}),\\ 
	\psi&=(\psi_{1},\psi_{2})=\mathbf{D}(\Theta),\\ 
	\bar{\psi}&=(\bar{\psi}_{1},\bar{\psi}_{2})=\psi\cdot\phi,
\end{align*}
where we denote $\tilde{\psi}_{1}=(\tilde{x}_{1},\tilde{U}_{1},\tilde{J}_{1},\tilde{K}_{1},\tilde{V}_{1},\tilde{H}_{1})$ and  $\tilde{\psi}_{2}=(\tilde{x}_{2},\tilde{U}_{2},\tilde{J}_{2},\tilde{K}_{2},\tilde{V}_{2},\tilde{H}_{2})$,
and similar for $\psi$ and $\bar{\psi}$. We want to show that $\tilde{\psi}=\bar{\psi}$. The proofs of $\tilde{\psi}_{1}=\bar{\psi}_{1}$ and $\tilde{\psi}_{2}=\bar{\psi}_{2}$ are similar and we only show that $\tilde{\psi}_{1}=\bar{\psi}_{1}$. We have
\begin{align*}
	\tilde{x}_{1}(\tilde{\X}(s))&=\tilde{\Z}_{2}(s) \quad \text{by } \eqref{eq:mapfromG0toF1}\\
	&=\Z_{2}(h(s)) \quad \text{by } \eqref{eq:Gaction2}\\
	&=x_{1}(\X\circ h(s)) \quad \text{by } \eqref{eq:mapfromG0toF1}\\
	&=x_{1}(f\circ\tilde{\X}(s)) \quad \text{by } \eqref{eq:Caction}\\
	&=\bar{x}_{1}(\tilde{\X}(s)) \quad \text{by } \eqref{eq:Faction1}.
\end{align*}
Similarly, one proves that $\tilde{U}_{1}=\bar{U}_{1}$.
\begin{align*}
	\tilde{J}_{1}(\tilde{\X}(s))&=\int_{-\infty}^{\tilde{\X}(s)}\tilde{\V}_{4}(X)\,dX \quad \text{by } \eqref{eq:mapfromG0toF3}\\
	&=\int_{-\infty}^{f^{-1}\circ\X\circ h(s)}f'(X)\V_{4}(f(X))\,dX \quad \text{by } \eqref{eq:Caction} \text{ and } \eqref{eq:Gaction3}\\
	&=\int_{-\infty}^{\X\circ h(s)}\V_{4}(X)\,dX \quad \text{by a change of variables}\\
	&=J_{1}(\X\circ h(s)) \quad \text{by } \eqref{eq:mapfromG0toF3}\\
	&=J_{1}(f\circ\tilde{\X}(s)) \quad \text{by } \eqref{eq:Caction}\\
	&=\bar{J}_{1}(\tilde{\X}(s)) \quad \text{by } \eqref{eq:Faction3}.
\end{align*}
By a similar calculation, one shows that $\tilde{K}_{1}=\bar{K}_{1}$. From \eqref{eq:mapfromG0toF6}, \eqref{eq:Gaction4} and \eqref{eq:Faction6}, we obtain
\begin{align*}
	\tilde{H}_{1}(\tilde{\X}(s))&=\tilde{\p}(\tilde{\X}(s))=f'(\tilde{\X}(s))\p(f\circ\tilde{\X}(s))=f'(\tilde{\X}(s))H_{1}(f\circ\tilde{\X}(s))=\bar{H}_{1}(\tilde{\X}(s)),\\
	\tilde{H}_{2}(\tilde{\Y}(s))&=\tilde{\q}(\tilde{\Y}(s))=g'(\tilde{\Y}(s))\q(g\circ\tilde{\Y}(s))=g'(\tilde{\Y}(s))H_{2}(g\circ\tilde{\Y}(s))=\bar{H}_{2}(\tilde{\Y}(s)).
\end{align*}
Similarly, one proves that $\tilde{V}_{1}=\bar{V}_{1}$.

\textbf{Step 5.}
We prove \eqref{eq:equivar5}. Given $\psi=(\psi_{1},\psi_{2})\in\F$, we denote
\begin{align*}
	\bar{\psi}&=(\bar{\psi}_{1},\bar{\psi}_{2})=\psi\cdot\phi,\\
	\bar{\Theta}&=(\bar{\X},\bar{\Y},\bar{\Z},\bar{\V},\bar{\W},\bar{\p},\bar{\q})=\mathbf{C}(\bar{\psi}),\\
	\Theta&=(\X,\Y,\Z,\V,\W,\p,\q)=\mathbf{C}(\psi),\\
	\tilde{\Theta}&=(\tilde{\X},\tilde{\Y},\tilde{\Z},\tilde{\V},\tilde{\W},\tilde{\p},\tilde{\q})=\Theta\cdot\phi.
\end{align*}
We want to prove that $\tilde{\Theta}=\bar{\Theta}$. We first show that $(\tilde{\X},\tilde{\Y})=(\bar{\X},\bar{\Y})$. By \eqref{eq:mapFtoGX} and \eqref{eq:Faction1}, we have
\begin{equation}
\label{eq:CequivX}
	\bar{\X}(s)=\sup\{X\in \mathbb{R} \ | \ x_{1}\circ f(X')<x_{2}\circ g(2s-X') \text{ for all } X'<X \}.
\end{equation}
From \eqref{eq:mapFtoG2}, we obtain $x_{1}\circ\X\circ h=x_{2}\circ\Y\circ h$ which implies, by \eqref{eq:Caction}, that $x_{1}\circ f\circ\tilde{\X}=x_{2}\circ g\circ\tilde{\Y}$. Hence, by \eqref{eq:CequivX}, $\bar{\X}\leq \tilde{\X}$. Assume that $\bar{\X}(s)<\tilde{\X}(s)$ for some $s\in\mathbb{R}$. Since $f$ and $g$ are strictly increasing functions, we have $f(\bar{\X}(s))<f(\tilde{\X}(s))$
and $g(\tilde{\Y}(s))<g(\bar{\Y}(s))$ which implies, by \eqref{eq:setFrel1}, that
\begin{equation*}
	x_{1}\circ f\circ\bar{\X}(s)\leq x_{1}\circ f\circ\tilde{\X}(s)=x_{2}\circ g\circ\tilde{\Y}(s)\leq x_{2}\circ g\circ\bar{\Y}(s).
\end{equation*}
By \eqref{eq:CequivX}, we have $x_{1}\circ f\circ\bar{\X}(s)=x_{2}\circ g\circ\bar{\Y}(s)$, so we must have $x_{1}\circ f\circ\bar{\X}(s)=x_{1}\circ f\circ\tilde{\X}(s)$ and $x_{2}\circ g\circ\tilde{\Y}(s)=x_{2}\circ g\circ\bar{\Y}(s)$. This implies, since $x_{1}$ and $x_{2}$ are nondecreasing, that $x_{1}$ and $x_{2}$ are constant on $[f\circ\bar{\X}(s),f\circ\tilde{\X}(s)]$ and $[g\circ\tilde{\Y}(s),g\circ\bar{\Y}(s)]$, respectively. If $f\circ\bar{\X}(s)\leq 2h(s)-g\circ\bar{\Y}(s)$, set $X'=2h(s)-g\circ\bar{\Y}(s)$ and $Y'=g\circ\bar{\Y}(s)$. We get
\begin{equation*}
	f\circ\bar{\X}(s)\leq X'<2h(s)-g\circ\tilde{\Y}(s)=2h(s)-\Y\circ h(s)=\X\circ h(s)=f\circ\tilde{\X}(s),
\end{equation*}
that is, $X'\in[f\circ\bar{\X}(s),f\circ\tilde{\X}(s)]$, so that $x_{1}(X')=x_{1}(f\circ\tilde{\X}(s))$. Similarly, we get that $x_{2}(Y')=x_{2}(g\circ\tilde{\Y}(s))$. Then, by \eqref{eq:Caction} and \eqref{eq:mapFtoG2}, we obtain
\begin{equation*}
	x_{1}(X')=x_{1}(f\circ\tilde{\X}(s))=x_{1}(\X\circ h(s))=x_{2}(\Y\circ h(s))=x_{2}(g\circ\tilde{\Y}(s))=x_{2}(Y').
\end{equation*} 
Thus, we have $x_{1}(X')=x_{2}(Y')$, $X'<\X\circ h(s)$ and $X'+Y'=2h(s)$, which contradicts the definition \eqref{eq:mapFtoGX} of $(\X,\Y)$ at $h(s)$. If $f\circ\bar{\X}(s)>2h(s)-g\circ\bar{\Y}(s)$, let $X'=f\circ\bar{\X}(s)$ and $Y'=2h(s)-f\circ\bar{\X}(s)$, which implies that $X'<f\circ\tilde{\X}(s)=\X\circ h(s)$ and $x_{1}(X')=x_{1}(f\circ\tilde{\X}(s))$. We get
\begin{align*}
	g\circ\tilde{\Y}(s)&=\Y\circ h(s)=2h(s)-\X\circ h(s)=2h(s)-f\circ\tilde{\X}(s)\\
	&<2h(s)-f\circ\bar{\X}(s)=Y'<g\circ\bar{\Y}(s),
\end{align*}
so that $Y'\in[g\circ\tilde{\Y}(s),g\circ\bar{\Y}(s)]$ and $x_{2}(Y')=x_{2}(g\circ\tilde{\Y}(s))$. Then, as before, we obtain $x_{1}(X')=x_{2}(Y')$ which is a contradiction, because $X'<\X\circ h(s)$ and $X'+Y'=2h(s)$. Hence, we must have $\bar{\X}=\tilde{\X}$, which implies that $\bar{\Y}=\tilde{\Y}$.
Then, by a straightforward calculation, one proves that $\bar{\Z}=\tilde{\Z}$, $\bar{\V}=\tilde{\V}$, $\bar{\W}=\tilde{\W}$, $\bar{\p}=\tilde{\p}$ and $\bar{\q}=\tilde{\q}$. For example, we have
\begin{align*}
	\bar{\Z}_{3}(s)&=\bar{U}_{1}(\bar{\X}(s)) \quad \text{by } \eqref{eq:mapFtoGZ3}\\
	&=\bar{U}_{1}(\tilde{\X}(s))\\
	&=U_{1}(f\circ\tilde{\X}(s)) \quad \text{by } \eqref{eq:Faction2}\\
	&=U_{1}(\X\circ h(s)) \quad \text{by } \eqref{eq:Caction}\\
	&=\Z_{3}(h(s)) \quad \text{by } \eqref{eq:mapFtoGZ3}\\
	&=\tilde{\Z}_{3}(s) \quad \text{by } \eqref{eq:Gaction2},
\end{align*}
\begin{align*}
	\bar{\V}_{1}(\bar{\X}(s))&=\frac{1}{2c(\bar{U}_{1}\circ\tilde{\X}(s))}\bar{x}_{1}'(\tilde{\X}(s)) \quad \text{by } \eqref{eq:mapFtoGV1}\\
	&=\frac{f'(\tilde{\X}(s))}{2c(U_{1}\circ f\circ\tilde{\X}(s))}x_{1}'(f\circ\tilde{\X}(s)) \quad \text{by } \eqref{eq:Faction1} \text{ and } \eqref{eq:Faction2}\\
	&=f'(\tilde{\X}(s))\V_{1}(f\circ\tilde{\X}(s)) \quad \text{by } \eqref{eq:mapFtoGV1}\\
	&=\tilde{\V}_{1}(\tilde{\X}(s)) \quad \text{by } \eqref{eq:Gaction3},
\end{align*}
and, by \eqref{eq:mapFtoGp}, \eqref{eq:Faction6} and \eqref{eq:Gaction4}, 
\begin{align*}
	\bar{\p}(\bar{\X}(s))&=\bar{H}_{1}(\tilde{\X}(s))=f'(\tilde{\X}(s))H_{1}(f\circ\tilde{\X}(s))=f'(\tilde{\X}(s))\p(f\circ\tilde{\X}(s))=\tilde{\p}(\tilde{\X}(s)),\\
	\bar{\q}(\bar{\Y}(s))&=\bar{H}_{2}(\tilde{\Y}(s))=g'(\tilde{\Y}(s))H_{2}(g\circ\tilde{\Y}(s))=g'(\tilde{\Y}(s))\q(g\circ\tilde{\Y}(s))=\tilde{\q}(\tilde{\Y}(s)).
\end{align*}
This concludes the proof of \eqref{eq:equivar5}.

\textbf{Step 6.}
We are now ready to prove \eqref{eq:equivar6}. Let
\begin{equation*}
	\Theta=\mathbf{C}(\psi),\quad (Z,p,q)=\mathbf{S}(\Theta),\quad (\bar{Z},\bar{p},\bar{q})
	=\mathbf{t}_{T}(Z,p,q),\quad \bar{\Theta}=\mathbf{E}(\bar{Z},\bar{p},\bar{q}),\quad \bar{\psi}=\mathbf{D}(\bar{\Theta}).
\end{equation*}
We claim that $S_{T}(\psi\cdot\phi)=S_T(\psi)\cdot\phi$. This follows from Step 1-6, as
\begin{align*}
	S_{T}(\psi\cdot\phi)&=\mathbf{D}\circ\mathbf{E}\circ\mathbf{t}_{T}\circ\mathbf{S}\circ\mathbf{C}(\psi\cdot\phi)\\
	&=\mathbf{D}\circ\mathbf{E}\circ\mathbf{t}_{T}\circ\mathbf{S}(\mathbf{C}(\psi)\cdot\phi) \quad \text{by } \eqref{eq:equivar5}\\
	&=\mathbf{D}\circ\mathbf{E}\circ\mathbf{t}_{T}\circ\mathbf{S}(\Theta\cdot\phi)\\
	&=\mathbf{D}\circ\mathbf{E}\circ\mathbf{t}_{T}(\mathbf{S}(\Theta)\cdot\phi) \quad \text{by } \eqref{eq:equivar3}\\
	&=\mathbf{D}\circ\mathbf{E}\circ\mathbf{t}_{T}((Z,p,q)\cdot\phi)\\
	&=\mathbf{D}\circ\mathbf{E}(\mathbf{t}_{T}(Z,p,q)\cdot\phi) \quad \text{by } \eqref{eq:equivar2}\\
	&=\mathbf{D}\circ\mathbf{E}((\bar{Z},\bar{p},\bar{q})\cdot\phi)\\
	&=\mathbf{D}(\mathbf{E}(\bar{Z},\bar{p},\bar{q})\cdot\phi) \quad \text{by } \eqref{eq:equivar1}\\
	&=\mathbf{D}(\bar{\Theta}\cdot\phi)\\
	&=\mathbf{D}(\bar{\Theta})\cdot\phi \quad \text{by } \eqref{eq:equivar4}\\
	&=\bar{\psi}\cdot\phi\\
	&=S_T(\psi)\cdot\phi.
\end{align*}
\end{proof}

\begin{definition}
We denote by $\F/G^{2}$ the quotient of $\F$ with respect to the action of the group $G^{2}$ on $\F$. More specifically, we define the equivalence relation, $\sim$, on $\F$ as
\begin{equation*}
	\text{for any } \psi,\bar{\psi}\in\F,\ \psi\sim\bar{\psi} \text{ if there exists } \phi\in G^{2} \text{ such that } \bar{\psi}=\psi\cdot\phi.
\end{equation*}
For an element $\psi\in\F$, we denote the equivalence class by
\begin{equation*}
	[\psi]=\{\bar{\psi}\in\F \ | \ \bar{\psi}\sim\psi\}.
\end{equation*} 
We define the quotient space as
\begin{equation*}
	\F/G^{2}=\{[\psi] \ | \ \psi\in\F\}.
\end{equation*}
\end{definition}

\begin{definition}
\label{def:F0}
Let
\begin{equation*}
	\F_{0}=\{\psi=(\psi_{1},\psi_{2})\in\F \ | \ x_{1}+J_{1}=\id \text{ and } x_{2}+J_{2}=\id\}
\end{equation*}
and $\Pi:\F\rightarrow\F_{0}$ be the projection on $\F_{0}$ given by $\bar{\psi}=(\bar{\psi}_{1},\bar{\psi}_{2})=\Pi(\psi)$ where $\bar{\psi}\in\F_{0}$ is defined as follows. Let
\begin{equation}
\label{eq:projectionfandg}
	f(X)=x_{1}(X)+J_{1}(X) \quad \text{and} \quad g(Y)=x_{2}(Y)+J_{2}(Y)
\end{equation}
and denote $\phi=(f,g)\in G^{2}$. We set
\begin{equation*}
	\bar{\psi}=\psi\cdot\phi^{-1}.
\end{equation*}
\end{definition}

\begin{proof}[Proof of the well-posedness of Definition \ref{def:F0}]
Since $x_{1}+J_{1}$ and $x_{2}+J_{2}$ belong to $G$, $f$ and $g$ belong to $G$ and their inverses exist. We claim that $f^{-1}$ and $g^{-1}$ belong to $G$. This immediately follows from \eqref{eq:groupcond} and the following estimate. Since $f\in G$, Lemma \ref{lemma:auxiliaryG} implies that there exists $\alpha\geq 0$ such that $\frac{1}{1+\alpha}\leq f'\leq 1+\alpha$ almost everywhere, and we get
\begin{align*}
	\int_{\mathbb{R}}(f^{-1}(X)'-1)^{2}\,dX&=\int_{\mathbb{R}}\bigg(\frac{1-f'(f^{-1}(X))}{f'(f^{-1}(X))}\bigg)^{2}\,dX\\
	&\leq(1+\alpha)\int_{\mathbb{R}}\frac{(1-f'(f^{-1}(X)))^{2}}{f'(f^{-1}(X))}\,dX\\
	&=(1+\alpha)\int_{\mathbb{R}}(1-f'(X))^{2}\,dX
\end{align*}
by a change of variables. Hence, $(f^{-1})'-1\in L^{2}(\mathbb{R})$ and $f^{-1}\in G$. The same argument shows that $g^{-1}\in G$. Then the proof of the well-posedness of Definition \ref{def:Faction} implies that $\bar{\psi}=\psi\cdot\phi^{-1}$ belongs to $\F$. Furthermore, we have
\begin{equation*}
	\bar{x}_{1}(X)+\bar{J}_{1}(X)=(x_{1}+J_{1})\circ f^{-1}(X)=(x_{1}+J_{1})\circ(x_{1}+J_{1})^{-1}(X)=X.
\end{equation*}
By a similar calculation, we get $\bar{x}_{2}(Y)+\bar{J}_{2}(Y)=Y$ and conclude that $\bar\psi\in\F_{0}$. 
\end{proof}	

\begin{lemma}
\label{lemma:semigroupDtoDaux}
The following statements hold:
\begin{enumerate}
	\item[(i)] For any $\psi$ and $\bar{\psi}$ in $\F$, we have
	\begin{subequations}
	\begin{equation}
	\label{eq:quotproj1}
		\psi\sim\bar{\psi}\quad\text{if and only if}\quad\Pi(\psi)=\Pi(\bar{\psi}),
	\end{equation} 
	so that the sets $\F/G^{2}$ and $\F_{0}$ are in bijection.
	\item[(ii)] We have
	\begin{equation}
	\label{eq:quotproj2}
		\mathbf{M}\circ\Pi=\mathbf{M}
	\end{equation}
	and
	\begin{equation}
	\label{eq:quotproj3}
		\mathbf{L}\circ\mathbf{M}|_{\F_{0}}=\id|_{\F_{0}} \quad \text{and} \quad \mathbf{M}\circ\mathbf{L}=\id, 
	\end{equation}
	so that the sets $\D$, $\F_{0}$ and $\F/G^{2}$ are in bijection.
	\item[(iii)] We have
	\begin{equation}
	\label{eq:quotproj4}
		\Pi\circ S_{T}\circ\Pi=\Pi\circ S_{T}.
	\end{equation}
	\end{subequations}
\end{enumerate}
\end{lemma}

Note that the first identity in \eqref{eq:quotproj3} is equivalent to
\begin{equation}
\label{eq:quotproj5}
	\mathbf{L}\circ\mathbf{M}\circ\Pi=\Pi.
\end{equation}

Before we prove the lemma we make some remarks. 

Let $\xi_{0}=(u_{0},R_{0},S_{0},\rho_{0},\sigma_{0},\mu_{0},\nu_{0})\in\D$. Consider 
$\psi=\mathbf{L}(\xi_{0})$, $\bar{\psi}=S_{T}(\psi)$ and $\xi_{T}=\mathbf{M}(\bar{\psi})$. 

Let $\phi\in G^{2}$ and use $\hat{\psi}=\psi\cdot\phi$ as initial data for the solution operator $S_{T}$. From \eqref{eq:equivar6} we have $S_{T}(\hat{\psi})=S_{T}(\psi)\cdot\phi=\bar{\psi}\cdot\phi=\tilde{\psi}$. Let
$\tilde{\xi}_{T}=\mathbf{M}(\tilde{\psi})$. Since $\phi\in G^{2}$ we have $\tilde{\psi}\sim\bar{\psi}$. Then, by \eqref{eq:quotproj1} we get $\Pi(\tilde{\psi})=\Pi(\bar{\psi})$ and from \eqref{eq:quotproj2} we get  
\begin{equation*}
	\mathbf{M}(\tilde{\psi})=\mathbf{M}(\Pi(\tilde{\psi}))=\mathbf{M}(\Pi(\bar{\psi}))=\mathbf{M}(\bar{\psi}).
\end{equation*}
This implies
\begin{equation*}
	\tilde{\xi}_{T}=\mathbf{M}(\tilde{\psi})=\mathbf{M}(\bar{\psi})=\xi_{T},
\end{equation*} 
so the two solutions are identical.

We can think of this as follows: To each element $\xi_{0}\in\D$ there correspond infinitely many elements in $\F$, all belonging to the same equivalence class. The mapping $\mathbf{L}:\D\to \F_0$ picks one member of the equivalence class, but we could also pick a different one. Applying the solution operator to all elements belonging to the same equivalence class yields infinitely many solutions in $\F$, which form an equivalence class. Using the mapping $\mathbf{M}:\F\to \D$ on all of these solutions yields the same element in $\D$. Since we get the same solution in the end, we can think of each member of the equivalence class as a different "parametrization" of the initial data in $\F$, which are connected through relabeling.

The following example shows how we can use relabeling in order to get different initial curves in $\G_{0}$. We use the same notation as above. Assume that $\xi_{0}=(u_{0},R_{0},S_{0},\rho_{0},\sigma_{0},\mu_{0},\nu_{0})$ belongs to $\D$ and satisfies the additional conditions $u_{0},R_{0},S_{0},\rho_{0},\sigma_{0}\in L^{\infty}(\mathbb{R})$ and that $\mu_{0}$ and $\nu_{0}$ are absolutely continuous. We have
\begin{equation*}
	f(x_{1}(X))=X \quad \text{and} \quad g(x_{2}(Y))=Y,
\end{equation*}
where
\begin{equation*}
	f(x)=x+\frac{1}{4}\int_{-\infty}^{x}(R_{0}^{2}+c(u_{0})\rho_{0}^{2})(z)\,dz \quad \text{and} \quad  g(x)=x+\frac{1}{4}\int_{-\infty}^{x}(S_{0}^{2}+c(u_{0})\sigma_{0}^{2})(z)\,dz.
\end{equation*}
Since the functions $f$ and $g$ are strictly increasing, they are invertible and $x_{1}$ and $x_{2}$ are given as
\begin{equation*}
	x_{1}(X)=f^{-1}(X) \quad \text{and} \quad x_{2}(Y)=g^{-1}(Y).
\end{equation*}
Let $\Theta=\mathbf{C}(\psi)$. The functions $\X$ and $\Y$ are given by
\begin{equation*}
	x_{1}(\X(s))=x_{2}(2s-\X(s)) \quad \text{and} \quad \Y(s)=2s-\X(s). 
\end{equation*}
From this we get that $\X$ and $\Y$ are strictly increasing functions. Furthermore, the functions $f$ and $g$ belong to $G$ and can therefore be used as relabeling functions. In the above notation this means $\phi=(f,g)$. Denote $\hat\psi=\psi\cdot\phi$. Then we get $\hat{x}_{1}(X)=X$ and $\hat{x}_{2}(Y)=Y$. Consider $\hat{\Theta}=\mathbf{C}(\hat{\psi})$. Now we get $\hat{\X}(s)=s$ and $\hat{\Y}(s)=s$. As above, using either $\Theta$ or $\tilde{\Theta}$ yields the same solution in $\D$.
Therefore, for this type of initial data in $\D$ we can without loss of generality assume that the initial curve $(\X,\Y)$ is the identity, i.e., $\X(s)=s$ and $\Y(s)=s$. 
Note that without the additional assumptions on the initial data in $\D$, it is not possible to relabel the initial curve into the identity.

\begin{proof}
	We decompose the proof into five steps.
	
\textbf{Step 1.}
We prove \eqref{eq:quotproj1}. If $\psi\sim\bar{\psi}$, there exists $\tilde{\phi}=(\tilde{f},\tilde{g})\in G^{2}$ such that $(\bar{\psi}_{1},\bar{\psi}_{2})=\bar{\psi}=\psi\cdot\tilde{\phi}=(\psi_{1},\psi_{2})\cdot\tilde{\phi}$. Let $\phi=(f,g)$ and $\bar{\phi}=(\bar{f},\bar{g})$ be given by \eqref{eq:projectionfandg} for $\psi$ and $\bar{\psi}$, respectively. We have $\bar{\phi}=\phi\circ\tilde{\phi}$ because
\begin{equation*}
	\bar{\phi}=(\bar{f},\bar{g})=(\bar{x}_{1}+\bar{J}_{1},\bar{x}_{2}+\bar{J}_{2})=((x_{1}+J_{1})\circ\tilde{f},(x_{2}+J_{2})\circ\tilde{g})=(f\circ\tilde{f},g\circ\tilde{g})=\phi\circ\tilde{\phi}.
\end{equation*}
Then, we get
\begin{align*}
	\Pi(\bar{\psi})&=\bar{\psi}\cdot(\bar{\phi})^{-1}=(\psi\cdot\tilde{\phi})\cdot(\phi\circ\tilde{\phi})^{-1}=\psi\cdot(\tilde{\phi}\circ(\phi\circ\tilde{\phi})^{-1})\\
	&=\psi\cdot(\tilde{\phi}\circ(\tilde{\phi})^{-1}\circ\phi^{-1})=\psi\cdot\phi^{-1}=\Pi(\psi),
\end{align*}
where the identity $(\psi\cdot\tilde{\phi})\cdot(\phi\circ\tilde{\phi})^{-1}=\psi\cdot(\tilde{\phi}\circ(\phi\circ\tilde{\phi})^{-1})$ follows from a straightforward calculation using Definition \ref{def:Faction}. For example, by \eqref{eq:Faction1}, we have
\begin{equation*}
	\bar{x}_{1}\circ F(X)=x_{1}\circ\tilde{f}\circ F(X)=x_{1}\circ f^{-1},
\end{equation*}
where we denote $F=(f \circ\tilde{f})^{-1}$, and
\begin{align*}
	F'(X)\bar{H}_{1}\circ F(X)&=F'(X)(\tilde{f})'\circ F(X)H_{1}\circ\tilde{f}\circ F(X)\\
	&=(\tilde{f}\circ F(X))_{X}H_{1}\circ\tilde{f}\circ F(X)\\
	&=(f^{-1})'H_{1}\circ f^{-1}(X).
\end{align*}
Conversely, if $\Pi(\psi)=\Pi(\bar{\psi})$, then $\psi\cdot\phi^{-1}=\bar{\psi}\cdot(\bar{\phi})^{-1}$, where $\phi=(f,g)$ and $\bar{\phi}=(\bar{f},\bar{g})$ are given by \eqref{eq:projectionfandg}. This implies that
\begin{equation}
\label{eq:projaux}
	\bar{\psi}=\psi\cdot(\phi^{-1}\circ\bar{\phi}).
\end{equation}
This also follows from a direct calculation. For example, we have
\begin{equation*}
	((\bar{f})^{-1})'(X)\bar{H}_{1}\circ(\bar{f})^{-1}(X)=(f^{-1})'(X)H_{1}\circ f^{-1}(X)
\end{equation*} 
which implies that
\begin{equation*}
	((\bar{f})^{-1})'\circ\bar{f}(X)\bar{H}_{1}(X)=(f^{-1})'\circ\bar{f}(X)H_{1}\circ f^{-1}\circ\bar{f}(X)
\end{equation*} 
and since $((\bar{f})^{-1})'\circ\bar{f}(X)=\frac{1}{(\bar{f})'(X)}$, we get\footnote{Since $\bar{f}\in G$, there exists $\delta>0$ such that $\bar{f}'\geq\delta$ almost everywhere, see Lemma \ref{lemma:auxiliaryG}.}
\begin{equation*}
	\bar{H}_{1}(X)=(f^{-1}\circ\bar{f}(X))_{X}H_{1}\circ f^{-1}\circ\bar{f}(X).
\end{equation*}
Then, since $\phi^{-1}\circ\bar{\phi}\in G^{2}$, \eqref{eq:projaux} implies that $\psi$ and $\bar{\psi}$ are equivalent.

\textbf{Step 2.}
We prove \eqref{eq:quotproj2}. Given $\psi=(\psi_{1},\psi_{2})\in\F$, let $(u,R,S,\rho,\sigma,\mu,\nu)=\mathbf{M}(\psi)$, $\bar{\psi}=(\bar{\psi}_{1},\bar{\psi}_{2})=\Pi(\psi)$ and $(\bar{u},\bar{R},\bar{S},\bar{\rho},\bar{\sigma},\bar{\mu},\bar{\nu})=\mathbf{M}(\bar{\psi})$. We want to prove that $(\bar{u},\bar{R},\bar{S},\bar{\rho},\bar{\sigma},\bar{\mu},\bar{\nu})=(u,R,S,\rho,\sigma,\mu,\nu)$. From \eqref{eq:mapFtoD1}, \eqref{eq:Faction1} and \eqref{eq:Faction2}, we get
\begin{equation*}
	\bar{u}(\bar{x}_{1}(X))=\bar{U}_{1}(X)=U_{1}(f^{-1}(X))=u(x_{1}\circ f^{-1}(X))=u(\bar{x}_{1}(X)).
\end{equation*}
For any Borel set $B$, we have
\begin{align*}
	\int_{B}\bar{R}(x)\,dx&=\int_{\bar{x}_{1}^{-1}(B)}2c(\bar{U}_{1}(X))\bar{V}_{1}(X)\,dX \quad \text{by } \eqref{eq:mapFtoD3}\\
	&=\int_{\{X\in\mathbb{R} \ | \ x_{1}(f^{-1}(X))\in B\}}2c(U_{1}(f^{-1}(X)))(f^{-1}(X))'V_{1}(f^{-1}(X))\,dX\\
	&=\int_{x_{1}^{-1}(B)}2c(U_{1}(X))V_{1}(X)\,dX \quad \text{by a change of variables}\\
	&=\int_{B}R(x)\,dx \quad \text{by } \eqref{eq:mapFtoD3},
\end{align*}
where we used \eqref{eq:Faction1}, \eqref{eq:Faction2} and \eqref{eq:Faction5}. Hence, $\bar{R}=R$ almost everywhere. By \eqref{eq:mapFtoD5} and \eqref{eq:Faction6}, we obtain
\begin{align*}
	\int_{B}\bar{\rho}(x)\,dx&=\int_{\bar{x}_{1}^{-1}(B)}2\bar{H}_{1}(X)\,dX\\
	&=\int_{\bar{x}_{1}^{-1}(B)}2(f^{-1}(X))'H_{1}(f^{-1}(X))\,dX\\
	&=\int_{x_{1}^{-1}(B)}2H_{1}(X)\,dX\\
	&=\int_{B}\rho(x)\,dx.
\end{align*} 
Similarly, one proves that $\bar{S}=S$ and $\bar \sigma=\sigma$ almost everywhere. Using \eqref{eq:mapFtoD7} and \eqref{eq:Faction3}, we find for any Borel set $B\subset \mathbb{R}$, that 
\begin{align*}
	\bar{\mu}(B)&=\int_{\bar{x}_{1}^{-1}(B)}\bar{J}_{1}(X)\,dX\\
	&=\int_{\bar{x}_{1}^{-1}(B)}(f^{-1}(X))'J_{1}'(f^{-1}(X))\,dX\\
	&=\int_{x_{1}^{-1}(B)}J_{1}'(X)\,dX\\
	&=\mu(B)
\end{align*}
and $\bar \mu=\mu$. One proves that $\bar{\nu}=\nu$ in a similar way.

\textbf{Step 3.}
We prove that $\mathbf{L}\circ\mathbf{M}|_{\F_{0}}=\id|_{\F_{0}}$. Given $\psi=(\psi_{1},\psi_{2})\in\F_{0}$, let $(u,R,S,\rho,\sigma,\mu,\nu)=\mathbf{M}(\psi)$ and $\bar{\psi}=(\bar{\psi}_{1},\bar{\psi}_{2})=\mathbf{L}(u,R,S,\rho,\sigma,\mu,\nu)$. We want to show that $\bar{\psi}=\psi$. We first prove that $\bar{x}_{1}=x_{1}$. Let
\begin{equation}
\label{eq:LcompMg}
	g(x)=\sup\{X\in\mathbb{R} \ | \ x_{1}(X)<x\}.
\end{equation} 
For all $x\in\mathbb{R}$, we have
\begin{equation}
\label{eq:LcompMx1g}
	x_{1}(g(x))=x
\end{equation}
and since $x_{1}$ is continuous and nondecreasing, $x_{1}^{-1}((-\infty,x))=(-\infty,g(x))$. From \eqref{eq:mapFtoD7} and \eqref{eq:setFrel5}, we get
\begin{equation}
\label{eq:LcompMmuJ1}
	\mu((-\infty,x))=\int_{x_{1}^{-1}((-\infty,x))}J_{1}'(X)\,dX=\int_{-\infty}^{g(x)}J_{1}'(X)\,dX=J_{1}(g(x)).
\end{equation}
Since $\psi\in\F_{0}$, $x_{1}+J_{1}=\id$ which implies, by \eqref{eq:LcompMx1g} and \eqref{eq:LcompMmuJ1}, that
\begin{equation}
\label{eq:LcompMmug}
	x+\mu((-\infty,x))=g(x).
\end{equation}
From the definition \eqref{eq:mapfromDtoF1} of $\bar{x}_{1}$, we then obtain
\begin{equation}
\label{eq:LcompMx1barg}
	\bar{x}_{1}(X)=\sup\{x\in\mathbb{R} \ | \ g(x)<X\}.
\end{equation}
This implies that, for any $X\in\mathbb{R}$, there exists an increasing sequence, $z_{i}$, such that $\displaystyle\lim_{i\rightarrow\infty}z_{i}=\bar{x}_{1}(X)$ and $g(z_{i})<X$. Using that $x_{1}$ is nondecreasing and \eqref{eq:LcompMx1g}, we get $z_{i}\leq x_{1}(X)$. Letting $i$ tend to infinity, we obtain $\bar{x}_{1}(X)\leq x_{1}(X)$. Assume that $\bar{x}_{1}(X)<x_{1}(X)$. Then, there exists $x\in\mathbb{R}$ such that $\bar{x}_{1}(X)<x<x_{1}(X)$ which implies, by \eqref{eq:LcompMx1barg}, that $g(x)\geq X$. On the other hand, $x_{1}(g(x))=x<x_{1}(X)$ implies that $g(x)<X$ because $x_{1}$ is nondecreasing, which gives us a contradiction. Hence, we must have $\bar{x}_{1}=x_{1}$. Then, by \eqref{eq:mapfromDtoF3} and since $x_{1}+J_{1}=\id$, we get
\begin{equation*}
	\bar{J}_{1}(X)=X-\bar{x}_{1}(X)=X-x_{1}(X)=J_{1}(X)
\end{equation*}
and from \eqref{eq:mapfromDtoF4} and \eqref{eq:mapFtoD1}, we obtain
\begin{equation*}
	\bar{U}_{1}(X)=u(\bar{x}_{1}(X))=u(x_{1}(X))=U_{1}(X).
\end{equation*}
By \eqref{eq:mapfromDtoF5} and \eqref{eq:mapFtoD9}, we have
\begin{equation*}
	\bar{V}_{1}(X)=\bar{x}_{1}'(X)\frac{R(\bar{x}_{1}(X))}{2c(\bar{U}_{1}(X))}=x_{1}'(X)\frac{R(x_{1}(X))}{2c(U_{1}(X))}=V_{1}(X)
\end{equation*}
and by \eqref{eq:mapfromDtoF6} and \eqref{eq:setFrel2}, we get
\begin{equation*}
	\bar{K}_{1}(X)=\int_{-\infty}^{X}\frac{\bar{J}_{1}'(\bar{X})}{c(\bar{U}_{1}(\bar{X}))}\,d\bar{X}=\int_{-\infty}^{X}\frac{J_{1}'(\bar{X})}{c(U_{1}(\bar{X}))}\,d\bar{X}=K_{1}(X).
\end{equation*}
Using \eqref{eq:mapfromDtoF7} and \eqref{eq:mapFtoD11}, we find that
\begin{equation*}
	\bar{H}_{1}(X)=\frac{1}{2}\rho(\bar{x}_{1}(X))\bar{x}_{1}'(X)=\frac{1}{2}\rho(x_{1}(X))x_{1}'(X)=H_{1}(X). 
\end{equation*}
Hence, we have proved that $\bar{\psi}_{1}=\psi_{1}$. Similarly, one proves that $\bar{x}_{2}=x_{2}$ and $\bar{\psi}_{2}=\psi_{2}$. For example, by \eqref{eq:mapfromDtoF7} and \eqref{eq:mapFtoD12}, we have
\begin{equation*}
	\bar{H}_{2}(Y)=\frac{1}{2}\sigma(\bar{x}_{2}(Y))\bar{x}_{2}'(Y)=\frac{1}{2}\sigma(x_{2}(Y))x_{2}'(Y)=H_{2}(Y).
\end{equation*} 

\textbf{Step 4.}
Let us prove that $\mathbf{M}\circ\mathbf{L}=\id$. Given $(u,R,S,\rho,\sigma,\mu,\nu)\in\D$, let $\psi=(\psi_{1},\psi_{2})=\mathbf{L}(u,R,S,\rho,\sigma,\mu,\nu)$ and $(\bar{u},\bar{R},\bar{S},\bar{\rho},\bar{\sigma},\bar{\mu},\bar{\nu})=\mathbf{M}(\psi)$. We want to prove that $(\bar{u},\bar{R},\bar{S},\bar{\rho},\bar{\sigma},\bar{\mu},\bar{\nu})=(u,R,S,\rho,\sigma,\mu,\nu)$. First we show that $\bar{\mu}=\mu$. Let $g$ be the function defined as before by \eqref{eq:LcompMg}. The same computation that leads to \eqref{eq:LcompMmug} now gives
\begin{equation}
\label{eq:LcompMmubarg}
	x+\bar{\mu}((-\infty,x))=g(x).
\end{equation}
By \eqref{eq:mapfromDtoF1}, for any $X\in\mathbb{R}$, there exists an increasing sequence, $x_{i}$, such that $\displaystyle\lim_{i\rightarrow\infty}x_{i}=x_{1}(X)$ and $x_i+\mu((-\infty,x_{i}))<X$. Sending $i$ to infinity, and since $x\mapsto\mu((-\infty,x))$ is lower semi-continuous, we obtain $x_1(X)+\mu((-\infty,x_{1}(X))))\leq X$. We set $X=g(x)$ and get, by \eqref{eq:LcompMx1g}, that
\begin{equation}
\label{eq:LcompMmuleqg}
	x+\mu((-\infty,x))\leq g(x).
\end{equation}
From the definition of $g$, we have that, for any $x\in\mathbb{R}$, there exists an increasing sequence, $X_{i}$, such that $\displaystyle\lim_{i\rightarrow\infty}X_{i}=g(x)$ and $x_{1}(X_{i})<x$. This implies, by \eqref{eq:mapfromDtoF1}, that $x+\mu((-\infty,x))\geq X_{i}$. Letting $i$ tend to infinity, we obtain
$x+\mu((-\infty,x))\geq g(x)$ which, together with \eqref{eq:LcompMmuleqg}, yields
\begin{equation}
\label{eq:LcompMmug2}
	x+\mu((-\infty,x))=g(x).
\end{equation}
Comparing \eqref{eq:LcompMmubarg} and \eqref{eq:LcompMmug2}, we get that $\bar{\mu}=\mu$. Similarly, one proves that $\bar{\nu}=\nu$. By \eqref{eq:mapFtoD1} and \eqref{eq:mapfromDtoF4}, we have
\begin{equation*}
	\bar{u}(x_{1}(X))=U_{1}(X)=u(x_{1}(X)).
\end{equation*}
For any Borel set $B$, we have
\begin{align*}
	\int_{B}\bar{R}(x)\,dx&=\int_{x_{1}^{-1}(B)}2c(U_{1}(X))V_{1}(X)\,dX \quad \text{by } \eqref{eq:mapFtoD3}\\
	&=\int_{x_{1}^{-1}(B)}2c(u\circ x_{1}(X))x_{1}'(X)\frac{R(x_{1}(X))}{2c(u\circ x_{1}(X))}\,dX \quad \text{by } \eqref{eq:mapfromDtoF4} \text{ and } \eqref{eq:mapfromDtoF5}\\
	&=\int_{B}R(x)\,dx \quad \text{by a change of variables},
\end{align*}
so that $\bar{R}=R$ almost everywhere. Similarly, one proves that $\bar{S}=S$ almost everywhere. 
From \eqref{eq:mapFtoD5} and \eqref{eq:mapfromDtoF7}, we get
\begin{equation*}
	\int_{B}\bar{\rho}(x)\,dx=\int_{x_{1}^{-1}(B)}2H_{1}(X)\,dX=\int_{x_{1}^{-1}(B)}\rho(x_{1}(X))x_{1}'(X)\,dX=\int_{B}\rho(x)\,dx.
\end{equation*}
Hence, $\bar{\rho}=\rho$ almost everywhere. Similarly, one proves that $\bar{\sigma}=\sigma$ almost everywhere. 

\textbf{Step 5.}
We prove \eqref{eq:quotproj4}. Given $\psi=(\psi_{1},\psi_{2})\in\F$, let $\phi=(f,g)\in G^{2}$ be defined as in \eqref{eq:projectionfandg} so that $\Pi(\psi)=\psi\cdot\phi^{-1}$. By \eqref{eq:equivar6}, we get
\begin{equation*}
	S_{T}\circ\Pi(\psi)=S_{T}(\psi\cdot\phi^{-1})=S_{T}(\psi)\cdot\phi^{-1},
\end{equation*}
which implies that $S_{T}\circ\Pi$ and $S_{T}$ are equivalent. Then, \eqref{eq:quotproj4} follows from
\eqref{eq:quotproj1}.
\end{proof}

Now we are finally in position to prove that $\bar{S}_{T}$ is a semigroup.

\begin{theorem}
\label{thm:STbarsemigroup}
The mapping $\bar{S}_{T}$ is a semigroup.
\end{theorem}

\begin{proof}
The proof relies on Lemma \ref{lemma:semigroupDtoDaux} and Theorem \ref{thm:STsemigroup}.
From Definition \ref{def:modsemigroup} we have
\begin{align*}
	\bar{S}_{T}\circ\bar{S}_{T'}&=\mathbf{M}\circ S_{T}\circ\mathbf{L}\circ\mathbf{M}\circ S_{T'}\circ\mathbf{L}\\
	&=\mathbf{M}\circ\Pi\circ S_{T}\circ\mathbf{L}\circ\mathbf{M}\circ\Pi\circ S_{T'}\circ\mathbf{L} \quad \text{by } \eqref{eq:quotproj2}\\
	&=\mathbf{M}\circ\Pi\circ S_{T}\circ\Pi\circ S_{T'}\circ\mathbf{L} \quad \text{by } \eqref{eq:quotproj5}\\
	&=\mathbf{M}\circ\Pi\circ S_{T}\circ S_{T'}\circ\mathbf{L} \quad \text{by } \eqref{eq:quotproj4}\\
	&=\mathbf{M}\circ S_{T}\circ S_{T'}\circ\mathbf{L} \quad \text{by } \eqref{eq:quotproj2}\\
	&=\mathbf{M}\circ S_{T+T'}\circ\mathbf{L} \quad \text{by Theorem } \ref{thm:STsemigroup}\\ 
	&=\bar{S}_{T+T'}.
\end{align*}	
\end{proof}


\section{Existence of Weak Global Conservative Solutions}
\label{sec:WeakSoln}

It remains to prove that the solution obtained by using the operator $\bar{S}_{T}$ is a weak solution of \eqref{eq:nvwsys}.

\begin{theorem}
Let $t>0$ and $(u_{0},R_{0},S_{0},\rho_{0},\sigma_{0},\mu_{0},\nu_{0})\in\D$. Then 
\begin{equation*}
(u,R,S,\rho,\sigma,\mu,\nu)(t)=\bar{S}_{t}(u_{0},R_{0},S_{0},\rho_{0},\sigma_{0},\mu_{0},\nu_{0})
\end{equation*}
is a weak solution of \eqref{eq:nvwsys}, meaning that
\begin{subequations}
\label{eq:Weak}
\begin{aalign}
	\label{eq:NVW1weakform}
	&\iint_{[0,\infty)\times\mathbb{R}}\bigg(\big[\phi_{t}-c(u)\phi_{x}\big]R+\big[\phi_{t}+c(u)\phi_{x}\big]S+\frac{c'(u)}{c(u)}RS\phi\bigg)\,dx\,dt\\
	&=\iint_{[0,\infty)\times\mathbb{R}}\frac{2c'(u)}{c(u)}\phi\,d\mu\,dt+\iint_{[0,\infty)\times\mathbb{R}}\frac{2c'(u)}{c(u)}\phi\,d\nu\,dt,
\end{aalign}
\begin{equation}
\label{eq:NVW2weakform}
	\iint_{[0,\infty)\times\mathbb{R}}\big[\phi_{t}-c(u)\phi_{x}\big]\rho\,dx\,dt=0,
\end{equation}
and
\begin{equation}
\label{eq:NVW3weakform}
	\iint_{[0,\infty)\times\mathbb{R}}\big[\phi_{t}+c(u)\phi_{x}\big]\sigma\,dx\,dt=0
\end{equation}
for all $\phi=\phi(t,x)$ in $C^{\infty}_{0}((0,\infty)\times\mathbb{R})$, where
\begin{equation}
\label{eq_NVW4weakform}
	R=u_{t}+c(u)u_{x} \quad \text{and} \quad S=u_{t}-c(u)u_{x}
\end{equation}
in the sense of distributions.
\end{subequations}

Moreover, the measures $\mu$ and $\nu$ satisfy the equations
\begin{subequations}
\label{eq:Weak2}	
\begin{equation}
\label{eq:mainthmmeaseqn1}
	(\mu+\nu)_{t}-(c(u)(\mu-\nu))_{x}=0
\end{equation}
and
\begin{equation}
\label{eq:mainthmmeaseqn2}
	\bigg(\frac{1}{c(u)}(\mu-\nu)\bigg)_{t}-(\mu+\nu)_{x}=0
\end{equation}
\end{subequations}
in the sense of distributions.
\end{theorem}

Note that if the two measures $\mu$ and $\nu$ are absolutely continuous,
\eqref{eq:mainthmmeaseqn1} and \eqref{eq:mainthmmeaseqn2} coincide with \eqref{eq:conslaw} in the sense of distributions, which we derived in the smooth case. Moreover, the difference of the sign in front of $\mu$ and $\nu$ indicates the two opposite traveling directions.

\begin{proof}
We decompose the proof into two steps.

\textbf{Step 1.}
We first show \eqref{eq:Weak}. Given $(u_{0},R_{0},S_{0},\rho_{0},\sigma_{0},\mu_{0},\nu_{0})\in\D$, we consider $(u,R,S,\rho,\sigma,\mu,\nu)(t)=\bar{S}_{t}(u_{0},R_{0},S_{0},\rho_{0},\sigma_{0},\mu_{0},\nu_{0})$, where $\bar{S}_{t}$ is given by Definition \ref{def:modsemigroup}. The identities in \eqref{eq_NVW4weakform} follow from \eqref{eq:lemmasemigprop6} in Lemma \ref{lemma:semigroupDtoDproperties}. By a change of variables, we get
\begin{aalign}
\label{eq:equivalentweakform1}
	&\iint_{[0,\infty)\times\mathbb{R}}(\phi_{t}-c(u)\phi_{x})R(t,x)\,dx\,dt\\
	&=\iint_{\mathbb{R}^{2}}(\phi_{t}-c(u)\phi_{x})R(t(X,Y),x(X,Y))(t_{X}x_{Y}-t_{Y}x_{X})(X,Y)\,dX\,dY\\
	&=2\iint_{\mathbb{R}^{2}}\bigg(\frac{\phi_{t}-c(u)\phi_{x}}{c(u)}R\bigg)(t(X,Y),x(X,Y))x_{X}x_{Y}(X,Y)\,dX\,dY \text{ by } \eqref{eq:setH1}\\
	&=-2\iint_{\mathbb{R}^{2}}\phi_{Y}R(t(X,Y),x(X,Y))x_{X}(X,Y)\,dX\,dY\\
	&=-2\iint_{\mathbb{R}^{2}}c(U)U_{X}(X,Y)\phi_{Y}(t(X,Y),x(X,Y))\,dX\,dY \text{ by } \eqref{eq:lemmasemigprop1}, \eqref{eq:lemmasemigprop2}\\
	&=2\iint_{\mathbb{R}^{2}}(c(U)U_{X})_{Y}(X,Y)\phi(t(X,Y),x(X,Y))\,dX\,dY\\
	&=2\iint_{\mathbb{R}^{2}}(c'(U)U_{X}U_{Y}+c(U)U_{XY})(X,Y)\phi(t(X,Y),x(X,Y))\,dX\,dY,
\end{aalign}
where we used integration by parts and 
\begin{aalign}
\label{eq:equivalentweakformphiY}
	&\phi_{Y}(t(X,Y),x(X,Y))\\
	&=\phi_{t}(t(X,Y),x(X,Y))t_{Y}(X,Y)+\phi_{x}(t(X,Y),x(X,Y))x_{Y}(X,Y)\\
	&=-\bigg(\frac{\phi_{t}-c(u)\phi_{x}}{c(u)}\bigg)(t(X,Y),x(X,Y))x_{Y}(X,Y),
\end{aalign}
which follows from \eqref{eq:setH1}. 

Similarly, we find that
\begin{aalign}
\label{eq:equivalentweakform2}
	&\iint_{[0,\infty)\times\mathbb{R}}(\phi_{t}+c(u)\phi_{x})S(t,x)\,dx\,dt\\
	&=2\iint_{\mathbb{R}^{2}}(c'(U)U_{X}U_{Y}+c(U)U_{XY})(X,Y)\phi(t(X,Y),x(X,Y))\,dX\,dY
\end{aalign}
and
\begin{aalign}
\label{eq:equivalentweakform22}
		&\iint_{[0,\infty)\times\mathbb{R}}\frac{c'(u)}{c(u)}RS\phi(t,x)\,dt\,dx\\
		&=-2\iint_{\mathbb{R}^{2}}c'(U)U_{X}U_{Y}(X,Y)\phi(t(X,Y),x(X,Y))\,dX\,dY.
\end{aalign}
Combining \eqref{eq:equivalentweakform1}, \eqref{eq:equivalentweakform2}, \eqref{eq:equivalentweakform22} and \eqref{eq:goveqU} yields
\begin{aalign}
\label{eq:LHS}
	&\iint_{[0,\infty)\times\mathbb{R}}\bigg(\big[\phi_{t}-c(u)\phi_{x}\big]R+\big[\phi_{t}+c(u)\phi_{x}\big]S+\frac{c'(u)}{c(u)}RS\phi\bigg)\,dx\,dt\\
	&=\iint_{\mathbb{R}^{2}}\frac{2c'(U)}{c^2(U)}(x_{Y}J_{X}+x_{X}J_{Y})\phi(t,x)\,dX\,dY.
\end{aalign}

 We have\footnote{Note that although the mapping $\mathbf{D}$ is from $\G_{0}$ to $\F$, the $t$ dependence in $\Theta$ makes sense, since, for any fixed $t\geq0$, we can consider the set $\G_{t}$, that is, the set $\G$ where $\Z_{1}(s)=t$. This still gives $\dot{\Z}_{1}(s)=0$, and $\G_{t}$ can be mapped to $\G_{0}$ by $\mathbf{t}_t$.} $(u,R,S,\rho,\sigma,\mu,\nu)(t)=\mathbf{M}\circ\mathbf{D}(\Theta(t))$ and
by \eqref{eq:mapGtoD7} 
\begin{align*}
	&\iint_{[0,\infty)\times\mathbb{R}}\frac{2c'(u)}{c(u)}\phi(t,x)\,d\mu(t)\,dt\\
	&=\iint_{[0,\infty)\times\mathbb{R}}\frac{2c'(u)}{c(u)}\phi(t,\Z_{2}(t,s))\V_{4}(t,\X(t,s))\X_{s}(t,s)\,ds\,dt,
\end{align*}
where we added the $t$ dependence in $\Theta(t)$, which gives $\X(t,s)$, $\Z_{2}(t,s)$ and $\V_{4}(t,s)$ in the equation above. The measures $\mu$ and $\nu$ integrate with respect to the $x$ variable and the notation $d\mu(t)$ and $d\nu(t)$ means that they depend on $t$. We proceed to the change of variables $s=\frac{1}{2}(X+Y)$ and $t=t(X,Y)$ and obtain
\begin{align*}
	&\iint_{[0,\infty)\times\mathbb{R}}\frac{2c'(u)}{c(u)}\phi(t,\Z_{2}(t,s))\V_{4}(t,\X(t,s))\X_{s}(t,s)\,ds\,dt\\	
	&=\iint_{\mathbb{R}^{2}}\frac{2c'(u)}{c(u)}\phi(t(X,Y),x(X,Y))J_{X}(X,Y)\\
	&\hspace{42pt}\times\X_{s}(t(X,Y),s(X,Y))\bigg(\frac{t_{X}-t_{Y}}{2}\bigg)(X,Y)\,dX\,dY\\
	&=\iint_{\mathbb{R}^{2}}\frac{2c'(U)}{c^{2}(U)}x_{Y}J_{X}(X,Y)\phi(t(X,Y),x(X,Y))\,dX\,dY \quad \text{by } \eqref{eq:setH1},
\end{align*}
where we used that $\dot{\Z}_1(t,s)=0$, which implies
\begin{aalign}
\label{eq:tId}
	0&=t_{X}(X,Y)\X_{s}(t(X,Y),s(X,Y))+t_{Y}(X,Y)\Y_{s}(t(X,Y),s(X,Y))\\
	&=(t_{X}-t_{Y})(X,Y)\X_{s}(t(X,Y),s(X,Y))-2\bigg(\frac{x_{Y}}{c(U)}\bigg)(X,Y),
\end{aalign}
by \eqref{eq:initialcurvenormalization} and \eqref{eq:setH1}. 

Similarly, we obtain
\begin{align*}
	&\iint_{[0,\infty)\times\mathbb{R}}\frac{2c'(u)}{c(u)}\phi(t,x)\,d\nu(t)\,dt\\
	&=\iint_{\mathbb{R}^{2}}\frac{2c'(U)}{c^{2}(U)}x_{X}J_{Y}(X,Y)\phi(t(X,Y),x(X,Y))\,dX\,dY,
\end{align*}
and we get
\begin{aalign}
\label{eq:RHS}
	&\iint_{[0,\infty)\times\mathbb{R}}\frac{2c'(u)}{c(u)}\phi(t,x)\,d\mu(t)\,dt+\iint_{[0,\infty)\times\mathbb{R}}\frac{2c'(u)}{c(u)}\phi(t,x)\,d\nu(t)\,dt\\
	&=\iint_{\mathbb{R}^{2}}\frac{2c'(U)}{c^{2}(U)}(x_{Y}J_{X}+x_{X}J_{Y})\phi(t,x)\,dX\,dY.
\end{aalign}
From \eqref{eq:LHS} and \eqref{eq:RHS} we conclude that \eqref{eq:NVW1weakform} holds.

It remains to prove \eqref{eq:NVW2weakform} and \eqref{eq:NVW3weakform}. We have
\begin{align*}
	&\iint_{[0,\infty)\times\mathbb{R}}(\phi_{t}-c(u)\phi_{x})\rho(t,x)\,dx\,dt\\
	&=2\iint_{\mathbb{R}^{2}}\bigg(\frac{\phi_{t}-c(u)\phi_{x}}{c(u)}\rho\bigg)(t(X,Y),x(X,Y))x_{X}x_{Y}(X,Y)\,dX\,dY\\
	&=-2\iint_{\mathbb{R}^{2}}\phi_{Y}\rho(t(X,Y),x(X,Y))x_{X}(X,Y)\,dX\,dY \quad \text{by } \eqref{eq:equivalentweakformphiY}\\
	&=-2\iint_{\mathbb{R}^{2}}p(X,Y)\phi_{Y}(t(X,Y),x(X,Y))\,dX\,dY \quad \text{by } \eqref{eq:lemmasemigprop1} \text{ and } \eqref{eq:lemmasemigprop3}\\
	&=2\iint_{\mathbb{R}^{2}}p_{Y}(X,Y)\phi(t(X,Y),x(X,Y))\,dX\,dY \quad \text{by integration by parts}\\
	&=0 \qquad \text{by } \eqref{eq:goveqp}
\end{align*} 
and
\begin{align*}
	&\iint_{[0,\infty)\times\mathbb{R}}(\phi_{t}+c(u)\phi_{x})\sigma(t,x)\,dx\,dt\\
	&=2\iint_{\mathbb{R}^{2}}\bigg(\frac{\phi_{t}+c(u)\phi_{x}}{c(u)}\sigma\bigg)(t(X,Y),x(X,Y))x_{X}x_{Y}(X,Y)\,dX\,dY\\
	&=2\iint_{\mathbb{R}^{2}}\phi_{X}\sigma(t(X,Y),x(X,Y))x_{Y}(X,Y)\,dX\,dY\\
	&=2\iint_{\mathbb{R}^{2}}q(X,Y)\phi_{X}(t(X,Y),x(X,Y))\,dX\,dY \quad \text{by } \eqref{eq:lemmasemigprop1} \text{ and } \eqref{eq:lemmasemigprop5}\\
	&=-2\iint_{\mathbb{R}^{2}}q_{X}(X,Y)\phi(t(X,Y),x(X,Y))\,dX\,dY \quad \text{by integration by parts}\\
	&=0 \qquad \text{by } \eqref{eq:goveqq}
\end{align*}
because
\begin{equation*}
	\phi_{X}(t(X,Y),x(X,Y))=\bigg(\frac{\phi_{t}+c(u)\phi_{x}}{c(u)}\bigg)(t(X,Y),x(X,Y))x_{X}(X,Y),
\end{equation*} 
which follows from a similar calculation as in \eqref{eq:equivalentweakformphiY}. Thus, we have proved \eqref{eq:NVW2weakform} and \eqref{eq:NVW3weakform}.

\textbf{Step 2.}
Now we prove \eqref{eq:Weak2}. First we show \eqref{eq:mainthmmeaseqn1}, that is,
\begin{equation*}
	\iint_{[0,\infty)\times\mathbb{R}}(\phi_{t}-c(u)\phi_{x})(t,x)\,d\mu(t)\,dt	+\iint_{[0,\infty)\times\mathbb{R}}(\phi_{t}+c(u)\phi_{x})(t,x)\,d\nu(t)\,dt=0
\end{equation*}
for all $\phi\in C^{\infty}_{0}((0,\infty)\times\mathbb{R})$. By a calculation as above we find
\begin{align*}
	&\iint_{[0,\infty)\times\mathbb{R}}(\phi_{t}-c(u)\phi_{x})(t,x)\,d\mu(t)\,dt\\
	&=\iint_{[0,\infty)\times\mathbb{R}}(\phi_{t}-c(u)\phi_{x})(t,\Z_{2}(t,s))\V_{4}(t,\X(t,s))\X_{s}(t,s)\,ds\,dt \quad \text{by } \eqref{eq:mapGtoD7}\\	
	&=\iint_{\mathbb{R}^{2}}(\phi_{t}-c(u)\phi_{x})(t(X,Y),x(X,Y))J_{X}(X,Y)\\
	&\hspace{42pt}\times\X_{s}(t(X,Y),s(X,Y))\bigg(\frac{t_{X}-t_{Y}}{2}\bigg)(X,Y)\,dX\,dY\\
	&=\iint_{\mathbb{R}^{2}}(\phi_{t}-c(u)\phi_{x})(t(X,Y),x(X,Y))\\
	&\hspace{42pt}\times J_{X}(X,Y)\bigg(\frac{x_{Y}}{c(U)}\bigg)(X,Y)\,dX\,dY \quad \text{by } \eqref{eq:tId} \text{ and } \eqref{eq:setH1}\\
	&=-\iint_{\mathbb{R}^{2}}\phi_{Y}(t(X,Y),x(X,Y))J_{X}(X,Y)\,dX\,dY \quad \text{by } \eqref{eq:equivalentweakformphiY},
\end{align*}
where we used the change of variables $s=\frac{1}{2}(X+Y)$ and $t=t(X,Y)$.

Similarly, one proves that
\begin{equation*}
	\iint_{[0,\infty)\times\mathbb{R}}(\phi_{t}+c(u)\phi_{x})(t,x)\,d\nu(t)\,dt=\iint_{\mathbb{R}^{2}}\phi_{X}(t(X,Y),x(X,Y))J_{Y}(X,Y)\,dX\,dY,
\end{equation*}
so that
\begin{align*}
	&\iint_{[0,\infty)\times\mathbb{R}}(\phi_{t}-c(u)\phi_{x})(t,x)\,d\mu(t)\,dt	+\int_{[0,\infty)\times\mathbb{R}}(\phi_{t}+c(u)\phi_{x})(t,x)\,d\nu(t)\,dt\\
	&=\iint_{\mathbb{R}^{2}}(-\phi_{Y}(t(X,Y),x(X,Y))J_{X}(X,Y)+\phi_{X}(t(X,Y),x(X,Y))J_{Y}(X,Y))\,dX\,dY\\
	&=0
\end{align*}
by integration by parts. This concludes the proof of \eqref{eq:mainthmmeaseqn1}. In a similar way, one proves \eqref{eq:mainthmmeaseqn2}.
\end{proof}

The semigroup of solutions, $\bar{S}_{t}$, is conservative in the following sense.

\begin{theorem}
\label{thm:cons}
Given $(u_{0},R_{0},S_{0},\rho_{0},\sigma_{0},\mu_{0},\nu_{0})\in\D$, let  
\begin{equation*}
	(u,R,S,\rho,\sigma,\mu,\nu)(t)=\bar{S}_{t}(u_{0},R_{0},S_{0},\rho_{0},\sigma_{0},\mu_{0},\nu_{0}).
\end{equation*}
We have:
\begin{enumerate}
\item[(i)] For all $t\geq 0$,
\begin{equation*}
	\mu(t)(\mathbb{R})+\nu(t)(\mathbb{R})=\mu_{0}(\mathbb{R})+\nu_{0}(\mathbb{R}).
\end{equation*}
\item[(ii)] For almost every $t\geq 0$, the singular parts of $\mu(t)$ and $\nu(t)$ are concentrated on the set where $c'(u)=0$.
\end{enumerate}
\end{theorem}

\begin{proof}
We prove (i). Given $\tau\geq 0$, let  
\begin{equation*}
	(u,R,S,\rho,\sigma,\mu,\nu)(\tau)=\bar{S}_{\tau}(u_{0},R_{0},S_{0},\rho_{0},\sigma_{0},\mu_{0},\nu_{0}).
\end{equation*}
We consider $\Theta(\tau)\in\G_{\tau}$ and $(Z,p,q)\in\H$ such that $(u,R,S,\rho,\sigma,\mu,\nu)(\tau)=\mathbf{M}\circ\mathbf{D}(\Theta(\tau))$ and $\Theta(\tau)=\mathbf{E}(Z,p,q)$. By Definition \ref{def:mapE}, we have \\$\Z_{2}(\tau,s)=x(\X(\tau,s),\Y(\tau,s))$, $\V_{4}(\X(\tau,s))=J_{X}(\X(\tau,s),\Y(\tau,s))$ 
and $\W_{4}(\Y(\tau,s))=J_{Y}(\X(\tau,s),\Y(\tau,s))$. Then, from \eqref{eq:mapGtoD7} and \eqref{eq:mapGtoD8}, we have for any Borel set $B$, that
\begin{align*}
	\mu(\tau)(B)&=\int_{\{s\in\mathbb{R} \, | \, \Z_{2}(\tau,s)\in B\}}\V_{4}(\X(\tau,s))\X_{s}(\tau,s)\,ds\\
	&=\int_{\{s\in\mathbb{R} \, | \, x(\X(\tau,s),\Y(\tau,s))\in B\}}J_{X}(\X(\tau,s),\Y(\tau,s))\X_{s}(\tau,s)\,ds
\end{align*}
and
\begin{align*}
	\nu(\tau)(B)&=\int_{\{s\in\mathbb{R} \, | \, \Z_{2}(\tau,s)\in B\}}\W_{4}(\Y(\tau,s))\Y_{s}(\tau,s)\,ds\\
	&=\int_{\{s\in\mathbb{R} \, | \, x(\X(\tau,s),\Y(\tau,s))\in B\}}J_{Y}(\X(\tau,s),\Y(\tau,s))\Y_{s}(\tau,s)\,ds
\end{align*}
respectively. Hence, 
\begin{align*}
	&\mu(\tau)(\mathbb{R})+\nu(\tau)(\mathbb{R})\\
	&=\int_{\mathbb{R}}(J_{X}(\X(\tau,s),\Y(\tau,s))\X_{s}(\tau,s)+J_{Y}(\X(\tau,s),\Y(\tau,s))\Y_{s}(\tau,s))\,ds\\
	&=\int_{\mathbb{R}}J_{s}(\X(\tau,s),\Y(\tau,s))\,ds\\
	&=\displaystyle\lim_{s\rightarrow\infty}J(\X(\tau,s),\Y(\tau,s))\\
	&=\displaystyle\lim_{s\rightarrow\infty}J(\X(0,s),\Y(0,s)) \quad \text{by Lemma } \ref{lemma:curveind}\\
	&=\mu_{0}(\mathbb{R})+\nu_{0}(\mathbb{R}),
\end{align*}
where we used that $\displaystyle\lim_{s\rightarrow-\infty}J(\X(\tau,s),\Y(\tau,s))=\displaystyle\lim_{s\rightarrow-\infty}\Z_{4}(\tau,s)=0$.

Let us prove (ii). We decompose $\mu(\tau)$ into its absolutely continuous and singular part, that is, $\mu(\tau)=\mu(\tau)_{\text{ac}}+\mu(\tau)_{\text{sing}}$. We want to prove that, for almost every time $\tau\geq 0$,
\begin{equation*}
	\mu(\tau)_{\text{sing}}(\{x\in\mathbb{R} \ | \ c'(u(\tau,x))\neq 0\})=0.
\end{equation*}
Consider the set
\begin{equation*}
	A_{\tau}=\{s\in\mathbb{R} \ | \ x_{X}(\X(\tau,s),\Y(\tau,s))>0 \}.
\end{equation*}
Since $x_{X}(\X(\tau,s),\Y(\tau,s))=\V_{2}(\X(\tau,s))$, $A_{\tau}$ corresponds to the set $A$ in \eqref{eq:setAandB} in the proof of Lemma \ref{lemma:mapG0toD}. Using $\V_{4}(\X(\tau,s))=J_{X}(\X(\tau,s),\Y(\tau,s))$ and $\Z_{2}(\tau,s)=x(\X(\tau,s),\Y(\tau,s))$ in \eqref{eq:mapGtoDsing1}, we get
\begin{aalign}
\label{eq:iimusing}
	\mu_{\text{sing}}(\tau)(B)&=\int_{\{s\in\mathbb{R} \, | \, \Z_{2}(\tau,s)\in B\}\cap A_{\tau}^{c}}\V_{4}(\X(\tau,s))\X_{s}(\tau,s)\,ds\\
	&=\int_{\{s\in A_{\tau}^{c} \, | \, x(\X(\tau,s),\Y(\tau,s))\in B\}}J_{X}(\X(\tau,s),\Y(\tau,s))\X_{s}(\tau,s)\,ds
\end{aalign}
for any Borel set $B$. Let
\begin{equation*}
	E=\{(X,Y)\in\mathbb{R}^{2} \ | \ x_{X}(X,Y)=0 \text{ and } c'(U(X,Y))\neq 0\}.
\end{equation*}
For a given time $\tau$, we consider the mapping $\Gamma_{\tau}:s\mapsto(\X(\tau,s),\Y(\tau,s))$. From \eqref{eq:iimusing}, we obtain
\begin{aalign}
\label{eq:muSingcDer}
	&\mu(\tau)_{\text{sing}}(\{x\in\mathbb{R} \ | \ c'(u(\tau,x))\neq 0\})\\
	&=\int_{\{s\in\mathbb{R} \, | \, (\X(\tau,s),\Y(\tau,s))\in E\}}J_{X}(\X(\tau,s),\Y(\tau,s))\X_{s}(\tau,s)\,ds\\
	&=\int_{\Gamma_{\tau}^{-1}(E)}J_{X}(\X(\tau,s),\Y(\tau,s))\X_{s}(\tau,s)\,ds\\
	&\leq 2||J_{X}||_{W^{1,\infty}_{Y}(\mathbb{R})}\text{meas}(\Gamma_{\tau}^{-1}(E)).
\end{aalign}
We claim that $\text{meas}(\Gamma_{\tau}^{-1}(E))=0$. By the area formula, see Section 3.3 in \cite{EvaGar}, we obtain
\begin{aalign}
\label{eq:measGammaTau}
	\text{meas}(\Gamma_{\tau}^{-1}(E))&=\int_{\Gamma_{\tau}^{-1}(E)}\frac{1}{2}(\X_{s}(\tau,s)+\Y_{s}(\tau,s))\,ds \quad \text{since } \X_{s}+\Y_{s}=2\\
	&\leq\int_{\Gamma_{\tau}^{-1}(E)}(\X_{s}^{2}(\tau,s)+\Y_{s}^{2}(\tau,s))^{\frac{1}{2}}\,ds\\
	&=\H^{1}(\Gamma_{\tau}\circ\Gamma_{\tau}^{-1}(E))\\
	&\leq\H^{1}(E\cap t^{-1}(\tau)),
\end{aalign}
where we used that $\X_{s}=(\X_{s}^{2})^{\frac{1}{2}}\leq(\X_{s}^{2}+\Y_{s}^{2})^{\frac{1}{2}}$ (and similarly for $\Y_{s}$). Here, $\H^{1}$ denotes the one-dimensional Hausdorff measure. The fact that 
$\Gamma_{\tau}\circ\Gamma_{\tau}^{-1}(E)\subset E\cap t^{-1}(\tau)$ follows from the following argument. If $E\cap t^{-1}(\tau)$ contains a rectangle, we have by the Definition \ref{def:mapE}, that the curve $\Gamma_{\tau}$ consists of the left vertical side and the upper horizontal side of the rectangle. We show that $\H^{1}(E\cap t^{-1}(\tau))=0$. We have
\begin{equation*}
	E=A_{1}\cup A_{3},
\end{equation*}
where
\begin{equation*}
	A_{1}=\{(X,Y)\in\mathbb{R}^{2} \ | \ x_{X}(X,Y)=0,\ x_{Y}(X,Y)>0 \text{ and } c'(U(X,Y))\neq 0\}
\end{equation*}
and
\begin{equation*}
	A_{3}=\{(X,Y)\in\mathbb{R}^{2} \ | \ x_{X}(X,Y)=0, \ x_{Y}(X,Y)=0 \text{ and } c'(U(X,Y))\neq 0\}.
\end{equation*}

By an argument as in the proof of Theorem 4 in \cite{HolRay:11}, we obtain $\text{meas}(A_{1})=0$. Hence, by the coarea formula, see Section 3.4 in \cite{EvaGar}, we get
\begin{align*}
	\int_{\mathbb{R}}\H^{1}(E\cap t^{-1}(\tau))\,d\tau
	&=\iint_{E}(t_{X}^{2}(X,Y)+t_{Y}^{2}(X,Y))^{\frac{1}{2}}\,dX\,dY\\
	&=\iint_{A_{3}}(t_{X}^{2}(X,Y)+t_{Y}^{2}(X,Y))^{\frac{1}{2}}\,dX\,dY\\
	&=\iint_{A_{3}}\frac{1}{c(U(X,Y))}(x_{X}^{2}(X,Y)+x_{Y}^{2}(X,Y))^{\frac{1}{2}}\,dX\,dY \quad \text{by } \eqref{eq:setH1}\\
	&=0.
\end{align*}
Therefore, we get from \eqref{eq:measGammaTau}, that $\text{meas}(\Gamma_{\tau}^{-1}(E))=0$, which inserted into \eqref{eq:muSingcDer} yields
\begin{equation*}
	\mu(\tau)_{\text{sing}}(\{x\in\mathbb{R} \ | \ c'(u(\tau,x))\neq 0\})=0.
\end{equation*}
\end{proof}

\begin{theorem}[Finite speed of propagation]
\label{thm:finiteSpeedOfProp}
For initial data $(u_{0},R_{0},S_{0},\rho_{0},\sigma_{0},\mu_{0},\nu_{0})$ and $(\bar{u}_{0},\bar{R}_{0},\bar{S}_{0},\bar{\rho}_{0},\bar{\sigma}_{0},\bar{\mu}_{0},\bar{\nu}_{0})$ in $\D$, we consider the solutions
\begin{equation*}
	(u,R,S,\rho,\sigma,\mu,\nu)(t)=\bar{S}_{t}(u_{0},R_{0},S_{0},\rho_{0},\sigma_{0},\mu_{0},\nu_{0})
\end{equation*}
and
\begin{equation*}
	(\bar{u},\bar{R},\bar{S},\bar{\rho},\bar{\sigma},\bar{\mu},\bar{\nu})(t)=\bar{S}_{t}(\bar{u}_{0},\bar{R}_{0},\bar{S}_{0},\bar{\rho}_{0},\bar{\sigma}_{0},\bar{\mu}_{0},\bar{\nu}_{0}).
\end{equation*}
Given $\mathbf{t}>0$ and $\mathbf{x}\in\mathbb{R}$, if
\begin{equation*}
	(u_{0},R_{0},S_{0},\rho_{0},\sigma_{0},\mu_{0},\nu_{0})(x)=(\bar{u}_{0},\bar{R}_{0},\bar{S}_{0},\bar{\rho}_{0},\bar{\sigma}_{0},\bar{\mu}_{0},\bar{\nu}_{0})(x)
\end{equation*} 
for $x\in[\mathbf{x}-\kappa\mathbf{t},\mathbf{x}+\kappa\mathbf{t}]$, then
\begin{equation*}
	u(\mathbf{t},\mathbf{x})=\bar{u}(\mathbf{t},\mathbf{x}).
\end{equation*}
\end{theorem}

In the case of the linear wave equation, i.e., $c$ is constant, one has $u(\mathbf{t},\mathbf{x})=\bar{u}(\mathbf{t},\mathbf{x})$ if the initial data are equal on the interval $[\mathbf{x}-c\mathbf{t},\mathbf{x}+c\mathbf{t}]$. If the function $c(u)$ satisfies $\frac{1}{\kappa}\leq c(u)\leq \kappa$ for some $\kappa\geq1$ the corresponding interval is contained in $[\mathbf{x}-\kappa\mathbf{t},\mathbf{x}+\kappa\mathbf{t}]$. Thus, we require the initial data to coincide on a slightly bigger interval.

\begin{proof}
We denote $x_{l}=\mathbf{x}-\kappa\mathbf{t}$ and $x_{r}=\mathbf{x}+\kappa\mathbf{t}$. For a given $(u_{0},R_{0},S_{0},\rho_{0},\sigma_{0},\mu_{0},\nu_{0})\in\D$, we define
\begin{equation*}
	(\bar{u}_{0},\bar{R}_{0},\bar{S}_{0},\bar{\rho}_{0},\bar{\sigma}_{0})(x)=
	\begin{cases}
	(u_{0},R_{0},S_{0},\rho_{0},\sigma_{0})(x) & \text{if } x\in[x_{l},x_{r}]\\
	(0,0,0,0,0) & \text{otherwise}
	\end{cases}
\end{equation*} 
and
\begin{equation*}
	\bar{\mu}_{0}(B)=\mu_{0}(B\cap[x_{l},x_{r}]), \quad
	\bar{\nu}_{0}(B)=\nu_{0}(B\cap[x_{l},x_{r}])
\end{equation*}
for any Borel set $B$. It is enough to prove the theorem for the initial data \\ $(u_{0},R_{0},S_{0},\rho_{0},\sigma_{0},\mu_{0},\nu_{0})$ and $(\bar{u}_{0},\bar{R}_{0},\bar{S}_{0},\bar{\rho}_{0},\bar{\sigma}_{0},\bar{\mu}_{0},\bar{\nu}_{0})$ in $\D$. We have to compute the solutions corresponding to these two initial data. We decompose the proof into five steps.

\textbf{Step 1.} Let $\psi=(\psi_{1},\psi_{2})=\mathbf{L}(u_{0},R_{0},S_{0},\rho_{0},\sigma_{0},\mu_{0},\nu_{0})$ and $\bar{\psi}=(\bar{\psi}_{1},\bar{\psi}_{2})=\mathbf{L}(\bar{u}_{0},\bar{R}_{0},\bar{S}_{0},\bar{\rho}_{0},\bar{\sigma}_{0},\bar{\mu}_{0},\bar{\nu}_{0})$. We denote $X_{l}=x_{l}$, $Y_{l}=x_{l}$, $X_{r}=x_{r}+\mu_{0}([x_{l},x_{r}])$, $Y_{r}=x_{r}+\nu_{0}([x_{l},x_{r}])$ and $\Omega=[X_{l},X_{r}]\times[Y_{l},Y_{r}]$. We claim that
\begin{equation}
\label{eq:x1barthreecases}
	\bar{x}_{1}(X)=\begin{cases}
	X & \text{if } X\leq X_{l},\\
	x_{1}(X+\mu_{0}((-\infty,x_{l}))) & \text{if } X_{l}<X\leq X_{r},\\
	X-\mu_{0}([x_{l},x_{r}]) & \text{if } X>X_{r}
	\end{cases}
\end{equation}
and
\begin{equation}
\label{eq:x2barthreecases}
	\bar{x}_{2}(Y)=\begin{cases}
	Y & \text{if } Y\leq Y_{l},\\
	x_{2}(Y+\nu_{0}((-\infty,x_{l}))) & \text{if } Y_{l}<Y\leq Y_{r},\\
	Y-\nu_{0}([x_{l},x_{r}]) & \text{if } Y>Y_{r}.
	\end{cases}
\end{equation}
From \eqref{eq:mapfromDtoF1}, we have
\begin{equation}
\label{eq:x1bardef}
	\bar{x}_{1}(X)=\sup \{x'\in \mathbb{R} \ | \ x'+\bar{\mu}_{0}((-\infty, x'))<X \}.
\end{equation}

First case: $X\leq X_{l}$. For any $x'$ such that $x'+\bar{\mu}_{0}((-\infty,x'))<X$, we have $x'<X$, so that $x'<X\leq X_{l}=x_{l}$. Then, $\bar{\mu}_{0}((-\infty,x'))=\mu_{0}((-\infty,x')\cap[x_{l},x_{r}])=0$ and $\bar{x}_{1}(X)=X$.

Second case: $X_{l}<X\leq X_{r}$. We have $x_{l}+\bar{\mu}_{0}((-\infty,x_{l}))=x_{l}=X_{l}<X$, so that
\begin{equation}
\label{eq:Step1SecondCase}
	\bar{x}_{1}(X)=\sup \{x'\in[x_{l},\infty) \ | \ x'+\mu_{0}([x_{l}, x'))<X \}.
\end{equation} 
We claim that for any $x'$ such that $x'+\bar{\mu}_{0}((-\infty,x'))<X$, we have $x'\leq x_{r}$. Assume the opposite, that is, $x'>x_{r}$. Then, $\bar{\mu}_{0}((-\infty,x'))=\mu_{0}((-\infty,x')\cap[x_{l},x_{r}])=\mu_{0}([x_{l},x_{r}])$ and we get
\begin{equation*}
	x'+\mu_{0}([x_{l},x_{r}])=x'+\bar{\mu}_{0}((-\infty,x'))<X\leq X_{r}=x_{r}+\mu_{0}([x_{l},x_{r}]),
\end{equation*} 
which contradicts the assumption $x'>x_{r}$. Hence,
\begin{equation}
\label{eq:Step1SecondCaseii}
	\bar{x}_{1}(X)=\sup \{x'\in[x_{l},x_{r}] \ | \ x'+\mu_{0}([x_{l}, x'))<X \}.
\end{equation}
We claim that
\begin{equation*}
	x_{1}(X+\mu_{0}((-\infty,x_{l})))=\sup \{x'\in[x_{l},x_{r}] \ | \ x'+\mu_{0}([x_{l}, x'))<X \}.
\end{equation*}
By \eqref{eq:mapfromDtoF1}, we have
\begin{equation*}
	x_{1}(X+\mu_{0}((-\infty,x_{l})))=\sup \{x'\in \mathbb{R} \ | \ x'+\mu_{0}((-\infty, x'))<X+\mu_{0}((-\infty,x_{l})) \}.
\end{equation*}
We have $x_{l}+\mu_{0}((-\infty,x_{l}))=X_{l}+\mu_{0}((-\infty,x_{l}))<X+\mu_{0}((-\infty,x_{l}))$, so that
\begin{align*}
	x_{1}(X+\mu_{0}((-\infty,x_{l})))&=\sup \{x'\in [x_{l},\infty) \ | \ x'+\mu_{0}((-\infty, x'))<X+\mu_{0}((-\infty,x_{l}))\}\\
	&=\sup \{x'\in [x_{l},\infty) \ | \ x'+\mu_{0}([x_{l},x'))<X \}\\
	&=\bar{x}_{1}(X) \quad \text{by } \eqref{eq:Step1SecondCase}.
\end{align*}
Then, by \eqref{eq:Step1SecondCaseii}, we get
\begin{equation*}
	x_{1}(X+\mu_{0}((-\infty,x_{l})))=\sup \{x'\in[x_{l},x_{r}] \ | \ x'+\mu_{0}([x_{l}, x'))<X \}.
\end{equation*}

Third case: $X>X_{r}$. Since
$x_{r}+\bar{\mu}_{0}((-\infty,x_{r}))=x_{r}+\mu_{0}([x_{l},x_{r}))\leq x_{r}+\mu_{0}([x_{l},x_{r}])=X_{r}<X$, we have
\begin{equation*}
	\bar{x}_{1}(X)=\sup \{x'\in[x_{r},\infty) \ | \ x'+\bar{\mu}_{0}((-\infty,x'))<X \}.
\end{equation*}
If $x'>x_{r}$, $\bar{\mu}_{0}((-\infty,x'))=\mu_{0}([x_{l},x_{r}])$, which implies that
\begin{equation*}
	\bar{x}_{1}(X)=X-\mu_{0}([x_{l},x_{r}]).
\end{equation*}
This concludes the proof of \eqref{eq:x1barthreecases}. Similarly, one proves \eqref{eq:x2barthreecases}.

Let $f(X)=X+\mu_{0}((-\infty,x_{l}))$ and $g(Y)=Y+\nu_{0}((-\infty,x_{l}))$. We claim that $\phi=(f,g)\in G^{2}$. Since $f'=1$, $f$ is invertible and $f^{-1}(X)=X-\mu_{0}((-\infty,x_{l}))$. We have $||f-\id||_{\Linf(\mathbb{R})}\leq \mu_{0}(\mathbb{R})$, $||f^{-1}-\id||_{\Linf(\mathbb{R})}\leq\mu_{0}(\mathbb{R})$, $f'-1=0$ and $(f^{-1})'-1=0$. Hence, $f$ belongs to $G$. Similarly, one shows that $g\in G$. We denote $\tilde{\psi}=\psi\cdot\phi$. We have proved that
\begin{equation}
\label{eq:x1barandtilde}
	\bar{x}_{1}(X)=\tilde{x}_{1}(X) \quad \text{for } X_{l}<X\leq X_{r}
\end{equation}
and
\begin{equation}
\label{eq:x2barandtilde}
	\bar{x}_{2}(Y)=\tilde{x}_{2}(Y) \quad \text{for } Y_{l}<Y\leq Y_{r}.
\end{equation}

\textbf{Step 2.} Let $\bar{\Theta}=\mathbf{C}(\bar{\psi})$ and $\tilde{\Theta}=\mathbf{C}(\tilde{\psi})$. We prove that
\begin{equation}
\label{eq:step2XandYbarandtilde}
	\bar{\X}(s)=\tilde{\X}(s) \quad \text{and} \quad \bar{\Y}(s)=\tilde{\Y}(s)
\end{equation}
for $s\in[s_{l},s_{r}]$ where $s_{l}=\frac{1}{2}(X_{l}+Y_{l})$ and $s_{r}=\frac{1}{2}(X_{r}+Y_{r})$.
First we show that
\begin{equation}
\label{eq:step2XandYtilde}
	\tilde{\X}(s_{l})=X_{l} \quad \text{and} \quad \tilde{\Y}(s_{l})=Y_{l}.
\end{equation}
By \eqref{eq:mapFtoGX}, we have
\begin{equation*}
	\tilde{\X}(s)=\sup\{X\in \mathbb{R} \ | \ \tilde{x}_{1}(X')<\tilde{x}_{2}(2s-X') \text{ for all } X'<X \}.
\end{equation*}
For any $X<X_{l}$, $\tilde{x}_{1}(X)\leq\tilde{x}_{1}(X_{l})$ because $\tilde{x}_{1}$ is nondecreasing. We claim that $\tilde{x}_{1}(X)<\tilde{x}_{1}(X_{l})$. Assume the opposite, that is, $\tilde{x}_{1}(X)=\tilde{x}_{1}(X_{l})$. Then, since $\tilde{x}_{1}(X_{l})=x_{l}$, there exists an increasing sequence $x_{i}$ such that $\displaystyle\lim_{i\rightarrow\infty}x_{i}=x_{l}$ and $x_{i}+\mu_{0}((-\infty,x_{i}))<X+\mu_{0}((-\infty,x_{l}))$. By sending $i$ to infinity, we get, since $x\mapsto\mu_{0}((-\infty,x))$ is lower semi-continuous, that $x_{l}+\mu_{0}((-\infty,x_{l}))\leq X+\mu_{0}((-\infty,x_{l}))$. This leads to the contradiction $x_{l}\leq X<X_{l}=x_{l}$. Therefore, for any $X<X_{l}$,
\begin{align*}
	\tilde{x}_{1}(X)&<\tilde{x}_{1}(X_{l})=x_{1}(x_{l}+\mu_{0}((-\infty,x_{l})))=x_{l}\\
	&=x_{2}(x_{l}+\nu_{0}((-\infty,x_{l})))=\tilde{x}_{2}(Y_{l})\leq\tilde{x}_{2}(2s_{l}-X),
\end{align*}
where we used that $\tilde{\psi}=\psi\cdot\phi$ and the fact that $\tilde{x}_{2}$ is nondecreasing. Hence, $\tilde{\X}(s_{l})=X_{l}$ and we have proved \eqref{eq:step2XandYtilde}. By using similar arguments, one obtains
\begin{equation}
\label{eq:Step2tildeXandY}
	\tilde{\X}(s_{r})=X_{r} \quad \text{and} \quad \tilde{\Y}(s_{r})=Y_{r}.
\end{equation}
The corresponding results for $\bar{\X}$ and $\bar{\Y}$, which we state here for completeness, are
\begin{equation*}
	\bar{\X}(s_{l})=X_{l}, \quad \bar{\Y}(s_{l})=Y_{l} \quad \text{and} \quad \bar{\X}(s_{r})=X_{r}, \quad \bar{\Y}(s_{r})=Y_{r}.
\end{equation*}
In particular, we have proved \eqref{eq:step2XandYbarandtilde} for $s=s_{l}$ and $s=s_{r}$. For $s\in(s_{l},s_{r})$, we have either $X_{l}<\bar{\X}(s)\leq X_{r}$ or $Y_{l}\leq\bar{\Y}(s)<Y_{r}$
by the definition of $\bar{\X}$ and $\bar{\Y}$. We only consider the case $X_{l}<\bar{\X}(s)\leq X_{r}$, as the other case can be treated similarly. By \eqref{eq:mapFtoGX}, there exists an increasing sequence $X_{i}$ such that $\displaystyle\lim_{i\rightarrow\infty}X_{i}=\bar{\X}(s)$ and $\bar{x}_{1}(X_{i})<\bar{x}_{2}(2s-X_{i})$. For sufficiently large $i$, we have $X_{l}<X_{i}\leq X_{r}$ and, by \eqref{eq:x1barandtilde}, we get $\tilde{x}_{1}(X_{i})=\bar{x}_{1}(X_{i})<\bar{x}_{2}(2s-X_{i})$. If $2s-X_{i}\leq Y_{r}$ then, by 
\eqref{eq:x2barandtilde}, $\bar{x}_{2}(2s-X_{i})=\tilde{x}_{2}(2s-X_{i})$, so that $\tilde{x}_{1}(X_{i})<\tilde{x}_{2}(2s-X_{i})$. If $2s-X_{i}>Y_{r}$ then also $\tilde{x}_{1}(X_{i})<\tilde{x}_{2}(2s-X_{i})$. Assume the opposite, that is, $\tilde{x}_{1}(X_{i})\geq\tilde{x}_{2}(2s-X_{i})$. Since $\tilde{x}_{1}$ and $\tilde{x}_{2}$ are nondecreasing, we have $\tilde{x}_{1}(X_{r})\geq \tilde{x}_{1}(X_{i})\geq \tilde{x}_{2}(2s-X_{i})\geq \tilde{x}_{2}(Y_{r})$. We have
\begin{align*}
	\tilde{x}_{1}(X_{r})&=x_{1}(X_{r}+\mu_{0}((-\infty,x_{l})))=x_{1}(x_{r}+\mu_{0}((-\infty,x_{r}]))\\
	&=x_{1}(x_{r}+\mu_{0}((-\infty,x_{r})))=x_{r}
\end{align*}
and similarly, we obtain $\tilde{x}_{2}(Y_{r})=x_{r}$. Thus, $\tilde{x}_{1}(X_{r})=\tilde{x}_{2}(Y_{r})$, which implies that $\tilde{x}_{2}(2s-X_{i})=x_{r}$. By 
\eqref{eq:mapfromDtoF2}, there exists a decreasing sequence $x_{j}$ such that $\displaystyle\lim_{j\rightarrow\infty}x_{j}=\tilde{x}_{2}(2s-X_{i})$ and \newline  $x_{j}+\nu_{0}((-\infty,x_{j}))\geq 2s-X_{i}+\nu_{0}((-\infty,x_{l}))$. Sending $j$ to infinity, we get
\begin{equation*}
	2s-X_{i}+\nu_{0}((-\infty,x_{l}))\leq \tilde{x}_{2}(2s-X_{i})+\nu_{0}((-\infty,\tilde{x}_{2}(2s-X_{i})])=x_{r}+\nu_{0}((-\infty,x_{r}]),
\end{equation*}
which leads to the contradiction
\begin{equation*}
	2s-X_{i}\leq x_{r}+\nu_{0}([x_{l},x_{r}])=Y_{r}.
\end{equation*}
Hence, we have shown that 
\begin{equation}
\label{eq:step2finalrelation1}
	\tilde{x}_{1}(X_{i})<\tilde{x}_{2}(2s-X_{i}).
\end{equation}
If $Y_{l}<\bar{\Y}(s)<Y_{r}$, we get, from \eqref{eq:x1barandtilde} and \eqref{eq:x2barandtilde}, that
\begin{equation}
\label{eq:step2finalrelation2}
	\tilde{x}_{1}(\bar{\X}(s))=\bar{x}_{1}(\bar{\X}(s))=\bar{x}_{2}(\bar{\Y}(s))=\tilde{x}_{2}(\bar{\Y}(s)).
\end{equation}
If $\bar{\Y}(s)=Y_{l}$, we have $\bar{x}_{2}(\bar{\Y}(s))=x_{l}=\tilde{x}_{2}(\bar{\Y}(s))$, so that
\eqref{eq:step2finalrelation2} also holds. It then follows from \eqref{eq:step2finalrelation1} and \eqref{eq:step2finalrelation2} that $\bar{\X}(s)=\tilde{\X}(s)$, which concludes the proof of \eqref{eq:step2XandYbarandtilde}.

\textbf{Step 3.} Let $(\bar{Z},\bar{p},\bar{q})=\mathbf{S}(\bar{\Theta})$ and $(\tilde{Z},\tilde{p},\tilde{q})=\mathbf{S}(\tilde{\Theta})$. We prove that
\begin{aalign}
\label{eq:step3Zpqbarandtilde}
	&\bar{t}(X,Y)=\tilde{t}(X,Y), \quad \bar{x}(X,Y)=\tilde{x}(X,Y), \quad \bar{U}(X,Y)=\tilde{U}(X,Y),\\
	&\hspace{47pt}\bar{p}(X,Y)=\tilde{p}(X,Y), \quad \bar{q}(X,Y)=\tilde{q}(X,Y)
\end{aalign}
for all $(X,Y)\in\Omega$, where $\Omega=[X_l,X_r]\times[Y_l,Y_r]$. We have $\bar{x}_{1}(X)=\tilde{x}_{1}(X)$ for $X\in[X_{l},X_{r}]$ and $\bar{x}_{2}(Y)=\tilde{x}_{2}(Y)$ for $Y\in[Y_{l},Y_{r}]$. Let us show that
\begin{align*}
	&\bar{U}_{1}(X)=\tilde{U}_{1}(X), \quad \bar{V}_{1}(X)=\tilde{V}_{1}(X), \quad \bar{J}_{1}(X)=\tilde{J}_{1}(X)-\tilde{J}_{1}(X_{l}),\\
	&\hspace{35pt}\bar{K}_{1}(X)=\tilde{K}_{1}(X)-\tilde{K}_{1}(X_{l}), \quad \bar{H}_{1}(X)=\tilde{H}_{1}(X)
\end{align*}
for $X\in[X_{l},X_{r}]$, and
\begin{aalign}
\label{eq:step3psibar2}
	&\bar{U}_{2}(Y)=\tilde{U}_{2}(Y), \quad \bar{V}_{2}(Y)=\tilde{V}_{2}(Y), \quad \bar{J}_{2}(Y)=\tilde{J}_{2}(Y)-\tilde{J}_{2}(Y_{l}),\\
	&\hspace{35pt}\bar{K}_{2}(Y)=\tilde{K}_{2}(Y)-\tilde{K}_{2}(Y_{l}), \quad \bar{H}_{2}(Y)=\tilde{H}_{2}(Y)
\end{aalign}
for $Y\in[Y_{l},Y_{r}]$. From \eqref{eq:mapfromDtoF4}, we have
\begin{equation*}
	\bar{U}_{1}(X)=\bar{u}_{0}(\bar{x}_{1}(X))=u_{0}(\tilde{x}_{1}(X))=u_{0}(x_{1}(f(X)))=U_{1}(f(X))=\tilde{U}_{1}(X).
\end{equation*}
By \eqref{eq:mapfromDtoF5}, we get
\begin{align*}
	\bar{V}_{1}(X)&=\bar{x}_{1}'(X)\frac{\bar{R}_{0}(\bar{x}_{1}(X))}{2c(\bar{U}_{1}(X))}\\
	&=\tilde{x}_{1}'(X)\frac{R_{0}(\tilde{x}_{1}(X))}{2c(\tilde{U}_{1}(X))}\\
	&=f'(X)x_{1}'(f(X))\frac{R_{0}(x_{1}(f(X)))}{2c(U_{1}(f(X)))}\\
	&=f'(X)V_{1}(f(X))\\
	&=\tilde{V}_{1}(X).
\end{align*}
We have, by \eqref{eq:mapfromDtoF3}, that
\begin{align*}
	\bar{J}_{1}(X)&=X-\bar{x}_{1}(X)\\
	&=X-\tilde{x}_{1}(X)\\
	&=X-x_{1}(X+\mu_{0}((-\infty,x_{l})))\\
	&=X+\mu_{0}((-\infty,x_{l}))-x_{1}(X+\mu_{0}((-\infty,x_{l})))\\
	&\quad-(X_{l}+\mu_{0}((-\infty,x_{l}))+x_{1}(X_{l}+\mu_{0}((-\infty,x_{l}))))\\ 
	&=J_{1}(X+\mu_{0}((-\infty,x_{l})))-J_{1}(X_{l}+\mu_{0}((-\infty,x_{l})))\\
	&=\tilde{J}_{1}(X)-\tilde{J}_{1}(X_{l})
\end{align*}
since $x_{1}(X_{l}+\mu_{0}((-\infty,x_{l})))=X_{l}$. From \eqref{eq:x1barthreecases}, $\bar{x}_{1}(X)=X$ for $X\leq X_{l}$, so that $\bar{J}_{1}(X)=X-\bar{x}_{1}(X)=0$ for $X\leq X_{l}$. This implies, by \eqref{eq:mapfromDtoF6}, that
\begin{align*}
	\bar{K}_{1}(X)&=\int_{-\infty}^{X}\frac{\bar{J}_{1}'(\bar{X})}{c(\bar{U}_{1}(\bar{X}))}\,d\bar{X}\\
	&=\int_{X_{l}}^{X}\frac{\tilde{J}_{1}'(\bar{X})}{c(\tilde{U}_{1}(\bar{X}))}\,d\bar{X}\\
	&=\int_{X_{l}}^{X}\frac{J_{1}'(\bar{X}+\mu_{0}((-\infty,x_{l})))}{c(U_{1}(\bar{X}+\mu_{0}((-\infty,x_{l}))))}\,d\bar{X}\\
	&=\int_{X_{l}+\mu_{0}((-\infty,x_{l}))}^{X+\mu_{0}((-\infty,x_{l}))}\frac{J_{1}'(\bar{X})}{c(U_{1}(\bar{X}))}\,d\bar{X} \quad \text{by a change of variables}\\
	&=\int_{-\infty}^{X+\mu_{0}((-\infty,x_{l}))}\frac{J_{1}'(\bar{X})}{c(U_{1}(\bar{X}))}\,d\bar{X}-\int_{-\infty}^{X_{l}+\mu_{0}((-\infty,x_{l}))}\frac{J_{1}'(\bar{X})}{c(U_{1}(\bar{X}))}\,d\bar{X}\\
	&=K_{1}(X+\mu_{0}((-\infty,x_{l})))-K_{1}(X_{l}+\mu_{0}((-\infty,x_{l})))\\
	&=\tilde{K}_{1}(X)-\tilde{K}_{1}(X_{l}).
\end{align*}
By \eqref{eq:mapfromDtoF7}, we have
\begin{align*}
	\bar{H}_{1}(X)&=\frac{1}{2}\bar{\rho}_{0}(\bar{x}_{1}(X))\bar{x}_{1}'(X)\\
	&=\frac{1}{2}\rho_{0}(\tilde{x}_{1}(X))\tilde{x}_{1}'(X)\\
	&=\frac{1}{2}\rho_{0}(x_{1}(f(X)))x_{1}'(f(X))f'(X)\\
	&=f'(X)H_{1}(f(X))\\
	&=\tilde{H}_{1}(X).
\end{align*}
In a similar way, one proves \eqref{eq:step3psibar2}.

We have $\bar{\X}(s)=\tilde{\X}(s)$ and $\bar{\Y}(s)=\tilde{\Y}(s)$ for $s\in[s_{l},s_{r}]$. We show that
\begin{align*}
	&\bar{\Z}_{1}(s)=\tilde{\Z}_{1}(s), \quad \bar{\Z}_{2}(s)=\tilde{\Z}_{2}(s), \quad \bar{\Z}_{3}(s)=\tilde{\Z}_{3}(s),\\ 
	&\bar{\Z}_{4}(s)=\tilde{\Z}_{4}(s)-\tilde{\Z}_{4}(s_{l}), \quad \bar{\Z}_{5}(s)=\tilde{\Z}_{5}(s)-\tilde{\Z}_{5}(s_{l})
\end{align*}
for $s\in[s_{l},s_{r}]$, and 
\begin{equation}
\label{eq:step3VWpqbartilde}
	\bar{\V}(X)=\tilde{\V}(X), \quad \bar{\W}(Y)=\tilde{\W}(Y), \quad \bar{\p}(X)=\tilde{\p}(X), \quad \bar{\q}(Y)=\tilde{\q}(Y)
\end{equation}
for $X\in[X_{l},X_{r}]$ and $Y\in[Y_{l},Y_{r}]$. From \eqref{eq:mapFtoGZ1}, we have $\bar{\Z}_{1}(s)=0=\tilde{\Z}_{1}(s)$. By \eqref{eq:mapFtoGZ2} and \eqref{eq:mapFtoGZ3}, we have $\bar{\Z}_{2}(s)=\bar{x}_{1}(\bar{\X}(s))=\tilde{x}_{1}(\tilde{\X}(s))=\tilde{\Z}_{2}(s)$ and $\bar{\Z}_{3}(s)=\bar{U}_{1}(\bar{\X}(s))=\tilde{U}_{1}(\tilde{\X}(s))=\tilde{\Z}_{3}(s)$, respectively. From \eqref{eq:mapFtoGZ4}, we get
\begin{align*}
	\bar{\Z}_{4}(s)&=\bar{J}_{1}(\bar{\X}(s))+\bar{J}_{2}(\bar{\Y}(s))\\
	&=\tilde{J}_{1}(\tilde{\X}(s))-\tilde{J}_{1}(X_{l})+\tilde{J}_{2}(\tilde{\Y}(s))-\tilde{J}_{2}(Y_{l})\\
	&=\tilde{\Z}_{4}(s)-\tilde{\Z}_{4}(s_{l})
\end{align*}
and similarly, we find that $\bar{\Z}_{5}(s)=\tilde{\Z}_{5}(s)-\tilde{\Z}_{5}(s_{l})$. Using \eqref{eq:mapFtoGV1}-\eqref{eq:mapFtoGp}, a straightforward calculation shows \eqref{eq:step3VWpqbartilde}. Hence, $\bar{\Theta}$ and $\tilde{\Theta}$ are equal in $\Omega$, except that the energy potentials differ up to a constant. However, since the governing equations \eqref{eq:goveq} are invariant with respect to addition of a constant to the energy potentials, we get, by Lemma \ref{lemma:SolnBigRectangle}, that \eqref{eq:step3Zpqbarandtilde} holds. 

\textbf{Step 4.} We prove that there exists $(X_{0},Y_{0})\in\Omega$ such that
\begin{equation}
\label{eq:step4tandxbar}
	\bar{t}(X_{0},Y_{0})=\mathbf{t} \quad \text{and} \quad \bar{x}(X_{0},Y_{0})=\mathbf{x}.
\end{equation}
We have
\begin{equation*}
	\bar{x}(X_{l},Y_{l})=\bar{x}_{1}(X_{l})=\bar{x}_{2}(Y_{l})=x_{l} \quad \text{and} \quad \bar{x}(X_{r},Y_{r})=\bar{x}_{1}(X_{r})=\bar{x}_{2}(Y_{r})=x_{r},
\end{equation*}
so that
\begin{aalign}
\label{eq:step4xbardiff}
	\bar{x}(X_{r},Y_{l})-x_{l}&=\bar{x}(X_{r},Y_{l})-\bar{x}(X_{l},Y_{l})\\
	&=\int_{X_{l}}^{X_{r}}\bar{x}_{X}(X,Y_{l})\,dX\\
	&=\int_{X_{l}}^{X_{r}}c(\bar{U}(X,Y_{l}))\bar{t}_{X}(X,Y_{l})\,dX \quad \text{by } \eqref{eq:setH1}\\
	&\leq\kappa\int_{X_{l}}^{X_{r}}\bar{t}_{X}(X,Y_{l})\,dX \quad \text{since } \frac{1}{\kappa}\leq c\leq\kappa \text{ and } \bar{t}_{X}\geq 0\\
	&=\kappa \bar{t}(X_{r},Y_{l}) \quad \text{since } \bar{t}(X_{l},Y_{l})=\bar{t}(\bar{\X}(s_{l}),\bar{\Y}(s_{l}))=0. 
\end{aalign}
In a similar way, one proves that $x_{r}-\bar{x}(X_{r},Y_{l})\leq\kappa\bar{t}(X_{r},Y_{l})$, which added to \eqref{eq:step4xbardiff} yields $x_{r}-x_{l}\leq 2\kappa\bar{t}(X_{r},Y_{l})$, or
\begin{equation}
\label{eq:step4lowerrightcorner}
	\mathbf{t}\leq\bar{t}(X_{r},Y_{l}).
\end{equation}
There exists $(X_{0},Y_{0})\in\mathbb{R}^{2}$, which may not be unique, such that 
\begin{equation*}
	\bar{t}(X_{0},Y_{0})=\mathbf{t} \quad \text{and} \quad \bar{x}(X_{0},Y_{0})=\mathbf{x}.
\end{equation*}
Assume that
\begin{equation}
\label{eq:step4notinomega}
	\bar{t}(X,Y)\neq\mathbf{t} \quad \text{or} \quad \bar{x}(X,Y)\neq\mathbf{x}
\end{equation}
for all $(X,Y)\in\Omega$. We claim that we cannot have
\begin{equation}
\label{eq:step4firstsituation}
	X_{0}>X_{r} \quad \text{and} \quad Y_{0}<Y_{l}
\end{equation}
or
\begin{equation}
\label{eq:step4secondsituation}
	X_{0}<X_{r} \quad \text{and} \quad Y_{0}>Y_{l},
\end{equation}
so that either $X_{0}>X_{r}$ and $Y_{0}\geq Y_{l}$ or $X_{0}\leq X_{r}$ and $Y_{0}<Y_{l}$.

If \eqref{eq:step4firstsituation} holds, we get, since $\bar{t}_{X}\geq 0$ and $\bar{t}_{Y}\leq 0$, that $\mathbf{t}=\bar{t}(X_{0},Y_{0})\geq\bar{t}(X_{r},Y_{l})\geq\mathbf{t}$, where we used \eqref{eq:step4lowerrightcorner}. Hence, $\bar{t}(X_{r},Y_{l})=\mathbf{t}$, which contradicts \eqref{eq:step4notinomega}.

If \eqref{eq:step4secondsituation} holds, we have either $Y_{0}>Y_{r}$ or $Y_{l}<Y_{0}\leq Y_{r}$. If $Y_{0}>Y_{r}$, we get a contradiction, since $\mathbf{t}=\bar{t}(X_{0},Y_{0})\leq \bar{t}(X_{r},Y_{r})=\bar{t}(\bar{\X}(s_{r}),\bar{\Y}(s_{r}))=0$. If $Y_{l}<Y_{0}\leq Y_{r}$, we have $X_{0}<X_{l}$ because $(X_{0},Y_{0})\notin\Omega$, which leads to the contradiction $\mathbf{t}=\bar{t}(X_{0},Y_{0})\leq \bar{t}(X_{l},Y_{l})=\bar{t}(\bar{\X}(s_{l}),\bar{\Y}(s_{l}))=0$.
Hence, we have either $X_{0}>X_{r}$ and $Y_{0}\geq Y_{l}$ or $X_{0}\leq X_{r}$ and $Y_{0}<Y_{l}$.
 
Assume that $X_{0}>X_{r}$ and $Y_{0}\geq Y_{l}$. If $Y_{0}>Y_{r}$, then $\mathbf{x}=\bar{x}(X_{0},Y_{0})\geq\bar{x}(X_{r},Y_{r})=x_{r}=\mathbf{x}+\kappa\mathbf{t}$, that is, $\mathbf{t}\leq 0$, which is a contradiction. Therefore, $Y_{l}\leq Y_{0}\leq Y_{r}$. Since $\bar{x}_{X}\geq 0$, we have $\bar{x}(X_{r},Y_{0})\leq \bar{x}(X_{0},Y_{0})$. We claim that $\bar{x}(X_{r},Y_{0})<\bar{x}(X_{0},Y_{0})$. Assume the opposite, that is, $\bar{x}(X_{r},Y_{0})=\bar{x}(X_{0},Y_{0})$. Then, since $\bar{x}$ is nondecreasing in the $X$ variable, we have $\bar{x}_{X}(X,Y_{0})=0$ for all $X\in[X_{r},X_{0}]$. By \eqref{eq:setH1}, we get $\bar{t}_{X}(X,Y_{0})=0$ for all $X\in[X_{r},X_{0}]$, so that $\bar{t}(X_{r},Y_{0})=\bar{t}(X_{0},Y_{0})=\mathbf{t}$. This contradicts \eqref{eq:step4notinomega} and we must have $\bar{x}(X_{r},Y_{0})<\bar{x}(X_{0},Y_{0})$. We obtain
\begin{align*}
	\mathbf{x}&=\bar{x}(X_{0},Y_{0})\\
	&>\bar{x}(X_{r},Y_{0})\\
	&=x_{r}-\int_{Y_{0}}^{Y_{r}}\bar{x}_{Y}(X_{r},Y)\,dY\\
	&=x_{r}+\int_{Y_{0}}^{Y_{r}}c(\bar{U}(X_{r},Y))\bar{t}_{Y}(X_{r},Y)\,dY \quad \text{by } \eqref{eq:setH1}\\
	&\geq x_{r}+\kappa\int_{Y_{0}}^{Y_{r}}\bar{t}_{Y}(X_{r},Y)\,dY \quad \text{since } \frac{1}{\kappa}\leq c\leq\kappa \text{ and } \bar{t}_{Y}\leq 0\\
	&=x_{r}-\kappa\bar{t}(X_{r},Y_{0}) \quad \text{since } \bar{t}(X_{r},Y_{r})=\bar{t}(\bar{\X}(s_{r}),\bar{\Y}(s_{r}))=0\\
	&\geq x_{r}-\kappa\bar{t}(X_{0},Y_{0}) \quad \text{because } \bar{t}_{X}\geq 0\\
	&=\mathbf{x},
\end{align*}
which is a contradiction. The situation $X_{0}\leq X_{r}$ and $Y_{0}<Y_{l}$ can be treated similarly. 
Hence, \eqref{eq:step4notinomega} cannot hold and we have proved \eqref{eq:step4tandxbar}.

\textbf{Step 5.}
We have
\begin{align*}
	\bar{u}(\mathbf{t},\mathbf{x})&=\bar{U}(X_{0},Y_{0}) \quad \text{by } \eqref{eq:step4tandxbar} \text{ and } \eqref{eq:lemmasemigprop1}\\
	&=\tilde{U}(X_{0},Y_{0}) \quad \text{by } \eqref{eq:step3Zpqbarandtilde}\\
	&=U(f(X_{0}),g(Y_{0}))\\
	&=u(t(f(X_{0}),g(Y_{0})),x(f(X_{0}),g(Y_{0}))) \quad \text{by } \eqref{eq:lemmasemigprop1}\\
	&=u(\tilde{t}(X_{0},Y_{0}),\tilde{x}(X_{0},Y_{0}))\\
	&=u(\bar{t}(X_{0},Y_{0}),\bar{x}(X_{0},Y_{0})) \quad \text{by } \eqref{eq:step3Zpqbarandtilde}\\
	&=u(\mathbf{t},\mathbf{x}) \quad \text{by } \eqref{eq:step4tandxbar}.
\end{align*}
This concludes the proof.
\end{proof}

\section{Regularity of Solutions}
\label{sec:Reg}

The theorems we prove in Section \ref{sec:SmoothSoln} and \ref{sec:Approx} are local results. The main reason for this is that we require the initial data $\rho_{0}$ and $\sigma_{0}$ corresponding to the equations \eqref{eq:nvwsys2} and \eqref{eq:nvwsys3} to be bounded from below by a strictly positive constant and to belong to $L^{2}$, which is not possible globally.

\subsection{Existence of Smooth Solutions}
\label{sec:SmoothSoln}

\begin{theorem}
\label{thm:Regular1st}
Let $-\infty<x_{l}<x_{r}<\infty$ and consider $(u_{0},R_{0},S_{0},\rho_{0},\sigma_{0},\mu_{0},\nu_{0})\in\D$. Let $m\in\mathbb{N}$ and assume that,
\begin{itemize}
\item[\mylabel{eq:smoothAssumption1}{(A1)}]
	$u_{0}\in\Linf([x_{l},x_{r}])$,
\item[\mylabel{eq:smoothAssumption2}{(A2)}]
	$R_{0},S_{0},\rho_{0},\sigma_{0}\in W^{m-1,\infty}([x_{l},x_{r}])$,
\item[\mylabel{eq:smoothAssumption3}{(A3)}]
there are constants $d>0$ and $e>0$ such that $\rho_{0}(x)\geq d$ and $\sigma_{0}(x)\geq e$ for all  $x\in[x_{l},x_{r}]$,
\item[\mylabel{eq:smoothAssumption4}{(A4)}]
	$\mu_{0}$ and $\nu_{0}$ are absolutely continuous on $[x_{l},x_{r}]$,
\item[\mylabel{eq:smoothAssumption5}{(A5)}]
	$c\in C^{m-1}(\mathbb{R})$ and $\displaystyle{\max_{u\in\mathbb{R}}}\bigg|\frac{d^{i}}{du^{i}}c(u)\bigg|\leq k_{i}$ for constants $k_{i}, i=3,4,5,\dots,m-1$.
\end{itemize}
For any $\tau\in\big[0,\frac{1}{2\kappa}(x_{r}-x_{l})\big]$ consider  
\begin{equation*}
	(u,R,S,\rho,\sigma,\mu,\nu)(\tau)=\bar{S}_{\tau}(u_{0},R_{0},S_{0},\rho_{0},\sigma_{0},\mu_{0},\nu_{0}).
\end{equation*}
Then
\begin{itemize}
\item[\mylabel{eq:smoothResult1}{(P1)}]
	$u(\tau,\cdot)\in W^{m,\infty}([x_{l}+\kappa\tau,x_{r}-\kappa\tau])$,
\item[\mylabel{eq:smoothResult2}{(P2)}]
	$R(\tau,\cdot),S(\tau,\cdot),\rho(\tau,\cdot),\sigma(\tau,\cdot)\in W^{m-1,\infty}([x_{l}+\kappa\tau,x_{r}-\kappa\tau])$,
\item[\mylabel{eq:smoothResult3}{(P3)}]
there are constants $\bar{d}>0$ and $\bar{e}>0$ such that $\rho(\tau,x)\geq \bar{d}$ and $\sigma(\tau,x)\geq \bar{e}$
	 for all $x\in[x_{l}+\kappa\tau,x_{r}-\kappa\tau]$,
\item[\mylabel{eq:smoothResult4}{(P4)}]
	$\mu(\tau,\cdot)$ and $\nu(\tau,\cdot)$ are absolutely continuous on $[x_{l}+\kappa\tau,x_{r}-\kappa\tau]$.
\end{itemize}
For $\tau\in\big[-\frac{1}{2\kappa}(x_{r}-x_{l}),0\big]$, the solution satisfies the same properties on the interval \newline
$\big[x_{l}-\kappa\tau,x_{r}+\kappa\tau\big]$.
\end{theorem}

Note that since $(u_{0})_{x}=\frac{1}{2c(u_{0})}(R_{0}-S_{0})$, it follows from assumptions \ref{eq:smoothAssumption1}, \ref{eq:smoothAssumption2} and \ref{eq:smoothAssumption5} that $u_{0}\in W^{m,\infty}([x_{l},x_{r}])$.

Specifically, \ref{eq:smoothAssumption4} means that $\mu_{0}((-\infty,x_{l}))=\mu_{0}((-\infty,x_{l}])$ and $\nu_{0}((-\infty,x_{l}))=\nu_{0}((-\infty,x_{l}])$.

By \eqref{eq:cassumption} and \eqref{eq:cderassumption}, \ref{eq:smoothAssumption5} holds for $i=0,1,2$.

\begin{proof}
In the following, we will consider the case $0<\tau\leq\frac{1}{2\kappa}(x_{r}-x_{l})$. The case $-\frac{1}{2\kappa}(x_{r}-x_{l})\leq\tau<0$ can be treated in the same way.

We decompose the proof into three steps.

\textbf{Step 1.} We first consider the case $m=1$.

\textbf{(i)} Consider $(\psi_{1},\psi_{2})=\mathbf{L}(u_{0},R_{0},S_{0},\rho_{0},\sigma_{0},\mu_{0},\nu_{0})$. Since $\mu_{0}$ is absolutely continuous on $[x_{l},x_{r}]$ we have from \eqref{eq:mapfromDtoF1},
\begin{equation*}
	x_{1}(X)+\mu_{0}((-\infty,x_{1}(X)))=X
\end{equation*}
for all $x_{1}(X)\in[x_{l},x_{r}]$. For $X_{l}$ and $X_{r}$ satisfying $x_{1}(X_{l})=x_{l}$ and $x_{1}(X_{r})=x_{r}$ we have
\begin{equation*}
	X_{l}=x_{l}+\mu_{0}((-\infty,x_{l})) \quad \text{and} \quad X_{r}=x_{r}+\mu_{0}((-\infty,x_{r})).
\end{equation*}
Therefore, since $x_{1}$ is nondecreasing, we get
\begin{equation}
\label{eq:x1onXlXR}
	x_{1}(X)+\mu_{0}((-\infty,x_{1}(X)))=X
\end{equation}
for all $X\in[X_{l},X_{r}]$. Similarly we find by using \eqref{eq:mapfromDtoF2},
\begin{equation*}
	x_{2}(Y)+\nu_{0}((-\infty,x_{2}(Y)))=Y
\end{equation*}
for all $Y\in[Y_{l},Y_{r}]$, where $x_{2}(Y_{l})=x_{l}$, $x_{2}(Y_{r})=x_{r}$
and
\begin{equation*}
	 Y_{l}=x_{l}+\nu_{0}((-\infty,x_{l})) \quad \text{and} \quad Y_{r}=x_{r}+\nu_{0}((-\infty,x_{r})).
\end{equation*}
We define $\Omega=[X_{l},X_{r}]\times[Y_{l},Y_{r}]$. From now on we only consider $(X,Y)\in\Omega$. Rewriting \eqref{eq:x1onXlXR} yields
\begin{equation*}
	x_{1}(X)+\mu_{0}((-\infty,x_{l}))+(\mu_{0})_{\text{ac}}((x_{l},x_{1}(X)))=X.
\end{equation*}
Hence,
\begin{equation*}
	x_{1}(X)+\mu_{0}((-\infty,x_{l}))+\frac{1}{4}\int_{x_{l}}^{x_{1}(X)}(R_{0}^{2}+c(u_{0})\rho_{0}^{2})(x)\,dx=X.
\end{equation*}
We differentiate and obtain
\begin{equation*}
	x_{1}'(X)+\frac{1}{4}x_{1}'(X)(R_{0}^{2}+c(u_{0})\rho_{0}^{2})\circ x_{1}(X)=1,
\end{equation*}
which implies that
\begin{equation}
\label{eq:x1derexpr}
	x_{1}'(X)=\frac{4}{(R_{0}^{2}+c(u_{0})\rho_{0}^{2})\circ x_{1}(X)+4}.
\end{equation}
Since $R_{0},\rho_{0}\in\Linf([x_{l},x_{r}])$, we get the lower bound
\begin{equation}
\label{eq:regx1strictincreasing}
	x_{1}'(X)\geq\frac{4}{||R_{0}||_{\Linf([x_{l},x_{r}])}^{2}+\kappa||\rho_{0}||_{\Linf([x_{l},x_{r}])}^{2}+4}=:d_{1}>0,
\end{equation}
and since $\rho_{0}(x)\geq d$, we find the upper bound
\begin{equation}
\label{eq:regx1UpperBound}
	x_{1}'(X)\leq \frac{4\kappa}{d^{2}+4\kappa}.
\end{equation}
Similarly, we find
\begin{equation*}
	x_{2}'(Y)=\frac{4}{(S_{0}^{2}+c(u_{0})\sigma_{0}^{2})\circ x_{2}(Y)+4},
\end{equation*}
\begin{equation}
\label{eq:regx2strictincreasing}
	x_{2}'(Y)\geq\frac{4}{||S_{0}||_{\Linf([x_{l},x_{r}])}^{2}+\kappa||\sigma_{0}||_{\Linf([x_{l},x_{r}])}^{2}+4}=:e_{1}>0
\end{equation}
and
\begin{equation}
\label{eq:regx2UpperBound}
	x_{2}'(Y)\leq \frac{4\kappa}{e^{2}+4\kappa}.
\end{equation}

\textbf{(ii)} Let $\Theta=(\X,\Y,\Z,\V,\W,\p,\q)=\mathbf{C}(\psi_{1},\psi_{2})$. By \eqref{eq:mapFtoGX} we have $x_{1}(\X(s))=x_{2}(\Y(s))$, which after differentiating and using $\X(s)+\Y(s)=2s$, yields
\begin{equation}
\label{eq:XandYdotx1x2der}
	\dot{\X}(s)=\frac{2x_{2}'(\Y(s))}{x_{1}'(\X(s))+x_{2}'(\Y(s))} \quad \text{and} \quad \dot{\Y}(s)=\frac{2x_{1}'(\X(s))}{x_{1}'(\X(s))+x_{2}'(\Y(s))}. 
\end{equation}
This implies, by \eqref{eq:regx1strictincreasing}-\eqref{eq:regx2UpperBound}, that
\begin{equation}
\label{eq:XandYdot}
	\dot{\X}(s)\geq 2e_{1}\bigg(\frac{4\kappa}{d^{2}+4\kappa}+\frac{4\kappa}{e^{2}+4\kappa}\bigg)^{-1} \text{ and } \dot{\Y}(s)\geq 2d_{1}\bigg(\frac{4\kappa}{d^{2}+4\kappa}+\frac{4\kappa}{e^{2}+4\kappa}\bigg)^{-1}
\end{equation}
for all $s$ such that $(\X(s),\Y(s))\in\Omega$, that is, for values of $s$ which satisfy $X_{l}\leq\X(s)\leq X_{r}$ and $Y_{l}\leq\Y(s)\leq Y_{r}$. Using this together with the identity $\X(s)+\Y(s)=2s$, we find that \eqref{eq:XandYdot} is valid for all $s\in[s_{l},s_{r}]$, where $s_{l}=\frac{1}{2}(X_{l}+Y_{l})$ and $s_{r}=\frac{1}{2}(X_{r}+Y_{r})$. Hence, $\X(s)$ and $\Y(s)$ are strictly increasing functions on $[s_l,s_r]$.

\textbf{(iii)} Consider $(Z,p,q)=\mathbf{S}(\Theta)$. By an argument as in the proof of Lemma~\ref{lemma:stripest}, we obtain the inequality
\begin{equation*}
	(x_{X}+J_{X})(X,Y)\leq (\V_{2}+\V_{4})(X)e^{C|Y-\Y(X)|}
\end{equation*}
for all $(X,Y)\in\Omega$, where $C$ depends on $|||\Theta|||_{\G(\Omega)}$. We have $\V_{2}+\V_{4}=\frac{1}{2}x_{1}'+J_{1}'$ and since $J_{1}'=1-x_{1}'$, $x_{1}'\geq 0$ and $J_{1}'\geq 0$, this implies that $\frac{1}{2}\leq \V_{2}+\V_{4}\leq 1$. Thus, we get
\begin{equation}
\label{eq:GronIneqxXJX}
	(x_{X}+J_{X})(X,Y)\leq e^{C|Y-\Y(X)|}.
\end{equation}
By \eqref{eq:mapFtoGp}, \eqref{eq:mapfromDtoF7}, \ref{eq:smoothAssumption3} and \eqref{eq:regx1strictincreasing} we obtain
\begin{equation}
\label{eq:pPos}
	p(X,\Y(X))=\p(X)=H_{1}(X)=\frac{1}{2}\rho_{0}(x_{1}(X))x_{1}'(X)\geq\frac{1}{2}dd_{1}=:d_{2}.
\end{equation}
Similarly we obtain
\begin{equation*}
	q(\X(Y),Y)=\q(Y)=H_{2}(Y)=\frac{1}{2}\sigma_{0}(x_{2}(Y))x_{2}'(Y)\geq\frac{1}{2}ee_{1}=:e_{2}.
\end{equation*}
Then, since $p_{Y}=0$, we get
\begin{align*}
	d_{2}^{2}&\leq p^{2}(X,\Y(X))\\
	&=p^{2}(X,Y)\\
	&\leq\bigg(\frac{1}{c(U)}((c(U)U_{X})^{2}+c(U)p^{2})\bigg)(X,Y)\\
	&\leq\kappa((c(U)U_{X})^{2}+c(U)p^{2})(X,Y)\\
	&=2\kappa(J_{X}x_{X})(X,Y)\\
	&\leq 2\kappa(J_{X}+x_{X})x_{X}(X,Y)\\
	&\leq 2\kappa e^{C|Y-\Y(X)|}x_{X}(X,Y),
\end{align*}
where we used \eqref{eq:setH3} and \eqref{eq:GronIneqxXJX}. Hence, 
\begin{equation}
\label{eq:xXpos}
	x_{X}(X,Y)\geq \frac{d_{2}^{2}}{2\kappa}e^{-C|Y-\Y(X)|}.
\end{equation}
Using that $q_{X}=0$, we find in the same way that
\begin{equation}
\label{eq:xYpos}
	x_{Y}(X,Y)\geq \frac{e_{2}^{2}}{2\kappa}e^{-\tilde{C}|X-\X(Y)|},
\end{equation}
where $\tilde{C}$ depends on $|||\Theta|||_{\G(\Omega)}$. From \eqref{eq:setH1}, we then get
\begin{equation}
\label{eq:tXandtYpos}
	t_{X}(X,Y)\geq \frac{d_{2}^{2}}{2\kappa^{2}}e^{-C|Y-\Y(X)|} \quad \text{and} \quad t_{Y}(X,Y)\leq-\frac{e_{2}^{2}}{2\kappa^{2}}e^{-\tilde{C}|X-\X(Y)|}.
\end{equation}

\textbf{(iv)} For any $0<\tau\leq\frac{1}{2\kappa}(x_{r}-x_{l})$ consider $\Theta(\tau)=\mathbf{E}\circ\mathbf{t}_{\tau}(Z,p,q)$. 

We claim that $(\X(\tau,s), \Y(\tau,s))$ lies below the curve\footnote{Note that this is also true outside of $\Omega$.} $(\X(s),\Y(s))$. For any $\bar{s}$ and $s$ such that $\X(\tau,\bar{s})=\X(s)$ we have
\begin{equation*}
	t(\X(s),\Y(s))=0<\tau=t(\X(\tau,\bar{s}),\Y(\tau,\bar{s})),
\end{equation*}
so that by \eqref{eq:setH1} and \eqref{eq:setH4},
\begin{equation}
\label{eq:curveTauInBox1}
	\Y(\tau,\bar{s})<\Y(s),
\end{equation}
which proves the claim.

We prove that there exist $\bar{s}_{\text{min}}$ and $\bar{s}_{\text{max}}$ satisfying $s_{l}<\bar{s}_{\text{min}}\leq\bar{s}_{\text{max}}<s_{r}$ such that $(\X(\tau,s),\Y(\tau,s))$ belongs to $\Omega$ for all $s\in[\bar{s}_{\text{min}},\bar{s}_{\text{max}}]$. First we need an estimate. Since $c(u)\leq\kappa$, $t_{X}(X,Y)>0$, and $t_{Y}(X,Y)<0$ we have
\begin{aalign}
\label{eq:maxTimeBound}
	\frac{1}{2\kappa}(x_{r}-x_{l})
	&=\frac{1}{2\kappa}(x(X_{r},Y_{r})-x(X_{l},Y_{l}))\\
	&=\frac{1}{2\kappa}(x(X_{r},Y_{r})-x(X_{r},Y_{l})+x(X_{r},Y_{l})-x(X_{l},Y_{l}))\\
	&=\frac{1}{2\kappa}\bigg(\int_{Y_{l}}^{Y_{r}}x_{Y}(X_{r},Y)\,dY+\int_{X_{l}}^{X_{r}}x_{X}(X,Y_{l})\,dX\bigg)\\
	&=\frac{1}{2\kappa}\bigg(-\int_{Y_{l}}^{Y_{r}}c(U(X_{r},Y))t_{Y}(X_{r},Y)\,dY\\
	&\hspace{34pt}+\int_{X_{l}}^{X_{r}}c(U(X,Y_{l}))t_{X}(X,Y_{l})\,dX\bigg) \quad \text{by } \eqref{eq:setH1}\\
	&\leq\frac{1}{2}\bigg(-\int_{Y_{l}}^{Y_{r}}t_{Y}(X_{r},Y)\,dY+\int_{X_{l}}^{X_{r}}t_{X}(X,Y_{l})\,dX\bigg)\\
	&=t(X_{r},Y_{l}).
\end{aalign}
Consider $\bar{s}_{\text{max}}$ such that $\X(\tau,\bar{s}_{\text{max}})=X_{r}$. Note that since $\X(\tau,\bar{s}_{\text{max}})=\X(s_{r})$ we get from \eqref{eq:curveTauInBox1}, $\Y(\tau,\bar{s}_{\text{max}})<\Y(s_{r})=Y_{r}$. From \eqref{eq:maxTimeBound} we get
\begin{equation*}
	t(X_{r},\Y(\tau,\bar{s}_{\text{max}}))=\tau\leq\frac{1}{2\kappa}(x_{r}-x_{l})\leq t(X_{r},Y_{l}),
\end{equation*}
so that $\Y(\tau,\bar{s}_{\text{max}})\geq Y_{l}$. In particular, we have
\begin{equation}
\label{eq:smax}
	Y_{l}\leq\Y(\tau,\bar{s}_{\text{max}})<Y_{r}
\end{equation}
and therefore, since $\X(\tau,\cdot)$ and $\Y(\tau,\cdot)$ are nondecreasing and continuous, there exists $\bar{s}_{\text{min}}$ such that $\bar{s}_{\text{min}}\leq\bar{s}_{\text{max}}$, $\Y(\tau,\bar{s}_{\text{min}})=Y_{l}$ and $\X(\tau,\bar{s}_{\text{min}})\leq X_{r}$. Since
\begin{equation*}
	t(X_{l},Y_{l})=0<\tau=t(\X(\tau,\bar{s}_{\text{min}}),Y_{l})
\end{equation*}
we find that $\X(\tau,\bar{s}_{\text{min}})>X_{l}$ and we have
\begin{equation}
\label{eq:smin}
	X_{l}<\X(\tau,\bar{s}_{\text{min}})\leq X_{r}.
\end{equation}
For any $s\in[\bar{s}_{\text{min}},\bar{s}_{\text{max}}]$ we get from \eqref{eq:smax} and \eqref{eq:smin}, since $\X(\tau,\cdot)$ and $\Y(\tau,\cdot)$ are nondecreasing, that
\begin{equation*}
	X_{l}<\X(\tau,\bar{s}_{\text{min}})\leq \X(\tau,s)\leq \X(\tau,\bar{s}_{\text{max}})=X_{r}
\end{equation*}
and
\begin{equation*}
	Y_{l}=\Y(\tau,\bar{s}_{\text{min}})\leq \Y(\tau,s)\leq \Y(\tau,\bar{s}_{\text{max}})<Y_{r}.
\end{equation*}
In other words, the curve $(\X(\tau,s),\Y(\tau,s))$ lies in $\Omega$ for all $s\in[\bar{s}_{\text{min}},\bar{s}_{\text{max}}]$,  and hence all the estimates obtained in (iii) are valid along this curve. Observe that
\begin{equation*}
	s_{l}=\frac{1}{2}(X_{l}+Y_{l})<\frac{1}{2}\big(\X(\tau,\bar{s}_{\text{min}})+\Y(\tau,\bar{s}_{\text{min}})\big)=\bar{s}_{\text{min}}
\end{equation*}
and
\begin{equation*}
	s_{r}=\frac{1}{2}(X_{r}+Y_{r})>\frac{1}{2}\big(\X(\tau,\bar{s}_{\text{max}})+\Y(\tau,\bar{s}_{\text{max}})\big)=\bar{s}_{\text{max}},
\end{equation*}
which implies $s_{l}<\bar{s}_{\text{min}}\leq\bar{s}_{\text{max}}<s_{r}$.

\begin{figure}
	\centerline{\hbox{\includegraphics[width=10cm]{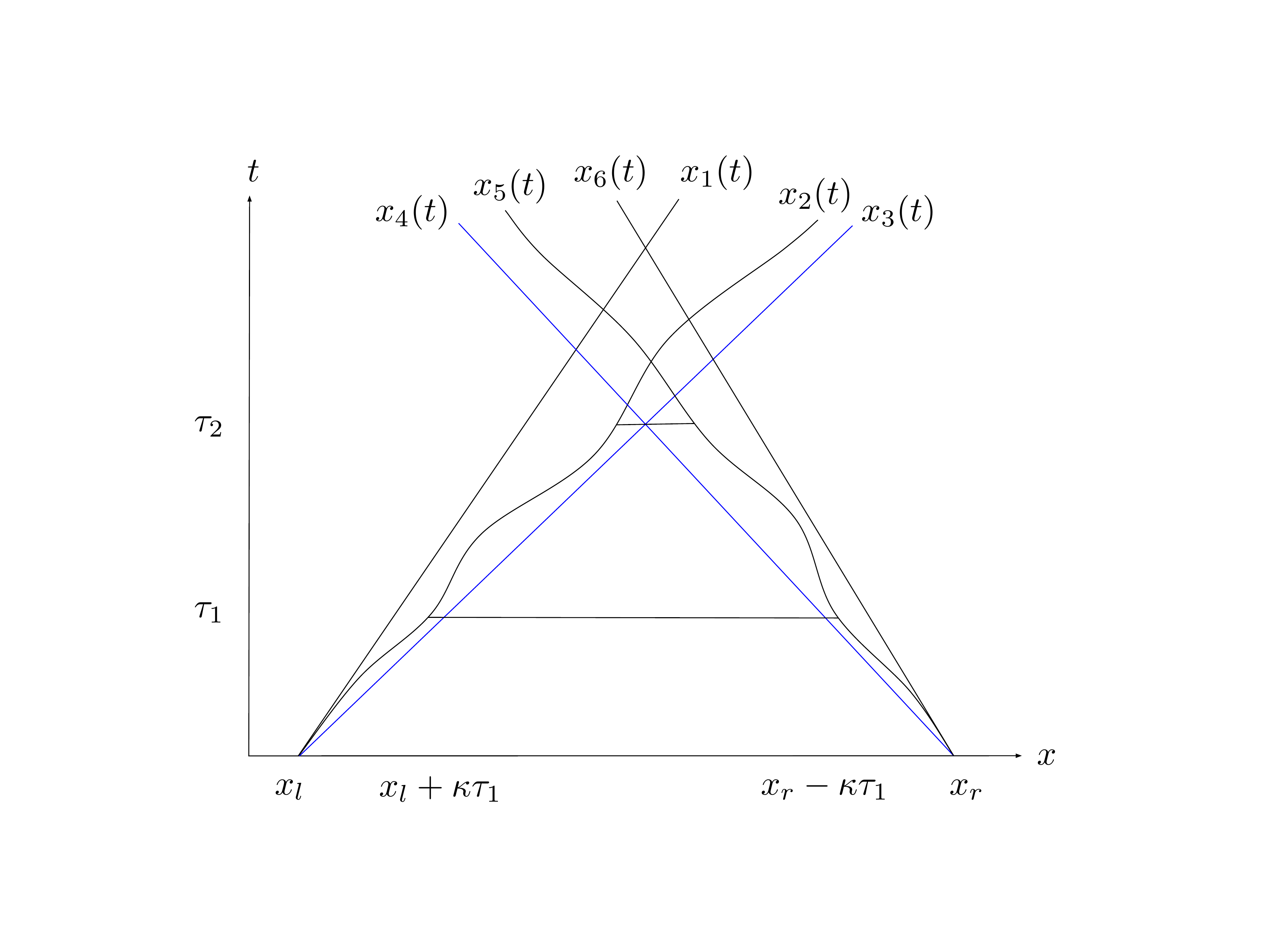}}}
	\caption{The region bounded by the characteristics $x_{2}(t)$ (forward) and $x_{5}(t)$ (backward) starting from $x_{l}$ and $x_{r}$, respectively, at $t=0$. Here, $0<\tau_{1}<\tau_{2}=\frac{1}{2\kappa}(x_{r}-x_{l})$. The remaining functions are given by $x_{1}(t)=x_{l}+\frac{t}{\kappa}$, $x_{3}(t)=x_{l}+\kappa t$, $x_{4}(t)=x_{r}-\kappa t$ and $x_{6}(t)=x_{r}-\frac{t}{\kappa}$. We have $x_{3}(\tau_{1})=x_{l}+\kappa\tau_{1}=x(\X(\tau_{1},\bar{s}_{1}),\Y(\tau_{1},\bar{s}_{1}))$ and $x_{4}(\tau_{1})=x_{r}-\kappa\tau_{1}=x(\X(\tau_{1},\bar{s}_{2}),\Y(\tau_{1},\bar{s}_{2}))$.}
	\label{fig:FigThm71i}
\end{figure}

\begin{figure}
	\centerline{\hbox{\includegraphics[width=10cm]{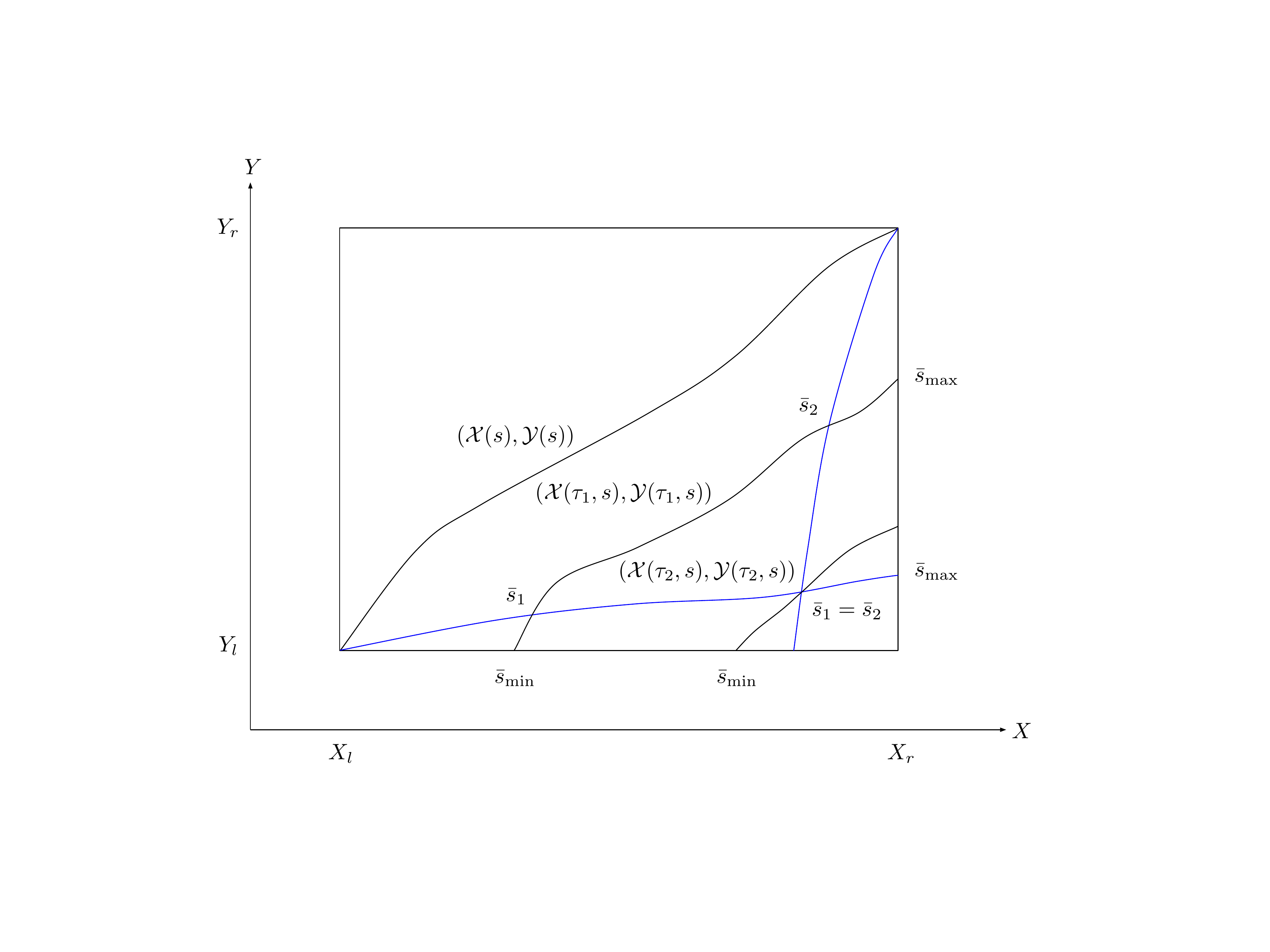}}}
	\caption{The region from Figure \ref{fig:FigThm71i} in Lagrangian coordinates. The curves $(\X(s),\Y(s))$, $(\X(\tau_{1},s),\Y(\tau_{1},s))$ and $(\X(\tau_{2},s),\Y(\tau_{2},s))$ correspond to $t=0$, $t=\tau_{1}$ and $t=\tau_{2}$, respectively.}
	\label{fig:FigThm71ii}
\end{figure}

By differentiating $t(\X(\tau,s),\Y(\tau,s))=\tau$ and using that $\dot{\X}(\tau,s)+\dot{\Y}(\tau,s)=2$, we obtain
\begin{equation}
\label{eq:XtauStrictPos}
	\dot{\X}(\tau,s)=\frac{-2t_{Y}(\X(\tau,s),\Y(\tau,s))}{t_{X}(\X(\tau,s),\Y(\tau,s))-t_{Y}(\X(\tau,s),\Y(\tau,s))}.
\end{equation}
By \eqref{eq:decayinfty}, we have
\begin{equation*}
	|t_{X}(X,Y)|\leq ||Z_{X}^{a}||_{W^{1,\infty}_{Y}(\Omega)}+\frac{\kappa}{2}, \quad |t_{Y}(X,Y)|\leq ||Z_{Y}^{a}||_{W^{1,\infty}_{X}(\Omega)}+\frac{\kappa}{2}
\end{equation*}
for all $(X,Y)\in\Omega$. Using this and \eqref{eq:tXandtYpos} in \eqref{eq:XtauStrictPos}, we find 
\begin{equation}
\label{eq:XdotLowerBound}
	\dot{\X}(\tau,s)\geq \frac{e_{2}^{2}\kappa^{-2}e^{-\tilde{C}|X-\X(Y)|}}{||Z_{X}^{a}||_{W^{1,\infty}_{Y}(\Omega)}+||Z_{Y}^{a}||_{W^{1,\infty}_{X}(\Omega)}+\kappa}=:\alpha_{1}e^{-\tilde{C}|X-\X(Y)|}
\end{equation}
for all $s\in[\bar{s}_{\text{min}},\bar{s}_{\text{max}}]$. Similarly, one proves that
\begin{equation}
\label{eq:YdotLowerBound}
	\dot{\Y}(\tau,s)\geq\alpha_{2}e^{-C|Y-\Y(X)|}
\end{equation}
for some positive constant $\alpha_{2}$ that depends on $d_{2}$, $\kappa$, $||Z_{X}^{a}||_{W^{1,\infty}_{Y}(\Omega)}$ and $||Z_{Y}^{a}||_{W^{1,\infty}_{X}(\Omega)}$.

Next, we prove that there exist $\bar{s}_{1}$ and $\bar{s}_{2}$ satisfying $\bar{s}_{\text{min}}\leq\bar{s}_{1}\leq\bar{s}_{2}\leq\bar{s}_{\text{max}}$ such that
\begin{equation}
\label{eq:s1bars2bar}
	x(\X(\tau,\bar{s}_{1}),\Y(\tau,\bar{s}_{1}))=x_{l}+\kappa\tau \quad \text{and} \quad x(\X(\tau,\bar{s}_{2}),\Y(\tau,\bar{s}_{2}))=x_{r}-\kappa\tau.
\end{equation}
Observe that since $x$ is increasing with respect to both variables this will imply
\begin{equation*}
	x_{l}+\kappa\tau\leq x(\X(\tau,s),\Y(\tau,s))\leq x_{r}-\kappa\tau
\end{equation*}
for all $s\in[\bar{s}_{1},\bar{s}_{2}]$. Furthermore, since $\bar{s}_{\text{min}}\leq\bar{s}_{1}\leq\bar{s}_{2}\leq\bar{s}_{\text{max}}$ and $\X(\tau,\cdot)$ and $\Y(\tau,\cdot)$ are increasing functions, $(\X(\tau,s),\Y(\tau,s))$ belongs to $\Omega$ for all $s\in[\bar{s}_{1},\bar{s}_{2}]$, see Figure \ref{fig:FigThm71i} and \ref{fig:FigThm71ii}.

By \eqref{eq:setH1} and \eqref{eq:cassumption}, we have
\begin{align*}
	x(\X(\tau,\bar{s}_{\text{min}}),Y_{l})&=x(X_{l},Y_{l})+\int_{X_{l}}^{\X(\tau,\bar{s}_{\text{min}})}x_{X}(X,Y_{l})\,dX\\
	&=x(X_{l},Y_{l})+\int_{X_{l}}^{\X(\tau,\bar{s}_{\text{min}})}c(U(X,Y_{l}))t_{X}(X,Y_{l})\,dX\\
	&\leq x(X_{l},Y_{l})+\kappa\int_{X_{l}}^{\X(\tau,\bar{s}_{\text{min}})}t_{X}(X,Y_{l})\,dX\\
	&=x(X_{l},Y_{l})+\kappa(t(\X(\tau,\bar{s}_{\text{min}}),Y_{l})-t(X_{l},Y_{l}))\\
	&=x_{l}+\kappa\tau.
\end{align*}

Similarly, by using the lower bound on $c$, we get
\begin{equation*}
	x(\X(\tau,\bar{s}_{\text{min}}),Y_{l})\geq x_{l}+\frac{1}{\kappa}\tau.
\end{equation*} 
From \eqref{eq:setH1} and \eqref{eq:cassumption}, we have
\begin{align*}
	x_r=x(X_{r},Y_{r})&=x(X_{r},\Y(\tau,\bar{s}_{\text{max}}))+\int_{\Y(\tau,\bar{s}_{\text{max}})}^{Y_{r}}x_{Y}(X_{r},Y)\,dY\\
	&=x(X_{r},\Y(\tau,\bar{s}_{\text{max}}))-\int_{\Y(\tau,\bar{s}_{\text{max}})}^{Y_{r}}c(U(X_{r},Y))t_{Y}(X_{r},Y)\,dY\\
	&\leq x(X_{r},\Y(\tau,\bar{s}_{\text{max}}))-\kappa\int_{\Y(\tau,\bar{s}_{\text{max}})}^{Y_{r}}t_{Y}(X_{r},Y)\,dY \quad \text{since } t_{Y}<0\\
	&=x(X_{r},\Y(\tau,\bar{s}_{\text{max}}))+\kappa\tau,
\end{align*}
so that
\begin{equation*}
	x(X_{r},\Y(\tau,\bar{s}_{\text{max}}))\geq x_{r}-\kappa\tau.
\end{equation*}
Similarly, we obtain
\begin{equation*}
	x(X_{r},\Y(\tau,\bar{s}_{\text{max}}))\leq x_{r}-\frac{1}{\kappa}\tau.
\end{equation*}
Hence, we end up with
\begin{equation}
\label{eq:Xsmin}
	x_{l}+\frac{1}{\kappa}\tau\leq x(\X(\tau,\bar{s}_{\text{min}}),Y_{l})\leq x_{l}+\kappa\tau
\end{equation}
and
\begin{equation}
\label{eq:Ysmax}
	x_{r}-\kappa\tau\leq x(X_{r},\Y(\tau,\bar{s}_{\text{max}}))\leq x_{r}-\frac{1}{\kappa}\tau.
\end{equation}
Since $0<\tau\leq\frac{1}{2\kappa}(x_{r}-x_{l})$ we have $x_{l}+\kappa\tau=x_{l}+2\kappa\tau-\kappa\tau\leq x_{l}+x_{r}-x_{l}-\kappa\tau= x_{r}-\kappa\tau$, which implies, since $x$ is continuous with respect to both variables, that there exist $\bar{s}_{1}$ and $\bar{s}_{2}$ such that $\bar{s}_{1}\leq\bar{s}_{2}$ and
\begin{equation}
\label{eq:s1bars2barProof}
	x(\X(\tau,\bar{s}_{1}),\Y(\tau,\bar{s}_{1}))=x_{l}+\kappa\tau \quad \text{and} \quad x(\X(\tau,\bar{s}_{2}),\Y(\tau,\bar{s}_{2}))=x_{r}-\kappa\tau.
\end{equation}
From \eqref{eq:Xsmin}, \eqref{eq:s1bars2barProof}, and the fact that $x$ increases along the curve $(\X(\tau,s), \Y(\tau,s))$, we have
\begin{equation*}
	x(\X(\tau,\bar{s}_{\text{min}}),Y_{l})\leq x_{l}+\kappa\tau=x(\X(\tau,\bar{s}_{1}),\Y(\tau,\bar{s}_{1})),
\end{equation*}
so that $\X(\tau,\bar{s}_{\text{min}})\leq \X(\tau,\bar{s}_{1})$ which implies $\bar{s}_{\text{min}}\leq \bar{s}_{1}$. By \eqref{eq:Ysmax} and \eqref{eq:s1bars2barProof} we have
\begin{equation*}
	x(\X(\tau,\bar{s}_{2}),\Y(\tau,\bar{s}_{2}))=x_{r}-\kappa\tau\leq x(X_{r},\Y(\tau,\bar{s}_{\text{max}})),
\end{equation*}
and we get $\Y(\tau,\bar{s}_{2})\leq \Y(\tau,\bar{s}_{\text{max}})$ and $\bar{s}_{2}\leq \bar{s}_{\text{max}}$. This concludes the proof of \eqref{eq:s1bars2bar}.

We prove \ref{eq:smoothResult1}. From \eqref{eq:mapGtoD1} we have
\begin{equation*}
	u(\tau,x)=\Z_{3}(\tau,s) \quad \text{if } x=\Z_{2}(\tau,s).
\end{equation*}	
Since the function $\Z_{2}(\tau,s)$ is nondecreasing, the smallest and biggest value 
it can attain for $s\in[\bar{s}_{1},\bar{s}_{2}]$ is
\begin{equation*}
	\Z_{2}(\tau,\bar{s}_{1})=x(\X(\tau,\bar{s}_{1}),\Y(\tau,\bar{s}_{1}))=x_{l}+\kappa\tau
\end{equation*}
and
\begin{equation*}
	\Z_{2}(\tau,\bar{s}_{2})=x(\X(\tau,\bar{s}_{2}),\Y(\tau,\bar{s}_{2}))=x_{r}-\kappa\tau,
\end{equation*}
respectively. The function $\Z_{2}(\tau,\cdot)$ is in fact strictly increasing for $s\in[\bar{s}_{1},\bar{s}_{2}]$, as we now show. We differentiate the relation $\Z_{2}(\tau,s)=x(\X(\tau,s),\Y(\tau,s))$ and get
\begin{equation*}
	\dot{\Z}_{2}(\tau,s)=x_{X}(\X(\tau,s),\Y(\tau,s))\dot{\X}(\tau,s)+ x_{Y}(\X(\tau,s),\Y(\tau,s))\dot{\Y}(\tau,s).
\end{equation*}
From \eqref{eq:xXpos}, \eqref{eq:xYpos}, \eqref{eq:XdotLowerBound} and \eqref{eq:YdotLowerBound} we have
\begin{equation}
\label{eq:Z2bound}
	\dot{\Z}_{2}(\tau,s)\geq \frac{1}{2\kappa}\big(\alpha_{1}d_{2}^{2}+\alpha_{2}e_{2}^{2}\big)
	e^{-C|Y-\Y(X)|-\tilde{C}|X-\X(Y)|}>0.
\end{equation}
Hence, $s\mapsto \Z_{2}(\tau,s)$ is strictly increasing for $s\in[\bar{s}_{1},\bar{s}_{2}]$ and therefore invertible on $[\bar{s}_{1},\bar{s}_{2}]$. For any $x\in[x_{l}+\kappa\tau,x_{r}-\kappa\tau]$ we get
\begin{equation}
\label{eq:utauZ3}
	u(\tau,x)=\Z_{3}(\tau,\Z_{2}^{-1}(\tau,x)),
\end{equation}
and since $|\Z_{3}(\tau,\Z_{2}^{-1}(\tau,x))|=|U(\X(\tau,\Z_{2}^{-1}(\tau,x)),\Y(\tau,\Z_{2}^{-1}(\tau,x)))|\leq ||U||_{L^{\infty}(\Omega)}$ we have
\begin{equation}
\label{eq:uLinf}
	u(\tau,\cdot)\in\Linf([x_{l}+\kappa\tau,x_{r}-\kappa\tau]).
\end{equation}
Next, we differentiate \eqref{eq:utauZ3} and get
\begin{equation}
\label{eq:uxbound}
	u_{x}(\tau,x)=\dot{\Z}_{3}(\tau,\Z_{2}^{-1}(\tau,x))\frac{d}{dx}\Z_{2}^{-1}(\tau,x)
	=\frac{\dot{\Z}_{3}(\tau,\Z_{2}^{-1}(\tau,x))}{\dot{\Z}_{2}(\tau,\Z_{2}^{-1}(\tau,x))}.
\end{equation}
We have
\begin{align*}
	|\dot{\Z}_{3}(\tau,s)|&=|\V_{3}(\tau,\X(\tau,s))\dot{\X}(\tau,s)+\W_{3}(\tau,\Y(\tau,s))\dot{\Y}(\tau,s)|\\
	&=|U_{X}(\X(\tau,s),\Y(\tau,s))\dot{\X}(\tau,s)+U_{Y}(\X(\tau,s),\Y(\tau,s))\dot{\Y}(\tau,s)| \\
	&\leq 2||U_{X}||_{W^{1,\infty}_{Y}(\Omega)}+2||U_{Y}||_{W^{1,\infty}_{X}(\Omega)},
\end{align*}
so that $\dot{\Z}_{3}(\tau,\cdot)\in\Linf([\bar{s}_{1},\bar{s}_{2}])$. By \eqref{eq:Z2bound} we end up with  
\begin{equation*}
	|u_{x}(\tau,x)|\leq\frac{2\kappa||\dot{\Z}_{3}(\tau,\cdot)||_{\Linf([\bar{s}_{1},\bar{s}_{2}])}}{\alpha_{1}d_{2}^{2}+\alpha_{2}e_{2}^{2}}e^{C|Y-\Y(X)|+\tilde{C}|X-\X(Y)|},
\end{equation*}
which implies that
\begin{equation}
\label{eq:uxLinf}
	u_{x}(\tau,\cdot)\in\Linf([x_{l}+\kappa\tau,x_{r}-\kappa\tau]).
\end{equation}
From \eqref{eq:uLinf} and \eqref{eq:uxLinf} we conclude that \ref{eq:smoothResult1} holds.

We prove \ref{eq:smoothResult2}. From \eqref{eq:mapGtoD9} and \eqref{eq:mapGtoD10}, we have
\begin{equation*}
	R(\tau,\Z_{2}(\tau,s))\V_{2}(\tau,\X(\tau,s))=c(\Z_{3}(\tau,s))\V_{3}(\tau,\X(\tau,s))
\end{equation*}	
and	
\begin{equation*}	
	\rho(\tau,\Z_{2}(\tau,s))\V_{2}(\tau,\X(\tau,s))=\p(\tau,\X(\tau,s))
\end{equation*}
for all $s\in[\bar{s}_{1},\bar{s}_{2}]$. We multiply these equations with $2\dot{\X}(\tau,s)$ and use \eqref{eq:setG0rel2} to get
\begin{equation*}
	R(\tau,\Z_{2}(\tau,s))\dot{\Z}_{2}(\tau,s)=2c(\Z_{3}(\tau,s))\V_{3}(\tau,\X(\tau,s))\dot{\X}(\tau,s)
\end{equation*}
and
\begin{equation*}
	\rho(\tau,\Z_{2}(\tau,s))\dot{\Z}_{2}(\tau,s)=2\p(\tau,\X(\tau,s))\dot{\X}(\tau,s),
\end{equation*}
which yields
\begin{equation}
\label{eq:RExpr}
	R(\tau,x)=\frac{2c(\Z_{3}(\tau,\Z_{2}^{-1}(\tau,x)))\V_{3}(\tau,\X(\tau,\Z_{2}^{-1}(\tau,x)))\dot{\X}(\tau,\Z_{2}^{-1}(\tau,x))}{\dot{\Z}_{2}(\tau,\Z_{2}^{-1}(\tau,x))}
\end{equation}
and
\begin{equation}
\label{eq:rhoExpr}
	\rho(\tau,x)=\frac{2\p(\tau,\X(\tau,\Z_{2}^{-1}(\tau,x)))\dot{\X}(\tau,\Z_{2}^{-1}(\tau,x))}{\dot{\Z}_{2}(\tau,\Z_{2}^{-1}(\tau,x))}
\end{equation} 
for all $x\in[x_{l}+\kappa\tau,x_{r}-\kappa\tau]$. Since $\V_{3}(\tau,\X(\tau,s))=U_{X}(\X(\tau,s),\Y(\tau,s))$ and $|\V_{3}(\tau,\X(\tau,s))|\leq ||U_{X}||_{W^{1,\infty}_{Y}(\Omega)}$ we have $\V_{3}(\tau,\X(\tau,\cdot))\in\Linf([\bar{s}_{1},\bar{s}_{2}])$, and since
$\p(\tau,\X(\tau,s))=p(\X(\tau,s),\Y(\tau,s))$ and $|\p(\tau,\X(\tau,s))|\leq ||p||_{W^{1,\infty}_{Y}(\Omega)}$ we have $\p(\tau,\X(\tau,\cdot))\in\Linf([\bar{s}_{1},\bar{s}_{2}])$. Using \eqref{eq:Z2bound} in \eqref{eq:RExpr} and \eqref{eq:rhoExpr} we obtain
\begin{equation*}
	|R(\tau,x)|\leq \frac{8\kappa^{2}||\V_{3}(\tau,\X(\tau,\cdot))||_{\Linf([\bar{s}_{1},\bar{s}_{2}])}}{\alpha_{1}d_{2}^{2}+\alpha_{2}e_{2}^{2}}e^{C|Y-\Y(X)|+\tilde{C}|X-\X(Y)|}
\end{equation*}
and
\begin{equation*}
	|\rho(\tau,x)|\leq \frac{8\kappa||\p(\tau,\X(\tau,\cdot))||_{\Linf([\bar{s}_{1},\bar{s}_{2}])}}{\alpha_{1}d_{2}^{2}+\alpha_{2}e_{2}^{2}}e^{C|Y-\Y(X)|+\tilde{C}|X-\X(Y)|}
\end{equation*}
for all $x\in[x_{l}+\kappa\tau,x_{r}-\kappa\tau]$. Therefore $R(\tau,\cdot),\rho(\tau,\cdot)\in\Linf([x_{l}+\kappa\tau,x_{r}-\kappa\tau])$. In a similar way one shows that $S(\tau,\cdot),\sigma(\tau,\cdot)\in\Linf([x_{l}+\kappa\tau,x_{r}-\kappa\tau])$ and we have proved \ref{eq:smoothResult2}.

We prove \ref{eq:smoothResult3}. By inserting  
\begin{equation*}
	\dot{\Z}_{2}(\tau,\Z_{2}^{-1}(\tau,x))=2x_{X}(\X(\tau,\Z_{2}^{-1}(\tau,x)),\Y(\tau,\Z_{2}^{-1}(\tau,x)))\dot{\X}(\tau,\Z_{2}^{-1}(\tau,x))
\end{equation*}
into \eqref{eq:rhoExpr} we get
\begin{equation*}
	\rho(\tau,x)=\frac{\p(\tau,\X(\tau,\Z_{2}^{-1}(\tau,x)))}{x_{X}(\X(\tau,\Z_{2}^{-1}(\tau,x)),\Y(\tau,\Z_{2}^{-1}(\tau,x)))}
\end{equation*}
for all $s\in[\bar{s}_{1},\bar{s}_{2}]$. Since $p_{Y}=0$ we get from \eqref{eq:pPos},
\begin{equation*}
	\p(\tau,X)=p(X,\Y(\tau,\X^{-1}(\tau,X)))=p(X,\Y(X))\geq d_{2}
\end{equation*}
for all $X\in[X_{l},X_{r}]$. Recalling \eqref{eq:GronIneqxXJX} and that $J_X(X,Y)\geq 0$ we get
\begin{equation*}
	\rho(\tau,x)\geq d_{2}e^{-C\vert Y_l-Y_r\vert}
\end{equation*}
for all $x\in[x_{l}+\kappa\tau,x_{r}-\kappa\tau]$. Similarly we find
\begin{equation*}
	\sigma(\tau,x)\geq e_{2}e^{-\tilde C\vert X_l-X_r\vert}
\end{equation*}
for all $x\in[x_{l}+\kappa\tau,x_{r}-\kappa\tau]$. This concludes the proof of \ref{eq:smoothResult3}.

We prove \ref{eq:smoothResult4}. Let $M\subset[x_{l}+\kappa\tau,x_{r}-\kappa\tau]$ be a Borel set. By \eqref{eq:mapGtoDsing1}, we have
\begin{equation*}
	\mu_{\text{sing}}(\tau,M)=\int_{\Z_{2}^{-1}(\tau,M)\cap A^{c}}\V_{4}(\tau,\X(\tau,s))\dot{\X}(\tau,s)\,ds,
\end{equation*}
where $A=\{s\in\mathbb{R} \ | \ \V_{2}(\tau,\X(\tau,s))>0 \}$. We have $\text{meas}(A^{c})=0$, which implies that
\begin{equation*}
	\mu_{\text{sing}}(\tau,M)\leq 2|||\Theta(\tau)|||_{\G(\Omega)}\,\text{meas}(A^{c})=0.
\end{equation*}
This proves that $\mu(\tau)$ is absolutely continuous on $[x_{l}+\kappa\tau,x_{r}-\kappa\tau]$. Similarly, one shows that $\nu(\tau)$ is absolutely continuous on $[x_{l}+\kappa\tau,x_{r}-\kappa\tau]$. 

\textbf{Step 2.} Assume that $m=2$. By Step 1, \ref{eq:smoothResult3} and \ref{eq:smoothResult4} hold. Moreover, \ref{eq:smoothResult1} and \ref{eq:smoothResult2} hold for $m=1$. It remains to prove that $u_{xx}(\tau,\cdot),R_{x}(\tau,\cdot),S_{x}(\tau,\cdot),\rho_{x}(\tau,\cdot),\sigma_{x}(\tau,\cdot)\in L^{\infty}([x_{l}+\kappa\tau,x_{r}-\kappa\tau])$. In order to do so, we have to show that \\ $Z_{XX}=(t_{XX},x_{XX},U_{XX},J_{XX},K_{XX})$ and $Z_{YY}=(t_{YY},x_{YY},U_{YY},J_{YY},K_{YY})$ exist and are bounded. We first consider $Z_{XX}$. There exists a unique solution $Z_{XX}$ in $\Omega$ of the system
\begin{equation*}
	Z_{XXY}(X,Y)=f(X,Y,Z_{XX})
\end{equation*}
since $f$ is Lipschitz continuous with respect to the $Z_{XX}$ variable, which comes from the fact that the system is semilinear. This can be seen by differentiating the governing equations \eqref{eq:goveq}. For instance, we have
\begin{aalign}
\label{eq:xXXY}
	x_{XXY}&=\frac{c'(U)}{2c(U)}(x_{X}U_{XY}-U_{X}x_{XY})+\frac{c''(U)}{2c(U)}U_{X}(U_{X}x_{Y}+U_{Y}x_{X})\\
	&\quad+\frac{c'(U)}{2c(U)}(x_{Y}U_{XX}+U_{Y}x_{XX})
\end{aalign}
and
\begin{equation}
\label{eq:xXXFundThmCalc}
	|x_{XX}(X,Y)|\leq |x_{XX}(X,\Y(X))|+\int_{Y}^{\Y(X)}|x_{XXY}(X,\tilde{Y})|\,d\tilde{Y},
\end{equation}
if we assume without loss of generality that $Y\leq\Y(X)$ (the other case is similar.)

Let us find a bound on $x_{XX}$ at time $\tau=0$. We differentiate \eqref{eq:x1derexpr} and get
\begin{equation*}
	x_{1}''(X)=-\frac{1}{4}x_{1}'(X)^{3}(2R_{0}R_{0x}+c'(u_{0})u_{0x}\rho_{0}^{2}+2c(u_{0})\rho_{0}\rho_{0x})\circ x_{1}(X),
\end{equation*}
which, by \eqref{eq:regx1UpperBound}, implies that
\begin{align}\label{eq:estx12}
	|x_{1}''(X)|&\leq \frac{1}{4}\bigg(\frac{4\kappa}{d^{2}+4\kappa}\bigg)^{3}(2||R_{0}||_{\Linf([x_{l},x_{r}])} ||(R_{0})_{x}||_{\Linf([x_{l},x_{r}])}\\ \nonumber
	&\hspace{89pt}+k_{1}||(u_{0})_{x}||_{\Linf([x_{l},x_{r}])} ||\rho_{0}||_{\Linf([x_{l},x_{r}])}^{2}\\ \nonumber
	&\hspace{89pt}+2\kappa||\rho_{0}||_{\Linf([x_{l},x_{r}])} ||(\rho_{0})_{x}||_{\Linf([x_{l},x_{r}])})
\end{align}
and we conclude that $x_{1}''\in\Linf([X_{l},X_{r}])$.

By Definition \ref{def:mapfromDtoF} and Definition \ref{def:mapfromFtoG0}, we have
\begin{equation*}
	x_{X}(X,\Y(X))=\V_{2}(X)=\frac{1}{2}x_{1}'(X).
\end{equation*}
We differentiate and get
\begin{equation*}
	x_{XX}(X,\Y(X))+x_{XY}(X,\Y(X))\bigg(\frac{\dot{\Y}}{\dot{\X}}\bigg)\circ \X^{-1}(X)=\frac{1}{2}x_{1}''(X),
\end{equation*}
so that by \eqref{eq:XdotLowerBound},
\begin{align*}
	|x_{XX}(X,\Y(X))|&\leq \frac{1}{2}|x_{1}''(X)|+|x_{XY}(X,\Y(X))|\bigg(\frac{\dot{\Y}}{\dot{\X}}\bigg)\circ \X^{-1}(X)\\
	&\leq \frac{1}{2}||x_{1}''||_{\Linf([X_{l},X_{r}])}+\frac{k_{1}\kappa} {\alpha_{1}e^{-\tilde{C}|X-\X(Y)|}}\bigg(|||\Theta|||_{\G(\Omega)}^{2}+\frac{1}{2}|||\Theta|||_{\G(\Omega)}\bigg)
\end{align*}
and $x_{XX}(\cdot,\Y(\cdot))\in\Linf([X_{l},X_{r}])$. Here we used the estimate
\begin{align*}
	&|x_{XY}(\X(s),\Y(s))|\\
	&=\bigg|\frac{c'(\Z_{3}(s))}{2c(\Z_{3}(s))}(\V_{3}(\X(s))\W_{2}(\Y(s))+\W_{3}(\Y(s))\V_{2}(\X(s)))\bigg| \quad \text{by } \eqref{eq:goveq}\\
	&\leq \frac{1}{2}k_{1}\kappa \bigg(||\V_{3}^{a}||_{\Linf([X_{l},X_{r}])}\bigg(||\W_{2}^{a}||_{\Linf([Y_{l},Y_{r}])}+\frac{1}{2}\bigg)\\
	&\hspace{50pt}+||\W_{3}^{a}||_{\Linf([Y_{l},Y_{r}])}\bigg(||\V_{2}^{a}||_{\Linf([X_{l},X_{r}])}+\frac{1}{2}\bigg)\bigg)\\
	&\leq \frac{1}{2}k_{1}\kappa \bigg(|||\Theta|||_{\G(\Omega)}^{2}+\frac{1}{2}|||\Theta|||_{\G(\Omega)}\bigg).
\end{align*}

We estimate $x_{XXY}$. Since $|Z_{X}^{a}|\leq ||Z_{X}^{a}||_{W^{1,\infty}_{Y}(\Omega)}$ and $|Z_{Y}^{a}|\leq ||Z_{Y}^{a}||_{W^{1,\infty}_{X}(\Omega)}$, we get from \eqref{eq:goveq} that 
\begin{equation}
\label{eq:ZXYeta}
	|Z_{XY}|\leq \eta,
\end{equation}
where $\eta$ depends on $||Z_{X}^{a}||_{W^{1,\infty}_{Y}(\Omega)}$, $||Z_{Y}^{a}||_{W^{1,\infty}_{X}(\Omega)}$, $\kappa$ and $k_{1}$. We obtain from \eqref{eq:xXXY},
\begin{aalign}
\label{eq:xXXYest}
	|x_{XXY}|&\leq \frac{1}{2}k_{1}\kappa\bigg[\bigg(||Z_{X}^{a}||_{W^{1,\infty}_{Y}(\Omega)}+\frac{1}{2}\bigg)\eta+||Z_{X}^{a}||_{W^{1,\infty}_{Y}(\Omega)}\eta\bigg]\\
	&\quad +\frac{1}{2}k_{2}\kappa||Z_{X}^{a}||_{W^{1,\infty}_{Y}(\Omega)}\bigg[||Z_{X}^{a}||_{W^{1,\infty}_{Y}(\Omega)}\bigg(||Z_{Y}^{a}||_{W^{1,\infty}_{X}(\Omega)}+\frac{1}{2}\bigg)\\
	&\hspace{120pt}+||Z_{Y}^{a}||_{W^{1,\infty}_{X}(\Omega)}\bigg(||Z_{X}^{a}||_{W^{1,\infty}_{Y}(\Omega)}+\frac{1}{2}\bigg)\bigg]\\
	&\quad +\frac{1}{2}k_{1}\kappa\bigg[\bigg(||Z_{Y}^{a}||_{W^{1,\infty}_{X}(\Omega)}+\frac{1}{2}\bigg)|U_{XX}|+||Z_{Y}^{a}||_{W^{1,\infty}_{X}(\Omega)}|x_{XX}|\bigg]\\
	&\leq k_{1}\kappa\bigg(||Z_{X}^{a}||_{W^{1,\infty}_{Y}(\Omega)}+\frac{1}{2}\bigg)\eta\\
	&\quad +k_{2}\kappa||Z_{X}^{a}||_{W^{1,\infty}_{Y}(\Omega)}\bigg(||Z_{X}^{a}||_{W^{1,\infty}_{Y}(\Omega)}+\frac{1}{2}\bigg)\bigg(||Z_{Y}^{a}||_{W^{1,\infty}_{X}(\Omega)}+\frac{1}{2}\bigg)\\
	&\quad +\frac{1}{2}k_{1}\kappa\bigg(||Z_{Y}^{a}||_{W^{1,\infty}_{X}(\Omega)}+\frac{1}{2}\bigg)(|U_{XX}|+|x_{XX}|).
\end{aalign}
We insert \eqref{eq:xXXYest} into \eqref{eq:xXXFundThmCalc} and get
\begin{align*}
	&|x_{XX}(X,Y)|\\
	&\leq ||x_{XX}(\cdot,\Y(\cdot))||_{\Linf([X_{l},X_{r}])}+\bigg[
	k_{1}\kappa\bigg(||Z_{X}^{a}||_{W^{1,\infty}_{Y}(\Omega)}+\frac{1}{2}\bigg)\eta\\
	&\quad +k_{2}\kappa||Z_{X}^{a}||_{W^{1,\infty}_{Y}(\Omega)}\bigg(||Z_{X}^{a}||_{W^{1,\infty}_{Y}(\Omega)}+\frac{1}{2}\bigg)\bigg(||Z_{Y}^{a}||_{W^{1,\infty}_{X}(\Omega)}+\frac{1}{2}\bigg)\bigg]|Y_{r}-Y_{l}|\\
	&\quad +\int_{Y}^{\Y(X)}\frac{1}{2}k_{1}\kappa\bigg(||Z_{Y}^{a}||_{W^{1,\infty}_{X}(\Omega)}+\frac{1}{2}\bigg)(|U_{XX}|+|x_{XX}|)(X,\tilde{Y})\,d\tilde{Y}.
\end{align*}
Following the same lines for the other components of $Z_{XX}$, we obtain
\begin{align*}
	&(|t_{XX}|+|x_{XX}|+|U_{XX}|+|J_{XX}|+|K_{XX}|)(X,Y)\\
	&\leq ||Z_{XX}(\cdot,\Y(\cdot))||_{\Linf([X_{l},X_{r}])}+C_{1}|Y_{r}-Y_{l}|\\
	&\quad +\int_{Y}^{\Y(X)}C_{2}(|t_{XX}|+|x_{XX}|+|U_{XX}|+|J_{XX}|+|K_{XX}|)(X,\tilde{Y})\,d\tilde{Y},
\end{align*}
where $C_{1}$ and $C_{2}$ depend on $||Z_{X}^{a}||_{W^{1,\infty}_{Y}(\Omega)}$, $||Z_{Y}^{a}||_{W^{1,\infty}_{X}(\Omega)}$, $\kappa$, $k_{1}$ and $k_{2}$. By Gronwall's lemma, we obtain
\begin{aalign}
\label{eq:ZXXest}
	&(|t_{XX}|+|x_{XX}|+|U_{XX}|+|J_{XX}|+|K_{XX}|)(X,Y)\\
	&\leq (||Z_{XX}(\cdot,\Y(\cdot))||_{\Linf([X_{l},X_{r}])}+C_{1}|Y_{r}-Y_{l}|)e^{C_{2}|Y-\Y(X)|}.
\end{aalign}
A similar procedure yields
\begin{aalign}
\label{eq:ZYYest}
	&(|t_{YY}|+|x_{YY}|+|U_{YY}|+|J_{YY}|+|K_{YY}|)(X,Y)\\
	&\leq (||Z_{YY}(\X(\cdot),\cdot)||_{\Linf([Y_{l},Y_{r}])}+\tilde{C}_{1}|X_{r}-X_{l}|)e^{\tilde{C}_{2}|X-\X(Y)|},
\end{aalign}
where $\tilde{C}_{1}$ and $\tilde{C}_{2}$ depend on $||Z_{X}^{a}||_{W^{1,\infty}_{Y}(\Omega)}$, $||Z_{Y}^{a}||_{W^{1,\infty}_{X}(\Omega)}$, $\kappa$, $k_{1}$ and $k_{2}$.

We also need estimates for $\ddot{\X}(\tau,s)$ and $\ddot{\Y}(\tau,s)$. By \eqref{eq:setG0rel}, we have
\begin{equation*}
	x_{X}(\X(\tau,s),\Y(\tau,s))\dot{\X}(\tau,s)=x_{Y}(\X(\tau,s),\Y(\tau,s))\dot{\Y}(\tau,s).
\end{equation*}
We differentiate and get
\begin{aalign}
\label{eq:xXXandxYYRel}
	&x_{XX}(\X(\tau,s),\Y(\tau,s))\dot{\X}(\tau,s)^{2}+x_{X}(\X(\tau,s),\Y(\tau,s))\ddot{\X}(\tau,s)\\
	&=x_{YY}(\X(\tau,s),\Y(\tau,s))\dot{\Y}(\tau,s)^{2}+x_{Y}(\X(\tau,s),\Y(\tau,s))\ddot{\Y}(\tau,s).
\end{aalign}
Since $\X(\tau,s)+\Y(\tau,s)=2s$, we have $\ddot{\Y}(\tau,s)=-\ddot{\X}(\tau,s)$, so that
\begin{equation*}
	\ddot{\X}(\tau,s)=\frac{x_{YY}(\X(\tau,s),\Y(\tau,s))\dot{\Y}(\tau,s)^{2}-x_{XX}(\X(\tau,s),\Y(\tau,s))\dot{\X}(\tau,s)^{2}}{x_{X}(\X(\tau,s),\Y(\tau,s))+x_{Y}(\X(\tau,s),\Y(\tau,s))}.
\end{equation*}
By \eqref{eq:xXpos}, \eqref{eq:xYpos}, \eqref{eq:ZXXest} and \eqref{eq:ZYYest}, we find
\begin{aalign}
\label{eq:Chi2ndDer1stEst}
	|\ddot{\X}(\tau,s)|&\leq 4\bigg(\frac{|x_{XX}(\X(\tau,s),\Y(\tau,s))|+|x_{YY}(\X(\tau,s),\Y(\tau,s))|}{x_{X}(\X(\tau,s),\Y(\tau,s))+x_{Y}(\X(\tau,s),\Y(\tau,s))}\bigg)\\
	&\leq \frac{4}{\frac{d_{2}^{2}}{2\kappa}e^{-C|Y-\Y(X)|}+\frac{e_{2}^{2}}{2\kappa}e^{-\tilde{C}|X-\X(Y)|}}\\
	&\quad \times ((||Z_{XX}(\cdot,\Y(\cdot))||_{\Linf([X_{l},X_{r}])}+C_{1}|Y_{r}-Y_{l}|)e^{C_{2}|Y-\Y(X)|}\\
	&\hspace{30pt}+(||Z_{YY}(\X(\cdot),\cdot)||_{\Linf([Y_{l},Y_{r}])}+\tilde{C}_{1}|X_{r}-X_{l}|)e^{\tilde{C}_{2}|X-\X(Y)|}).
\end{aalign}
If we define
\begin{align*}
	D&=\big(4\max\{||Z_{XX}(\cdot,\Y(\cdot))||_{\Linf([X_{l},X_{r}])}+C_{1}|Y_{r}-Y_{l}|,\\
	&\hspace{60pt}||Z_{YY}(\X(\cdot),\cdot)||_{\Linf([Y_{l},Y_{r}])}+\tilde{C}_{1}|X_{r}-X_{l}| \}\big)\bigg(\min\Big\{\frac{d_{2}^{2}}{2\kappa},\frac{e_{2}^{2}}{2\kappa}\Big\}\bigg)^{-1}
\end{align*}
and
\begin{equation*}
	\bar{C}=\max\{C,\tilde{C},C_{2},\tilde{C}_{2}\},
\end{equation*}
we get from \eqref{eq:Chi2ndDer1stEst},
\begin{aalign}
\label{eq:Chi2ndDer2ndEst}
	|\ddot{\X}(\tau,s)|&\leq D\bigg(\frac{e^{\bar{C}|X-\X(Y)|}+e^{\bar{C}|Y-\Y(X)|}}{e^{-\bar{C}|X-\X(Y)|}+e^{-\bar{C}|Y-\Y(X)|}}\bigg)\\
	&=De^{\bar{C}(|X-\X(Y)|+|Y-\Y(X)|)}.
\end{aalign}
We conclude that $\ddot{\X}(\tau,\cdot)\in\Linf([\bar{s}_{1},\bar{s}_{2}])$. Here we used that 
\begin{equation*}
	\frac{e^{a}+e^{b}}{e^{-a}+e^{-b}}=e^{a+b}
\end{equation*}
for $a,b\in\mathbb{R}$. We differentiate $\Z_{2}(\tau,s)=x(\X(\tau,s),\Y(\tau,s))$ twice and get, by \eqref{eq:xXXandxYYRel}, that
\begin{align*}
	\ddot{\Z}_{2}(\tau,s)&=x_{XX}(\X(\tau,s),\Y(\tau,s))\dot{\X}(\tau,s)^{2}+x_{X}(\X(\tau,s),\Y(\tau,s))\ddot{\X}(\tau,s)\\
	&\quad+2x_{XY}(\X(\tau,s),\Y(\tau,s))\dot{\X}(\tau,s)\dot{\Y}(\tau,s)\\
	&\quad+x_{YY}(\X(\tau,s),\Y(\tau,s))\dot{\Y}(\tau,s)^{2}+x_{Y}(\X(\tau,s),\Y(\tau,s))\ddot{\Y}(\tau,s)\\
	&=2x_{XX}(\X(\tau,s),\Y(\tau,s))\dot{\X}(\tau,s)^{2}+2x_{X}(\X(\tau,s),\Y(\tau,s))\ddot{\X}(\tau,s)\\
	&\quad+2x_{XY}(\X(\tau,s),\Y(\tau,s))\dot{\X}(\tau,s)\dot{\Y}(\tau,s).
\end{align*}
Using the estimates \eqref{eq:ZXYeta}, \eqref{eq:ZXXest}, and \eqref{eq:Chi2ndDer2ndEst}, we get
\begin{align*}
	|\ddot{\Z}_{2}(\tau,s)|&\leq 8(||Z_{XX}(\cdot,\Y(\cdot))||_{\Linf([X_{l},X_{r}])}+C_{1}|Y_{r}-Y_{l}|)e^{C_{2}|Y-\Y(X)|}\\
	&\quad +2\bigg(||Z_{X}^{a}||_{W^{1,\infty}_{Y}(\Omega)}+\frac{1}{2}\bigg)De^{\bar{C}(|X-\X(Y)|+|Y-\Y(X)|)}+8\eta
\end{align*}
and $\ddot{\Z}_{2}(\tau,\cdot)\in\Linf([\bar{s}_{1},\bar{s}_{2}])$.

We have
\begin{align*}
	\ddot{\Z}_{3}(\tau,s)&=U_{XX}(\X(\tau,s),\Y(\tau,s))\dot{\X}(\tau,s)^{2}+U_{X}(\X(\tau,s),\Y(\tau,s))\ddot{\X}(\tau,s)\\
	&\quad+2U_{XY}(\X(\tau,s),\Y(\tau,s))\dot{\X}(\tau,s)\dot{\Y}(\tau,s)\\
	&\quad+U_{YY}(\X(\tau,s),\Y(\tau,s))\dot{\Y}(\tau,s)^{2}+U_{Y}(\X(\tau,s),\Y(\tau,s))\ddot{\Y}(\tau,s),
\end{align*}
which implies, by \eqref{eq:ZXYeta}, \eqref{eq:ZYYest} and \eqref{eq:Chi2ndDer2ndEst}, that
\begin{align*}
	|\ddot{\Z}_{3}(\tau,s)|&\leq 4(||Z_{XX}(\cdot,\Y(\cdot))||_{\Linf([X_{l},X_{r}])}+C_{1}|Y_{r}-Y_{l}|)e^{C_{2}|Y-\Y(X)|}\\
	&\quad +||Z_{X}^{a}||_{W^{1,\infty}_{Y}(\Omega)}De^{\bar{C}(|X-\X(Y)|+|Y-\Y(X)|)}+8\eta\\
	&\quad +4(||Z_{YY}(\X(\cdot),\cdot)||_{\Linf([Y_{l},Y_{r}])}+\tilde{C}_{1}|X_{r}-X_{l}|)e^{\tilde{C}_{2}|X-\X(Y)|}\\
	&\quad +||Z_{Y}^{a}||_{W^{1,\infty}_{X}(\Omega)}De^{\bar{C}(|X-\X(Y)|+|Y-\Y(X)|)}. 
\end{align*}
Hence, $\ddot{\Z}_{3}(\tau,\cdot)\in\Linf([\bar{s}_{1},\bar{s}_{2}])$.

We compute $u_{xx}$ from \eqref{eq:uxbound} and get
\begin{equation*}
	u_{xx}(\tau,x)=\frac{\ddot{\Z}_{3}(\tau,\Z_{2}^{-1}(\tau,x))}{\dot{\Z}_{2}(\tau,\Z_{2}^{-1}(\tau,x))^{2}}-\frac{\dot{\Z}_{3}(\tau,\Z_{2}^{-1}(\tau,x))\ddot{\Z}_{2}(\tau,\Z_{2}^{-1}(\tau,x))}{\dot{\Z}_{2}(\tau,\Z_{2}^{-1}(\tau,x))^{3}}.
\end{equation*}
By \eqref{eq:Z2bound} we obtain
\begin{align*}
	|u_{xx}(\tau,x)|&\leq\frac{||\ddot{\Z}_{3}(\tau,\cdot)||_{\Linf([\bar{s}_{1},\bar{s}_{2}])}}{\big[\frac{1}{2\kappa}\big(\alpha_{1}d_{2}^{2}+\alpha_{2}e_{2}^{2}\big)e^{-C|Y-\Y(X)|-\tilde{C}|X-\X(Y)|}\big]^{2}}\\
	&\quad+\frac{||\dot{\Z}_{3}(\tau,\cdot)||_{\Linf([\bar{s}_{1},\bar{s}_{2}])}||\ddot{\Z}_{2}(\tau,\cdot)||_{\Linf([\bar{s}_{1},\bar{s}_{2}])}}{\big[\frac{1}{2\kappa}\big(\alpha_{1}d_{2}^{2}+\alpha_{2}e_{2}^{2}\big)e^{-C|Y-\Y(X)|-\tilde{C}|X-\X(Y)|}\big]^{3}},
\end{align*}
and we conclude that $u_{xx}(\tau,\cdot)\in L^{\infty}([x_{l}+\kappa\tau,x_{r}-\kappa\tau])$.

By differentiating \eqref{eq:RExpr} we get
\begin{align*}
	&R_{x}(\tau,x)\\
	&=\frac{2\dot{\X}(\tau,\Z_{2}^{-1}(\tau,x))}{\big[\dot{\Z}_{2}(\tau,\Z_{2}^{-1}(\tau,x))\big]^{2}}\big[c'(\Z_{3}(\tau,\Z_{2}^{-1}(\tau,x)))\dot{\Z}_{3}(\tau,\Z_{2}^{-1}(\tau,x))\V_{3}(\tau,\X(\tau,\Z_{2}^{-1}(\tau,x)))\\
	&\hspace{115pt}+c(\Z_{3}(\tau,\Z_{2}^{-1}(\tau,x)))\dot{\V}_{3}(\tau,\X(\tau,\Z_{2}^{-1}(\tau,x)))\big]\\
	&\quad+\frac{2c(\Z_{3}(\tau,\Z_{2}^{-1}(\tau,x)))\V_{3}(\tau,\X(\tau,\Z_{2}^{-1}(\tau,x)))\ddot{\X}(\tau,\Z_{2}^{-1}(\tau,x))}{\big[\dot{\Z}_{2}(\tau,\Z_{2}^{-1}(\tau,x))\big]^{2}}\\
	&\quad-\frac{2c(\Z_{3}(\tau,\Z_{2}^{-1}(\tau,x)))\V_{3}(\tau,\X(\tau,\Z_{2}^{-1}(\tau,x)))\dot{\X}(\tau,\Z_{2}^{-1}(\tau,x))\ddot{\Z}_{2}(\tau,\Z_{2}^{-1}(\tau,x))}{\big[\dot{\Z}_{2}(\tau,\Z_{2}^{-1}(\tau,x))\big]^{3}},
\end{align*}
where we denote $\dot{\V}_{3}(\tau,\X(\tau,s))=\frac{d}{ds}\V_{3}(\tau,\X(\tau,s))$. Since $\V_{3}(\tau,\X(\tau,s))=\\U_{X}(\X(\tau,s),\Y(\tau,s))$, we have
\begin{equation*}
	\dot{\V}_{3}(\tau,\X(\tau,s))=U_{XX}(\X(\tau,s),\Y(\tau,s))\dot{\X}(\tau,s)+U_{XY}(\X(\tau,s),\Y(\tau,s))\dot{\Y}(\tau,s).
\end{equation*}
From \eqref{eq:ZXYeta} and \eqref{eq:ZXXest} we obtain
\begin{equation*}
	|\dot{\V}_{3}(\tau,\X(\tau,s))|\leq 2\big[(||Z_{XX}(\cdot,\Y(\cdot))||_{\Linf([X_{l},X_{r}])}+C_{1}|Y_{r}-Y_{l}|)e^{C_{2}|Y-\Y(X)|}+\eta\big],
\end{equation*}
where $\eta$ is a constant that depends on $||Z_{X}^{a}||_{W^{1,\infty}_{Y}(\Omega)}$, $||Z_{Y}^{a}||_{W^{1,\infty}_{X}(\Omega)}$, $\kappa$ and $k_{1}$. Therefore we have $\dot{\V}_{3}(\tau,\X(\tau,\cdot))\in\Linf([\bar{s}_{1},\bar{s}_{2}])$. By \eqref{eq:Z2bound} we get
\begin{align*}
	|R_{x}(\tau,x)|&\leq\frac{2}{\big[\frac{1}{2\kappa}\big(\alpha_{1}d_{2}^{2}+\alpha_{2}e_{2}^{2}\big)
	e^{-C|Y-\Y(X)|-\tilde{C}|X-\X(Y)|}\big]^{2}}\\
	&\quad\times\Big(2k_{1}||\dot{\Z}_{3}(\tau,\cdot)||_{\Linf([\bar{s}_{1},\bar{s}_{2}])}||\V_{3}(\tau,\X(\tau,\cdot))||_{\Linf([\bar{s}_{1},\bar{s}_{2}])}\\
	&\hspace{32pt}+2\kappa||\dot{\V}_{3}(\tau,\X(\tau,\cdot))||_{\Linf([\bar{s}_{1},\bar{s}_{2}])}\\
	&\hspace{32pt}+\kappa||\V_{3}(\tau,\X(\tau,\cdot))||_{\Linf([\bar{s}_{1},\bar{s}_{2}])}||\ddot{\X}(\tau,\cdot)||_{\Linf([\bar{s}_{1},\bar{s}_{2}])}\Big)\\
	&\quad+\frac{4\kappa||\V_{3}(\tau,\X(\tau,\cdot))||_{\Linf([\bar{s}_{1},\bar{s}_{2}])}||\ddot{\Z}_{2}(\tau,\cdot)||_{\Linf([\bar{s}_{1},\bar{s}_{2}])}}{\big[\frac{1}{2\kappa}\big(\alpha_{1}d_{2}^{2}+\alpha_{2}e_{2}^{2}\big)
		e^{-C|Y-\Y(X)|-\tilde{C}|X-\X(Y)|}\big]^{3}},
\end{align*}
which implies that $R_{x}(\tau,\cdot)\in L^{\infty}([x_{l}+\kappa\tau,x_{r}-\kappa\tau])$.

We differentiate \eqref{eq:rhoExpr} and get
\begin{align*}
	&\rho_{x}(\tau,x)\\
	&=\frac{2\big[\dot{\p}(\tau,\X(\tau,\Z_{2}^{-1}(\tau,x)))\dot{\X}(\tau,\Z_{2}^{-1}(\tau,x))+\p(\tau,\X(\tau,\Z_{2}^{-1}(\tau,x)))\ddot{\X}(\tau,\Z_{2}^{-1}(\tau,x))\big]}{\big[\dot{\Z}_{2}(\tau,\Z_{2}^{-1}(\tau,x))\big]^{2}}\\
	&\quad-\frac{2\p(\tau,\X(\tau,\Z_{2}^{-1}(\tau,x)))\dot{\X}(\tau,\Z_{2}^{-1}(\tau,x))\ddot{\Z}_{2}(\tau,\Z_{2}^{-1}(\tau,x))}{\big[\dot{\Z}_{2}(\tau,\Z_{2}^{-1}(\tau,x))\big]^{3}},
\end{align*}
where we denote $\dot{\p}(\tau,\X(\tau,s))=\frac{d}{ds}\p(\tau,\X(\tau,s))$. Since $\p(\tau,\X(\tau,s))=\\p(\X(\tau,s),\Y(\tau,s))$ and $p_{Y}(X,Y)=0$ we have
\begin{equation*}
	\dot{\p}(\tau,\X(\tau,s))=p_{X}(\X(\tau,s),\Y(\tau,s))\dot{\X}(\tau,s).
\end{equation*}
Furthermore, by \eqref{eq:mapFtoGp} and \eqref{eq:mapfromDtoF7}, we have
\begin{equation}
\label{eq:pInitial}
	p(X,\Y(X))=\p(X)=H_{1}(X)=\frac{1}{2}\rho_{0}(x_{1}(X))x_{1}'(X).
\end{equation}
We differentiate and get, since $p_{Y}(X,Y)=0$,
\begin{equation*}
	p_{X}(X,Y)=p_{X}(X,\Y(X))=\frac{1}{2}(\rho_{0})_{x}(x_{1}(X))x_{1}'(X)^{2}+\frac{1}{2}\rho_{0}(x_{1}(X))x_{1}''(X).
\end{equation*}
Using \eqref{eq:regx1UpperBound}, this leads to the estimate
\begin{equation*}
	|p_{X}(X,Y)|\leq \frac{1}{2}||(\rho_{0})_{x}||_{\Linf([x_{l},x_{r}])}\bigg(\frac{4\kappa}{d^{2}+4\kappa}\bigg)^{2} +\frac{1}{2}||\rho_{0}||_{\Linf([x_{l},x_{r}])}||x_{1}''||_{\Linf([X_{l},X_{r}])},
\end{equation*}
which implies that
\begin{equation*}
	|\dot{\p}(\tau,\X(\tau,s))|\leq ||(\rho_{0})_{x}||_{\Linf([x_{l},x_{r}])}\bigg(\frac{4\kappa}{d^{2}+4\kappa}\bigg)^{2} +||\rho_{0}||_{\Linf([x_{l},x_{r}])}||x_{1}''||_{\Linf([X_{l},X_{r}])}.
\end{equation*}
Recalling \eqref{eq:estx12}, it follows that $\dot{\p}(\tau,\X(\tau,\cdot))\in\Linf([\bar{s}_{1},\bar{s}_{2}])$. Using \eqref{eq:Z2bound}, we end up with
\begin{align*}
	|\rho_{x}(\tau,x)|&\leq\frac{2\big(2||\dot{\p}(\tau,\X(\tau,\cdot))||_{\Linf([\bar{s}_{1},\bar{s}_{2}])}+||\p(\tau,\X(\tau,\cdot))||_{\Linf([\bar{s}_{1},\bar{s}_{2}])}||\ddot{\X}(\tau,\cdot)||_{\Linf([\bar{s}_{1},\bar{s}_{2}])}\big)}{\big[\frac{1}{2\kappa}\big(\alpha_{1}d_{2}^{2}+\alpha_{2}e_{2}^{2}\big)
	e^{-C|Y-\Y(X)|-\tilde{C}|X-\X(Y)|}\big]^{2}}\\
	&\quad+\frac{4||\p(\tau,\X(\tau,\cdot))||_{\Linf([\bar{s}_{1},\bar{s}_{2}])}||\ddot{\Z}_{2}(\tau,\cdot)||_{\Linf([\bar{s}_{1},\bar{s}_{2}])}}{\big[\frac{1}{2\kappa}\big(\alpha_{1}d_{2}^{2}+\alpha_{2}e_{2}^{2}\big)e^{-C|Y-\Y(X)|-\tilde{C}|X-\X(Y)|}\big]^{3}}
\end{align*}
and we conclude that $\rho_{x}(\tau,\cdot)\in L^{\infty}([x_{l}+\kappa\tau,x_{r}-\kappa\tau])$. 

In a similar way one shows that $S_{x}(\tau,\cdot),\sigma_{x}(\tau,\cdot)\in L^{\infty}([x_{l}+\kappa\tau,x_{r}-\kappa\tau])$.

\textbf{Step 3.} Assume that the result holds for $m=n$, that is, if $u_{0}\in\Linf([x_{l},x_{r}])$ and
$R_{0},S_{0},\rho_{0},\sigma_{0}\in W^{n-1,\infty}([x_{l},x_{r}])$, then
$u(\tau,\cdot)\in W^{n,\infty}([x_{l}+\kappa\tau,x_{r}-\kappa\tau])$, $R(\tau,\cdot)$, $S(\tau,\cdot),\rho(\tau,\cdot)$, $\sigma(\tau,\cdot)\in W^{n-1,\infty}([x_{l}+\kappa\tau,x_{r}-\kappa\tau])$ and \ref{eq:smoothResult3}  and \ref{eq:smoothResult4} hold. We show by induction that the result also holds for $m=n+1$, that is, we assume that $R_{0}$, $S_{0}$, $\rho_{0}$, $\sigma_{0}$ belong $W^{n,\infty}([x_{l},x_{r}])$, and prove that $u(\tau,\cdot)\in W^{n+1,\infty}([x_{l}+\kappa\tau,x_{r}-\kappa\tau])$ and $R(\tau,\cdot),S(\tau,\cdot),\rho(\tau,\cdot),\sigma(\tau,\cdot)\in W^{n,\infty}([x_{l}+\kappa\tau,x_{r}-\kappa\tau])$. 

Since the result holds for $m=n$ we get, following closely the argument used in Step 2 to derive \eqref{eq:ZXXest} and \eqref{eq:ZYYest}, that
\begin{equation}
\label{eq:ZnDerBounded}
	\frac{\partial^{\alpha+\beta}}{\partial X^{\alpha}\partial Y^{\beta}}Z\in[\Linf(\Omega)]^{5}, \quad \alpha,\beta=0,1,\dots,n, \quad \alpha+\beta\leq n.
\end{equation}

Since $R_{0},S_{0},\rho_{0},\sigma_{0}\in W^{n,\infty}([x_{l},x_{r}])$ we get by Definition \ref{def:mapfromDtoF} and Definition \ref{def:mapfromFtoG0} that 
\begin{equation*}
	\frac{\partial^{\alpha+\beta}}{\partial X^{\alpha}\partial Y^{\beta}}Z, \quad \alpha,\beta=0,1,\dots,n+1, \quad \alpha+\beta=n+1
\end{equation*}
is bounded on the curve $(\X(s),\Y(s))$, $s\in[s_{l},s_{r}]$. 

Since the governing equation \eqref{eq:goveq} is semilinear, there exists a unique solution $\frac{\partial^{n+1}}{\partial X^{\alpha}\partial Y^{\beta}}Z$, \newline $\alpha,\beta=0,1,\dots,n+1$, $\alpha+\beta=n+1$ in $\Omega$ of the system
\begin{aalign}
\label{eq:ZnDerExpression}
	&\frac{\partial}{\partial Y}\sum_{\substack{\alpha,\beta=0,1,\dots,n+1 \\ \alpha+\beta=n+1}} \frac{\partial^{n+1}}{\partial X^{\alpha}\partial Y^{\beta}}(t+x+U+J+K)\\
	&=f+\sum_{\substack{\alpha,\beta=0,1,\dots,n+1 \\ \alpha+\beta=n+1}} \bigg< g_{\alpha,\beta},\frac{\partial^{n+1}}{\partial X^{\alpha}\partial Y^{\beta}}Z \bigg>,
\end{aalign}
where $f$ and $g_{\alpha,\beta}$ depend on derivatives up to order $n$. By \eqref{eq:ZnDerBounded}, the functions $f$ and $g_{\alpha,\beta}$ are bounded. Here, $g_{\alpha,\beta}$ denotes $n+1$ five dimensional vectors. To clarify the notation, let us compute \eqref{eq:ZnDerExpression} for $n=2$. We have
\begin{align*}
	&\frac{\partial}{\partial Y}\sum_{\substack{\alpha,\beta=0,1,2,3 \\ \alpha+\beta=3}} \frac{\partial^{3}}{\partial X^{\alpha}\partial Y^{\beta}}(t+x+U+J+K)\\
	&=f+\langle g_{3,0},Z_{XXX} \rangle+\langle g_{2,1},Z_{XXY} \rangle+\langle g_{1,2},Z_{XYY} \rangle+\langle g_{0,3},Z_{YYY} \rangle.
\end{align*}  
By Gronwall's lemma, we obtain 
\begin{equation*}
	\frac{\partial^{n+1}}{\partial X^{\alpha}\partial Y^{\beta}}Z\in[\Linf(\Omega)]^{5}, \quad \alpha,\beta=0,1,\dots,n+1, \quad \alpha+\beta=n+1.
\end{equation*}
This implies, since $x_{X}(\X(\tau,s),\Y(\tau,s))\dot{\X}(\tau,s)=x_{Y}(\X(\tau,s),\Y(\tau,s))\dot{\Y}(\tau,s)$, $\Z_{2}(\tau,s)=x(\X(\tau,s),\Y(\tau,s))$ and $\Z_{3}(\tau,s)=U(\X(\tau,s),\Y(\tau,s))$, that
\begin{equation*}
	\frac{d^{n+1}}{ds^{n+1}}\X(\tau,\cdot),\frac{d^{n+1}}{ds^{n+1}}\Y(\tau,\cdot),\frac{d^{n+1}}{ds^{n+1}}\Z_{2}(\tau,\cdot),\frac{d^{n+1}}{ds^{n+1}}\Z_{3}(\tau,\cdot) \in\Linf([s_{l},s_{r}]).
\end{equation*}
It then follows from \eqref{eq:Z2bound} and \eqref{eq:utauZ3} that
\begin{equation*}
	\frac{\partial^{n+1}}{\partial x^{n+1}}u(\tau,\cdot) \in L^{\infty}([x_{l}+\kappa\tau,x_{r}-\kappa\tau]),
\end{equation*}
and from \eqref{eq:Z2bound} and \eqref{eq:RExpr} we obtain
\begin{equation*}
	\frac{\partial^{n}}{\partial x^{n}}R(\tau,\cdot) \in L^{\infty}([x_{l}+\kappa\tau,x_{r}-\kappa\tau]).
\end{equation*}

By \eqref{eq:pInitial}, $\frac{\partial^{k}}{\partial X^{k}}p(\cdot,\Y(\cdot))$, $k=0,1,\dots,n$, is bounded on the curve $(\X(s),\Y(s))$, $s\in[s_{l},s_{r}]$. Since
\begin{equation*}
	\frac{\partial^{k}}{\partial X^{k}}p(X,\Y(X))=\frac{\partial^{k}}{\partial X^{k}}p(X,Y),
\end{equation*}
we get from \eqref{eq:Z2bound} and \eqref{eq:rhoExpr},
\begin{equation*}
	\frac{\partial^{n}}{\partial x^{n}}\rho(\tau,\cdot) \in L^{\infty}([x_{l}+\kappa\tau,x_{r}-\kappa\tau]).
\end{equation*}

Similarly, one proves that 
\begin{equation*}
	\frac{\partial^{n}}{\partial x^{n}}S(\tau,\cdot),\frac{\partial^{n}}{\partial x^{n}}\sigma(\tau,\cdot) \in L^{\infty}([x_{l}+\kappa\tau,x_{r}-\kappa\tau]).
\end{equation*}
\end{proof}

From Theorem \ref{thm:Regular1st} we obtain the following result.

\begin{corollary}
\label{cor:smooth}
Let $-\infty<x_{l}<x_{r}<\infty$ and consider $(u_{0},R_{0},S_{0},\rho_{0},\sigma_{0},\mu_{0},\nu_{0})\in\D$. Assume that
\begin{itemize}
\item[(A1')]
	$u_{0},R_{0},S_{0},\rho_{0},\sigma_{0}\in C^{\infty}([x_{l},x_{r}])$,
\item[(A2')]
	there are constants $d>0$ and $e>0$ such that $\rho_{0}(x)\geq d$ and $\sigma_{0}(x)\geq e$ for all $x\in[x_{l},x_{r}]$,
\item[(A3')]
	$\mu_{0}$ and $\nu_{0}$ are absolutely continuous on $[x_{l},x_{r}]$,
\item[(A4')]
	$c\in C^{\infty}(\mathbb{R})$ and $c^{(m)}\in L^{\infty}(\mathbb{R})$ for $m=3,4,5,\dots$.
\end{itemize}
For any $\tau\in\big[0,\frac{1}{2\kappa}(x_{r}-x_{l})\big]$ consider  
\begin{equation*}
	(u,R,S,\rho,\sigma,\mu,\nu)(\tau)=\bar{S}_{\tau}(u_{0},R_{0},S_{0},\rho_{0},\sigma_{0},\mu_{0},\nu_{0}).
\end{equation*}
Then
\begin{itemize}
\item[(P1')]
	$u(\tau,\cdot),R(\tau,\cdot),S(\tau,\cdot),\rho(\tau,\cdot),\sigma(\tau,\cdot)\in C^{\infty}([x_{l}+\kappa\tau,x_{r}-\kappa\tau])$,
\item[(P2')]
there are constants $\bar{d}>0$ and $\bar{e}>0$ such that $\rho(\tau,x)\geq \bar{d}$ and $\sigma(\tau,x)\geq \bar{e}$ for all $x\in[x_{l}+\kappa\tau,x_{r}-\kappa\tau]$,
\item[(P3')]
$\mu(\tau,\cdot)$ and $\nu(\tau,\cdot)$ are absolutely continuous on $[x_{l}+\kappa\tau,x_{r}-\kappa\tau]$.
\end{itemize}
For $\tau\in\big[-\frac{1}{2\kappa}(x_{r}-x_{l}),0\big]$, the solution satisfies the same properties on the interval \newline
$\big[x_{l}-\kappa\tau,x_{r}+\kappa\tau\big]$.
\end{corollary}

\subsection{Approximation by Smooth Solutions}
\label{sec:Approx}

The following theorem is our main result. In the proof we use Lemma \ref{lemma:1} and Lemma \ref{lemma:2} which are stated and proved in Section \ref{section:aux}.

\begin{theorem}
	\label{thm:main}
	Let $-\infty<x_{l}<x_{r}<\infty$. Consider $(u_{0},R_{0},S_{0},\rho_{0},\sigma_{0},\mu_{0},\nu_{0})$ and $(u_{0}^{n},R_{0}^{n},S_{0}^{n},\rho_{0}^{n},\sigma_{0}^{n},\mu_{0}^{n},\nu_{0}^{n})$ in $\D$. Assume that for all $n\in\mathbb{N}$,
	
	\begin{itemize}
		\item[\mylabel{eq:approxAssumption1}{(A1'')}]
		$u_{0}, R_{0}, S_{0}, u_{0}^{n}, R_{0}^{n}, S_{0}^{n}, \rho_{0}^{n}, \sigma_{0}^{n}\in C^{\infty}([x_{l},x_{r}])$,
		\item[\mylabel{eq:approxAssumption2}{(A2'')}]
		$\rho_{0}(x)=0$ and $\sigma_{0}(x)=0$ for all $x\in[x_{l},x_{r}]$,
		\item[\mylabel{eq:approxAssumption3}{(A3'')}]
		there are constants $d_{n}>0$ and $e_{n}>0$ such that $\rho_{0}^{n}(x)\geq d_{n}$ and $\sigma_{0}^{n}(x)\geq e_{n}$ for all $x\in[x_{l},x_{r}]$,
		\item[\mylabel{eq:approxAssumption4}{(A4'')}]
		$u_{0}^{n}\rightarrow u_{0}$ in $L^{\infty}([x_{l},x_{r}])$, $R_{0}^{n}\rightarrow R_{0}, \ S_{0}^{n}\rightarrow S_{0}, \ \rho_{0}^{n}\rightarrow \rho_{0}$ and $\sigma_{0}^{n}\rightarrow \sigma_{0}$ in $L^{2}([x_{l},x_{r}])$,
		\item[\mylabel{eq:approxAssumption5}{(A5'')}]
		$\mu_{0}, \ \nu_{0}, \ \mu_{0}^{n} \text{ and } \nu_{0}^{n}$ are absolutely continuous on $[x_{l},x_{r}]$,
		\item[\mylabel{eq:approxAssumption6}{(A6'')}]
		$\mu_{0}((-\infty,x_{l}))=\mu_{0}^{n}((-\infty,x_{l}))$, $\mu_{0}((-\infty,x_{r},))=\mu_{0}^{n}((-\infty,x_{r}))$,\\
		$\nu_{0}((-\infty,x_{l}))=\nu_{0}^{n}((-\infty,x_{l}))$ and  $\nu_{0}((-\infty,x_{r},))=\nu_{0}^{n}((-\infty,x_{r}))$,
		\item[\mylabel{eq:approxAssumption7}{(A7'')}]
		$c\in C^{\infty}(\mathbb{R})$ and $c^{(m)}\in L^{\infty}(\mathbb{R})$ for $m=3,4,5,\dots$.		
	\end{itemize}
	For any $\tau\in\big[0,\frac{1}{2\kappa}(x_{r}-x_{l})\big]$ consider
	\begin{equation*}
	(u^{n},R^{n},S^{n},\rho^{n},\sigma^{n},\mu^{n},\nu^{n})(\tau)=\bar{S}_{\tau}(u^{n}_{0},R^{n}_{0},S^{n}_{0},\rho^{n}_{0},\sigma^{n}_{0},\mu^{n}_{0},\nu^{n}_{0})
	\end{equation*}
	and 
	\begin{equation*}
	(u,R,S,\rho,\sigma,\mu,\nu)(\tau)=\bar{S}_{\tau}(u_{0},R_{0},S_{0},\rho_{0},\sigma_{0},\mu_{0},\nu_{0}).
	\end{equation*}
	Then we have
	\begin{itemize}
		\item[\mylabel{eq:approxResult1}{(P1'')}]
		$u^{n}(\tau,\cdot)\rightarrow u(\tau,\cdot)$ in $\Linf([x_{l}+\kappa\tau,x_{r}-\kappa\tau])$,
		\item[\mylabel{eq:approxResult2}{(P2'')}]
		$\rho^{n}(\tau,\cdot)\rightarrow 0$ and $\sigma^{n}(\tau,\cdot)\rightarrow 0$ in $L^{1}([x_{l}+\kappa\tau,x_{r}-\kappa\tau])$.
	\end{itemize}
	The same conclusion holds on the interval $\big[x_{l}-\kappa\tau,x_{r}+\kappa\tau\big]$ in the case where $\tau\in\big[-\frac{1}{2\kappa}(x_{r}-x_{l}),0\big]$.
\end{theorem}

Observe that we assume in \ref{eq:approxAssumption4} $u_{0}^{n}\rightarrow u_{0}$ in $L^{\infty}([x_{l},x_{r}])$, in contrast to the $L^2(\mathbb{R})$ convergence in the global case. This is because functions in $H^{1}(\mathbb{R})$ are zero at infinity, while here we have no assumptions on the values of $u_{0}$ and $u_{0}^{n}$ at the endpoints of $[x_{l},x_{r}]$. Since $[x_{l},x_{r}]$ is a bounded interval, the convergence in $L^{\infty}([x_{l},x_{r}])$ implies that $u_{0}^{n}\rightarrow u_{0}$ in $L^{2}([x_{l},x_{r}])$.

By the assumptions \ref{eq:approxAssumption5} and \ref{eq:approxAssumption6} we mean that
\begin{equation*}
\mu_{0}([x_{l},x_{r}])=\mu_{0}^{n}([x_{l},x_{r}]) \quad \text{and} \quad \nu_{0}([x_{l},x_{r}])=\nu_{0}^{n}([x_{l},x_{r}])
\end{equation*}
for all $n$.

Note that the approximating sequence in Theorem \ref{thm:main} is smooth. Indeed, by Corollary \ref{cor:smooth} we have $u^{n}(\tau,\cdot),R^{n}(\tau,\cdot),S^{n}(\tau,\cdot),\rho^{n}(\tau,\cdot),\sigma^{n}(\tau,\cdot)\in C^{\infty}([x_{l}+\kappa\tau,x_{r}-\kappa\tau])$. Furthermore, there are constants $\bar{d}_{n}>0$ and $\bar{e}_{n}>0$ such that $\rho^{n}(\tau,x)\geq \bar{d}_{n}$ and $\sigma^{n}(\tau,x)\geq \bar{e}_{n}$ for all $x\in[x_{l}+\kappa\tau,x_{r}-\kappa\tau]$, and $\mu^{n}(\tau,\cdot)$ and $\nu^{n}(\tau,\cdot)$ are absolutely continuous on $[x_{l}+\kappa\tau,x_{r}-\kappa\tau]$ for all $n$. 
However, the limit solution does not in general satisfy these properties. Of course, we have that $(u,R,S,\rho,\sigma,\mu,\nu)(\tau)$ belongs to $\D$, but the functions are not necessarily smooth and the measures are not necessarily absolutely continuous. We illustrate this with an example. Consider the function
\begin{equation*}
	f_{n}(x)=\bigg(\frac{2}{\pi}\bigg)^{\frac{1}{4}}\sqrt{n}e^{-(nx)^{2}}.
\end{equation*}
We have $||f_{n}||_{L^{2}(\mathbb{R})}=1$, $f_{n}(x)\rightarrow 0$ for $x\neq 0$ and $f_{n}(0)\rightarrow +\infty$,
and $f_{n}\rightarrow 0$ in $L^{1}(\mathbb{R})$ as $n\to\infty$. By standard calculations we get that
\begin{equation*}
	\lim_{n\rightarrow\infty}\int_{\mathbb{R}}\phi(x)f_{n}(x)\,dx=0
\end{equation*}
and
\begin{equation*}
	\lim_{n\rightarrow\infty}\int_{\mathbb{R}}\phi(x)f_{n}^{2}(x)\,dx=\phi(0)
\end{equation*}
for all $\phi\in C_{c}^{\infty}(\mathbb{R})$. In other words, $f_{n}\overset{*}{\rightharpoonup}0$ and $f_{n}^{2}\,dx\overset{*}{\rightharpoonup}\delta_{0}$, where $\delta_{0}$ is the Dirac delta at zero, which is a singular measure. 
Returning to our setting, since $\mu^{n}(\tau,\cdot)$ is absolutely continuous on $[x_{l}+\kappa\tau,x_{r}-\kappa\tau]$ we have
\begin{equation*}
	\mu^{n}(\tau,[x_{l}+\kappa\tau,x_{r}-\kappa\tau])=\frac{1}{4}\int_{x_{l}+\kappa\tau}^{x_{r}-\kappa\tau}[(R^{n})^{2}+c(u^{n})(\rho^{n})^{2}](\tau,x)\,dx.
\end{equation*}
We can think of $\sqrt{c(u^{n}(\tau,x))}\rho^{n}(\tau,x)$ as the function $f_{n}$ in the example above. We then have that $\sqrt{c(u^{n}(\tau,\cdot))}\rho^{n}(\tau,\cdot)$ is in $C^{\infty}([x_{l}+\kappa\tau,x_{r}-\kappa\tau])$ and satisfies $\sqrt{c(u^{n}(\tau,\cdot))}\rho^{n}(\tau,\cdot)\rightarrow 0$ in $L^{1}([x_{l}+\kappa\tau,x_{r}-\kappa\tau])$, but $c(u^n(\tau,\cdot))(\rho^n)^2(\tau,\cdot)\overset{*}{\rightharpoonup}\delta_{0}$.

\begin{proof}
	We will only consider the case $0<\tau\leq \frac{1}{2\kappa}(x_{r}-x_{l})$. The case $-\frac{1}{2\kappa}(x_{r}-x_{l})\leq\tau<0$ can be treated in the same way.
	
	We split the proof into four steps.
	
	\textbf{Step 1.} Set
	\begin{equation*}
		(\psi_{1},\psi_{2})=\mathbf{L}(u_{0},R_{0},S_{0},\rho_{0},\sigma_{0},\mu_{0},\nu_{0})
	\end{equation*}	
	and
	\begin{equation*}	
		(\psi_{1}^{n},\psi_{2}^{n})=\mathbf{L}(u^{n}_{0},R^{n}_{0},S^{n}_{0},\rho^{n}_{0},\sigma^{n}_{0},\mu^{n}_{0},\nu^{n}_{0}),
	\end{equation*}
	where 
	\begin{equation*}
		\psi_{1}=(x_{1},U_{1},J_{1},K_{1},V_{1},H_{1}), \quad \psi_{2}=(x_{2},U_{2},J_{2},K_{2},V_{2},H_{2})
	\end{equation*}
	and
	\begin{equation*}	
		 \psi_{1}^{n}=(x_{1}^{n},U_{1}^{n},J_{1}^{n},K_{1}^{n},V_{1}^{n},H_{1}^{n}), \quad \psi_{2}^{n}=(x_{2}^{n},U_{2}^{n},J_{2}^{n},K_{2}^{n},V_{2}^{n},H_{2}^{n}). 	
	\end{equation*}
	Let us find out what kind of region the interval $[x_{l},x_{r}]$ corresponds to in Lagrangian coordinates $(X,Y)$. Since the measures are assumed to be absolutely continuous we get from \eqref{eq:mapfromDtoF1},
	\begin{equation*}
	x_{1}(X)+\mu_{0}((-\infty,x_{1}(X)))=X
	\end{equation*}
	for all $x_{1}(X)\in[x_{l},x_{r}]$. We show which range of the $X$-variable this corresponds to. If $\hat{X}$ is such that $x_{1}(\hat{X})=x_{l}$ then
	\begin{equation*}
	x_{l}+\mu_{0}((-\infty,x_{l}))=\hat{X}.
	\end{equation*}
	We also have
	\begin{equation*}
	x_{1}^{n}(X)+\mu_{0}^{n}((-\infty,x_{1}^{n}(X)))=X
	\end{equation*}
	for all $x_{1}^{n}(X)\in[x_{l},x_{r}]$. If $\check{X}$ is such that $x_{1}^{n}(\check{X})=x_{l}$ we get
	\begin{equation*}
	x_{l}+\mu_{0}^{n}((-\infty,x_{l}))=\check{X}.
	\end{equation*}
	Using \ref{eq:approxAssumption6} we obtain
	\begin{equation*}
	\hat{X}=\check{X}
	\end{equation*}
	and we denote $X_{l}=\hat{X}=\check{X}$. 
	
	If $\bar{X}$ and $\tilde{X}$ are such that $x_{1}(\bar{X})=x_{1}^{n}(\tilde{X})=x_{r}$ then
	\begin{equation*}
	x_{r}+\mu_{0}((-\infty,x_{r}))=\bar{X}
	\end{equation*}
	and
	\begin{equation*}
	x_{r}+\mu_{0}^{n}((-\infty,x_{r}))=\tilde{X},
	\end{equation*}
	and by \ref{eq:approxAssumption6} we get
	\begin{equation*}
	\bar{X}=\tilde{X}
	\end{equation*}
	which we denote $X_{r}=\bar{X}=\tilde{X}$. We have $X_{l}\leq X_{r}$ since
	\begin{equation*}
	X_{l}=x_{l}+\mu_{0}((-\infty,x_{l}))\leq x_{r}+\mu_{0}((-\infty,x_{r}))=X_{r}.
	\end{equation*}
	In other words, we can define the interval $[X_{l},X_{r}]$ using either the measure $\mu_{0}$ or $\mu_{0}^{n}$, and observe that this is a consequence of the assumptions \ref{eq:approxAssumption5} and \ref{eq:approxAssumption6}.
	
	In a similar way we find that
	\begin{equation*}
	x_{2}(Y)+\nu_{0}((-\infty,x_{2}(Y)))=Y \quad \text{and} \quad x_{2}^{n}(Y)+\nu_{0}^{n}((-\infty,x_{2}^{n}(Y)))=Y
	\end{equation*}
	for all $Y\in [Y_{l},Y_{r}]$ where
	\begin{equation*}
	Y_{l}=x_{l}+\nu_{0}((-\infty,x_{l}))\quad \text{and} \quad Y_{r}=x_{r}+\nu_{0}((-\infty,x_{r})).
	\end{equation*}
	We denote $\Omega=[X_{l},X_{r}]\times[Y_{l},Y_{r}]$.
	
	Following closely the proof of Lemma \ref{lemma:1}, we obtain that $x_{1}$, $x_{1}^{n}$, $x_{2}$ and $x_{2}^{n}$ are strictly increasing for $X\in [X_l, X_r]$ and $Y\in [Y_l,Y_r]$, respectively, and
	\begin{equation*}
	x_{1}^{n}\rightarrow x_{1}, \quad (x_{1}^{n})^{-1}\rightarrow x_{1}^{-1}, \quad U_{1}^{n}\rightarrow U_{1}, \quad J_{1}^{n}\rightarrow J_{1} \quad \text{in } L^{\infty}([X_{l},X_{r}]),
	\end{equation*}
	\begin{equation*}
	x_{2}^{n}\rightarrow x_{2}, \quad (x_{2}^{n})^{-1}\rightarrow x_{2}^{-1}, \quad U_{2}^{n}\rightarrow U_{2}, \quad J_{2}^{n}\rightarrow J_{2} \quad \text{in } L^{\infty}([Y_{l},Y_{r}]),
	\end{equation*}
	\begin{equation*}
     V_{1}^{n}\rightarrow V_{1}, \quad H_{1}^{n}\rightarrow H_{1}, \quad (x_{1}^{n})'\rightarrow x_{1}', \quad (J_{1}^{n})'\rightarrow J_{1}', \quad (K_{1}^{n})'\rightarrow K_{1}' \quad \text{in } L^{2}([X_{l},X_{r}]),
	\end{equation*}
	\begin{equation*}
	 V_{2}^{n}\rightarrow V_{2}, \quad H_{2}^{n}\rightarrow H_{2}, \quad (x_{2}^{n})'\rightarrow x_{2}', \quad(J_{2}^{n})'\rightarrow J_{2}', \quad (K_{2}^{n})'\rightarrow K_{2}' \quad \text{in } L^{2}([Y_{l},Y_{r}]).
	\end{equation*}
	Note that \eqref{eq:mapfromDtoF6} does not imply $K_{1}^{n}\rightarrow K_{1}$ in $L^{\infty}([X_{l},X_{r}])$ and $K_{2}^{n}\rightarrow K_{2}$ in $L^{\infty}([Y_{l},Y_{r}])$. However, we have $K_{1}^{n}-K_{1}^{n}(X_{l})\rightarrow K_{1}-K_{1}(X_{l})$ in $L^{\infty}([X_{l},X_{r}])$ and $K_{2}^{n}-K_{2}^{n}(Y_{l})\rightarrow K_{2}-K_{2}(Y_{l})$ in $L^{\infty}([Y_{l},Y_{r}])$. The proof of this closely follows the procedure in \eqref{eq:K1Diff}.
	
	\textbf{Step 2.} Let 
	\begin{equation*}
		\Theta=(\X,\Y,\Z,\V,\W,\p,\q)=\mathbf{C}(\psi_{1},\psi_{2})
	\end{equation*}
	and  
	\begin{equation*}
		\Theta^{n}=(\X^{n},\Y^{n},\Z^{n},\V^{n},\W^{n},\p^{n},\q^{n})=\mathbf{C}(\psi_{1}^{n},\psi_{2}^{n}).
	\end{equation*}
	From \eqref{eq:mapFtoGX} we have
	\begin{equation*}
	x_{1}(\X(s))=x_{2}(2s-\X(s)) \quad \text{and} \quad x_{1}^{n}(\X^{n}(s))=x_{2}^{n}(2s-\X^{n}(s)).
	\end{equation*}
	Since we only consider the functions $x_{1}$ and $x_{1}^{n}$ on $[X_{l},X_{r}]$, and $x_{2}$ and $x_{2}^{n}$ on $[Y_{l},Y_{r}]$, the relevant values of $s$ are those satisfying $\X(s)\in[X_{l},X_{r}]$ and $2s-\X(s)\in[Y_{l},Y_{r}]$. 
In other words, since $x_1(X_l)=x_2(Y_l)=x_l$ and $x_1(X_r)=x_2(Y_r)=x_r$, we have 	
	\begin{equation*}
	\frac{1}{2}(X_{l}+Y_{l})\leq s\leq \frac{1}{2}(X_{r}+Y_{r})
	\end{equation*}
	and we call $s_{l}=\frac{1}{2}(X_{l}+Y_{l})$ and $s_{r}=\frac{1}{2}(X_{r}+Y_{r})$. The same conclusion holds for the $s$-values of $(\X^{n}(s),\Y^{n}(s))$. The points $s_{l}$ and $s_{r}$ correspond to the two diagonal points in the box for both curves, i.e., $(\X(s_l),\Y(s_l))=(X_l,Y_l)=(\X^n(s_l),\Y^n(s_l))$ and $(\X(s_r),\Y(s_r))=(X_r,Y_r)=(\X^n(s_r),\Y^n(s_r))$. 	
	
	Following closely the proof of Lemma \ref{lemma:2}, we obtain that $\X$, $\Y$, $\X^{n}$ and $\Y^{n}$ are strictly increasing on $[s_l, s_r]$, and
	\begin{equation*}
	\X^{n}\rightarrow\X, \quad \Y^{n}\rightarrow\Y, \quad \Z_{j}^{n}\rightarrow\Z_{j} \quad \text{in } \Linf([s_{l},s_{r}]),
	\end{equation*}
	\begin{equation*}
	\V_{i}^{n}\rightarrow\V_{i}, \quad \p^{n}\rightarrow\p \quad \text{in } L^{2}([X_{l},X_{r}]),
	\end{equation*}
	\begin{equation*}
	\W_{i}^{n}\rightarrow\W_{i}, \quad \q^{n}\rightarrow\q \quad \text{in } L^{2}([Y_{l},Y_{r}])
	\end{equation*}
	for $i=1,\dots,5$, $j=1,\dots,4$. From \eqref{eq:mapFtoGp} and \eqref{eq:mapfromDtoF7} we have
	\begin{equation*}
		\p(X)=H_{1}(X)=\frac{1}{2}\rho_{0}(x_{1}(X))x_{1}'(X)
	\end{equation*}
	and
	\begin{equation*}
		\p^{n}(X)=H_{1}^{n}(X)=\frac{1}{2}\rho_{0}^{n}(x_{1}^{n}(X))(x_{1}^{n})'(X)
	\end{equation*}
	for all $X\in[X_{l},X_{r}]$. Using \ref{eq:approxAssumption2}, \ref{eq:approxAssumption3} and \eqref{eq:regx1strictincreasing} we get for all $X\in[X_{l},X_{r}]$, $\p(X)=0$ and $\p^{n}(X)\geq k_{n}>0$ for some constant $k_{n}$.

	Since $K_{1}^{n}-K_{1}^{n}(X_{l})\rightarrow K_{1}-K_{1}(X_{l})$ in $L^{\infty}([X_{l},X_{r}])$ and $K_{2}^{n}-K_{2}^{n}(Y_{l})\rightarrow K_{2}-K_{2}(Y_{l})$ in $L^{\infty}([Y_{l},Y_{r}])$ we get from \eqref{eq:mapFtoGZ5} that $\Z_{5}^{n}-\Z_{5}^{n}(s_{l})\rightarrow\Z_{5}-\Z_{5}(s_{l})$ in $\Linf([s_{l},s_{r}])$.
	
	\textbf{Step 3.} Consider $(Z,p,q)=\mathbf{S}(\Theta)$ and $(Z^{n},p^{n},q^{n})=\mathbf{S}(\Theta^{n})$. We prove a Gronwall type estimate. We claim that for all $(X,Y)$ in $\Omega$,
	\begin{aalign}
	\label{eq:Gron}
	&\bigg[\big[Z_{3}-Z_{3}^{n}\big]^{2}+\sum_{i=1}^{5}\Big(\big[Z_{i,X}-Z_{i,X}^{n}\big]^{2}+\big[Z_{i,Y}-Z_{i,Y}^{n}\big]^{2}\Big)\bigg](X,Y)\\
	&\leq K\Bigg\{
	||U_{1}-U_{1}^{n}||_{L^{\infty}([X_{l},X_{r}])}^{2}\\
	&\hspace{35pt}+\sum_{j=1}^{5}\Big(\big[\V_{j}(X)-\V_{j}^{n}(X)\big]^{2}+\big[\W_{j}(Y)-\W_{j}^{n}(Y)\big]^{2}\Big)\\
	&\hspace{35pt}+\big[\Y\circ\X^{-1}(X)-\Y^{n}\circ(\X^{n})^{-1}(X)\big]^{2}\\
	&\hspace{35pt}+\big[\X\circ\Y^{-1}(Y)-\X^{n}\circ(\Y^{n})^{-1}(Y)\big]^{2}\Bigg\},
	\end{aalign}
	where $K$ depends on $\kappa$, $k_{1}$, $k_{2}$, $|||\Theta|||_{\G(\Omega)}$ and the size of $\Omega$.

	Let $(X,Y)\in \Omega$. Subtracting the equations
	\begin{equation*}
	t_{X}(X,Y)=t_{X}(X,\Y(X))+\int_{\Y(X)}^{Y}t_{XY}(X,\tilde{Y})\,d\tilde{Y}
	\end{equation*}
	and
	\begin{equation}
	\label{eq:tXnIntegral}
	t^{n}_{X}(X,Y)=t^{n}_{X}(X,\Y^{n}(X))+\int_{\Y^{n}(X)}^{Y}t^{n}_{XY}(X,\tilde{Y})\,d\tilde{Y}
	\end{equation}
	yields
	\begin{aalign}
	\label{eq:tXtXnDiff}
	t_{X}(X,Y)-t^{n}_{X}(X,Y)&=\V_{1}(X)-\V_{1}^{n}(X)-\int_{\Y^{n}(X)}^{\Y(X)}t^{n}_{XY}(X,\tilde{Y})\,d\tilde{Y}\\
	&\quad+\int_{\Y(X)}^{Y}(t_{XY}(X,\tilde{Y})-t^{n}_{XY}(X,\tilde{Y}))\,d\tilde{Y}.
	\end{aalign}
	Using \eqref{eq:goveqt} we get
	\begin{aalign}
	\label{eq:tXYtXYnDiff1}
	&t_{XY}-t^{n}_{XY}\\
	&=-\frac{c'(U)}{2c(U)}(U_{X}t_{Y}+U_{Y}t_{X})+\frac{c'(U^{n})}{2c(U^{n})}(U_{X}^{n}t_{Y}^{n}+U_{Y}^{n}t_{X}^{n})\\
	&=-\frac{c'(U)}{2c(U)}\Big(U_{X}(t_{Y}-t_{Y}^{n})+t_{Y}^{n}(U_{X}-U_{X}^{n})+U_{Y}(t_{X}-t_{X}^{n})+t_{X}^{n}(U_{Y}-U_{Y}^{n})\Big)\\
	&\quad -\frac{1}{2}(U_{X}^{n}t_{Y}^{n}+U_{Y}^{n}t_{X}^{n})\int_{U^{n}}^{U}\bigg(\frac{c''(V)}{c(V)}-\frac{c'(V)^{2}}{c(V)^2}\bigg)\,dV
	\end{aalign}
	where we used
	\begin{equation*}
	\frac{d}{dV}\bigg(\frac{c'(V)}{c(V)}\bigg)=\frac{c''(V)}{c(V)}-\frac{c'(V)^{2}}{c(V)^2}.
	\end{equation*}
	
	We need a pointwise uniform bound on the components of $Z_{X}^{n}$ and $Z_{Y}^{n}$. This will be done in the same way as in Lemma \ref{lemma:stripest}. We show the details here for completeness. By \eqref{eq:goveqt} we get
	\begin{equation*}
	|t_{XY}^{n}|\leq\frac{1}{2}k_{1}\kappa\Big(|t_{X}^{n}|+|x_{X}^{n}|+|U_{X}^{n}|+|J_{X}^{n}|+|K_{X}^{n}|\Big)\Big(|t_{Y}^{n}|+|x_{Y}^{n}|+|U_{Y}^{n}|+|J_{Y}^{n}|+|K_{Y}^{n}|\Big),
	\end{equation*}
	and by doing the same kind of estimate for the other components we get
	\begin{aalign}
	\label{eq:ZXYnEst}
	&\Big(|t_{XY}^{n}|+|x_{XY}^{n}|+|U_{XY}^{n}|+|J_{XY}^{n}|+|K_{XY}^{n}|\Big)\\
	&\leq B_{1}\Big(|t_{X}^{n}|+|x_{X}^{n}|+|U_{X}^{n}|+|J_{X}^{n}|+|K_{X}^{n}|\Big)\\
	&\quad\times\Big(|t_{Y}^{n}|+|x_{Y}^{n}|+|U_{Y}^{n}|+|J_{Y}^{n}|+|K_{Y}^{n}|\Big)
	\end{aalign}
	for a constant $B_{1}$ that only depends on $\kappa$ and $k_{1}$. By \eqref{eq:tXnIntegral} and the corresponding expressions for $x^{n}_{X}$, $U^{n}_{X}$, $J^{n}_{X}$ and $K^{n}_{X}$ we obtain 
	\begin{aalign}
	\label{eq:tXnIntermediateIneq}
	&\Big(|t_{X}^{n}|+|x_{X}^{n}|+|U_{X}^{n}|+|J_{X}^{n}|+|K_{X}^{n}|\Big)(X,Y)\\
	&\leq \Big(|t_{X}^{n}|+|x_{X}^{n}|+|U_{X}^{n}|+|J_{X}^{n}|+|K_{X}^{n}|\Big)(X,\Y^{n}(X))\\
	&\quad +\bigg|\int_{\Y^{n}(X)}^{Y}\Big(|t_{XY}^{n}|+|x_{XY}^{n}|+|U_{XY}^{n}|+|J_{XY}^{n}|+|K_{XY}^{n}|\Big)(X,\tilde{Y})\,d\tilde{Y}\bigg|\\
	&\leq \Big(|t_{X}^{n}|+|x_{X}^{n}|+|U_{X}^{n}|+|J_{X}^{n}|+|K_{X}^{n}|\Big)(X,\Y^{n}(X))\\
	&\quad +\bigg|\int_{\Y^{n}(X)}^{Y}B_{1}\bigg\{\Big(|t_{X}^{n}|+|x_{X}^{n}|+|U_{X}^{n}|+|J_{X}^{n}|+|K_{X}^{n}|\Big)\\
	&\hspace{86pt}\times\Big(|t_{Y}^{n}|+|x_{Y}^{n}|+|U_{Y}^{n}|+|J_{Y}^{n}|+|K_{Y}^{n}|\Big)\bigg\}(X,\tilde{Y})\,d\tilde{Y}\bigg|,
	\end{aalign}
	where we used \eqref{eq:ZXYnEst}. By Gronwall's inequality,
	\begin{align*}
	&\Big(|t_{X}^{n}|+|x_{X}^{n}|+|U_{X}^{n}|+|J_{X}^{n}|+|K_{X}^{n}|\Big)(X,Y)\\
	&\leq \Big(|t_{X}^{n}|+|x_{X}^{n}|+|U_{X}^{n}|+|J_{X}^{n}|+|K_{X}^{n}|\Big)(X,\Y^{n}(X))\\
	&\quad\times\exp\Bigg\{B_{1}\bigg|\int_{\Y^{n}(X)}^{Y}\Big(|t_{Y}^{n}|+|x_{Y}^{n}|+|U_{Y}^{n}|+|J_{Y}^{n}|+|K_{Y}^{n}|\Big)(X,\tilde{Y})\,d\tilde{Y}\bigg|\Bigg\}.
	\end{align*}
	By \eqref{eq:setH1}, \eqref{eq:setH2}, \eqref{eq:setH4} and \eqref{eq:setH5} we get
	\begin{aalign}
	\label{eq:tYnIntermediateIneq}
	|t_{Y}^{n}|+|x_{Y}^{n}|+|U_{Y}^{n}|+|J_{Y}^{n}|+|K_{Y}^{n}|&=\frac{1}{c(U^{n})}x_{Y}^{n}+x_{Y}^{n}+|U_{Y}^{n}|+J_{Y}^{n}+\frac{1}{c(U^{n})}J_{Y}^{n}\\
	&\leq(1+\kappa)(x_{Y}^{n}+J_{Y}^{n})+|U_{Y}^{n}|.
	\end{aalign}
	From \eqref{eq:setH3} we have
	\begin{equation*}
	2J_{Y}^{n}x_{Y}^{n}=(c(U^{n})U_{Y}^{n})^{2}+c(U^{n})(q^{n})^{2}\geq (c(U^{n})U_{Y}^{n})^{2},
	\end{equation*}
	and by Young's inequality,
	\begin{equation*}
	|U_{Y}^{n}|\leq \kappa\sqrt{2J_{Y}^{n}x_{Y}^{n}}\leq \frac{\kappa}{\sqrt{2}}(J_{Y}^{n}+x_{Y}^{n}),
	\end{equation*}
	which by \eqref{eq:tYnIntermediateIneq} implies
	\begin{equation*}
	|t_{Y}^{n}|+|x_{Y}^{n}|+|U_{Y}^{n}|+|J_{Y}^{n}|+|K_{Y}^{n}|\leq\bigg[1+\bigg(1+\frac{1}{\sqrt{2}}\bigg)\kappa\bigg](x_{Y}^{n}+J_{Y}^{n}).
	\end{equation*}
	Using this in \eqref{eq:tXnIntermediateIneq} we obtain
	\begin{aalign}
	\label{eq:ZXnUnifBound1}
	&\Big(|t_{X}^{n}|+|x_{X}^{n}|+|U_{X}^{n}|+|J_{X}^{n}|+|K_{X}^{n}|\Big)(X,Y)\\
	&\leq \Big(|t_{X}^{n}|+|x_{X}^{n}|+|U_{X}^{n}|+|J_{X}^{n}|+|K_{X}^{n}|\Big)(X,\Y^{n}(X))\\
	&\quad\times\exp\Bigg\{B_{2}\bigg|\int_{\Y^{n}(X)}^{Y}(x_{Y}^{n}+J_{Y}^{n})(X,\tilde{Y})\,d\tilde{Y}\bigg|\Bigg\}
	\end{aalign}
	for a new constant $B_{2}$ that only depends on $\kappa$ and $k_{1}$. Since $x^{n}$ and $J^{n}$ are nondecreasing with respect to both variables, we have
	\begin{align*}
	&\bigg|\int_{\Y^{n}(X)}^{Y}(x_{Y}^{n}+J_{Y}^{n})(X,\tilde{Y})\,d\tilde{Y}\bigg|\\
	&=\big|x^{n}(X,Y)-x^{n}(X,\Y^{n}(X))+J^{n}(X,Y)-J^{n}(X,\Y^{n}(X))\big|\\
	&\leq \big|x^{n}(X_{r},Y_{r})-x^{n}(X_{l},Y_{l})\big|+\big|J^{n}(X_{r},Y_{r})-J^{n}(X_{l},Y_{l})\big|\\
	&=\big|x^{n}(\X^{n}(s_{r}),\Y^{n}(s_{r}))-x^{n}(\X^{n}(s_{l}),\Y^{n}(s_{l}))\big|\\
	&\quad+\big|J^{n}(\X^{n}(s_{r}),\Y^{n}(s_{r}))-J^{n}(\X^{n}(s_{l}),\Y^{n}(s_{l}))\big| \\
	&=\big|\Z_{2}^{n}(s_{r})-\Z_{2}^{n}(s_{l})\big|+\big|\Z_{4}^{n}(s_{r})-\Z_{4}^{n}(s_{l})\big|.
	\end{align*}
	Since
	\begin{align*}
	\Z_{2}^{n}(s_{r})-\Z_{2}^{n}(s_{l})&=\Z_{2}^{n}(s_{r})-\Z_{2}(s_{r})+\Z_{2}(s_{r})-\Z_{2}^{n}(s_{l})+\Z_{2}(s_{l})-\Z_{2}(s_{l})\\
	&=\Z_{2}^{n}(s_{r})-\Z_{2}(s_{r})+\Z_{2}^{a}(s_{r})+s_{r}-\Z_{2}^{n}(s_{l})+\Z_{2}(s_{l})-\Z_{2}^{a}(s_{l})-s_{l}
	\end{align*}
	and
	\begin{equation*}
	\Z_{4}^{n}(s_{r})-\Z_{4}^{n}(s_{l})=\Z_{4}^{n}(s_{r})-\Z_{4}(s_{r})+\Z_{4}^{a}(s_{r})-\Z_{4}^{n}(s_{l})+\Z_{4}(s_{l})-\Z_{4}^{a}(s_{l})
	\end{equation*}
	we end up with
	\begin{align*}
	&\bigg|\int_{\Y^{n}(X)}^{Y}(x_{Y}^{n}+J_{Y}^{n})(X,\tilde{Y})\,d\tilde{Y}\bigg|\\
	&\leq 2||\Z_{2}-\Z_{2}^{n}||_{L^{\infty}([s_{l},s_{r}])}+2||\Z_{2}^{a}||_{L^{\infty}([s_{l},s_{r}])}+s_{r}-s_{l}\\
	&\quad+2||\Z_{4}-\Z_{4}^{n}||_{L^{\infty}([s_{l},s_{r}])}+2||\Z_{4}^{a}||_{L^{\infty}([s_{l},s_{r}])}.
	\end{align*}
	The convergence $\Z_{i}^{n}\rightarrow\Z_{i}$ in $L^{\infty}([s_{l},s_{r}])$ implies that for every $\varepsilon>0$ we can choose $n$ so large that $||\Z_{i}-\Z_{i}^{n}||_{L^{\infty}([s_{l},s_{r}])}\leq\varepsilon$, $i=1,2$. Hence,
	\begin{equation}
	\label{eq:ZYnUnifBound2}
	\bigg|\int_{\Y^{n}(X)}^{Y}(x_{Y}^{n}+J_{Y}^{n})(X,\tilde{Y})\,d\tilde{Y}\bigg|\leq 4\varepsilon+2||\Z_{2}^{a}||_{L^{\infty}([s_{l},s_{r}])}+2||\Z_{4}^{a}||_{L^{\infty}([s_{l},s_{r}])}+s_{r}-s_{l}.
	\end{equation}
	We use \eqref{eq:ZYnUnifBound2} and $Z_{i,X}^{n}(X,\Y^{n}(X))=\V_{i}^{n}(X)$ in \eqref{eq:ZXnUnifBound1} and obtain
	\begin{equation}
	\label{eq:ZXnUnifBound2}
	\Big(|t_{X}^{n}|+|x_{X}^{n}|+|U_{X}^{n}|+|J_{X}^{n}|+|K_{X}^{n}|\Big)(X,Y)\leq B_{3}\sum_{l=1}^{5}\big|\V_{l}^{n}(X)\big|
	\end{equation}
	where $B_{3}$ only depends on $\kappa$, $k_{1}$ and $|||\Theta|||_{\G(\Omega)}$. Since $\mu_{0}^{n}$ is absolutely continuous in $[x_{l},x_{r}]$ we obtain as in \eqref{eq:x1derexpr} that
	\begin{equation*}
	0\leq (x_{1}^{n})'(X)\leq 1
	\end{equation*}
	for all $X\in[X_{l},X_{r}]$. Using \eqref{eq:mapFtoGV1}, \eqref{eq:mapFtoGV2}, \eqref{eq:mapFtoGV4}, \eqref{eq:mapFtoGV5}, \eqref{eq:mapfromDtoF3} and \eqref{eq:mapfromDtoF6} we get
	\begin{equation*}
	0\leq\V_{1}^{n}(X)\leq\frac{1}{2}\kappa, \quad 0\leq\V_{2}^{n}(X)\leq\frac{1}{2}, \quad 0\leq\V_{4}^{n}(X)\leq 1, \quad 0\leq\V_{5}^{n}(X)\leq\kappa
	\end{equation*}
	for all $X\in[X_{l},X_{r}]$. From \eqref{eq:mapFtoGV3} and \eqref{eq:setFrel3} we have
	\begin{equation*}
	|\V_{3}^{n}(X)|\leq \frac{1}{c(U_{1}^{n}(X))}\sqrt{(x_{1}^{n})'(X)(J_{1}^{n})'(X)}\leq \kappa.
	\end{equation*}
	By inserting these estimates in \eqref{eq:ZXnUnifBound2} we get
	\begin{equation}
	\label{eq:ZXnUnifBound3}
	\Big(|t_{X}^{n}|+|x_{X}^{n}|+|U_{X}^{n}|+|J_{X}^{n}|+|K_{X}^{n}|\Big)(X,Y)\leq B_{4} \quad \text{ for all } (X,Y)\in \Omega
	\end{equation}
	for a new constant $B_{4}$, which only depends on $\kappa$, $k_{1}$ and $|||\Theta|||_{\G(\Omega)}$. Similarly we can show that there exist constants $B_{5}$, $B_{6}$ and $B_{7}$, which only depend on $\kappa$, $k_{1}$ and $|||\Theta|||_{\G(\Omega)}$, such that for all $(X,Y)\in \Omega$,
	\begin{equation}
	\label{eq:ZYnUnifBound}
	\Big(|t_{Y}^{n}|+|x_{Y}^{n}|+|U_{Y}^{n}|+|J_{Y}^{n}|+|K_{Y}^{n}|\Big)(X,Y)\leq B_{5},
	\end{equation}
	\begin{equation}
	\label{eq:ZXUnifBound}
	\Big(|t_{X}|+|x_{X}|+|U_{X}|+|J_{X}|+|K_{X}|\Big)(X,Y)\leq B_{6}
	\end{equation}
	and
	\begin{equation}
	\label{eq:ZYUnifBound}
	\Big(|t_{Y}|+|x_{Y}|+|U_{Y}|+|J_{Y}|+|K_{Y}|\Big)(X,Y)\leq B_{7}.
	\end{equation}
	From \eqref{eq:ZXYnEst} we get for all $(X,Y)\in \Omega$ that
	\begin{equation}
	\label{eq:tXYBound}
	\Big(|t_{XY}^{n}|+|x_{XY}^{n}|+|U_{XY}^{n}|+|J_{XY}^{n}|+|K_{XY}^{n}|\Big)(X,Y)\leq D
	\end{equation}
	for a constant $D$ that depends on $\kappa$, $k_{1}$ and $|||\Theta|||_{\G(\Omega)}$.
	From \eqref{eq:ZXnUnifBound3}-\eqref{eq:ZYUnifBound} we get in \eqref{eq:tXYtXYnDiff1}, for all $(X,Y)\in \Omega$ that
	\begin{aalign}
	\label{eq:tXYtXYnDiff2}
	&\big|t_{XY}-t^{n}_{XY}\big|(X,Y)\\
	&\leq C_{1}\Big(|U-U^{n}|+|t_{X}-t_{X}^{n}|+|U_{X}-U_{X}^{n}|+|t_{Y}-t_{Y}^{n}|+|U_{Y}-U_{Y}^{n}|\Big)(X,Y),
	\end{aalign}
	where $C_{1}$ depends on $\kappa$, $k_{1}$, $k_{2}$ and $|||\Theta|||_{\G(\Omega)}$. 	
	
	Using the estimates \eqref{eq:tXYtXYnDiff2} and \eqref{eq:tXYBound} in \eqref{eq:tXtXnDiff} we get
	\begin{align*}
	&|t_{X}(X,Y)-t^{n}_{X}(X,Y)|\\
	&\leq|\V_{1}(X)-\V_{1}^{n}(X)|+D|\Y\circ\X^{-1}(X)-\Y^{n}\circ(\X^{n})^{-1}(X)|\\
	&\quad+C_{1}\bigg|\int_{\Y(X)}^{Y}\big(|U-U^{n}|+|t_{X}-t_{X}^{n}|+|U_{X}-U_{X}^{n}|\\
	&\hspace{80pt}+|t_{Y}-t_{Y}^{n}|+|U_{Y}-U_{Y}^{n}|\big)(X,\tilde{Y})\,d\tilde{Y}\bigg|.
	\end{align*}
	To write the estimates more compactly, we denote $Z=(Z_{1},Z_{2},Z_{3},Z_{4},Z_{5})=(t,x,U,J,K)$, and similar for $Z^{n}$. With this notation we get from the above estimate that
	\begin{aalign}
	\label{eq:V1XV1nXDiff}
	&|Z_{1,X}(X,Y)-Z_{1,X}^{n}(X,Y)|\\
	&\leq|\V_{1}(X)-\V_{1}^{n}(X)|+D|\Y\circ\X^{-1}(X)-\Y^{n}\circ(\X^{n})^{-1}(X)|\\
	&\quad+C_{1}\bigg|\int_{\Y(X)}^{Y}\bigg[|Z_{3}-Z_{3}^{n}|+\sum_{i=1}^{5}\big(|Z_{i,X}-Z_{i,X}^{n}|+|Z_{i,Y}-Z_{i,Y}^{n}|\big)\bigg](X,\tilde{Y})\,d\tilde{Y}\bigg|.
	\end{aalign}  
	By the same procedure as above, we obtain
	\begin{aalign}
	\label{eq:VjXVjnXDiff}
	&|Z_{j,X}(X,Y)-Z_{j,X}^{n}(X,Y)|\\
	&\leq|\V_{j}(X)-\V_{j}^{n}(X)|+D|\Y\circ\X^{-1}(X)-\Y^{n}\circ(\X^{n})^{-1}(X)|\\
	&\quad+C_{j}\bigg|\int_{\Y(X)}^{Y}\bigg[|Z_{3}-Z_{3}^{n}|+\sum_{i=1}^{5}\big(|Z_{i,X}-Z_{i,X}^{n}|+|Z_{i,Y}-Z_{i,Y}^{n}|\big)\bigg](X,\tilde{Y})\,d\tilde{Y}\bigg|.  
	\end{aalign}
	for $j=2,3,4,5$, where $C_{j}$ depends on $\kappa$, $k_{1}$, $k_{2}$ and $|||\Theta|||_{\G(\Omega)}$, and $D$ depends on $\kappa$, $k_{1}$ and $|||\Theta|||_{\G(\Omega)}$.
	
	Let us find similar estimates for the partial derivatives with respect to $Y$. Now we subtract
	\begin{equation*}
	t_{Y}(\X(Y),Y)=t_{Y}(X,Y)+\int_{X}^{\X(Y)}t_{XY}(\tilde{X},Y)\,d\tilde{X}
	\end{equation*}
	and
	\begin{equation*}
	t^{n}_{Y}(\X^{n}(Y),Y)=t^{n}_{Y}(X,Y)+\int_{X}^{\X^{n}(Y)}t^{n}_{XY}(\tilde{X},Y)\,d\tilde{X}
	\end{equation*}
	to get
	\begin{align*}
	t_{Y}(X,Y)-t^{n}_{Y}(X,Y)&=\W_{1}(Y)-\W_{1}^{n}(Y)+\int_{\X(Y)}^{\X^{n}(Y)}t^{n}_{XY}(\tilde{X},Y)\,d\tilde{X}\\
	&\quad-\int_{X}^{\X(Y)}(t_{XY}(\tilde{X},Y)-t^{n}_{XY}(\tilde{X},Y))\,d\tilde{X},
	\end{align*}
	and we obtain
	\begin{align*}
	&|t_{Y}(X,Y)-t^{n}_{Y}(X,Y)|\\
	&\leq|\W_{1}(Y)-\W_{1}^{n}(Y)|+D|\X\circ\Y^{-1}(Y)-\X^{n}\circ(\Y^{n})^{-1}(Y)|\\
	&\quad+C_{1}\bigg|\int_{X}^{\X(Y)}\big(|U-U^{n}|+|t_{X}-t_{X}^{n}|+|U_{X}-U_{X}^{n}|\\
	&\hspace{85pt}+|t_{Y}-t_{Y}^{n}|+|U_{Y}-U_{Y}^{n}|\big)(\tilde{X},Y)\,d\tilde{X}\bigg|,
	\end{align*}
	which in the alternative notation takes the form
	\begin{aalign}
	\label{eq:V1YV1YnDiff}
	&|Z_{1,Y}(X,Y)-Z_{1,Y}^{n}(X,Y)|\\
	&\leq|\W_{1}(Y)-\W_{1}^{n}(Y)|+D|\X\circ\Y^{-1}(Y)-\X^{n}\circ(\Y^{n})^{-1}(Y)|\\
	&\hspace{8pt}+C_{1}\bigg|\int_{X}^{\X(Y)}\bigg[|Z_{3}-Z_{3}^{n}|+\sum_{i=1}^{5}\big(|Z_{i,X}-Z_{i,X}^{n}|+|Z_{i,Y}-Z_{i,Y}^{n}|\big)\bigg](\tilde{X},Y)\,d\tilde{X}\bigg|.
	\end{aalign}
	Similarly, we get
	\begin{aalign}
	\label{eq:VjYVjYnDiff}
	&|Z_{j,Y}(X,Y)-Z_{j,Y}^{n}(X,Y)|\\
	&\leq|\W_{j}(Y)-\W_{j}^{n}(Y)|+D|\X\circ\Y^{-1}(Y)-\X^{n}\circ(\Y^{n})^{-1}(Y)|\\
	&\hspace{8pt}+C_{j}\bigg|\int_{X}^{\X(Y)}\bigg[|Z_{3}-Z_{3}^{n}|+\sum_{i=1}^{5}\big(|Z_{i,X}-Z_{i,X}^{n}|+|Z_{i,Y}-Z_{i,Y}^{n}|\big)\bigg](\tilde{X},Y)\,d\tilde{X}\bigg|
	\end{aalign}
	for $j=2,3,4,5$. 
	
	We have
	\begin{equation*}
	U(X,Y)=U(X,\Y(X))+\int_{\Y(X)}^{Y}U_{Y}(X,\tilde{Y})\,d\tilde{Y}
	\end{equation*}
	and
	\begin{equation*}
	U^{n}(X,Y)=U^{n}(X,\Y^{n}(X))+\int_{\Y^{n}(X)}^{Y}U^{n}_{Y}(X,\tilde{Y})\,d\tilde{Y},
	\end{equation*}
	so that
	\begin{align*}
	U(X,Y)-U^{n}(X,Y)&=U(X,\Y(X))-U^{n}(X,\Y^{n}(X))
	-\int_{\Y^{n}(X)}^{\Y(X)}U^{n}_{Y}(X,\tilde{Y})\,d\tilde{Y}\\
	&\quad+\int_{\Y(X)}^{Y}(U_{Y}(X,\tilde{Y})-U^{n}_{Y}(X,\tilde{Y}))\,d\tilde{Y}.
	\end{align*}
	To any $X$ in $[X_l, X_r]$, there exist unique $s$ and $s^{n}$ in $[s_l,s_r]$ such that $X=\X(s)$ and $X=\X^{n}(s^{n})$ and we can write
	\begin{align*}
	U(X,\Y(X))-U^{n}(X,\Y^{n}(X))&=U(\X(s),\Y(s))-U^{n}(\X^{n}(s^n),\Y^{n}(s^{n}))\\
	&=\Z_{3}(s)-\Z_{3}^{n}(s^{n})\\
	&=U_{1}(\X(s))-U_{1}^{n}(\X^{n}(s^{n}))\\
	&=U_{1}(X)-U_{1}^{n}(X).
	\end{align*}
	Therefore, 
	\begin{align*}
	&|U(X,Y)-U^{n}(X,Y)|\\
	&\leq ||U_{1}-U_{1}^{n}||_{L^{\infty}([X_{l},X_{r}])}+B_{5}|\Y\circ\X^{-1}(X)-\Y^{n}\circ(\X^{n})^{-1}(X)|\\
	&\quad+\bigg|\int_{\Y(X)}^{Y}|U_{Y}(X,\tilde{Y})-U^{n}_{Y}(X,\tilde{Y})|\,d\tilde{Y}\bigg|,
	\end{align*}	
	where we used \eqref{eq:ZYnUnifBound}. In the new notation this implies	
	\begin{aalign}
	\label{eq:V3V3nDiff}
	&|Z_{3}(X,Y)-Z^{n}_{3}(X,Y)|\\
	&\leq ||U_{1}-U_{1}^{n}||_{L^{\infty}([X_{l},X_{r}])}+B_{5}|\Y\circ\X^{-1}(X)-\Y^{n}\circ(\X^{n})^{-1}(X)|\\
	&\quad+\bigg|\int_{\Y(X)}^{Y}\bigg[|Z_{3}-Z_{3}^{n}|+\sum_{i=1}^{5}\big(|Z_{i,X}-Z_{i,X}^{n}|+|Z_{i,Y}-Z_{i,Y}^{n}|\big)\bigg](X,\tilde{Y})\,d\tilde{Y}\bigg|.
	\end{aalign}
	
	From \eqref{eq:V1XV1nXDiff}--\eqref{eq:VjYVjYnDiff} we get by using H\"{o}lder's inequality\footnote{The factor $3$ comes from that we first split the right-hand side in three terms, the factor $33=3\cdot 11$ comes from splitting the $11$ terms in the integral.},
	\begin{aalign}
	\label{eq:VjXVjnXDiff2}
	&\big[Z_{j,X}(X,Y)-Z_{j,X}^{n}(X,Y)\big]^{2}\\
	&\leq 3\big[\V_{j}(X)-\V_{j}^{n}(X)\big]^{2}+3D^{2}\big[\Y\circ\X^{-1}(X)-\Y^{n}\circ(\X^{n})^{-1}(X)\big]^{2}\\
	&\quad+33C_{j}^{2}|Y-\Y(X)|\bigg|\int_{\Y(X)}^{Y}\bigg[\big[Z_{3}-Z_{3}^{n}\big]^{2}
	+\sum_{i=1}^{5}\Big(\big[Z_{i,X}-Z_{i,X}^{n}\big]^{2}\\
	&\hspace{230pt}+\big[Z_{i,Y}-Z_{i,Y}^{n}\big]^{2}\Big)\bigg](X,\tilde{Y})\,d\tilde{Y}\bigg|
	\end{aalign}
	and
	\begin{aalign}
	\label{eq:V1YV1YnDiff2}
	&\big[Z_{j,Y}(X,Y)-Z_{j,Y}^{n}(X,Y)\big]^{2}\\
	&\leq 3\big[\W_{j}(Y)-\W_{j}^{n}(Y)\big]^{2}+3D^{2}\big[\X\circ\Y^{-1}(Y)-\X^{n}\circ(\Y^{n})^{-1}(Y)\big]^{2}\\
	&\quad+33C_{j}^{2}|\X(Y)-X|\bigg|\int_{X}^{\X(Y)}\bigg[\big[Z_{3}-Z_{3}^{n}\big]^{2}+\sum_{i=1}^{5}\Big(\big[Z_{i,X}-Z_{i,X}^{n}\big]^{2}\\
	&\hspace{228pt}+\big[Z_{i,Y}-Z_{i,Y}^{n}\big]^{2}\Big)\bigg](\tilde{X},Y)\,d\tilde{X}\bigg|
	\end{aalign}
	for $j=1,\dots,5$. Similarly, from \eqref{eq:V3V3nDiff} we get
	\begin{aalign}
	\label{eq:V3V3nDiff2}
	&\big[Z_{3}(X,Y)-Z^{n}_{3}(X,Y)\big]^{2}\\
	&\leq 3||U_{1}-U_{1}^{n}||_{L^{\infty}([X_{l},X_{r}])}^{2}+3B_{5}^{2}\big[\Y\circ\X^{-1}(X)-\Y^{n}\circ(\X^{n})^{-1}(X)\big]^{2}\\
	&\quad+33|Y-\Y(X)|\bigg|\int_{\Y(X)}^{Y}\bigg[\big[Z_{3}-Z_{3}^{n}\big]^{2}+\sum_{i=1}^{5}\Big(\big[Z_{i,X}-Z_{i,X}^{n}\big]^{2}\\
	&\hspace{227pt}+\big[Z_{i,Y}-Z_{i,Y}^{n}\big]^{2}\Big)\bigg](X,\tilde{Y})\,d\tilde{Y}\bigg|.
	\end{aalign}
	
	We combine \eqref{eq:VjXVjnXDiff2}-\eqref{eq:V3V3nDiff2} and get for all $(X,Y)\in \Omega$ that
	\begin{aalign}
	\label{eq:GronCond}
	&\bigg[\big[Z_{3}-Z_{3}^{n}\big]^{2}+\sum_{i=1}^{5}\Big(\big[Z_{i,X}-Z_{i,X}^{n}\big]^{2}+\big[Z_{i,Y}-Z_{i,Y}^{n}\big]^{2}\Big)\bigg](X,Y)\\
	&\leq 
	3||U_{1}-U_{1}^{n}||_{L^{\infty}([X_{l},X_{r}])}^{2}+3(B_{5}^{2}+5D^2)\big[\Y\circ\X^{-1}(X)-\Y^{n}\circ(\X^{n})^{-1}(X)\big]^{2}\\
	&\quad+15D^2\big[\X\circ\Y^{-1}(Y)-\X^{n}\circ(\Y^{n})^{-1}(Y)\big]^{2}\\
	&\quad+3\sum_{j=1}^{5}\big[\V_{j}(X)-\V_{j}^{n}(X)\big]^{2}+3\sum_{j=1}^{5}\big[\W_{j}(Y)-\W_{j}^{n}(Y)\big]^{2}\\
	&\quad+33(1+C)|Y-\Y(X)|\bigg|\int_{\Y(X)}^{Y}\bigg[\big[Z_{3}-Z_{3}^{n}\big]^{2}+\sum_{i=1}^{5}\Big(\big[Z_{i,X}-Z_{i,X}^{n}\big]^{2}\\
	&\hspace{230pt}+\big[Z_{i,Y}-Z_{i,Y}^{n}\big]^{2}\Big)\bigg](X,\tilde{Y})\,d\tilde{Y}\bigg|\\
	&\quad+33C|\X(Y)-X|\bigg|\int_{X}^{\X(Y)}\bigg[\big[Z_{3}-Z_{3}^{n}\big]^{2}+\sum_{i=1}^{5}\Big(\big[Z_{i,X}-Z_{i,X}^{n}\big]^{2}\\
	&\hspace{225pt}+\big[Z_{i,Y}-Z_{i,Y}^{n}\big]^{2}\Big)\bigg](\tilde{X},Y)\,d\tilde{X}\bigg|,
	\end{aalign}
	where we introduced $C=\sum_{j=1}^{5}C_{j}^{2}$. 
	
	At this point we need the following Gronwall inequality, which can be found in \cite{FiMiPe}, see chapter "Gronwall inequalities in higher dimensions". For completeness, we state and prove the inequality in the following lemma.
	
	\begin{lemma}
		\label{lemma:Gron}
		Consider a nonnegative function $u(x,y)$ in the region $x\geq 0$ and $y\geq 0$. 
		Assume that 
		\begin{equation*}
		u(x,y)\leq c+a\int_{0}^{x}u(r,y)\,dr+b\int_{0}^{y}u(x,s)\,ds
		\end{equation*}
		where $a$, $b$ and $c$ are nonnegative constants. Then we have
		\begin{equation*}
		u(x,y)\leq ce^{2ax+2by+abxy}.
		\end{equation*}
	\end{lemma}
	
	\begin{proof}
		Set
		\begin{equation*}
		G(x,y)=\int_{0}^{y}u(x,s)\,ds.
		\end{equation*}
		We have $G_{y}(x,y)=u(x,y)$, so that
		\begin{equation*}
		G_{y}(x,y)\leq c+a\int_{0}^{x}u(r,y)\,dr+bG(x,y)
		\end{equation*}
		and
		\begin{equation*}
		\frac{d}{dy}\Big(G(x,y)e^{-by}\Big)\leq \bigg(c+a\int_{0}^{x}u(r,y)\,dr\bigg)e^{-by}. 
		\end{equation*}
		Integration yields
		\begin{equation}
		\label{eq:Gbound}
		G(x,y)\leq \frac{c}{b}(e^{by}-1)+a\int_{0}^{y}\int_{0}^{x}u(r,s)e^{b(y-s)}\,dr\,ds.
		\end{equation}
		Therefore we get
		\begin{align*}
		u(x,y)&\leq c+a\int_{0}^{x}u(r,y)\,dr+bG(x,y)\\ 
		&\leq ce^{by}+a\int_{0}^{x}u(r,y)\,dr+ab\int_{0}^{y}\int_{0}^{x}u(r,s)e^{b(y-s)}\,dr\,ds.
		\end{align*}
		Denote
		\begin{equation*}
		g(x,y)=ce^{by}+a\int_{0}^{x}u(r,y)\,dr+ab\int_{0}^{y}\int_{0}^{x}u(r,s)e^{b(y-s)}\,dr\,ds
		\end{equation*}
		which implies
		\begin{equation}
		\label{eq:uBoundg}
		u(x,y)\leq g(x,y).
		\end{equation}
		From \eqref{eq:Gbound} we have
		\begin{aalign}
		\label{eq:GBoundg}
		G(x,y)&\leq \frac{c}{b}(e^{by}-1)+a\int_{0}^{y}\int_{0}^{x}u(r,s)e^{b(y-s)}\,dr\,ds\\
		&\leq \frac{1}{b}\bigg(ce^{by}+ab\int_{0}^{y}\int_{0}^{x}u(r,s)e^{b(y-s)}\,dr\,ds\bigg)\\
		&\leq \frac{1}{b}\bigg(ce^{by}+a\int_{0}^{x}u(r,y)\,dr+ab\int_{0}^{y}\int_{0}^{x}u(r,s)e^{b(y-s)}\,dr\,ds\bigg)\\
		&=\frac{1}{b}g(x,y).
		\end{aalign}
		We compute
		\begin{equation*}
		g_{x}(x,y)=au(x,y)+ab\int_{0}^{y}u(x,s)e^{b(y-s)}\,ds.
		\end{equation*}
		Integration by parts yields
		\begin{equation}
		\label{eq:gxIntByPart}
		g_{x}(x,y)=au(x,y)+ab\int_{0}^{y}u(x,s)\,ds+ab^{2}\int_{0}^{y}\int_{0}^{s}u(x,l)e^{b(y-s)}\,dl\,ds.
		\end{equation}
		In the last integral on the right-hand side we use
		\begin{equation*}
		\int_{0}^{s}u(x,l)\,dl=G(x,s)\leq \frac{c}{b}(e^{bs}-1)+a\int_{0}^{s}\int_{0}^{x}u(r,t)e^{b(s-t)}\,dr\,dt
		\end{equation*}
		and get
		\begin{aalign}
		\label{eq:MainBound}	
		&\int_{0}^{y}\int_{0}^{s}u(x,l)e^{b(y-s)}\,dl\,ds\\
		&\leq\int_{0}^{y}e^{b(y-s)}\bigg(\frac{c}{b}(e^{bs}-1)+a\int_{0}^{s}\int_{0}^{x}u(r,t)e^{b(s-t)}\,dr\,dt\bigg)ds\\
		&\leq
		\int_{0}^{y}e^{b(y-s)}\bigg(\frac{c}{b}e^{bs}+a\int_{0}^{y}\int_{0}^{x}u(r,t)e^{b(s-t)}\,dr\,dt\bigg)ds\\
		&=\int_{0}^{y}\bigg(\frac{c}{b}e^{by}+a\int_{0}^{y}\int_{0}^{x}u(r,t)e^{b(y-t)}\,dr\,dt\bigg)ds\\
		&=\frac{c}{b}ye^{by}+ay\int_{0}^{y}\int_{0}^{x}u(r,t)e^{b(y-t)}\,dr\,dt.
		\end{aalign}
		Using the estimates \eqref{eq:uBoundg}, \eqref{eq:GBoundg} and \eqref{eq:MainBound} in \eqref{eq:gxIntByPart} gives
		\begin{align*}
		g_{x}(x,y)&\leq ag(x,y)+ab\frac{1}{b}g(x,y)+abcye^{by}+a^{2}b^{2}y\int_{0}^{y}\int_{0}^{x}u(r,t)e^{b(y-t)}\,dr\,dt\\
		&=2ag(x,y)+aby\bigg(ce^{by}+ab\int_{0}^{y}\int_{0}^{x}u(r,t)e^{b(y-t)}\,dr\,dt\bigg)\\
		&\leq 2ag(x,y)+aby\bigg(ce^{by}+a\int_{0}^{x}u(r,y)\,dr+ab\int_{0}^{y}\int_{0}^{x}u(r,t)e^{b(y-t)}\,dr\,dt\bigg)\\
		&=2ag(x,y)+abyg(x,y).
		\end{align*} 
		Integration leads to
		\begin{equation*}
		g(x,y)\leq g(0,y)e^{2ax+abxy}=ce^{2ax+by+abxy}
		\end{equation*}
		and using \eqref{eq:uBoundg} finally implies
		\begin{equation*}
		u(x,y)\leq ce^{2ax+by+abxy}.
		\end{equation*}
		If we instead considered 
		\begin{equation*}
		\tilde G(x,y)=\int_{0}^{x}u(r,y)\,dr
		\end{equation*}
		we would have ended up with
		\begin{equation*}
		u(x,y)\leq ce^{ax+2by+abxy}.
		\end{equation*}
	\end{proof}
	
	We return to \eqref{eq:GronCond}. Using Lemma \ref{lemma:Gron} we get
	\begin{aalign}
	\label{eq:Gron1}
	&\bigg[\big[Z_{3}-Z_{3}^{n}\big]^{2}+\sum_{i=1}^{5}\Big(\big[Z_{i,X}-Z_{i,X}^{n}\big]^{2}+\big[Z_{i,Y}-Z_{i,Y}^{n}\big]^{2}\Big)\bigg](X,Y)\\
	&\leq \Bigg\{3||U_{1}-U_{1}^{n}||_{L^{\infty}([X_{l},X_{r}])}^{2}\\
	&\hspace{23pt}+3(B_{5}^{2}+5D^2)\big[\Y\circ\X^{-1}(X)-\Y^{n}\circ(\X^{n})^{-1}(X)\big]^{2}\\
	&\hspace{23pt}+15D^2\big[\X\circ\Y^{-1}(Y)-\X^{n}\circ(\Y^{n})^{-1}(Y)\big]^{2}\\
	&\hspace{23pt}+3\sum_{j=1}^{5}\big[\V_{j}(X)-\V_{j}^{n}(X)\big]^{2}+3\sum_{j=1}^{5}\big[\W_{j}(Y)-\W_{j}^{n}(Y)\big]^{2}\Bigg\}\\
	&\quad \times \exp\bigg\{66C\big[\X(Y)-X\big]^{2}+66(1+C)\big[Y-\Y(X)\big]^{2}\\
	&\hspace{53pt}+33^{2}C(1+C)\big[\X(Y)-X\big]^{2}\big[Y-\Y(X)\big]^{2}\bigg\}.
	\end{aalign}
	Since all the differences appearing in the exponential function are bounded by either $X_{r}-X_{l}$ or $Y_{r}-Y_{l}$ we can find a new constant $K$ which depends on $\kappa$, $k_{1}$, $k_{2}$, $|||\Theta|||_{\G(\Omega)}$ and the size of $\Omega$ such that	
	\begin{align*}
	&\bigg[\big[Z_{3}-Z_{3}^{n}\big]^{2}+\sum_{i=1}^{5}\Big(\big[Z_{i,X}-Z_{i,X}^{n}\big]^{2}+\big[Z_{i,Y}-Z_{i,Y}^{n}\big]^{2}\Big)\bigg](X,Y)\\
	&\leq K\Bigg\{
	||U_{1}-U_{1}^{n}||_{L^{\infty}([X_{l},X_{r}])}^{2}+\sum_{j=1}^{5}\Big(\big[\V_{j}(X)-\V_{j}^{n}(X)\big]^{2}+\big[\W_{j}(Y)-\W_{j}^{n}(Y)\big]^{2}\Big)\\
	&\hspace{33pt}+\big[\Y\circ\X^{-1}(X)-\Y^{n}\circ(\X^{n})^{-1}(X)\big]^{2}
	+\big[\X\circ\Y^{-1}(Y)-\X^{n}\circ(\Y^{n})^{-1}(Y)\big]^{2}\Bigg\},
	\end{align*}
	and we have proved the claim \eqref{eq:Gron}.
	
	From \eqref{eq:Gron} we obtain an estimate for the difference $Z_{i}(X,2s-X)-Z_{i}^{n}(X,2s-X)$ for $i=1,\dots,4$. We have
	\begin{align*}
	&Z_{i}(X,2s-X)-Z_{i}^{n}(X,2s-X)\\
	&=Z_{i}(\X(s),2s-\X(s))+\int_{\X(s)}^{X}(Z_{i,X}-Z_{i,Y})(\xi,2s-\xi)\,d\xi\\
	&\quad -Z_{i}^{n}(\X^{n}(s),2s-\X^{n}(s))-\int_{\X^{n}(s)}^{X}(Z_{i,X}^{n}-Z_{i,Y}^{n})(\xi,2s-\xi)\,d\xi\\
	&=\Z_{i}(s)-\Z_{i}^{n}(s)+
	\int_{\X(s)}^{X}\big[(Z_{i,X}-Z_{i,X}^{n})-(Z_{i,Y}-Z_{i,Y}^{n})\big](\xi,2s-\xi)\,d\xi\\
	&\quad -\int_{\X^{n}(s)}^{\X(s)}(Z_{i,X}^{n}-Z_{i,Y}^{n})(\xi,2s-\xi)\,d\xi,
	\end{align*}
	which implies by \eqref{eq:ZXnUnifBound3}, \eqref{eq:ZYnUnifBound}, and the Cauchy--Schwarz inequality, that
	\begin{aalign}
	\label{eq:IntEst0}
	&\big|Z_{i}(X,2s-X)-Z_{i}^{n}(X,2s-X)\big|\\
	&\leq ||\Z_{i}-\Z_{i}^{n}||_{L^{\infty}([s_{l},s_{r}])}+(B_{4}+B_{5})||\X-\X^{n}||_{L^{\infty}([s_{l},s_{r}])} \\
	&\quad+\big[X-\X(s)\big]^{\frac{1}{2}}\Bigg\{\bigg|\int_{\X(s)}^{X}(Z_{i,X}-Z_{i,X}^{n})^2(\xi,2s-\xi)\,d\xi\bigg|^{\frac{1}{2}}\\
	&\hspace{100pt} +\bigg|\int_{\X(s)}^{X}(Z_{i,Y}-Z_{i,Y}^{n})^2(\xi,2s-\xi)\,d\xi\bigg|^{\frac{1}{2}}\Bigg\}\\
	&\leq ||\Z_{i}-\Z_{i}^{n}||_{L^{\infty}([s_{l},s_{r}])}+(B_{4}+B_{5})||\X-\X^{n}||_{L^{\infty}([s_{l},s_{r}])} \\
	&\quad+2(X_{r}-X_{l})^{\frac{1}{2}}\bigg|\int_{\X(s)}^{X}\bigg[\big[Z_{3}-Z_{3}^{n}\big]^{2}+\sum_{i=1}^{5}\Big(\big[Z_{i,X}-Z_{i,X}^{n}\big]^{2}\\
	&\hspace{206pt}+\big[Z_{i,Y}-Z_{i,Y}^{n}\big]^{2}\Big)\bigg](\xi,2s-\xi)\,d\xi\bigg|^{\frac{1}{2}}.
	\end{aalign}
	From \eqref{eq:Gron}, we have
	\begin{align*}
	&\bigg[\big[Z_{3}-Z_{3}^{n}\big]^{2}+\sum_{i=1}^{5}\Big(\big[Z_{i,X}-Z_{i,X}^{n}\big]^{2}+\big[Z_{i,Y}-Z_{i,Y}^{n}\big]^{2}\Big)\bigg](\xi,2s-\xi)\\
	&\leq K\Bigg\{
	||U_{1}-U_{1}^{n}||_{L^{\infty}([X_{l},X_{r}])}^{2}+\sum_{j=1}^{5}\Big(\big[\V_{j}(\xi)-\V_{j}^{n}(\xi)\big]^{2}+\big[\W_{j}(2s-\xi)-\W_{j}^{n}(2s-\xi)\big]^{2}\Big)\\
	&\hspace{33pt}+\big[\Y\circ\X^{-1}(\xi)-\Y^{n}\circ(\X^{n})^{-1}(\xi)\big]^{2}\\
	&\hspace{33pt}+\big[\X\circ\Y^{-1}(2s-\xi)-\X^{n}\circ(\Y^{n})^{-1}(2s-\xi)\big]^{2}\Bigg\}.
	\end{align*}
	Integration and a change of variables leads to
	\begin{aalign}
	\label{eq:IntEst1}
	&\bigg|\int_{\X(s)}^{X}\bigg[\big[Z_{3}-Z_{3}^{n}\big]^{2}+\sum_{i=1}^{5}\Big(\big[Z_{i,X}-Z_{i,X}^{n}\big]^{2}+\big[Z_{i,Y}-Z_{i,Y}^{n}\big]^{2}\Big)\bigg](\xi,2s-\xi)\,d\xi\bigg|\\
	&\leq K\Bigg\{|X-\X(s)|||U_{1}-U_{1}^{n}||_{L^{\infty}([X_{l},X_{r}])}^{2}\\
	&\hspace{30pt}+\sum_{j=1}^{5}\Bigg(\bigg|\int_{\X(s)}^{X}\big[\V_{j}(\xi)-\V_{j}^{n}(\xi)\big]^{2}\,d\xi\bigg|+\bigg|\int_{2s-X}^{\Y(s)}\big[\W_{j}(\xi)-\W_{j}^{n}(\xi)\big]^{2}\,d\xi\bigg|\Bigg)\\
	&\hspace{30pt}+\bigg|\int_{\X(s)}^{X}\big[\Y\circ\X^{-1}(\xi)-\Y^{n}\circ(\X^{n})^{-1}(\xi)\big]^{2}\,d\xi\bigg|\\
	&\hspace{30pt}+\bigg|\int_{2s-X}^{\Y(s)}\big[\X\circ\Y^{-1}(\xi)-\X^{n}\circ(\Y^{n})^{-1}(\xi)\big]^{2}\,d\xi\bigg|\Bigg\}\\
	&\leq K\Bigg\{(X_{r}-X_{l})||U_{1}-U_{1}^{n}||_{L^{\infty}([X_{l},X_{r}])}^{2}\\
	&\hspace{30pt}+\sum_{j=1}^{5}\Big(||\V_{j}-\V_{j}^{n}||_{L^{2}([X_{l},X_{r}])}^{2}+||\W_{j}-\W_{j}^{n}||_{L^{2}([Y_{l},Y_{r}])}^{2}\Big)\\
	&\hspace{30pt}+\int_{X_{l}}^{X_{r}}\big[\Y\circ\X^{-1}(\xi)-\Y^{n}\circ(\X^{n})^{-1}(\xi)\big]^{2}\,d\xi\\
	&\hspace{30pt}+\int_{Y_{l}}^{Y_{r}}\big[\X\circ\Y^{-1}(\xi)-\X^{n}\circ(\Y^{n})^{-1}(\xi)\big]^{2}\,d\xi\Bigg\}.
	\end{aalign}
	From \eqref{eq:initialcurvenormalization} we have $X+\Y\circ\X^{-1}(X)=2\X^{-1}(X)$ and $X+\Y^{n}\circ(\X^{n})^{-1}(X)=2(\X^{n})^{-1}(X)$, so that
	\begin{equation}
	\label{eq:YChiInvDiff}
		\Y\circ\X^{-1}(X)-\Y^{n}\circ(\X^{n})^{-1}(X)=2(\X^{-1}(X)-(\X^{n})^{-1}(X)).
	\end{equation}
	This leads to
	\begin{aalign}
		\label{eq:YChiInvDiff2}
		&\int_{X_{l}}^{X_{r}}\big[\Y\circ\X^{-1}(\xi)-\Y^{n}\circ(\X^{n})^{-1}(\xi)\big]^{2}\,d\xi\\
		&=4\int_{X_{l}}^{X_{r}}\big[\X^{-1}(\xi)-(\X^{n})^{-1}(\xi)\big]^{2}\,d\xi\\
		&\leq 4(s_{r}-s_{l})\int_{X_{l}}^{X_{r}}|\X^{-1}(\xi)-(\X^{n})^{-1}(\xi)|\,d\xi
	\end{aalign}

To estimate the above integral, we need to introduce another sequence of curves on $[s_l, s_r]$ given by $(\hat\X^n(s), \hat \Y^n(s))=(\max\{\X(s), \X^n(s)\}, \min\{\Y(s), \Y^n(s)\})$. This sequence satisfies that both $(\X, \Y)$ and $(\X^n, \Y^n)$ lie above or are equal to $(\hat \X^n, \hat \Y^n)$. This implies that $\X^{-1}(X)\geq (\hat\X^{n})^{-1}(X)$ and $(\X^n)^{-1}(X)\geq (\hat\X^{n})^{-1}(X)$ for all $X$ in $[X_l,X_r]$, and that 
\begin{aalign}\label{curve}
	\int_{X_{l}}^{X_{r}}|\X^{-1}(\xi)-(\X^{n})^{-1}(\xi)|\,d\xi& =\int_{X_l}^{X_r}(\X^{-1}(\xi)-(\hat\X^n)^{-1}(\xi))\,d\xi\\
	&\quad+\int_{X_l}^{X_r}((\X^n)^{-1}(\xi)-(\hat\X^n)^{-1}(\xi))\,d\xi.
\end{aalign}
By a change of variables and integration by parts we get
	\begin{equation*}
		\int_{X_l}^{X_r}\X^{-1}(\xi)\,d\xi=\int_{\X^{-1}(X_{l})}^{\X^{-1}(X_{r})}s\dot{\X}(s)\,ds=X_r\X^{-1}(X_r)-X_{l}\X^{-1}(X_{l})-\int_{\X^{-1}(X_{l})}^{\X^{-1}(X_{r})}\X(s)\,ds,
	\end{equation*}
	and similarly we find
	\begin{equation*}
		\int_{X_{l}}^{X_{r}}(\hat\X^{n})^{-1}(\xi)\,d\xi=X_{r}(\hat\X^{n})^{-1}(X_{r})-X_{l}(\hat\X^{n})^{-1}(X_{l})-\int_{(\hat\X^{n})^{-1}(X_{l})}^{(\hat\X^{n})^{-1}(X_{r})}\hat\X^{n}(s)\,ds.
	\end{equation*}
	Note that
	\begin{equation*}
		\X^{-1}(X_{l})=s_l=(\hat\X^{n})^{-1}(X_{l}) \quad \text{ and } \quad  \X^{-1}(X_{r})=s_r=(\hat{\X}^{n})^{-1}(X_{r}).
	\end{equation*} 
	Therefore, 
	\begin{equation*}
		\int_{X_{l}}^{X_{r}}|\X^{-1}(\xi)-(\hat\X^{n})^{-1}(\xi)|\,d\xi=\int_{s_l}^{s_r}(\hat\X^n(s)-\X(s))\,ds.
	\end{equation*}
	We obtain in a similar way,
	\begin{equation*}
		\int_{X_{l}}^{X_{r}}|(\X^n)^{-1}(\xi)-(\hat \X^{n})^{-1}(\xi)|\,d\xi=\int_{s_l}^{s_r}(\hat\X^n(s)-\X^{n}(s))\,ds.
	\end{equation*}
	Combining \eqref{curve} and the above estimates we end up with 
	\begin{align*}
        \int_{X_{l}}^{X_{r}}|\X^{-1}(\xi)-(\X^{n})^{-1}(\xi)|\,d\xi& = \int_{s_l}^{s_r}(\hat\X^n(s)-\X(s))\,ds + \int_{s_l}^{s_r}(\hat\X^n(s)-\X^{n}(s))\,ds\\
        & = \int_{s_l}^{s_r}\vert \X^{n}(s)-\X(s)\vert\,ds,
        \end{align*}
    which inserted in \eqref{eq:YChiInvDiff2} yields
	\begin{equation}
	\label{eq:YChiInvDiff3}
		\int_{X_{l}}^{X_{r}}\big[\Y\circ\X^{-1}(\xi)-\Y^{n}\circ(\X^{n})^{-1}(\xi)\big]^{2}\,d\xi\leq 4(s_{r}-s_{l})^{2}||\X-\X^{n}||_{L^{\infty}([s_{l},s_{r}])}.
	\end{equation}

	By similar calculations we find
	\begin{equation}
	\label{eq:ChiYInvDiff}
		\int_{Y_{l}}^{Y_{r}}\big[\X\circ\Y^{-1}(\xi)-\X^{n}\circ(\Y^{n})^{-1}(\xi)\big]^{2}\,d\xi\leq 4(s_{r}-s_{l})^{2}||\Y-\Y^{n}||_{L^{\infty}([s_{l},s_{r}])}.
	\end{equation}
	
	Returning to \eqref{eq:IntEst1} we now get
	\begin{align*}
		&\bigg|\int_{\X(s)}^{X}\bigg[\big[Z_{3}-Z_{3}^{n}\big]^{2}+\sum_{i=1}^{5}\Big(\big[Z_{i,X}-Z_{i,X}^{n}\big]^{2}+\big[Z_{i,Y}-Z_{i,Y}^{n}\big]^{2}\Big)\bigg](\xi,2s-\xi)\,d\xi\bigg|\\
		&\leq K\Bigg\{(X_{r}-X_{l})||U_{1}-U_{1}^{n}||_{L^{\infty}([X_{l},X_{r}])}^{2}\\
		&\hspace{33pt}+\sum_{j=1}^{5}\Big(||\V_{j}-\V_{j}^{n}||_{L^{2}([X_{l},X_{r}])}^{2}+||\W_{j}-\W_{j}^{n}||_{L^{2}([Y_{l},Y_{r}])}^{2}\Big)\\
		&\hspace{33pt}+4(s_{r}-s_{l})^{2}\Big(||\X-\X^{n}||_{L^{\infty}([s_{l},s_{r}])}+||\Y-\Y^{n}||_{L^{\infty}([s_{l},s_{r}])}\Big)\Bigg\},
	\end{align*}
	which we insert in \eqref{eq:IntEst0} and get 
	\begin{aalign}
	\label{eq:ZIneq2}
		&\big|Z_{i}(X,2s-X)-Z_{i}^{n}(X,2s-X)\big|\\
		&\leq ||\Z_{i}-\Z_{i}^{n}||_{L^{\infty}([s_{l},s_{r}])}+(B_{4}+B_{5})||\X-\X^{n}||_{L^{\infty}([s_{l},s_{r}])} \\
		&\quad+2\sqrt{K(X_{r}-X_{l})} \Bigg\{(X_{r}-X_{l})||U_{1}-U_{1}^{n}||_{L^{\infty}([X_{l},X_{r}])}^{2}\\
		&\hspace{30pt}+\sum_{j=1}^{5}\Big(||\V_{j}-\V_{j}^{n}||_{L^{2}([X_{l},X_{r}])}^{2}+||\W_{j}-\W_{j}^{n}||_{L^{2}([Y_{l},Y_{r}])}^{2}\Big)\\
		&\hspace{30pt}+4(s_{r}-s_{l})^{2}\Big(||\X-\X^{n}||_{L^{\infty}([s_{l},s_{r}])}+||\Y-\Y^{n}||_{L^{\infty}([s_{l},s_{r}])}\Big)\Bigg\}^{\frac{1}{2}}
	\end{aalign}
	for $i=1,\dots,4$. A similar inequality is valid for
	\begin{equation*}
		|(Z_{5}(X,2s-X)-Z_{5}(X_{l},2s_{l}-X_{l}))-(Z_{5}^{n}(X,2s-X)-Z_{5}^{n}(X_{l},2s_{l}-X_{l}))\big|,
	\end{equation*} 
	the only difference from \eqref{eq:ZIneq2} being that we get $||(\Z_{5}-\Z_{5}(s_{l}))-(\Z_{5}^{n}-\Z_{5}^{n}(s_{l}))||_{L^{\infty}([s_{l},s_{r}])}$ on the right-hand side.
	
	\textbf{Step 4.} For any $0<\tau\leq\frac{1}{2\kappa}(x_{r}-x_{l})$ consider
	\begin{equation*}
		\Theta(\tau)=\mathbf{E}\circ\mathbf{t}_{\tau}(Z,p,q), \quad \Theta^{n}(\tau)=\mathbf{E}\circ\mathbf{t}_{\tau}(Z^{n},p^{n},q^{n}),
	\end{equation*}
	\begin{equation*}
		(u,R,S,\rho,\sigma,\mu,\nu)(\tau)=\mathbf{M}\circ\mathbf{D}(\Theta(\tau))
	\end{equation*}
	and 
	\begin{equation*}
		(u^{n},R^{n},S^{n},\rho^{n},\sigma^{n},\mu^{n},\nu^{n})(\tau)=\mathbf{M}\circ\mathbf{D}(\Theta^{n}(\tau)).
	\end{equation*}
	We prove \ref{eq:approxResult1}.
	
	A close inspection of the proof of \eqref{eq:s1bars2bar} reveals that there exist $\bar{s}_{1}$ and $\bar{s}_{2}$ such that $s_{l}<\bar{s}_{1}\leq\bar{s}_{2}<s_{r}$, $(\X(\tau,s),\Y(\tau,s))\in\Omega$ for all $s\in[\bar{s}_{1},\bar{s}_{2}]$ and
	\begin{equation*}
	x(\X(\tau,\bar{s}_{1}),\Y(\tau,\bar{s}_{1}))=x_{l}+\kappa\tau \quad \text{and} \quad x(\X(\tau,\bar{s}_{2}),\Y(\tau,\bar{s}_{2}))=x_{r}-\kappa\tau.
	\end{equation*}
	Moreover, there exist $\bar{s}_{1}^{n}$ and $\bar{s}_{2}^{n}$ such that $s_{l}<\bar{s}_{1}^{n}\leq\bar{s}_{2}^{n}<s_{r}$, $(\X^{n}(\tau,s),\Y^{n}(\tau,s))\in\Omega$ for all $s\in[\bar{s}_{1}^{n},\bar{s}_{2}^{n}]$ and
	\begin{equation*}
	x^{n}(\X^{n}(\tau,\bar{s}_{1}^{n}),\Y^{n}(\tau,\bar{s}_{1}^{n}))=x_{l}+\kappa\tau \quad \text{and} \quad x^{n}(\X^{n}(\tau,\bar{s}_{2}^{n}),\Y^{n}(\tau,\bar{s}_{2}^{n}))=x_{r}-\kappa\tau.
	\end{equation*}
	Consider $z\in[x_{l}+\kappa\tau,x_{r}-\kappa\tau]$. There are $s\in[\bar{s}_{1},\bar{s}_{2}]$ and $s^{n}\in[\bar{s}_{1}^{n},\bar{s}_{2}^{n}]$ such that $z=\Z_{2}(\tau,s)=\Z_{2}^{n}(\tau,s^{n})$. Using \eqref{eq:mapGtoD1} we obtain
	\begin{aalign}
	\label{eq:uDiff}
		u(\tau,z)-u^{n}(\tau,z)&=\Z_{3}(\tau,s)-\Z_{3}^{n}(\tau,s^{n})\\
		&=U(\X(\tau,s),\Y(\tau,s))-U^{n}(\X^{n}(\tau,s^{n}),\Y^{n}(\tau,s^{n})).
	\end{aalign}
	Now we can have several different scenarios depending on the order of the points $s$ and $s^{n}$, and the intervals $[\bar{s}_{1},\bar{s}_{2}]$ and $[\bar{s}_{1}^{n},\bar{s}_{2}^{n}]$. We only show one of the challenging cases, the others can be treated in a similar way. Assume that $s\leq \bar{s}_{1}^{n}\leq \bar{s}_{\text{max}}$. For the definition of $\bar{s}_{\text{max}}$ see Step 1 (iv) in the proof of Theorem \ref{thm:Regular1st}. Observe that in this case the point $(\X(\tau,\bar{s}_{1}^{n}),\Y(\tau,\bar{s}_{1}^{n}))$ is in $\Omega$, but the point
	$(\X^{n}(\tau,s),\Y^{n}(\tau,s))$ may be outside $\Omega$. So when estimating \eqref{eq:uDiff} we have to carefully choose points on the curve so that we do not end up outside $\Omega$, see Figure \ref{fig:FigThm73}. 
	
	\begin{figure}
		\centerline{\hbox{\includegraphics[width=10cm]{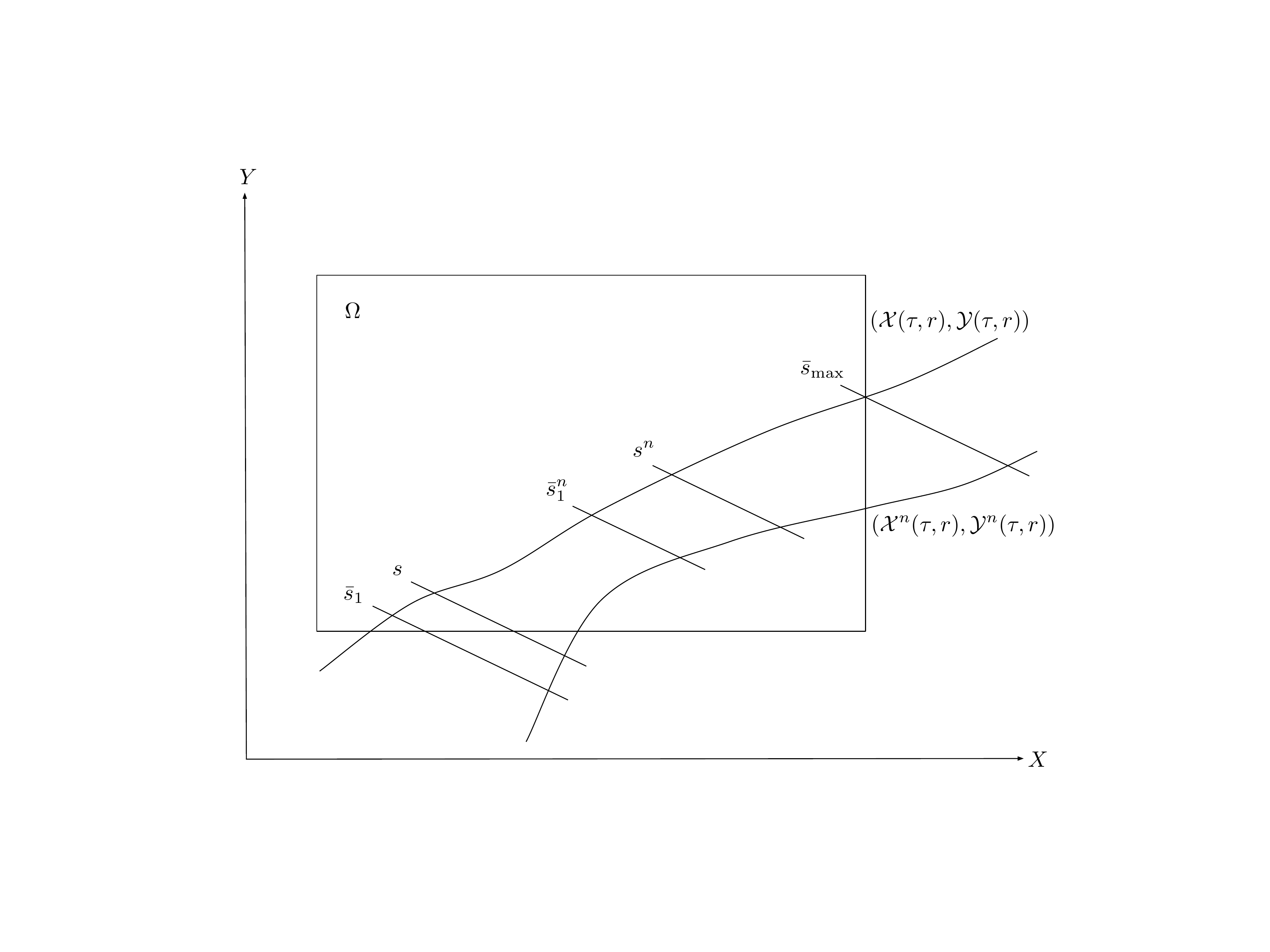}}}
		\caption{An example of the situation described in Step 4. Note that the point $(\X^{n}(\tau,s),\Y^{n}(\tau,s))$ lies outside $\Omega$.}
		\label{fig:FigThm73}
	\end{figure}
	
	Write\footnote{In the case when $\bar{s}_{1}^{n}\geq \bar{s}_{\text{max}}$ the point $(\X(\tau,\bar{s}_{1}^{n}),\Y(\tau,\bar{s}_{1}^{n}))$ may be outside $\Omega$. In that case the proof can be done in a similar way by replacing $\bar{s}_{1}^{n}$ with $\bar{s}_{\text{max}}$.}
	\begin{align*}
		&U(\X(\tau,s),\Y(\tau,s))-U^{n}(\X^{n}(\tau,s^{n}),\Y^{n}(\tau,s^{n}))\\
		&=U(\X(\tau,s),\Y(\tau,s))-U(\X(\tau,\bar{s}_{1}^{n}),\Y(\tau,\bar{s}_{1}^{n})) \quad (A_{1}^{n})\\
		&\quad + U(\X(\tau,\bar{s}_{1}^{n}),\Y(\tau,\bar{s}_{1}^{n}))-U(\X^{n}(\tau,\bar{s}_{1}^{n}),\Y^{n}(\tau,\bar{s}_{1}^{n})) \quad (A_{2}^{n})\\
		&\quad +U(\X^{n}(\tau,\bar{s}_{1}^{n}),\Y^{n}(\tau,\bar{s}_{1}^{n}))-U^{n}(\X^{n}(\tau,\bar{s}_{1}^{n}),\Y^{n}(\tau,\bar{s}_{1}^{n})) \quad (A_{3}^{n})\\
		&\quad +U^{n}(\X^{n}(\tau,\bar{s}_{1}^{n}),\Y^{n}(\tau,\bar{s}_{1}^{n}))-U^{n}(\X^{n}(\tau,s^{n}),\Y^{n}(\tau,s^{n})) \quad (A_{4}^{n}).
	\end{align*}
	
	By \eqref{eq:setGcomp} and the Cauchy--Schwarz inequality we have
	\begin{aalign}
	\label{eq:Z3FirstDiff}	
		|A_{1}^{n}|&=|\Z_{3}(\tau,\bar{s}_{1}^{n})-\Z_{3}(\tau,s)|\\
		&=\left| \int_{s}^{\bar{s}_{1}^{n}}\dot{\Z}_{3}(\tau,r)\,dr\,\right|\\
		&=\left|\int_{s}^{\bar{s}_{1}^{n}}[\V_{3}(\tau,\X(\tau,r))\dot{\X}(\tau,r)+\W_{3}(\tau,\Y(\tau,r))\dot{\Y}(\tau,r)]\,dr\,\right|\\
		&\leq \bigg(\int_{s}^{\bar{s}_{1}^{n}}\dot{\X}(\tau,r)\,dr\bigg)^{\frac{1}{2}}\bigg(\int_{s}^{\bar{s}_{1}^{n}}\V_{3}^{2}(\tau,\X(\tau,r))\dot{\X}(\tau,r)\,dr\bigg)^{\frac{1}{2}}\\
		&\quad +\bigg(\int_{s}^{\bar{s}_{1}^{n}}\dot{\Y}(\tau,r)\,dr\bigg)^{\frac{1}{2}}\bigg(\int_{s}^{\bar{s}_{1}^{n}}\W_{3}^{2}(\tau,\Y(\tau,r))\dot{\Y}(\tau,r)\,dr\bigg)^{\frac{1}{2}}.
	\end{aalign}
	From \eqref{eq:setGrel3} and \eqref{eq:setG0rel2} we get
	\begin{aalign}
	\label{eq:V3XdotInt}
		0\leq &\int_{s}^{\bar{s}_{1}^{n}}\V_{3}^{2}(\tau,\X(\tau,r))\dot{\X}(\tau,r)\,dr\\
		&=\int_{s}^{\bar{s}_{1}^{n}}\bigg(\frac{2\V_{2}(\tau,\X(\tau,r))\V_{4}(\tau,\X(\tau,r))}{c^{2}(\Z_{3}(\tau,r))}-\frac{\p^{2}(\tau,\X(\tau,r))}{c(\Z_{3}(\tau,r))}\bigg)\dot{\X}(\tau,r)\,dr\\
		&\leq \int_{s}^{\bar{s}_{1}^{n}}\frac{2\V_{2}(\tau,\X(\tau,r))\V_{4}(\tau,\X(\tau,r))}{c^{2}(\Z_{3}(\tau,r))}\dot{\X}(\tau,r)\,dr\\
		&\leq \kappa^{2}B_{6}\int_{s}^{\bar{s}_{1}^{n}}2\V_{2}(\tau,\X(\tau,r))\dot{\X}(\tau,r)\,dr\\
		&=\kappa^{2}B_{6}\int_{s}^{\bar{s}_{1}^{n}}\dot{\Z}_{2}(\tau,r)\,dr\\
		&=\kappa^{2}B_{6}\big[\Z_{2}(\tau,\bar{s}_{1}^{n})-\Z_{2}(\tau,s)\big]\\
		&=\kappa^{2}B_{6}\big[\Z_{2}(\tau,\bar{s}_{1}^{n})-\Z_{2}^{n}(\tau,\bar{s}_{1}^{n})+\Z_{2}^{n}(\tau,\bar{s}_{1}^{n})-\Z_{2}^{n}(\tau,s^{n})\big]\\
		&\leq\kappa^{2}B_{6}\big[\Z_{2}(\tau,\bar{s}_{1}^{n})-\Z_{2}^{n}(\tau,\bar{s}_{1}^{n})\big],
	\end{aalign}
	where we used that $\Z_{2}(\tau,s)=\Z_{2}^{n}(\tau,s^{n})$ and $\Z_{2}^{n}(\tau,\bar{s}_{1}^{n})\leq\Z_{2}^{n}(\tau,s^{n})$ since $\bar{s}_{1}^{n}\leq s^{n}$ and $\Z_{2}^{n}(\tau,\cdot)$ is nondecreasing. We also used that since $(\X(\tau,r),\Y(\tau,r))\in\Omega$ for $r\in[s,\bar{s}_{1}^{n}]$ we can use the estimate in \eqref{eq:ZXUnifBound} to obtain $|\V_{4}(\tau,\X(\tau,r))|=|J_{X}(\X(\tau,r),\Y(\tau,r))|\leq B_{6}$. We have
	\begin{aalign}
	\label{eq:Z2tauDiff}
		&\Z_{2}(\tau,\bar{s}_{1}^{n})-\Z_{2}^{n}(\tau,\bar{s}_{1}^{n})\\
		&=x(\X(\tau,\bar{s}_{1}^{n}),\Y(\tau,\bar{s}_{1}^{n}))-x^{n}(\X^{n}(\tau,\bar{s}_{1}^{n}),\Y^{n}(\tau,\bar{s}_{1}^{n}))\\
		&=x(\X(\tau,\bar{s}_{1}^{n}),\Y(\tau,\bar{s}_{1}^{n}))-x(\X^{n}(\tau,\bar{s}_{1}^{n}),\Y^{n}(\tau,\bar{s}_{1}^{n}))\\
		&\quad +x(\X^{n}(\tau,\bar{s}_{1}^{n}),\Y^{n}(\tau,\bar{s}_{1}^{n})) -x^{n}(\X^{n}(\tau,\bar{s}_{1}^{n}),\Y^{n}(\tau,\bar{s}_{1}^{n})).
	\end{aalign}
	By \eqref{eq:setH1} and since $t_{X}\geq 0$ and $t_{Y}\leq 0$ we get 
	\begin{align*}
		&|x(\X(\tau,\bar{s}_{1}^{n}),\Y(\tau,\bar{s}_{1}^{n}))-x(\X^{n}(\tau,\bar{s}_{1}^{n}),\Y^{n}(\tau,\bar{s}_{1}^{n}))|\\
		&=\left|\int_{\X^{n}(\tau,\bar{s}_{1}^{n})}^{\X(\tau,\bar{s}_{1}^{n})}(x_{X}-x_{Y})(X,2\bar{s}_{1}^{n}-X)\,dX\,\right|\\
		&=\left|\int_{\X^{n}(\tau,\bar{s}_{1}^{n})}^{\X(\tau,\bar{s}_{1}^{n})}[c(U)(t_{X}+t_{Y})](X,2\bar{s}_{1}^{n}-X)\,dX\,\right|
		\\
		&\leq\kappa\left|\int_{\X^{n}(\tau,\bar{s}_{1}^{n})}^{\X(\tau,\bar{s}_{1}^{n})}(t_{X}-t_{Y})(X,2\bar{s}_{1}^{n}-X)\,dX\,\right|\\
		&=\kappa |t(\X(\tau,\bar{s}_{1}^{n}),\Y(\tau,\bar{s}_{1}^{n}))-t(\X^{n}(\tau,\bar{s}_{1}^{n}),\Y^{n}(\tau,\bar{s}_{1}^{n}))|.
	\end{align*}
	Since $t(\X(\tau,\bar{s}_{1}^{n}),\Y(\tau,\bar{s}_{1}^{n}))=\tau=t^{n}(\X^{n}(\tau,\bar{s}_{1}^{n}),\Y^{n}(\tau,\bar{s}_{1}^{n}))$ we have
	\begin{aalign}
		\label{eq:tIdentity}
		&t(\X(\tau,\bar{s}_{1}^{n}),\Y(\tau,\bar{s}_{1}^{n}))-t(\X^{n}(\tau,\bar{s}_{1}^{n}),\Y^{n}(\tau,\bar{s}_{1}^{n}))\\
		&=t^{n}(\X^{n}(\tau,\bar{s}_{1}^{n}),\Y^{n}(\tau,\bar{s}_{1}^{n}))-t(\X^{n}(\tau,\bar{s}_{1}^{n}),\Y^{n}(\tau,\bar{s}_{1}^{n})).
	\end{aalign}
	Therefore,
	\begin{align*}
		&|x(\X(\tau,\bar{s}_{1}^{n}),\Y(\tau,\bar{s}_{1}^{n}))-x(\X^{n}(\tau,\bar{s}_{1}^{n}),\Y^{n}(\tau,\bar{s}_{1}^{n}))|\\
		&\leq \kappa |t-t^{n}|(\X^{n}(\tau,\bar{s}_{1}^{n}),\Y^{n}(\tau,\bar{s}_{1}^{n}))
	\end{align*}
	and \eqref{eq:Z2tauDiff} yields
	\begin{equation}
	\label{eq:Z2Z2nDiff}
		0\leq \Z_{2}(\tau,\bar{s}_{1}^{n})-\Z_{2}^{n}(\tau,\bar{s}_{1}^{n})
		\leq \big(|x-x^{n}|+\kappa|t-t^{n}|\big)(\X^{n}(\tau,\bar{s}_{1}^{n}),\Y^{n}(\tau,\bar{s}_{1}^{n})).
	\end{equation}
	We insert this in \eqref{eq:V3XdotInt} and get 
	\begin{equation*}
		0\leq \int_{s}^{\bar{s}_{1}^{n}}\V_{3}^{2}(\tau,\X(\tau,r))\dot{\X}(\tau,r)\,dr
		\leq\kappa^{2}B_{6}\big(|x-x^{n}|+\kappa|t-t^{n}|\big)(\X^{n}(\tau,\bar{s}_{1}^{n}),\Y^{n}(\tau,\bar{s}_{1}^{n})).
	\end{equation*}
	Similarly we get
	\begin{align*}
		0\leq \int_{s}^{\bar{s}_{1}^{n}}\W_{3}^{2}(\tau,\Y(\tau,r))\dot{\Y}(\tau,r)\,dr
		\leq\kappa^{2}B_{7}\big(|x-x^{n}|+\kappa|t-t^{n}|\big)(\X^{n}(\tau,\bar{s}_{1}^{n}),\Y^{n}(\tau,\bar{s}_{1}^{n})).
	\end{align*}
	We have
	\begin{equation*}
		\int_{s}^{\bar{s}_{1}^{n}}\dot{\X}(\tau,r)\,dr=\X(\tau,\bar{s}_{1}^{n})-\X(\tau,s)\leq X_{r}-X_{l}
	\end{equation*}
	and
	\begin{equation*}
		\int_{s}^{\bar{s}_{1}^{n}}\dot{\Y}(\tau,r)\,dr\leq Y_{r}-Y_{l}.
	\end{equation*}
	Using these estimates in \eqref{eq:Z3FirstDiff} gives
	\begin{equation}
	\label{eq:A1n}
		|A_{1}^{n}|\leq v_{1}
		\big(|x-x^{n}|+\kappa|t-t^{n}|\big)^{\frac{1}{2}}(\X^{n}(\tau,\bar{s}_{1}^{n}),\Y^{n}(\tau,\bar{s}_{1}^{n}))
	\end{equation}
	for a constant $v_{1}$ that is independent of $n$.
	
	By \eqref{eq:setH3} we have the estimates
	\begin{equation*}
		|U_{X}|\leq\frac{\sqrt{2J_{X}x_{X}}}{c(U)} \quad \text{and} \quad |U_{Y}|\leq\frac{\sqrt{2J_{Y}x_{Y}}}{c(U)}
	\end{equation*}
	which leads to
	\begin{align*}
		|A_{2}^{n}|&\leq|U(\X(\tau,\bar{s}_{1}^{n}),\Y(\tau,\bar{s}_{1}^{n}))-U(\X^{n}(\tau,\bar{s}_{1}^{n}),\Y^{n}(\tau,\bar{s}_{1}^{n}))|\\
		&=\left|\int_{\X^{n}(\tau,\bar{s}_{1}^{n})}^{\X(\tau,\bar{s}_{1}^{n})}(U_{X}-U_{Y})(X,2\bar{s}_{1}^{n}-X)\,dX\,\right|\\
		&\leq \sqrt{2}\kappa\left|\int_{\X^{n}(\tau,\bar{s}_{1}^{n})}^{\X(\tau,\bar{s}_{1}^{n})}\Big(\sqrt{J_{X}x_{X}}+\sqrt{J_{Y}x_{Y}}\Big)(X,2\bar{s}_{1}^{n}-X)\,dX\,\right|\\
		&\leq \sqrt{2}\kappa\left|\int_{\X^{n}(\tau,\bar{s}_{1}^{n})}^{\X(\tau,\bar{s}_{1}^{n})}J_{X}(X,2\bar{s}_{1}^{n}-X)\,dX\,\right|^{\frac{1}{2}}\left|\int_{\X^{n}(\tau,\bar{s}_{1}^{n})}^{\X(\tau,\bar{s}_{1}^{n})}x_{X}(X,2\bar{s}_{1}^{n}-X)\,dX\,\right|^{\frac{1}{2}}\\
		&\quad +\sqrt{2}\kappa\left|\int_{\X^{n}(\tau,\bar{s}_{1}^{n})}^{\X(\tau,\bar{s}_{1}^{n})}J_{Y}(X,2\bar{s}_{1}^{n}-X)\,dX\,\right|^{\frac{1}{2}}\left|\int_{\X^{n}(\tau,\bar{s}_{1}^{n})}^{\X(\tau,\bar{s}_{1}^{n})}x_{Y}(X,2\bar{s}_{1}^{n}-X)\,dX\,\right|^{\frac{1}{2}}\\
		&\leq 2\sqrt{2}\kappa\left|\int_{\X^{n}(\tau,\bar{s}_{1}^{n})}^{\X(\tau,\bar{s}_{1}^{n})}(J_{X}+J_{Y})(X,2\bar{s}_{1}^{n}-X)\,dX\,\right|^{\frac{1}{2}}\\
		&\quad\times\left|\int_{\X^{n}(\tau,\bar{s}_{1}^{n})}^{\X(\tau,\bar{s}_{1}^{n})}(x_{X}+x_{Y})(X,2\bar{s}_{1}^{n}-X)\,dX\,\right|^{\frac{1}{2}}
	\end{align*}
	since $x_{Y}\geq 0$ and $J_{Y}\geq 0$. From \eqref{eq:setH1} we get
	\begin{align*}
		&\left|\int_{\X^{n}(\tau,\bar{s}_{1}^{n})}^{\X(\tau,\bar{s}_{1}^{n})}(x_{X}+x_{Y})(X,2\bar{s}_{1}^{n}-X)\,dX\,\right|\\
		&=\left|\int_{\X^{n}(\tau,\bar{s}_{1}^{n})}^{\X(\tau,\bar{s}_{1}^{n})}\big[c(U)(t_{X}-t_{Y})\big](X,2\bar{s}_{1}^{n}-X)\,dX\,\right|\\
		&\leq\kappa\left|\int_{\X^{n}(\tau,\bar{s}_{1}^{n})}^{\X(\tau,\bar{s}_{1}^{n})}(t_{X}-t_{Y})(X,2\bar{s}_{1}^{n}-X)\,dX\,\right|\\
		&=\kappa|t(\X(\tau,\bar{s}_{1}^{n}),\Y(\tau,\bar{s}_{1}^{n}))-t(\X^{n}(\tau,\bar{s}_{1}^{n}),\Y^{n}(\tau,\bar{s}_{1}^{n}))|\\
		&=\kappa|t-t^{n}|(\X^{n}(\tau,\bar{s}_{1}^{n}),\Y^{n}(\tau,\bar{s}_{1}^{n})),
	\end{align*}	
	where the last equality follows from \eqref{eq:tIdentity}. Using \eqref{eq:ZXUnifBound} and \eqref{eq:ZYUnifBound} yields
	\begin{align*}
		&\left|\int_{\X^{n}(\tau,\bar{s}_{1}^{n})}^{\X(\tau,\bar{s}_{1}^{n})}(J_{X}+J_{Y})(X,2\bar{s}_{1}^{n}-X)\,dX\,\right|\\
		&\leq (B_{6}+B_{7})|\X(\tau,\bar{s}_{1}^{n})-\X^{n}(\tau,\bar{s}_{1}^{n})|\\
		&\leq (B_{6}+B_{7})(X_{r}-X_{l}).
	\end{align*}
	Now we get
	\begin{equation}
	\label{eq:A2n}
		|A_{2}^{n}|\leq v_{2}|t-t^{n}|^{\frac{1}{2}}(\X^{n}(\tau,\bar{s}_{1}^{n}),\Y^{n}(\tau,\bar{s}_{1}^{n}))
	\end{equation}
	for a constant $v_{2}$ which is independent of $n$. 
	
	By estimating as we did for $A_{1}^{n}$ we get
	\begin{align*}
		|A_{4}^{n}|&\leq \bigg(\int_{\bar{s}_{1}^{n}}^{s^{n}}\dot{\X}^{n}(\tau,r)\,dr\bigg)^{\frac{1}{2}}\bigg(\int_{\bar{s}_{1}^{n}}^{s^{n}}(\V_{3}^{n})^{2}(\tau,\X^{n}(\tau,r))\dot{\X}^{n}(\tau,r)\,dr\bigg)^{\frac{1}{2}}\\
		&\quad +\bigg(\int_{\bar{s}_{1}^{n}}^{s^{n}}\dot{\Y}^{n}(\tau,r)\,dr\bigg)^{\frac{1}{2}}\bigg(\int_{\bar{s}_{1}^{n}}^{s^{n}}(\W_{3}^{n})^{2}(\tau,\Y^{n}(\tau,r))\dot{\Y}^{n}(\tau,r)\,dr\bigg)^{\frac{1}{2}}
	\end{align*}
	and
	\begin{align*}
		0\leq &\int_{\bar{s}_{1}^{n}}^{s^{n}}(\V_{3}^{n})^{2}(\tau,\X^{n}(\tau,r))\dot{\X}^{n}(\tau,r)\,dr\\
		&\leq \kappa^{2}B_{4}\big[\Z_{2}^{n}(\tau,s^{n})-\Z_{2}^{n}(\tau,\bar{s}_{1}^{n})\big]\\
		&=\kappa^{2}B_{4}\big[\Z_{2}(\tau,s)-\Z_{2}(\tau,\bar{s}_{1}^{n})+\Z_{2}(\tau,\bar{s}_{1}^{n})-\Z_{2}^{n}(\tau,\bar{s}_{1}^{n})\big]\\
		&\leq \kappa^{2}B_{4}\big[\Z_{2}(\tau,\bar{s}_{1}^{n})-\Z_{2}^{n}(\tau,\bar{s}_{1}^{n})\big],
	\end{align*}
	where we used that $\Z_{2}^{n}(\tau,s^{n})=\Z_{2}(\tau,s)$ and $\Z_{2}(\tau,s)\leq\Z_{2}(\tau,\bar{s}_{1}^{n})$ since $s\leq\bar{s}_{1}^{n}$ and $\Z_{2}(\tau,\cdot)$ is nondecreasing. By using \eqref{eq:Z2Z2nDiff} we end up with
	\begin{equation}
	\label{eq:A4n}
		|A_{4}^{n}|\leq v_{3}
		\big(|x-x^{n}|+\kappa|t-t^{n}|\big)^{\frac{1}{2}}(\X^{n}(\tau,\bar{s}_{1}^{n}),\Y^{n}(\tau,\bar{s}_{1}^{n}))
	\end{equation}
	for a constant $v_{3}$ that is independent of $n$.
	
	Using \eqref{eq:A1n}, \eqref{eq:A2n} and \eqref{eq:A4n} in \eqref{eq:uDiff} yields
	\begin{align*}
		&|u(\tau,z)-u^{n}(\tau,z)|\\
		&\leq\Big[(v_{1}+v_{3})\big(|x-x^{n}|+\kappa|t-t^{n}|\big)^{\frac{1}{2}}+v_{2}|t-t^{n}|^{\frac{1}{2}}+|U-U^{n}|\Big](\X^{n}(\tau,\bar{s}_{1}^{n}),\Y^{n}(\tau,\bar{s}_{1}^{n}))
	\end{align*}
	and we conclude, after using \eqref{eq:ZIneq2}, that $u^{n}(\tau,\cdot)\rightarrow u(\tau,\cdot)$ in $\Linf([x_{l}+\kappa\tau,x_{r}-\kappa\tau])$.
	
	For any $z\in[x_{l}+\kappa\tau,x_{r}-\kappa\tau]$, there are $s\in[\bar{s}_{1},\bar{s}_{2}]$ and $s^{n}\in[\bar{s}_{1}^{n},\bar{s}_{2}^{n}]$ such that $z=\Z_{2}(\tau,s)=\Z_{2}^{n}(\tau,s^{n})$.
	From \eqref{eq:mapGtoD1} and \eqref{eq:mapGtoD3} we have
	\begin{align*}
		\int_{x_{l}+\kappa\tau}^{z}\frac{R(\tau,y)}{c(u(\tau,y))}\,dy&=\int_{\bar{s}_{1}}^{s}\frac{1}{c(u(\tau,\Z_{2}(\tau,r)))}2c(\Z_{3}(\tau,r))\V_{3}(\tau,\X(\tau,r))\dot{\X}(\tau,r)\,dr\\
		&=\int_{\bar{s}_{1}}^{s}2\V_{3}(\tau,\X(\tau,r))\dot{\X}(\tau,r)\,dr,
	\end{align*}
	and similarly we have
	\begin{equation*}
		\int_{x_{l}+\kappa\tau}^{z}\frac{S(\tau,y)}{c(u(\tau,y))}\,dy=\int_{\bar{s}_{1}}^{s}-2\W_{3}(\tau,\Y(\tau,r))\dot{\Y}(\tau,r)\,dr.
	\end{equation*}
	Using \eqref{eq:setDux} and \eqref{eq:setGcomp} yields
	\begin{align*}
		\int_{x_{l}+\kappa\tau}^{z}u_{x}(\tau,y)\,dy&=\int_{x_{l}+\kappa\tau}^{z}\bigg(\frac{R-S}{2c(u)}\bigg)(\tau,y)\,dy\\
		&=\int_{\bar{s}_{1}}^{s}\dot{\Z}_{3}(\tau,r)\,dr\\
		&=\Z_{3}(\tau,s)-\Z_{3}(\tau,\bar{s}_{1})\\
		&=u(\tau,z)-u(\tau,x_{l}+\kappa\tau),
	\end{align*}
	where the last equality follows from \eqref{eq:mapGtoD1}. According to Theorem 3.35 in \cite{Folland}, the function $u(\tau,\cdot)$ is absolutely continuous on $[x_{l}+\kappa\tau,x_{r}-\kappa\tau]$. We also have
	\begin{equation*}
		\int_{x_{l}+\kappa\tau}^{z}u_{x}^{n}(\tau,y)\,dy=
		u^{n}(\tau,z)-u^{n}(\tau,x_{l}+\kappa\tau).
	\end{equation*}
	Therefore
	\begin{equation*}
		\left|\int_{x_{l}+\kappa\tau}^{z}(u_{x}-u_{x}^{n})(\tau,y)\,dy\,\right|\leq 2||u(\tau,\cdot)-u^{n}(\tau,\cdot)||_{\Linf([x_{l}+\kappa\tau,x_{r}-\kappa\tau])},
	\end{equation*}
	so that
	\begin{equation*}
		\lim_{n\rightarrow\infty}\int_{x_{l}+\kappa\tau}^{z}u_{x}^{n}(\tau,y)\,dy=\int_{x_{l}+\kappa\tau}^{z}u_{x}(\tau,y)\,dy.
	\end{equation*}
		
	Since $p_{Y}=p_{Y}^{n}=0$ we have
	\begin{equation*}
	\p(\tau,X)=p(X,\Y(\tau,\X^{-1}(\tau,X)))=p(X,\Y(\X^{-1}(X)))=\p(X)
	\end{equation*}
	and
	\begin{equation*}
		\p^{n}(\tau,X)=p^{n}(X,\Y^{n}(\tau,(\X^{n})^{-1}(\tau,X)))=p^{n}(X,\Y^{n}((\X^{n})^{-1}(X)))=\p^{n}(X).
	\end{equation*}
	From Step 2 we get $\p(\tau,X)=0$ and $\p^{n}(\tau,X)\geq k_{n}>0$ for all $X\in[X_{l},X_{r}]$. Then by \eqref{eq:mapGtoD5} we obtain
	\begin{equation*}
		\int_{x_{l}+\kappa\tau}^{z}\rho(\tau,y)\,dy=\int_{\bar{s}_{1}}^{s}2\p(\tau,\X(\tau,r))\dot{\X}(\tau,r)\,dr=0,	
	\end{equation*}
	so that by a change of variables,
	\begin{equation*}
		\left|\int_{x_{l}+\kappa\tau}^{z}(\rho^{n}-\rho)(\tau,y)\,dy\,\right|=\int_{\X^{n}(\tau,\bar{s}_{1}^{n})}^{\X^{n}(\tau,s^{n})}2\p^{n}(X)\,dX\leq 2\sqrt{X_{r}-X_{l}}||\p^{n}||_{L^{2}([X_{l},X_{r}])}.
	\end{equation*}
	Since $\p^{n}\rightarrow 0$ in $L^{2}([X_{l},X_{r}])$ this implies that
	\begin{equation}
	\label{eq:rhonConv2}
		\lim_{n\rightarrow\infty}\int_{x_{l}+\kappa\tau}^{z}\rho^{n}(\tau,y)\,dy=\int_{x_{l}+\kappa\tau}^{z}\rho(\tau,y)\,dy=0.
	\end{equation}
	As mentioned in the comment before the proof, there is for any $n$ a constant $\bar{d}_{n}>0$ such that
	\begin{equation*}
	\rho^{n}(\tau,z)\geq \bar{d}_{n}
	\end{equation*}
	for all $z\in[x_{l}+\kappa\tau,x_{r}-\kappa\tau]$. Since $\rho^{n}(\tau,\cdot)$ is positive, \eqref{eq:rhonConv2} is the same as $\rho^{n}(\tau,\cdot)\rightarrow 0$ in $L^{1}([x_{l}+\kappa\tau,x_{r}-\kappa\tau])$. 
	
	Similarly we obtain $\sigma^{n}(\tau,\cdot)\rightarrow 0$ in $L^{1}([x_{l}+\kappa\tau,x_{r}-\kappa\tau])$. This concludes the proof of \ref{eq:approxResult2}.

\end{proof}

\subsection{Convergence Results}
\label{section:aux}

\begin{lemma}
	\label{lemma:1}
	Let $(u,R,S,\rho,\sigma,\mu,\nu)$ and $(u^{n},R^{n},S^{n},\rho^{n},\sigma^{n},\mu^{n},\nu^{n})$ belong to $\D$, and assume that $\mu$, $\nu$, $\mu^{n}$ and $\nu^{n}$ are absolutely continuous. Consider 
	\begin{equation*}
		\psi_{1}=(x_{1},U_{1},J_{1},K_{1},V_{1},H_{1}), \quad \psi_{2}=(x_{2},U_{2},J_{2},K_{2},V_{2},H_{2}),
	\end{equation*}
	and 
	\begin{equation*}
		\psi_{1}^{n}=(x_{1}^{n},U_{1}^{n},J_{1}^{n},K_{1}^{n},V_{1}^{n},H_{1}^{n}), \quad \psi_{2}^{n}=(x_{2}^{n},U_{2}^{n},J_{2}^{n},K_{2}^{n},V_{2}^{n},H_{2}^{n})
	\end{equation*}
	defined by
	\begin{equation*}
		(\psi_{1},\psi_{2})=\mathbf{L}(u,R,S,\rho,\sigma,\mu,\nu)
	\end{equation*}
	and 
	\begin{equation*}
		(\psi_{1}^{n},\psi_{2}^{n})=\mathbf{L}(u^{n},R^{n},S^{n},\rho^{n},\sigma^{n},\mu^{n},\nu^{n}).
	\end{equation*}
	Assume that
	\begin{equation*}
	u^{n}\rightarrow u, \quad R^{n}\rightarrow R, \quad S^{n}\rightarrow S, \quad \rho^{n}\rightarrow\rho \quad \text{and} \quad \sigma^{n}\rightarrow\sigma \quad \text{in} \quad L^{2}(\mathbb{R}).
	\end{equation*}
	Then, $x_{i}$ and $x_{i}^{n}$ are strictly increasing, and
	\begin{equation*}
	x_{i}^{n}\rightarrow x_{i}, \quad (x_{i}^{n})^{-1}\rightarrow x_{i}^{-1}, \quad U_{i}^{n}\rightarrow U_{i}, \quad J_{i}^{n}\rightarrow J_{i}, \quad K_{i}^{n}\rightarrow K_{i} \quad \text{in} \quad L^{\infty}(\mathbb{R}),
	\end{equation*}
	\begin{equation*}
	U_{i}^{n}\rightarrow U_{i}, \quad V_{i}^{n}\rightarrow V_{i}, \quad H_{i}^{n}\rightarrow H_{i} \quad \text{in} \quad L^{2}(\mathbb{R}),
	\end{equation*}
	\begin{equation*}
	(x_{i}^{n})'\rightarrow x_{i}', \quad (J_{i}^{n})'\rightarrow J_{i}', \quad (K_{i}^{n})'\rightarrow K_{i}' \quad \text{in} \quad L^{1}(\mathbb{R})
	\end{equation*}
	for $i=1,2$.
\end{lemma}

We mention that the functions which converge in $L^{1}(\mathbb{R})$ also converge in $L^{2}(\mathbb{R})$, because convergence in $L^{1}(\mathbb{R})$ and (uniform) boundedness in $L^{\infty}(\mathbb{R})$ imply convergence in $L^{2}(\mathbb{R})$. 

Let us show this in detail for $(x_{1}^{n})'\to x_{1}'$. Since the measures are absolutely continuous we get as in \eqref{eq:x1derexpr},
\begin{equation*}
x_{1}'(X)=\frac{4}{(R_{0}^{2}+c(u_{0})\rho_{0}^{2})\circ x_{1}(X)+4},
\end{equation*}
so that $0\leq x_{1}'(X)\leq 1$. We also have $0\leq (x_{1}^{n})'(X)\leq 1$. We get
\begin{align*}
\int_{\mathbb{R}}\big(x_{1}'(X)-(x_{1}^{n})'(X)\big)^{2}\,dX&\leq\int_{\mathbb{R}}\big|x_{1}'(X)-(x_{1}^{n})'(X)\big|\big(x_{1}'(X)+(x_{1}^{n})'(X)\big)\,dX\\
&\leq 2\int_{\mathbb{R}}\big|x_{1}'(X)-(x_{1}^{n})'(X)\big|\,dX.
\end{align*}
Therefore, $(x_{1}^{n})'\rightarrow x_{1}'$ in  $L^{1}(\mathbb{R})$ and $L^{2}(\mathbb{R})$. Notice that it was important that $(x_{1}^{n})'$ was \emph{uniformly} bounded in $L^{\infty}(\mathbb{R})$, not just that $(x_{1}^{n})'-1\in L^{\infty}(\mathbb{R})$ for all $n$, which is what we get from the definition of $\F$.

\begin{proof}	
	We start by proving that $u^{n}\rightarrow u$ in $\Linf(\mathbb{R})$. By \eqref{eq:setDux} and the Cauchy Schwarz inequality, we have
	\begin{align*}
	&(u(x)-u^{n}(x))^{2}\\
	&=2\int_{-\infty}^{x}(u-u^{n})(u_{x}-u^{n}_{x})(y)\,dy\\
	&=\int_{-\infty}^{x}(u-u^{n})\bigg(\frac{1}{c(u)}(R-S)-\frac{1}{c(u^{n})}(R^{n}-S^{n})\bigg)(y)\,dy\\
        &\leq \int_{\mathbb{R}}\kappa|u-u^{n}|(|R|+|R^{n}|+|S|+|S^{n}|)(y)\,dy\\
	&\leq \kappa||u-u^{n}||_{L^{2}(\mathbb{R})}\big(||R||_{L^{2}(\mathbb{R})}+||R^{n}||_{L^{2}(\mathbb{R})}+||S||_{L^{2}(\mathbb{R})}+||S^{n}||_{L^{2}(\mathbb{R})}\big)\\
	&\leq\kappa||u-u^{n}||_{L^{2}(\mathbb{R})}(2||R||_{L^{2}(\mathbb{R})}+||R-R^{n}||_{L^{2}(\mathbb{R})}+2||S||_{L^{2}(\mathbb{R})}+||S-S^{n}||_{L^{2}(\mathbb{R})}),
        \end{align*}
	which implies, since $R^{n}\rightarrow R$, $S^{n}\rightarrow S$ and $u^{n}\rightarrow u$ in $L^{2}(\mathbb{R})$, that $u^{n}\rightarrow u$ in $\Linf(\mathbb{R})$. 
	
	Let 
	\begin{equation*}
		\psi=(\psi_{1},\psi_{2})=\mathbf{L}(u,R,S,\rho,\sigma,\mu,\nu)
	\end{equation*}
	and 
	\begin{equation*}
		\psi^{n}=(\psi_{1}^{n},\psi_{2}^{n})=\mathbf{L}(u^{n},R^{n},S^{n},\rho^{n},\sigma^{n},\mu^{n},\nu^{n}).
	\end{equation*}
	We will only prove the results for the components of $\psi_{1}$, the proof for $\psi_{2}$ can be done in a similar way.
	
	 Let $f=\frac{1}{4}(R^{2}+c(u)\rho^{2})$ and $f^{n}=\frac{1}{4}((R^{n})^{2}+c(u^{n})(\rho^{n})^{2})$. From \eqref{eq:setDabscont} 
	we have $\mu((-\infty,x))=\int_{-\infty}^{x}f(z)\,dz$ and $\mu^{n}((-\infty,x))=\int_{-\infty}^{x}f^{n}(z)\,dz$, since $\mu$ and $\mu^{n}$ are absolutely continuous measures. Consider the functions
	\begin{equation*}
		F(x)=x+\int_{-\infty}^{x}f(z)\,dz \quad \text{and} \quad F^{n}(x)=x+\int_{-\infty}^{x}f^{n}(z)\,dz,
	\end{equation*}
	which are strictly increasing and continuous.
	 By \eqref{eq:mapfromDtoF1}, we have that $F(x_{1}(X))=X$ and $F^{n}(x_{1}^{n}(X))=X$ for all $X\in \mathbb{R}$, which implies that $x_{1}$ and $x_{1}^{n}$ are strictly increasing. Since $x_{1}^{-1}=F$ and $(x_{1}^{n})^{-1}=F^{n}$, the inverses $x_{1}^{-1}$ and $(x_{1}^{n})^{-1}$ exist and are strictly increasing.
	We prove that $(x_{1}^{n})^{-1}\rightarrow x_{1}^{-1}$ in $\Linf(\mathbb{R})$. Since
	\begin{equation*}
		|x_{1}^{-1}(x)-(x_{1}^{n})^{-1}(x)|\leq ||f-f^{n}||_{L^{1}(\mathbb{R})}
	\end{equation*}
	we have to show that $f^{n}\rightarrow f$ in $L^{1}(\mathbb{R})$. By using the estimate
	\begin{equation}
	\label{eq:oneOvercestALT}
	\bigg|\frac{1}{c(u)}-\frac{1}{c(u^{n})}\bigg|=\bigg|\int_{u^{n}}^{u}-\frac{c'(w)}{c^{2}(w)}\,dw\bigg|\leq \kappa^{2}k_{1}|u-u^{n}|
	\end{equation}
	and the Cauchy--Schwarz inequality, we find
	\begin{align*}
	||f-f^{n}||_{L^{1}(\mathbb{R})}&=\frac{1}{4}\int_{\mathbb{R}}|R^{2}-(R^{n})^{2}+\rho^{2}(c(u)-c(u^{n}))+c(u^n)(\rho^{2}-(\rho^{n})^{2})|(x)\,dx\\
	&\leq \frac{1}{4}\int_{\mathbb{R}}\big(|R^{2}-(R^{n})^{2}|+k_{1}\rho^{2}|u-u^{n}|+\kappa|\rho^{2}-(\rho^{n})^{2}|\big)(x)\,dx\\
	&\leq \frac{1}{4}||R+R^{n}||_{L^{2}(\mathbb{R})}||R-R^{n}||_{L^{2}(\mathbb{R})}+\frac{k_{1}}{4}||\rho||_{L^{2}(\mathbb{R})}^2||u-u^{n}||_{\Linf(\mathbb{R})}\\
	&\quad+\frac{\kappa}{4}||\rho+\rho^{n}||_{L^{2}(\mathbb{R})}||\rho-\rho^{n}||_{L^{2}(\mathbb{R})}\\
	&\leq \frac{1}{4}\big(2||R||_{L^{2}(\mathbb{R})}+||R-R^{n}||_{L^{2}(\mathbb{R})}\big)||R-R^{n}||_{L^{2}(\mathbb{R})}\\
	&\quad+\frac{k_{1}}{4}||\rho||_{L^{2}(\mathbb{R})}^2||u-u^{n}||_{\Linf(\mathbb{R})}\\
	&\quad+\frac{\kappa}{4} \big(2||\rho||_{L^{2}(\mathbb{R})}+||\rho-\rho^{n}||_{L^{2}(\mathbb{R})}\big)||\rho-\rho^{n}||_{L^{2}(\mathbb{R})}.
	\end{align*}
	Since $R^{n}\rightarrow R$ and $\rho^{n}\rightarrow \rho$ in $L^{2}(\mathbb{R})$, and $u^{n}\rightarrow u$ in $\Linf(\mathbb{R})$, it follows that $f^{n}\rightarrow f$ in $L^{1}(\mathbb{R})$. Therefore $(x_{1}^{n})^{-1}\rightarrow x_{1}^{-1}$ in $\Linf(\mathbb{R})$.
	
	We prove that $x_{1}^{n}\rightarrow x_{1}$ in $\Linf(\mathbb{R})$. We first consider the case $x_{1}^{n}(X)\leq x_{1}(X)$. By direct calculations we get
	\begin{aalign}
	\label{eq:x1UnifConv1estALT}
	x_{1}(X)-x_{1}^{n}(X)&=-\int_{-\infty}^{x_{1}(X)}f(x)\,dx+\int_{-\infty}^{x_{1}^{n}(X)}f^{n}(x)\,dx\\
	&=-\int_{-\infty}^{x_{1}(X)}(f(x)-f^{n}(x))\,dx-\int_{x_{1}^{n}(X)}^{x_{1}(X)}f^{n}(x)\,dx\\
	&\leq -\int_{-\infty}^{x_{1}(X)}(f(x)-f^{n}(x))\,dx,
	\end{aalign}
	since $f^{n}\geq 0$. If $ x_{1}(X)\leq x_{1}^{n}(X)$, we get in a similar way that
	\begin{equation}
	\label{eq:x1UnifConv2estALT}
	x_{1}^{n}(X)-x_{1}(X)\leq \int_{-\infty}^{x_{1}(X)}(f(x)-f^{n}(x))\,dx.
	\end{equation}
	Combining \eqref{eq:x1UnifConv1estALT} and \eqref{eq:x1UnifConv2estALT}, we end up with
	\begin{equation*}
	|x_{1}(X)-x_{1}^{n}(X)|\leq ||f-f^{n}||_{L^{1}(\mathbb{R})},
	\end{equation*}
	which implies, since $f^{n}\rightarrow f$ in $L^{1}(\mathbb{R})$, that $x_{1}^{n}\rightarrow x_{1}$ in $\Linf(\mathbb{R})$. Using \eqref{eq:mapfromDtoF3} we get $J_{1}^{n}\rightarrow J_{1}$ in $\Linf(\mathbb{R})$. 

	We prove that $(x_{1}^{n})'\rightarrow x_{1}'$ in $L^{1}(\mathbb{R})$. As in \eqref{eq:x1derexpr} we get 
	\begin{equation*}
	(x_{1}^{n})'=\frac{1}{f^{n}\circ x_{1}^{n}+1} \quad \text{and} \quad x_{1}'=\frac{1}{f\circ x_{1}+1},
	\end{equation*} 
	so that
	\begin{equation}
	\label{eq:x1L1DiffALT}
	(x_{1}^{n})'-x_{1}'=(f\circ x_{1}-f^{n}\circ x_{1}^{n})(x_{1}^{n})'x_{1}'.
	\end{equation}
	For every $\varepsilon>0$, there exists a function $l$ in $C_{c}(\mathbb{R})$ such that $||f-l||_{L^{1}(\mathbb{R})}\leq\varepsilon$. Here $C_{c}(\mathbb{R})$ denotes the space of continuous functions with compact support. Applying the triangle inequality in \eqref{eq:x1L1DiffALT} yields
	\begin{aalign}
	\label{eq:x1DerL1convALT}
	&\int_{\mathbb{R}}|x_{1}'(X)-(x_{1}^{n})'(X)|\,dX\\
	&\leq \int_{\mathbb{R}}|f\circ x_{1}(X)-l\circ x_{1}(X)|x_{1}'(X)\,dX+\int_{\mathbb{R}}|l\circ x_{1}(X)-l\circ x_{1}^{n}(X)|\,dX\\
	&\quad+\int_{\mathbb{R}}|l\circ x_{1}^{n}(X)-f\circ x_{1}^{n}(X)|(x_{1}^{n})'(X)\,dX\\
	&\quad+\int_{\mathbb{R}}|f\circ x_{1}^{n}(X)-f^{n}\circ x_{1}^{n}(X)|(x_{1}^{n})'(X)\,dX\\
	&=2||f-l||_{L^{1}(\mathbb{R})}+||l\circ x_{1}-l\circ x_{1}^{n}||_{L^{1}(\mathbb{R})}+||f-f^{n}||_{L^{1}(\mathbb{R})},
	\end{aalign}
	by a change of variables. Furthermore we used that $0\leq x_{1}'\leq 1$ and $0\leq (x_{1}^{n})'\leq 1$. It remains to show that $\displaystyle\lim_{n\rightarrow\infty}||l\circ x_{1}-l\circ x_{1}^{n}||_{L^{1}(\mathbb{R})}=0$. We have $l\circ x_{1}^{n}\rightarrow l\circ x_{1}$ pointwise almost everywhere. In order to use the dominated convergence theorem we have to show that $l\circ x_{1}^{n}$ can be uniformly bounded by a function which belongs to $L^{1}(\mathbb{R})$. We prove a slightly more general result which will be used many times throughout the text.

	\begin{lemma}\label{lem:UnifBdd}
	Assume that $g\in C_{c}(\mathbb{R})$, and that $h$ and $h_{n}$ satisfy $h-\id,h_{n}-\id\in L^{\infty}(\mathbb{R})$ and $h_{n}\rightarrow h$ in $L^{\infty}(\mathbb{R})$. Then there exists a constant $0<K<\infty$ that is independent of $n$ such that
	\begin{equation}
	\label{eq:UnifBdd}
		|g\circ h_{n}|\leq ||g||_{L^{\infty}(\mathbb{R})}\chi_{[-K,K]},
	\end{equation}  
	where $\chi_{[-K,K]}$ denotes the indicator function of the interval $[-K, K]$.
	\end{lemma}
	\begin{proof}
	Since $g$ has compact support there is a constant $0<k<\infty$ such that $\text{supp}(g)\subset[-k,k]$.
	Writing $h_{n}(x)=h_{n}(x)-h(x)+h(x)-x+x$, we get $\{x\mid h_{n}(x)\in[-k,k]\}\subset[-K_{n},K_{n}]$ where $K_{n}=k+||h-h_{n}||_{L^{\infty}(\mathbb{R})}+||h-\id||_{L^{\infty}(\mathbb{R})}$. Since $h_{n}\rightarrow h$ in $L^{\infty}(\mathbb{R})$ we can find a constant $M$ such that  $||h-h_{n}||_{L^{\infty}(\mathbb{R})}\leq M$ for all $n$. If we set $K=k+M+||h-\id||_{L^{\infty}(\mathbb{R})}$ we get
	\begin{equation*}
		|g\circ h_n(x)|\leq \vert g\circ h_n(x)\vert \chi_{[-K,K]}\leq ||g||_{L^\infty(\mathbb{R})}\chi_{[-K,K]} ,	\end{equation*}
	which proves \eqref{eq:UnifBdd}.
	\end{proof}
	
	From \eqref{eq:UnifBdd} we conclude that $l\circ x_{1}^{n}$ can be uniformly bounded by an $L^{1}(\mathbb{R})$ function. By the dominated convergence theorem we obtain $\displaystyle\lim_{n\rightarrow\infty}||l\circ x_{1}-l\circ x_{1}^{n}||_{L^{1}(\mathbb{R})}=0$. We conclude that the right-hand side of \eqref{eq:x1DerL1convALT} can be made arbitrarily small,
	so that $(x_{1}^{n})'\rightarrow x_{1}'$ in $L^{1}(\mathbb{R})$ and as an immediate consequence $(J_1^n)'\to J_1'$ in $L^1(\Real)$.	
	
	We show that $U_{1}^{n}\rightarrow U_{1}$ in $\Linf(\mathbb{R})$. By \eqref{eq:mapfromDtoF4} and the Cauchy--Schwarz inequality we get
	\begin{align*}
	\big|U_{1}(X)-U_{1}^{n}(X)\big|&\leq\big|u(x_{1}(X))-u(x_{1}^{n}(X))\big|+\big|u(x_{1}^{n}(X))-u^{n}(x_{1}^{n}(X))\big|\\
	&\leq\bigg|\int_{x_{1}^{n}(X)}^{x_{1}(X)}u_{x}(x)\,dx\bigg|+\big|u(x_{1}^{n}(X))-u^{n}(x_{1}^{n}(X))\big|\\
	&\leq ||u_{x}||_{L^{2}(\mathbb{R})}||x_{1}-x_{1}^{n}||_{\Linf(\mathbb{R})}^{\frac{1}{2}}+||u-u^{n}||_{\Linf(\mathbb{R})}.
	\end{align*}
	From \eqref{eq:setDux} we have $||u_{x}||_{L^{2}(\mathbb{R})}\leq \frac{\kappa}{2}(||R||_{L^{2}(\mathbb{R})}+||S||_{L^{2}(\mathbb{R})})$, and since $x_{1}^{n}\rightarrow x_{1}$ and $u^{n}\rightarrow u$ in $L^{\infty}(\mathbb{R})$ we conclude that $U_{1}^{n}\rightarrow U_{1}$ in $\Linf(\mathbb{R})$.
	
	Let us prove that $U_{1}^{n}\rightarrow U_{1}$ in $L^{2}(\mathbb{R})$. Since $u\in H^{1}(\mathbb{R})$ there is for every $\varepsilon>0$ a continuous function $\eta$ with compact support such that $||u-\eta||_{L^{2}(\mathbb{R})}\leq\varepsilon$ and $||u-\eta||_{L^{\infty}(\mathbb{R})}\leq\varepsilon$. We have
	\begin{aalign}
	\label{eq:U1nconvU1}
	||U_{1}-U_{1}^{n}||_{L^{2}(\mathbb{R})}&\leq ||u\circ x_{1}-\eta\circ x_{1}||_{L^{2}(\mathbb{R})}+||\eta\circ x_{1}-\eta\circ x_{1}^{n}||_{L^{2}(\mathbb{R})}\\
	&\quad +||\eta\circ x_{1}^{n}-u\circ x_{1}^{n}||_{L^{2}(\mathbb{R})}+||u\circ x_{1}^{n}-u^{n}\circ x_{1}^{n}||_{L^{2}(\mathbb{R})}.
	\end{aalign}
	Let
	\begin{equation*}
	D_{1}=\bigg\{X\in\mathbb{R} \ | \ x_{1}'(X)<\frac{1}{2}\bigg\} \quad \text{and} \quad D_{2}=\bigg\{Y\in\mathbb{R} \ | \ x_{2}'(Y)<\frac{1}{2}\bigg\}.
	\end{equation*}
	Observe that 
	\begin{equation*}
	\meas(D_{1})\leq \int_\Real 2J_1'(X)\,dX \quad \text{and} \quad \meas(D_{2})\leq \int_\Real 2J_2'(Y)\,dY,
	\end{equation*}
	since $J_i'=1-x_i'$, $i=1,2$. Since $\displaystyle\lim_{X\rightarrow-\infty}J_{1}(X)=\displaystyle\lim_{Y\rightarrow-\infty}J_{2}(Y)=0$ we get
	\begin{equation*}
	\text{meas}(D_{1})\leq 2||J_{1}||_{L^{\infty}(\mathbb{R})} \quad \text{and} \quad \text{meas}(D_{2})\leq 2||J_{2}||_{L^{\infty}(\mathbb{R})}.
	\end{equation*}
	In a similar way we find that the measures of the sets
	\begin{equation*}
	D_{1}^{n}=\bigg\{X\in\mathbb{R} \ | \ (x_{1}^{n})'(X)<\frac{1}{2}\bigg\} \quad \text{and} \quad D_{2}^{n}=\bigg\{Y\in\mathbb{R} \ | \ (x_{2}^{n})'(Y)<\frac{1}{2}\bigg\}.
	\end{equation*}
	have the bounds
	\begin{equation*}
	\text{meas}(D_{1}^{n})\leq 2||J_{1}^{n}||_{L^{\infty}(\mathbb{R})} \quad \text{and} \quad \text{meas}(D_{2}^{n})\leq 2||J_{2}^{n}||_{L^{\infty}(\mathbb{R})}.
	\end{equation*}
	Since $J_{i}^{n}\rightarrow J_{i}$ in $\Linf(\mathbb{R})$, we can find constants $E_{1}$ and $E_{2}$ that are independent of $n$ such that $\text{meas}(D_{1}^{n})\leq E_{1}$ and $\text{meas}(D_{2}^{n})\leq E_{2}$ for all $n$. We have
	\begin{aalign}
	\label{eq:U1Conv1}
	&||u\circ x_{1}-\eta\circ x_{1}||_{L^{2}(\mathbb{R})}^{2}\\
	&=\int_{D_{1}}\big(u\circ x_{1}(X)-\eta\circ x_{1}(X)\big)^{2}\,dX+\int_{D_{1}^{c}}\big(u\circ x_{1}(X)-\eta\circ x_{1}(X)\big)^{2}\,dX\\
	&\leq\text{meas}(D_{1})||u-\eta||_{L^{\infty}(\mathbb{R})}^{2}+2\int_{D_{1}^{c}}\big(u\circ x_{1}(X)-\eta\circ x_{1}(X)\big)^{2}x_{1}'(X)\,dX\\
	&\leq\text{meas}(D_{1})||u-\eta||_{L^{\infty}(\mathbb{R})}^{2}+2\int_{\mathbb{R}}\big(u\circ x_{1}(X)-\eta\circ x_{1}(X)\big)^{2}x_{1}'(X)\,dX\\
	&=\text{meas}(D_{1})||u-\eta||_{L^{\infty}(\mathbb{R})}^{2}+2||u-\eta||_{L^{2}(\mathbb{R})}^{2}
	\end{aalign}
	by a change of variables. Similarly,
	\begin{equation}
	\label{eq:U1Conv2}
	||\eta\circ x_{1}^{n}-u\circ x_{1}^{n}||_{L^{2}(\mathbb{R})}^{2}\leq E_{1}||u-\eta||_{L^{\infty}(\mathbb{R})}^{2}+2||u-\eta||_{L^{2}(\mathbb{R})}^{2}
	\end{equation}
	and
	\begin{equation*}
	||u\circ x_{1}^{n}-u^{n}\circ x_{1}^{n}||_{L^{2}(\mathbb{R})}^{2}\leq E_{2}||u-u^{n}||_{L^{\infty}(\mathbb{R})}^{2}+2||u-u^{n}||_{L^{2}(\mathbb{R})}^{2}.
	\end{equation*}
	Since $u_{n}\rightarrow u$ in $L^{2}(\mathbb{R})$ and $L^{\infty}(\mathbb{R})$ we get that
	\begin{equation}
	\label{eq:U1Conv3}
		\displaystyle\lim_{n\rightarrow\infty}||u\circ x_{1}^{n}-u^{n}\circ x_{1}^{n}||_{L^{2}(\mathbb{R})}=0.
	\end{equation}
	
	For the remaining term in \eqref{eq:U1nconvU1} we use the dominated convergence theorem. We have $\eta\circ x_{1}^{n}\rightarrow \eta\circ x_{1}$ pointwise almost everywhere, and by Lemma~\ref{lem:UnifBdd} we find that $\eta\circ x_{1}^{n}$ can be uniformly bounded by an $L^{2}(\mathbb{R})$ function. By the dominated convergence theorem we get 
	\begin{equation}
	\label{eq:U1Conv4}
		\displaystyle\lim_{n\rightarrow\infty}||\eta\circ x_{1}-\eta\circ x_{1}^{n}||_{L^{2}(\mathbb{R})}=0.
	\end{equation} 
	
	From \eqref{eq:U1Conv3} and \eqref{eq:U1Conv4}, and since the right-hand sides of \eqref{eq:U1Conv1} and \eqref{eq:U1Conv2} can be made arbitrarily small, we conclude that $U_{1}^{n}\rightarrow U_{1}$ in $L^{2}(\mathbb{R})$.
	
	We prove that $K_{1}^{n}\rightarrow K_{1}$ in $\Linf(\mathbb{R})$. By \eqref{eq:mapfromDtoF6} we get
	\begin{aalign}
	\label{eq:K1Diff}
	K_{1}(X)-K_{1}^{n}(X)&=\int_{-\infty}^{X}\bigg[\frac{J_{1}'(\bar{X})}{c(U_{1}(\bar{X}))}-\frac{(J_{1}^{n})'(\bar{X})}{c(U_{1}^{n}(\bar{X}))}\bigg]\,d\bar{X}\\
	&=\int_{-\infty}^{X}J_{1}'(\bar{X})\bigg[\frac{1}{c(U_{1}(\bar{X}))}-\frac{1}{c(U_{1}^{n}(\bar{X}))}\bigg]\,d\bar{X}\\
	&\quad +\int_{-\infty}^{X}\frac{1}{c(U_{1}^{n}(\bar{X}))}\big[J_{1}'(\bar{X})-(J_{1}^{n})'(\bar{X})\big]\,d\bar{X}.
	\end{aalign}
	For the first term on the right-hand side we use the Cauchy--Schwarz inequality and get
	\begin{align*}
	&\bigg|\int_{-\infty}^{X}J_{1}'(\bar{X})\bigg[\frac{1}{c(U_{1}(\bar{X}))}-\frac{1}{c(U_{1}^{n}(\bar{X}))}\bigg]\,d\bar{X}\bigg|\\
	&\leq ||J_{1}'||_{L^{2}(\mathbb{R})}\bigg|\int_{-\infty}^{X}\bigg(\int_{U_{1}^{n}(\bar{X})}^{U_{1}(\bar{X})}-\frac{c'(U)}{c^{2}(U)}\,dU\bigg)^{2}\,d\bar{X}\bigg|^{\frac{1}{2}}\\
	&\leq k_{1}\kappa^{2}||J_{1}'||_{L^{2}(\mathbb{R})}||U_{1}-U_{1}^{n}||_{L^{2}(\mathbb{R})},
	\end{align*}
	and for the second term we have
	\begin{equation*}
	\bigg|\int_{-\infty}^{X}\frac{1}{c(U_{1}^{n}(\bar{X}))}\big[J_{1}'(\bar{X})-(J_{1}^{n})'(\bar{X})\big]\,d\bar{X}\bigg|\leq \kappa||J_{1}'-(J_{1}^{n})'||_{L^{1}(\mathbb{R})},
	\end{equation*}
	which implies that $K_{1}^{n}\rightarrow K_{1}$ in $\Linf(\mathbb{R})$.
	
	The above proof in fact also shows that $(K_{1}^{n})'\rightarrow K_{1}'$ in $L^{1}(\mathbb{R})$. This is because
	\begin{equation*}
	\int_{\mathbb{R}}\big|K_{1}'(X)-(K_{1}^{n})'(X)\big|\,dX=\int_{\mathbb{R}}\bigg|\frac{J_{1}'(X)}{c(U_{1}(X))}-\frac{(J_{1}^{n})'(X)}{c(U_{1}^{n}(X))}\bigg|\,dX,
	\end{equation*}
	which is very similar to the term we estimated above.
	
	We prove that $H_{1}^{n}\rightarrow H_{1}$ in $L^{2}(\mathbb{R})$. Since $\rho$ and $\rho^{n}$ belong to $L^{2}(\mathbb{R})$ there exist for every $\varepsilon>0$ functions $\phi$ and $\phi^{n}$ in $C_{c}(\mathbb{R})$ such that $||\rho-\phi||_{L^{2}(\mathbb{R})}\leq\varepsilon$ and $||\rho^{n}-\phi^{n}||_{L^{2}(\mathbb{R})}\leq\varepsilon$. Since $\rho^{n}\rightarrow\rho$ in $L^{2}(\mathbb{R})$ we can for every $\varepsilon>0$ choose $n$ so large that $||\rho-\rho^{n}||_{L^{2}(\mathbb{R})}\leq\varepsilon$. This implies that for large $n$ we have $||\phi-\phi^{n}||_{L^{2}(\mathbb{R})}\leq 3\varepsilon$. From \eqref{eq:mapfromDtoF7} we have
	\begin{align*}
		H_{1}-H_{1}^{n}&=\frac{1}{2}x_{1}'(\rho\circ x_{1}-\phi\circ x_{1})
		+\frac{1}{2}((\phi\circ x_{1})x_{1}'-(\phi^{n}\circ x_{1}^{n})(x_{1}^{n})')\\
		&\quad+\frac{1}{2}(x_{1}^{n})'(\phi^{n}\circ x_{1}^{n}-\rho^{n}\circ x_{1}^{n}),
	\end{align*}
	so that
	\begin{aalign}
	\label{eq:H1Conv}
		||H_{1}-H_{1}^{n}||_{L^{2}(\mathbb{R})}&\leq\frac{1}{2}||x_{1}'(\rho\circ x_{1}-\phi\circ x_{1})||_{L^{2}(\mathbb{R})}\\
		&\quad+\frac{1}{2}||(\phi\circ x_{1})x_{1}'-(\phi^{n}\circ x_{1}^{n})(x_{1}^{n})'||_{L^{2}(\mathbb{R})}\\
		&\quad+\frac{1}{2}||(x_{1}^{n})'(\phi^{n}\circ x_{1}^{n}-\rho^{n}\circ x_{1}^{n})||_{L^{2}(\mathbb{R})}.
	\end{aalign}
	Since $0\leq x_{1}'\leq 1$ we get for the first term on the right-hand side by a change of variables,  
	\begin{aalign}
	\label{eq:H1Conv1}
		||x_{1}'(\rho\circ x_{1}-\phi\circ x_{1})||_{L^{2}(\mathbb{R})}^{2}&\leq\int_{\mathbb{R}}x_{1}'(X)(\rho(x_{1}(X))-\phi(x_{1}(X)))^{2}\,dX\\
		&=||\rho-\phi||_{L^{2}(\mathbb{R})}^{2}.
	\end{aalign}
	Similarly we get for the third term,
	\begin{equation}
	\label{eq:H1Conv2}
		||(x_{1}^{n})'(\phi^{n}\circ x_{1}^{n}-\rho^{n}\circ x_{1}^{n})||_{L^{2}(\mathbb{R})}\leq ||\rho^{n}-\phi^{n}||_{L^{2}(\mathbb{R})}.
	\end{equation}
	For the second term we have
	\begin{aalign}
		\label{eq:H1Conv3}
		&||(\phi\circ x_{1})x_{1}'-(\phi^{n}\circ x_{1}^{n})(x_{1}^{n})'||_{L^{2}(\mathbb{R})}\\
		&\leq ||(\phi\circ x_{1})(x_{1}'-(x_{1}^{n})')||_{L^{2}(\mathbb{R})}+||(x_{1}^{n})'(\phi\circ x_{1}-\phi\circ x_{1}^{n})||_{L^{2}(\mathbb{R})}\\
		&\quad+||(x_{1}^{n})'(\phi\circ x_{1}^{n}-\phi^{n}\circ x_{1}^{n})||_{L^{2}(\mathbb{R})}\\
		&\leq \sqrt{2}||\phi||_{L^{\infty}(\mathbb{R})}||x_{1}'-(x_{1}^{n})'||_{L^{1}(\mathbb{R})}^{\frac{1}{2}}+||\phi\circ x_{1}-\phi\circ x_{1}^{n}||_{L^{2}(\mathbb{R})}\\
		&\quad+||\phi-\phi^{n}||_{L^{2}(\mathbb{R})},
	\end{aalign}
	where we used that $0\leq x_{1}'\leq 1$, $0\leq (x_{1}^{n})'\leq 1$, and a change of variables for the last term on the right-hand side. We have $\phi\circ x_{1}^{n}\rightarrow\phi\circ x_{1}$ pointwise almost everywhere, and by Lemma~\ref{lem:UnifBdd} we get that $\phi\circ x_{1}^{n}$ can be uniformly bounded by an $L^{2}(\mathbb{R})$ function. By the dominated convergence theorem we have $\displaystyle\lim_{n\rightarrow\infty}||\phi\circ x_{1}-\phi\circ x_{1}^{n}||_{L^{2}(\mathbb{R})}=0$. 
	
	Using the estimates \eqref{eq:H1Conv1}-\eqref{eq:H1Conv3} in \eqref{eq:H1Conv} we observe that all terms on the right-hand side can be made arbitrarily small, which implies that $H_{1}^{n}\rightarrow H_{1}$ in $L^{2}(\mathbb{R})$.

	We can prove in more or less the same way that $V_{1}^{n}\rightarrow V_{1}$ in $L^{2}(\mathbb{R})$, where we also have to use that $U_{1}^{n}\rightarrow U_{1}$ in $L^{\infty}(\mathbb{R})$ and the boundedness of $c$ and $c'$. 
\end{proof}

\begin{lemma}
	\label{lemma:2}
	Let $(\psi_{1},\psi_{2})$ and $(\psi_{1}^{n},\psi_{2}^{n})$ belong to $\F$, where $\psi_{i}=(x_{i},U_{i},J_{i},K_{i},V_{i},H_{i})$ and $\psi_{i}^{n}=(x_{i}^{n},U_{i}^{n},J_{i}^{n},K_{i}^{n},V_{i}^{n},H_{i}^{n})$. Assume that $x_{i}$ and $x_{i}^{n}$ are strictly increasing, and that $x_{i}+J_{i}=\id$, and $x_i^n+J_i^n=\id$. Consider
	\begin{equation*}
		(\X,\Y,\Z,\V,\W,\p,\q)=\mathbf{C}(\psi_{1},\psi_{2})
	\end{equation*}
	and
	\begin{equation*}
		(\X^{n},\Y^{n},\Z^{n},\V^{n},\W^{n},\p^{n},\q^{n})=\mathbf{C}(\psi_{1}^{n},\psi_{2}^{n}).
	\end{equation*}
	Assume
	\begin{equation*}
	x_{i}^{n}\rightarrow x_{i}, \quad (x_{i}^{n})^{-1}\rightarrow x_{i}^{-1}, \quad U_{i}^{n}\rightarrow U_{i}, \quad J_{i}^{n}\rightarrow J_{i}, \quad K_{i}^{n}\rightarrow K_{i} \quad \text{in} \quad L^{\infty}(\mathbb{R}),
	\end{equation*}
	\begin{equation*}
	U_{i}^{n}\rightarrow U_{i}, \quad V_{i}^{n}\rightarrow V_{i}, \quad H_{i}^{n}\rightarrow H_{i} \quad \text{in} \quad L^{2}(\mathbb{R}),
	\end{equation*}
	\begin{equation*}
	(x_{i}^{n})'\rightarrow x_{i}', \quad (J_{i}^{n})'\rightarrow J_{i}', \quad (K_{i}^{n})'\rightarrow K_{i}' \quad \text{in} \quad L^{2}(\mathbb{R})
	\end{equation*}
	for $i=1,2$. Then $\X$, $\Y$, $\X^{n}$ and $\Y^{n}$ are strictly increasing, and
	\begin{align*}
	&\X^{n}\rightarrow\X, \quad \Y^{n}\rightarrow\Y, \quad \Z_{i}^{n}\rightarrow\Z_{i} \quad \text{in} \quad \Linf(\mathbb{R}),
	\end{align*}
	\begin{equation*}
	\Z_{3}^{n}\rightarrow\Z_{3}, \quad \V_{i}^{n}\rightarrow\V_{i}, \quad \W_{i}^{n}\rightarrow\W_{i}, \quad \p^{n}\rightarrow\p, \quad \q^{n}\rightarrow\q \quad \text{in} \quad L^{2}(\mathbb{R})
	\end{equation*}
	for $i=1,\dots,5$.
\end{lemma}

\begin{proof} 
	Since $x_{i}$ and $x_{i}^{n}$ are continuous and strictly increasing for $i=1,2$, we have by \eqref{eq:mapFtoGX} that for every $s\in\mathbb{R}$ there exist \emph{unique} points $\X(s)$ and $\X^{n}(s)$ such that 
	\begin{equation*}
	x_{1}(\X(s))=x_{2}(2s-\X(s))\quad \text{and} \quad x_{1}^{n}(\X^{n}(s))=x_{2}^{n}(2s-\X^{n}(s)).
	\end{equation*}
	Moreover, $\X$ and $\X^{n}$ are strictly increasing and continuous. It follows that $\Y$ and $\Y^{n}$ are strictly increasing and continuous. Thus, the inverse functions $\X^{-1}$, $\Y^{-1}$, $(\X^{n})^{-1}$ and $(\Y^{n})^{-1}$ exist, and they are continuous and strictly increasing.
	
	We prove that $\X^{n}\rightarrow\X$ in $\Linf(\mathbb{R})$. To begin with we show that $J_{i}^{n}\circ (x_{i}^{n})^{-1}\rightarrow J_{i}\circ x_{i}^{-1}$ in $\Linf(\mathbb{R})$. Write
	\begin{equation*}
		J_{i}\circ x_{i}^{-1}(X)-J_{i}^{n}\circ (x_{i}^{n})^{-1}(X)=\int_{(x_{i}^{n})^{-1}(X)}^{x_{i}^{-1}(X)}J_{i}'(Z)\,dZ+
		J_{i}\circ (x_{i}^{n})^{-1}(X)-J_{i}^{n}\circ (x_{i}^{n})^{-1}(X)
	\end{equation*}
	to get
	\begin{equation*}
		|J_{i}\circ x_{i}^{-1}(X)-J_{i}^{n}\circ (x_{i}^{n})^{-1}(X)|\leq||x_{i}^{-1}-(x_{i}^{n})^{-1}||_{\Linf(\mathbb{R})}+||J_{i}-J_{i}^{n}||_{\Linf(\mathbb{R})}.
	\end{equation*}
	We used that $0\leq J_{i}'\leq 1$, where the upper bound comes from the identity $x_{i}'+J_{i}'=1$, and $x_{i}'\geq 0$. Using the convergence assumptions implies that $J_{i}^{n}\circ (x_{i}^{n})^{-1}\rightarrow J_{i}\circ x_{i}^{-1}$ in $\Linf(\mathbb{R})$.
	
	We insert $X=\X(s)$ in $x_{1}(X)+J_{1}(X)=X$ and get
	\begin{equation}
	\label{eq:ChiFtoG}
		\X(s)=x_1\circ \X(s)+J_1\circ \X(s)=x_{1}\circ\X(s)+J_{1}\circ x_1^{-1}\circ x_1\circ\X(s).
	\end{equation}
 Similarly we get
	\begin{equation}
	\label{eq:YFtoG}
		\Y(s)=x_2\circ \Y(s)+J_2\circ \Y(s)=x_{1}\circ\X(s)+J_{2}\circ x_{2}^{-1}\circ x_{1}\circ\X(s),
	\end{equation}
	where we used that $x_{1}(\X(s))=x_{2}(\Y(s))$.
	The expressions for $\X^{n}$ and $\Y^{n}$ are defined in a similar way. By \eqref{eq:initialcurvenormalization},
	\begin{equation*}
		\X(s)+\Y(s)-\X^{n}(s)-\Y^{n}(s)=0,
	\end{equation*} 
	which combined with the above expressions yields
	\begin{align*}
		&2\big(x_{1}\circ\X(s)-x_{1}^{n}\circ\X^{n}(s)\big)+J_{1}\circ x_{1}^{-1}\circ x_1\circ\X(s)-J_{1}^{n}\circ (x_1^{n})^{-1}\circ x_{1}^{n}\circ\X^{n}(s)\\
		&\quad +J_{2}\circ x_{2}^{-1}\circ x_{1}\circ\X(s)-
		J_{2}^{n}\circ (x_{2}^{n})^{-1}\circ x_{1}^{n}\circ\X^{n}(s)=0.
	\end{align*}
	Since $J_{i}^{n}\circ(x_{i}^{n})^{-1}$ is increasing we get
	\begin{align*}
		&2\big\vert x_{1}\circ\X(s)-x_{1}^{n}\circ\X^{n}(s)\big\vert\\ 
		&+\vert J_{1}\circ x_1^{-1}\circ x_1\circ \X(s)-J_{1}\circ x_1^{-1}\circ x_1^n\circ \X^n(s)\vert\\
		&+\vert J_2\circ x_2^{-1}\circ x_1\circ \X(s)-J_2\circ x_2^{-1}\circ x_1^n\circ \X^n(s)\vert \\
		&\quad \leq \vert J_1^n\circ (x_1^n)^{-1}\circ x_1^n\circ \X^n(s)-J_1\circ x_1^{-1}\circ x_1^n\circ \X^n(s)\vert \\
		&\quad \quad + \vert J_2^n\circ (x_2^n)^{-1}\circ x_1^n\circ \X^n(s)-J_2\circ x_2^{-1}\circ x_1^n\circ \X^n(s)\vert,
	\end{align*}
	Therefore
	\begin{align*}
		||x_{1}\circ\X-x_{1}^{n}\circ\X^{n}||_{\Linf(\mathbb{R})}&\leq\frac12||J_1^n\circ (x_1^n)^{-1}-J_1\circ x_1^{-1}||_{\Linf(\mathbb{R})}\\
		&\quad+ \frac{1}{2}||J_{2}^{n}\circ (x_{2}^{n})^{-1}-J_{2}\circ x_{2}^{-1}||_{\Linf(\mathbb{R})},
	\end{align*}
	\begin{align*}
		||J_{1}\circ x_{1}^{-1}\circ x_{1}\circ\X-
		J_{1}\circ x_{1}^{-1}\circ x_{1}^{n}\circ\X^{n}||_{\Linf(\mathbb{R})}&\leq||J_1^n\circ (x_1^n)^{-1}-J_1\circ x_1^{-1}||_{\Linf(\mathbb{R})}\\
		&\quad+||J_{2}^{n}\circ (x_{2}^{n})^{-1}-J_{2}\circ x_{2}^{-1}||_{\Linf(\mathbb{R})},
	\end{align*}
	and 	
	\begin{align*}
		|| J_1\circ \X-J_1^n\circ \X^n||_{\Linf(\mathbb{R})}&=||J_{1}\circ x_{1}^{-1}\circ x_{1}\circ\X-
		J_{1}^n\circ (x_{1}^n)^{-1}\circ x_{1}^{n}\circ\X^{n}||_{\Linf(\mathbb{R})}\\ &\leq 2||J_1^n\circ (x_1^n)^{-1}-J_1\circ x_1^{-1}||_{\Linf(\mathbb{R})}\\
		&\quad+||J_{2}^{n}\circ (x_{2}^{n})^{-1}-J_{2}\circ x_{2}^{-1}||_{\Linf(\mathbb{R})}.
	\end{align*}
	Thus, we showed that $x_{1}^{n}\circ\X^{n}\rightarrow x_{1}\circ\X$ and 
	$J_{1}^{n}\circ \X^{n}\rightarrow J_{1}\circ\X$ in $\Linf(\mathbb{R})$ since $J_{i}^{n}\circ (x_{i}^{n})^{-1}-J_{i}\circ x_{i}^{-1}$ in $\Linf(\mathbb{R})$.
	
	From \eqref{eq:ChiFtoG}, \eqref{eq:initialcurvenormalization}, \eqref{eq:YFtoG}, \eqref{eq:mapFtoGZ2} and \eqref{eq:mapFtoGZ4} it immediately follows that $\X^{n}\rightarrow\X$, $\Y^{n}\rightarrow\Y$, $\Z_{2}^{n}\rightarrow\Z_{2}$ and $\Z_{4}^{n}\rightarrow\Z_{4}$ in $L^{\infty}(\mathbb{R})$. 
	
	We show that $\Z_{5}^{n}\rightarrow\Z_{5}$ in $L^{\infty}(\mathbb{R})$. By \eqref{eq:mapFtoGZ5} we have
	\begin{align*}
	\Z_{5}(s)-\Z_{5}^{n}(s)&=K_{1}(\X(s))-K_{1}(\X^{n}(s))+K_{1}(\X^{n}(s))-K_{1}^{n}(\X^{n}(s))\\
	&\quad +K_{2}(\Y(s))-K_{2}(\Y^{n}(s))+K_{2}(\Y^{n}(s))-K_{2}^{n}(\Y^{n}(s))
	\end{align*}
	and for the first line we get
	\begin{align*}
		&\big|K_{1}(\X(s))-K_{1}(\X^{n}(s))+K_{1}(\X^{n}(s))-K_{1}^{n}(\X^{n}(s))\big|\\ 
		&=\bigg|\int_{\X^{n}(s)}^{\X(s)}K_{1}'(X)\,dX+K_{1}(\X^{n}(s))-K_{1}^{n}(\X^{n}(s))\bigg|\\
		&\leq ||K_{1}'||_{L^{\infty}(\mathbb{R})}||\X-\X^{n}||_{L^{\infty}(\mathbb{R})}+||K_{1}-K_{1}^{n}||_{L^{\infty}(\mathbb{R})}.
	\end{align*}
	A similar estimate for the second line yields $\Z_{5}^{n}\rightarrow\Z_{5}$ in $L^{\infty}(\mathbb{R})$. By \eqref{eq:mapFtoGZ3} and the Cauchy--Schwarz inequality we have
	\begin{align*}
	|\Z_{3}(s)-\Z_{3}^{n}(s)|&\leq \bigg|\int_{\X^{n}(s)}^{\X(s)}U_{1}'(X)\,dX\bigg|+|U_{1}(\X^{n}(s))-U_{1}^{n}(\X^{n}(s))|\\
	&\leq ||U_{1}'||_{L^{2}(\mathbb{R})}  ||\X-\X^{n}||_{L^{\infty}(\mathbb{R})}^{\frac{1}{2}}+||U_{1}-U_{1}^{n}||_{L^{\infty}(\mathbb{R})},
	\end{align*}
	which shows that $\Z_{3}^{n}\rightarrow\Z_{3}$ in $L^{\infty}(\mathbb{R})$.
	
	We prove that $\Z_{3}^{n}\rightarrow\Z_{3}$ in $L^{2}(\mathbb{R})$. We have
	\begin{equation}
	\label{eq:Z3L2Diff}
		||\Z_{3}-\Z_{3}^{n}||_{L^{2}(\mathbb{R})}^{2}=||\Z_{3}||_{L^{2}(\mathbb{R})}^{2}-2\langle \Z_{3},\Z_{3}^{n} \rangle+||\Z_{3}^{n}||_{L^{2}(\mathbb{R})}^{2},
	\end{equation}
	where $\langle \cdot,\cdot \rangle$ denotes the inner product on $L^{2}(\mathbb{R})$. Since $\dot{\X}+\dot{\Y}=2$ we get from \eqref{eq:mapFtoGZ3} and a change of variables,
	\begin{aalign}
	\label{eq:Z3L2Diff1}
		||\Z_{3}||_{L^{2}(\mathbb{R})}^{2}&=\frac{1}{2}\int_{\mathbb{R}}U_{1}^{2}(\X(s))\dot{\X}(s)\,ds+\frac{1}{2}\int_{\mathbb{R}}U_{2}^{2}(\Y(s))\dot{\Y}(s)\,ds\\
		&=\frac{1}{2}||U_{1}||_{L^{2}(\mathbb{R})}^{2}+\frac{1}{2}||U_{2}||_{L^{2}(\mathbb{R})}^{2},
	\end{aalign} 
	and similarly
	\begin{equation}
	\label{eq:Z3L2Diff2}
		||\Z_{3}^{n}||_{L^{2}(\mathbb{R})}^{2}=\frac{1}{2}||U_{1}^{n}||_{L^{2}(\mathbb{R})}^{2}+\frac{1}{2}||U_{2}^{n}||_{L^{2}(\mathbb{R})}^{2}.
	\end{equation}
	Using that $\dot{\X}^{n}+\dot{\Y}^{n}=2$ we have
	\begin{equation*}
		2\langle \Z_{3},\Z_{3}^{n} \rangle=\int_{\mathbb{R}}\Z_{3}(s)\Z_{3}^{n}(s)\dot{\X}^{n}(s)\,ds+\int_{\mathbb{R}}\Z_{3}(s)\Z_{3}^{n}(s)\dot{\Y}^{n}(s)\,ds.
	\end{equation*}
	Since $\Z_{3}\in L^{2}(\mathbb{R})$ there exists for every $\varepsilon>0$ a function $\phi\in C^{\infty}_{c}(\mathbb{R})$ such that $||\Z_{3}-\phi||_{L^{2}(\mathbb{R})}\leq\varepsilon$. Write
	\begin{align*}
		&\int_{\mathbb{R}}\Z_{3}(s)\Z_{3}^{n}(s)\dot{\X}^{n}(s)\,ds\\
		&=\int_{\mathbb{R}}\big[\Z_{3}(s)-\phi(s)\big]\Z_{3}^{n}(s)\dot{\X}^{n}(s)\,ds+\int_{\mathbb{R}}\phi(s)\Z_{3}^{n}(s)\dot{\X}^{n}(s)\,ds-\int_{\mathbb{R}}\phi(s)\Z_{3}(s)\dot{\X}(s)\,ds\\
		&\quad +\int_{\mathbb{R}}\Z_{3}(s)\dot{\X}(s)\big[\phi(s)-\Z_{3}(s)\big]\,ds +\int_{\mathbb{R}}\Z_{3}^{2}(s)\dot{\X}(s)\,ds\\
		&=T_{1}^{n}+\int_{\mathbb{R}}\Z_{3}^{2}(s)\dot{\X}(s)\,ds. 
	\end{align*}
	By a change of variables we get 
	\begin{equation}
	\label{eq:Z3L2Diff3}
		\int_{\mathbb{R}}\Z_{3}(s)\Z_{3}^{n}(s)\dot{\X}^{n}(s)\,ds=T_{1}^{n}+||U_{1}||_{L^{2}(\mathbb{R})}^{2}
	\end{equation}
	and in a similar way we obtain
	\begin{equation}
	\label{eq:Z3L2Diff4}
		\int_{\mathbb{R}}\Z_{3}(s)\Z_{3}^{n}(s)\dot{\Y}^{n}(s)\,ds=T_{2}^{n}+||U_{2}||_{L^{2}(\mathbb{R})}^{2},
	\end{equation}
	where $T_{2}^{n}$ is equal to $T_{1}^{n}$ with $\X(s)$ and $\X^{n}(s)$ replaced by $\Y(s)$ and $\Y^{n}(s)$, respectively.
	
	Using \eqref{eq:Z3L2Diff1}-\eqref{eq:Z3L2Diff4} in \eqref{eq:Z3L2Diff} we get
	\begin{aalign}
	\label{eq:Z3L2Diff5}
		||\Z_{3}-\Z_{3}^{n}||_{L^{2}(\mathbb{R})}^{2}&=\frac{1}{2}\Big(||U_{1}^{n}||_{L^{2}(\mathbb{R})}^{2}-||U_{1}||_{L^{2}(\mathbb{R})}^{2}+||U_{2}^{n}||_{L^{2}(\mathbb{R})}^{2}-||U_{2}||_{L^{2}(\mathbb{R})}^{2}\Big)\\
		&\quad-T_{1}^{n}-T_{2}^{n}.
	\end{aalign}	
	The strong convergence $U_{i}^{n}\rightarrow U_{i}$ in $L^{2}(\mathbb{R})$ implies that $||U_{i}^{n}||_{L^{2}(\mathbb{R})}\rightarrow||U_{i}||_{L^{2}(\mathbb{R})}$ for $i=1,2$. Thus, it remains to show that $T_{i}^{n}\rightarrow 0$ for $i=1,2$. 
	
	Using \eqref{eq:mapFtoGZ3} and $0\leq\dot{\X}^{n}\leq 2$ we get by the Cauchy--Schwarz inequality and a change of variables,
	\begin{equation}
	\label{eq:T1n1}
		\bigg|\int_{\mathbb{R}}\Z_{3}(s)\dot{\X}(s)\big[\phi(s)-\Z_{3}(s)\big]\,ds\bigg|\leq\sqrt{2}||U_{1}||_{L^{2}(\mathbb{R})}||\Z_{3}-\phi||_{L^{2}(\mathbb{R})},
	\end{equation}
	and similarly
	\begin{aalign}
	\label{eq:T1n2}
		&\bigg|\int_{\mathbb{R}}\big[\Z_{3}(s)-\phi(s)\big]\Z_{3}^{n}(s)\dot{\X}^{n}(s)\,ds\bigg|\\
		&\leq\sqrt{2}\Big(||U_{1}^{n}-U_{1}||_{L^{2}(\mathbb{R})}+||U_{1}||_{L^{2}(\mathbb{R})}\Big)||\Z_{3}-\phi||_{L^{2}(\mathbb{R})}.
	\end{aalign}
	
	Since $\phi$ has compact support, there exists $k>0$ such that $\text{supp}(\phi)\subset[-k,k]$. 
	Integration by parts yields
	\begin{align*}
		\int_{\mathbb{R}}\phi(s)\Z_{3}(s)\dot{\X}(s)\,ds=\int_{-k}^{k}\phi(s)\Z_{3}(s)\dot{\X}(s)\,ds=-\int_{-k}^{k}\phi'(s)\int_{-k}^{s}\Z_{3}(t)\dot{\X}(t)\,dt\,ds
	\end{align*}
	where the first term in the integration by parts equals zero because $\phi$ has compact support and the second integral is finite, since
	\begin{equation*}
		\bigg|\int_{-k}^{s}\Z_{3}(t)\dot{\X}(t)\,dt\bigg|=\bigg|\int_{\X(-k)}^{\X(s)}U_{1}(X)\,dX\bigg|\leq \big(2||\X-\id||_{L^{\infty}(\mathbb{R})}+2k\big)^{\frac{1}{2}} ||U_{1}||_{L^{2}(\mathbb{R})}
	\end{equation*}
	for $s\in[-k,k]$. Here we used a change of variables and the estimate
	\begin{aalign}
	\label{eq:ChiDiffk}
		\big|\X(k)-\X(-k)\big|&=\big|(\X(k)-k)-(\X(-k)-(-k))+2k\big|\\
		&\leq 2||\X-\id||_{L^{\infty}(\mathbb{R})}+2k.
	\end{aalign}
	We have
	\begin{equation*}
		\int_{\mathbb{R}}\phi(s)\Z_{3}(s)\dot{\X}(s)\,ds=-\int_{-k}^{k}\phi'(s)\int_{\X(-k)}^{\X(s)}U_{1}(X)\,dX\,ds
	\end{equation*}
	and similarly we obtain
	\begin{equation*}
		\int_{\mathbb{R}}\phi(s)\Z_{3}^{n}(s)\dot{\X}^{n}(s)\,ds=-\int_{-k}^{k}\phi'(s)\int_{\X^{n}(-k)}^{\X^{n}(s)}U_{1}^{n}(X)\,dX\,ds.
	\end{equation*}
	By the Cauchy--Schwarz inequality we get
	\begin{aalign}
	\label{eq:U1U1nDiff}
	&\bigg|\int_{\X(-k)}^{\X(s)}U_{1}(X)\,dX-\int_{\X^{n}(-k)}^{\X^{n}(s)}U_{1}^{n}(X)\,dX\bigg|\\
	&\leq\bigg|\int_{\X(-k)}^{\X(s)}\big[U_{1}(X)-U_{1}^{n}(X)\big]\,dX\bigg|\\
	&\quad+\bigg|\int_{\X^{n}(-k)}^{\X(-k)}U_{1}^{n}(X)\,dX\bigg|+\bigg|\int_{\X^{n}(s)}^{\X(s)}U_{1}^{n}(X)\,dX\bigg|\\
	&\leq ||U_{1}-U_{1}^{n}||_{L^{2}(\mathbb{R})}\big|\X(k)-\X(-k)\big|^{\frac{1}{2}}+2||U_{1}^{n}||_{L^{2}(\mathbb{R})}||\X-\X^{n}||_{L^{\infty}(\mathbb{R})}^{\frac{1}{2}},
	\end{aalign}
	which implies by \eqref{eq:ChiDiffk} that 
	\begin{align*}
		&\bigg|\int_{\mathbb{R}}\phi(s)\Z_{3}^{n}(s)\dot{\X}^{n}(s)\,ds-\int_{\mathbb{R}}\phi(s)\Z_{3}(s)\dot{\X}(s)\,ds\bigg|\\
		&\leq 2k||\phi'||_{L^{\infty}(\mathbb{R})}\Big[||U_{1}-U_{1}^{n}||_{L^{2}(\mathbb{R})}\big(2||\X-\id||_{L^{\infty}(\mathbb{R})}+2k\big)^{\frac{1}{2}}\\
		&\hspace{80pt}+2\big(||U_{1}^{n}-U_{1}||_{L^{2}(\mathbb{R})}+||U_{1}||_{L^{2}(\mathbb{R})}\big)||\X-\X^{n}||_{L^{\infty}(\mathbb{R})}^{\frac{1}{2}}\Big].
	\end{align*}
	Therefore
	\begin{equation}
	\label{eq:T1n3}
		\lim_{n\rightarrow\infty}\bigg|\int_{\mathbb{R}}\phi(s)\Z_{3}^{n}(s)\dot{\X}^{n}(s)\,ds-\int_{\mathbb{R}}\phi(s)\Z_{3}(s)\dot{\X}(s)\,ds\bigg|=0.
	\end{equation}
	
	From \eqref{eq:T1n1}, \eqref{eq:T1n2} and \eqref{eq:T1n3} we conclude that $T_{1}^{n}\rightarrow 0$ as $n\rightarrow\infty$. In a similar way we can prove that $T_{2}^{n}\rightarrow 0$. This implies by \eqref{eq:Z3L2Diff5} that $\Z_{3}^{n}\rightarrow\Z_{3}$ in $L^{2}(\mathbb{R})$.

	Using \eqref{eq:mapFtoGV1} we get
	\begin{equation*}
	\big|\big|\V_{1}-\V_{1}^{n}\big|\big|_{L^{2}(\mathbb{R})}\leq\bigg|\bigg|x_{1}'\bigg(\frac{1}{2c(U_{1})}-\frac{1}{2c(U_{1}^{n})}\bigg)\bigg|\bigg|_{L^{2}(\mathbb{R})}+\bigg|\bigg|\frac{1}{2c(U_{1}^{n})}\Big(x_{1}'-(x_{1}^{n})'\Big)\bigg|\bigg|_{L^{2}(\mathbb{R})},
	\end{equation*}
	and by inserting the estimates
	\begin{align*}
	&\bigg|\bigg|x_{1}'\bigg(\frac{1}{2c(U_{1})}-\frac{1}{2c(U_{1}^{n})}\bigg)\bigg|\bigg|_{L^{2}(\mathbb{R})}^{2}\\
	&\leq\frac{1}{4}\Big(||x_{1}'-1||_{L^{\infty}(\mathbb{R})}+1\Big)^{2}\int_{\mathbb{R}}\bigg(\int_{U_{1}^{n}(X)}^{U_{1}(X)}-\frac{c'(U)}{c^{2}(U)}\,dU\bigg)^{2}\,dX\\
	&\leq\frac{1}{4}k_{1}^{2}\kappa^{4}\Big(||x_{1}'-1||_{L^{\infty}(\mathbb{R})}+1\Big)^{2}\big|\big|U_{1}-U_{1}^{n}\big|\big|_{L^{2}(\mathbb{R})}^{2}
	\end{align*}
	and
	\begin{equation*}
	\bigg|\bigg|\frac{1}{2c(U_{1}^{n})}\Big(x_{1}'-(x_{1}^{n})'\Big)\bigg|\bigg|_{L^{2}(\mathbb{R})}
	\leq\frac{1}{2}\kappa\big|\big|x_{1}'-(x_{1}^{n})'\big|\big|_{L^{2}(\mathbb{R})}
	\end{equation*}
	we see that $\V_{1}^{n}\rightarrow\V_{1}$ in $L^{2}(\mathbb{R})$.

	From \eqref{eq:mapFtoGV2}-\eqref{eq:mapFtoGp} and the assumptions we immediately get that $\V_{i}^{n}\rightarrow\V_{i}$, $i=2,\dots,5$ and $\p^{n}\rightarrow\p$ in $L^{2}(\mathbb{R})$. 
	
	The corresponding results for $\W$ and $\q$ can be proved in a similar way.  
\end{proof}

Theorem \ref{thm:main} deals with convergence of the elements $u$, $\rho$ and $\sigma$ for $(u,R,S,\rho,\sigma,\mu,\nu)$ in $\D$, see \ref{eq:approxResult1} and \ref{eq:approxResult2}. The following result indicates the type of convergence we have to assume for the elements of the set $\G_{0}$ in order to get weak-star convergence of the remaining elements in $\D$.

\begin{lemma}
	\label{lemma:3}
	Let $\Theta=(\X,\Y,\Z,\V,\W,\p,\q)$ and $\Theta^{n}=(\X^{n},\Y^{n},\Z^{n},\V^{n},\W^{n},\p^{n},\q^{n})$ belong to $\G_{0}$. Consider 
	\begin{equation*}
		(u,R,S,\rho,\sigma,\mu,\nu)=\mathbf{M}\circ\mathbf{D}(\Theta)
	\end{equation*}
	and
	\begin{equation*}
		(u^{n},R^{n},S^{n},\rho^{n},\sigma^{n},\mu^{n},\nu^{n})=\mathbf{M}\circ\mathbf{D}(\Theta^{n}).
	\end{equation*}
	Assume that
	\begin{equation*}
	\X^{n}\rightarrow\X, \quad \Y^{n}\rightarrow\Y, \quad \Z_{i}^{n}\rightarrow\Z_{i} \quad \text{in} \quad \Linf(\mathbb{R}),
	\end{equation*}
	\begin{equation*}
	\Z_{3}^{n}\rightarrow\Z_{3}, \quad \V_{i}^{n}\rightarrow\V_{i}, \quad \W_{i}^{n}\rightarrow\W_{i}, \quad \p^{n}\rightarrow\p, \quad \q^{n}\rightarrow\q  \quad \text{in} \quad L^{2}(\mathbb{R}),
	\end{equation*}
	for $i=1,\dots,5$. Then
	\begin{equation*}
	u^{n}\rightarrow u \quad \text{in} \quad \Linf(\mathbb{R})
	\end{equation*}
	and
	\begin{equation*}
	R^{n}\overset{*}{\rightharpoonup}R, \quad S^{n}\overset{*}{\rightharpoonup}S, \quad
	\rho^{n}\overset{*}{\rightharpoonup}\rho, \quad \sigma^{n}\overset{*}{\rightharpoonup}\sigma, \quad
	\mu^{n}\overset{*}{\rightharpoonup}\mu \quad \text{and} \quad \nu^{n}\overset{*}{\rightharpoonup}\nu.
	\end{equation*}
\end{lemma}

Observe that there are no assumptions on the monotonicity of $(\X,\Y)$ and $(\X^{n},\Y^{n})$. This means that as functions of $s$ they are nondecreasing, but not necessarily strictly increasing.

From these results it follows immediately that $u^{n}_{x}\overset{*}{\rightharpoonup}u_{x}$ due to \eqref{eq:setDux}. 

\begin{proof}
	We will use Lemma \ref{lemma:mapG0toD}. For any $x$, there exist $s$ and $s^{n}$, which are not necessarily unique, such that $x=\Z_{2}(s)$ and $x=\Z_{2}^{n}(s^{n})$. By \eqref{eq:mapGtoD1}, we have $u(x)=\Z_{3}(s)$ and $u^{n}(x)=\Z_{3}^{n}(s^{n})$. We have
	\begin{equation}
	\label{eq:ConvInGImpliesInDuDiffALT}
	u(x)-u^{n}(x)=\Z_{3}(s)-\Z_{3}^{n}(s^{n})=\Z_{3}(s)-\Z_{3}(s^{n})+\Z_{3}(s^{n})-\Z_{3}^{n}(s^{n}).
	\end{equation}
	We estimate the difference $\Z_{3}(s)-\Z_{3}(s^{n})$. We assume that $s^{n}\leq s$, the other case can be treated similar. We have
	\begin{aalign}
		\label{eq:ConvInGImpliesInDZ3DiffALT}
		|\Z_{3}(s)-\Z_{3}(s^{n})|&=\left| \int_{s^{n}}^{s}\dot{\Z}_{3}(\bar{s})\,d\bar{s} \,\right|\\
		&=\left|\int_{s^{n}}^{s}(\V_{3}(\X(\bar{s}))\dot{\X}(\bar{s})+\W_{3}(\Y(\bar{s}))\dot{\Y}(\bar{s}))\,d\bar{s}\,\right|\\
		&\leq \bigg(\int_{s^{n}}^{s}\dot{\X}(\bar{s})\,d\bar{s}\bigg)^{\frac{1}{2}}\bigg(\int_{s^{n}}^{s}\V_{3}^{2}(\X(\bar{s}))\dot{\X}(\bar{s})\,d\bar{s}\bigg)^{\frac{1}{2}}\\
		&\quad +\bigg(\int_{s^{n}}^{s}\dot{\Y}(\bar{s})\,d\bar{s}\bigg)^{\frac{1}{2}}\bigg(\int_{s^{n}}^{s}\W_{3}^{2}(\Y(\bar{s}))\dot{\Y}(\bar{s})\,d\bar{s}\bigg)^{\frac{1}{2}}
	\end{aalign}
	by the Cauchy--Schwarz inequality. From \eqref{eq:setGrel3}, we get
	\begin{aalign}
		\label{eq:ConvInGImpliesInDV3ALT}
		&\int_{s^{n}}^{s}\V_{3}^{2}(\X(\bar{s}))\dot{\X}(\bar{s})\,d\bar{s}\\
		&=\int_{s^{n}}^{s}\bigg(\frac{2\V_{2}(\X(\bar{s}))\V_{4}(\X(\bar{s}))}{c^{2}(\Z_{3}(\bar{s}))}-\frac{\p^{2}(\X(\bar{s}))}{c(\Z_{3}(\bar{s}))}\bigg)\dot{\X}(\bar{s})\,d\bar{s}\\
		&\leq \int_{s^{n}}^{s}\frac{2\V_{2}(\X(\bar{s}))\V_{4}(\X(\bar{s}))}{c^{2}(\Z_{3}(\bar{s}))}\dot{\X}(\bar{s})\,d\bar{s}\\
		&\leq \kappa^{2}||\V_{4}^{a}||_{\Linf(\mathbb{R})}\int_{s^{n}}^{s}2\V_{2}(\X(\bar{s}))\dot{\X}(\bar{s})\,d\bar{s}\\
		&=\kappa^{2}||\V_{4}^{a}||_{\Linf(\mathbb{R})}\int_{s^{n}}^{s}\dot{\Z}_{2}(\bar{s})\,d\bar{s} \quad \text{by } \eqref{eq:setG0rel2}\\
		&=\kappa^{2}||\V_{4}^{a}||_{\Linf(\mathbb{R})}(\Z_{2}(s)-\Z_{2}(s^{n}))\\
		&=\kappa^{2}||\V_{4}^{a}||_{\Linf(\mathbb{R})}(\Z_{2}^{n}(s^{n})-\Z_{2}(s^{n})) \quad \text{since } \Z_{2}(s)=\Z_{2}^{n}(s^{n})\\
		&\leq\kappa^{2}||\V_{4}^{a}||_{\Linf(\mathbb{R})}||\Z_{2}^{n}-\Z_{2}||_{\Linf(\mathbb{R})}.
	\end{aalign}
	In a similar way, we obtain
	\begin{equation}
	\label{eq:ConvInGImpliesInDW3ALT}
	\int_{s^{n}}^{s}\W_{3}^{2}(\Y(\bar{s}))\dot{\Y}(\bar{s})\,d\bar{s}\leq\kappa^{2}||\W_{4}^{a}||_{\Linf(\mathbb{R})}||\Z_{2}^{n}-\Z_{2}||_{\Linf(\mathbb{R})}.
	\end{equation}
	
	Using $\Z_{2}(s)=\Z_{2}^{n}(s^{n})$ once more we get 
	\begin{aalign}
		\label{eq:ConvInGImpliesInDXALT}
		\int_{s^{n}}^{s}\dot{\X}(\bar{s})\,d\bar{s}&=\X(s)-\X(s^{n})\\
		&=\big(\X(s)-s\big)-\big(\X(s^{n})-s^{n}\big)-\big(\Z_{2}(s)-s\big)+\big(\Z_{2}(s)-s^{n}\big)\\
		&=\big(\X(s)-s\big)-\big(\X(s^{n})-s^{n}\big)-\big(\Z_{2}(s)-s\big)+\big(\Z_{2}^{n}(s^{n})-s^{n}\big)\\
		&=\big(\X(s)-s\big)-\big(\X(s^{n})-s^{n}\big)-\big(\Z_{2}(s)-s\big)\\
		&\quad+\big(\Z_{2}^{n}(s^{n})-\Z_{2}(s^{n})\big)+\big(\Z_{2}(s^{n})-s^{n}\big)\\
		&\leq 2||\X-\id||_{L^{\infty}(\mathbb{R})}+2||\Z_{2}-\id||_{L^{\infty}(\mathbb{R})}+||\Z_{2}-\Z_{2}^{n}||_{L^{\infty}(\mathbb{R})}
	\end{aalign}
	Similarly, we get
	\begin{equation}
	\label{eq:ConvInGImpliesInDYALT}
	\int_{s^{n}}^{s}\dot{\Y}(\bar{s})\,d\bar{s}\leq 2||\Y-\id||_{L^{\infty}(\mathbb{R})}+2||\Z_{2}-\id||_{L^{\infty}(\mathbb{R})}+||\Z_{2}-\Z_{2}^{n}||_{L^{\infty}(\mathbb{R})}.
	\end{equation}
	
	Combining \eqref{eq:ConvInGImpliesInDV3ALT}-\eqref{eq:ConvInGImpliesInDYALT} in \eqref{eq:ConvInGImpliesInDZ3DiffALT} and using that $\Z_{2}^{n}\rightarrow\Z_{2}$ in $\Linf(\mathbb{R})$ we find that $\Z_{3}(s^{n})\rightarrow \Z_{3}(s)$.
	Using this and that $\Z_{3}^{n}\rightarrow\Z_{3}$ in $L^{\infty}(\mathbb{R})$ in \eqref{eq:ConvInGImpliesInDuDiffALT} we conclude that $u^{n}\rightarrow u$ in $\Linf(\mathbb{R})$.
	
	We prove that $\mu^{n}\overset{*}{\rightharpoonup}\mu$, that is, 
	\begin{equation*}
	\lim_{n\rightarrow\infty}\int_{\mathbb{R}}\phi(x)\,d\mu^{n}=\int_{\mathbb{R}}\phi(x)\,d\mu
	\end{equation*}
	for all $\phi\in C_{0}(\mathbb{R})$. Here $C_{0}(\mathbb{R})$ is the space of continuous functions that vanish at infinity. Since $C_{c}^{\infty}(\mathbb{R})$ is dense in $C_{0}(\mathbb{R})$ it suffices to consider test functions $\phi$ in $C_{c}^{\infty}(\mathbb{R})$. By \eqref{eq:mapGtoD7} we have
	\begin{equation}
	\label{eq:rewritemu}
	\int_{\mathbb{R}}\phi(x)\,d\mu=\int_{\mathbb{R}}\phi(\Z_{2}(s))\V_{4}(\X(s))\dot{\X}(s)\,ds.
	\end{equation}
	By \eqref{eq:setGcomp} and \eqref{eq:setGrel2} we have
	\begin{align*}
	2\V_{4}(\X(s))\dot{\X}(s)&=\V_{4}(\X(s))\dot{\X}(s)+\W_{4}(\Y(s))\dot{\Y}(s)\\
	&\quad +\V_{4}(\X(s))\dot{\X}(s)-\W_{4}(\Y(s))\dot{\Y}(s)\\
	&=\dot{\Z}_{4}(s)+c(\Z_{3}(s))\big[\V_{5}(\X(s))\dot{\X}(s)+\W_{5}(\Y(s))\dot{\Y}(s)\big]\\
	&=\dot{\Z}_{4}(s)+c(\Z_{3}(s))\dot{\Z}_{5}(s).
	\end{align*}
	Inserting this in \eqref{eq:rewritemu} and using integration by parts yields
	\begin{align*}
	\int_{\mathbb{R}}\phi(x)\,d\mu&=\frac{1}{2}\int_{\mathbb{R}}\big[\phi(\Z_{2}(s))\dot{\Z}_{4}(s)+\phi(\Z_{2}(s))c(\Z_{3}(s))\dot{\Z}_{5}(s)\big]\,ds\\
	&=-\frac{1}{2}\int_{\mathbb{R}}\phi'(\Z_{2}(s))\dot{\Z}_{2}(s)\Z_{4}(s)\,ds\\
	&\quad-\frac{1}{2}\int_{\mathbb{R}}\big[\phi'(\Z_{2}(s))\dot{\Z}_{2}(s)c(\Z_{3}(s))+\phi(\Z_{2}(s))c'(\Z_{3}(s))\dot{\Z}_{3}(s)\big]\Z_{5}(s)\,ds.
	\end{align*}
	Similarly we find
	\begin{align*}
	\int_{\mathbb{R}}\phi(x)\,d\mu^{n}&=-\frac{1}{2}\int_{\mathbb{R}}\phi'(\Z_{2}^{n}(s))\dot{\Z}_{2}^{n}(s)\Z_{4}^{n}(s)\,ds\\
	&\quad-\frac{1}{2}\int_{\mathbb{R}}\big[\phi'(\Z_{2}^{n}(s))\dot{\Z}_{2}^{n}(s)c(\Z_{3}^{n}(s))+\phi(\Z_{2}^{n}(s))c'(\Z_{3}^{n}(s))\dot{\Z}_{3}^{n}(s)\big]\Z_{5}^{n}(s)\,ds.
	\end{align*}
	We subtract and get
	\begin{aalign}
		\label{eq:muDiff0}
		&\int_{\mathbb{R}}\phi(x)\,d\mu-\int_{\mathbb{R}}\phi(x)\,d\mu^{n}\\
		&=-\frac{1}{2}\bigg(\int_{\mathbb{R}}\big[\phi'(\Z_{2}(s))\dot{\Z}_{2}(s)\Z_{4}(s)-\phi'(\Z_{2}^{n}(s))\dot{\Z}_{2}^{n}(s)\Z_{4}^{n}(s)\big]\,ds\\
		&\hspace{45pt}+\int_{\mathbb{R}}\big[\phi'(\Z_{2}(s))\dot{\Z}_{2}(s)c(\Z_{3}(s))\Z_{5}(s)\\
		&\hspace{80pt}-\phi'(\Z_{2}^{n}(s))\dot{\Z}_{2}^{n}(s)c(\Z_{3}^{n}(s))\Z_{5}^{n}(s)\big]\,ds\\
		&\hspace{45pt}+\int_{\mathbb{R}}\big[\phi(\Z_{2}(s))c'(\Z_{3}(s))\dot{\Z}_{3}(s)\Z_{5}(s)\\
		&\hspace{80pt}-\phi(\Z_{2}^{n}(s))c'(\Z_{3}^{n}(s))\dot{\Z}_{3}^{n}(s)\Z_{5}^{n}(s)\big]\,ds\bigg).
	\end{aalign}
	The three integrals on the right-hand side of \eqref{eq:muDiff0} can be treated in more or less the same way, and we only consider the second one. We have
	\begin{aalign}
		\label{eq:muDiff}
		&\int_{\mathbb{R}}\big[\phi'(\Z_{2}(s))\dot{\Z}_{2}(s)c(\Z_{3}(s))\Z_{5}(s)-\phi'(\Z_{2}^{n}(s))\dot{\Z}_{2}^{n}(s)c(\Z_{3}^{n}(s))\Z_{5}^{n}(s)
		\big]\,ds\\
		&=\int_{\mathbb{R}}\dot{\Z}_{2}(s)c(\Z_{3}(s))\Z_{5}(s)\big[\phi'(\Z_{2}(s))-\phi'(\Z_{2}^{n}(s))\big]\,ds \quad (I_{1}^{n})\\
		&\quad+\int_{\mathbb{R}}\phi'(\Z_{2}^{n}(s))c(\Z_{3}(s))\Z_{5}(s)\big[\dot{\Z}_{2}(s)-\dot{\Z}_{2}^{n}(s)\big]\,ds \quad (I_{2}^{n})\\
		&\quad+\int_{\mathbb{R}}\phi'(\Z_{2}^{n}(s))\dot{\Z}_{2}^{n}(s)c(\Z_{3}(s))\big[\Z_{5}(s)-\Z_{5}^{n}(s)\big]\,ds \quad (I_{3}^{n})\\
		&\quad+\int_{\mathbb{R}}\phi'(\Z_{2}^{n}(s))\dot{\Z}_{2}^{n}(s)\Z_{5}^{n}(s)\big[c(\Z_{3}(s))-c(\Z_{3}^{n}(s))\big]\,ds \quad (I_{4}^{n}).
	\end{aalign}
	
	We have
	\begin{equation}
	\label{eq:I1Bound1}
	\big|I_{1}^{n}\big|\leq 4\kappa\bigg(||\V_{2}^{a}||_{L^{\infty}(\mathbb{R})}+\frac{1}{2}\bigg)||\Z_{5}||_{L^{\infty}(\mathbb{R})}||\phi'\circ\Z_{2}-\phi'\circ\Z_{2}^{n}||_{L^{1}(\mathbb{R})}
	\end{equation}
	since 
	\begin{equation*}
	0\leq\dot{\Z}_{2}(s)=2\V_{2}(\X(s))\dot{\X}(s)\leq 4\V_{2}(\X(s))=4\bigg(\V_{2}^{a}(\X(s))+\frac{1}{2}\bigg).
	\end{equation*}
	We have $\phi'\circ\Z_{2}^{n}\rightarrow\phi'\circ\Z_{2}$ pointwise almost everywhere, and by Lemma~\ref{lem:UnifBdd} we find that $\phi'\circ\Z_{2}^{n}$ can be uniformly bounded by an $L^{1}(\mathbb{R})$ function. By the dominated convergence theorem we get $\displaystyle\lim_{n\rightarrow\infty}||\phi'\circ\Z_{2}-\phi'\circ\Z_{2}^{n}||_{L^{1}(\mathbb{R})}=0$, which from \eqref{eq:I1Bound1} implies $\displaystyle\lim_{n\rightarrow\infty}I_{1}^{n}=0$.
	
	Integration by parts yields
	\begin{align*}
	I_{2}^{n}&=\int_{\mathbb{R}}\big[\phi''(\Z_{2}^{n}(s))\dot{\Z}_{2}^{n}(s)c(\Z_{3}(s))\Z_{5}(s)+\phi'(\Z_{2}^{n}(s))c'(\Z_{3}(s))\dot{\Z}_{3}(s)\Z_{5}(s)\\
	&\hspace{28pt}+\phi'(\Z_{2}^{n}(s))c(\Z_{3}(s))\dot{\Z}_{5}(s)\big]\big[\Z_{2}(s)-\Z_{2}^{n}(s)\big]\,ds
	\end{align*}
	which implies
	\begin{align*}
	|I_{2}^{n}|&\leq\bigg(\kappa||\Z_{5}||_{L^{\infty}(\mathbb{R})}\int_{\mathbb{R}}\big|\phi''(\Z_{2}^{n}(s))\big|\dot{\Z}_{2}^{n}(s)\,ds\\
	&\hspace{25pt}+\Big[2k_{1}||\Z_{5}||_{L^{\infty}(\mathbb{R})}\big(||\V_{3}^{a}||_{L^{\infty}(\mathbb{R})}+||\W_{3}^{a}||_{L^{\infty}(\mathbb{R})}\big)\\
	&\hspace{45pt}+2\kappa\big(||\V_{5}^{a}||_{L^{\infty}(\mathbb{R})}+||\W_{5}^{a}||_{L^{\infty}(\mathbb{R})}\big)\Big]\\
	&\hspace{35pt}\times\int_{\mathbb{R}}\big|\phi'(\Z_{2}^{n}(s))\big|\,ds	
	\bigg)||\Z_{2}-\Z_{2}^{n}||_{L^{\infty}(\mathbb{R})}.
	\end{align*}
	By a change of variables and an estimate as in the proof of Lemma~\ref{lem:UnifBdd}
	we get
	\begin{align*}
	|I_{2}^{n}|&\leq\bigg(\kappa||\Z_{5}||_{L^{\infty}(\mathbb{R})}||\phi''||_{L^{\infty}(\mathbb{R})}\text{meas}\big(\text{supp}(\phi'')\big)\\
	&\hspace{25pt}+\Big[2k_{1}||\Z_{5}||_{L^{\infty}(\mathbb{R})}\big(||\V_{3}^{a}||_{L^{\infty}(\mathbb{R})}+||\W_{3}^{a}||_{L^{\infty}(\mathbb{R})}\big)\\
	&\hspace{45pt}+2\kappa\big(||\V_{5}^{a}||_{L^{\infty}(\mathbb{R})}+||\W_{5}^{a}||_{L^{\infty}(\mathbb{R})}\big)\Big]\\
	&\hspace{35pt}\times \tilde{C}||\phi'||_{L^{\infty}(\mathbb{R})}	
	\bigg)||\Z_{2}-\Z_{2}^{n}||_{L^{\infty}(\mathbb{R})}
	\end{align*}
	for a constant $\tilde{C}$ that is independent of $n$. Since $\Z_{2}^{n}\rightarrow\Z_{2}$ in $L^{\infty}(\mathbb{R})$ we get $\displaystyle\lim_{n\rightarrow\infty}I_{2}^{n}=0$.
	
	For the third integral we use a change of variables and get
	\begin{align*}
	|I_{3}^{n}|&\leq \kappa ||\Z_{5}-\Z_{5}^{n}||_{L^{\infty}(\mathbb{R})} \int_{\mathbb{R}}\big|\phi'(\Z_{2}^{n}(s))\big|\dot{\Z}_{2}^{n}(s)\,ds\\
	&\leq\kappa\,\text{meas}\big(\text{supp}(\phi')\big)||\phi'||_{L^{\infty}(\mathbb{R})}||\Z_{5}-\Z_{5}^{n}||_{L^{\infty}(\mathbb{R})},
	\end{align*}
	which implies since $\Z_{5}^{n}\rightarrow\Z_{5}$ in $L^{\infty}(\mathbb{R})$ that $\displaystyle\lim_{n\rightarrow\infty}I_{3}^{n}=0$.
	
	By the Cauchy--Schwarz inequality we get
	\begin{align*}
	|I_{4}^{n}|&\leq ||c(\Z_{3})-c(\Z_{3}^{n})||_{L^{\infty}(\mathbb{R})} ||\Z_{5}^{n}||_{L^{\infty}(\mathbb{R})}\int_{\mathbb{R}}\big|\phi'(\Z_2^n(s))\big|\dot{\Z}_{2}^{n}(s)\,ds\\
	&\leq ||c(\Z_{3})-c(\Z_{3}^{n})||_{L^{\infty}(\mathbb{R})}\big(||\Z_{5}^{n}-\Z_{5}||_{L^{\infty}(\mathbb{R})}+||\Z_{5}||_{L^{\infty}(\mathbb{R})}\big)||\phi'(\Z_2^n(s))\dot{\Z}_{2}^{n}||_{L^{1}(\mathbb{R})}
	\end{align*}
	and by inserting the estimates
	\begin{align*}
	||\phi'(\Z_2^n)\dot{\Z}_{2}^{n}||_{L^{1}(\mathbb{R})}&=||\phi'||_{L^{1}(\mathbb{R})}\\
	&\leq ||\phi'||_{L^{\infty}(\mathbb{R})}\text{meas}\big(\text{supp}(\phi')\big)	
	\end{align*}
	and
	\begin{equation*}
	||c(\Z_{3})-c(\Z_{3}^{n})||_{L^\infty(\mathbb{R})} \leq k_{1}||\Z_{3}-\Z_{3}^{n}||_{L^\infty(\mathbb{R})}
	\end{equation*}
	we obtain
	\begin{align*}	
	|I_{4}^{n}|\leq k_{1}||\phi'||_{L^{\infty}(\mathbb{R})}\text{meas}\big(\text{supp}(\phi')\big)\big(||\Z_{5}^{n}-\Z_{5}||_{L^{\infty}(\mathbb{R})}+||\Z_{5}||_{L^{\infty}(\mathbb{R})}\big)	||\Z_{3}-\Z_{3}^{n}||_{L^\infty(\mathbb{R})}
	\end{align*}
	Since $\Z_{5}^{n}\rightarrow\Z_{5}$ and $\Z_{3}^{n}\rightarrow\Z_{3}$ in $L^{\infty}(\mathbb{R})$ we obtain $\displaystyle\lim_{n\rightarrow\infty}I_{4}^{n}=0$.
	
	We return to \eqref{eq:muDiff} and obtain
	\begin{aalign}
	\label{eq:I1toI4}
	&\lim_{n\rightarrow\infty}\bigg|\int_{\mathbb{R}}\big[\phi'(\Z_{2}(s))\dot{\Z}_{2}(s)c(\Z_{3}(s))\Z_{5}(s)\\
	&\hspace{55pt}-\phi'(\Z_{2}^{n}(s))\dot{\Z}_{2}^{n}(s)c(\Z_{3}^{n}(s))\Z_{5}^{n}(s)\big]\,ds\bigg|=0.
	\end{aalign}
	
	By using \eqref{eq:I1toI4} in \eqref{eq:muDiff0} we get
	\begin{equation}
	\label{eq:measDiff1}
	\lim_{n\rightarrow\infty}\bigg|\int_{\mathbb{R}}\phi(x)\,d\mu-\int_{\mathbb{R}}\phi(x)\,d\mu^{n}\bigg|=0
	\end{equation}
	for all $\phi\in C_{c}^{\infty}(\mathbb{R})$, and we conclude that $\mu^{n}\overset{*}{\rightharpoonup}\mu$. Similarly we prove that $\nu^{n}\overset{*}{\rightharpoonup}\nu$.
	
	Next we show that $\rho^{n}\overset{*}{\rightharpoonup}\rho$, that is,
	\begin{equation*}
	\lim_{n\rightarrow\infty}\int_{\mathbb{R}}\rho^{n}(x)\phi(x)\,dx=\int_{\mathbb{R}}\rho(x)\phi(x)\,dx
	\end{equation*}
	for all $\phi\in L^{2}(\mathbb{R})$. Since $C_{c}^{\infty}(\mathbb{R})$ is dense in $L^{2}(\mathbb{R})$, it suffices to consider test functions $\phi$ in $C_{c}^{\infty}(\mathbb{R})$.
	
	Consider a test function $\phi$ in $C_{c}^{\infty}(\mathbb{R})$. By \eqref{eq:mapGtoD5}, we have
	\begin{equation*}
	\int_{\mathbb{R}}\phi(x)\rho(x)\,dx=2\int_{\mathbb{R}}\phi(\Z_{2}(s))\p(\X(s))\dot{\X}(s)\,ds
	\end{equation*}
	and
	\begin{equation*}
	\int_{\mathbb{R}}\phi(x)\rho^{n}(x)\,dx=2\int_{\mathbb{R}}\phi(\Z^{n}_{2}(s))\p^{n}(\X^{n}(s))\dot{\X}^{n}(s)\,ds.
	\end{equation*}
	By subtracting we get
	\begin{aalign}
		\label{eq:A1andA2}
		&\int_{\mathbb{R}}\phi(x)\big[\rho(x)-\rho^{n}(x)\big]\,dx\\
		&=2\int_{\mathbb{R}}\phi(\Z_{2}(s))\big[\p(\X(s))\dot{\X}(s)-\p^{n}(\X^{n}(s))\dot{\X}^{n}(s)\big]\,ds \quad (A_{1}^{n})\\
		&\quad+2\int_{\mathbb{R}}\p^{n}(\X^{n}(s))\dot{\X}^{n}(s)\big[\phi(\Z_{2}(s))-\phi(\Z^{n}_{2}(s))\big]\,ds \quad (A_{2}^{n}).
	\end{aalign}
	
	Since $\phi$ has compact support, there exists $k>0$ such that $\text{supp}(\phi)\subset[-k,k]$. Thus, we only integrate over $s$ such that $-k\leq\Z_{2}(s)\leq k$ in $A_1^n$. Since the quantity $\Z_{2}-\id$ belongs to $L^{\infty}(\mathbb{R})$, the region we integrate over is contained in 
	\begin{equation*}
		-k-||\Z_{2}-\id||_{L^{\infty}(\mathbb{R})}\leq s \leq k+||\Z_{2}-\id||_{L^{\infty}(\mathbb{R})}, 
	\end{equation*}
	or written more compactly, $-M\leq s\leq M$ with $M=k+||\Z_{2}-\id||_{L^{\infty}(\mathbb{R})}$. Integration by parts yields
	\begin{aalign}
		\label{eq:A1}
		A_{1}^{n}&=2\int_{-M}^{M}\phi(\Z_{2}(s))\big[\p(\X(s))\dot{\X}(s)-\p^{n}(\X^{n}(s))\dot{\X}^{n}(s)\big]\,ds\\
		&=2\bigg[\phi(\Z_{2}(s))\int_{-M}^{s}\big[\p(\X(\tau))\dot{\X}(\tau)-\p^{n}(\X^{n}(\tau))\dot{\X}^{n}(\tau)\big]\,d\tau\bigg]_{s=-M}^{s=M}\\
		&\quad-2\int_{-M}^{M}\phi'(\Z_{2}(s))\dot{\Z}_{2}(s)\\
		&\hspace{30pt}\times\bigg(\int_{-M}^{s}\big[\p(\X(\tau))\dot{\X}(\tau)-\p^{n}(\X^{n}(\tau))\dot{\X}^{n}(\tau)\big]\,d\tau\bigg)\,ds.
	\end{aalign}
	By a change of variables we have
	\begin{equation*}
	\int_{-M}^{s}\big[\p(\X(\tau))\dot{\X}(\tau)-\p^{n}(\X^{n}(\tau))\dot{\X}^{n}(\tau)\big]\,d\tau=\int_{\X(-M)}^{\X(s)}\p(X)\,dX-\int_{\X^{n}(-M)}^{\X^{n}(s)}\p^{n}(X)\,dX.
	\end{equation*}
	Using an estimate as in \eqref{eq:U1U1nDiff} yields
	\begin{align*}
	&\bigg|\int_{-M}^{s}\big[\p(\X(\tau))\dot{\X}(\tau)-\p^{n}(\X^{n}(\tau))\dot{\X}^{n}(\tau)\big]\,d\tau\bigg|\\
	&\leq ||\p-\p^{n}||_{L^{2}(\mathbb{R})}\big(2||\X-\id||_{L^{\infty}(\mathbb{R})}+2M\big)^{\frac{1}{2}}\\
	&\quad+2\big(||\p^{n}-\p||_{L^{2}(\mathbb{R})}+||\p||_{L^{2}(\mathbb{R})}\big)||\X-\X^{n}||_{L^{\infty}(\mathbb{R})}^{\frac{1}{2}}.
	\end{align*}
	This implies that the first term on the right-hand side of \eqref{eq:A1} equals zero. Moreover, we get
	\begin{align*}	
	|A_{1}^{n}|&\leq 2\Big[||\p-\p^{n}||_{L^{2}(\mathbb{R})}\big(2||\X-\id||_{L^{\infty}(\mathbb{R})}+2M\big)^{\frac{1}{2}}\\
	&\hspace{20pt}+2\big(||\p^{n}-\p||_{L^{2}(\mathbb{R})}+||\p||_{L^{2}(\mathbb{R})}\big)||\X-\X^{n}||_{L^{\infty}(\mathbb{R})}^{\frac{1}{2}}\Big]\int_{-M}^{M}|\phi'(\Z_{2}(s))|\dot{\Z}_{2}(s)\,ds\\
	&\leq 2\Big[||\p-\p^{n}||_{L^{2}(\mathbb{R})}\big(2||\X-\id||_{L^{\infty}(\mathbb{R})}+2M\big)^{\frac{1}{2}}\\
	&\hspace{20pt}+2\big(||\p^{n}-\p||_{L^{2}(\mathbb{R})}+||\p||_{L^{2}(\mathbb{R})}\big)||\X-\X^{n}||_{L^{\infty}(\mathbb{R})}^{\frac{1}{2}}\Big]||\phi'||_{L^{\infty}(\mathbb{R})}\text{meas}\big(\text{supp}(\phi')\big)
	\end{align*}
	where we used that
	\begin{align*}
	\int_{-M}^{M}|\phi'(\Z_{2}(s))|\dot{\Z}_{2}(s)\,ds&\leq \int_{\mathbb{R}}|\phi'(\Z_{2}(s))|\dot{\Z}_{2}(s)\,ds\\
	&=\int_{\mathbb{R}}|\phi'(x)|\,dx\leq ||\phi'||_{L^{\infty}(\mathbb{R})}\text{meas}\big(\text{supp}(\phi')\big).
	\end{align*}
	Since $\X^{n}\rightarrow\X$ in $L^{\infty}(\mathbb{R})$ and $\p^{n}\rightarrow\p$ in $L^{2}(\mathbb{R})$ we find that $\displaystyle\lim_{n\rightarrow\infty}A_{1}^{n}=0$.
	
	The second integral in \eqref{eq:A1andA2} is estimated as follows. Using the Cauchy--Schwarz inequality and that $0\leq\dot{\X}^{n}\leq 2$ we obtain
	\begin{align*}
	|A_{2}^{n}|&\leq 2\bigg(\int_{\mathbb{R}}\big(\p^{n}(\X^{n}(s))\dot{\X}^{n}(s)\big)^{2}\,ds\bigg)^{\frac{1}{2}}||\phi\circ\Z_{2}-\phi\circ\Z^{n}_{2}||_{L^{2}(\mathbb{R})}\\
	&\leq 2\sqrt{2}||\p^{n}||_{L^{2}(\mathbb{R})}||\phi\circ\Z_{2}-\phi\circ\Z^{n}_{2}||_{L^{2}(\mathbb{R})}\\
	&\leq 2\sqrt{2}\big(||\p^{n}-\p||_{L^{2}(\mathbb{R})}+||\p||_{L^{2}(\mathbb{R})}\big)||\phi\circ\Z_{2}-\phi\circ\Z^{n}_{2}||_{L^{2}(\mathbb{R})}
	\end{align*}
	by a change of variables. We have $\phi\circ\Z_{2}^{n}\rightarrow \phi\circ\Z_{2}$ pointwise almost everywhere, and by Lemma~\ref{lem:UnifBdd} we get that $\phi\circ\Z^{n}_{2}$ can be uniformly bounded by an $L^{2}(\mathbb{R})$ function, so the dominated convergence theorem implies that $\displaystyle\lim_{n\rightarrow\infty} ||\phi\circ\Z_{2}-\phi\circ\Z^{n}_{2}||_{L^{2}(\mathbb{R})}=0$. Since also $\p^{n}\rightarrow\p$ in $L^{2}(\mathbb{R})$ we find that $\displaystyle\lim_{n\rightarrow\infty}A_{2}^{n}=0$.  
	
	We return to \eqref{eq:A1andA2}, and conclude that
	\begin{equation*}
	\lim_{n\rightarrow\infty}\bigg|\int_{\mathbb{R}}\phi(x)\big[\rho(x)-\rho^{n}(x)\big]\,dx\bigg|=0.
	\end{equation*}
	Therefore, $\rho^{n}\overset{*}{\rightharpoonup}\rho$. In a similar way one shows that $\sigma^{n}\overset{*}{\rightharpoonup}\sigma$.
	
	We prove that $R^{n}\overset{*}{\rightharpoonup}R$. Consider a test function $\phi$ in $C_{c}^{\infty}(\mathbb{R})$. By \eqref{eq:mapGtoD3} we can write
	\begin{align*}
	&\int_{\mathbb{R}}\phi(x)\big[R(x)-R^{n}(x)\big]\,dx\\
	&=2\int_{\mathbb{R}}\phi(\Z_{2}(s))c(\Z_{3}(s))\big[\V_{3}(\X(s))\dot{\X}(s)-\V_{3}^{n}(\X^{n}(s))\dot{\X}^{n}(s)\big]\,ds\\
	&\quad +2\int_{\mathbb{R}}c(\Z_{3}(s))\V_{3}^{n}(\X^{n}(s))\dot{\X}^{n}(s)\big[\phi(\Z_{2}(s))-\phi(\Z_{2}^{n}(s))\big]\,ds\\
	&\quad +2\int_{\mathbb{R}}\phi(\Z_{2}^{n}(s))\V_{3}^{n}(\X^{n}(s))\dot{\X}^{n}(s)\big[c(\Z_{3}(s))-c(\Z_{3}^{n}(s))\big]\,ds.
ig	\end{align*}
	The first and second integral can be treated more or less like $A_{1}^{n}$ and $A_{2}^{n}$ in \eqref{eq:A1andA2}, respectively. The last integral is estimated as follows
	\begin{align*}
	&\bigg|\int_{\mathbb{R}}\phi(\Z_{2}^{n}(s))\V_{3}^{n}(\X^{n}(s))\dot{\X}^{n}(s)\big[c(\Z_{3}(s))-c(\Z_{3}^{n}(s))\big]\,ds\bigg|\\
	&\leq ||\phi||_{L^{\infty}(\mathbb{R})}||(\V_{3}^{n}\circ\X^{n})\dot{\X}^{n}||_{L^{2}(\mathbb{R})} ||c\circ\Z_{3}-c\circ\Z_{3}^{n}||_{L^{2}(\mathbb{R})}\\
	&\leq \sqrt{2}k_{1} ||\phi||_{L^{\infty}(\mathbb{R})}\big( ||\V_{3}^{n}-\V_{3}||_{L^{2}(\mathbb{R})}+||\V_{3}||_{L^{2}(\mathbb{R})}\big) ||\Z_{3}-\Z_{3}^{n}||_{L^{2}(\mathbb{R})}.
	\end{align*} 
	This implies that $R^{n}\overset{*}{\rightharpoonup}R$. In a similar way one shows that $S^{n}\overset{*}{\rightharpoonup}S$.
\end{proof}

\textbf{Acknowledgments.} The authors are very grateful to Xavier Raynaud for proposing the regularized system, and to Helge Holden for many discussions.

\end{document}